\documentclass[a4paper]{amsart}

\RequirePackage{amsmath} \RequirePackage{amssymb}
\usepackage{amscd,latexsym,amsthm,amsfonts,amssymb,amsmath,amsxtra}
\usepackage[colorlinks=true,urlcolor=blue,citecolor=blue]{hyperref}
\usepackage{color}
\usepackage[all]{xy}
\usepackage[OT2,T1]{fontenc}
\usepackage{mathtools}
\usepackage{mathrsfs}

\DeclareSymbolFont{cyrletters}{OT2}{wncyr}{m}{n}
\DeclareMathSymbol{\Sha}{\mathalpha}{cyrletters}{"58}

\let\Re\undefined
\let\Im\undefined

\DeclareMathOperator{\Re}{Re}
\DeclareMathOperator{\Im}{Im}

\DeclareMathOperator{\Tr}{Tr}

\DeclareMathOperator{\supp}{supp}

\DeclareMathOperator{\GL}{GL}

\DeclareMathOperator{\vol}{vol}

\begin{document}
	
	\theoremstyle{plain}
\newtheorem{thm}{Theorem} \newtheorem{cor}[thm]{Corollary}
\newtheorem{thmy}{Theorem}
\renewcommand{\thethmy}{\Alph{thmy}}
\newenvironment{thmx}{\stepcounter{thm}\begin{thmy}}{\end{thmy}}
	\newtheorem{lemma}[thm]{Lemma}  \newtheorem{prop}[thm]{Proposition}
	\newtheorem{conj}[thm]{Conjecture}  \newtheorem{fact}[thm]{Fact}
	\newtheorem{claim}[thm]{Claim}
	\theoremstyle{definition}
	\newtheorem{defn}[thm]{Definition}
	\newtheorem{example}[thm]{Example}
	\newtheorem{exercise}[thm]{Exercise}
	\theoremstyle{remark}
	\newtheorem*{remark}{Remark}

	\newcommand{\BA}{{\mathbb {A}}} \newcommand{\BB}{{\mathbb {B}}}
	\newcommand{\BC}{{\mathbb {C}}} \newcommand{\BD}{{\mathbb {D}}}
	\newcommand{\BE}{{\mathbb {E}}} \newcommand{\BF}{{\mathbb {F}}}
	\newcommand{\BG}{{\mathbb {G}}} \newcommand{\BH}{{\mathbb {H}}}
	\newcommand{\BI}{{\mathbb {I}}} \newcommand{\BJ}{{\mathbb {J}}}
	\newcommand{\BK}{{\mathbb {K}}} \newcommand{\BL}{{\mathbb {L}}}
	\newcommand{\BM}{{\mathbb {M}}} \newcommand{\BN}{{\mathbb {N}}}
	\newcommand{\BO}{{\mathbb {O}}} \newcommand{\BP}{{\mathbb {P}}}
	\newcommand{\BQ}{{\mathbb {Q}}} \newcommand{\BR}{{\mathbb {R}}}
	\newcommand{\BS}{{\mathbb {S}}} \newcommand{\BT}{{\mathbb {T}}}
	\newcommand{\BU}{{\mathbb {U}}} \newcommand{\BV}{{\mathbb {V}}}
	\newcommand{\BW}{{\mathbb {W}}} \newcommand{\BX}{{\mathbb {X}}}
	\newcommand{\BY}{{\mathbb {Y}}} \newcommand{\BZ}{{\mathbb {Z}}}
	
	\newcommand{\CA}{{\mathcal {A}}} \newcommand{\CB}{{\mathcal {B}}}
	\newcommand{\CC}{{\mathcal {C}}} \renewcommand{\CD}{{\mathcal {D}}}
	\newcommand{\CE}{{\mathcal {E}}} \newcommand{\CF}{{\mathcal {F}}}
	\newcommand{\CG}{{\mathcal {G}}} \newcommand{\CH}{{\mathcal {H}}}
	\newcommand{\CI}{{\mathcal {I}}} \newcommand{\CJ}{{\mathcal {J}}}
	\newcommand{\CK}{{\mathcal {K}}} \newcommand{\CL}{{\mathcal {L}}}
	\newcommand{\CM}{{\mathcal {M}}} \newcommand{\CN}{{\mathcal {N}}}
	\newcommand{\CO}{{\mathcal {O}}} \newcommand{\CP}{{\mathcal {P}}}
	\newcommand{\CQ}{{\mathcal {Q}}} \newcommand{\CR}{{\mathcal {R}}}
	\newcommand{\CS}{{\mathcal {S}}} \newcommand{\CT}{{\mathcal {T}}}
	\newcommand{\CU}{{\mathcal {U}}} \newcommand{\CV}{{\mathcal {V}}}
	\newcommand{\CW}{{\mathcal {W}}} \newcommand{\CX}{{\mathcal {X}}}
	\newcommand{\CY}{{\mathcal {Y}}} \newcommand{\CZ}{{\mathcal {Z}}}
	
	\newcommand{\RA}{{\mathrm {A}}} \newcommand{\RB}{{\mathrm {B}}}
	\newcommand{\RC}{{\mathrm {C}}} \newcommand{\RD}{{\mathrm {D}}}
	\newcommand{\RE}{{\mathrm {E}}} \newcommand{\RF}{{\mathrm {F}}}
	\newcommand{\RG}{{\mathrm {G}}} \newcommand{\RH}{{\mathrm {H}}}
	\newcommand{\RI}{{\mathrm {I}}} \newcommand{\RJ}{{\mathrm {J}}}
	\newcommand{\RK}{{\mathrm {K}}} \newcommand{\RL}{{\mathrm {L}}}
	\newcommand{\RM}{{\mathrm {M}}} \newcommand{\RN}{{\mathrm {N}}}
	\newcommand{\RO}{{\mathrm {O}}} \newcommand{\RP}{{\mathrm {P}}}
	\newcommand{\RQ}{{\mathrm {Q}}} \newcommand{\RR}{{\mathrm {R}}}
	\newcommand{\RS}{{\mathrm {S}}} \newcommand{\RT}{{\mathrm {T}}}
	\newcommand{\RU}{{\mathrm {U}}} \newcommand{\RV}{{\mathrm {V}}}
	\newcommand{\RW}{{\mathrm {W}}} \newcommand{\RX}{{\mathrm {X}}}
	\newcommand{\RY}{{\mathrm {Y}}} \newcommand{\RZ}{{\mathrm {Z}}}
	
	\newcommand{\fa}{{\mathfrak{a}}} \newcommand{\fb}{{\mathfrak{b}}}
	\newcommand{\fc}{{\mathfrak{c}}} \newcommand{\fd}{{\mathfrak{d}}}
	\newcommand{\fe}{{\mathfrak{e}}} \newcommand{\ff}{{\mathfrak{f}}}
	\newcommand{\fg}{{\mathfrak{g}}} \newcommand{\fh}{{\mathfrak{h}}}
	\newcommand{\fii}{{\mathfrak{i}}} \newcommand{\fj}{{\mathfrak{j}}}
	\newcommand{\fk}{{\mathfrak{k}}} \newcommand{\fl}{{\mathfrak{l}}}
	\newcommand{\fm}{{\mathfrak{m}}} \newcommand{\fn}{{\mathfrak{n}}}
	\newcommand{\fo}{{\mathfrak{o}}} \newcommand{\fp}{{\mathfrak{p}}}
	\newcommand{\fq}{{\mathfrak{q}}} \newcommand{\fr}{{\mathfrak{r}}}
	\newcommand{\fs}{{\mathfrak{s}}} \newcommand{\ft}{{\mathfrak{t}}}
	\newcommand{\fu}{{\mathfrak{u}}} \newcommand{\fv}{{\mathfrak{v}}}
	\newcommand{\fw}{{\mathfrak{w}}} \newcommand{\fx}{{\mathfrak{x}}}
	\newcommand{\fy}{{\mathfrak{y}}} \newcommand{\fz}{{\mathfrak{z}}}
	\newcommand{\fA}{{\mathfrak{A}}} \newcommand{\fB}{{\mathfrak{B}}}
	\newcommand{\fC}{{\mathfrak{C}}} \newcommand{\fD}{{\mathfrak{D}}}
	\newcommand{\fE}{{\mathfrak{E}}} \newcommand{\fF}{{\mathfrak{F}}}
	\newcommand{\fG}{{\mathfrak{G}}} \newcommand{\fH}{{\mathfrak{H}}}
	\newcommand{\fI}{{\mathfrak{I}}} \newcommand{\fJ}{{\mathfrak{J}}}
	\newcommand{\fK}{{\mathfrak{K}}} \newcommand{\fL}{{\mathfrak{L}}}
	\newcommand{\fM}{{\mathfrak{M}}} \newcommand{\fN}{{\mathfrak{N}}}
	\newcommand{\fO}{{\mathfrak{O}}} \newcommand{\fP}{{\mathfrak{P}}}
	\newcommand{\fQ}{{\mathfrak{Q}}} \newcommand{\fR}{{\mathfrak{R}}}
	\newcommand{\fS}{{\mathfrak{S}}} \newcommand{\fT}{{\mathfrak{T}}}
	\newcommand{\fU}{{\mathfrak{U}}} \newcommand{\fV}{{\mathfrak{V}}}
	\newcommand{\fW}{{\mathfrak{W}}} \newcommand{\fX}{{\mathfrak{X}}}
	\newcommand{\fY}{{\mathfrak{Y}}} \newcommand{\fZ}{{\mathfrak{Z}}}
	\newcommand{\Int}{\operatorname{Int}}
	\newcommand{\Res}{\operatorname{Res}}
	\newcommand{\tr}{\operatorname{tr}}
	\newcommand{\Eis}{\operatorname{Eis}}
	\newcommand{\Geo}{\operatorname{Geo}}
	\newcommand{\End}{\operatorname{End}}
	\newcommand{\K}{\operatorname{K}}
	\newcommand{\sgn}{\operatorname{sgn}}
	\newcommand{\Ker}{\operatorname{Ker}}
	\newcommand{\Ad}{\operatorname{Ad}}
	\newcommand{\Weyl}{\operatorname{Weyl}}
	\newcommand{\ram}{\operatorname{ram}}
	\newcommand{\Supp}{\operatorname{Supp}}
	\newcommand{\Cond}{\operatorname{Cond}}
	\newcommand{\diag}{\operatorname{diag}}
	\newcommand{\Ind}{\operatorname{Ind}}
	\newcommand{\E}{\operatorname{E}}
	\newcommand{\coker}{\operatorname{coker}}
	\newcommand{\Out}{\operatorname{Out}}
	\newcommand{\V}{\operatorname{V}}
	\newcommand{\Span}{\operatorname{Span}}
	\newcommand{\Sin}{\operatorname{Sing}}
	\newcommand{\Con}{\operatorname{Const}}
	\newcommand{\Reg}{\operatorname{Reg}}
	\newcommand{\dis}{\operatorname{simp}}
	\newcommand{\fin}{\operatorname{fin}}
	\newcommand{\Whi}{\operatorname{Whittaker}}
	\newcommand{\RomanNumeralCaps}[1]
	{\MakeUppercase{\romannumeral #1}}
	\title[A Coarse Jacquet-Zagier Trace Formula for $\GL(n)$]{A Coarse Jacquet-Zagier Trace Formula for GL($n$) with Applications}%
	\author{Liyang Yang}

	\address{253-37 Caltech, Pasadena\\
		CA 91125, USA}
	\email{lyyang@caltech.edu}

\begin{abstract}
In this paper we establish a coarse Jacquet-Zagier trace identity for $\GL(n).$ We prove the absolute convergence when $\Re(s)>1$ and $0<\Re(s)<1;$ and obtain holomorphic continuation under almost all character twist. Moreover, as an application, we prove that holomorphy of certain adjoint $L$-functions for $\GL(n)$ implies Dedekind conjecture of degree $n$. Some nonvanishing results are also discussed. 
	\end{abstract}
	
	\date{\today}%
	\maketitle
	\tableofcontents
	\section{Introduction}
	\subsection{Trace Formula: from Arthur-Selberg to Jacquet-Zagier}
	Let $F$ be a global field, with adele ring $\mathbb{A}_F.$ Let $G=\GL(n).$ We consider a smooth function $\varphi:$ $G(\mathbb{A}_F) \rightarrow \mathbb{C}$ which is left and right $K$-finite for a compact subgroup $K$ of $G(\mathbb{A}_F)$, transforms by a unitary character $\omega$ of $Z_G\left(\mathbb{A}_F\right),$ and has compact support modulo $Z_G\left(\mathbb{A}_F\right).$ Denote by $\mathcal{H}(G(\mathbb{A}_F),\omega)$ the set of such functions. Then $\varphi\in \mathcal{H}(G(\mathbb{A}_F),\omega)$ defines an integral operator 
	$$
	R(\varphi)f(y)=\int_{Z_G(\mathbb{A}_F)\backslash G(\mathbb{A}_F)}\varphi(x)f(yx)dx,
	$$ 
	on the space $L^2\left(G(F)\backslash G(\mathbb{A}_F),\omega^{-1}\right)$ of functions on $G(F)\backslash G(\mathbb{A}_F)$ which transform under $Z_{G}(\mathbb{A}_F)$ by $\omega^{-1}$ and are square integrable on $G(F)Z_{G}(\mathbb{A}_F)\backslash G(\mathbb{A}_F).$ This operator can clearly be represented by the kernel function
	$$
	\K^{\varphi}(x,y)=\sum_{\gamma\in Z_G(F)\backslash G(F)}\varphi(x^{-1}\gamma y).
	$$
	We often omit the superscript $\varphi$ and simply write $\K(x,y)$ for $\K^{\varphi}(x,y).$
	\medskip
	
	It is well known that $L^2\left(G(F)\backslash G(\mathbb{A}_F),\omega^{-1}\right)$ decomposes into the direct sum of the space $L_0^2\left(G(F)\backslash G(\mathbb{A}_F),\omega^{-1}\right)$ of cusp forms and spaces $L_{\Eis}^2\left(G(F)\backslash G(\mathbb{A}_F),\omega^{-1}\right)$ and $L_{\Res}^2\left(G(F)\backslash G(\mathbb{A}_F),\omega^{-1}\right)$ defined using Eisenstein series and residues of Eisenstein series respectively. Then $\K$ splits up as: $\K=\K_0+\K_{\Eis}+\K_{\Res}.$ Selberg trace formula gives an expression for the trace of the operator $R(\varphi)$ restricted to the discrete spectrum, and is roughly of the form 
\begin{equation}\label{01}
	\int_{G(F)Z(\mathbb{A}_F)\backslash G(\mathbb{A}_F)}\K_0(x,x)dx=\Sigma_{\Geo}-\Sigma_{\Eis}-\Sigma_{\Res},
\end{equation}
where $\Sigma_{\Geo}, \Sigma_{\Eis}, \Sigma_{\Res}$ are contributions from geometric side, continuous spectrum and residual spectrum, respectively. Typically, the right hand side of \eqref{01} has some convergence issue; so a truncation is usually needed.
\medskip

	In \cite{Zag77}, Zagier investigated an analogue of \eqref{01} to prove holomorphy of $L$-function associated to symmetric square of classical cusp forms. Precisely, he considered
	\begin{equation}\label{1'}
	I_0^{\varphi}(s)=\int_{GL(2,\mathbb{Q})Z(\mathbb{A}_{\mathbb{Q}})\backslash GL(2,\mathbb{A}_{\mathbb{Q}})}\K_0^{\varphi}(x,x)E(x,s)dx,
	\end{equation}
	where $E(x,s)$ is an Eisenstein series.	Note that $\K_0(x,x)$ is rapidly decreasing and $E(x,s)$ is slowly increasing outside $s=1,$ thus the right hand side of \eqref{1'} is well defined as a meromorphic function, which has a simple pole at $s=1.$ Zagier obtained a spectral expansion: $I_0^{\varphi}(s)=\Sigma_{\Geo}(s)-\Sigma_{\Eis}(s)-\Sigma_{\Res}(s).$ Ignoring the convergence issue on the right hand side, one can take residue at $s=1,$ then it would in principle recover \eqref{01}, the Arthur-Selberg trace formula. This can be made rigorous by applying suitable regularization, see \cite{Wu19}.
	\medskip 
	
	Zagier's trace identity \eqref{1'} was further developed by Jacquet and Zagier \cite{JZ87} in terms of representation theoretical language to give a new proof of holomorphy of adjoint L-functions on $\GL(2,\mathbb{A}_F).$ They show (after continuation) the contribution from continuous and residual spectrum is a holomorphic multiple of Dedekind zeta function, and the contribution from elliptic regular conjugacy classes gives certain Artin $L$-series associated to finitely many quadratic extensions of $F.$ Hence the holomorphy of adjoint $L$-functions can be deduced from class field theory, or more generally, the (twisted) Dedekind conjecture (see Conjecture \ref{De} below). However, the treatment of $\Sigma_{\Eis}(s)$ in loc. cit. is not quite complete as there are infinitely many cuspidal data when $F\neq \mathbb{Q}.$

\subsection{Statement of the Main Results}
In this paper, we will investigate a distribution $I_0^{\varphi}(s;\tau),$ (see \eqref{X} in Sec. \ref{sec2.1} for precise definition), which is a generalization of $I_0^{\varphi}(s)$ defined in \eqref{1'} to $\GL(n)$ over a global field $F;$ and prove the absolute convergence when $\Re(s)>1$ and $0<\Re(s)<1.$ One of our main results (Theorem \ref{reg ell}, Theorem \ref{aa}, Theorem \ref{39'} and Theorem \ref{47'} ) may be summarized informally as follows:
\begin{thmx}\label{A.}
	Let notation be as before. Let $\Re(s)>1.$ Let $\varphi\in \mathcal{H}(G(\mathbb{A}_F),\omega).$ Then $I_0^{\varphi}(s;\tau)$ admits an expansion:
	\begin{equation}\label{m}
	I_0^{\varphi}(s;\tau)=I_{\Geo,\Reg}(s,\tau)+I_{\infty,\Reg}(s,\tau)+I_{\Sin}(s,\tau)+ \sum_{\chi}\int_{(i\mathbb{R})^{n-1}}I_{\chi}(s,\tau,\lambda)d\lambda,
	\end{equation}
	where all the sums on the right hand side of \eqref{m} converge absolutely; moreover, the first term $I_{\Geo,\Reg}(s,\tau)$ is a finite sum of $I_E(s,\tau)$ over certain direct sum of \'Etale extensions $E/F$ of degree $\leq n,$ and $I_E(s,\tau)$ is a multiple of $\Lambda(s,\tau\circ N_{E/F});$ the last sum is an infinite sum over cuspidal data $\chi$ associated to proper standard parabolic subgroups of $G,$ $I_{\chi}(s,\tau,\lambda)$ is a multiple of Rankin-Selberg period attached to $\chi;$ $I_{\Con}(s,\tau)$ is a multiple of 
	\begin{align*}
	\frac{\Lambda(s,\tau)\Lambda(2s,\tau^2)\cdots \Lambda((n-1)s,\tau^{n-1})\Lambda(ns,\tau^n)}{\Lambda(s+1,\tau)\Lambda(2s+1,\tau^2)\cdots \Lambda((n-1)s+1,\tau^{n-1})}.
	\end{align*}
Here $\Lambda(s,\cdot)$ refers to complete Hecke $L$-functions.	Furthermore, all  the terms on the right hand side of \eqref{m} admit holomorphic continuation to the whole $s$-plane if $\tau^k\neq 1$ for $1\leq k\leq n.$
\end{thmx}
\begin{remark}
	\begin{enumerate}
\item[(1).] The expansion \eqref{m} generalizes Jacquet and Zagier's formula for $\GL(2)$ (see \cite{JZ87}) to $\GL(n).$ A restricted version was obtained by Flicker \cite{Fli92} under some choice of test functions $\varphi$ so that only elliptic regular part of $I_{\Geo,\Reg}(s,\tau)$ shows up on the right hand side of \eqref{m}. New ideas of our proof are briefly summarized in Section \ref{1.3} below.
\item[(2).]	$I_{\Sin}(s,\tau)$ arises in the continuous spectrum, and it appears only when $n\geq 3.$  For certain applications, one can be eliminated it by choosing discrete and cuspidal test functions in the sense of \cite{FK88}. Such test functions will be used to deduce Theorem \ref{D} (see Section \ref{1.2.2} below), as an application of Theorem \ref{A.}. 
	
\item[(3).] For fixed $\lambda,$ each individual $I_{\chi}(s,\tau,\lambda)$ is a period of automorphic forms in the case of $(\GL(n)\times\GL(n),\GL(n))$ over the diagonal, in parallel to the $(\GL(n+1)\times\GL(n),\GL(n))$ studied in \cite{IY15}. 
\end{enumerate}
\end{remark}

\subsubsection{Basic Notation}\label{sec2.1}
Denote by $\mathcal{S}(\mathbb{A}_F^n)$ the space of Schwartz-Bruhat functions on the vector space $\mathbb{A}_F^n$ and by $\mathcal{S}_0(\mathbb{A}_F^n)$ the subspace spanned by products $\Phi=\prod_v\Phi_v$ whose components at real and complex places have the form
\begin{align*}
\Phi_v(x_v)=e^{-\pi \sum_{j=1}^nx_{v,j}^2}\cdot Q(x_{v,1},x_{v,2},\cdots,x_{v,n}),\ x_v=(x_{v,1},x_{v,2},\cdots,x_{v,n})\in F_v^n,
\end{align*}
where $F_v\simeq \mathbb{R},$ and $Q(x_{v,1},x_{v,2},\cdots,x_{v,n})\in \mathbb{C}[x_{v,1},x_{v,2},\cdots,x_{v,n}];$ and
\begin{align*}
\Phi_v(x_v)=e^{-2\pi \sum_{j=1}^nx_{v,j}\bar{x}_{v,j}}\cdot Q(x_{v,1},\bar{x}_{v,1},x_{v,2},\bar{x}_{v,2},\cdots,x_{v,n},\bar{x}_{v,n}),
\end{align*}
where $F_v\simeq \mathbb{C}$ and $Q(x_{v,1},\bar{x}_{v,1},x_{v,2},\bar{x}_{v,2},\cdots,x_{v,n},\bar{x}_{v,n})$ is a polynomial in the ring $\mathbb{C}[x_{v,1},\bar{x}_{v,1},x_{v,2},\bar{x}_{v,2},\cdots,x_{v,n},\bar{x}_{v,n}].$ 

Denote by $\Xi_F$ the set of unitary characters on $F^{\times}\backslash\mathbb{A}_F^{\times}$ which are trivial on $\mathbb{R}_+^{\times}.$ For any $\xi\in \Xi_F,$ denote by $\Lambda(s,\xi)$ the \textit{complete} Hecke $L$-function associated to $\xi.$ Let $\Phi\in\mathcal{S}_0(\mathbb{A}_F^n).$ Let $\tau\in\Xi_F$ be fixed. Let $\eta=(0,\cdots,0,1)\in F^n.$ Set
$$
f(x,\Phi,\tau;s)=\tau(\det x)|\det x|^s\int_{\mathbb{A}_F^{\times}}\Phi(\eta tx)\tau(t)^n|t|^{ns}d^{\times}t,
$$
which is a Tate integral (up to holomorphic factors) for the complete $L$-function  $\Lambda(ns,x.\Phi,\tau^{n})=L_{\infty}(ns,x_{\infty}.\Phi_{\infty},\tau_{\infty}^{n})\cdot L_{\fin}(ns,x_{\fin}.\Phi_{\fin},\tau_{\fin}^{n}).$ It converges absolutely uniformly in compact subsets of $\Re(s)>1/n.$ Since the mirabolic subgroup $P_0$ is the stabilizer of $\eta.$ Let $P=P_0Z_G$ be the full $(n-1,1)$ parabolic subgroup of $G,$ then $f(x,s)\in \Ind_{P(\mathbb{A}_F)}^{G(\mathbb{A}_F)}(\delta_P^{s-1/2}\tau^{-n}),$ where $\delta_P$ is the modulus character for the parabolic $P.$ Then we can define the Eisenstein series 
\begin{align*}
E_P(x,\Phi,\tau;s)=\sum_{\gamma\in P(F)\backslash G(F)}f(x,\Phi,\tau;s),
\end{align*}
which converges absolutely for $\Re(s)>1.$ Also, we define the integral:
\begin{equation}\label{X}
I_0^{\varphi}(s,\tau)=\int_{G(F)Z(\mathbb{A}_F)\backslash G(\mathbb{A}_F)}\K^{\varphi}_0(x,x)E_P(x,\Phi,\tau;s)dx.
\end{equation}
If there is no confusion in the context, we will alway write $I_0(s)$ (resp. $f_{\tau}(x,s)$ or $f(x,s)$) instead of $I_0^{\varphi}(s;\tau)$ (resp. $f(x,\Phi,\tau;s)$) for simplicity.
\medskip 

The distribution $I_0^{\varphi}(s,\tau)$ is interesting for several reasons: it involves more information than the Arthur-Selberg trace formula, e.g., one can take $\tau$ to be of order n and evaluate \eqref{X} at $s=1,$ then the spectral expansion of $I_0^{\varphi}(s,\tau)$ would in principle  provide a twisted trace formula for $G=\GL(n);$ the calculation has been carried out in \cite{Kaz83} when $n$ is a prime; and Theorem 2 of \cite{JZ87} reinterpreted the $\GL(2)$ case of the twisted trace identity as essentially equivalent to a theorem of Labesse and Langlands \cite{LL79}. 
\medskip 

On the other hand, the spectral expansion of $I_0^{\varphi}(s,\tau)$ is quite involved. When $n>2,$ the continuous spectrum has not been investigated before. Nevertheless, the expansion turns out to convey some interesting information connecting $L$-functions defined analytically and algebraically.  In fact, we shall compute the expansion and deduce from it that \textbf{holomorphy of certain adjoint $L$-functions} for $G=\GL(n)$ implies \textbf{Dedekind conjecture} for degree $n$ extensions (see Theorem \ref{D} on P. 5). This implication has been conjectured for a long time, e.g., see \cite{JZ87} and \cite{JR97}. 
\medskip 

Another consequence of studying $I_0^{\varphi}(s,\tau)$ is holomorphy of adjoint $L$-functions (and their twists) for all cuspidal representations on $\GL(n),$ $n\leq 4.$ This would be done in sequel \cite{Yan19}.

\subsubsection{Some Applications}\label{1.2.2}
The distribution $I_{\Geo,\Reg}(s,\tau)$ in \eqref{m} turns out to play a role in certain cases of \textit{beyond endoscopy}, see Altu{\u{g}}'s work \cite{Alt15a}, \cite{Alt15b} and \cite{Alt17}. In this section, we give other applications of \eqref{m} to some conjectures on holomorphy of $L$-functions and nonvanishing problem. First, we recall
\begin{conj}[$\tau$-twisted Dedekind Conjecture]\label{De}
	Let notation be as before. Let $E/F$ be an extension of global fields. Then $\Lambda_E(s,\tau\circ N_{E/F})/\Lambda_F(s,\tau)$ is holomorphic when $s\neq 1$, where $N_{E/F}$ is the relative norm.
\end{conj}
When $\tau$ is trivial, the above conjecture is conventionally called Dedekind conjecture, which is known when $E/F$ is Galois by the work of Aramata and Brauer (see Chap. 1 of \cite{Mar77}) or has a solvable Galois closure by the work of Uchida \cite{Uch75} and van der Waall \cite{Waa75}. Moreover, Dedekind conjecture is the prototype of Artin's holomorphy conjecture. The $\tau$-twisted version of Conjecture \ref{De} has been proved by Murty \cite{MR00} when $E/F$ is either Galois or has a solvable closure. However, the general case (even general degree 5 extensions) is not yet known. 
\medskip

When $n=2,$ \cite{JZ87} provides a connection between adjoint $L$-functions associated to $\pi\in\mathcal{A}_0(\GL(2,F)\backslash \GL(2,\mathbb{A}_F),\omega^{-1})$ and $\Lambda_E(s,\tau\circ N_{E/F})/\Lambda_F(s,\tau)$ when $E/F$ is quadratic. It was noted in \cite{JR97} that, at least for degree/rank $n$ up to 5, the two families seem to be related on a nuts-and-bolts level in the theory of integral representations, in addition to the relationships suggested by \cite{JZ87}. 

\medskip

Let $\mathcal{A}_0^{\dis}(G(F)\backslash G(\mathbb{A}_F),\omega^{-1})$ be the subspace generated by cuspidal representations $\pi\in \mathcal{A}_0(G(F)\backslash G(\mathbb{A}_F),\omega^{-1})$ such that $\pi$ has a supercuspidal component. Following \cite{JZ87}, Flicker \cite{Fli92} used a simple trace formula to deduce that Conjecture \ref{De} implies holomorphy of adjoint $L$-functions
\begin{align*}
L(s,\pi,\Ad\otimes\tau)=\frac{L(s,\pi\times\widetilde{\pi}\otimes\tau)}{\Lambda(s,\tau)}, \quad \pi \in \mathcal{A}_0^{\dis}(G(F)\backslash G(\mathbb{A}_F),\omega^{-1}),
\end{align*}
when $s\neq 1,$ modulo the key Lemma 4 in \cite{Fli92}. However, this lemma is not correct as pointed out by himself (ref. \cite{Fli93}, P. 202). Consequently, the asserted implication is not complete. In this section we will prove an implication in the opposite direction, obtaining
\begin{thmx}\label{D}
Let notation be as before. Assume the twist adjoint $L$-functions $L(s,\pi,\Ad\otimes\tau)$ are holomorphic at $s\neq 1$ for all $\pi\in \mathcal{A}_0^{\dis}(G(F)\backslash G(\mathbb{A}_F),\omega^{-1}).$ Then the $\tau$-twisted Dedekind conjecture holds for all fields extensions of $E/F$ of degree $n.$ 
\end{thmx}
\begin{remark}
\begin{enumerate}
\item[(1).] This relation provides a new perspective in the study of Dedekind conjecture, which is currently wide open when the degree is larger or equal to $5.$
\item[(2).] Suppose $\tau^k\neq 1,$ $1\leq k\leq n.$ We can conclude from Theorem \ref{A.}, Theorem \ref{D} and Theorem \ref{47'} (see Sec. \ref{3.}) that, if $I_{\Sin}(s,\tau)/\Lambda(s,\tau)$ admits a holomorphic continuation, then the twisted adjoint $L$-functions $L(s,\pi,\Ad\otimes\tau)$ are holomorphic at $s\neq 1$ for all $\pi\in \mathcal{A}_0(G(F)\backslash G(\mathbb{A}_F),\omega^{-1})$ if and only if the $\tau$-twisted Dedekind conjecture holds for all fields extensions of $E/F$ of degree $n.$ 
\end{enumerate}
\end{remark}

In Section \ref{sec9}, we will see the proof of Theorem \ref{D} would provide a result on the nonvanishing of $L(1/2,\pi\times\widetilde{\pi}):$

\begin{thmx}\label{N}
Let notation be as before. Let $n\geq 2.$ Suppose there exists an extension $E/F$ with degree $[E:F]=n,$ and $\zeta_E(1/2)\neq 0.$ Then there exists a $\pi=\pi(E)\in \mathcal{A}_0(G(F)\backslash G(\mathbb{A}_F),\omega^{-1}),$ such that $L(1/2,\pi\times\widetilde{\pi})\neq0.$
\end{thmx}

\begin{remark}

Fr\"{o}hlich \cite{Fro72} proved that there are infinitely many number fields $F$ such that $\zeta_F(1/2)=0.$ According to Langlands program, $L(s,\pi,\Ad)$ is holomorphic. In this case, if $\zeta_F(1/2)=0,$ then for \textbf{all} $\pi\in \mathcal{A}_0(G(F)\backslash G(\mathbb{A}_F),\omega^{-1}),$ $L(1/2,\pi\times\widetilde{\pi})=0.$ Hence, Theorem \ref{N} follows from Dedekind conjecture and Langlands program.
\end{remark}

\subsection{Idea of Proofs and Structure of the Paper}\label{1.3}
Starting with the spectral decomposition $\K_0(x,x)=\K(x,x)-(\K_{\Eis}(x,x)+\K_{\Res}(x,x)),$ we will further decompose these kernel functions by algebraic and analytic expansion:

Let $\mathfrak{S}$ be the union of $p^{-1}\gamma p$ modulo the center $Z_G(F),$ where $\gamma$ runs through $F$-points of standard parabolic subgroups of $G,$ and $p\in P_0(F).$ Then 
\begin{equation}\label{2}
\K(x,y)=\sum_{\gamma\in Z_G(F)\backslash G(F)-\mathfrak{S}}\varphi(x^{-1}\gamma y)+\sum_{\gamma\in \mathfrak{S}}\varphi(x^{-1}\gamma y).
\end{equation}
By Proposition \ref{reg} in Section \ref{6.1}, the set $Z_G(F)\backslash G(F)-\mathfrak{S}$ consists of $P_0(F)$-conjugacy classes, giving rise to regular $G(F)$-conjugacy classes.
\medskip 

On the other hand, by Proposition \ref{Fourier} (see Section \ref{4.1}), we have the Fourier expansion for $\K_{\Eis}(x,x)+\K_{\Res}(x,x):$
\begin{equation}\label{3}
\K_{\Eis}(x,x)+\K_{\Res}(x,x)=\int_{[N_P]}\K(ux,x)du+\sum_{k=2}^{n-1}\mathcal{F}_k\K(x,x)+W_{\K_{\Eis}}(x,x).
\end{equation}

Thus, combining \eqref{2} and \eqref{3} together we then obtain 
\begin{equation}\label{4.}
\K_0(x,x)=\K_{\Reg}(x)+\K_{\Con}(x)+\K_{\Sin}(x)+\K_{\Whi}(x),
\end{equation}
where $\K_{\Whi}(x)=-W_{\K_{\Eis}}(x,x),$ and 
\begin{align*}
\K_{\Reg}(x)&=\sum_{\gamma\in Z_G(F)\backslash G(F)-\mathfrak{S}}\varphi(x^{-1}\gamma x),\\ \K_{\Con}(x)&=-\int_{[N_P]}\sum_{\gamma\in Z_G(F)\backslash G(F)-\mathfrak{S}}\varphi(x^{-1}u^{-1}\gamma x)du,\\
\K_{\Sin}(x)&=\sum_{\gamma\in \mathfrak{S}}\varphi(x^{-1}\gamma y)-\int_{[N_P]}\sum_{\gamma\in \mathfrak{S}}\varphi(x^{-1}u^{-1}\gamma x)du-\sum_{k=2}^{n-1}\mathcal{F}_k\K(x,x).
\end{align*}

One then substitutes \eqref{4.} into \eqref{X} to obtain formally
\begin{equation}\label{5.}
I_0(s,\tau)=I_{\Reg}(s,\tau)+I_{\Con}(s,\tau)+I_{\Sin}(s,\tau)+I_{\Whi}(s,\tau),
\end{equation}
where $I_{\Whi}(s,\tau)$ is an infinite sum of general Rankin-Selberg periods involving Whittaker functions, and we also denote it by $I_{\infty}^{(1)}(s,\tau)$ from Section \ref{4.1}.
\medskip 

As will be seen in Section \ref{sec2}, stabilizers of elements in $Z_G(F)\backslash G(F)-\mathfrak{S}$ are direct sums of \'Etale algebras over $F$ of degree less or equal to $n.$ Hence the corresponding distribution $I_{\Reg}(s,\tau)$ would be a sum of certain Artin $L$-series associated to these \'Etale algebras. This has been treated in Theorem \ref{reg ell} in Section \ref{sec2.2}. 
\medskip

In Section \ref{sec3} we prove Fourier expansion of automorphic forms on $P_0(F)\backslash G(\mathbb{A}_F),$ which implies the decomposition \eqref{3}. 

\medskip 
In Section \ref{sec4.1}, we find explicitly representatives of $Z_G(F)\backslash G(F)-\mathfrak{S}$ as $P_0(F)$-conjugacy classes. Then, very roughly, we develop a geometric reduction (in $\GL(2)$ case, this is amounts to using Poisson summation, which is not available for $\GL(n),$ $n\geq 3$), to relate $I_{\Con}(s,\tau)$ to certain intertwining operators. Hence the convergence and analytic properties follow from theory of intertwining operators. The results are summarized in Theorem \ref{aa} in Section \ref{6.2}.

\medskip 

Then the rest of this paper is devoted to the distribution $I_{\Whi}(s,\tau).$ We follow Arthur's approach with modified truncation operators to work with the trace identity under consideration, and show the absolute convergence over any Siegel domain. However, unfolding the Eisenstein series one then sees the fundamental domain is much larger than any Siegel domain. We then prove Proposition \ref{27''} to resolve this technical problem and show that $I_{\Whi}(s,\tau),$ when $\Re(s)>1,$ is an absolute convergent infinite sum of Mellin transforms of certain Rankin-Selberg convolution for \textit{non-discrete} representations. Concrete statements are given in Theorem \ref{39'} in Section \ref{6sec}.

\medskip 

In Section \ref{6.2.}, we prove some properties of Rankin-Selberg periods for non-discrete representations. These results will be used in Section \ref{3.} to show absolute convergence of $I_{\Whi}(s,\tau)$ in the strip $0<\Re(s)<1,$ and thus get a holomorphic function therein, see Theorem \ref{47'} for details. So $I_{\Whi}(s,\tau)$ is holomorphic when $0<\Re(s)<1$ and $\Re(s)>1.$ However, for $\tau$ such that $\tau^k=1$ for some $1\leq k\leq n,$ the function $I_{\Whi}(s,\tau)$ has singularities on the whole boundary $\Re(s)=1.$ So we need to find a meromorphic continuation for $I_{\Whi}(s,\tau).$ This is investigated in Section \ref{7.2}, where we obtain continuation of each individual summand of $I_{\Whi}(s,\tau)$ to some zero-free region of Rankin-Selberg $L$-functions, proving Theorem \ref{57}, which will be of independent interest, e.g., it will be used in \cite{Yan19}. Further continuation to some open region containing $\Re(s)\geq 1/2$ are obtained for $\GL(n),$ $n\leq 4,$ in the Appendix \ref{app}.
\medskip

In Section \ref{sec9}, we gather Theorem \ref{reg ell}, Theorem \ref{aa}, Theorem \ref{39'} and Theorem \ref{47'} to deduce Theorem \ref{A.}. Furthermore, applying some special test functions $\varphi$ into Theorem \ref{A.} and dealing with some generalized Tate integral, we then prove Theorem \ref{D} and Theorem \ref{N}.

\begin{remark}
When $n=2,$ $I_{\Sin}(s,\tau)$ has no Fourier part, and has been dealt with in \cite{JZ87}. Foe general $n\geq 3,$  combining Theorem \ref{reg ell}, Theorem \ref{aa}, Theorem \ref{39'}, Theorem \ref{47'} and functional equation of Eisenstein series, we then conclude $I_{\Sin}(s,\tau)$ is uniformly convergent when $\Re(s)>1;$ and it admits a meromorphic continuation to the whole $s$-plane if $\tau^k\neq 1$ for $1\leq k\leq n.$ Nevertheless, the distribution $I_{\Sin}(s,\tau)$ is rather involved. We will handle it for general $\tau$ and $G=\GL(n),$ $n\leq 4,$ in the sequel \cite{Yan19} by developing different methods from this paper.

\end{remark}

\textbf{Acknowledgements}
I am very grateful to my advisor Dinakar Ramakrishnan for instructive discussions and helpful comments. I would like to thank Ashay Burungale, Li Cai, Herv\'e Jacquet, Dihua Jiang, Simon Marshall, Kimball Martin, Philippe Michel, Chen Wan, Song Wang and Xinwen Zhu for their precise comments and useful suggestions. Part of this paper was revised during my visit to the Morningside Center of Mathematics in China and \'{E}cole polytechnique f\'ed\'erale de Lausanne in Switzerland and I would like to thank their hospitality.
	
	\section{Contributions from Geometric Sides}\label{sec2}
	Let $\mathcal{H}\left(G(\mathbb{A}_F)\right)$ be the Hecke algebra of $\mathcal{H}\left(G(\mathbb{A}_F)\right)$ and $\varphi\in \mathcal{H}\left(G(\mathbb{A}_F)\right).$ For any character $\omega$ of $\mathbb{A}_F^{\times}/F^{\times}.$ Let $\varphi\in \mathcal{C}_c^{\infty}\left(Z_G(\mathbb{A}_F)\setminus G(\mathbb{A}_F)\right)\cap\mathcal{H}\left(G(\mathbb{A}_F)\right)$ be of central character $\omega.$ Denote by $V_0$ the Hilbert space 
	$$
	L_0^2\left(G(F)\setminus G(\mathbb{A}_F),\omega^{-1}\right)=\bigoplus_{\pi}V_{\pi},
	$$
	where $\pi\in \mathcal{A}_0\left(G(F)\setminus G(\mathbb{A}_F),\omega^{-1}\right),$ the set of irreducible cuspidal representation of $G(\mathbb{A}_F)$ with central character $\omega$ and $V_{\pi}$ is the corresponding isotypical component. By multiplicity one, the representation of $G(\mathbb{A}_F)$ on $V_{\pi}$ is equivalent to $\pi.$ For each $\pi,$ we choose an orthonormal basis $\mathcal{B}_{\pi}$ of $V_{\pi}$ consisting of $K$-finite vectors.  Let $\K_0(x,y)$ be the kernel function for the right regular representation $R(\varphi)$ on $V_0.$ Then we have the decomposition
	\begin{equation}\label{ker_0}
	\K_0(x,y)=\sum_{\pi}\K_{\pi}(x,y),\  \text{where}\ \K_{\pi}(x,y)=\sum_{\phi\in\mathcal{B}_{\pi}}\pi(\varphi)\phi(x)\overline{\phi(y)}.
	\end{equation}
	All the functions in the summands are of rapid decay in $x$ and $y.$ The sum of $\K_{\pi}(x,y)$ converges in the space of rapidly decaying functions, by the usual estimates on the growth of cusp forms. The sum over $\mathcal{B}_{\pi}$ is finitely uniformly in $x$ and $y$ for a given $\varphi$ because of the $K$-finiteness of $\varphi.$

	\medskip

\subsection{Structure of $G(F)$-Conjugacy Classes}\label{6.1}
Let $B$ be the subgroup of upper triangular matrices of $G,$ and $T$ the Levi component of $B.$ Let $W=W_n$ be Weyl group of $G$ with respect to $(B,T).$ Then one can take $W$ to be the subgroup consisting of all $n\times n$ matrices which have exactly $1$ in each row and each column, and zeros elsewhere. Let $\Delta=\{\alpha_{1,2},\alpha_{2,3},\cdots,\alpha_{n-1,n}\}$ be the set of simple roots, and for each simple root $\alpha_{k,k+1},$ $1\leq k\leq n-1,$ denote by $w_k$ the corresponding reflection. Explicitly, for each $1\leq k\leq n-1,$
\begin{align*}
w_k=\begin{pmatrix}
I_{k-1} &\\
& S&\\
&&I_{n-k-1}
\end{pmatrix},\ \text{where $S=\begin{pmatrix}
	&1\\
	1&
	\end{pmatrix}$}.
\end{align*}
For each $1\leq k\leq n-1,$ let $I_k=\Delta\setminus\{\alpha_{k,k+1}\}$ and $W_{I_{k}}$ be the subgroup generated by elements in $I_k.$ Write $Q_{k}=BW_{I_{k}}B.$ Then $Q_{k}$ is a standard maximal parabolic subgroup of $G$ corresponding to the simple root $\alpha_{k,k+1}.$ And every maximal parabolic subgroup is conjugate to some $Q_k,$ $1\leq k\leq n-1.$ Clearly, under this notation, one has $P=Q_{n-1}.$ Denote by
$$
Q_k(F)^{P(F)}=\{pqp^{-1}:\ p\in P(F),\ q\in Q_k(F)\},\ 1\leq k\leq n-1. 
$$
\begin{prop}\label{irre}
	Let $\mathcal{C}$ be an irregular $G(F)$-conjugacy class, then one has
	\begin{equation}\label{9}
	\mathcal{C}=\bigcup _{k=1}^{n-1}\mathcal{C}\cap Q_k(F)^{P(F)}.
	\end{equation} 	
\end{prop}
To prove \eqref{9}, we need rational canonical forms of $g\in G(F),$ which is an analogue of Jordan canonical forms of matrices over $\mathbb{C}$ (of course $F$ is not algebraically closed). The decomposition is given below: 
\begin{lemma}\label{Jordan}
	Let $V$ be a $n$-dimensional vector space over $F,$ and $\mathscr{A}\in \End(V).$ Then there exist subspaces $V_l\subseteq V$, $1\leq l\leq r,$ such that 
	\begin{equation}\label{10.5}
	V=V_1\oplus V_2\oplus\cdots\oplus V_r,
	\end{equation}
	and for each $i,$ both of the minimal polynomial and characteristic polynomial of $\mathscr{A}_{V_l}=\mathscr{A}\mid_{V_l}$ are of the form $\wp(\lambda)^k,$ where $k\in \mathbb{N}_{\geq 1}$ and $\wp(\lambda)\in F[\lambda]$ is a irreducible polynomial over $F.$ Furthermore, for each $l,$ there exists a basis $\alpha_l=\{\alpha_{l_1},\cdots,\alpha_{l_m}\}$ of $V_l$ such that under $\alpha_l,$ $\mathscr{A}_{V_l}$ has the following quasi-rational canonical form
	\begin{equation}\label{10}
	\mathcal{J}\left(\wp(\lambda)^k\right):=\begin{pmatrix}
	C(\wp) &&&\\
	N& C(\wp)&&\\
	&\ddots&\ddots&\\
	&&N&C(\wp)
	\end{pmatrix},
	\end{equation}
	where $C(\wp)$ is the companion matrix of $\wp(\lambda)$ and $N=\begin{pmatrix}
	&&&1\\
	&&0&\\
	&\ldots&&\\
	0&&&
	\end{pmatrix}.$
\end{lemma}
\begin{proof}
	Let $m(\lambda)$ (resp. $f(\lambda)$) be the minimal polynomial (resp. characteristic polynomial) of $\mathscr{A}.$ Consider their primary decompositions over $F:$
	\begin{align*}
	m(\lambda)=\prod_{i}\wp_i(\lambda)^{e_i'}\ \text{and}\ f(\lambda)=\prod_{i}\wp_i(\lambda)^{e_i},
	\end{align*} 
	where $\wp(\lambda)_i's$ are distinct irreducible monic polynomials over $F,$ $0\leq e_i'\leq e_i,$ $\forall$ $i.$ Take $U_i=\ker \wp_i(\mathscr{A})^{e_i'}.$ Then $U_i$ is $\mathscr{A}$-invariant. By cyclic decomposition theorem (which holds for general fields), we have 
	$$
	U_i=F[\mathscr{A}]\alpha^{i,1}\oplus F[\mathscr{A}]\alpha^{i,2}\oplus \cdots\oplus F[\mathscr{A}]\alpha^{i,r_i},
	$$
	where each $F[\mathscr{A}]\alpha^{i,j}$ is a cyclic subspace of $U_i.$ Then one has the decomposition \eqref{10.5} and both of the minimal polynomial and characteristic polynomial of $\mathscr{A}_{V_{i_j}}=\mathscr{A}\mid_{F[\mathscr{A}]\alpha^{i,j}}$ are powers of $\wp_i(\lambda).$
	
	For any $i$ and $1\leq j\leq r_i,$ we may assume that the minimal polynomial of $\mathscr{A}_{V_{i_j}}$ on $F[\mathscr{A}]\alpha^{i,j}$ is $\wp_{i}(\lambda)^{e_{i,j}'}$ with some $0\leq e_{i,j}'\leq e_{i}'.$ Write $\wp_{i}(\lambda)=\lambda^{d_{i}}-c_{d_{i}-1}\lambda^{d_{i}-1}-\cdots-c_0.$ Define
	$$
	\alpha_{sd_{i}+t}=\mathscr{A}_{V_{i_j}}^{t-1}\wp_i(\mathscr{A}_{V_{i_j}})^{s}\alpha^{i,j},\ \text{$1\leq s\leq e_{i,j}',$ $1\leq t\leq d_{i}$}.
	$$
	Note that for any $1\leq s\leq e_{i,j}',$
	\begin{align*}
	\mathscr{A}_{V_{i_j}}\alpha_{sd_{i}}&=\mathscr{A}_{V_{i_j}}^{d_{i}}\wp_i(\mathscr{A}_{V_{i_j}})^{s-1}\alpha^{i,j}\\
	&=\left(\mathscr{A}_{V_{i_j}}^{d_{i}}-\wp_i(\mathscr{A}_{V_{i_j}})^{s-1}\right)\wp_i(\mathscr{A}_{V_{i_j}})^{s-1}\alpha^{i,j}+\wp_i(\mathscr{A}_{V_{i_j}})^{s}\alpha^{i,j}\\
	&=c_0\alpha_{(s-1)d_i+1}+c_1\alpha_{(s-1)d_i+2}+\cdots+c_{d_i-1}\alpha_{sd_i}+\alpha_{sd_i+1}.
	\end{align*}
	Therefore, under the basis $\{\alpha_{sd_{i}+t}:\ 1\leq s\leq e_{i,j}',\ 1\leq t\leq d_{i}\},$ $\mathscr{A}_{V_{i_j}}$ is represented by $\mathcal{J}\left(\wp(\lambda)^{e_{i,j}'}\right)$ defined in \eqref{10}.
\end{proof}

Note that we have a bijection $W_{I_{n-1}}\backslash W/W_{I_{n-1}}\longleftrightarrow \{1,w_{n-1}\}.$
By Bruhat decomposition $G(F)$ is equal to
\begin{equation}\label{a}
\coprod_{w\in W_{I_{n-1}}\backslash W/W_{I_{n-1}}}P(F)wP(F)=P(F)\coprod P(F)\times \{w_{n-1}\}\times N_P^{w_n}(F),
\end{equation}
where $N_P^{w_n}(F)=\left(w_{n-1}N_P(F)w_{n-1}\cap N_P(F)\right)\backslash N_P(F).$ Combining \eqref{a} with the above lemma leads to the proof of \eqref{9}:

\begin{proof}[Proof of Proposition \ref{irre}]
	Let $g\in \mathcal{C}$ be an representative. Set $m(\lambda)$ (resp. $f(\lambda)$) to be its minimal polynomial (resp. characteristic polynomial) over $F$. Consider their primary decompositions over $F:$
	\begin{align*}
	m(\lambda)=\prod_{i\in I}\wp_i(\lambda)^{e_i'}\ \text{and}\ f(\lambda)=\prod_{i\in I}\wp_i(\lambda)^{e_i},
	\end{align*} 
	where $\wp(\lambda)_i's$ are distinct irreducible monic polynomials over $F,$ $I$ is a finite index set such that $e_i>0,$ $\forall$ $i\in I.$ Write $d_i=\deg \wp_i(\lambda),$ $\forall$ $i\in I.$ We may assume that $d_1\leq d_2\leq \cdots\leq d_{\#I}$. Also, write $d_0=0$. Since the conjugacy class $\mathcal{C}$ is irregular, $m(\lambda)$ is a proper factor of $f(\lambda).$ Thus we have the following cases:
	\begin{description}
		\item[Case I] If $\#I=1$ and $d_1>1$, then $m(\lambda)=\wp(\lambda)^{e'},$ $f(\lambda)=\wp(\lambda)^{e},$ and $0<e'<e=d_1^{-1}n.$ Hence by Lemma \ref{Jordan}, $g$ is $G(F)$-conjugate to some element $\tilde{g}$ in 
		$$
		\begin{pmatrix}
		GL_{n-kd_1}(F)&\\
		&GL_{kd_1}(F)
		\end{pmatrix}\ \text{for some $1\leq k\leq e-1$}.
		$$
		For any $h\in G(F),$ if $h\in P(F),$ then clearly $h\tilde{g}h^{-1}\in pQ_{n-kd_1}(F)p^{-1};$ if $h\in G(F)-P(F),$ it can be uniquely written as $h=pw_{n-1}n,$ where $p\in P(F)$ and $n\in N_P(F).$ Since $n\tilde{g}n^{-1}\in Q_{n-kd_1}(F),$ one has 
		\begin{align*}
		h\tilde{g}h^{-1}\in pw_{n-1}Q_{n-kd_1}(F)w_{n-1}p^{-1}\subseteq pQ_{n-kd_1}(F)p^{-1}.
		\end{align*}
		Hence, in this case one always has that 
		$$
		\mathcal{C}=\{h\tilde{g}h^{-1}:\ h\in G(F)\}\subseteq \bigcup _{k=1}^{n-1} Q_k(F)^{P(F)}.
		$$
		
		\item[Case II] If $\#I>1,$ and $d_i>1,$ $d_{i-1}=1$ for some $1\leq i\leq \#I,$ then similarly by Lemma \ref{Jordan} there exists some $\tilde{g}\in Q_{(n-kd_i-f_{i-1},f_{i-1},kd_i)}(F)$ for some $0<k<d_1^{-1}(n-f_{i-1}),$ where $f_{i-1}=\sum_{j=1}^{d_{i-1}}e_j$ and $Q_{(n-kd_i-f_{i-1},f_{i-1},kd_i)}(F)$ is a standard parabolic subgroup of $G(F)$ of type $(n-kd_i-f_{i-1},f_{i-1},kd_i).$ Then if wither $f_{i-1}>0$ or $n-kd_i-f_{i-1}\neq0,$ then $n-kd_i>0$ and $Q_{(n-kd_i-f_{i-1},f_{i-1},kd_i)}(F)\subseteq Q_{n-kd_i}(F).$ Otherwise, $i=1,$ and one can pick up a representative $\tilde{g}\in Q_{n-e_1'd_1}(F).$ Noting that $n-e_1'd_1>0$ since $\#I>1,$ $Q_{n-e_1'd_1}(F)$ is then a well defined standard maximal parabolic subgroup of $G(F).$ Then as in Case I, \eqref{9} holds.
		
		\item[Case III] If $d_{\#I}=1,$ then there exists some $\lambda_1,\cdots,\lambda_{\#I}\in F,$ integers $0<e_i'<e_i,$ $1\leq i\leq \#I$ (since $\mathcal{C}$ is irregular), such that $m(\lambda)=\prod(\lambda-\lambda_i)^{e_i'}$ and $f(\lambda)=\prod(\lambda-\lambda_i)^{e_i}.$ Then by the usual Jordan canonical decomposition one can find some $\tilde{g}\in B(F)$ such that $\tilde{g}$ represents the conjugacy class $\mathcal{C}.$ Then a similar trick as in Case I implies \eqref{9}.
	\end{description}
	Proposition \ref{irre} thus follows.
\end{proof}

Now we consider regular $G(F)$-conjugacy classes in $G(F).$ The structure of such conjugacy classes are described in Proposition \ref{reg}. To prove $\eqref{9'}$ we need some preparation. Given arbitrarily an $k\in \mathbb{N}_{\geq1},$ denote by $H(F)=H_{k}(F)=GL_k(F).$ Let $\gamma\in H(F)$ be regular and denote by $f(\lambda)=\wp_1(\lambda)^{e_1}\cdots \wp_{m_k}(\lambda)^{e_{m_k}}$ its characteristic polynomial, where $e_i\geq 1,$ $\wp_i$ is monic and irreducible over $F,$ $1\leq i\leq {m_k}.$ Let $d_i=\deg \wp_i.$ Set $F(\gamma)=F[\lambda]/\left(f(\lambda)\right)$ be the polynomial algebra generated by $\gamma,$ and denote by $F(\gamma)^{\times}$ the set of invertible elements in $F(\gamma).$ Let $P_{0}^H(F)$ be the mirabolic subgroup of $H(F).$ Also, for any $\delta\in H(F),$ we always write $H_{\delta}(F)$ for the centralizer of $\delta$ in $H(F).$ We will always use these notation henceforth.
\begin{lemma}\label{A}
	Let $\gamma\in H(F)$ be regular elliptic, then for any $(a_1,a_2,\cdots,a_k)\in F^k,$ there exists a unique element $x\in F(\gamma)$ such that the last row of $x$ is exactly $(a_1,a_2,\cdots,a_k).$
\end{lemma}	
\begin{proof}
	Since $\gamma$ is regular, $H_{\gamma}(F)=F(\gamma),$ and $\dim F(\gamma)=k.$ Let $\eta=(0,\cdots,0,1)\in F^k.$ Consider the linear map:
	$$
	\tau:\ F(\gamma)\rightarrow F^k,\qquad x\mapsto\tau(x)=\eta x.
	$$
	Since $\gamma$ is elliptic, $F(\gamma)$ is a field, so any nonzero element is invertible. Consequently, the map $\tau$ is injective, and hence surjective. Thus $\tau$ is an isomorphism of $k$-dimensional $F$-vector spaces. Then the lemma follows.
\end{proof}
\begin{remark}
	Let $\gamma\in H(F)$ be regular elliptic, we have $H(F)=P_0^H(F)F(\gamma)^{\times}.$ In fact, since $\tau$ is a bijection,  given $g\in H(F),$ there exists $h\in H(F)$ such that $\eta g=\eta h$ which implies that $gh^{-1}\in P_0^H(F),$ the isotropy subgroup of $\eta,$ i.e., $g\in P_0^H(F)F(\gamma)^{\times}.$
\end{remark}

\begin{lemma}\label{A'}
	Let $\gamma\in H(F)$ be regular. Assume further that the characteristic polynomial of $\gamma$ has only one irreducible factor. Then for any $(a_1,a_2,\cdots,a_k)\in F^k,$ there exists a unique element $x\in F(\gamma)$ such that the last row of $x$ is exactly $(a_1,a_2,\cdots,a_k).$
\end{lemma}	
\begin{proof}
	Let $f(\lambda)=\wp(\lambda)^e$ be the characteristic polynomial of $\gamma,$ where $\wp(\lambda)=\lambda^d+c_{d-1}\lambda^{d-1}+\cdots+c_1\lambda+c_0\in F[\lambda]$ is irreducible. Then $de=k$. By definition, $F[\gamma]=F[\lambda]/\left(\wp(\lambda)^e\right).$ Consider the filtration 
	$$
	\wp(\lambda)^{i-1}F[\lambda]/\left(\wp(\lambda)^i\right)\supseteq \wp(\lambda)^{i}F[\lambda]/\left(\wp(\lambda)^{i+1}\right),\ 1\leq i\leq e-1.
	$$
	Pick the basis for $F[\gamma]$ over $F$ as in the proof of Lemma \ref{A}, i.e., $\{\lambda^i\wp(\lambda)^j:\ 0\leq i\leq d,\ 0\leq j\leq e-1\}.$ With respect to this basis, each element of $F[\gamma]$ has the following type
	\begin{equation}\label{B'}
	\mathcal{S}_{\gamma}=\Bigg\{A=\begin{pmatrix}
	A_0 &&&\\
	A_1& A_0&&\\
	A_2&A_1&A_0&\\
	\vdots&\ddots&\ddots&\ddots\\
	A_{e-1}&\ldots &A_2&A_1&A_0
	\end{pmatrix},\ A_i\in {M}_{d\times d}(F),\ 0\leq i\leq e-1\Bigg\},
	\end{equation}
	and under this basis, and the assumption that $\gamma$ is regular, $\gamma$ has the quasi-rational canonical form
	\begin{equation}\label{0012}
	\mathcal{J}=\begin{pmatrix}
	C &&&\\
	N& C&&\\
	&\ddots&\ddots&\\
	&&N&C
	\end{pmatrix}\in GL_{k}(F),
	\end{equation}
	i.e., $\gamma$ is conjugate to $\mathcal{J},$ where $C=C(\wp)$ be the companion matrix of $\wp(\lambda)$, i.e., 
	$$
	C=\begin{pmatrix}
	0&& &-c_0\\
	1&0&&\\
	&\ddots&\ddots&\\
	&&1&-c_{d-1}
	\end{pmatrix},\ \text{and}\ N=\begin{pmatrix}
	&& 1\\
	&\ldots&\\
	0&&
	\end{pmatrix}\in GL_{d}(F).
	$$ 
	Since elements in the same $G(F)$-conjugacy class have the same characteristic polynomial, we may assume that $\gamma=\mathcal{J}$ is a quasi-rational canonical form. 
	
	Then necessarily if $A\in F[\gamma]$ of the form in $\eqref{B'},$ then it commutes with $\gamma.$ Indeed, since $\gamma$ is regular, i.e., the minimal polynomial of $\gamma$ coincides with its characteristic polynomial, any nonsingular matrix commuting with $\gamma$ must lie in $F[\gamma]^{\times}.$ Thus $F[\gamma]^{\times}=\{A\in \mathcal{S}_{\gamma}\cap GL_{k}(F):\ A\gamma=\gamma A\}.$ 
	
	Now we consider the equation $A\gamma=\gamma A,$ $A\in \mathcal{S}_{\gamma}.$ Clearly this is equivalent to a system of Sylvester equations
	\begin{equation}\label{233}
	\begin{cases}
	CA_0=A_0C\\
	NA_0+CA_1=A_1C+A_0N\\
	\qquad\vdots\\
	NA_{e-2}+CA_{e-1}=A_{e-1}C+A_{e-2}N.
	\end{cases}
	\end{equation}
	Since $A_0\in F[C]^{\times},$ and $C$ is regular elliptic, $A_0$ commuting with $C$ implies that there exists some $h_0(\lambda)\in F[\lambda],$ such that $A_0=h_0(C).$ We may assume that $d_0=\deg h_0\leq d-1.$ Let $\eta=(0,\cdots,0,1)$ and write $\eta C^{i}=(b_{d,1}^{(i)},b_{d,2}^{(i)},\cdots,b_{d,d}^{(i)}),$ $1\leq i\leq d-1,$ for the last row of $C^{i}.$ Define 
	$$
	X_{(i)}=\begin{pmatrix}
	0&b^{(i)}_{d,1}&b^{(i)}_{d,2}&\cdots &b^{(i)}_{d,d-1}\\
	0&0&b^{(i)}_{d,1}&\ddots&\vdots\\
	\vdots&\ddots&\ddots&\ddots&b^{(i)}_{d,2}\\
	\vdots&&\ddots&\ddots&b^{(i)}_{d,1}\\
	0&\ldots&\ldots&0&0
	\end{pmatrix}\in GL_{d}(F).
	$$
	\begin{claim}\label{13}
		Let notation be as before, then for any $1\leq i\leq d-1,$ $X=X_{(i)}$ is a solution to the Sylvester equation
		$$
		NC^i+CX=XC+NC^i.
		$$
	\end{claim}
	Write $A_0=h_0(C)=c'_{d_0}C^{d_0}+c'_{d_0-1}C^{d_0-1}+\cdots+c'_1C+c'_0I_{d},$ $c'_{d_0}\neq 0$. Define 
	$$
	A_{h_0}^{sp}=c'_{d_0}X_{(d_0)}+c'_{d_0-1}X_{(d_0-1)}+\cdots+c'_1X_{(1)}.
	$$
	Clearly $A_1=A_{h_0}^{sp}$ gives a special solution of the equation $NA_0+CA_1=A_1C+A_0N$ (the superscript 'sp' refers to 'special'). Given $A_0=h_0(C)$ as above, one then claims that
	$$
	\mathcal{A}_1=\{A_{h_0}^{sp}+h_1(C):\ h_1\in F[\lambda],\ \deg h_1\leq d-1\}
	$$
	gives all solutions to the equation $NA_0+CA_1=A_1C+A_0N.$ In fact, on the one hand, elements in $\mathcal{A}_1$ obviously satisfies the equation; on the other hand, let $A_1'$ be any solution to the equation, then $A_{h_0}^{sp}-A_1'$ commutes with $C,$ thus it is a polynomial of $C,$ namely, $A_1'\in \mathcal{A}_1.$ This proves the claim.
	
	Note that $NA_{h_0}^{sp}=A_{h_0}^{sp}N=0,$ when substitute $A_1=A_{h_0}^{sp}+h_1(C)$ into the equation $NA_1+CA_2=A_2C+A_1N,$ to get $Nh_1(C)+CA_2=A_2C+h_1(C)N.$ Write $h_1(\lambda)=c''_{d_1}\lambda^{d_1}+c''_{d_1-1}\lambda^{d_1-1}+\cdots+c''_{1}\lambda+c''_{0},$ and set 
	$$
	A_{h_1}^{sp}=c''_{d_1}X_{(d_1)}+c''_{d_1-1}X_{(d_1-1)}+\cdots+c''_1X_{(1)}.
	$$
	Then $\mathcal{A}_2=\{A_{h_1}^{sp}+h_2(C):\ h_2\in F[\lambda],\ \deg h_2\leq d-1\}$ gives all solutions to the equation $NA_1+CA_2=A_2C+A_1N.$ Generally we define $\mathcal{A}_i,$ $1\leq i\leq e-1$ similarly, and set $\mathcal{A}_0=\{h_0(C):\ h_0\in F[\lambda],\ \deg h_0\leq d-1\}.$ These $\mathcal{A}_i$'s describe the structure of $F[\gamma]^{\times}.$
	
	Therefore, given any $\mathfrak{a}=(a_1,a_2,\cdots,a_k)\in F^k,$ by Lemma \ref{A} one can find uniquely an $A_0\in F[C]$ such that $\eta A_0=(a_{k-d+1},a_{k-d+2},\cdots,a_k).$ Denote the sections of $\mathfrak{a}$ by $\mathfrak{a}_i=(a_{(i-1)d+1},a_{(i-1)d+2},\cdots, a_{id}),$ $1\leq i\leq e-1.$ Let $1\leq i_0\leq e-1,$ assume that for any $0\leq i<i_0$ one can find uniquely an element $A_{i}\in {M}_{d\times d}(F)$ such that the last row of $A_i$ is exactly $\mathfrak{a}_{e-i},$ then let $h_{i_0}(C)\in F[C]^{\times}$ be the unique element whose last row is $\mathfrak{a}_{e-i_0},$ take $A_{i_0}=A_{h_{i_0}}^{sp}+h_{i_0}(C).$ Then $\eta A_{i_0}=\eta h_{i_0}(C)=\mathfrak{a}_{e-i_0}.$ Moreover, such an $A_{i_0}$ is actually unique. Let $A_{i_0}'$ be another matrix satisfying that $\eta A_{i_0}'=\mathfrak{a}_{e-i_0}.$ Since $A_{i_0}'$ is a solution of $NA_{i_0-1}+CX=XC+A_{i_0-1}N,$ $A_{i_0}-A_{i_0}'$ commutes with $C.$ Thus $A_{i_0}-A_{i_0}'\in F[C].$ Note that the last row of $A_{i_0}-A_{i_0}'$ is $0,$ so by the uniqueness from Lemma \ref{A}, $A_{i_0}-A_{i_0}'=0.$ This shows the uniqueness of $A_{i_0}.$
	
	Therefore, the proof ends with an induction on $i_0$ and the following proof of Claim \ref{13}.
\end{proof}

\begin{proof}[Proof of Claim \ref{13}]
	We give a proof based on induction, although one might verify the claim by brute force computation (which is pretty complicated). Note that the case $i=1$ is trivial, since $X_{(1)}=N.$ Now we assume that there exists an $i_0$ such that $1<i_0\leq d-1,$ and for any $1\leq i< i_0,$ $X=X_{(i)}$ is a solution to the Sylvester equation $NC^i+CX=XC+NC^i.$ Write $C^j=\left(b_{s,t}^{(j)}\right)_{1\leq s,t\leq d},$ $1\leq j\leq d-1,$ then a straightforward expansion implies that $X=X_{(i_0)}$ is a solution to the Sylvester equation $NC^{i_0}+CX=XC+NC^{i_0}$ if and only if the following system of linear equations holds
	\begin{equation}\label{2333}
	\begin{cases}
	b_{d,d}^{(i_0)}=-c_1b_{d,1}^{(i_0)}-c_2b_{d,2}^{(i_0)}\cdots-c_{d-1}b_{d,d-1}^{(i_0)},\\
	b_{d,d-1}^{(i_0)}=-c_2b_{d,1}^{(i_0)}-c_3b_{d,2}^{(i_0)}-\cdots-c_{d-1}b_{d,d-2}^{(i_0)}+b_{2,1}^{(i_0)},\\
	b_{d,d-2}^{(i_0)}=-c_3b_{d,1}^{(i_0)}-c_4b_{d,2}^{(i_0)}-\cdots-c_{d-1}b_{d,d-3}^{(i_0)}+b_{3,1}^{(i_0)},\\
	\qquad\qquad\vdots\\
	b_{d,2}^{(i_0)}=-c_{d-1}b_{d,1}^{(i_0)}+b_{d-1,1}^{(i_0)}.
	\end{cases}
	\end{equation}
	Comparing entries on both sides of $C^{i_0}=C^{i_0-1}C$ leads to the recurrence relations
	\begin{equation}\label{23333}
	\begin{cases}
	b_{d,j}^{(i_0)}=b_{d,j+1}^{(i_0-1)},\ 1\leq j\leq d-1,\\
	b_{d,d}^{(i_0)}=-c_{0}b_{d,1}^{(i_0-1)}-c_{1}b_{d,2}^{(i_0-1)}-\cdots-c_{d-1}b_{d,d}^{(i_0-1)}.
	\end{cases}
	\end{equation}
	Since $i_0\leq d-1,$ $i_0-1<d-1,$ then $b_{d,1}^{(i_0-1)}=0.$ Therefore, relations \eqref{23333} implies that 
	$$
	b_{d,d}^{(i_0)}=-c_{0}b_{d,1}^{(i_0-1)}-c_1b_{d,1}^{(i_0)}-c_2b_{d,2}^{(i_0)}\cdots-c_{d-1}b_{d,d-1}^{(i_0)},
	$$
	which is exactly the first equation in \eqref{2333}. By other assumption, the system of relations \eqref{2333} holds when $i_0$ replaced by $i_0-1.$ Therefore, for any $1\leq j\leq d-2,$ one has
	$$
	b_{d,d-j+1}^{(i_0-1)}=-c_jb_{d,1}^{(i_0-1)}-c_{j+1}b_{d,2}^{(i_0-1)}-\cdots-c_{d-1}b_{d,d-j}^{(i_0-1)}+b_{j,1}^{(i_0-1)}.
	$$
	Note that $b_{j,1}^{(i_0-1)}=b_{j+1,1}^{(i_0)},$ $1\leq j\leq d-2,$ $b_{d,1}^{(i_0-1)}=0,$ and thus \eqref{23333} implies that 
	$$
	b_{d,d-j}^{(i_0)}=-c_{j+1}b_{d,1}^{(i_0)}-c_{j+2}b_{d,2}^{(i_0)}-\cdots-c_{d-1}b_{d,d-j-1}^{(i_0)}+b_{j+1,1}^{(i_0)},
	$$
	which is exactly the $(1+j)$-th equation in \eqref{2333}. Hence the proof follows from induction.
\end{proof}

\begin{lemma}\label{B.}
	Let $\gamma\in G(F)$ be regular. Then there exists a finite set of elements $\Gamma_{reg}=\{\gamma_i\in G(F):\ 0\leq i\leq m_0\}$ such that 
	\begin{enumerate}
		\item $G(F)=\bigcup_{0\leq i\leq m_0} P_0(F)\gamma_iF(\gamma),$ where $P_0$ is the mirabolic subgroup of $G;$
		\item There are at most one $\gamma_i\in \Gamma_{reg}$ satisfying that 
		$$
		\gamma_iF(\gamma)\gamma_i^{-1}\nsubseteq \bigcup _{k=1}^{n-1}Q_k(F).
		$$
	\end{enumerate}
\end{lemma}
\begin{proof}
	Denote by $f(\lambda)=\wp_1(\lambda)^{e_1}\cdots \wp_m(\lambda)^{e_m}$ the characteristic polynomial of $\gamma\in G(F)$, where $e_i\geq 1,$ $\wp_i$'s are distinct monic and irreducible polynomials over $F,$ $1\leq i\leq m.$ Let $d_i=\deg \wp_i.$ Then $d_1e_1+d_2e_2+\cdots +d_me_m=\deg f=n.$ Set 
	$$
	F(\gamma)=F[\lambda]/\left(f(\lambda)\right)=\bigoplus_{i=1}^mF[\lambda]/\left(\wp_i(\lambda)^{e_i}\right)
	$$ 
	be the polynomial algebra generated by $\gamma,$ and denote by $F(\gamma)^{\times}$ the set of invertible elements in $F(\gamma).$ Since $\gamma$ is regular, then by Lemma \ref{Jordan}, $\gamma$ is $G(F)$-conjugate to a matrix of the form
	\begin{align*}
	\gamma^*=\begin{pmatrix}
	\mathcal{J}(\wp_1^{e_1})&& &\\
	&\mathcal{J}(\wp_2^{e_2})&&\\
	&&\ddots&\\
	&&&\mathcal{J}(\wp_m^{e_m})
	\end{pmatrix}\in \begin{pmatrix}
	G^{(1)}(F)&& &\\
	&G^{(2)}(F)&&\\
	&&\ddots&\\
	&&&G^{(m)}(F)
	\end{pmatrix},
	\end{align*}
	where $G^{(1)}(F):=GL_{d_ie_i}(F),$ $1\leq i\leq m.$ We may assume $\gamma=\gamma^*.$ Write $k_i=d_ie_i,$ $1\leq i\leq m.$ For any $\mathfrak{a}=(a_1,a_2,\cdots,a_n)\in F^n,$ let $\eta=(0,0,\cdots,0,1)$ and $\widetilde{\eta}_i:$ $F^n\longrightarrow F^{k_i},$ such that
	$$
	\widetilde{\eta}_i\left(\mathfrak{a}\right)=(a_{k_1+\cdots+k_{i-1}+1},\cdots,a_{k_1+\cdots+k_{i-1}+k_i}),\ 1\leq i\leq m.
	$$
	Also, for convenience we write $\eta^{(e_i)}\left(\mathfrak{a}\right)$ for the last $d_i$ components of $\widetilde{\eta}_i\left(\mathfrak{a}\right),$ namely $\eta^{(e_i)}\left(\mathfrak{a}\right)=(a_{k_1+\cdots+k_{i-1}+(e_i-1)d_i+1},a_{k_1+\cdots+k_{i-1}+(e_i-1)d_i+2},\cdots,a_{k_1+\cdots+k_{i-1}+k_i}),$ $1\leq i\leq m.$
	We then split $G(F)$ into a disjoint union of sets following the conditions on the $\mathfrak{a}_{i}$'s and show that each of the sets is a $P_0(F)\gamma_iF(\gamma)^{\times}$ for a specific $\gamma_i.$
	
	Let $\mathcal{S}_0=\{\delta\in G(F):\ \eta\delta=\mathfrak{a}=(a_1,a_2,\cdots,a_n)\in F^n,\ \text{such that for}\ 1\leq i\leq m,\ \eta^{(e_i)}\left(\mathfrak{a}\right)\neq \boldsymbol{0}\}.$ Let $\eta_i=(0,0,\cdots,0,1)\in F^{k_i},$ $1\leq i\leq m.$ Denote by 
	\begin{align*}
	\gamma_0=\begin{pmatrix}
	I_{k_1}&& &\\
	\vdots&\ddots&&\\
	\vdots&&I_{k_{m-1}}&\\
	\eta_{1}&\ldots&\eta_{m-1}&I_{k_m}
	\end{pmatrix}.
	\end{align*}
	Then applying Lemma \ref{A'} to each $\widetilde{\eta}_i\left(\mathfrak{a}\right)\in F^{k_i},$ we find for each $1\leq i\leq m,$ for any $\delta\in\mathcal{S}_0,$ a unique $x_i\in F[\mathcal{J}(\wp_i^{e_i})]^{\times},$ such that $\eta x_i=\widetilde{\eta}_i\delta.$ (Write $x_i$ in the form in \eqref{B'}, the definition of $\mathcal{S}_0$ implies that $A_0\neq0,$ thus $A_0\in F[C]^{\times},$ so $x_i\in F[\mathcal{J}(\wp_i^{e_i})]^{\times}.$) Let $x=\diag(x_1,\cdots,x_m),$ then $\eta(\gamma_0x)=\eta\delta.$ Consequently, $\delta(\gamma_0x)^{-1}\in P_0(F),$ i.e., $\delta\in P_0(F)\gamma_0F(\gamma)^{\times}.$ Moreover, one has $P_0(F)\cap \gamma_0F(\gamma)^{\times}\gamma_0^{-1}=\{I_n\}.$ To see this, look at the last row of $\gamma_0x\gamma_0^{-1}.$ A straightforward computation shows that 
	$$
	\widetilde{\eta}_i\left(\gamma_0x\gamma_0^{-1}\right)=\eta_{i}x_i-\eta_i,\ 1\leq i\leq m-1, 
	$$
	and $\widetilde{\eta}_m\left(\gamma_0x\gamma_0^{-1}\right)=\eta_m.$ Then by uniqueness part of Lemma \ref{A'}, it follows that $x_i=I_{k_i},$ $1\leq i\leq m.$
	
	For any $1\leq l\leq m-1$ and $1\leq i_1<\cdots<i_l\leq m,$ let $\mathcal{S}^{(l)}_{(i_1,\cdots,i_l)}=\{\delta\in G(F):\ \eta\delta=\mathfrak{a}=(a_1,a_2,\cdots,a_n)\in F^n,\ \text{such that}\ \eta^{(e_j)}\left(\mathfrak{a}\right)=\boldsymbol{0}\ \text{iff}\ j\in \{i_1,\cdots, i_l\}\}.$ For any $1\leq i\leq m,$ $1\leq e\leq e_i-1,$ define $\eta_i^{e}=(0,\cdots,0,1,0,\cdots,0)\in F^{k_i},$ where the only nonzero entry (i.e. $1$) occurs in the $ed_i$-position, namely, there are $(e_i-e)d_i$ zeros on the right hand side of the entry $1.$ Let $v_l$ denote the $l$-th element of $\eta_i^e,$ define 
	$$
	\eta_i^*=(0,\cdots0,v_{d_i+1},v_{d_i+2},\cdots,v_{k_i-d_i},v_1,v_2,\cdots,v_{d_i}).
	$$
	For any $1\leq i\leq m$ and $1\leq s<t\leq m,$ define the Weyl elements $w^{(1)}_{i}$ and $w^{(2)}_{s,t}$ as 
	$$
	w^{(1)}_{i}=\begin{pmatrix}
	I_{k_i'}&&&&&&\\
	&0&&\cdots&&I_{d_i} & \\
	&& I_{d_i}  &&&& \\
	&\vdots&&\ddots &&\vdots&\\
	&&&&I_{d_i}&&\\
	&I_{d_i}  &&\cdots&&0& \\
	&&&&&&I_{k_i''}
	\end{pmatrix},
	$$
	where $k_i'=k_1+\cdots k_{i-1},$ $k_i''=k_{i+1}+\cdots+k_m,$ $\forall$ $1\leq i\leq m;$ and 
	$$
	w^{(2)}_{s,t}=  \begin{pmatrix}
	I_{k_s'}&&&&&&\\
	&0&&\cdots&&I_{k_t} & \\
	&& I_{k_{s+1}}  &&&& \\
	&\vdots&&\ddots &&\vdots&\\
	&&&&I_{k_{t-1}}&&\\
	&I_{k_s}  &&\cdots&&0& \\
	&&&&&&I_{k_t''}
	\end{pmatrix}.
	$$
	We write simply that $w^{(2)}_{s,t}=I_n$ if $s=t.$
	
	Given $\delta\in \mathcal{S}^{(l)}_{(i_1,\cdots,i_l)},$ let $\eta\delta=\mathfrak{a}=(a_1,a_2,\cdots,a_n)\in F^n.$ Then $\eta^{(e_j)}\left(\mathfrak{a}\right)=\boldsymbol{0}$ if and only if $j\in \{i_1,\cdots,i_l\}.$ If $\widetilde{\eta}_j(\mathfrak{a})\neq0,$ let $e^0_j\leq e_j-1$ be the maximal integral such that $(a_{k_1+\cdots+(e^0_j-1)d_j+1},a_{k_1+\cdots+k_{j-1}+(e^0_j-1)d_j+2},\cdots,a_{k_1+\cdots+k_{j-1}+e_j^0d_j})\neq 0.$ Then by Lemma \ref{A'}, for each such $j,$ one can find an element $x_j$ of the form in \eqref{B'} such that the $e^0_jd_j$-row of $x_j$ is exactly 
	$$
	(a_{k_1+\cdots+k_{j-1}+1},a_{k_1+\cdots+k_{j-1}+1},\cdots,a_{k_1+\cdots+k_{j-1}+e_j^0d_j})\neq 0.
	$$ 
	If $\widetilde{\eta}_j(\mathfrak{a})=0,$ then take $x_j$ to be an arbitrary element in $F(\mathcal{J}(\wp_j^{e_j}))^{\times}.$ Since $\eta^{(e_i)}\left(\mathfrak{a}\right)\neq \boldsymbol{0}$ for any $i\notin \{i_1,\cdots,i_l\},$ we can pick uniquely elements $x_i\in F(	\mathcal{J}(\wp_i^{e_i}))^{\times}$ such that their last row is $\widetilde{\eta}_i\mathfrak{a},$ $ i\in \{1,2,\cdots,m\}\setminus \{i_1,\cdots,i_l\}.$ Let $x=\diag(x_1,x_2,\cdots,x_m).$
	
	Apply the transposition $\prod_{1\leq j\leq l}(j,i_j)\in S_n$ on the ordered $n$-set $(1,2,\cdots,n)$ to get another ordered $n$-set $(i_1,\cdots,i_l,i_{l+1}',\cdots,i_m').$ Then clearly $\eta^{(e_{i_m'})}(\mathfrak{a})\neq0,$ hence by our choice, the last row of $x_{i_m'}$ is exactly $\eta_{i_m'}(\mathfrak{a}).$ Thus we can define the element $\gamma_{(i_1,\cdots,i_l)}\in G(F)$ as
	$$
	\begin{pmatrix}
	I_{k_{i_1}} &  & &&&&&& \\
	& I_{k_{i_2}} & &&&&& &\\
	& & \ddots & &&&&&\\
	&&& I_{k_{i_l}}&&& &&\\
	\vdots   &&&& I_{k_{i_{l+1}'}}& &&&\\
	&&&&& I_{k_{i_{l+2}'}}&&&\\
	&&&&& &\ddots & &\\
	&&&&  &  &&  I_{k_{i_{m-1}'}}&\\
	\eta^*_{i_1} & \eta^*_{i_2}& \cdots &\eta^*_{i_l}&\eta_{i_{l+1}'}& \eta_{i_{l+2}'}& \cdots &\eta_{i_{m-1}'}& I_{k_{i_{m-1}'}}
	\end{pmatrix}\cdot \prod_{j=1}^lw^{(1)}_{j}w^{(2)}_{j,i_j}.
	$$
	Then by our setting, $\eta\gamma^{(l)}_{(i_1,\cdots,i_l)}x=\mathfrak{a}=\eta\delta.$ Then $\delta\in P_0(F)\gamma_{(i_1,\cdots,i_l)}F(\gamma)^{\times}.$ Given any $x\in F(\gamma)^{\times},$, one checks directly that 
	$$
	\prod_{j=1}^lw^{(1)}_{j}w^{(2)}_{j,i_j}x\left(\prod_{j=1}^lw^{(1)}_{j}w^{(2)}_{j,i_j}\right)^{-1}\in Q_{d_{i_1}}(F).
	$$
	Therefore, $\gamma_{(i_1,\cdots,i_l)}x\gamma_{(i_1,\cdots,i_l)}^{-1}\in Q_{d_{i_1}}(F),$ i.e., $\gamma_{(i_1,\cdots,i_l)}F(\gamma)^{\times}\gamma_{(i_1,\cdots,i_l)}^{-1}\subseteq Q_{d_{i_1}}(F).$
	
	Now we consider $\mathcal{S}^{(m)}=\{\delta\in G(F):\ \eta\delta=\mathfrak{a}=(a_1,a_2,\cdots,a_n)\in F^n-\{\boldsymbol{0}\},\ \eta^{(e_j)}\left(\mathfrak{a}\right)=\boldsymbol{0},\  1\leq j\leq m\}.$ Let $\delta\in \mathcal{S}^{(m)}$ such that $\eta\delta=\mathfrak{a}=(a_1,a_2,\cdots,a_n)\in F^n-\{\boldsymbol{0}\}.$ For each  $1\leq j\leq m$ such that $\widetilde{\eta}_j(\mathfrak{a})\neq0,$ denote by $e^0_j\leq e_j-1$ the maximal integral such that 
	$$
	(a_{k_1+\cdots+(e^0_j-1)d_j+1},a_{k_1+\cdots+k_{j-1}+(e^0_j-1)d_j+2},\cdots,a_{k_1+\cdots+k_{j-1}+e_j^0d_j})\neq 0.
	$$
	Likewise, for each such $j,$ one can find an element $x_j$ of the form in \eqref{B'} such that the $e^0_jd_j$-row of $x_j$ is exactly 
	$(a_{k_1+\cdots+k_{j-1}+1},a_{k_1+\cdots+k_{j-1}+1},\cdots,a_{k_1+\cdots+k_{j-1}+e_j^0d_j}).$ For the remaining $j$'s, take arbitrary $x_j\in F(	\mathcal{J}(\wp_j^{e_j}))^{\times}.$ Let $x=\diag(x_1,\cdots,x_m).$
	
	Now we pick arbitrarily a $j_0$ such that $\widetilde{\eta}_{j_0}(\mathfrak{a})\neq0.$ Let $j_0'\neq j_0$ be another integer. Denote by 
	$$
	w^{(1)}_{j_0,e_{j_0}}=\begin{pmatrix}
	I_{k_{j_0}'}&&&\\
	&0&I_{(e_j-e_{j_0})d_{j_0}} & \\
	&I_{e_{j_0}d_{j_0}} &0& \\
	&&&I_{k_{j_0}''}
	\end{pmatrix}.
	$$
	Let 
	$$
	\gamma_{m}=\begin{pmatrix}
	I_{k_{j_0'}} &  & &&&&&& \\
	& & \ddots & &&&&&\\
	&&& I_{k_{1}}&&& &&\\
	\vdots   &&&& \ddots& &&&\\
	&&&&& I_{k_{m}}&&&\\
	&&&&& &\ddots & &\\
	\eta^*_{i_1} & & \cdots &\eta_{1}&\ldots& \eta_{m}& \cdots & I_{k_{j_0}}
	\end{pmatrix}\cdot w^{(1)}_{j_0'}w^{(2)}_{1,j_0'}w^{(1)}_{j_0,e_{j_0}}w^{(2)}_{j_0,m}.
	$$	
	Then $\eta\gamma_{m}x=\eta\delta.$ So $\delta\in P_0(F)\gamma_mF(\gamma)^{\times}.$ Moreover, for any $x'\in F(\gamma)^{\times},$ $\gamma_mx'\gamma_m^{-1}\in Q_{d_{j_0'}}(F),$ the standard maximal parabolic subgroup of type $(j_0',n-j_0').$
	
	In all, we see that 
	\begin{align*}
	G(F)&=\mathcal{S}_0\coprod\bigcup_{l=1}^{m-1}\bigcup_{1\leq i_1<\cdots<i_l\leq m}\mathcal{S}^{(l)}_{(i_1,\cdots,i_l)}\coprod\mathcal{S}^{(m)}\\
	&=P_0(F)\gamma_0F(\gamma)^{\times}\bigcup\bigcup_{\substack{1\leq l\leq m-1\\1\leq i_1<\cdots<i_l\leq m}}P_0(F)\gamma_{(i_1,\cdots,i_l)}F(\gamma)^{\times}\bigcup P_0(F)\gamma_mF(\gamma)^{\times},
	\end{align*}
	where $\gamma_mF(\gamma)^{\times}\gamma_m^{-1}$ and each $\gamma_{(i_1,\cdots,i_l)}F(\gamma)^{\times}\gamma_{(i_1,\cdots,i_l)}^{-1}$ are contained in some standard maximal parabolic subgroup, and $P_0(F)\cap \gamma_0F(\gamma)^{\times}\gamma_0^{-1}=\{I_n\}.$
\end{proof}
Now we prove the result on the structure of regular conjugacy classes: 
\begin{prop}\label{reg}
	Let $\mathcal{C}$ be a regular $G(F)$-conjugacy classes in $G(F).$ Then there exists a $P(F)$-conjugacy class $\mathcal{C}_0$ such that 
	\begin{equation}\label{9'}
	\mathcal{C}=\mathcal{C}_0\coprod\bigcup _{k=1}^{n-1}\mathcal{C}\cap Q_k(F)^{P(F)}.
	\end{equation} 
\end{prop}
\begin{proof}
	By Lemma \ref{B.} we have  $G(F)=\bigcup_{0\leq i\leq m_0} P_0(F)\gamma_iF(\gamma),$ where $P_0$ is the mirabolic subgroup of $G,$ and $\gamma\in C.$ If $\delta\in G(F),$ there exists $p\in P_0(F)$ and $i\in \{0,2,\cdots, m_0\}$ and $x\in F(\gamma)^{\times},$ such that $\delta=p\gamma_ix.$ So one has 
	$$
	\delta\gamma\delta^{-1}=p\gamma_ix\gamma x^{-1}\gamma_i^{-1}p^{-1}=p\gamma_i\gamma \gamma_i^{-1}p^{-1}.
	$$
	If $i\geq 1,$ then $\delta\gamma\delta^{-1}\in \mathcal{C}\cap Q_j(F)^{P(F)},$ for some standard maximal parabolic subgroup $Q_j$ of type $(j,n-j),$ $1\leq j\leq n-1.$ And for $i=0,$ $\delta\gamma\delta^{-1}=p\gamma_0\gamma \gamma_0^{-1}p^{-1}.$ Take $\gamma'=\gamma_0\gamma \gamma_0^{-1}.$ Then $\mathcal{C}_0=\{p\gamma'p^{-1}:\ p\in P(F)\}.$ This proves the result.
\end{proof}

	\subsection{Contributions from Nonsingular Conjugacy Classes}\label{sec2.2}
Let $s\neq 1.$ Consider the well defined distribution 
	$$
	I(s)=\int_{G(F)Z_G(\mathbb{A}_F)\setminus G(\mathbb{A}_F)}\K_0(x,x)E(x,\Phi;s)dx.  
	$$
	
Let $Q_k$ be the standard parabolic subgroup of $\GL(n)$ of type $(k,n-k).$ In Proposition \ref{reg} we show that for any regular $G(F)$-conjugacy classes $\mathcal{C}$ in $G(F),$ there exists a $P(F)$-conjugacy class $\mathcal{C}_0$ such that 
\begin{align*}
\mathcal{C}=\mathcal{C}_0\coprod\bigcup _{k=1}^{n-1}\mathcal{C}\cap Q_k(F)^{P(F)}.
\end{align*} 
Moreover, such a $\mathcal{C}_0$ is uniquely determined by $\mathcal{C}.$ When $\mathcal{C}$ is a non-regular $G(F)$-conjugacy class, then by Proposition \ref{irre}, we have
\begin{align*}
\mathcal{C}=\bigcup _{k=1}^{n-1}\mathcal{C}\cap Q_k(F)^{P(F)}.
\end{align*} 	

Take $\mathcal{C}_0$ to be empty set in this case. Denote by
\begin{align*}
\mathfrak{S}=\bigcup _{k=1}^{n-1}\left(Z_G(F)\backslash Q_k(F)\right)^{P_0(F)}.
\end{align*}	

	Following the approach in \cite{JZ87}, we will treat $I(s)$ via the decomposition
	$$
	\K_0(x,x)=\sum_{\mathcal{C}}\K_{\mathcal{C}}(x)+\K_{\Sin}(x)+\K_{\infty}(x),
	$$
	where $\mathcal{C}$ runs through all  conjugacy classes in $G(F)/Z_G(F)$ and 
	\begin{align*}
	&\K_{\mathcal{C}}(x,y)=\sum_{\substack{\gamma\in \mathcal{C}_0}}\varphi(x^{-1}\gamma y)=\sum_{\substack{\gamma\in \mathcal{C}-\mathfrak{G}}}\varphi(x^{-1}\gamma y),\quad \K_{\Geo, \Reg}(x,y)=\sum_{\mathcal{C}} \K_{\mathcal{C}}(x,y),\\
	&\K_{\Geo,\Sin}(x,y)=\sum_{\gamma\in \mathfrak{G}}\varphi(x^{-1}\gamma y),\qquad \K_{\infty}(x,y)=-\K_{\Eis}(x,y)-\K_{\Res}(x,y).
	\end{align*}

	So correspondingly, integrating against the Eisenstein series $E(x,\Phi;s)$ associated to $P$ implies that $I(s)$ can be decomposed (at least formally) as
\begin{equation}\label{008}
	I(s)=I_{\Geo,\Reg}(s)+I_{\Geo,\Sin}(s)+I_{\infty}(s),\ \Re(s)>1.
\end{equation}
When $I_{\Geo,\Reg}(s),$ $I_{\Geo,\Sin}(s)$ and $I_{\infty}(s)$ all converge absolutely in $\Re(s)>1,$ the formula \eqref{008} would be rigorous.
\medskip 

	When $G=GL(2),$ Jacquet and Zagier (see \cite{JZ87}) computed $I_{\Reg}(s),$ $I_{\Geo,\Sin}(s)$ and $I_{\infty}(s)$ for general test function $\varphi,$ and verified the convergence. Note that the contribution from $I_{\Reg}(s)$ would give Artin $L$-functions of degree less or equal to $n.$ We shall deal with $I_{\Reg}(s)$ in this section, and leaving the computation of $I_{\Geo,\Sin}(s)$ and $I_{\infty}(s)$ in the following parts. For each $\mathcal{C},$ let (at least formally)
	\begin{align*}
	I_{\mathcal{C}}(s):=\int_{G(F)Z_G(\mathbb{A}_F)\setminus G(\mathbb{A}_F)}\K_{\mathcal{C}}(x,x)E(x,\Phi;s)dx.  
	\end{align*}
	
	Then by definition, $I_{\mathcal{C}}(s)=0$ unless $\mathcal{C}$ is regular. To describe these conjugacy classes, we introduce the classification of them by factorization of their characteristic polynomials. 
	
	Let $\mathcal{C}$ be a conjugacy class in $G(F).$ Denote by $P(t;\mathcal{C})$ the characteristic polynomial of $\mathcal{C}.$ Factorize it into irreducible ones with multiplicities as 
	\begin{align*}
	f(\lambda;\mathcal{C})=\prod_{i=1}^{g}\wp_i(\lambda;\mathcal{C})^{e_i},
	\end{align*} 
	where $P_i(t;\mathcal{C})\in F[\lambda]$ is an irreducible polynomial of degree $f_i.$ We may assume $f_1\geq \cdots\geq f_g.$ Denote by $\textbf{f}=(f_1,\cdots,f_g)\in \mathbb{Z}_{\geq 1}^g$ and $\textbf{e}=(e_1,\cdots,e_g)\in \mathbb{Z}_{\geq 1}^g.$ Then $\langle\textbf{f},\textbf{e}\rangle=\sum f_ie_i=n.$
\begin{defn}
Let notation be as before. We say $\mathcal{C}$ is of type $(\textbf{f},\textbf{e};g).$ Let $\Gamma_{\textbf{f},\textbf{e};g}$ be the collection of regular $G(F)$-conjugacy classes of type $(\textbf{f},\textbf{e};g).$ 
\end{defn}
	\medskip 

With the above definition, we have the decomposition:
\begin{equation}\label{009}
\bigsqcup_{\text{$\mathcal{C}$ regular}} \mathcal{C}=\bigsqcup_{\substack{\textbf{f},\ \textbf{e}\in \mathbb{Z}_{\geq 1}^g\\ \langle\textbf{f},\textbf{e}\rangle=n}}\Gamma_{\textbf{f},\textbf{e};g}.
\end{equation}

A useful observation is that if $\mathcal{C}$ is a regular conjugacy class in $G(F)$ of type $(f_1,\cdots,f_g;e_1,\cdots,e_g),$ and $\gamma\in\mathcal{C},$ then the centralizer of $\gamma$ in $G(F)$ is isomorphic to the invertible elements of the algebra $\oplus_{1\leq i\leq g}E_i^{e_i},$ where $E_i$ is a field extension of $F$ with $[E_i;F]=f_i;$ and $\oplus_{1\leq i\leq g}E_i^{e_i}$ denotes the direct sum of $e_i$ copies of $E_i.$
\medskip 

Let $\mathcal{C}\in \Gamma_{\textbf{f},\textbf{e};g}.$ Let $\gamma_{\mathcal{C}}\in \mathcal{C}$ be a fixed element. Let $\lambda_{\textbf{f},\textbf{e};g}\in G(F)$ be defined by 
\begin{equation}\label{0010}
\lambda_{\textbf{f},\textbf{e};g}=\begin{pmatrix}
I_{f_1}&&&\\
&\ddots& &\\
&&I_{f_1}\\
&&&\ddots \\
&&&&I_{f_g}\\
&&&&&\ddots\\
\eta_{f_1}&\cdots&\eta_{f_1}&\cdots&\eta_{f_g}&\cdots&I_{f_g}
\end{pmatrix}^{-1},
\end{equation}
where for each integer $m,$ $\eta_m=(0,\cdots,1)\in F^m,$ the row vector with the last entry being 1 and the rest being 0; and $I_m$ is the identity matrix of rank $m.$ 
\medskip

Then by Proposition \ref{reg} and unfolding $E(s,\Phi;s)$, we have, when $\Re(s)>1,$ that 
\begin{align*}
I_{\mathcal{C}}(s)=&\int_{Z_G(\mathbb{A}_F)P_0(F)\backslash G(\mathbb{A}_F)}\sum_{p\in P_0(F)}\varphi(x^{-1}p^{-1}\lambda_{\textbf{f},\textbf{e};g}^{-1}\gamma_{\mathcal{C}}\lambda_{\textbf{f},\textbf{e};g}px)f(x,s)dx\\
=&\int_{Z_G(\mathbb{A}_F)\backslash G(\mathbb{A}_F)}\varphi(x^{-1}\gamma_{\mathcal{C}}x)f(\lambda_{\textbf{f},\textbf{e};g}^{-1}x,s)dx,
\end{align*} 
supposing the above integrals converge absolutely. Combing this with \eqref{009} we then deduce (at least formally) that, when $\Re(s)>1,$  
\begin{equation}\label{0011}
\sum_{\mathcal{C}}I_{\mathcal{C}}(s)=\sum_{\substack{\textbf{f},\ \textbf{e}\in \mathbb{Z}_{\geq 1}^g\\ \langle\textbf{f},\textbf{e}\rangle=n}}\int_{Z_G(\mathbb{A}_F)\backslash G(\mathbb{A}_F)}\sum_{\mathcal{C}\in \Gamma_{\textbf{f},\textbf{e};g}}\varphi(x^{-1}\gamma_{\mathcal{C}}x)f(\lambda_{\textbf{f},\textbf{e};g}^{-1}x,s)dx.
\end{equation}
Moreover, \eqref{0011} would be rigorous if the right hand side converges absolutely, which is indeed the case. To verify, we will consider each type $(\textbf{f},\textbf{e};g)$ separately in the following subsections. 

\subsubsection{Type $(n;1)$}\label{n,1}
We treat the conjugacy classes of type $(\textbf{f},\textbf{e};g)=((n),(1);1)$ first, these are exactly elliptic regular conjugacy classes. Denote by 
\begin{align*}
I_{r.e.}(s)=I_{r.e.}^{\varphi}(s)=\sum_{\mathcal{C}\ \text{regular elliptic}}I_{\mathcal{C}}(s).
\end{align*}
	
	\begin{lemma}\label{A.1}
		If $\gamma\in G(F)$ is regular elliptic, then $G(F)=P(F)F[\gamma]^{\times},$ where $P$ is the parabolic subgroup of $\GL(n)$ of type $(n-1,1).$
	\end{lemma}
	\begin{proof}
		Since $\gamma$ is regular elliptic, $G_{\gamma}(F)=F[\gamma]^{\times},$ and $\dim F[\gamma]=n.$ 
		
		Let $y=(0,0,\cdots,0,1)\in F^n.$ Consider the linear map:
		$$
		\iota:\ F[\gamma]\rightarrow F^n,\qquad x\mapsto\iota(x)=yx.
		$$
		Since $\gamma$ is regular elliptic, $F[\gamma]$ is a field, so any nonzero element is invertible. Consequently, the map $\iota$ is injective, and hence surjective. So, given $g\in G(F),$ there exists $h\in G(F)$ such that $yg=yh$ which implies that $gh^{-1}\in P(F),$ the isotropy subgroup of $y,$ i.e., $g\in P(F)F[\gamma]^{\times}.$
	\end{proof}
	
	\begin{prop}\label{re}
Let notation be as before. Then for every field extension $E/F$ of degree $n,$ there is an analytic function $Q_E(s)$ such that 
		\begin{equation}\label{main geo} 
		I_{r.e.}(s)=\frac1n\sum_{[E:F]=n}Q_E(s)\Lambda_E\left(s,\tau\circ N_{E/F}\right),
		\end{equation}
		where the summation is taken over only finitely many $E$'s, depending implicitly only on the test function $\varphi.$
	\end{prop}
	\begin{proof}
		Since $\Gamma_{r.e.}\left(G(F)/Z(F)\right)$ is invariant under $P(F)$-conjugation, we have
		\begin{align*}
		I_{r.e.}(s)=\int_{P(F)Z_G(\mathbb{A}_F)\setminus G(\mathbb{A}_F)}\sum_{\gamma\in \Gamma_{r.e.}\left(G(F)/Z(F)\right)}\varphi(x^{-1}\gamma x)\cdot f(x,s)dx.
		\end{align*} 
		
		Denote by $\{\Gamma^{r.e.}\}$ a set of representatives for the regular elliptic conjugacy classes in $\Gamma_{r.e.}\left(G(F)/Z(F)\right).$ For any $\gamma\in \{\Gamma^{r.e.}\},$ the centralizer of $\gamma$ in $G(F)/Z(F)$ is exactly $F[\gamma]^{\times}.$ Then we have
		\begin{equation}\label{5}
		\sum_{\gamma\in \Gamma_{r.e.}\left(G(F)/Z(F)\right)}\varphi(x^{-1}\gamma x)=\sum_{\gamma\in \{\Gamma^{r.e.}\}}\sum_{\delta\in F[\gamma]^{\times}Z_G(F)\backslash G(F)}\varphi(x^{-1}\delta^{-1}\gamma\delta x).
	    \end{equation}
	    By Lemma \ref{A.1} one has $G(F)=P(F)F[\gamma]^{\times}.$ Since $P(F)\cap F[\gamma]^{\times}=\{I_n\},$ every element $\delta\in Z_G(F)\backslash G(F)$ can be written unique as $\delta=p\nu,$ where $p\in Z_G(F)\backslash P(F)$ and $\nu\in F[\gamma]^{\times}.$ Hence the inner sum of \eqref{5} could be taken over $p\in Z_G(F)\backslash P(F).$ Therefore, substituting these into the expression of $I_{r.e.}(s)$ one will obtain
		\begin{equation}\label{6}
		I_{r.e.}(s)=\int_{Y_G}\sum_{\gamma\in \{\Gamma^{r.e.}\}}\sum_{p\in Z_G(F)\backslash P(F)}\varphi(x^{-1}p^{-1}\gamma p x)f(x,s)dx,
		\end{equation} 
		where $Y_G=P(F)Z_G(\mathbb{A}_F)\backslash G(\mathbb{A}_F).$ For the given $F,$ let $E/F$ be a field extension of degree $n.$ Fix an algebraic closure $\bar{F}$ of $F$, then $E$ embeds into $\bar{F},$ we look at the contribution from all the regular elliptic conjugacy classes together. We say that a conjugacy class belongs to an extension $E$ of $F$ if it consists of the conjugates of some element $\gamma\in E^{\times}/F^{\times}-\{1\}$ with the usual identification. We have to distinguish between two cases:
		\begin{description}
			\item[(a)] $E/F$ is Galois.
			\item[(b)] $E/F$ is not Galois.
		\end{description}
		The idea is to replace the summation over $\gamma\in \{\Gamma^{r.e.}\}$ by summation over extensions $E/F$ of degree $n;$ and inside, summation over elements of $E.$
		
		\begin{description}
			\item[Case (a)] When $\gamma$ varies over $E^{\times}/F^{\times}$ we get each conjugacy class belonging to $E$ exactly $n$ times.
			\item[Case (b)] When $\gamma$ varies over $E^{\times}/F^{\times}$ we get each conjugacy class belonging to $E$ once; but the sets of conjugacy classes belonging to the $n$ embeddings of $E$ in $\bar{F}$ are identical.
		\end{description}
		
		So in either case, we can rewrite the integral in \eqref{6} as 
		\begin{equation}\label{7'}
		I_{r.e.}(s)=\frac1n\int_{Z_G(\mathbb{A}_F)\setminus G(\mathbb{A}_F)}\sum_{[E:F]=n}\sum_{\gamma\in E^{\times}/F^{\times}-\{1\}}\varphi(x^{-1}\gamma x)f(x,s)dx
		\end{equation}
		where the right hand summation is over all extensions $E/F$ of degree $n.$ Note that $\eqref{7'}$ is the same as \eqref{0011}. We claim that the sum over $E/F$ such that $[E:F]=n$ is actually finite, depending only on $\varphi.$ To see this, consider a function $\theta:$ $Z_G(\mathbb{A}_F)\setminus G(\mathbb{A}_F)\rightarrow \mathbb{A}_F^{n-1}$ defined by
		$$
		\theta(x)=\left(\frac{\sigma_1(x)^n}{\sigma_n(x)},\frac{\sigma_2(x)^n}{\sigma_n(x)^2},\cdots,\frac{\sigma_{n-1}(x)^n}{\sigma_n(x)^{n-1}}\right),
		$$
		where $\sigma_i(x)$ is the $i$-th symmetric polynomial in the eigenvalues of $x.$ Since $\sigma_n(x)=\det x\neq0,$ $\forall$ $x\in G(\mathbb{A}_F),$ $\theta$ is continuous. It is obvious that $\theta$ is conjugation invariant. Since $\supp \varphi$ is compact, so is $\theta(\supp \varphi).$ So when $x$ varies over a compact set $C \subset G(\mathbb{A}_F),$ the set
		$\big\{\theta(\gamma):\ \gamma\in G(F)/Z_G(F),\ \varphi(x^{-1}\gamma x)\neq0,\ \forall\ x\in C\big\}$ is the intersection of a compact set with a discrete set, hence is finite.
		\begin{itemize}
			\item If $\theta(\gamma)\neq 0,$ the number of distinct fields $F[\gamma]$ with a given value of $\theta(\gamma)$ is at most $n,$ hence is finite.
			\item If $\theta(\gamma)=0,$ the map $\gamma\mapsto\det \gamma$ from
			$$
			G(\mathbb{A}_F)/Z_G(\mathbb{A}_F)\rightarrow \mathbb{A}_F^{\times}/\mathbb{A}_F^{\times^n}U_F
			$$
			(with $\mathbb{A}_F^{\times^n}=\{a^n:\ a\in\mathbb{A}_F^{\times}\}$ and $U_F$ is the maximal compact subgroup of $\mathbb{A}_F^{\times}$) is continuous; so the image of $\supp \varphi$ is also compact. Since $\mathbb{A}_F^{\times^n}$ has finite index in $\mathbb{A}_F^{\times^n}U_F\cap F,$ there are only finitely many values for the image $\det \gamma\mod\mathbb{A}_F^{\times^n}$ for $\gamma\in G(F)/Z_G(F),$ $\varphi(x^{-1}\gamma x)\neq0,$ $x\in C.$ When $\theta(\gamma)=0,$ $F[\gamma]$ is determined (up to embedding) by $\det \gamma,$ so we are done.
		\end{itemize}
	    Moreover, since the coefficients of the characteristic polynomial of every $\gamma\in E^{\times}/F^{\times}$ are rational, and lie in a compact set depending on $\supp\varphi$ (and a discrete subset of a compact set is finite), the sum over $\gamma\in E^{\times}/F^{\times}-\{1\}$ is a finite sum. Thus we can interchange integrals in \eqref{7'} to get
	    \begin{equation}\label{8'}
	    I_{r.e.}(s)=\frac1n\sum_{[E:F]=n}\sum_{\gamma\in E^{\times}/F^{\times}-\{1\}}\int_{G(\mathbb{A}_F)}\varphi(x^{-1}\gamma x)\Phi(\eta x)\tau(\det x)|\det x|^sdx,
	    \end{equation}
		where $\eta=(0,\cdots,0,1)\in \mathbb{A}_F^n.$ Let $I_E(s)$ be the inner integral in \eqref{8'}, then
		\begin{align*}
		I_E(s)=\int_{G_{\gamma}(\mathbb{A}_F)\backslash G(\mathbb{A}_F)}\varphi(x^{-1}\gamma x)\int_{G_{\gamma}(\mathbb{A}_F)}\Phi[(0,\cdots,0,1)tx]\tau(\det tx)|\det tx|^sd^{\times}tdx,
		\end{align*}
		where $G_{\gamma}$ is the centralizer of $\gamma$ in $G.$ Hence, $G_{\gamma}(\mathbb{A}_F)\simeq \mathbb{A}_E^{\times}.$ If we identify $G_{\gamma}(\mathbb{A}_F)$ with $\mathbb{A}_E^{\times},$ $\det\mid_{E^{\times}}:$ $E^{\times}\rightarrow F^{\times}$ with the norm map $N_{E/F},$ $t\mapsto |\det t|_{\mathbb{A}_F}$ with the idele norm in $E,$ and $\mathcal{S}(\mathbb{A}_F^n)$ with $\mathcal{S}(\mathbb{A}_E),$ we see that the inner integral is just the Tate integral for $\Lambda_E\left(s,\tau\circ N_{E/F}\right).$ So there is some elementary function $Q(s)$ of $s$ such that $I_E(s)=Q(s)\Lambda_E\left(s,\tau\circ N_{E/F}\right),$ where $\Lambda_E\left(s,\tau\circ N_{E/F}\right)$ is the complete $L$-function attached to $E.$
		
		Consequently, $I_E(s)$ itself converges normally for $\Re(s)>1$ and its behavior is given by $Q(s)L_E\left(s,\tau\circ N_{E/F}\right).$ This also given the meromorphic continuation of $I_E(s)$ to the entire $s$-plane. Since
		$$
		I_{r.e.}(s)=\sum_{\text{$\mathcal{C}$ of type $((n),(1))$}}I_{\mathcal{C}}(s)=\frac1n\sum_{[E:F]=n}\sum_{\gamma\in E^{\times}/F^{\times}-\{1\}}I_E(s),
		$$
		where the sums are finite, then $I_{r.e.}(s)$ is well defined when $\Re(s)>1,$ admits a meromorphic continuation to $s\in\mathbb{C},$ moreover, \eqref{main geo} holds. 
	\end{proof}

\subsubsection{Orbital Integrals of General Type}
In this subsection, we deal with orbital integrals of general type $(\textbf{f},\textbf{e};g).$ Note that one of the key ingredients to handle the elliptic regular case is that 
\begin{equation}\label{0013}
x\mapsto \sum_{\lambda\in E^{\times}/F^{\times}-\{1\}}\varphi(x^{-1}\lambda x) 
\end{equation}
has compact modulo $G_{\gamma}(\mathbb{A}_F).$ However, the function \eqref{0013} is not compactly supported for general type $(\textbf{f},\textbf{e};g),$ for example, \eqref{25} has no compact support for regular unipotent conjugacy classes. So we must proceed differently from the elliptic regular case in Subsection \ref{n,1}. Denote by 
\begin{align*}
I_{\textbf{f},\textbf{e};g}(s)=\int_{Z_G(\mathbb{A}_F)\backslash G(\mathbb{A}_F)}\sum_{\mathcal{C}\in \Gamma_{\textbf{f},\textbf{e};g}}\varphi(x^{-1}\gamma_{\mathcal{C}}x)f(\lambda_{\textbf{f},\textbf{e};g}^{-1}x,s)dx,
\end{align*}
where $g\geq 1,$ $\textbf{f}, \textbf{e}\in \mathbb{Z}_{\geq 1}^g,$ $\langle\textbf{f},\textbf{e}\rangle=n;$ and $\Re(s)>1.$ We may write $\textbf{f}=(f_1,\cdots,f_g)$ with $f_1\geq \cdots\geq f_g;$ and $\textbf{e}=(e_1,\cdots,e_g).$ Denote by $N_{k,l}$ the number of pairs $(f_i,e_i)$ such that $f_i=k$ and $e_i=l.$ Define 
\begin{align*}
C_{\textbf{f},\textbf{e}}=\prod_{k=1}^n\prod_{l=1}^nN_{k,l}!.
\end{align*}

Let $E$ be a finite extension of $F.$ Let $\chi$ be an idele class character of $\mathbb{A}_E^{\times}.$ Let $j$ be a positive integer. Denote by 
\begin{align*}
\Lambda_{E}[j](s,\chi):=\Lambda_{E}(js-j+1,\chi^j),
\end{align*}
where $\Lambda_E(s,\chi)$ is the complete Hecke $L$-function associated to $\chi.$

\begin{prop}\label{gel}
Let notation be as before. Then 
\begin{equation}\label{ge} 
I_{\textbf{f},\textbf{e};g}(s)=\frac{1}{C_{\textbf{f},\textbf{e}}}\prod_{i=1}^g\frac{1}{f_i^{e_i}}\sum_{[E_{i}:F]=f_i}Q_{E_{i}}(s)\prod_{j=1}^{e_i}\Lambda_{E_{i}}[j]\left(s,\tau\circ N_{E_i/F}\right),
\end{equation}
where for each $i,$ the innermost summation is taken over only finitely many fields $E_{i}$'s, depending implicitly only on the test function $\varphi;$ and each $Q_{E_{i}}(s)$ is an entire function.
\end{prop}
\begin{proof}
Let $\mathcal{C}$ be a regular conjugacy class of type $(\textbf{f},\textbf{e};g).$ By Lemma \ref{Jordan} and proof of Lemma \ref{A'}, we can write $\mathcal{C}=\{\gamma\},$ with a typical element $\gamma$ given by 
\begin{align*}
\gamma=\begin{pmatrix}
\alpha_{1}&&&\\
\vdots&\ddots& &\\
*&\cdots&\alpha_{1}\\
&&&\ddots \\
&&&&\alpha_{g}\\
&&&&\vdots&\ddots\\
&&&&*&\cdots&\alpha_{g}
\end{pmatrix},
\end{align*}
where the entries $\alpha_{i}\in E_{i}^{\times}/F^{\times}-\{1\}\hookrightarrow\GL(f_i,F),$ with some field $[E_{i}: F]=f_i,$ $1\leq i\leq g,$ $1\leq j\leq e_i.$ We may choose $\gamma$ such that each block is of the form \eqref{0012}, namely, its quasi-rational canonical form. Denote by $G_{\gamma}$ the stabilizer of $\gamma.$ Let $P_{\textbf{f},\textbf{e};g}$ be the standard parabolic subgroup of type $(f_1e_1,\cdots,f_ge_g).$ Then $G_{\gamma}$ is a subgroup of $P_{\textbf{f},\textbf{e};g}.$
\medskip 

Let $\mathcal{I}_1$ be the set of integers $1\leq i\leq g$ such that $e_i=1;$ let $\mathcal{I}_2$ be the set of integers $1\leq i\leq g$ such that $e_i>1.$ Then we can write $\gamma=\gamma_1\gamma_2,$ where $\gamma_1$ is formed by replacing $\alpha_{f_i}$ with $I_{f_i}$ for all $i\in \mathcal{I}_2;$ and $\gamma_2:=\gamma\gamma_1^{-1}.$ Also, we have
\begin{align*}
P_{\textbf{f},\textbf{e};g}(\mathbb{A}_F)=\prod_{i\in \mathcal{I}_1}\GL(f_i,\mathbb{A}_F)\times \prod_{j\in \mathcal{I}_2}N_j(\mathbb{A}_F)M_j(\mathbb{A}_F)K_j,
\end{align*}
where $\GL(f_je_j,\mathbb{A}_F)=N_j(\mathbb{A}_F)M_j(\mathbb{A}_F)K_j$ is the Iwasawa decomposition for the group  $\GL(f_je_j,\mathbb{A}_F),$ where $M_j$ is the product of $e_j$ copies of $\GL(f_j).$ Also, there exists a compact subgroup $K_{\textbf{f},\textbf{e}}$ of $G(\mathbb{A}_F)$ such that $G(\mathbb{A}_F)=P_{\textbf{f},\textbf{e}(\mathbb{A}_F)}K_{\textbf{f},\textbf{e}}.$ Hence, one can write at least formally,
\begin{align*}
&\int_{Z_G(\mathbb{A}_F)\backslash G(\mathbb{A}_F)}\sum_{i=1}^g\sum_{\alpha_{i}\in E_{i}^{\times}/F^{\times}-\{1\}}\varphi(x^{-1}\gamma x)f(\lambda_{\textbf{f},\textbf{e};g}^{-1}x,s)dx\\
=&\int_{K_{\textbf{f},\textbf{e};g}}\int_{Z_G(\mathbb{A}_F)\backslash P_{\textbf{f},\textbf{e};g}(\mathbb{A}_F)}\sum_{i=1}^g\sum_{\alpha_{i}\in E_{i}^{\times}/F^{\times}-\{1\}}\varphi(k^{-1}x^{-1}\gamma xk)f(\lambda_{\textbf{f},\textbf{e};g}^{-1}xk,s)dxdk.
\end{align*}
Denote by $J_1(s)$ the above distribution. Then we deduce that 
\begin{align*}
J_1(s)=&\prod_{i\in\mathcal{I}_1}\int_{G_{\alpha_{i}}(\mathbb{A}_F)\backslash \GL(f_i,\mathbb{A}_F)}\sum_{\alpha_{i}}\int_{G_{\alpha_{i}}(\mathbb{A}_F)}J_2(s;t_i,x_i: i\in\mathcal{I}_1)dt_idx_1^{i},
\end{align*}
where the function $J_2(s)=J_2(s;t_i,x_i: i\in\mathcal{I}_1)$ is defined by 
\begin{align*}
J_2(s)=&\int_{K}\prod_{j\in\mathcal{I}_2}\int_{X}\sum_{\substack{\alpha_{j,k}\\ 1\leq k\leq e_j}}\int_{Y}\varphi(k^{-1}y^{-1}x_2^{-1}x_1^{-1}\gamma x_1x_2yk)f(\lambda_{\textbf{f},\textbf{e};g}^{-1}x_1x_2yk,s)dydx_2dk,
\end{align*}
where $K=K_{\textbf{f},\textbf{e};g},$ $X=G_{\alpha_{j}}(\mathbb{A}_F)^{e_j}\backslash \GL(f_j,\mathbb{A}_F)^{e_j},$ and $Y=N_j(\mathbb{A}_F)G_{\alpha_{j}}(\mathbb{A}_F)^{e_j}K_j.$
\medskip 

Then by Tate's thesis, we deduce that 
\begin{align*}
\int_{Y}\varphi(k^{-1}y^{-1}x^{-1}\gamma xyk)f(\lambda_{\textbf{f},\textbf{e};g}^{-1}xyk,s)dy\cdot \prod_{i\in \mathcal{I}_1}\tau(t_i)^{-1}|\det t_i|^{-s} 
\end{align*}
is equal to the product of $Q_{E_j}(s;\varphi, \Phi,t_i,x_i: i\in\mathcal{I}_1)\widetilde{\Phi}(x;\eta_{f_i}t_i: i\in\mathcal{I}_1)$ and $\prod_{k=1}^{e_j}\Lambda_{E_{j}}[k]\left(s,\tau\circ N_{E_j/F}\right),$  where $Q_{E_j}(s;\varphi, \Phi,t_i,x_i: i\in\mathcal{I}_1)$ is an entire function and $\widetilde{\Phi}(x;\eta_{f_i}t_i: i\in\mathcal{I}_1)$ is Schwartz. Moreover, the function 
\begin{align*}
x_2\mapsto\sum_{\substack{\alpha_{j,k}\\ 1\leq k\leq e_j}}\int_{N_j(\mathbb{A}_F)K_j}\varphi(k^{-1}y^{-1}x_2^{-1}x_1^{-1}\gamma x_1x_2yk)f(\lambda_{\textbf{f},\textbf{e};g}^{-1}x_1x_2y_1y_2k,s)dy_2
\end{align*}
has compact support in $G_{\alpha_{j}}(\mathbb{A}_F)^{e_j}\backslash \GL(f_j,\mathbb{A}_F)^{e_j};$ and the function
\begin{align*}
x\mapsto\sum_{i\in \mathcal{I}_1}\sum_{\alpha_i}\sum_{\gamma_2}\int_{Y}\varphi(k^{-1}y^{-1}x_2^{-1}x_1^{-1}\gamma x_1x_2yk)f(\lambda_{\textbf{f},\textbf{e};g}^{-1}x_1x_2yk,s)dy
\end{align*}
is compactly supported as well. Hence $J_1(s)$ converges absolutely when $\Re(s)>1,$ and has a meromorphic continuation to the whole $s$-plane, with only possible poles at $s\in\{1, 1/2, \cdots, 1/n\}.$ Then Proposition \ref{gel} follows from the similar combinatorial analysis in the proof of Proposition \ref{re}.
\end{proof}

Combining \eqref{009} and Proposition \ref{gel} we then obtain

\begin{thmx}\label{reg ell}
	Let notation be as before. Let $\Re(s)>1.$ Then 
	\begin{align*}
	I_{\Geo,\Reg}(s)=\sum_{g=1}^n\sum_{\substack{\textbf{f},\ \textbf{e}\in \mathbb{Z}_{\geq 1}^g\\ \langle\textbf{f},\textbf{e}\rangle=n}}\frac{1}{C_{\textbf{f},\textbf{e}}}\prod_{i=1}^g\frac{1}{f_i^{e_i}}\sum_{[E_{i}:F]=f_i}Q_{E_{i}}(s)\prod_{j=1}^{e_i}\Lambda_{E_{i}}[j]\left(s,\tau\circ N_{E_i/F}\right),
	\end{align*}
	where for each $i,$ the innermost summation is taken over only finitely many fields $E_{i}$'s, depending implicitly only on the test function $\varphi;$ and each $Q_{E_{i}}(s)$ is an entire function. Moreover, the right hand side of $I_{\Geo,\Reg}(s)$ in the above formula gives a meromorphic continuation of $I_{\Geo,\Reg}(s)$ to the whole $s$-plane, with only possible poles at $s\in\{1, 1/2, \cdots, 1/n\}.$ 
\end{thmx}

	\section{Mirabolic Fourier Expansion of $\K_{\infty}(s)$}\label{sec3}
	Take a test function $\varphi$ as before, then by the definition of $E_P(x,\Phi;s)$ we have
	\begin{align*}
	I_{\infty}(s)=I_{\infty}^{\varphi}(s)=-\int_{G(F)Z_G(\mathbb{A}_F)\setminus G(\mathbb{A}_F)}\K_{\infty}(x,x)\sum_{\gamma\in P(F)\setminus G(F)}f(\gamma x,s) dx.
	\end{align*} 
	where $\K_{\infty}(x,y)=\K_{Eis}(x,y)+\K_{\Res}(x,y)$ is left $N(F)$-invariant. Then 
	\begin{equation}\label{sp1}
	I_{\infty}(s)=-\int_{Z_G(\mathbb{A}_F)P(F)\setminus G(\mathbb{A}_F)}\K_{\infty}(x,x)f(x,s) dx.
	\end{equation} 
	Now we proceed to compute \eqref{sp1} by considering the Fourier expansion of $\K_{\infty}(x,y).$

	\subsection{Mirabolic Fourier Expansions of Automorphic Forms}\label{4.1} Fourier expansions of automorphic forms of $\GL_n$ are well known (see \cite{PP75}). Following the idea of Piatetski-Shapiro in \cite{PP75}, we give a new form of Fourier expansions of weak automorphic forms in terms of generalized mirabolic subgroups, via which a further decomposition of $I_{\infty}(s)$ is obtained. Here we call a function $f\in \mathcal{C}\left(G(\mathbb{A}_F)\right)$ a weak automorphic form if it is slowly increasing on $G(\mathbb{A}_F),$ right $K$-finite and $P_0(F)$-invariant, where $P_0$ is the mirabolic subgroup of $G=GL_n.$
	
	Fix an integer $n\geq2.$ The maximal unipotent subgroup of $G(\mathbb{A}_F),$ denoted by $N(\mathbb{A}_F),$ is defined to be the set of all $n\times n$ upper triangular matrices in $G(\mathbb{A}_F)$ with ones on the diagonal and arbitrary entries above the diagonal. Let $\psi_{F/\mathbb{Q}}(\cdot)=e^{2\pi i\Tr_{F/\mathbb{Q}}(\cdot)}$ be the standard additive character, then for any $\alpha=(\alpha_1,\cdots,\alpha_{n-1})\in F^{n-1},$ define a character $\psi_{\alpha}:$ $N(\mathbb{A}_F)\rightarrow \mathbb{C}$ by
	\begin{align*}
	\psi_{\alpha}(u)=\prod_{i=1}^{n-1}\psi_{F/\mathbb{Q}}\left(\alpha_iu_{i,i+1}\right),\quad \forall\ u=(u_{i,j})_{n\times n}\in N(\mathbb{A}_F).
	\end{align*}  
	Write $\psi_k=\psi_{(0,\cdots,0,1,\cdots,1)}$ (where the first $n-k$ components are 0 and the remaining $k$ components are $1$) and $\theta=\psi_{(1,\cdots,1)},$ the standard generic character used to define Whittaker functions. 
	
	For $1\leq k\leq n-1,$ let $B_{n-k}$ be the standard Borel subgroup (i.e. the subgroup consisting of nonsingular upper triangular matrices) of $\GL_{n-k};$ let $N_{n-k}$ be the unipotent radical of $B_{n-k}.$ For any $i,$ $j\in\mathbb{N},$ let $M_{i\times j}$ be the additive group scheme of $i\times j$-matrices. Define the unipotent radicals 
	$$
	N_{(k,1,\cdots,1)}=\bigg\{\begin{pmatrix}
	I_{k} & B\\
	& D\\
	\end{pmatrix}:\ B\in M_{k\times(n-k)},\ D\in N_{n-k}\bigg\},\ \text{$1\leq k\leq n-1$ }.
	$$
	For $1\leq k\leq n-1,$ set the generalized mirabolic subgroups 
	\begin{align*}
	R_{k}&=\bigg\{\left(
	\begin{array}{cc}
	A&C\\
	0&B
	\end{array}
	\right):\ A\in GL_k,\ C\in M_{k\times(n-k)},\ B\in N_{n-k}
	\bigg\}.
	\end{align*}
	For $2\leq k\leq n-1,$ define subgroups of $R_k$ by
	\begin{align*}
	R_{k}^0&=\Bigg\{\left(
	\begin{array}{ccc}
	A&B'&C\\
	0&a&D\\
	0&0&B
	\end{array}
	\right):\ \begin{pmatrix}
	A&B'\\
	&a
	\end{pmatrix}\in GL_{k},\  \begin{pmatrix}
	C\\
	D
	\end{pmatrix}\in M_{k\times(n-k)},\ B\in N_{n-k}
	\Bigg\}.
	\end{align*} 
	Also we define $R_0=R_1^0=N_{(0,1,\cdots,1)}:=N_{(1,1,\cdots,1)}$ to be the unipotent radical of the standard Borel subgroup of $\GL_n.$
	\begin{prop}[Mirabolic Fourier Expansion]\label{Fourier}
		Let $h$ be a continuous function on $P_0(F)\setminus G(\mathbb{A}_F).$ Then we have 
		\begin{equation}\label{fourier}
		h(x)=\sum_{k=1}^n\sum_{\delta_k\in R_{k-1}\backslash R_{n-1}}\int_{N_{(k-1,1,\cdots,1)}(F)\backslash N_{(k-1,1,\cdots,1)}(\mathbb{A}_F)}h(u\delta_kx)\psi_{n-k}(u)du
		\end{equation}
		if the right hand side converges absolutely and locally uniformly.
	\end{prop}
	\begin{proof}
		For $1\leq k\leq n,$ we define
		\begin{align*}
		M_k^0&=\bigg\{\left(
		\begin{array}{cc}
		A&0\\
		0&I_{n-k}
		\end{array}
		\right):\ A\in GL(k,F) 
		\bigg\},\\
		M_k^{\infty}&=\bigg\{\left(
		\begin{array}{ccc}
		A&b&0\\
		0&c&0\\
		0&0&I_{n-k}
		\end{array}
		\right):\ A\in GL(k-1,F),\ b\in F^{k-1},\ c\in F^{\times}
		\bigg\},
		\end{align*}
		where $I_{n-k}$ is the unit matrix of dimension $n-k.$ For the sake of simplicity, write
		$$
		J_kh(x)=\sum_{\delta_k\in R_{k-1}\backslash R_{n-1}}\int_{N_{(k-1,1,\cdots,1)}(F)\backslash N_{(k-1,1,\cdots,1)}(\mathbb{A}_F)}h(n\delta_kx)\psi_{n-k}(n)dn,
		$$
		where $\psi_0\equiv1.$ Let $N_1\subset N$ be the subgroup consisting of elements of the form
		$$
		\mathfrak{n}^{(1)}(u_n)=\begin{pmatrix}
		1 &  & & u_{1,n}\\
		& 1&& \vdots\\
		&&\ddots& u_{n-1,n}\\
		&  && 1
		\end{pmatrix},\ \text{where $u_n=(u_{1,n},\cdots,u_{n-1,n})\in \mathbb{A}_F^{n-1}.$}
		$$
		Since $N_1$ is abelian, $h$ has the Fourier expansion with respect to $N_1:$
		$$
		h(x)=\sum_{\alpha^{(1)}=(\alpha_{1,n},\cdots,\alpha_{n-1,n})\in F^{n-1}}\int_{N_1(F)\setminus N_1(\mathbb{A}_F)}h(\mathfrak{n}^{(1)}(u_n)x)\prod_{i=1}^{n-1}\psi_{F/\mathbb{Q}}\left(\alpha_{i,n}u_{i,n}\right)du_n.
		$$
		Denote the inner integral by $W_{\alpha^{(1)}}^1h(x).$ Since $h$ is $P_{0}(F)$-invariant, then
		\begin{align*}
		W_{(0,0,\cdots,\alpha_{n-1,n})}^1h(\gamma x)&=\int_{N_1(F)\setminus N_1(\mathbb{A}_F)}h(\mathfrak{n}^{(1)}(u_n)\gamma x)\psi_{F/\mathbb{Q}}\left(\alpha_{n-1,n}u_{n-1,n}\right)du_n\\
		&=\int_{N_1(F)\setminus N_1(\mathbb{A}_F)}h\left(\gamma^{-1}\mathfrak{n}^{(1)}(u_n)\gamma x\right)\psi_{F/\mathbb{Q}}\left(\alpha_{n-1,n}u_{n-1,n}\right)du_n,
		\end{align*}
		for any $\gamma=\diag(A,1),$ where $A\in GL(n-1,F).$ An easy computation shows that $\gamma^{-1}\mathfrak{n}^{(1)}(u_n)\gamma=\mathfrak{n}^{(1)}(u_n'),$ where $u_n'=A^{-1}u_n.$ Write $A=(a_{i,j})_{(n-1)\times (n-1)},$ then $u_{n-1,n}=a_{n-1,1}u_{1,n}'+\cdots+a_{n-1,n-1}u_{n-1,n}'.$ This implies that for any such $\gamma,$
		$$
		W_{(0,0,\cdots,\alpha_{n-1,n})}^1h(\gamma x)=W_{(a_{n-1,1}\alpha_{n-1,n},a_{n-1,2}\alpha_{n-1,n},\cdots,a_{n-1,n-1}\alpha_{n-1,n})}^1h(x).
		$$ 
		Hence one has
		\begin{equation}\label{9.1}
		h(x)=\sum_{\substack{\gamma_{n-1}\in M^{\infty}_{n-1}\backslash M^0_{n-1}\\ \alpha_{n-1,n}\in F^{\times}}}W_{(0,0,\cdots,\alpha_{n-1,n})}^1h(\gamma_{n-1} x)+W_{(0,0,\cdots,0)}^1h(x).
		\end{equation}
		 For any $\alpha_{n-1,n}\in F^{\times},$ let $\mathfrak{a}_{n-1,n}=\diag(1,\cdots,1,\alpha_{n-1,n},1)\in R_{n-1}^0\backslash R_{n-1}.$ Since $h$ is left invariant by $\mathfrak{a}_{n-1,n},$ $M^{\infty}_{n-1}\backslash M^0_{n-1}=R^{0}_{n-1}\backslash R_{n-1}$ and
		\begin{align*}
		\begin{pmatrix}
		1 &  & & u_{1,n}\\
		& 1&& \vdots\\
		&&\ddots& \alpha_{n-1,n}u_{n-1,n}\\
		&  && 1
		\end{pmatrix}=\mathfrak{a}_{n-1,n}\begin{pmatrix}
		1 &  & & u_{1,n}\\
		& 1&& \vdots\\
		&&\ddots& u_{n-1,n}\\
		&  && 1
		\end{pmatrix}\mathfrak{a}_{n-1,n}^{-1},
		\end{align*}
		\begin{align*}
		W_{(0,0,\cdots,\alpha_{n-1,n})}^1h(\gamma_{n-1} x)=\int_{N_{(n-1,1)}(F)\backslash N_{(n-1,1)}(\mathbb{A}_F)}h(n\mathfrak{a}_{n-1,n}^{-1}\gamma_{n-1}x)\psi(u_{n-1,n})dn.
		\end{align*}
		Note that $W_{(0,0,\cdots,0)}^1h(x)=J_nh(x)$ and $R_{n-1}^0=R_{n-2},$ \eqref{9.1} then becomes
		\begin{equation}\label{9.11}
		h(x)=\sum_{\substack{\delta_{n-1}\in R_{n-2}\backslash R_{n-1}}}\int_{N_{(n-1,1)}(F)\backslash N_{(n-1,1)}(\mathbb{A}_F)}h(n\delta_{n-1}x)\psi_1(n)dn+J_nh(x).
		\end{equation}
		Let  $N_2\subset N$ be the subgroup consisting of elements of the form
		$$
		\mathfrak{n}^{(2)}(u_{n-1})=\begin{pmatrix}
		1 &  & & u_{1,n-1}\\
		& 1&& \vdots\\
		&&\ddots& u_{n-2,n-1}\\
		&  && 1\\
		&&&&1
		\end{pmatrix},\ u_{n-1}=(u_{1,n},\cdots,u_{n-2,n-1})\in \mathbb{A}_F^{n-2}.
		$$
		Since $N_2(F)\subset M^{\infty}_{n-1},$ $W_{(0,0,\cdots,\alpha_{n-1,n})}^1h(u_{n-1}x)=W_{(0,0,\cdots,\alpha_{n-1,n})}^1h(x),$ $\forall$ $u_{n-1}\in F^{n-1}.$ Then we have the Fourier expansion of $W_{(0,0,\cdots,\alpha_{n-1,n})}^1h(x)$ with respect to $N_2:$
		$$
		W_{(0,0,\cdots,\alpha_{n-1,n})}^1h(x)=\sum_{\alpha^{(2)}=(\alpha_{1,n-1},\cdots,\alpha_{n-2,n-1})\in F^{n-2}}W_{\alpha^{(2)}}^2h(x),
		$$
		where $W_{\alpha^{(2)}}^2h(x)=W_{(\alpha_{1,n-1},\alpha_{2,n-1},\cdots,\alpha_{n-2,n-1})}^2h(x)$ is defined to be
		$$
		\int_{N_2(F)\setminus N_2(\mathbb{A}_F)}W_{(0,0,\cdots,\alpha_{n-1,n})}^1h\left(\mathfrak{n}^{(2)}(u_{n-1}) x\right)\prod_{i=1}^{n-2}\psi_{F/\mathbb{Q}}\left(\alpha_{i,n-1}u_{i,n-1}\right)du_{n-1}.
		$$
		Likewise, we obtain
		$$
		W_{(0,0,\cdots,\alpha_{n-1,n})}^1h(x)=\sum_{\substack{\gamma_{n-2}\in M^{\infty}_{n-2}\backslash M^0_{n-2}\\ \alpha_{n-2,n-2}\in F^{\times}}}W_{(0,0,\cdots,\alpha_{n-2,n-1})}^2h(\gamma_{n-2} x)+W_{(0,\cdots,0)}^2h(x),
		$$
		where, by a direct computation, one has
		\begin{align*}
		&W_{(0,\cdots,0)}^2h(x)=\int_{N_{(n-2,1,1)}(F)\backslash N_{(n-2,1,1)}(\mathbb{A}_F)}h(nx)\psi(u_{n-2,n-1})dn,\\
		&W_{(0,0,\cdots,\alpha_{n-2,n-1})}^2h(\gamma_{n-2} x)=\int_{[N_{(n-2,1,1)}]}h(n\mathfrak{a}_{n-2,n-1}^{-1}\gamma_{n-2}x)\psi(u_{n-2,n-1})dn.
		\end{align*}
		Moreover, noting that $R_{n-2}^0=R_{n-3},$ then substituting the above computation into \eqref{9.1} implies that 
		\begin{align*}
		h(x)=\sum_{\substack{\delta_{n-2}\in R_{n-3}\backslash R_{n-1}}}\int_{[N_{(n-2,1,1)}]}h(n\delta_{n-2}x)\psi_{2}(n)dn+J_{n-1}h(x)+J_nh(x).
		\end{align*}
	Then clearly the expansion \eqref{fourier} follows from repeating this process $n-2$ more times.
	\end{proof}

	\subsection{Decomposition of $I_{\infty}(s)$}\label{sec3.2}
	Applying Proposition \ref{Fourier} to the kernel function $\K(x,y)$ viewed as a function of $x,$ we thus obtain a formal decomposition of the distribution $I_{\infty}(s)$ when $\Re(s)>1.$ In fact, by the spectral decomposition of the kernel function $\K_{\infty}(x,y)$ (see \cite{Art79}), one has
\begin{align*}
\K_{\infty}(x,y)=\sum_{\chi\in \mathfrak{S}_{\Eis}(G)}\sum_Pn(A)^{-1}\left(\frac{1}{2\pi i}\right)^{\dim (A/Z_G)}\int_{i\mathfrak{a}^G}\sum_{\phi\in \mathfrak{B}_{P,\chi}}\mathfrak{E}_{P,\chi}(x,y;\phi,\lambda)d\lambda,
\end{align*}
where $\mathfrak{E}_{P,\chi}(x,y;\phi,\lambda)=E(x,\mathcal{I}_P(\lambda_{\xi})\phi,\lambda)\overline{E(y,\phi,\lambda)},$ and the integrals on the right hand side converges absolutely. Since for any $m\in M_P(F),$ we have
\begin{align*}
E(my,\phi,\lambda)=\sum_{\delta\in P(F)\setminus G(F)}\phi(\delta my)e^{(\rho+\lambda)H_P(\delta my)}=E(y,\phi,\lambda).
\end{align*}
Hence $\K_{\infty}(x,y)$ is $M_P(F)$-invariant with respect to both variables. Then we can apply Proposition \ref{Fourier} with respect to the first variable of $\K_{\infty}(x,y)$ to get, at least formally, that 
	\begin{equation}\label{8}
	I_{\infty}(s)=\sum_{k=1}^n\int_{X_k}\int_{[N_k^*]}\int_{[N_k']}\K_{\infty}(n^*n_1x,x)\theta(n_1)dn_1dn^*f(x,s)dx,
	\end{equation}
	where $X_k=Z_G(\mathbb{A}_F)R_{k-1}(F)\backslash G(\mathbb{A}_F),$ and $N_k'=N_{(k,1,\cdots,1)}$ and 
	$$
	N_k^*=\bigg\{\begin{pmatrix}
	I_{k-1} & C&\\
	& 1&\\
	&&I_{n-k}
	\end{pmatrix}:\ C\in \mathbb{G}_a^{k-1}\bigg\}.
	$$
	Moreover, when both sides of \eqref{8} converge absolutely, the identity is rigorous. 
	\medskip
	
	However, there are usually convergence problem with the decomposition \eqref{8}. In fact, for $1\leq k\leq n,$ if we write $I_{\infty}^{(k)}(s)$ for the above (formal) integral, namely,
	\begin{align*}
	I_{\infty}^{(k)}(s)=\int_{Z_G(\mathbb{A}_F)R_{k-1}(F)\backslash G(\mathbb{A}_F)}\int_{[N_k^*]}\int_{[N_k']}\K_{\infty}(n^*n_1x,x)\theta(n_1)dn_1dn^*f(x,s)dx.
	\end{align*}
Then in fact $I_{\infty}^{(k)}(s)$ might diverge when $2\leq k\leq n,$ if $\varphi$ does not support in elliptic regular sets. Nevertheless, we can show $I_{\infty}^{(1)}(s)$ actually converges absolutely when $\Re(s)>1,$ and thus it defines a holomorphic function therein.
\medskip

To start with, the first observation is that one can replace $\K_{\infty}$ by $\K$ in the definition of $I_{\infty}^{(k)}(s),$ $2\leq k\leq n.$ Denote by $V_k'=\diag(I_k, N_{n-k})$ and
\begin{align*}
V_k=\Bigg\{\begin{pmatrix}
I_{k-1} & B&C\\
& 1&D\\
&&I_{n-k}
\end{pmatrix}:\ B\in \mathbb{G}_a^{k-1},\ \begin{pmatrix}
C\\
D
\end{pmatrix}\in {M_{k\times (n-k)}}\Bigg\}.
\end{align*}
Then for any function $\phi$ on $G(\mathbb{A}_F)$ one has, for any $x\in G(\mathbb{A}_F),$ that
\begin{align*}
\int_{[N_k^*]}\int_{[N_k']}\phi(n^*nx)\theta(n)dndn^*=\int_{[V_k']}\int_{[V_k]}\phi(uu'x)du\theta(u')du'.
\end{align*}	
Note that $V_k$ is the unipotent radical of the standard parabolic subgroup of type $(k-1,1,n-k),$ so one has
\begin{equation}\label{cusp}
\int_{[N_k^*]}\int_{[N_k']}\phi(n^*nx)\theta(n)dndn^*=0,\ \forall\ \phi\in \mathcal{A}_0\left(G(F)\backslash G(\mathbb{A}_F)\right).
\end{equation}
Then by \eqref{cusp} and the discrete spectral decomposition \eqref{ker_0} one has
\begin{align*}
\int_{[N_k^*]}\int_{[N_k']}\K_{0}(n^*n_1x,x)\theta(n_1)dn_1dn_2dn^*=0.
\end{align*}
Since $\K=\K_{\infty}+\K_0,$ then for $2\leq k\leq n,$ one sees that (at least formally)
\begin{align*}
I_{\infty}^{(k)}(s)=&\int_{Z_G(\mathbb{A}_F)R_{k-1}(F)\backslash G(\mathbb{A}_F)}\int_{[N_k^*]}\int_{[N_k']}\K(n^*n_1x,x)\theta(n_1)dn_1dn^*\cdot f(x,s)dx\\
=&\int_{Z_G(\mathbb{A}_F)R_{k-1}(F)\backslash G(\mathbb{A}_F)}\int_{[V_k']}\int_{[V_k]}\K(uu'x,x)du\theta(u')du'\cdot f(x,s)dx.
\end{align*}

\medskip

Let $n\geq 2.$ Recall that $\mathfrak{G}$ is the union of $\left(Z_G(F)\backslash Q_k(F)\right)^{P_0(F)},$ $1\leq k<n.$ Let 
\begin{align*}
&\K_{\infty,\Sin}^{(n)}(x,y)=\int_{N_P(F)\backslash N_P(\mathbb{A}_F)}\sum_{\gamma\in \mathfrak{G}}\varphi(x^{-1}u^{-1}\gamma y)du,\\ 
&\K_{\infty,\Reg}^{(n)}(x,y)=\int_{N_P(F)\backslash N_P(\mathbb{A}_F)}\K(ux,y)du-\K_{\infty,\Sin}^{(n)}(x,y);\\
&\K_{\infty}^{(k)}(x,y)=\sum_{\delta_k\in R_{k-1}(F)\backslash P_{n-1}(F)}\int_{[V_k']}\int_{[V_k]}\K(uu'\delta_kx,y)du\theta(u')du',
\end{align*}
where $2\leq k\leq n.$ Define
\begin{equation}\label{infty}
\K_{\Sin}(x,y):=\K_{\Geo,\Sin}(x,y)-\K_{\infty,\Sin}^{(n)}(x,y)-\sum_{k=2}^{n-1}\K_{\infty}^{(k)}(x,y).
\end{equation}

Let $\Re(s)>1.$ Correspondingly, we define the distributions by
\begin{align*}
&I_{\Geo,\Reg}(s,\tau)=\int_{Z_G(\mathbb{A}_F)P_0(F)\backslash G(\mathbb{A}_F)}\K_{\Geo,\Reg}^{(n)}(x,x)\cdot f_{\tau}(x,s)dx;\\
&I_{\infty,\Reg}(s,\tau)=\int_{Z_G(\mathbb{A}_F)P_0(F)\backslash G(\mathbb{A}_F)}\K_{\infty,\Reg}^{(n)}(x,x)\cdot f_{\tau}(x,s)dx;\\
&I_{\Sin}(s,\tau)=\int_{Z_G(\mathbb{A}_F)P_0(F)\backslash G(\mathbb{A}_F)}\K_{\Sin}(x,x)\cdot f_{\tau}(x,s)dx;\\
&I_{\infty}^{(1)}(s,\tau)=\int_{Z_G(\mathbb{A}_F)P_0(F)\backslash G(\mathbb{A}_F)}\K_{\infty}^{(1)}(x,x)\cdot f_{\tau}(x,s)dx;\\
\end{align*}

In the following sections we will deal with $I_{\infty,\Reg}(s,\tau)$ and $I_{\infty}^{(1)}(s,\tau),$ and the analytic behavior of $I_{\Sin}(s,\tau)$ would follow from spectral expansion and functional equation. As we will see, $I_{\infty,\Reg}(s,\tau)$ will be handled by Langlands-Shahidi's method after applying some geometric auxiliary results (see Section \ref{6.1}); and $I_{\infty}^{(1)}(s)$ can be reduced to an infinite sum of Rankin-Selberg convolutions of irreducible generic non-cuspidal representations of $\GL(n,\mathbb{A}_F)$ (see Section \ref{6sec}). We also obtain a  meromorphic continuation of $I_{\infty}^{(1)}(s,\tau)$ in Section \ref{6.2.} and Section \ref{3.}.  Hence the expansion \eqref{m} is well defined on both sides for $\Re(s)>1,$ and can be regarded as an identity between their continuations when $s\in\mathbb{C}$ is arbitrary and $\tau$ is such that $\tau^k\neq 1,$ $\forall$ $1\leq k\leq n.$

	\section{Contributions from $I_{\infty,\Reg}(s,\tau)$}\label{4}
	Now we start with handling the last term $I_{\infty,\Reg}(s,\tau),$ since the approach here applies to part of the computation of $I_{\infty}^{(k)}(s),$ $2\leq k\leq n-1,$ as well. Recall that 
	\begin{align*}
	I_{\infty}^{(n)}(s)=\int_{Z_G(\mathbb{A}_F)R_{n-1}(F)\backslash G(\mathbb{A}_F)}\int_{N_P(F)\backslash N_P(\mathbb{A}_F)}\K_{\infty}(nx,x)dnf(x,s)dx.
	\end{align*}
	Since $\K_{0}(x,y)$ by definition has no constant term with respect to $x$ or $y;$ we can replace $\K_{\infty}$ by $\K=\K_{0}+\K_{\infty}$ to rewrite $I_{\infty}^{(n)}(s)$ as 
	\begin{align*}
I_{\infty,\Reg}(s,\tau)&=\int_{X_n}\int_{N_P(F)\backslash N_P(\mathbb{A}_F)}\K_{\infty,\Reg}(nx,x)dnf(x,s)dx\\
	&=\int_{X_{n}}\int_{[N_P]}\sum_{\gamma\in Z_G(F)\backslash G(F)-\mathfrak{S}}\varphi(x^{-1}n^{-1}\gamma x)dnf(x,s)dx,
	\end{align*}
where $X_n=Z_G(\mathbb{A}_F)R_{n-1}(F)\backslash G(\mathbb{A}_F).$ To simplify $I_{\infty,\Reg}(s,\tau),$ write $Z_G(F)\backslash G(F)=\sqcup \mathcal{C}$ as a disjoint union of $G(F)$-conjugacy classes, and further decompose each class $\mathcal{C}$ into a disjoint union of $P(F)$-conjugacy classes. Then we will find representatives of these $P(F)$-conjugacy classes explicitly. So eventually one can get rid of the factor $R_{n-1}(F)=Z_G(F)\backslash P(F)$ in the domain; moreover, one can now apply Iwasawa decomposition to the domain $Z_G(\mathbb{A}_F)\backslash G(\mathbb{A}_F)$ to compute this integral. 
	
\subsection{$P(F)$-conjugacy Classes}\label{sec4.1}

	For any $G(F)$-conjugacy class $\mathcal{C},$ denote by $\mathcal{C}_{r.e.}^{P(F)}$ the component $\mathcal{C}_0$ given in \eqref{9'} if $\mathcal{C}$ is regular, and take $\mathcal{C}_{r.e.}^{P(F)}$ to be an empty set if $\mathcal{C}$ is irregular.  Since $\mathcal{C}_{r.e.}^{P(F)}$ does not intersect any standard maximal parabolic subgroups and is nontrivial only when $\mathcal{C}$ is regular, for convenience, we call $\mathcal{C}_{r.e.}^{P(F)}$ the regular elliptic component of $\mathcal{C},$ despite of the fact that it might not be elliptic. 
	
	Let $\mathfrak{C}_{r.e.}^{P(F)}$ be the union of regular elliptic components of all $G(F)$-conjugacy classes in $G(F).$ Then $\mathfrak{C}_{r.e.}^{P(F)}$ is a disjoint union of $P(F)$-conjugacy classes in $G(F)$ by Proposition \ref{reg}. Moreover,  Proposition \ref{irre} and Proposition \ref{reg} give a decomposition of $G(F)$ as $P(F)$-conjugacy classes
    \begin{equation}\label{14.2}
    G(F)=\mathfrak{C}_{r.e.}^{P(F)}\coprod \bigcup _{k=1}^{n-1}Q_k(F)^{P(F)}.
    \end{equation}
    
    For any $2\leq k\leq n,$ let $P_k$ be the standard maximal parabolic subgroup of $\GL_k$ of type $(k-1,1).$ In the following, we identify $P_k$ with $\diag(P_k, I_{n-k})$ when view it as a subgroup of $G=GL_n.$ Write $W_k$ the Weyl group of $\GL_k$ with respect to the standard Borel subgroup and its Levi component. Let $\Delta_{k}$ be the set of simple roots. Let $S_k$ be the subgroup of symmetric groups $S_n$ generated by permutations among $\{1,2,\cdots,k\}\subset\{1,2,\cdots,n\}.$ For any $\alpha\in\Delta_k,$ via the isomorphisms and natural inclusion $W_k\xrightarrow{\sim} S_k\hookrightarrow S_n\xrightarrow{\sim} W_n,$ we identify it with its natural extension in $\Delta,$ the set of simple roots for $G(F)=GL_n(F).$ Henceforth, write $\Delta_k=\{\alpha_{i,i+1}:\ 1\leq i\leq k-1\},$ and for each simple root $\alpha_{i,i+1},$ write $w^k_{i,i+1}$ for the corresponding simple reflection and identify it with $w_{i,i+1}$ by the natural embedding. 
    
    Denote by $\mathfrak{C}_{r.e.}^{P_k(F)}$ the union of regular elliptic components of all $G(F)$-conjugacy classes in $\GL_{k}(F),$ $2\leq k\leq n.$ Let $\mathcal{R}_k$ be a set consisting of exactly all representatives of the $P_{k}(F)$-conjugacy classes $\mathfrak{C}_{r.e.}^{P_k(F)}.$ 
    
	To compute $I_{\infty}^{(n)}(s),$ an explicit choice of representatives of $\mathfrak{C}_{r.e.}^{P(F)}$ in Bruhat normal form needs to be taken. We will find at the end of this subsection that for each $2\leq k\leq n,$ there exists a particular choice of each $\mathcal{R}_k,$ such that $\mathcal{R}_{k}$ is determined by $\mathcal{R}_{k-1}.$ Thus a desired $\mathcal{R}_{n}$ could be obtained by induction. This will be illustrated in Proposition \ref{repre}, to prove which, we start with the following result to narrow the candidates of representatives for $\mathfrak{C}_{r.e.}^{P(F)}.$ 
	\begin{lemma}\label{122}
	Let notation be as before. Set $\mathcal{R}_P=\{w_{n-1}w_{n-2}\cdots w_{1}b:\ b\in B(F)\}.$ Denote by $\mathcal{R}_P^{P(F)}$ the union of $P(F)$-conjugacy classes of elements in $\mathcal{R}_P.$ Then one has
	\begin{equation}\label{rep}
	\mathfrak{C}_{r.e.}^{P(F)}=\mathcal{R}_P^{P(F)}.
	\end{equation} 
	\end{lemma}

	\begin{proof}
	By Bruhat decomposition, one has
	\begin{align*}
	 G(F)=P(F)\coprod P(F)\cdot w_{n-1}\cdot P(F).
	\end{align*}
	For any $g_1\in P(F)$ and $g_2\in P(F)\cdot w_{n-1}\cdot P(F),$ since different Bruhat cells do not intersect, the $P(F)$-conjugacy class of $g_1$ does not intersect with that of $g_2.$ Also note that $P(F)$-conjugacy classes of $P(F)$ lie in $P(F),$ so they are not regular elliptic. Hence we reject all representatives in $P(F),$ and see clearly that $P(F)$-conjugacy classes in $\mathfrak{C}_{r.e.}^{P(F)}$ are represented by elements in $w_{n-1}P(F).$
	
	For any $g=w_{n-1}\begin{pmatrix}
	A_{n-1} & b\\
	&d_n
	\end{pmatrix}\in w_{n-1}P(F)\cap \mathfrak{C}_{r.e.}^{P(F)},$ by Bruhat decomposition, either $A_{n-1} \in P_{n-1}(F)$ or $A_{n-1} \in P_{n-1}(F)w_{n-2}P_{n-1}(F),$ where $P_{n-1}$ is the standard maximal parabolic subgroup of $\GL_{n-1}(F)$ of type $(n-2,1).$ If $A_{n-1} \in P_{n-1}(F),$ then $g\in Q_{n-2}(F)\subset \bigcup _{1\leq k\leq n-1}Q_k(F)^{P(F)}.$ Thus $g\notin \mathfrak{C}_{r.e.}^{P(F)}.$ Therefore, $A_{n-1} \in P_{n-1}(F)w_{n-2}P_{n-1}(F).$ For any $1\leq k\leq n-1,$ write $R_{k}^*$ the standard parabolic subgroup of $G=GL_n$ of type $(k,1,\cdots,1).$ So we can write 
	\begin{align*}
	g^{(0)}=g=w_{n-1}\begin{pmatrix}
	I_{n-2}&c& b_1\\
	&1&b_2\\
	&&1
	\end{pmatrix}w_{n-2}\begin{pmatrix}
	A_{n-2}&c_{n-2} & \\
	&d_{n-1}&\\
	&&d_n
	\end{pmatrix}\in w_{n-1}R_{n-1}^*(F),
	\end{align*}
	which is conjugate by $w_{n-2}\begin{pmatrix}
	A_{n-2}&c_{n-2} & \\
	&d_{n-1}&\\
	&&d_n
	\end{pmatrix}\in P(F)$ to 
	\begin{align*}
	g^{(1)}&=w_{n-2}\begin{pmatrix}
	A_{n-2}&c_{n-2} & \\
	&d_{n-1}&\\
	&&d_n
	\end{pmatrix}w_{n-1}\begin{pmatrix}
	I_{n-2}&c& b_1\\
	&1&b_2\\
	&&1
	\end{pmatrix}\\
	&=w_{n-2}w_{n-1}\begin{pmatrix}
	A_{n-2}& & c_{n-2}\\
	&d_{n}&\\
	&&d_{n-1}
	\end{pmatrix}\begin{pmatrix}
	I_{n-2}&c& b_1\\
	&1&b_2\\
	&&1
	\end{pmatrix}\in w_{n-2}w_{n-1}R_{n-2}^*(F).
	\end{align*}
	Again, apply Bruhat decomposition to $\GL_{n-2}(F)\hookrightarrow GL_{n}(F)$ to see either $A_{n-2} \in P_{n-2}(F)$ or $A_{n-2} \in P_{n-2}(F)w_{n-3}P_{n-2}(F),$ where $P_{n-2}$ is the standard maximal parabolic subgroup of $\GL_{n-2}(F)$ of type $(n-3,1).$ If $A_{n-2} \in P_{n-2}(F),$ then $g^{(1)}\in Q_{n-3}(F)\subset \bigcup _{1\leq k\leq n-1}Q_k(F)^{P(F)}.$ Thus $g^{(1)}\notin \mathfrak{C}_{r.e.}^{P(F)}.$ Therefore, $A_{n-2} \in P_{n-2}(F)w_{n-3}P_{n-2}(F).$ So we can write 
	\begin{align*}
	g^{(1)}=w_{n-2}w_{n-1}\begin{pmatrix}
	I_{n-3}&c'&c_{n-2}^{(1)} &b_1' \\
	&1&c_{n-2}^{(2)}&b_2'\\
	&&1&b_3'\\
	&&&1
	\end{pmatrix}w_{n-3}\begin{pmatrix}
	A_{n-3}&c_{n-3}&& \\
	&d_{n-2}&&\\
	&&d_{n}&\\
	&&&d_{n-1}
	\end{pmatrix},
	\end{align*}
	which is conjugate by $w_{n-3}\begin{pmatrix}
	A_{n-3}&c_{n-3}&& \\
	&d_{n-2}&&\\
	&&d_{n}&\\
	&&&d_{n-1}
	\end{pmatrix}\in P(F)$ to 
	\begin{align*}
	g^{(2)}&=w_{n-3}\begin{pmatrix}
	A_{n-3}&c_{n-3}&& \\
	&d_{n-2}&&\\
	&&d_{n}&\\
	&&&d_{n-1}
	\end{pmatrix}w_{n-2}w_{n-1}\begin{pmatrix}
	I_{n-3}&c'&c_{n-2}^{(1)} &b_1' \\
	&1&c_{n-2}^{(2)}&b_2'\\
	&&1&b_3'\\
	&&&1
	\end{pmatrix}\\
	&=w_{n-3}w_{n-2}w_{n-1}\begin{pmatrix}
	A_{n-3}&&&c_{n-3} \\
	&d_{n}&&\\
	&&d_{n-1}&\\
	&&&d_{n-2}
	\end{pmatrix}\begin{pmatrix}
	I_{n-3}&c'&c_{n-2}^{(1)} &b_1' \\
	&1&c_{n-2}^{(2)}&b_2'\\
	&&1&b_3'\\
	&&&1
	\end{pmatrix}.
	\end{align*}
	Clearly, $g^{(2)}\in w_{n-3}w_{n-2}w_{n-1}R_{n-3}^*(F).$ Continue this process inductively to see that $g$ is $P(F)$-conjugate to some element $g^{(n-2)}\in w_1w_2\cdots w_{n-1}R_1^*(F).$ 
	
	Therefore, $\mathfrak{C}_{r.e.}^{P(F)}\subseteq  \{\gamma^{P(F)}:\ \gamma\in w_1w_2\cdots w_{n-1}R_1^*(F)\}.$ So we have 
	\begin{align*}
	\{g^{-1}:\ g\in\mathfrak{C}_{r.e.}^{P(F)}\}&\subseteq  \{\gamma^{P(F)}:\ \gamma\in R_1^*(F)w_{n-1}\cdots w_2w_1\}\\
	&=\{\gamma^{P(F)}:\ \gamma\in w_{n-1}w_{n-2}\cdots w_1B(F)\},
	\end{align*}
	since $R_1^*(F)=B(F)\subseteq P(F).$ Denote by $\iota:$ $G(F)\xrightarrow{\sim} G(F),$ $g\mapsto g^{-1},$ the inversion isomorphism. Then $\mathfrak{C}_{r.e.}^{P(F)}$ is stable under $\iota,$ since $\bigcup _{1\leq k\leq n-1}Q_k(F)^{P(F)}$ is stable under $\iota.$ Hence, 
	$$
	\mathfrak{C}_{r.e.}^{P(F)}=\{g^{-1}:\ g\in\mathfrak{C}_{r.e.}^{P(F)}\}\subseteq\{\gamma^{P(F)}:\ \gamma\in w_{n-1}w_{n-2}\cdots w_1B(F)\}=\mathcal{R}_P^{P(F)}.
	$$
	Now we show that $\mathcal{R}_P^{P(F)}\cap \bigcup _{1\leq k\leq n-1}Q_k(F)^{P(F)}=\emptyset,$ which implies by \eqref{14.2} that $\mathcal{R}_P^{P(F)}\subseteq \mathfrak{C}_{r.e.}^{P(F)}.$ Therefore, $\mathfrak{C}_{r.e.}^{P(F)}=\mathcal{R}_P^{P(F)}.$
	
	Assume that $\mathcal{R}_P^{P(F)}\cap Q_k(F)^{P(F)}\neq \emptyset$ for some $1\leq k\leq n-1.$ If $k=n-1,$ then $Q_k(F)^{P(F)}=P(F).$ Then the assumption forces that $w_{n-1}w_{n-1}\cdots w_1\in P(F),$ which is obviously a contradiction. Thus we may assume that $1\leq k\leq n-2.$ Then by Bruhat decomposition, one has
	$$
	P(F)=\coprod_{w\in W_{I_{n-1}}}N(F)wB(F),\ \text{and}\ Q_k(F)=\coprod_{w'\in W_{I_{k}}}N(F)w'B(F). 
	$$
	For $w\in W_n,$ denote by $C(w)=B(F)wB(F),$ the Bruhat cell with respect to $w.$ Then the assumption $\mathcal{R}_P^{P(F)}\cap Q_k(F)^{P(F)}\neq \emptyset$ leads to that
	\begin{equation}\label{12}
	C(w)C(w_{n-1}w_{n-2}\cdots w_1)C(w^{-1})\cap C(w')\neq \emptyset.
	\end{equation}
	However, Lemma \ref{14.3} below shows that for any $1\leq k\leq n-2,$ any $(w,w')\in W_{I_{n-1}}\times W_{I_{k}},$ the intersection in the left hand side of \eqref{12} is always empty, which gives a contradiction and thus ends the proof. 
    \end{proof}
    
	\begin{lemma}\label{14.3}
    Let notation be as before, $1\leq k\leq n-2,$ then one has
    \begin{align*}
    C(w)C(w_{n-1}w_{n-2}\cdots w_1)C(w^{-1})\cap C(w')= \emptyset,\ \forall\ w\in W_{I_{n-1}},\ w'\in W_{I_{k}}.
    \end{align*}		
	\end{lemma}
	\begin{proof}
	Recall that for any $w\in W_n$ and $\alpha\in \Delta,$ we have (see \cite{Spr09}, Lemma 8.3.7)
	\begin{equation}\label{15}
	C(s_{\alpha})C(w)=
	\begin{cases}
	C(s_{\alpha}w)& \text{if $l(s_{\alpha}w)=l(w)+1$,}\\
	C(w)\sqcup C(s_{\alpha}w)& \text{if $l(s_{\alpha}w)=l(w)-1$},
	\end{cases}
	\end{equation}
	where $l:$ $W\rightarrow \mathbb{Z}$ is the length function. Also, a similar computation shows that 
	\begin{equation}\label{15.5}
	C(w)C(s_{\alpha})=
	\begin{cases}
	C(ws_{\alpha})& \text{if $l(ws_{\alpha})=l(w)+1$,}\\
	C(w)\sqcup C(ws_{\alpha})& \text{if $l(ws_{\alpha})=l(w)-1$}.
	\end{cases}
	\end{equation}
	Then by \eqref{15} and \eqref{15.5}, one obtains that 
	\begin{equation}\label{20}
	C(w)^{s_{\alpha}}=
	\begin{cases}
	C(s_{\alpha}ws_{\alpha}),\ \text{if $l(s_{\alpha}ws_{\alpha})=l(w)+2$;}\\
	C(s_{\alpha}w)\sqcup C(s_{\alpha}ws_{\alpha}),\ \text{if $l(s_{\alpha}w)<l(w),$ $l(s_{\alpha}ws_{\alpha})>l(s_{\alpha}w);$}\\
	C(ws_{\alpha})\sqcup C(s_{\alpha}ws_{\alpha}),\ \text{if $l(ws_{\alpha})<l(w),$ $l(s_{\alpha}ws_{\alpha})>l(ws_{\alpha});$}\\
	C(w)\sqcup C(s_{\alpha}w)\sqcup C(ws_{\alpha})\sqcup C(s_{\alpha}ws_{\alpha}),\ \text{otherwise,}
	\end{cases}
	\end{equation}
	where we use $C(w)^{s_{\alpha}}$ to denote by $C(s_{\alpha})C(w)C(s_{\alpha}).$
	
	Let $w'\in W_{I_{k}}$ and $w\in W_{I_{n-1}}.$ Let $l(w)$ be the length of $w.$ Then $w$ could be written as a products of $l(w)$ simple reflections $s_{i},$ $1\leq i\leq n-1,$ and each $s_i$ corresponds to the associated reflection of some simple roots in $W_{I_{n-1}}.$ 
	\begin{itemize}
	\item Assume that $l(w_{n-1}\cdots w_2w_1w^{-1})=l(w_{n-1}\cdots w_2w_1)+l(w^{-1}).$ Take $w=s_{l(w)}\cdots s_2s_1$ to be a reduced representation by simple reflections and apply \eqref{15} and \eqref{15.5} inductively one then sees that
	$$
	C(w)C(w_{n-1}w_{n-2}\cdots w_1)C(w^{-1})=C(ww_{n-1}w_{n-2}\cdots w_1w^{-1}).
	$$ 
	We will simply identify Weyl elements in $W=W_n$ with translations on the set $\{1,2,\cdots,n\}$ under the isomorphism $W_n\xrightarrow{\sim} S_n.$ Then the cycle type decomposition of $ww_{n-1}w_{n-2}\cdots w_1w^{-1}$ is the same as that of $w_{n-1}w_{n-2}\cdots w_1,$ which is an $n$-cycle. However, since elements in $W_{I_{k}}$ can never be $n$-cycles, $C(ww_{n-1}w_{n-2}\cdots w_1w^{-1})\cap C(w')= \emptyset,$ $\forall$ $w'\in W_{I_{k}}.$
	\item Assume that $l(w_{n-1}\cdots w_2w_1w^{-1})<l(w_{n-1}\cdots w_2w_1)+l(w^{-1}).$ Denote by $\mathcal{D}(w)$ the set of all possible reduced representations of $w'$ by simple reflections. Then by our assumption, one can take a reduced representation of $w=s'_{l(w)}\cdots s'_2s'_1$ such that $s'_1=w_1.$ Hence one can well define 
	$$
	j_{w}:=\max_{1\leq j\leq l(w)}\big\{s_{l(w)}\cdots s_2s_1\in \mathcal{D}(w):\ s_i=w_i,\ 1\leq i\leq j\big\}.
	$$
	Let $w=s_{l(w)}\cdots s_2s_1$ be a reduced representation such that $s_i=w_i,$ $1\leq i\leq j_w.$ Then $w^{-1}=s_1s_2\cdots s_{l(w)}.$ Also, by \eqref{15} or \eqref{15.5} we have
	$$
	\qquad \qquad C(w)=C(s_{l(w)})\cdots C(s_2)C(s_1),\ \text{and}\ C(w^{-1})=C(s_1)C(s_2)\cdots C(s_{l(w)}),
	$$
	so $C(w)C(\widetilde{w})C(w^{-1})=C(s_{l(w)})\cdots C(s_2)C(s_1)C(\widetilde{w})C(s_1)C(s_2)\cdots C(s_{l(w)}),$ where we denote by $\widetilde{w}=w_{n-1}w_{n-2}\cdots w_1$ for convenience. According to \eqref{20}, a brute force computation shows that
	$$
	C(w)C(\widetilde{w})C(w^{-1})=C(w^*)\coprod\coprod_{1\leq i\leq j_w}C(w^{(i)}),
	$$
	where $w^*=ww_{n-1}\cdots w_{j_{w}+2}w_{j_{w}+1}s_{j_{w+1}}\cdots s_{l(w)},$ and for $1\leq i\leq j_w,$
	$w^{(i)}=ww_{n-1}\cdots w_{i+1}w_{i}w_{i+1}\cdots w_{j_{w}}s_{j_{w+1}}\cdots s_{l(w)}.$ 
	
	Let $w_{(j_w)}=s_{l(w)}\cdots s_{j_{w+1}},$ $w^*_{(j_{w})}=w_{j_w}\cdots w_1w_{n-1}\cdots w_{j_{w}+2}w_{j_{w}+1}.$ Then $w^*_{(j_{w})}$ is an $n$-cycle, and thus $w^*=w_{(j_w)}w^*_{(j_{w})}w_{(j_w)}^{-1}$ is also an $n$-cycle. So $w^*\notin W_{I_{k}},$ implying that $C(w^*)\cap C(w')= \emptyset,$ $\forall$ $w'\in W_{I_{k}}.$
	
	For each $1\leq i\leq j_w,$ let $w_{(i)}^*=w_{i-1}\cdots w_1w_{n-1}w_{n-2}\cdots w_{i+1},$ $w_{(i)}=s_{l(w)}\cdots s_{j_{w+1}}w_{j_w}\cdots w_i.$ Then $w^{(i)}=w_{(i)}w_{(i)}^*w_{(i)}^{-1}.$ One can check that $w_{(i)}^*=(1,2,\cdots,i)(i+1,\cdots,n),$ i.e. the cycle type of $w_{(i)}^*$ is $(i,n-i).$ So $w^{(i)}$ also has cycle type $(i,n-i).$ Since elements in $W_{I_{k}}$ can never have type of the form $(i,n-i),$ $w^{(i)}\notin W_{I_{k}}.$ Therefore, $C(w^*)\cap C(w^{(i)})= \emptyset,$ $\forall$ $1\leq i\leq j_w,$ $w'\in W_{I_{k}}.$ This completes the proof.
    \end{itemize}
	\end{proof}

	Now we consider $P(F)$-conjugation among elements in $\mathcal{R}_P=\{w_{n-1}w_{n-2}\cdots w_{1}b:\ b\in B(F)\}$ to determine representatives of $\mathcal{R}_P^{P(F)}.$ Define a relation $\mathscr{R}$ on the set $\bigg\{w_{n-1}w_{n-2}\cdots w_1\mathfrak{t}\mathfrak{u}:\ \mathfrak{t}\in T\left(F^{\times}\right),\ \mathfrak{u}\in N(F)\bigg\}$ such that $w_{n-1}w_{n-2}\cdots w_1\mathfrak{t}\mathfrak{u}$ is related to $w_{n-1}w_{n-2}\cdots w_1\mathfrak{t}'\mathfrak{u}'$ if and only if $\mathfrak{u}=\mathfrak{u}'$ and there are elements $a_1,\cdots, a_n=a_0\in F^{\times}$ such that
	\begin{equation}\label{23}
	\begin{cases}
	a_0t_1a_1^{-1}=t_1'\\
	a_1t_2a_2^{-1}=t_2'\\
	\qquad\vdots\\
	a_{n-1}t_{n}a_{n}^{-1}=t_{n}'.
	\end{cases}
	\end{equation}
	One can check easily that $\mathcal{R}_P^{P(F)}$ forms an equivalence relation.
	
	\begin{prop}\label{repre}
	Let notation be as before. Set 
	\begin{align*}
	\widetilde{\mathcal{R}}_P=\bigg\{w_{n-1}w_{n-2}\cdots w_1\mathfrak{t}\mathfrak{u}:\ \mathfrak{t}\in T\left(F^{\times}\right),\ \mathfrak{u}\in N_P(F)\bigg\}/\mathscr{R}.
	\end{align*}
	Then $\widetilde{\mathcal{R}}_P$ forms a family of representatives of $\mathcal{R}_P^{P(F)}.$  
    \end{prop}
	\begin{proof}
	Let $w_{n-1}w_{n-2}\cdots w_1b$ and $w_{n-1}w_{n-2}\cdots w_1b'$ be two elements in $\mathcal{R}_P,$ and write $b=\mathfrak{t}_n\mathfrak{u},$ $b'=\mathfrak{t}_n'\mathfrak{u}',$ the corresponding Levi decomposition. 
	Assume that there exists some $p_n\in P_n(F)=P(F)$ such that 
	\begin{equation}\label{n-1}
	p_nw_{n-1}w_{n-2}\cdots w_1bp_n^{-1}=w_{n-1}w_{n-2}\cdots w_1b'.
	\end{equation}
	Then $w_{n-1}p_nw_{n-1}=w_{n-2}\cdots w_1b'p_nb^{-1}w_1\cdots w_{n-2}\in P(F)=Q_{n-1}(F).$ Since $p_n\in P(F),$ it is necessary of the following form
	\begin{align*}
	p_n=\begin{pmatrix}
	A_{n-2} & \mathfrak{c}_{n-1}&\mathfrak{c}_{n}\\
	& a_{n-1}&0\\
	&&a_{n}
	\end{pmatrix}\in \begin{pmatrix}
	GL_{n-2}(F) & \ast&\ast\\
	& F^{\times}&0\\
	&&F^{\times}
	\end{pmatrix}\subset Q_{n-2}(F).
	\end{align*}
	Hence, $w_{n-2}w_{n-1}p_nw_{n-1}w_{n-2}=w_{n-3}\cdots w_1b'p_nb^{-1}w_1\cdots w_{n-3}\in Q_{n-2}(F),$ i.e.,
	\begin{align*}
	w_{n-2}\begin{pmatrix}
	A_{n-2} &\mathfrak{c}_{n} &\mathfrak{c}_{n-1}\\
	& a_{n}&0\\
	&&a_{n-1}
	\end{pmatrix}w_{n-2}\in \begin{pmatrix}
	GL_{n-2}(F) & \ast&\ast\\
	& F^{\times}&0\\
	&&F^{\times}
	\end{pmatrix}\subset Q_{n-2}(F).
	\end{align*}
	Then $A_{n-2}$ must lie in a maximal parabolic subgroup of $\GL_{n-2}(F)$ of type $(n-3,1),$ and the last component of $\mathfrak{c}_{n}$ must be vanishing. Thus we can write 
	\begin{align*}
	p_n=\begin{pmatrix}
	A_{n-3} & \mathfrak{c}_{n-2}&\mathfrak{c}_{n-1}^{(n-3)}&\mathfrak{c}_{n}^{(n-3)}\\
	&a_{n-2}&c_{n-2,n-1}&0\\
	&& a_{n-1}&0\\
	&&&a_{n}
	\end{pmatrix}\in \begin{pmatrix}
	GL_{n-3}(F) & \ast&\ast&\ast\\
	& F^{\times}&\ast&0\\
	&& F^{\times}&0\\
	&&&F^{\times}
	\end{pmatrix},
	\end{align*}
	where for any column vector $\mathfrak{c}_i=(c_{1,i},c_{2,i},\cdots,c_{m,i})^{\mathrm{T}},$ any $1\leq k\leq m,$  write $\mathfrak{c}_i^{(k)}=(c_{1,i},c_{2,i},\cdots,c_{k,i})^{\mathrm{T}},$ namely, the first $k$-entries. Now a similar analysis on the identity 
	$$
	w_{n-3}w_{n-2}w_{n-1}p_nw_{n-1}w_{n-2}w_{n-3}=w_{n-4}\cdots w_1b'p_nb^{-1}w_1\cdots w_{n-4}\in Q_{n-3}(F)
	$$ 
	leads to $\mathfrak{c}_{n}=\mathfrak{c}_{n}^{(n-4)},$ namely, the last $4$ elements of $\mathfrak{c}_n$ are all zeros. Likewise, continue this process $(n-4)$-more times to get $\mathfrak{c}_{n}=\boldsymbol{0}.$ Now \eqref{n-1} becomes
	\begin{equation}\label{25}
	\begin{pmatrix}
	a_n & 0&0&\ldots&0\\
	&a_{1}&c_{1,2}&\cdots&c_{1,n-1}\\
	&&\ddots&&\vdots\\
	&&& a_{n-2}&c_{n-2,n-1}\\
	&&&&a_{n-1}
	\end{pmatrix}b\begin{pmatrix}
	a_1 & c_{1,2}&\cdots&c_{1,n-1}&0\\
	&a_{2}&\cdots&c_{2,n-1}&0\\
	&&\ddots&\vdots&\vdots\\
	&&& a_{n-1}&0\\
	&&&&a_{n}
	\end{pmatrix}^{-1}=b'.
	\end{equation}
	When expanded, \eqref{25} becomes \eqref{29}, which will be investigated below. Before seeking for a solution to \eqref{25}, we will simplify it by taking showing that one can actually only consider some special $b$ and $b'$. This is justified by Claim \ref{19} below.
	
	Write $\mathfrak{t}_{n}=\diag(t_1,\cdots,t_n),$ and set $t_{i,j}=t_it_j^{-1}$; for any $n-2\leq k\leq n-1,$ define 
	\begin{align*}
	\mathfrak{n}_{n-k}(F)=\Bigg\{\begin{pmatrix}
	I_{k} &u_{k} &\\
	& 1&\\
	&&I_{n-k-1}
	\end{pmatrix}:\ u_{k}\in {M_{k\times 1}}(F)\Bigg\}.
	\end{align*}
	Let $\mathfrak{u}_{n-k}=\mathfrak{u}\cap \mathfrak{n}_{n-k}(F),$ $n-2\leq k\leq n-1.$ Then 
	$\mathfrak{u}=\mathfrak{u}_{1}\mathfrak{u}_{2}.$
	\begin{claim}\label{19}
	For any $b\in B(F),$ there exists a unique $\mathfrak{u}\in \mathfrak{n}_2(F)\subset P(F),$ such that $\mathfrak{u}^{-1}w_{n-1}w_{n-2}\cdots w_1b\mathfrak{u}\in w_{n-1}w_{n-2}\cdots w_1T(F)\mathfrak{n}_1(F).$
	\end{claim}
	So we only need to consider the $P(F)$-conjugacy of among elements in $\mathcal{R}=\{w_{n-1}w_{n-2}\cdots w_1\mathfrak{t}\mathfrak{u}:\ \mathfrak{t}\in T\left(F^{\times}\right),\ \mathfrak{u}\in N_P(F)\}.$
	
	Let $w_{n-1}w_{n-2}\cdots w_1\mathfrak{t}_n\mathfrak{u}_1,$ $w_{n-1}w_{n-2}\cdots w_1\mathfrak{t}_n'\mathfrak{u}_1'\in \mathfrak{R}$ be $P(F)$-conjugate. Then there exists some $p_n=\diag(b,a_n)\in \diag(B(F),F^{\times})$ such that 
	\begin{equation}\label{26}
	p_nw_{n-1}w_{n-2}\cdots w_1\mathfrak{t}_n\mathfrak{u}_1p_n^{-1}=w_{n-1}w_{n-2}\cdots w_1\mathfrak{t}_n'\mathfrak{u}_1'.
	\end{equation}
	Write $b=\diag(a_1,a_2,\cdots,a_{n-1})\mathfrak{u},$ where we identity $\mathfrak{u}$ with $\diag(\mathfrak{u},1)\in \mathfrak{n}_2(F).$ Then comparing the Levi components of both sides in \eqref{26} leads exactly the system of relations \eqref{23}; while the unipotent radical gives the equation \eqref{29} with $c^0_{i,j}=1,$ $1\leq i<j\leq n-1.$ By the uniqueness of solution (shown in the proof of Claim \ref{19}), $\mathfrak{u}=I_{n}.$ Therefore, $\mathfrak{u}_1=\mathfrak{u}_1'\in \mathfrak{n}_1(F)=N_P(F).$ Then the proof follows.
    \end{proof}
    \begin{proof}[Proof of Claim \ref{19}]
    Let notation be as in the proof of Proposition \ref{repre}. Let $\mathfrak{u}=\{u_{i,j}\}_{1\leq i,j\leq n}\in \mathfrak{n}_2(F),$ and $\mathfrak{c}=\{c_{i,j}\}_{1\leq i,j\leq n}=w_1\cdots w_{n-1}\mathfrak{u}^{-1}w_{n-1}\cdots w_1\mathfrak{t}_n\mathfrak{u}\mathfrak{t}_n^{-1}.$ Denote by $\mathfrak{u}^{\times}=\{u_{i,j}'\}_{1\leq i,j\leq n}\in \mathfrak{n}_2(F).$ Then one has, for any $1\leq i<j\leq n,$ that
    \begin{equation}\label{27}
    u_{i,j}+u_{i,i+1}'u_{i+1,j}+u_{i,i+2}'u_{i+2,j}+\cdots+u_{i,j-1}'u_{j-1,j}+u_{i,j}'=0.
    \end{equation}
    Also, an elementary computation shows that for any $1\leq i<j\leq n,$ one has
    \begin{equation}\label{28}
    c_{i,j}=t_{i}t_j^{-1}u_{i,j}+t_{i+1}t_j^{-1}u_{i-1,i}'u_{i+1,j}+\cdots+t_{j-1}t_j^{-1}u_{i-1,j-2}'u_{j-1,j}+u_{i-1,j-1}'.
    \end{equation}
    Now fix $c^0_{i,j},$ $1<i<j<n,$ and $\mathfrak{t}_n,$ then we show by a double induction that there exists uniquely $u_{i,j},$ $1<i<j<n,$ such that $c_{i,j}=c^0_{i,j},$ $1<i<j<n,$ i.e., we want to solve the system of equations, for fixed $t_1,$ $\cdots,$ $t_n,$
    \begin{equation}\label{29}
    t_{i}t_j^{-1}u_{i,j}+t_{i+1}t_j^{-1}u_{i-1,i}'u_{i+1,j}+\cdots+t_{j-1}t_j^{-1}u_{i-1,j-2}'u_{j-1,j}+u_{i-1,j-1}'=c_{i,j}^0.
    \end{equation}
    
    When $n\leq 4$, one can check directly by hand that the solution to \eqref{29} exists and is unique. So from now on we assume that $n>4.$ Let $\mathcal{D}_i=\{D_{i,j}=u_{j+1,j+i+1}:\ 1\leq j\leq n-2-i\},$ $1\leq i\leq n-3.$ By \eqref{28}, $c_{1,j}=t_{1}t_j^{-1}u_{1,j},$ so to make $c^{1,j}=c^0_{1,j},$ one takes $u_{1,j}=t_{1}^{-1}t_jc^0_{1,j},$ $1\leq j\leq n-1.$ Also, by \eqref{27}, $u_{i,i+1}'=-u_{i,i+1},$ so \eqref{28} shows that $c_{i,i+1}=u_{i,i+1}+u_{i-1,i}'=u_{i,i+1}-u_{i-1,i},$ $2\leq i\leq n-2.$ Let $c_{i,i+1}=c_{i,i+1}^0,$ then $u_{i,i+1}=u_{i-1,i}+c_{i,i+1}^0,$ $2\leq i\leq n-2.$ Since $u_{12}=t_{1}^{-1}t_2c^0_{1,2},$ then a simple induction shows that elements in $\mathcal{D}_1$ are uniquely determined by the equation $c_{i,j}=c^0_{i,j},$ $1<i<j<n.$
    	
    Now, let $1<i_0\leq n-3,$ assume that $\mathcal{D}_i$ are uniquely solved out by \eqref{29} for any $1<i<i_0.$  By our assumption and \eqref{27}, $u_{1,i}'$ are now uniquely determined, $1\leq i\leq i_0-1.$ Then according to \eqref{28}, $D_{i_0,1}=u_{i_0+1,i_0+2}=c_{i_0+1,i_0+2}^0+u_{1,i_0+1}+u_{1,2}'u_{3,i_0+2}+u_{1,3}'u_{4,i_0+2}+\cdots+u_{1,i_0}'u_{i_0+1,i_0+2},$  where $u_{i,i_0+2}\in \mathcal{D}_{i_0+2-i},$ $3\leq i\leq i_0+1.$ So $D_{i_0,1}$ is uniquely determined. Assume that we have solved out all $D_{i_0,j},$ $j<j_0,$ in $\mathcal{D}_{i_0}.$ Then by \eqref{28}, $D_{i_0,j_0}=u_{j_0+1,j_0+i_0+1}$ completely depends on $u_{j_0,j}',$ $j_0+1\leq j\leq j_0+i_0,$ and $u_{j_0+k,j_0+i_0+1}\in \mathcal{D}_{i_0+1-k},$ $2\leq k\leq i_0.$ Again, by \eqref{27}, we can inductively compute each $u_{j_0,j}'$ in terms of $u_{i',j'},$ $j_0\leq i'<j'\leq j$ such that $(i',j')\neq (j_0,j).$ By our inductive assumption, all these $u_{i',j'}$'s and $u_{j_0+k,j_0+i_0+1}$ have been solved out uniquely. Then $D_{i_0,j_0}$ is thus obtained. By induction, elements in $\mathcal{D}_{i_0}$ are uniquely determined. Therefore, by induction on the index $i_0,$ one verifies that the solution to \eqref{29} does exist and in fact is unique.
    
    Denote by $\mathfrak{u}_0=\mathfrak{u}_{\mathfrak{c}^0,\mathfrak{t}_n}\in \mathfrak{n}_2(F)$ the solution to \eqref{29}. Then $\mathfrak{u}_0$ depends only on $\mathfrak{c}^0=\{c_{i,j}^0\}_{1\leq i,j\leq n}$ and $\mathfrak{t}_n,$ where we define $c_{i,n}^0=\delta_{i,n},$ $1\leq i\leq n,$ here $\delta$ is the Kronecker symbol. Let $b=\mathfrak{u}_2\mathfrak{u}_1\mathfrak{t}_n$ be an arbitrary element in $B(F).$ Take $\mathfrak{c}^0=\mathfrak{u}_2,$ $\mathfrak{u}_0\in \mathfrak{n}_2(F)$ the solution to \eqref{29}, and define $\mathfrak{n}_P=\mathfrak{u}_0\mathfrak{t}_n^{-1}\mathfrak{u}_1\mathfrak{t}_n\mathfrak{u}_0^{-1}\in \mathfrak{n}_1(F).$ Then the following conjugacy equation holds:
    $$
    \mathfrak{u}_0^{-1}w_{n-1}w_{n-2}\cdots w_1\mathfrak{t}_n\mathfrak{n}_P\mathfrak{u}_0=w_{n-1}w_{n-2}\cdots w_1\mathfrak{u}_2\mathfrak{u}_1\mathfrak{t}_n.
    $$
    Therefore, one can take representatives of $P(F)$-conjugacy classes $\mathcal{R}_P^{P(F)}$ in the set $\mathcal{R}=\{w_{n-1}w_{n-2}\cdots w_1\mathfrak{t}\mathfrak{u}:\ \mathfrak{t}\in T\left(F^{\times}\right),\ \mathfrak{u}\in N_P(F)\}.$
    \end{proof}
    \begin{remark}
    Let $\gamma=w_{n-1}w_{n-2}\cdots w_1\mathfrak{t}\mathfrak{u}\in GL_{n}(F),$ $ \mathfrak{t}\in T\left(F^{\times}\right),$ $\mathfrak{u}\in N(F),$ then the $P(F)$-conjugacy class of $\gamma$ is thoroughly determined by $\det\gamma$ and $\mathfrak{u}\cap N_P(F).$
    \end{remark}

	Now we consider for our purpose the decomposition of $Z_G(F)\backslash G(F)$ into $P(F)$-conjugacy classes. By \eqref{14.2} one has the following decomposition 
	\begin{equation}\label{12'}
	Z_G(F)\backslash G(F)=\left(Z_G(F)\cap \mathfrak{C}_{r.e.}^{P(F)}\right)\backslash\mathfrak{C}_{r.e.}^{P(F)}\coprod \bigcup _{k=1}^{n-1}\left(Z_G(F)\backslash Q_k(F)\right)^{P(F)}.
	\end{equation}

	\begin{cor}\label{repres}
	Let notation be as before. Set $\left(F^{\times}\right)^n=\{t^n:\ t\in F^{\times}\},$ and let
	\begin{equation}\label{mainrep}
	\widetilde{\mathcal{R}}_P^{*}=\Bigg\{w_1w_2\cdots w_{n-1}\begin{pmatrix}
	I_{n-3} & &\\
	& t&\\
	&&I_2
	\end{pmatrix}\mathfrak{u}:\ t\in F^{\times}/\left(F^{\times}\right)^n,\ \mathfrak{u}\in N_P(F)\Bigg\}.
	\end{equation}
	Then $\widetilde{\mathcal{R}}_P^{*}$ forms a family of representatives of $\left(Z_G(F)\cap \mathfrak{C}_{r.e.}^{P(F)}\right)\backslash\mathfrak{C}_{r.e.}^{P(F)}.$ 
	\end{cor}
	\begin{proof}
	By Lemma \ref{122} and Proposition \ref{repre}, the regular elliptic direct summand $\left(Z_G(F)\cap \mathfrak{C}_{r.e.}^{P(F)}\right)\backslash\mathfrak{C}_{r.e.}^{P(F)}$ in \eqref{12'} can be represented by 
	$\left(Z_G(F)\cap \widetilde{\mathcal{R}}_P\right)\backslash\widetilde{\mathcal{R}}_P\simeq Z_G(F)\backslash Z_G(F)\widetilde{\mathcal{R}}_P.$ Given any $\mathfrak{t}=\diag(t_1,t_2,\cdots,t_n)\in T(F),$ one can take 
	\begin{align*}
	p_n=\begin{pmatrix}
	t_1 & &&&\\
	& t_1t_2&&&\\
	&&\ddots&&\\
	&&&t_1t_2\cdots t_{n-1}\\
	&&&&1
	\end{pmatrix},\ \text{and}\ \mathfrak{t}'=\begin{pmatrix}
	I_{n-2}&&\\
	&\prod_{i=1}^nt_i&\\
	&&1
	\end{pmatrix},
	\end{align*}
	then clearly $p_nw_{n-1}w_{n-2}\cdots w_1\mathfrak{t}\mathfrak{u}p_n^{-1}=w_{n-1}w_{n-2}\cdots w_1\mathfrak{t}'\mathfrak{u}'$ holds with $\mathfrak{u}=\mathfrak{u}'.$ 
	
	For any $\mathfrak{z}=\diag(z,z,\cdots,z)\in Z_G(F),$ $\mathfrak{z}\mathfrak{t}=\diag(zt_1,zt_2,\cdots,zt_n)$ is $P(F)$-conjugate to 
	$$
	\begin{pmatrix}
		I_{n-2}&&\\
		&z^nt_1t_2\cdots t_{n}&\\
		&&1
	\end{pmatrix}\in \begin{pmatrix}
	I_{n-2}&&\\
	&F^{\times}/\left(F^{\times}\right)^n&\\
	&&1
	\end{pmatrix}.
	$$
	Hence $\left(Z_G(F)\cap \mathfrak{C}_{r.e.}^{P(F)}\right)\backslash\mathfrak{C}_{r.e.}^{P(F)}\subseteq \widetilde{\mathcal{R}}_P^{*,P(F)}:=\bigcup_{p\in P(F)}\big\{p\gamma p^{-1}\in \widetilde{\mathcal{R}}_P^{*}\big\}.$
	
	Conversely, for any $z\in F^{\times},$ $t\in F^{\times},$ there exists some $p\in P(F),$ such that  
	$$
	p\begin{pmatrix}
	I_{n-2}&&\\
	&z^nt&\\
	&&1
	\end{pmatrix}p^{-1}=z\begin{pmatrix}
	I_{n-2}&&\\
	&t&\\
	&&1
	\end{pmatrix}\in Z_G(F)\widetilde{\mathcal{R}}_P,
	$$
	thus any two distinguished elements in $\widetilde{\mathcal{R}}_P^{*}$ will not $P(F)$-conjugate. Hence the set 
	\begin{equation}\label{33}
	\Bigg\{w_{n-1}w_{n-2}\cdots w_1\begin{pmatrix}
	I_{n-2} & &\\
	& t&0\\
	&&1
	\end{pmatrix}\mathfrak{u}:\ t\in F^{\times}/\left(F^{\times}\right)^n,\ \mathfrak{u}\in N_P(F)\Bigg\}
	\end{equation}
	forms a family of representatives of $\left(Z_G(F)\cap \mathfrak{C}_{r.e.}^{P(F)}\right)\backslash\mathfrak{C}_{r.e.}^{P(F)}.$ Then the inverse of elements in the set defined in \eqref{33} also form a family of representatives of $\left(Z_G(F)\cap \mathfrak{C}_{r.e.}^{P(F)}\right)\backslash\mathfrak{C}_{r.e.}^{P(F)}.$ Note that these inverses are bijectively $P_0(F)$-conjugate to $\widetilde{\mathcal{R}}_P^{*},$ then the proof follows.
	\end{proof}

	\subsection{Holomorphic Continuation}\label{6.2}
	Let $P_0(F)$ be the mirabolic subgroup of $G(F),$ then by definition we have $P_0(F)=R_{n-1}(F).$ For any $\gamma\in G(F),$ write $\gamma^{P_0(F)}$ for the $P_0(F)$-conjugacy class of $\gamma,$ which is the same as $P(F)$-conjugacy class of $\gamma.$ Then by Corollary \ref{repres} one can decompose $Z_G(F)\backslash G(F)$ as 
	\begin{equation}\label{7}
	Z_G(F)\backslash G(F)=\coprod_{\gamma\in \widetilde{\mathcal{R}}_P^{*}}\gamma^{P_0(F)}\coprod \bigcup _{k=1}^{n-1}\left(Z_G(F)\backslash Q_k(F)\right)^{P_0(F)}.
	\end{equation}
	By the decomposition \eqref{7}, one can write $\K(x,y)=\K_{\Geo,\Reg}(x,y)+	\K^{\Geo,\Sin}(x,y),$ where 
	\begin{align*}
	\K_{\Geo,\Reg}(x,y)&=\sum_{\gamma\in \widetilde{\mathcal{R}}_P^{*}}\sum_{p\in P_0(F)}\varphi(x^{-1}p^{-1}\gamma py),\\
	\K^{\Geo,\Sin}(x,y)&=\sum_{\gamma^{P_0(F)}\in \mathcal{P}}\sum_{p\in P_0(F)}\varphi(x^{-1}p^{-1}\gamma py).
	\end{align*}
Hence we have the decomposition	
	\begin{align*}
I_{\infty,\Reg}(s,\tau)=\int_{X_n}\int_{[N_P]}	\K^{r.e.}(nx,x)dnf(x,s)dx.
	\end{align*}
	where $X_n=Z_G(\mathbb{A}_F)R_{n-1}(F)\backslash G(\mathbb{A}_F)=Z_G(\mathbb{A}_F)P_{0}(F)\backslash G(\mathbb{A}_F)$ and 
	$$
	\mathcal{P}=\big\{\gamma^{P(F)}:\ \gamma\in Z_G(F)\backslash Q_k(F)\ \text{for some $1\leq k\leq n-1$}\big\}.
	$$

	Note that $f(x,s)=f_{\tau}(x,s)$ is $P(F)$-invariant, then by \eqref{7} and a similar trick of changing variables and interchanging integrals one has formally that 
	\begin{align*}
I_{\infty,\Reg}(s,\tau)&=\int_{X_n}\int_{[N_P]}\sum_{\gamma\in \widetilde{\mathcal{R}}_P^{*}}\sum_{p\in P_0(F)}\varphi(x^{-1}n^{-1}p^{-1}\gamma px)dnf(x,s)dx\\
	&=\int_{Z_G(\mathbb{A}_F)\backslash G(\mathbb{A}_F)}\int_{N_P(F)\backslash N_P(\mathbb{A}_F)}\sum_{\gamma\in \widetilde{\mathcal{R}}_P^{*}}\varphi(x^{-1}n^{-1}\gamma x)dnf(x,s)dx\\
	&=\int_{X'(\mathbb{A}_F)}\int_{[N_P]}\int_{N_P(\mathbb{A}_F)}\sum_{\gamma\in \widetilde{\mathcal{R}}_P^{*}}\varphi(x^{-1}u^{-1}n^{-1}\gamma ux)dnduf(x,s)dx\\
	&=\int_{X'(\mathbb{A}_F)}\int_{[N_P]}\sum_{\gamma\in \widetilde{\mathcal{R}}_P^{*}}\varphi(x^{-1}u^{-1}\gamma nx)dnduf(x,s)dx.
	\end{align*}
	Recall that for any $2\leq k\leq n,$ any $v\in\Sigma_F,$  we have defined 
	\begin{align*}
	N_k^*=\Bigg\{\begin{pmatrix}
	I_{k-1} &u_{k} &\\
	& 1&\\
	&&I_{n-k}
	\end{pmatrix}:\ u_{k}\in {M_{(k-1)\times 1}}\Bigg\}.
	\end{align*}
	Let $N_k^*(\mathbb{A}_F)$ be the restricted product of $N_k^*(F_v)$'s, over $v\in \Sigma_F.$ 
	Then 
	$$
	N(\mathbb{A}_F)=\prod_{k=2}^{n}N_k^*(\mathbb{A}_F)=N_n^*(\mathbb{A}_F)N_{n-1}^*(\mathbb{A}_F)\cdots N_2^*(\mathbb{A}_F),\ \text{and}\ N(F)=\prod_{k=2}^{n}N_k^*(F).
	$$ Write $N^{P}=\prod_{k=2}^{n-1}N_k^*.$ Then one can write that  $N(\mathbb{A}_F)=N_{P}(\mathbb{A}_F)N^{P}(\mathbb{A}_F)$ and $N(F)=N_{P}(F)N^{P}(F).$ Apply Iwasawa decomposition to $X'(\mathbb{A}_F)$ to see
	$$
	X'(\mathbb{A}_F)=Z_G(\mathbb{A}_F)N_P(\mathbb{A}_F)\backslash G(\mathbb{A}_F)=Z_G(\mathbb{A}_F)\backslash T(\mathbb{A}_F)N^{P}(\mathbb{A}_F)K,
	$$
	where $K$ is a maximal compact subgroup. Set $T^*(\mathbb{A}_F)=Z_G(\mathbb{A}_F)\backslash T(\mathbb{A}_F)$ for convenience. For any $\gamma\in \widetilde{\mathcal{R}}_P^{*},$ write it uniquely as $\gamma=w_1w_{2}\cdots w_{n-1}\begin{pmatrix}
	I_{n-2} & &\\
	& t&0\\
	&&1
	\end{pmatrix}\mathfrak{u},$ with $t\in F^{\times}/\left(F^{\times}\right)^n,$ and $\mathfrak{u}\in N(F).$ Set $\widetilde{w}=w_1w_{2}\cdots w_{n-1},$ $\mathfrak{u}_P=\mathfrak{u}\cap N_P(F)$ and $\mathfrak{u}^P=\mathfrak{u}\cap N^P(F).$ Then $\mathfrak{u}=\mathfrak{u}_P\mathfrak{u}^P.$ Let 
	$\rho_{T^*}$ be the half-sum of positive roots of $T^*,$ and set $\delta_{T^*}(\mathfrak{t})=\mathfrak{t}^{2\rho_{T^*}}$ to be the modular character, explicitly, for any $\mathfrak{t}=\diag(t_1,t_2,\cdots,t_{n-1},1)\in T^*(F),$ $\delta_{T^*}(\mathfrak{t})=\prod_{i=1}^{n-1}|t_i|^{n-2i+1}_{\mathbb{A}_F}.$ Substitute these into the expression of $I_{\infty}^{r.e.}(s)$ to get
	\begin{align*}
I_{\infty,\Reg}(s,\tau)&=\int_{X'(\mathbb{A}_F)}f(x,s)dx\int_{N_P(\mathbb{A}_F)}\int_{N_P(F)\backslash N_P(\mathbb{A}_F)}\sum_{t\in F^{\times}/\left(F^{\times}\right)^n}\sum_{\mathfrak{u}\in N_P(F)}\\
	&\qquad\qquad\qquad\qquad\qquad\qquad\varphi\left(x^{-1}u^{-1}\widetilde{w}v\begin{pmatrix}
	I_{n-2} & &\\
	& t&\\
	&&1
	\end{pmatrix}\mathfrak{u} nx\right)dndu\\
	&=\int_K\int_{N^{P}(\mathbb{A}_F)}\int_{ T^*(\mathbb{A}_F)}f(n^P\mathfrak{t}k,s)\frac{d^{\times}\mathfrak{t}}{\delta_{T^*}(\mathfrak{t})}dn^Pdk\int_{N_P(\mathbb{A}_F)}\sum_{t}\int_{[N_P]}\\
	&\quad \sum_{\mathfrak{u}_P\in N_P(F)}\varphi\left(k^{-1}\mathfrak{t}^{-1}(n^P)^{-1}u^{-1}\widetilde{w}\begin{pmatrix}
	I_{n-2} & &\\
	& t&\\
	&&1
	\end{pmatrix}\mathfrak{u}_Pnn^P\mathfrak{t}k\right)dndu\\
	&=\int_K\int_{N^{P}(\mathbb{A}_F)}\int_{ T^*(\mathbb{A}_F)}f(\mathfrak{t}n^Pk,s)d^{\times}\mathfrak{t}dn^Pdk\int_{N_P(\mathbb{A}_F)}\sum_{t\in F^{\times}/\left(F^{\times}\right)^n}\\
	&\quad\times\int_{N_P(\mathbb{A}_F)}\varphi\left(k^{-1}(n^P)^{-1}u^{-1}\mathfrak{t}^{-1}\widetilde{w}\begin{pmatrix}
	I_{n-2} & &\\
	& t&\\
	&&1
	\end{pmatrix}n\mathfrak{t}n^Pk\right)dndu.
	\end{align*}
	Recall that $f(x,s)$ is defined by
	$$
	f(x,s)=\tau(\det x)|\det x|_{\mathbb{A}_F}^s\int_{\mathbb{A}_F^{\times}}\Phi[(0,\cdots,t)x]\tau^{n}(t)|t|^{ns}d^{\times}t.
	$$
    Then $f(\mathfrak{t}n^Pk,s)=\tau(\det \mathfrak{t})|\det \mathfrak{t}|^sf(k,s),$ where we identify $T^*(\mathbb{A}_F)$ with the subgroup $\{\diag(t_1,\cdots,t_{n-1},1):\ t_i\in \mathbb{A}_F^{\times},\ 1\leq i\leq n-1\}.$ Therefore one has
	\begin{align*}
I_{\infty,\Reg}(s,\tau)&=\int_Kf(k,s)dk\int_{N^{P}(\mathbb{A}_F)}dn^P\int_{ T^*(\mathbb{A}_F)}\int_{N_P(\mathbb{A}_F)}\sum_{t}\tau(\det \mathfrak{t})|\det \mathfrak{t}|_{\mathbb{A}_F}^{s+1}\\
	&\quad \times\int_{N_P(\mathbb{A}_F)}\varphi\left(k^{-1}(n^P)^{-1}u^{-1}\mathfrak{t}^{-1}\widetilde{w}\begin{pmatrix}
	I_{n-2} & &\\
	& t&\\
	&&1
	\end{pmatrix}\mathfrak{t}nn^Pk\right)dndud^{\times}\mathfrak{t},
	\end{align*}
	where $t$ runs through $F^{\times}/\left(F^{\times}\right)^n.$
	
	Given any $\mathfrak{u}'\in N^P(\mathbb{A}_F),$ and any $\mathfrak{t}'\in T(\mathbb{A}_F),$ we consider the following system of equations with respect to variables $c_{i,j},$ $1\leq i<j\leq n-1,$ and $\mathfrak{u}\in N_P(\mathbb{A}_F),$
	\begin{equation}\label{35}
	\begin{pmatrix}
	1 & c_{1,2}&\cdots&c_{1,n-1}&0\\
	&1&\cdots&c_{2,n-1}&0\\
	&&\ddots&\vdots&\vdots\\
	&&& 1&0\\
	&&&&1
	\end{pmatrix}^{-1}\cdot\mathfrak{t}'\cdot\begin{pmatrix}
	1 & 0&0&\ldots&0\\
	&1&c_{1,2}&\cdots&c_{1,n-1}\\
	&&\ddots&&\vdots\\
	&&& 1&c_{n-2,n-1}\\
	&&&&1
	\end{pmatrix}=\mathfrak{t}'\mathfrak{u}'\mathfrak{u}.
	\end{equation}
	One sees easily that equation \eqref{35} is equivalent to \eqref{25} or the system of equations \eqref{29}. By the existence of solutions to equation \eqref{25} (with fixed initial datum), we can find some $\mathfrak{u}=\mathfrak{u}_0\in N_P(\mathbb{A}_F),$ and $c_{i,j}=c_{i,j}^0\in \mathbb{A}_F,$ $1\leq i<j\leq n-1,$ such that \eqref{35} holds. Therefore, one can always find some element $\mathfrak{c}\in N^P(\mathbb{A}_F)$ such that 
	$$
	\mathfrak{c}^{-1}\mathfrak{t}^{-1}\widetilde{w}\begin{pmatrix}
	I_{n-2} & &\\
	& t&\\
	&&1
	\end{pmatrix}\mathfrak{t}\mathfrak{c}=\mathfrak{t}^{-1}\widetilde{w}\begin{pmatrix}
	I_{n-2} & &\\
	& t&\\
	&&1
	\end{pmatrix}\mathfrak{t}\mathfrak{u}_0\mathfrak{u}'.
	$$
	Hence for any $\mathfrak{u}'\in N^P(\mathbb{A}_F),$ one can rewrite $I_{\infty,\Reg}(s,\tau)=I_{\infty}^{\Reg}(\mathfrak{u}';s),$ where   
	\begin{align*}
	I_{\infty}^{\Reg}(\mathfrak{u}';s)&=\int_Kf(k,s)dk\int_{N^{P}(\mathbb{A}_F)}dn^P\int_{ T^*(\mathbb{A}_F)}\int_{N_P(\mathbb{A}_F)}\sum_{t}\tau(\det \mathfrak{t})|\det \mathfrak{t}|_{\mathbb{A}_F}^{s+1}\\
	&\ \times\int_{N_P(\mathbb{A}_F)}\varphi\left(k^{-1}(n^P)^{-1}u^{-1}\mathfrak{t}^{-1}\widetilde{w}\begin{pmatrix}
	I_{n-2} & &\\
	& t&\\
	&&1
	\end{pmatrix}\mathfrak{t}n\mathfrak{u}'n^Pk\right)dndud^{\times}\mathfrak{t}.
	\end{align*}
	Let $c_P=\vol\left(N^P(F)\backslash N^P(\mathbb{A}_F)\right).$ Since $I_{\infty}^{\Reg}(\mathfrak{u}';s)$ is $N^P(\mathbb{A}_F)$-invariant, one can integrate $I_{\infty}^{r.e.}(\mathfrak{u}^{-1};s)$ over the compact domain $N^P(F)\backslash N^P(\mathbb{A}_F)$ to get
	\begin{align*}
	I_{\infty}^{\Reg}(s)&=\frac1{c_P}\int_Kf(k,s)dk\int_{N^{P}(\mathbb{A}_F)}dn^P\int_{[N^{P}]}\int_{ T^*(\mathbb{A}_F)}\int_{N_P(\mathbb{A}_F)}\sum_{t}\tau(\det \mathfrak{t})|\det \mathfrak{t}|_{\mathbb{A}_F}^{s+1}\\
	&\ \times\int_{N_P(\mathbb{A}_F)}\varphi\left(k^{-1}(n^P)^{-1}u^{-1}\mathfrak{t}^{-1}\widetilde{w}\begin{pmatrix}
	I_{n-2} & &\\
	& t&\\
	&&1
	\end{pmatrix}\mathfrak{t}n\mathfrak{u}'n^Pk\right)dndud^{\times}\mathfrak{t}d\mathfrak{u}'\\
	&=\frac1{c_P}\int_Kf(k,s)dk\int_{N(\mathbb{A}_F)}d\mathfrak{u}\int_{[N^{P}]}\int_{ T^*(\mathbb{A}_F)}\sum_{t\in F^{\times}/\left(F^{\times}\right)^n}\tau(\det \mathfrak{t})|\det \mathfrak{t}|_{\mathbb{A}_F}^{s+1}\\
	&\qquad \times\int_{N_P(\mathbb{A}_F)}\varphi\left(k^{-1}\mathfrak{u}\mathfrak{t}^{-1}\widetilde{w}\begin{pmatrix}
	I_{n-2} & &\\
	& t&\\
	&&1
	\end{pmatrix}\mathfrak{t}\widetilde{w}^{-1}\cdot\widetilde{w}n\mathfrak{u}'k\right)dnd^{\times}\mathfrak{t}d\mathfrak{u}',
	\end{align*}
	where $I_{\infty}^{\Reg}(s)$ refers to $I_{\infty,\Reg}(s,\tau).$ After a changing of variables one obtains
	\begin{align*}
	I_{\infty}^{\Reg}(s)&=\frac1{c_P}\int_Kf(k,s)dk\int_{N(\mathbb{A}_F)}d\mathfrak{u}\int_{[N^{P}]}d\mathfrak{u}'\int_{N_P(\mathbb{A}_F)}dn\sum_{t\in F^{\times}/\left(F^{\times}\right)^n}\Delta^{(1)}_{s,\tau}(\mathfrak{t})\\
	&\ \times \int_{ \mathbb{A}_F^{\times}}\cdots\int_{ \mathbb{A}_F^{\times}}\varphi\left(k^{-1}\mathfrak{u}\begin{pmatrix}
	1& &&&\\
	& t_2^{-1}&&&\\
	&&\ddots&&\\
	&&&t_{n-1}^{-1}&\\
	&&&&t_1^nt\prod_{i=2}^{n-1}t_i
	\end{pmatrix}\widetilde{w}n\mathfrak{u}'k\right)d^{\times}\mathfrak{t},
	\end{align*}
	where $d^{\times}\mathfrak{t}=d^{\times}t_1d^{\times}t_2\cdots d^{\times}t_{n-1},$ and for any $\mathfrak{t}=\diag(t_1,t_2,\cdots,t_{n-1},1)\in T^*(\mathbb{A}_F),$ 
	$$
	\Delta^{(1)}_{s,\tau}(\mathfrak{t})=\tau(t_1)^{\frac{n(n-1)}{2}}|t_1|_{\mathbb{A}_F}^{\frac{n(n-1)}{2}(s+1)}\prod_{i=2}^{n-1}\tau(t_i^{n-i})|t_i|_{\mathbb{A}_F}^{(n-i)(s+1)}.
	$$
	Depending on the purity of $n,$ we can further simplify $I_{\infty}^{\Reg}(s)=I_{\infty,\Reg}(s,\tau).$ Recall the test function $\varphi$ has the central character $\omega,$ $\Xi$ is the set of idele class characters on $\mathbb{A}_F,$ which is trivial on the archimedean places. Denote by $\Xi_{\omega,n}$ the subset $\{\chi\in \Xi:\ \chi^n=\omega\}\subset \Xi.$ Also, let $\Xi_{\tau,2}^n=\{\xi\in \Xi:\ \xi^2=\tau\}$ if $n$ is even, and set $\Xi_{\tau,2}^n$ to be the empty set if $n$ is odd. Then both $\#\Xi_{\tau,2}^n<\infty$ and $\#\Xi_{\tau,2}^n<\infty.$ 
	
	When $n$ is odd, we have, by the computation above, that 
	\begin{align*}
	I_{\infty}^{\Reg}(s)&=\frac1{c_P}\int_Kf(k,s)dk\int_{N(\mathbb{A}_F)}d\mathfrak{u}\int_{[N^{P}]}d\mathfrak{u}'\int_{N_P(\mathbb{A}_F)}dn\sum_{\chi\in \Xi_{\omega,n}} \int_{ \mathbb{A}_F^{\times}}\Delta^{od}_{s,\tau,\chi}(\mathfrak{t})d^{\times}t_1\\
	&\ \times \int_{ \mathbb{A}_F^{\times}}\cdots\int_{ \mathbb{A}_F^{\times}}\varphi\left(k^{-1}\mathfrak{u}\begin{pmatrix}
	1& &&&\\
	& t_2^{-1}&&&\\
	&&\ddots&&\\
	&&&t_{n-1}^{-1}&\\
	&&&&t_1
	\end{pmatrix}\widetilde{w}n\mathfrak{u}'k\right)d^{\times}t_2\cdots d^{\times}t_{n-1},
	\end{align*}
	where we use the fact that $\left(\mathbb{A}_F^{\times}\right)^n\cdot F^{\times}/\left(F^{\times}\right)^n=F^{\times}\cdot \left(F^{\times}\backslash \mathbb{A}_F^{\times} \right)^n,$ and $\tau|\cdot|_{\mathbb{A}_F}$ is $F^{\times}$-invariant, and
	$$
	\Delta^{od}_{s,\tau,\chi}(\mathfrak{t})=\bar{\chi}(t_1)\tau(t_1)^{\frac{n-1}{2}}|t_1|_{\mathbb{A}_F}^{\frac{(n-1)(s+1)}{2}}\prod_{i=2}^{n-1}\tau(t_i^{\frac{n+1}2-i})|t_i|_{\mathbb{A}_F}^{[\frac{n+1}2-i](s+1)}.
	$$
	When $n$ is even, one has a similar simplification as follows
	\begin{align*}
	I_{\infty}^{\Reg}(s)&=\frac1{c_P}\int_Kf(k,s)dk\int_{N(\mathbb{A}_F)}d\mathfrak{u}\int_{[N^{P}]}d\mathfrak{u}'\int_{N_P(\mathbb{A}_F)}dn\sum_{\chi\in \Xi_{\omega,n}}\sum_{\xi\in \Xi_{\tau,2}^n} \Delta^{en}_{s,\tau,\chi,\xi}(\mathfrak{t})\\
	&\ \times \int_{ \mathbb{A}_F^{\times}}\cdots\int_{ \mathbb{A}_F^{\times}}\varphi\left(k^{-1}\mathfrak{u}\begin{pmatrix}
	1& &&&\\
	& t_2^{-1}&&&\\
	&&\ddots&&\\
	&&&t_{n-1}^{-1}&\\
	&&&&t_1
	\end{pmatrix}\widetilde{w}n\mathfrak{u}'k\right)d^{\times}t_1\cdots d^{\times}t_{n-1},
	\end{align*}
	where the weighted character $\Delta^{en}_{s,\tau,\chi,\xi}$ is defined to be
	\begin{align*}
	\Delta^{en}_{s,\tau,\chi,\xi}(\mathfrak{t})=\bar{\chi}(t_1)\xi(t_1)\tau(t_1)^{\frac{n-2}{2}}|t_1|_{\mathbb{A}_F}^{\frac{(n-1)(s+1)}{2}}\prod_{i=2}^{n-1}\xi(t_i)\tau(t_i^{\frac{n}2-i})|t_i|_{\mathbb{A}_F}^{[\frac{n+1}2-i](s+1)}.
	\end{align*}
	
	Let $T_*(\mathbb{A}_F^{\times})=\{\diag(1,t_1,t_2,\cdots,t_{n-1})\in T(\mathbb{A}_F):\ t_i\in \mathbb{A}_F^{\times},\ 1\leq i\leq n-1\}.$ Set
	\begin{align*}
	\iota:\ T^*(\mathbb{A}_F^{\times})\longrightarrow T_*(\mathbb{A}_F^{\times}),\ \mathfrak{t}\mapsto \mathfrak{t}^{\iota}=\diag(1,t_2^{-1},t_3^{-1},\cdots,t_{n-1}^{-1},t_1).
	\end{align*}
	
	For any $n\in\mathbb{N}_{\geq 2},$ define $\mathfrak{F}_{\chi,\xi}(x;k,s)=\mathfrak{F}_{\chi,\xi}(x;k,s,\varphi,\Phi,\tau)$ by
	\begin{align*}
	\mathfrak{F}_{\chi,\xi}(x;k,s)&=\int_{N(\mathbb{A}_F)}d\mathfrak{u}\int_{[N^{P}]}d\mathfrak{u}'\int_{ T^*(\mathbb{A}_F^{\times})}\varphi\left(k^{-1}\mathfrak{u}\mathfrak{t}^{\iota}x\mathfrak{u}'k\right)\Delta_{s,\tau,\chi,\xi,n}(\mathfrak{t})d^{\times}\mathfrak{t},
	\end{align*}
	where we write $\delta_n=-\frac{1+(-1)^n}2$ and denote by $\Delta_{s,\tau,\chi,\xi,n}(\mathfrak{t})$ the following character
	\begin{align*}
	\bar{\chi}(t_1)\xi(t_1)^{-\delta_n}\tau(t_1)^{\frac{n-1-\delta_n}{2}}|t_1|_{\mathbb{A}_F}^{\frac{(n-1)(s+1)}{2}}\prod_{i=2}^{n-1}\chi(t_i)\xi(t_i)^{\delta_n}\tau(t_i)^{\frac{n+1-\delta_n}2-i}|t_i|_{\mathbb{A}_F}^{[\frac{n+1}2-i](s+1)}.
	\end{align*}
	Since $[N^{P}]=N^{P}(F)\backslash N^P(\mathbb{A}_F)$ is compact and $\varphi$ is compactly supported, the function $\mathfrak{F}_{\chi,\xi}(x;k,s)$ is well defined for any $\chi,$ $\xi$ and $\Re(s)>1.$
	
	Let $b=\mathfrak{u}\mathfrak{t}\in B(\mathbb{A}_F),$ where $\mathfrak{u}\in N(\mathbb{A}_F),$ $\mathfrak{t}=\diag(t_1,t_2,\cdots, t_n)\in T(\mathbb{A}_F).$ Then
	\begin{align*}
	\mathfrak{F}_{\chi,\xi}(bx;k,s)=\prod_{i=1}^{n}\chi(t_i)\xi(t_i)^{\delta_n}\tau(t_i)^{\frac{n+1-\delta_n}2-i}|t_i|_{\mathbb{A}_F}^{[\frac{n+1}2-i](s+1)}\cdot\mathfrak{F}_{\chi,\xi}(x;k,s).
	\end{align*}
	Since the modular character of $T(\mathbb{A}_F)$ is $\delta_{T(\mathbb{A}_F)}(\mathfrak{t})=\prod_{i=1}^nt_i^{n+1-2i},$ so one has
	\begin{align*}
	\mathfrak{F}_{\chi,\xi}(x;k,s)\in\Ind_{B(\mathbb{A}_F)}^{G(\mathbb{A}_F)}\left(\chi\xi^{\delta_n}\tau^{\lambda_1}|\cdot|_{\mathbb{A}_F}^{\lambda_1s},\cdots,\chi\xi^{\delta_n}\tau^{\lambda_{n-1}}|\cdot|_{\mathbb{A}_F}^{\lambda_{n-1}s},\chi\xi^{\delta_n}\tau^{\lambda_n}|\cdot|_{\mathbb{A}_F}^{\lambda_ns}\right),
	\end{align*}
	where for $1\leq i\leq n,$ $\lambda_i=\frac{n+1-\delta_n}2-i.$ Denote by
	\begin{align*}
	G_{\chi,\xi}(x;s)=G_{\chi,\xi}(x;s,\varphi,\Phi,\tau)=\frac1{c_P}\int_Kf(k,s)\mathfrak{F}_{\chi,\xi}(x;k,s)dk.
	\end{align*}
Then at least formally one can write $I_{\infty,\Reg}(s,\tau)=I_{\infty}^{\Reg}(s)$ as a finite sum: 
	\begin{align*}
I_{\infty,\Reg}(s,\tau)=\sum_{\chi\in \Xi_{\omega,n}}\sum_{\xi\in \Xi_{\tau,2}^n}\int_{N_P(\mathbb{A}_F)}G_{\chi,\xi}(\widetilde{w}n;s)dn,\ \Re(s)>1.
	\end{align*}
	Let $\mathfrak{F}_{1,1,+}(x;k,s)=\mathfrak{F}_{1,1}(x;k,s,|\varphi|,|\Phi|,1)$ and $G_{1,1,+}(x;s)=G_{1,1}(x;s,|\varphi|,|\Phi|,1).$ Then the above interchanging orders of integrals is justified by Fubini's theorem on integral of nonnegative functions. One then has 
	\begin{align*}
I_{\infty,\Reg}^{+}(s,\tau)=\sum_{\chi\in \Xi_{1,n}}\sum_{\xi\in \Xi_{1,2}^n}\int_{N_P(\mathbb{A}_F)}G_{1,1,+}(\widetilde{w}n;s)dn,
	\end{align*}
	where the sums are finite. Then $\int_{N_P(\mathbb{A}_F)}G_{1,1,+}(\widetilde{w}n;s)dn$ converges absolutely and locally normally in $\Re(s)>1$ according to Langlands' theory on intertwining operators. Therefore, by dominant control theorem, $\int_{N_P(\mathbb{A}_F)}G_{\chi,\xi}(\widetilde{w}n;s)dn$ converges absolutely and locally normally in $\Re(s)>1.$ Moreover, apply Langlands' theory and Godement-Jacquet's theory we obtain:
\begin{thmx}\label{aa}
Let notation be as before, then $I_{\infty,\Reg}(s,\tau)$ converges absolutely and locally normally in the domain $\Re(s)>1.$ Moreover, $I_{\infty,\Reg}(s,\tau)$ admits a meromorphic continuation. Precisely, one has 
	\begin{align*}
I_{\infty,\Reg}(s,\tau)\sim \frac{\Lambda(s,\tau)\Lambda(2s,\tau^2)\cdots \Lambda((n-1)s,\tau^{n-1})\Lambda(ns,\tau^n)}{\Lambda(s+1,\tau)\Lambda(2s+1,\tau^2)\cdots \Lambda((n-1)s+1,\tau^{n-1})}.
	\end{align*}
\end{thmx}

	\section{Contributions from $I_{\infty}^{(1)}(s)$}\label{6sec}
In this section, we will deal with $I_{\infty}^{(1)}(s)$ and obtain the holomorphic continuation to eventually. To achieve that, we starting with spectral decomposition via which one can write $I_{\infty}^{(1)}(s)$ as Mellin transforms of Rankin-Selberg convolutions of some parabolic-induced representations. To handle the Rankin-Selberg convolution, local computations are needed. In Archimedean places, it requires some estimate of Whittaker functions; while for the non-archimedean unramified places, one can reduce to the usual Rankin-Selberg convolution multiplied by some extra factor from spectral parameters. This factor, according to GLobal Langlands-Shahidi method for Borel induction, turns out to be the L-function associated to exterior square of the corresponding induced representation. Combining the above work, one gets explicit expression for $I_{\infty}^{(1)}(s)$ when $\Re(s)>1.$ 
	
	One can write the test function $\varphi\in\mathcal{C}_0\left(G(\mathbb{A}_F)\right)$ as a finite linear combination of convolutions $\varphi_1*\varphi_2$ with functions $\varphi_i\in C_c^r\left(G(\mathbb{A}_F)\right),$ whose archimedean components are differentiable of arbitrarily high order $r.$ However, one may not be able to make each $\varphi_1*\varphi_2$ have some component supported in regular elliptic set even if $\varphi$ has such a component. Since we do need such a decomposition for the test function $\varphi,$ we will just take, in this section, our test function $\varphi\in\mathcal{C}_0\left(G(\mathbb{A}_F)\right)$ to be general, namely, we do not require $\varphi$ have a component supported in regular elliptic set. 

In this subsection, steps of proving absolute convergence of $I_{\infty}^{(1)}(s)$ when $\Re(s)$ is large will be indicated. We start with formal expression of $I_{\infty}^{(1)}(s),$ which is defined by 
\begin{equation}\label{{1}}
	I_{\infty}^{(1)}(s)=\int_{Z_G(\mathbb{A}_F)N(F)\backslash G(\mathbb{A}_F)}\int_{N(F)\backslash N(\mathbb{A}_F)}\K_{\infty}(n_1x,x)\theta(n_1)dn_1f(x,s)dx.
\end{equation}

Then we will substitute the spectral expansion of $\K_{\infty}(x,y),$ the kernel function from the continuous part, into \eqref{{1}}, to see $I_{\infty}^{(1)}(s)$ can be written (formally) as  
\begin{align*}
\int_{X_G}\sum_{\chi\in\mathfrak{X}}\sum_{P\in \mathcal{P}}\frac{1}{c_P}\int_{\Lambda^*}\sum_{\phi_1}\sum_{\phi_2}\langle\mathcal{I}_P(\lambda,\varphi)\phi_2,\phi_1\rangle W_{1}(x;\lambda)\overline{W_{2}(x;\lambda)}d\lambda f(x,s)dx,
\end{align*}
where $X_G=Z_G(\mathbb{A}_F)N(\mathbb{A}_F)\backslash G(\mathbb{A}_F)$ and $\mathfrak{X}$ is the (infinite) set of cuspidal data.  See \eqref{11} below for details. To show $I_{\infty}^{(1)}(s)$ is actually well-defined and for our continuation of $I_{\infty}^{(1)}(s)$ in Section \ref{7.2}, we shall prove in this section that the sum 
\begin{align*}
\sum_{\chi\in\mathfrak{X}}\sum_{P\in \mathcal{P}}\frac{1}{c_P}\int_{\Lambda^*}\bigg|\int_{X_G}\sum_{\phi_1}\sum_{\phi_2}\langle\mathcal{I}_P(\lambda,\varphi)\phi_2,\phi_1\rangle W_{1}(x;\lambda)\overline{W_{2}(x;\lambda)}f(x,s)dx\bigg| d\lambda 
\end{align*}
converges. Note that the integral over $X_G$ is a Rankin-Selberg convolution for (non-cuspidal) automorphic representations on $\GL(n),$ which will be handled in Section \ref{6.2.}. To justify the interchange of integrals and summations, we will follow Arthur's strategy from his profound work on trace formula theory, e.g., see \cite{Art78}, \cite{Art79}, \cite{Art80} and \cite{Art81}. Along Arthur's truncation on Eisenstein series, we have (see Lemma \ref{55'}):
\begin{equation}\label{00}
\int_{\mathcal{S}_{0}}\sum_{\chi}\Big|\int_{N(F)\backslash N(\mathbb{A}_F)}\Lambda^T_2\K_{\chi}(nx,x)\theta(n)dn\cdot f(x,s)\Big|dx<\infty,
\end{equation}
where $\mathcal{S}_{0}$ is a Siegel domain, $\Lambda^T_2$ is the truncation operator, and $\Re(s)>1.$

However, our situation is quite different from Arthur's case. One of the major differences is the fundamental domain: noting that $X_G=Z_G(\mathbb{A}_F)N(F)\backslash G(\mathbb{A}_F)$ is typically much larger than the classical region $Z_G(\mathbb{A}_F)G(F)\backslash G(\mathbb{A}_F),$ and $X_G$ cannot be covered by a Siegel domain. To remedy this, we establish a covering result Proposition \ref{27''} to show that 
\begin{equation}\label{000}
\int_{X_G}\sum_{\chi}\Big|\int_{N(F)\backslash N(\mathbb{A}_F)}\Lambda^T_2\K_{\chi}(nx,x)\theta(n)dn\cdot f(x,s)\Big|dx
\end{equation}
can be bounded by a finite linear combination of the left hand side of \eqref{00} over different Siegel domains. One then deduces that \eqref{000} converges. However, due to the Fourier transform, the geometry is quite different from Arthur's case. We thus construct a truncation operator, which fits our particular setting, acting on the geometric expansion of kernel function; and we then show the two truncations do match on $Z_G(\mathbb{A}_F)G(F)\backslash G(\mathbb{A}_F)$ (see Proposition \ref{29'}). So we do a geometric computation using Arthur's techniques to get an explicit expression in Proposition \ref{30'}, which consists of three parts: \textit{constant term}, \textit{apidly decaying terms} and \textit{polynomially increasing terms}. The last step is to use properties of Fourier coefficients to exclude the existence of polynomially increasing terms, so that one can take limit to recover from the truncation to the original expression \eqref{{1}}. We then summarize the final result on convergence as Theorem \ref{39'} at the end of this section.
	
	\subsection{Spectral Decomsition of the Kernel Function}\label{7.1}
	In this subsection, we review briefly the spectral theory of automorphic representation of reductive groups, and then apply the results to $\K_{Eis}$. Denote by $H$ a general reductive group and $P$ a standard parabolic subgroup of $H.$ Let $M_P$ (resp. $N_P$) be the Levi component (resp. unipotent radical) of $P.$ Then by spectral theory, the decomposition of the Hilbert space $L^2\left(Z_{H}(\mathbb{A}_F)N_P(\mathbb{A}_F)M_P(F)\setminus H(\mathbb{A}_F)\right)$ into right $H(\mathbb{A}_F)$-invariant subspaces is determined by the spectral data $\chi=\{(M,\sigma)\},$ where the pair $(M,\sigma)$ consists of a Levi subgroup $M$ of $H$ and a cuspidal representation $\sigma\in\mathcal{A}_0\left(Z_H(\mathbb{A}_F)\setminus M^1(\mathbb{A}_F)\right);$ the class $(M,\sigma)$ derives from the equivalence relation $(M,\sigma)\sim(M',\sigma')$ if and only if $M$ is conjugate to $M'$ by a Weyl group element $w,$ and $\sigma'=\sigma^w$ on $Z_H(\mathbb{A}_F)\setminus M^1(\mathbb{A}_F).$ Let $\mathfrak{X}$ be the set of equivalence classes $\chi=\{(M,\sigma)\}$ of these pairs, we thus have
	\begin{equation}\label{49}
	L^2\left(P\right):=L^2\left(Z_{H}(\mathbb{A}_F)N_P(\mathbb{A}_F)M_P(F)\setminus H(\mathbb{A}_F)\right)=\bigoplus_{\chi\in\mathfrak{X}}L^2\left(P\right)_{\chi},
	\end{equation}
	where $L^2\left(P\right)_{\chi}$ consists of functions $\phi\in L^2\left(Z_{H}(\mathbb{A}_F)N_P(\mathbb{A}_F)M_P(F)\setminus H(\mathbb{A}_F)\right)$ such that: for each standard parabolic subgroup $Q$ of $G,$ with $Q\subset P,$ and almost all $x\in H(\mathbb{A}_F),$ the projection of the function 
	$$
	m\mapsto x.\phi_{Q}(m)=\int_{N_Q(F)\setminus N_Q(\mathbb{A}_F)}\phi(nmx)dn
	$$
	onto the space of cusp forms in $L^2\left(Z_{H}(\mathbb{A}_F)M_Q(F)\setminus M_Q^1(\mathbb{A}_F)\right)$ transforms under $M_Q^1(\mathbb{A}_F)$ as a sum of representations $\sigma,$ in which $(M_Q,\sigma)\in\chi.$ If there is no such pair in $\chi,$ $x.\phi_{Q}$ will be orthogonal to $\mathcal{A}_0\left(Z_{H}(\mathbb{A}_F)M_Q(F)\setminus M_Q^1(\mathbb{A}_F)\right).$ Denote by $\mathcal{H}_P$ the space of such $\phi$'s. Let $\mathcal{H}_{P,\chi}$ be the subspace of $\mathcal{H}_P$ such that for any $(M,\sigma)\notin \chi,$ with $M=M_{P_1}$ and $P_1\subset P,$ we have
	$$
	\int_{M(F)\setminus M(\mathbb{A}_F)^1}\int_{N_{P_1}(F)\setminus N_{P_1}(\mathbb{A}_F)}\psi_0(m)\phi(nmx)dn=0,
	$$
	for any $\psi_0\in L^2_{cusp}\left(M(F)\setminus M(\mathbb{A}_F)^1\right)_{\sigma},$ and almost all $x.$ This leads us to Langlands' result to decompose $\mathcal{H}_P$ as an orthogonal direst sum $\mathcal{H}_P=\bigoplus_{\chi\in\mathfrak{X}}\mathcal{H}_{P,\chi}.$ Let $\mathcal{B}_P$ be an orthonormal basis of $\mathcal{H}_P,$ then we can choose $\mathcal{B}_P=\bigcup_{\chi\in\mathfrak{X}}\mathcal{B}_{P,\chi},$ where $\mathcal{B}_{P,\chi}$ is an orthonormal basis of the Hilbert space $\mathcal{H}_{P,\chi}.$ We may assume that vectors in each $\mathcal{B}_{P,\chi}$ are $K$-finite and are pure tensors.
	
	Let $H^1(\mathbb{A}_F)=\{g\in H(\mathbb{A}_F): |\lambda(g)|_{\mathbb{A}_F}=1,\ \forall\ \lambda\in X(H)_F\},$ where $X(H)_F$ is space set of $F$-rational characters of $H.$ Denote by $R$ the regular representation of $H(\mathbb{A}_F)^1$ on the Hilbert space $L^2\left(H(F)\setminus H(\mathbb{A}_F)^1\right),$ $R(\varphi)$ the operator with respect to $\varphi.$
	Let $R_P(\varphi)$ be the operator $R(\varphi)$ restricted to $L^2\left(P\right),$ and denote by $\K_{P}(x,y)$ the kernel of $R_P(\varphi).$ Then a direct computation shows that
	\begin{equation}\label{ker}
	\K_{P}(x,y)=\K_{\varphi,P}(x,y)=\int_{N_P(\mathbb{A}_F)}\sum_{\mu\in Z_H(F)\setminus M_{P}(F)}\varphi(x^{-1}\mu ny)dn.
	\end{equation}
	Then the decomposition \eqref{49}  induces a corresponding decomposition of kernels
	\begin{equation}\label{52'}
	\K_{P}(x,y)=\sum_{\chi\in\mathfrak{X}}\K_{P,\chi}(x,y),
	\end{equation}
	where each $\K_{P,\chi}(x,y)$ can be written explicitly in terms of Eisenstein series:
	\begin{equation}\label{53'}
	K_{P,\chi}(x,y)=\sum_{Q\subset P}\frac{1}{n_Q^P}\int_{i\mathfrak{a}^*_{Q}/i\mathfrak{a}^*_{H}}\sum_{\phi\in\mathcal{B}_{Q,\chi}}E_Q^P\left(x,\mathcal{I}_{Q}(\lambda, \varphi)\phi,\lambda\right)\overline{E_Q^P\left(y,\phi,\lambda\right)}d\lambda,
	\end{equation}
	where $n_Q^P$ is the number of of chambers in $\mathfrak{a}_Q/\mathfrak{a}_P,$ $\mathcal{B}_{Q,\chi}$ is an orthonormal basis of the Hilbert space $\mathcal{H}_{Q,\chi},$ and
	\begin{align*}
	&E_Q^P\left(x,\phi,\lambda\right)=\sum_{\delta\in M_{P}(F)\cap Q(F)\setminus M_P(F)}\phi(\delta x)e^{(\lambda+\rho_Q)H_Q(\delta x)},\\
	&E_Q^P\left(x,\mathcal{I}_{Q}(\lambda, \varphi)\phi,\lambda\right)=\int_{H(\mathbb{A}_F)}\varphi(y)E_Q^P\left(xy,\phi,\lambda\right)dy.
	\end{align*}

	\subsection{Spectral Expansion of $I_{\infty}^{(1)}(s)$} By definition, we have, by \eqref{{1}}, that 
	\begin{align*}
	I_{\infty}^{(1)}(s)=\int_{Z_G(\mathbb{A}_F)N(\mathbb{A}_F)\backslash G(\mathbb{A}_F)}\int_{[N]}\int_{[N]}\K_{\infty}(n_1x,n_2x)\theta(n_1)\bar{\theta}(n_2)dn_1dn_2f(x,s)dx.
	\end{align*}
	Denote by $\widehat{\K}_{\infty}(x,y)$ the Fourier expansion of $\K_{\infty}(x,y),$ namely,
	\begin{equation}\label{51}
	\widehat{\K}_{\infty}(x,y):=\int_{N(F)\backslash N(\mathbb{A}_F)}\int_{N(F)\backslash N(\mathbb{A}_F)}\K_{\infty}(n_1x,n_2y)\theta(n_1)\bar{\theta}(n_2)dn_1dn_2.
	\end{equation}
	For our particular purpose here, we take in Section \ref{7.1} that $H=G$ and the standard parabolic subgroup $P=G$ as well. Then a further computation leads to the explicit spectral decomposition of $\K_{\infty}(x,y),$ expressing it as 
	\begin{equation}\label{52}
	\sum_{\chi\in\mathfrak{X}}\sum_{P\in \mathcal{P}}\frac{1}{k_P!(2\pi)^{k_P}}\int_{i\mathfrak{a}^*_P/i\mathfrak{a}^*_G}\sum_{\phi\in \mathfrak{B}_{P,\chi}}E(x,\mathcal{I}_P(\lambda,\varphi)\phi,\lambda)\overline{E(y,\phi,\lambda)}d\lambda,
	\end{equation}
   where $\mathcal{P}$ is the set of standard parabolic subgroups which are not $G;$ and for any such $P,$ $k_P$ is the number of blocks of the Levi part of $P.$ Also, \eqref{52} converges absolutely. (see \cite{Art79}). Since $N(F)\backslash N(\mathbb{A}_F)$ is compact, $\widehat{\K}_{\infty}(x,y)$ is then well defined. 
    \begin{lemma}\label{25'}
    Let notation be as before. Then one can interchange the integrals in the definition of $\widehat{\K}_{\infty}(x,y),$ namely, one has 
	\begin{equation}\label{53}
	\widehat{\K}_{\infty}(x,y)=\sum_{\chi\in\mathfrak{X}}\sum_{P\in \mathcal{P}}\frac{1}{k_P!(2\pi)^{k_P}}\int_{i\mathfrak{a}^*_P/i\mathfrak{a}^*_G}\sum_{\phi\in \mathfrak{B}_{P,\chi}}W_{\Eis,1}(x;\lambda)\overline{W_{\Eis,2}(y;\lambda)}d\lambda,
	\end{equation}
	where the Fourier coefficient $W_{\Eis,1}(x;\lambda)=W_{Eis}(x,\mathcal{I}_P(\lambda,\varphi)\phi,\lambda)$ is defined by
	\begin{align*}
	W_{Eis}(x,\mathcal{I}_P(\lambda,\varphi)\phi,\lambda):=\int_{N(F)\backslash N(\mathbb{A}_F)}E(n_1x,\mathcal{I}_P(\lambda,\varphi)\phi,\lambda)\theta(n_1)dn_1;
	\end{align*}
	and similarly, $W_{\Eis,2}(y;\lambda)=W_{\Eis}(y,\phi,\lambda)$ is given as
	\begin{align*}
	W_{\Eis}(y,\phi,\lambda):=\int_{N(F)\backslash N(\mathbb{A}_F)}E(n_2y,\phi,\lambda)\theta(n_2)dn_2.
	\end{align*}
	\end{lemma}
    \begin{proof}
    The main idea of the proof is similar to that in \cite{Art78} (see p. 928-934).
    For any $P\in \mathcal{P}$, let $c_P=k_P!(2\pi)^{k_P}.$ Substitute \eqref{52} into \eqref{51} to get a formal expansion of $\widehat{\K}_{\infty}(x,y),$ which is clearly dominated by the following formal expression
    \begin{align*}
    \int_{[N]}\int_{[N]}\sum_{\chi\in\mathfrak{X}}\sum_{P\in \mathcal{P}}\frac{1}{c_P}\bigg|\int_{i\mathfrak{a}^*_P/i\mathfrak{a}^*_G}\sum_{\phi\in \mathfrak{B}_{P,\chi}}E(x,\mathcal{I}_P(\lambda,\varphi)\phi,\lambda)\overline{E(y,\phi,\lambda)}d\lambda\bigg|dn_1dn_2.
    \end{align*}
    Denote by $J_G$ the above integral. We will show $J_G$ is finite, hence \eqref{53} is well defined. One can write the test function $\varphi$ as a finite linear combination of convolutions $\varphi_1*\varphi_2$ with functions $\varphi_i\in C_c^r\left(G(\mathbb{A}_F)\right),$ whose archimedean components are differentiable of arbitrarily high order $r.$ Then one applies H\"older inequality to it. Clearly it is enough to deal with the special case that $\varphi=\varphi_j*\varphi_j^*,$ where $\varphi_j^*(x)=\overline{\varphi_j(x^{-1})},$ and $x=y.$ Note that $\mathfrak{B}_{P,\chi}$ is finite due to the $K$-finiteness assumption, and Eisenstein series converge absolutely for our $\lambda,$ hence the integrand 
    \begin{align*}
    \sum_{\phi\in \mathfrak{B}_{P,\chi}}E(x,\mathcal{I}_P(\lambda,\varphi)\phi,\lambda)\overline{E(x,\phi,\lambda)}=\sum_{\phi\in \mathfrak{B}_{P,\chi}}E(x,\mathcal{I}_P(\lambda,\varphi_j)\phi,\lambda)\overline{E(x,\mathcal{I}_P(\lambda,\varphi_j)\phi,\lambda)}
    \end{align*}
    is well defined and obviously nonnegative. In fact, the double integral over $\lambda$ and $\phi$ can be expressed as an increasing limit of nonnegative functions, each of which is the kernel of the restriction of $R(\varphi_j*\varphi_j^*),$ a positive semidefinite operator, to an invariant subspace. Since this limit is bounded by the nonnegative function 
    $$
    \K_j(x,x)=\sum_{\gamma\in Z(F)\backslash G(F)}\varphi_j*\varphi_j^*(x^{-1}\gamma x),
    $$
    and the domain $[N]=N(F)\backslash N(\mathbb{A}_F)$ is compact, the integral $J_G$ converges.
    \end{proof}
    Since $\mathfrak{B}_{P,\chi}$ is finite, and Eisenstein series converges absolutely for any $\lambda\in i\mathfrak{a}^*_P/i\mathfrak{a}^*_G,$ we can apply spectral decomposition to get 
    \begin{align*}
    E(x,\mathcal{I}_P(\lambda,\varphi)\phi_2,\lambda)=\sum_{\phi_1\in \mathfrak{B}_{P,\chi}}\langle\mathcal{I}_P(\lambda,\varphi)\phi_2,\phi_1\rangle E(x,\phi_1,\lambda),\quad \forall \ \phi_2\in \mathfrak{B}_{P,\chi}.
    \end{align*}
	For $1\leq i\leq 2,$ define the Whittaker function associated to $\phi_{i}$ parameterized by $\lambda\in i\mathfrak{a}^*_P/i\mathfrak{a}^*_G$ as 
	\begin{align*}
	W_{i}\left(x,\lambda\right)=W_{i}\left(x,\phi_{\alpha},\lambda\right):=\int_{N(\mathbb{A}_F)}\phi_{i}(w_0nx)e^{(\lambda+\rho_P)H_P(w_0 nx)}\theta(n)dn,
	\end{align*}
	where $w_0$ is the longest element in the Weyl group $W_n.$ 
	
	Since the residual spectrum is degenerate, that is, has no Whittaker model, the integral is zero unless the representation is cuspidal. Hence, by Bruhat decomposition one can write the non-constant terms $W_{\Eis,i}(x;\lambda)$ in terms of $W_{i}\left(x,\lambda\right).$ Set $X_G=Z_G(\mathbb{A}_F)N(F)\backslash G(\mathbb{A}_F),$ $c_P=k_P!(2\pi)^{k_P},$ and $\Lambda^*=i\mathfrak{a}^*_P/i\mathfrak{a}^*_G.$ Then by \eqref{53} one can rewrite (at least formally) $I_{\infty}^{(1)}(s)$ as
	\begin{equation}\label{11}
	\int_{X_G}\sum_{\chi\in\mathfrak{X}}\sum_{P\in \mathcal{P}}\frac{1}{c_P}\int_{\Lambda^*}\mathop{\sum\sum}_{\phi_1, \phi_2}\langle\mathcal{I}_P(\lambda,\varphi)\phi_2,\phi_1\rangle W_{1}(x;\lambda)\overline{W_{2}(x;\lambda)}d\lambda f(x,s)dx,
	\end{equation}
	where $\phi_i\in \mathfrak{B}_{P,\chi},$ $1\leq i\leq 2.$
	 
	To compute the above integral, one nice approach is to change the order of integrals so that we integrate over $X_G$ first, because this type of integral is exactly integral representation of Rankin-Selberg convolutions (parametrized by continuous spectrum). To verify that one can indeed interchange the integrals, we utilize Arthur's truncation operator $\Lambda^T.$ To start with, let us briefly recall its definition. 
	
	For any parabolic subgroup $P$ of $G,$ let $\widehat{\tau}_P$ be the characteristic function of the subset $\{t\in \mathfrak{a}_P:\ \varpi(t)>0,\ \forall\ \varpi\in\widehat{\Delta}_P\}$ of $\mathfrak{a}_P.$ Let $T\in \mathfrak{a}_0^+,$ we say $T$ is suitably regular if $T(\alpha^{\vee})$ is large, for each simple coroot $\alpha^{\vee}.$ For any suitably regular point $T\in\mathfrak{a}_0^{+}$ and any function $\phi\in\mathcal{B}_{loc}\left(Z_G(\mathbb{A}_F)G(F)\backslash G(\mathbb{A}_F)\right),$ define the truncation function $\Lambda^T\phi$ to be the function in $\mathcal{B}_{loc}\left(Z_G(\mathbb{A}_F)G(F)\backslash G(\mathbb{A}_F)\right)$ such that
	\begin{align*}
	\Lambda^T\phi(x)=\sum_{P}(-1)^{\dim(A_P/A_G)}\sum_{\delta\in P(F)\backslash G(F)}\widehat{\tau}_P\left(H_P(\delta x)-T\right) \int_{[N_P]}\phi(n\delta x)dn.
	\end{align*}
	The inner sum may be taken over a finite set depending on $x,$ while the integrand is a bounded function of $n.$
	
	Before moving on, we still need to choose a height function $\|\cdot\|$ on $G(\mathbb{A}_F)$ to describe properties of the truncation operator $\Lambda^T$ quantitatively.
	
	For any $x=(x_v)_v\in G(\mathbb{A}_F).$ We define $\|x_v\|_v=\max_{i,j}|x_{i,j,v}|_v$ if $v<\infty;$ and 
	$$
	\|x_v\|_v=\Big[\sum_{i,j}|x_{i,j,v}|_v^2\Big]^{1/2},\ \text{if}\ v\mid \infty.
	$$
	Then $\|x_v\|_v=1$ for almost all $v.$ The height function $\|x\|=\prod_v\|x_v\|_v$ is therefore well defined by a finite product. Then one has $\|xy\|\leq \|x\|\cdot\|y\|,$ $\forall$ $x,y\in G(\mathbb{A}_F).$ Also one can check that there is some absolute constants $C_0$ and $N_0$ such that for any $x\in G(\mathbb{A}_F),$ we have $\|x\|^{-1}\leq C_0\|x\|^{N_0},$ and
	\begin{align*}
	\#\{x\in G(F):\ \|x\|\leq t\}\leq C_0t^{N_0},\ t\geq 0.
	\end{align*}
	Note that the test function $\varphi$ is compactly supported, then one sees that 
	$$
	|\K(x,y)|=|\K_{\varphi}(x,y)|\leq \sum_{\substack{\gamma\in Z_G(F)\backslash G(F)\\ \|\gamma\|\leq \|x\cdot\supp\varphi\cdot y^{-1}\|}}\sup_{g\in Z_G(\mathbb{A}_F)\backslash G(\mathbb{A}_F)}|\varphi(g)|.
	$$
	Hence $|\K(x,y)|\ll_{\varphi}\#\{\gamma\in G(F):\ \|\gamma\|\leq \|x\cdot\supp\varphi\cdot y^{-1}\|\}\ll_{\varphi}\|x\|^{N_0'}\cdot\|y\|^{N_0'},$ where $N_0'$ is some absolute constant, i.e., independent of the choice of test function $\varphi.$ Thus there exists some constant $c(\varphi)$ such that $|\K(x,y)|\leq c(\varphi)\|x\|^{N_0'}\cdot\|y\|^{N_0'}.$
	
	Now we consider derivatives of the kernel $\K(x,y).$ Suppose $X$ and $Y$ are left invariant differential operators on $G(\mathbb{A}_{F,\infty})$ of degrees $d_1$ and $d_2.$ Suppose also that the test function $\varphi\in C_c^r\left(G(\mathbb{A}_F)\right),$ for some large positive $r.$ For any cuspidal datum $\chi\in\mathfrak{X},$ define the corresponding kernel function as
	\begin{equation}\label{56'''}
	\K_{\chi}(x,y)=\sum_{P\in \mathcal{P}}\frac{1}{k_P!(2\pi)^{k_P}}\int_{i\mathfrak{a}^*_P/i\mathfrak{a}^*_G}\sum_{\phi\in \mathfrak{B}_{P,\chi}}E(x,\mathcal{I}_P(\lambda,\varphi)\phi,\lambda)\overline{E(y,\phi,\lambda)}d\lambda,
	\end{equation}
	which is convergent absolutely according to the proof of Lemma \ref{25'}. Then there exists some function $\varphi_{X,Y}\in C_c^{r-d_1-d_2}\left(G(\mathbb{A}_F)\right)$ such that its corresponding kernel function $\K_{\chi,\varphi_{X,Y}}(x,y)$ is equal to $XY\K_{\chi}(x,y),$ for all $x,y\in G(\mathbb{A}_F).$ Then one can apply the above estimate for kernel functions to obtain
	\begin{equation}\label{55}
	\sum_{\chi}\Big|XY\K_{\chi}(x,y)\Big|\leq c(\varphi_{X,Y})\|x\|^{N_0'}\cdot\|y\|^{N_0'},\ \forall\ x,y\in G(\mathbb{A}_F).
	\end{equation}
	Also, for any function $H(x,y)$ in $\mathcal{B}_{loc}\left(G(F)\backslash G(\mathbb{A}_F)^1 \right),$ define the partial Fourier transform of $H(x,y)$ with respect to the $x$-variable as 
	\begin{align*}
	\mathcal{F}_1H(x,y):=\int_{N(F)\backslash N(\mathbb{A}_F)}H(n_1x,y)\theta(n_1)dn_1.
	\end{align*}
	This is well defined since $H(x,y)$ is $N(F)$-invariant on the $x$-variable, and the quotient $N(F)\backslash N(\mathbb{A}_F)$ is compact. 
	
	Let $A_{F,\infty}=Z_G(\mathbb{A}_{F,\infty})\backslash T(\mathbb{A}_{F,\infty}).$ For any $c>0,$ let $A_{c,\infty}$ be set consisting of all
	\begin{equation}\label{57'''}
	a=\begin{pmatrix}
	a_1a_2\cdots a_{n-1}& &&&\\
	&a_1a_2\cdots a_{n-2}&&&\\
	&&\ddots&&\\
	&&&a_1&\\
	&&&&1
	\end{pmatrix}\in A(\mathbb{A}_{F,\infty})
	\end{equation}
	with $|a_i|_{\infty}\geq c$ for $1\leq i\leq n-1.$ Clearly $A_{c,\infty}\subsetneq Z_G(\mathbb{A}_{F,\infty})\backslash T(\mathbb{A}_{F,\infty}),$ $\forall$ $c>0.$
	
	A Siegel set $\mathcal{S}_{c}$ for $c>0$ is defined to be the set of elements of the form $nak,$ $n\in \Omega\subseteq N(\mathbb{A}_F),$ compact, such that $N(F)\Omega=N(\mathbb{A}_F);$ $k\in K,$ and $a\in A_{\geq c}=\{a\in A(\mathbb{A}_F):\ |\alpha_i(a)|\geq c,\ 1\leq i\leq n-1\}.$ Then from reduction theory (see \cite{Bor69}), there exists some $c_0>0$ such that $G(\mathbb{A}_F)=Z_G(\mathbb{A}_F)G(F)\mathcal{S}_{c_0}.$ Denote by $S_0=S_{c_0},$ and we may assume that $0<c_0<1.$ Let $R$ be a function on $Z_G(\mathbb{A}_F)G(F)\backslash G(\mathbb{A}_F),$ we say that $R$ is slowly increasing if there exists some $r>0$ and $C>0$ such that $R(x)\leq C\|x\|^r,$ $\forall$ $x\in G(\mathbb{A}_F);$ we say $R$ is rapidly decreasing if for any positive integer $N$ and any Siegel set $\mathcal{S}_c$ for $G(\mathbb{A}_F),$ there is a positive constant $C$ such that $|R(x)|\leq C\|x\|^{-N}$ for every $x\in \mathcal{S}_c.$ 
	
	We will apply the truncation operator $\Lambda^T$ to the kernel functions $\K_{\chi}(x,y)$ and show that it is absolutely integrable twisted by any slowing increasing functions over $A_{h,\infty}A_{\varphi,fin}\cdot K$ for some $h>0,$ where $A_{\varphi,fin}$ be a compact subgroup of $A_{fin}=Z_G(\mathbb{A}_{F,fin})\backslash T(\mathbb{A}_{F,fin})$ depending on $\varphi.$ 
	\begin{lemma}\label{26'}
	Let notation be as above. Then $\sum_{\chi}\Big|\mathcal{F}_1\Lambda^T_2\K_{\chi}(x,x)\Big|$ is rapidly decreasing on $\mathcal{S}_0.$ Moreover, let $R$ be a slowly increasing function on $G(\mathbb{A}_F).$ Then we have
	\begin{equation}\label{55'}
	\int_{\mathcal{S}_{0}}\sum_{\chi}\Big|\mathcal{F}_1\Lambda^T_2\K_{\chi}(x,x)\cdot R(x)\Big|dx<\infty,
	\end{equation}
	where $\chi$ runs over all the equivalent classes of cuspidal datum.
	\end{lemma}
	\begin{proof}
	Clearly for any given $x\in G(\mathbb{A}_F),$ $\mathcal{F}_1\K_{\chi}(x,y)$ is a well defined function (with respect to $y$) in $\mathcal{B}_{loc}\left(Z_G(\mathbb{A}_F)G(F)\backslash G(\mathbb{A}_F)\right).$ Then according to Lemma 1.4 in \cite{Art80} one sees that given a Siegel set $\mathcal{S},$ positive integers $N$ and $N_1,$ and an open compact subgroup $K_0$ of $G(\mathbb{A}_{F,fin}),$ one can choose a finite set $\{X_i\}$ of left invariant differential operators on $Z_G(\mathbb{A}_{F,\infty})\backslash G(\mathbb{A}_{F,\infty})$ and a positive integer $r$ with the property that if $(\Omega,d\omega)$ is a measure space and $\phi(\omega)\mapsto \phi(\omega,x)$ is any measurable function from $\Omega$ to $C^r\left(Z_G(\mathbb{A}_{F})G(F)\backslash G(\mathbb{A}_F)/K_0\right),$ then 
	\begin{equation}\label{56}
	\sup_{x\in \mathcal{S}}\left(\|x\|^N \int_{\Omega}|\Lambda^T\phi(\omega,x)|d\omega\right)\leq \sup_{x\in \mathcal{S}}\left(\|x\|^{-N_1}\sum_{i} \int_{\Omega}|X_i\phi(\omega,x)|d\omega\right).
	\end{equation}
	 Since our test functions lie in the Hecke algebra, we may assume that $\varphi$ is biinvariant under $K_0,$ where $K_0$ is an open compact subgroup of $G(\mathbb{A}_{F,fin}).$ Take the measure space $(\Omega,d\omega)$ to be $(\mathfrak{X},d),$ where $d$ is the discrete measure. Also, set $N_1=N_0'$ and $N$ large, then we can find a finite set $\{Y_i\}$ of left invariant differential operators on $G(\mathbb{A}_{F,\infty})$ such that \eqref{56} holds, which under our current particular choices becomes 
	\begin{align*}
	\sup_{y\in \mathcal{S}_0}\left(\|y\|^N \sum_{\chi}\Big|\Lambda^T_2\K_{\chi}(x,y)\Big|\right)&\leq \sup_{y\in G(\mathbb{A})}\left(\sum_{\chi}\|y\|^{-N_1} \Big|\sum_i(Y_i)_2\K_{\chi}(x,y)\Big|\right)\\
	&\leq \sup_{y\in G(\mathbb{A})}\left(\sum_i\|y\|^{-N_1} \sum_{\chi}\Big|(Y_i)_2\K_{\chi}(x,y)\Big|\right),
	\end{align*} 
	where $(Y_i)_2$ above indicates that the differential operator acts on the $y$-variable. By the estimate \eqref{55}, the right hand side of the above inequality is bounded by
	\begin{align*}
	\sup_{y\in G(\mathbb{A})}\left(\sum_i\|y\|^{-N_1}c(\varphi_{I,Y_i})\|x\|^{N_0'}\cdot\|y\|^{N_0'}\right)\leq \sum_{i}c(\varphi_{I,Y_i})\|x\|^{N_1},
	\end{align*} 
	for any $x\in \mathcal{S}_0.$ Then setting $x=y,$ $V_0=\vol\left(N(F)\backslash N(\mathbb{A}_F)\right),$ we see that 
	\begin{align*}
	\sum_{\chi}\Big|\mathcal{F}_1\Lambda^T_2\K_{\chi}(x,x)\Big|\leq V_0\sum_{\chi}\Big|\Lambda^T_2\K_{\chi}(x,x)\Big|\leq V_0\sum_{i}c(\varphi_{I,Y_i})\|x\|^{N_1-N},
	\end{align*}
	for any $x\in \mathcal{S}_0,$ where $\chi$ runs over all the equivalent classes of cuspidal datum. Also, since $R(x)$ is slowly increasing, then by taking $N$ to be large enough we conclude that $\sum_{\chi}\Big|\mathcal{F}_1\Lambda^T_2\K_{\chi}(x,x)\Big|\cdot |R(x)|$ is a bounded function on $\mathcal{S}_{0},$ hence it is integrable.
	\end{proof}
	
	Let $S_n$ be the permutation group on $n$ letters. Let $\sigma\in S_n.$ For any $\mathfrak{a}\in A_{F,\infty},$ write it in its Iwasawa normal form given in \eqref{57'''}. Let $\mathfrak{a}'_{i}=a_1a_2\cdots a_{n-i},$ $1\leq i\leq n-1;$ and set $\mathfrak{a}'_{n}=1.$ For any $1\leq i\leq n,$ let $\mathfrak{a}_{i}=a_{1}^{-1}\cdots a_{\sigma(n)}^{-1}\mathfrak{a}'_{i}.$ Set $\mathfrak{a}_{\sigma}=\diag(\mathfrak{a}_{\sigma(1)},\mathfrak{a}_{\sigma(2)},\cdots,\mathfrak{a}_{\sigma(n)}).$ Clearly $\mathfrak{a}_{\sigma(n)}=1,$ and each $\mathfrak{a}_{\sigma(i)}$ is a rational function of $a_1,\cdots, a_{n-1}.$ Moreover, each $\mathfrak{a}_{\sigma(i)}$ is a monomial, $1\leq i\leq n-1.$ Note that $\mathfrak{a}_{\sigma}$ is of the form in \eqref{57'''}. So $\sigma$ induces a well defined map $A_{F,\infty}\longrightarrow A_{F,\infty},$ $\mathfrak{a}\mapsto \mathfrak{a}_{\sigma}.$ Denote by $\iota_{\sigma}$ this map. Then $\iota_{\sigma}$ is actually a bijection from $A_{F,\infty}$ to itself. Write $\mathfrak{a}_{\sigma(i)}=\mathfrak{a}_{\sigma(i)}(a_1,a_2,\cdots,a_{n-1})$ to indicate that $\mathfrak{a}_{\sigma(i)}$ is a fractional function of $a_1, a_2,\cdots a_{n-1}.$ For any $c>0,$ let
	\begin{align*}
	\mathcal{H}^c_{\sigma}=\Bigg\{(a_1,a_2,\cdots,a_{n-1}):\ \Big|\frac{\mathfrak{a}_{\sigma(i)}(a_1,a_2,\cdots,a_{n-1})}{\mathfrak{a}_{\sigma(i+1)}(a_1,a_2,\cdots,a_{n-1})}\Big|_{\infty}\geq c,\ 1\leq i\leq n-1\Bigg\}.
	\end{align*}
	Let $S_{n}^{reg}:=\{\sigma\in S_n:\ \sigma(i)\leq \sigma(i+1),\ 3\leq i\leq n-1\}.$ Let $0<c\leq 1.$ Denote by \begin{align*}
	\iota(A_{F,\infty})^{reg}=\big\{(a_1,a_2,\cdots,a_{n-1})\in \mathbb{G}_m(\mathbb{A}_{F,\infty})^{n-1}:\ \big|a_i\big|_{\infty}\geq c,\ 1\leq i\leq n-3\big\}.
	\end{align*}
	\begin{lemma}\label{26''}
		Let notation be as above. Let $n\geq 4.$ Then one has
		\begin{equation}\label{63'''}
		\iota(A_{F,\infty})^{reg}\subseteq \bigcup_{\sigma\in S_{n}^{reg}}\mathcal{H}^c_{\sigma}.
		\end{equation}
	\end{lemma}
\begin{proof}
For any $\sigma\in S_{n}^{reg},$ if $\sigma(2)<\sigma(3),$ then define $i_{\sigma}=2;$ if $\sigma(i_0)<\sigma(2)<\sigma(i_{0}+1)$ for some $3\leq i_0\leq n-1,$ then define $i_{\sigma}=i_0;$ if $\sigma(2)>\sigma(n),$ then define $i_{\sigma}=n.$ Since such an $i_0$ (if exists) is unique, then $i_{\sigma}$ is well defined. It induces an $n$ to $1$ surjection $\iota_{reg}:$ $S_{n}^{reg}\rightarrow \{2,3,\cdots,n\},$ given by $\sigma\mapsto i_{\sigma}.$ For $2\leq i\leq n,$ let $S_{n,i}^{reg}$ be the fibre at $i_{\sigma}=i,$ i.e., $S_{n,i}^{reg}=\iota_{reg}(i)^{-1}.$ Then
\begin{equation}\label{64}
\bigcup_{\sigma\in S_{n}^{reg}}\mathcal{H}^c_{\sigma}=\bigcup_{i=2}^n\bigcup_{\sigma\in S_{n,i}^{reg}}\mathcal{H}^c_{\sigma}.
\end{equation}
Let $\iota_2(A_{F,\infty})^{reg}=\big\{(a_1,a_2,\cdots,a_{n-1})\in \iota(A_{F,\infty})^{reg}:\ \big|a_{n-2}\big|_{\infty}\geq c\big\}.$ Denote by $\iota_n(A_{F,\infty})^{reg}=\big\{(a_1,a_2,\cdots,a_{n-1})\in \iota(A_{F,\infty})^{reg}:\ \big|a_{1}a_{2}\cdots a_{n-2}\big|_{\infty}\leq 1/c\big\}.$ Define, for any $3\leq i\leq n-1,$ that $\iota_i(A_{F,\infty})^{reg}=\big\{(a_1,a_2,\cdots,a_{n-1})\in \iota(A_{F,\infty})^{reg}:\ \big|a_{n-i}a_{n-i+1}\cdots a_{n-2}\big|_{\infty}\geq c,\ \big|a_{n-i+1}a_{n-i+2}\cdots a_{n-2}\big|_{\infty}\leq 1/c\big\}.$ Note that, for any $2\leq i\leq n,$ $\iota_i(A_{F,\infty})^{reg}$ is well defined.
\begin{claim}\label{28''}
Let notation be as before. Then one has, for any $2\leq i\leq n,$ that 
\begin{equation}\label{63''}
\iota_i(A_{F,\infty})^{reg}\subseteq \bigcup_{\sigma\in S_{n,i}^{reg}}\mathcal{H}^c_{\sigma}.
\end{equation}
\end{claim}
A straightforward combinatorial analysis shows that $\iota(A_{F,\infty})^{reg}$ is contained in the union of $\iota_i(A_{F,\infty})^{reg}$ over $2\leq i\leq n.$ Hence \eqref{63'''} comes from \eqref{64} and \eqref{63''}:
\begin{align*}
\iota(A_{F,\infty})^{reg}\subseteq \bigcup_{i=2}^n\iota_i(A_{F,\infty})^{reg}\subseteq\bigcup_{i=2}^n\bigcup_{\sigma\in S_{n,i}^{reg}}\mathcal{H}^c_{\sigma}=\bigcup_{\sigma\in S_{n}^{reg}}\mathcal{H}^c_{\sigma}.
\end{align*}
\end{proof}
\begin{proof}[Proof of Claim \ref{28''}]
For any $2\leq i\leq n$ and any $\sigma\in S_{n,i}^{reg},$ recall that $\mathcal{H}^c_{\sigma}$ is equal to 
\begin{align*}
\Bigg\{(a_1,a_2,\cdots,a_{n-1})\in \iota(A_{F,\infty})^{reg}:\ \Big|\frac{\mathfrak{a}_{\sigma(i)}(a_1,a_2,\cdots,a_{n-1})}{\mathfrak{a}_{\sigma(i+1)}(a_1,a_2,\cdots,a_{n-1})}\Big|_{\infty}\geq c,\ i\leq n-1\Bigg\}.
\end{align*}
\begin{itemize}
	\item[Case 1] Let $i=2$ and $\sigma\in S_{n,i}^{reg}.$ If $\sigma(2)+1<\sigma(1)<n,$ then there exists a unique $j_{\sigma}$ such that $\sigma(j_{\sigma})<\sigma(2)<\sigma(j_{\sigma}+1).$ In this case $\mathcal{H}^c_{\sigma}$ is equal to $\big\{(a_1,a_2,\cdots,a_{n-1})\in \iota(A_{F,\infty})^{reg}:\ |a_{n-2}|_{\infty}\geq c,\ |a_{n-j_{\sigma}}\cdots a_{n-1}|_{\infty}\geq c,\ |a_{n-j_{\sigma}+1}\cdots a_{n-1}|_{\infty}\leq 1/c\big\}.$ If $\sigma(1)=n,$ then $\mathcal{H}^c_{\sigma}=\big\{(a_1,a_2,\cdots,a_{n-1})\in \iota(A_{F,\infty})^{reg}:\ |a_{n-i}\cdots a_{n-2}|_{\infty}\geq c,\ |a_{1}\cdots a_{n-1}|_{\infty}\leq 1/c\big\}.$ If $\sigma(1)=\sigma(2)+1,$ then $\mathcal{H}^c_{\sigma}$ is equal to $\big\{(a_1,a_2,\cdots,a_{n-1})\in \iota(A_{F,\infty})^{reg}:\ |a_{n-2}a_{n-1}|_{\infty}\geq c,\ |a_{n-1}|_{\infty}\leq 1/c,\ |a_{n-1}a_{n-2}|_{\infty}\leq 1/c\big\}.$ If $\sigma(1)=\sigma(2)-1,$ then $\mathcal{H}^c_{\sigma}$ is equal to $\big\{(a_1,a_2,\cdots,a_{n-1})\in \iota(A_{F,\infty})^{reg}:\ |a_{n-2}|_{\infty}\geq c,\ |a_{n-1}|_{\infty}\geq c\big\}.$ Now one sees clearly that the union of $\mathcal{H}^c_{\sigma}$ over $\sigma\in S_{n,i}^{reg}$ does cover the sets $\big\{(a_1,a_2,\cdots,a_{n-1})\in \iota(A_{F,\infty})^{reg}:\ \big|a_{n-2}\big|_{\infty}\geq c,\ |a_{n-1}|_{\infty}\geq c\big\}$ and $\big\{(a_1,a_2,\cdots,a_{n-1})\in \iota(A_{F,\infty})^{reg}:\ \big|a_{n-2}\big|_{\infty}\geq c,\ |a_{n-1}|_{\infty}\leq 1/c\big\}.$ Therefore it covers $\iota_2(A_{F,\infty})^{reg}.$ Hence \eqref{63''} holds in the case where $i=2.$
	\item[Case 2] Let $3\leq i\leq n-1$ and $\sigma\in S_{n,i}^{reg}.$ If $\sigma(2)+1<\sigma(1)<n,$ then there exists a unique $j_{\sigma}$ such that $\sigma(j_{\sigma})<\sigma(2)<\sigma(j_{\sigma}+1).$ In this case $\mathcal{H}^c_{\sigma}$ is equal to $\big\{(a_1,a_2,\cdots,a_{n-1})\in \iota(A_{F,\infty})^{reg}:\ |a_{n-i}\cdots a_{n-2}|_{\infty}\geq c,\ |a_{n-i+1}\cdots a_{n-2}|_{\infty}\leq 1/c,\ |a_{n-j_{\sigma}}\cdots a_{n-1}|_{\infty}\geq c,\ |a_{n-j_{\sigma}+1}\cdots a_{n-1}|_{\infty}\leq 1/c\big\}.$ If $\sigma(1)=n,$ then $\mathcal{H}^c_{\sigma}$ is equal to $\big\{(a_1,a_2,\cdots,a_{n-1})\in \iota(A_{F,\infty})^{reg}:\ |a_{n-i}\cdots a_{n-2}|_{\infty}\geq c,\ |a_{n-i+1}\cdots a_{n-2}|_{\infty}\leq 1/c,\ |a_{1}\cdots a_{n-1}|_{\infty}\leq 1/c\big\}.$ If $\sigma(1)=\sigma(2)+1,$ then $\mathcal{H}^c_{\sigma}$ is equal to $\big\{(a_1,a_2,\cdots,a_{n-1})\in \iota(A_{F,\infty})^{reg}:\ |a_{n-i}\cdots a_{n-2}a_{n-1}|_{\infty}\geq c,\ |a_{n-1}|_{\infty}\leq 1/c,\ |a_{n-i+1}\cdots a_{n-2}|_{\infty}\leq 1/c\big\}.$ If $\sigma(1)=\sigma(2)-1,$ then $\mathcal{H}^c_{\sigma}$ is equal to $\big\{(a_1,a_2,\cdots,a_{n-1})\in \iota(A_{F,\infty})^{reg}:\ |a_{n-i}\cdots a_{n-2}|_{\infty}\geq c,\ |a_{n-1}|_{\infty}\geq c,\ |a_{n-i+1}\cdots a_{n-1}|_{\infty}\leq 1/c\big\}.$ If $1<\sigma(1)<\sigma(2)-1,$ then there exists a unique $j_{\sigma}$ such that $\sigma(j_{\sigma})<\sigma(2)<\sigma(j_{\sigma}+1).$ In this case $\mathcal{H}^c_{\sigma}$ is equal to $\big\{(a_1,a_2,\cdots,a_{n-1})\in \iota(A_{F,\infty})^{reg}:\ |a_{n-i}\cdots a_{n-2}|_{\infty}\geq c,\ |a_{n-i+1}\cdots a_{n-2}|_{\infty}\leq 1/c,\ |a_{n-j_{\sigma}}\cdots a_{n-2}a_{n-1}|_{\infty}\geq c,\ |a_{n-j_{\sigma}+1}\cdots a_{n-1}|_{\infty}\leq 1/c\big\}.$ If $\sigma(1)=1,$ then $\mathcal{H}^c_{\sigma}=\big\{(a_1,a_2,\cdots,a_{n-1})\in \iota(A_{F,\infty})^{reg}:\ |a_{n-i}a_{n-i+1}\cdots a_{n-2}|_{\infty}\geq c,\ |a_{n-i+1}\cdots a_{n-2}|_{\infty}\leq 1/c,\ c\leq |a_{n-2}a_{n-1}|_{\infty}\big\}.$ Now one sees clearly that the union of $\mathcal{H}^c_{\sigma}$ over $\sigma\in S_{n,i}^{reg}$ does cover the sets $\big\{(a_1,a_2,\cdots,a_{n-1})\in \iota(A_{F,\infty})^{reg}:\ \big|a_{n-i}a_{n-i+1}\cdots a_{n-2}\big|_{\infty}\geq c,\ \big|a_{n-i+1}a_{n-i+2}\cdots a_{n-2}\big|_{\infty}\leq 1/c,\ |a_{n-1}|_{\infty}\geq c\big\}$ and $\big\{(a_1,a_2,\cdots,a_{n-1})\in \iota(A_{F,\infty})^{reg}:\ \big|a_{n-i}a_{n-i+1}\cdots a_{n-2}\big|_{\infty}\geq c,\ \big|a_{n-i+1}a_{n-i+2}\cdots a_{n-2}\big|_{\infty}\leq 1/c,\ |a_{n-1}|_{\infty}\leq 1/c\big\}.$ Therefore, it covers $\iota_i(A_{F,\infty})^{reg}.$ Hence \eqref{63''} holds.
	\item[Case 3] Let $i=n$ and $\sigma\in S_{n,i}^{reg}.$  If $\sigma(1)=\sigma(2)+1,$ then $\mathcal{H}^c_{\sigma}=\big\{(a_1,a_2,\cdots,a_{n-1})\in \iota(A_{F,\infty})^{reg}:\ |a_{n-1}|_{\infty}\leq 1/c,\ |a_{1}a_2\cdots a_{n-2}|_{\infty}\leq 1/c\big\}.$ If $\sigma(1)=\sigma(2)-1,$ then $\mathcal{H}^c_{\sigma}$ is equal to $\big\{(a_1,a_2,\cdots,a_{n-1})\in \iota(A_{F,\infty})^{reg}:\ |a_{n-1}|_{\infty}\geq c,\ |a_{1}a_2\cdots a_{n-1}|_{\infty}\leq 1/c\big\}.$ If $1<\sigma(1)<\sigma(2)-1,$ then there exists a unique $j_{\sigma}$ such that $\sigma(j_{\sigma})<\sigma(2)<\sigma(j_{\sigma}+1).$ In this case $\mathcal{H}^c_{\sigma}$ is equal to $\big\{(a_1,\cdots,a_{n-1})\in \iota(A_{F,\infty})^{reg}:\ |a_{1}\cdots a_{n-2}|_{\infty}\leq 1/c,\ |a_{n-j_{\sigma}}\cdots a_{n-1}|_{\infty}\geq c,\ |a_{n-j_{\sigma}+1}\cdots a_{n-1}|_{\infty}\leq 1/c\big\}.$ If $\sigma(1)=1,$ then $\mathcal{H}^c_{\sigma}=\big\{(a_1,a_2,\cdots,a_{n-1})\in \iota(A_{F,\infty})^{reg}:\ |a_{n-i+1}\cdots a_{n-2}|_{\infty}\leq 1/c,\ |a_1a_2\cdots a_{n-2}a_{n-1}|_{\infty}\geq c\big\}.$ Now one sees clearly that the union of $\mathcal{H}^c_{\sigma}$ over $\sigma\in S_{n,i}^{reg}$ does cover the sets $\big\{(a_1,a_2,\cdots,a_{n-1})\in \iota(A_{F,\infty})^{reg}:\ \big|a_{1}a_{2}\cdots a_{n-2}\big|_{\infty}\leq 1/c,\ |a_{n-1}|_{\infty}\geq c\big\}$ and $\big\{(a_1,a_2,\cdots,a_{n-1})\in \iota(A_{F,\infty})^{reg}:\ \big|a_{1}a_{2}\cdots a_{n-2}\big|_{\infty}\leq 1/c,\ |a_{n-1}|_{\infty}\leq 1/c\big\}.$ Therefore, it covers $\iota_n(A_{F,\infty})^{reg}.$ Hence \eqref{63''} holds for $i=n.$
\end{itemize}
Therefore, Claim \ref{28''} follows from the above discussions.
\end{proof}

	\begin{prop}\label{27''}
	Let notation be as before. Let $\iota:$ $A_{F,\infty}\longrightarrow  \mathbb{G}_m(\mathbb{A}_{F,\infty})^{n-1}$ be the map given by $\mathfrak{a}\mapsto (a_1,a_2,\cdots,a_{n-1}).$ Then for any $0<c\leq 1,$ one has
	\begin{equation}\label{60'}
	\iota(A_{F,\infty})\subseteq \bigcup_{\sigma\in S_n}\mathcal{H}^c_{\sigma}.
	\end{equation}
	\end{prop}
    \begin{proof}
    When $n=2,$ \eqref{60'} is immediately. When $n=3,$ there are six different $\mathcal{H}^c_{\sigma}$'s. In this case, one can verify by brute force computation that the union of these hyperboloids does cover $\iota(A_{F,\infty}).$ Hence \eqref{60'} holds for $n=3.$ From now on, we may assume $n\geq 4.$ For any $m\in\mathbb{Z},$ let $\tau_m:$ $\mathbb{Z}\rightarrow \mathbb{Z}$ be the shifting map defined by $j\mapsto j+m,$ $\forall$ $j\in\mathbb{Z}.$ Set $S_{n-2}[2]:=\{\tau_2\circ\sigma\circ\tau_{-2}:\ \sigma\in S_{n-2}\}$ as a set of bijections from the set $\{3,4,\cdots,n\}$ to itself. Regard naturally $S_{n-2}[2]$ as the stabilizer of $\{1,2\}$ of $S_n.$ Then clearly $S_{n-2}[2]$ is isomorphic to $S_{n-2}.$ Denote by $\varrho$ the natural isomorphism $S_{n-2}\xrightarrow{\sim} S_{n-2}[2],$ $\sigma\mapsto \tau_2\circ\sigma\circ\tau_{-2}.$ Note that $\#S_n^{reg}=n(n-1),$ then we have a bijection:
    \begin{equation}\label{66'}
    S_{n-2}\times S_n^{reg}\xrightarrow{1:1}S_n,\quad (\sigma,\sigma')\mapsto \varrho(\sigma)\circ \sigma'.
    \end{equation}
   Assume that \eqref{60'} holds for any $n_0\leq n-2.$ Let $(a_1,a_2,\cdots,a_{n-1})\in \iota(A_{F,\infty}).$ Let $\mathfrak{a}\in A_{F,\infty}$ be such that $\iota(\mathfrak{a})=(a_1,a_2,\cdots,a_{n-1}).$ Then $(a_1,a_2,\cdots,a_{n-3})\in \iota^{\leq n-2}(A_{F,\infty}),$  where $\iota^{\leq n-2}:$ $A_{F,\infty}\longrightarrow  \mathbb{G}_m(\mathbb{A}_{F,\infty})^{n-3}$ is the map given by $\mathfrak{a}\mapsto (a_1,a_2,\cdots,a_{n-3}).$ Then by our induction assumption, there exists some $\sigma\in S_{n-2}$ such that $(a_1,a_2,\cdots,a_{n-3})\in \mathcal{H}_{\sigma_{\varrho}}^c,$ where $\sigma_{\varrho}:=\varrho^{-1}(\sigma).$ Note that for each $1\leq i\leq n-3,$  $\mathfrak{a}_{\sigma_{\varrho}(i)}(a_1,a_2,\cdots,a_{n-1})$ is independent of $a_{n-2}$ and $a_{n-1},$ we may write $\mathfrak{a}_{\sigma_{\varrho}(i)}(a_1,a_2,\cdots,a_{n-1})=\mathfrak{a}_{\sigma_{\varrho}(i)}(a_1,a_2,\cdots,a_{n-3}),$ $1\leq i\leq n-3.$ Let $\mathfrak{b}_{2}=a_{1}a_2\cdots a_{n-2}\cdot \mathfrak{a}_{\sigma_{\varrho}(1)}(a_1,a_2,\cdots,a_{n-3})^{-1},$ $\mathfrak{b}_{1}=a_{1}a_2\cdots a_{n-2}a_{n-1}\cdot \mathfrak{b}_{2}^{-1},$ and $\mathfrak{b}_{n}=1.$ Let $\mathfrak{b}_{i+2}=\mathfrak{a}_{\sigma_{\varrho}(i)}(a_1,a_2,\cdots,a_{n-3}),$ $1\leq i\leq n-3.$ Set $\mathfrak{b}=\diag(\mathfrak{b}_{1},\mathfrak{b}_{2},\cdots,\mathfrak{b}_{n}).$ Then clearly $\mathfrak{b}=\mathfrak{a}_{\sigma_{\varrho}}.$ Hence $\mathfrak{b}\in A_{F,\infty}$ and $\iota(\mathfrak{b})=(b_1,b_2,\cdots,b_{n-1}),$ where $b_i=\mathfrak{b}_{n-i}\cdot\mathfrak{b}_{n-i+1}^{-1},$ $1\leq i\leq n-1.$ By our definition of $\mathfrak{b},$ one sees that $\iota(\mathfrak{b})\in\iota(A_{F,\infty})^{reg}.$ Then by Lemma \ref{26''} there exists some $\sigma'\in S_n^{reg}$ such that $\iota(\mathfrak{b})\in \mathcal{H}_{\sigma'}^c.$ Therefore, 
   \begin{align*}
   |\mathfrak{b}_{\sigma'(i)}(b_1,b_2,\cdots,b_{n-1})|_{\infty}\geq c\cdot|\mathfrak{b}_{\sigma'(i+1)}(b_1,b_2,\cdots,b_{n-1})|_{\infty},\ 1\leq i\leq n-1.
   \end{align*}
   Since $\mathfrak{b}=\mathfrak{a}_{\sigma_{\varrho}},$ the above system of inequalities becomes, for $1\leq i\leq n-1,$ that 
   \begin{align*}
   |{\mathfrak{a}_{\varrho(\sigma)\circ\sigma'(i)}(a_1,a_2,\cdots,a_{n-1})}|_{\infty}\geq c\cdot |{\mathfrak{a}_{\varrho(\sigma)\circ\sigma'(i+1)}(a_1,a_2,\cdots,a_{n-1})}|_{\infty}.
   \end{align*}
   Then according to \eqref{66'}, $\iota(\mathfrak{a})\in \mathcal{H}^c_{\tilde{\sigma}}$ for $\tilde{\sigma}=\varrho(\sigma)\circ\sigma'\in S_n.$ Hence \eqref{60'} holds for $n_0=n.$ Since it holds for $n_0=2$ and $n_0=3,$ Proposition \ref{27''} follows by induction on initial cases.
    \end{proof}

	\begin{prop}\label{27'}
	Let notation be as above. Let $A_{\varphi,fin}$ be a compact subgroup of $Z_G(\mathbb{A}_{F,fin})\backslash T(\mathbb{A}_{F,fin})$ depending only on $\varphi$ and $F.$ Let $R(x)$ be a slowly increasing function on $\mathcal{S}_0.$ Then we have
	\begin{equation}\label{56''}
	\int_K\int_{[N]}\int_{A_{F,\infty}A_{\varphi,fin}}\sum_{\chi}\Big|\mathcal{F}_1\Lambda^T_2\K_{\chi}(nak,nak)\cdot R(nak)\Big|d^{\times}adndk<\infty,
	\end{equation}
	where $\chi$ runs over all the equivalent classes of cuspidal datum.
	\end{prop}
	\begin{proof}
	Let $A_{F,\infty}^{\sigma}=\{\mathfrak{a}\in A_{F,\infty}:\ \iota(\mathfrak{a})\in \mathcal{H}_{\sigma}^c\}.$ Then Proposition \ref{27''} implies that 
	\begin{equation}\label{69}
	A_{F,\infty}=\bigcup_{\sigma\in S_n}A_{F,\infty}^{\sigma}.
	\end{equation}
    The decomposition \eqref{69} implies that the left hand side of \eqref{56''} is not more than 
    \begin{align*}
    J_c:=\sum_{\sigma\in S_n}\int_K\int_{[N]}\int_{A_{F,\infty}^{\sigma}A_{\varphi,fin}}\sum_{\chi}\Big|\mathcal{F}_1\Lambda^T_2\K_{\chi}(nak,nak)\cdot R(nak)\Big|d^{\times}adndk.
    \end{align*}	
	Note that for any $\sigma\in S_n,$ let $\mathfrak{a}\in A_{F,\infty}^{\sigma},$ then $\mathfrak{a}_{\sigma}\in\mathcal{S}_c.$ Let $\iota(\mathfrak{a})=(a_{1},a_{2},\cdots,a_{n-1}),$ and $\iota(\mathfrak{a}_{\sigma})=(a_{\sigma,1},a_{\sigma,2},\cdots,a_{\sigma,n-1}),$ then a straightforward computation shows that each $a_{i}$ is a rational monomial $P_{\sigma}$ of $a_{\sigma,1}, a_{\sigma,2},\cdots, a_{\sigma,n-1}.$ Since such a monomial is at most polynomially increasing on $\mathcal{S}_c,$ one then changes variables to see 
	\begin{align*}
	\int_K\int_{[N]}\int_{A_{F,\infty}^{\sigma}A_{\varphi,fin}}\sum_{\chi}\Big|\mathcal{F}_1\Lambda^T_2\K_{\chi}(nak,nak)\cdot R(nak)\Big|d^{\times}adndk
	\end{align*}
	is bounded by the integral over $K,$ $N(F)\backslash N(\mathbb{A}_F)$ of
	\begin{align*}
	\int_{A_{F,\infty}^{\sigma}A_{\varphi,fin}}\sum_{\chi}\Big|\mathcal{F}_1\Lambda^T_2\K_{\chi}(n\mathfrak{a}_{\sigma}a_{fin}k,n\mathfrak{a}_{\sigma}a_{fin}k)\cdot R(n\mathfrak{a}_{\sigma}a_{fin}k)P_{\sigma}(\mathfrak{a}_{\sigma})\Big|d^{\times}\mathfrak{a}_{\sigma}d^{\times}a_{fin},
	\end{align*}
	which is bounded by the integral over $K,$ $N(F)\backslash N(\mathbb{A}_F)$ and $A_{\varphi,fin}$ of $G_{\sigma}(n,a_{\infty},k),$ where $G_{\sigma}(n,a_{\infty},k)$ is defined to be
	\begin{align*}
	\int_{A_{F,\infty}\cap\mathcal{S}_c}\sum_{\chi}\Big|\mathcal{F}_1\Lambda^T_2\K_{\chi}(na_{\infty}a_{fin}k,na_{\infty}a_{fin}k)\cdot R(na_{\infty}a_{fin}k)P_{\sigma}(a_{\infty})\Big|d^{\times}a_{\infty}.
	\end{align*}
	Since the function $a_{\infty}\mapsto R(na_{\infty}a_{fin}k)P_{\sigma}(a_{\infty})$ is slowly increasing on $A_{F,\infty}\cap\mathcal{S}_c,$ Lemma \ref{26'} implies that $G_{\sigma}(n,a_{\infty},k)$ is well defined and thus it is continuous. So
	\begin{align*}
	\int_K\int_{N(F)\backslash N(\mathbb{A}_F)}\int_{A_{\varphi,fin}}G_{\sigma}(n,a_{\infty},k)d^{\times}a_{fin}dndk<\infty,\ \forall\ \sigma\in S_n.
	\end{align*}
	Therefore, \eqref{56''} follows from the estimate below:
	\begin{align*}
	J_c\leq \sum_{\sigma\in S_n}\int_K\int_{N(F)\backslash N(\mathbb{A}_F)}\int_{A_{\varphi,fin}}G_{\sigma}(n,a_{\infty},k)d^{\times}a_{fin}dndk<\infty.
	\end{align*}
	\end{proof}
    
    \begin{cor}\label{28'}
    Let notation be as above. Let $R(x)$ be a slowly increasing function on $\mathcal{S}_0.$ Then we have
    \begin{equation}\label{56'}
    \int_{Z_G(\mathbb{A}_F)N(F)\backslash G(\mathbb{A}_F)}\sum_{\chi}\Big|\mathcal{F}_1\Lambda^T_2\K_{\chi}(x,x)\cdot R(x)\Big|dx<\infty,
    \end{equation}
    where $\chi$ runs over all the equivalent classes of cuspidal data.
    \end{cor}
    \begin{proof}
    According to \eqref{56'''}, which is absolutely convergent, one has explicitly that 
    \begin{align*}
    \mathcal{F}_1\Lambda^T_2\K_{\chi}(x,y)=\sum_{P\in \mathcal{P}}\frac{1}{c_P}\int_{i\mathfrak{a}^*_P/i\mathfrak{a}^*_G}\sum_{\phi\in \mathfrak{B}_{P,\chi}}\mathcal{F}E(x,\mathcal{I}_P(\lambda,\varphi)\phi,\lambda)\Lambda^T\overline{E(y,\phi,\lambda)}d\lambda,
    \end{align*}
    which is an easier analogue of Lemma \ref{25'} and the fact that for given $x$ and $y,$ $\Lambda^T$ is a finite sum, to write the operators $\mathcal{F}$ and $\Lambda^T$ inside the integral over $i\mathfrak{a}^*_P/i\mathfrak{a}^*_G.$ Then as before, the non-constant Fourier coefficient $\mathcal{F}E(x,\mathcal{I}_P(\lambda,\varphi)\phi,\lambda)$ of $E(x,\mathcal{I}_P(\lambda,\varphi)\phi,\lambda)$ becomes a Whittaker function $W(x;\lambda).$ 
    
    Recall that our test function $\varphi$ is $K$-finite. Hence there is some compact subgroup $K_0\subset G(\mathbb{A}_{F,fin})^1$ such that $\varphi$ is right $K_0$-invariant. Then for $\phi\in \mathfrak{B}_{P,\chi},$ $\mathcal{F}E(x,\mathcal{I}_P(\lambda,\varphi)\phi,\lambda)\Lambda^T\overline{E(y,\phi,\lambda)}=0$ unless $\phi$ is right $K_0$-invariant. Let $K_0=\prod_{v<\infty}K_{0,v}.$ Note that the Whittaker functions are decomposable, i.e., 
    $$
    W(x;\lambda)=\prod_{v\in\Sigma_F}W_{v}(x;\lambda).
    $$
    Then for each finite place $v,$ $W_{v}(x;\lambda)$ is right $K_{0,v}$-invariant. So there exists a compact subgroup $N_{0,v}\subseteq K_{0,v}\cap N(F_v),$ depending only on $\varphi,$ such that 
    $$
    W_{v}(t_vu_v;\lambda)=W_{v}(t_v;\lambda),\ \text{for all $t_v\in T(F_v)$ and $u_v\in N_{0,v}$}.
    $$
    On the other hand, $W_{v}(t_vu_v;\lambda)=\theta_{t_v}(u_v)W_{v}(t_v;\lambda),$ where $\theta_{t_v}(n_v)=\theta(t_vn_vt_v^{-1}),$ for any $n_v\in N(F_v).$ But then, there exists a constant $C_v$ depending only on $N_{0,v}$ and $\theta$ (hence not on $\lambda$) such that $\theta_{t_v}(u_v)=1$ if and only if $|\alpha_i(t_v)|\leq C_v,$ where $\alpha_i$'s are the simple roots of $G(F).$ Note that for all but finitely many $v<\infty,$ $K_{0,v}=GL_n(\mathcal{O}_{F,v}),$ thus we can take the corresponding $C_v=1.$ Hence for any $t_v\in T(F_v),$ $W_{v}(t_v;\lambda)\neq 0$ implies that $|\alpha_i(t_v)|\leq C_v,$ $1\leq i\leq n-1,$ and $C_v=1$ for all but finitely many finite places $v.$ Set $A=Z_G\backslash T,$ and
    $$
    A_{\varphi,fin}=\Big\{a=(a_v)\in A(\mathbb{A}_{F,fin}):\ |\alpha_i(a_v)|\leq C_v,\ 1\leq i\leq n-1\Big\}.
    $$
    Then $\supp W(x;\lambda)\mid_{A(\mathbb{A}_F)}\subseteq A(\mathbb{A}_{F,\infty})A_{\varphi,fin},$ $\forall$ $\lambda\in i\mathfrak{a}^*_P/i\mathfrak{a}^*_G,$ $1\leq i\leq 2.$ So after applying Iwasawa decomposition, we see that the integrand in \eqref{56'} are supported in $[N]\cdot A(\mathbb{A}_{F,\infty})A_{\varphi,fin}\cdot K,$ independent of $\chi.$ Therefore, \eqref{56'} follows from \eqref{56''}.
    \end{proof}
   
    Let $\K_P(x,y)$ be the restriction of $\K(x,y)$ to $\mathcal{H}_{P},$ explicitly given by \eqref{ker}. Define
    \begin{align*}
    \mathcal{F}_1k^T(x)=\sum_{P}(-1)^{\dim(A_P/A_G)}\sum_{\delta\in P(F)\backslash G(F)}\widehat{\tau}_P\left(H_P(\delta x)-T\right)\cdot\mathcal{F}_1\K_P(\delta x,\delta x),
    \end{align*}
    where $P$ runs over standard parabolic subgroups of $G,$ and for each non-minimal $P,$ the partial Fourier term $\mathcal{F}_1\K_P(x,y)$ is defined explicitly by
    \begin{align*}
    \mathcal{F}_1\K_P(x,y)=\int_{N_P(\mathbb{A}_F)N(F)\backslash N(\mathbb{A}_F)}\K_P(n x,y)\theta(n)dn,\ \forall\ x,y\in Z_G(\mathbb{A}_F)G(F)\backslash G(\mathbb{A}_F).
    \end{align*}
    
    Recall that by \eqref{52'} we have $\K_{P}(x,y)=\sum_{\chi}\K_{P,\chi}(x,y),$ where $\K_{P,\chi}(x,y)$ is defined by \eqref{53'}. We then obtain a decomposition $\mathcal{F}_1k^T(x)=\sum_{\chi\in\mathfrak{X}}\mathcal{F}_1k_{\chi}^T(x),$ where
    \begin{align*}
    \mathcal{F}_1k_{\chi}^T(x)=\sum_{P}(-1)^{\dim(A_P/A_G)}\sum_{\delta\in P(F)\backslash G(F)}\widehat{\tau}_P\left(H_P(\delta x)-T\right)\cdot\mathcal{F}_1\K_{P,\chi}(\delta x,\delta x),
    \end{align*}
    where $P$ runs over non-minimal standard parabolic subgroups and the partial Fourier term is defined explicitly by
    \begin{align*}
    \mathcal{F}_1\K_{P,\chi}(x,y)=\int_{N_P(\mathbb{A}_F)N(F)\backslash N(\mathbb{A}_F)}\K_{P,\chi}(n x,y)\theta(n)dn,
    \end{align*}
    where $\K_{P,\chi}(x,y)$ is defined by \eqref{53'}, i.e., for $x,y\in Z_G(\mathbb{A}_F)G(F)\backslash G(\mathbb{A}_F),$
    \begin{align*}
    K_{P,\chi}(x,y)=\sum_{Q\subset P}\frac{1}{n_Q^P}\int_{i\mathfrak{a}^*_{Q}/i\mathfrak{a}^*_{H}}\sum_{\phi\in\mathcal{B}_{Q,\chi}}E_Q^P\left(x,\mathcal{I}_{Q}(\lambda, \varphi)\phi,\lambda\right)\overline{E_Q^P\left(y,\phi,\lambda\right)}d\lambda.
    \end{align*} 
   Since each individual function $\K_{P,\chi}(x,y)$ has no nice geometric interpretation, when we deal with it we need to go back to $\K_{P}(x,y)$ and take advantage of its geometric expression; then "project" our results on $\K_{P}(x,y)$ to $\K_{P,\chi}(x,y).$ To make things more precise, we need to establish a variant of Lemma 2.3 of \cite{Art80}.

    \begin{lemma}\label{2.3}
    Let notation be as before. Suppose that for each $i,$ $1\leq i\leq N,$ we are given a parabolic subgroup $Q_i\supset Q,$ points $x_i, y_i\in G(\mathbb{A}_F)$ and a number $c_i$ such that 
    \begin{equation}\label{65'}
    \sum_{i=1}^Nc_i\int_{N_{Q}(F)\backslash N_Q(\mathbb{A}_F)}\int_{N_{Q_i}(\mathbb{A}_F)N(F)\backslash N(\mathbb{A}_F)}\K_{Q_i}(n_1x_i,nmy_i)\theta(n_1)dn_1dn=0,
    \end{equation}
    for all $m\in Z_Q(\mathbb{A}_F)M_Q(F)\backslash M_Q(\mathbb{A}_F).$ Then for any $\chi\in\mathfrak{X},$ 
    \begin{align*}
    h_{\chi}(m):=\sum_{i=1}^Nc_i\int_{N_{Q}(F)\backslash N_Q(\mathbb{A}_F)}\int_{N_{Q_i}(\mathbb{A}_F)N(F)\backslash N(\mathbb{A}_F)}\K_{Q_i,\chi}(n_1x_i,nmy_i)\theta(n_1)dn_1dn
    \end{align*}
    also vanishes for all $m\in Z_Q(\mathbb{A}_F)M_Q(F)\backslash M_Q(\mathbb{A}_F).$
    \end{lemma}
    \begin{proof}
    Given any $\chi'\in\mathfrak{X}.$ Suppose that there exists some group $R$ in $\mathcal{P}_{\chi'}$ which is contained in $Q.$ Let $\phi_{\chi'}\in L^2\left(M_R(F)Z_{M_R}(\mathbb{A}_F)\backslash M_R(\mathbb{A}_F)\right)_{\chi'}.$ For any function $h\in L^2\left(R(F)Z_{M_R}(\mathbb{A}_F)\backslash R(\mathbb{A}_F)\right)_{\chi},$ we define 
    \begin{align*}
    I(h,\phi_{\chi'}):=\int_{M_R(F)Z_{M_R}(\mathbb{A}_F)\backslash M_R(\mathbb{A}_F)}\int_{N_R(F)\backslash N_R(\mathbb{A}_F)}h(nm)\phi_{\chi'}(m)dndm.
    \end{align*}
    For any parabolic subgroup $Q'\subset Q_i,$ Let $h_{i,Q'}$ be the function defined by 
    \begin{align*}
    m\longmapsto \int_{N_{Q_i}(F)\backslash N_{Q_i}(\mathbb{A}_F)}E_{Q'}^{Q_i}(nmy,\phi,\lambda)dn,\ m\in M_{Q}(F)Z_{M_{Q}}(\mathbb{A}_F)\backslash M_{Q}(\mathbb{A}_F).
    \end{align*}
    Then $I(h_{i,Q'},\phi_{\chi'})=0$ by the construction of Eisenstein series. Hence, apply the spectral expansion of each $\mathcal{F}_1\K_{Q_i,\chi}(x,y)$ (noting that due to the same philosophy as Lemma \ref{25'} one can interchange the partial Fourier transform and the integral over the spectral parameters) one concludes that $I(h_{\chi},\phi_{\chi'})=0$ for any $\chi\neq \chi'.$ 
    
    Also the same estimates as in the proof of Lemma \ref{26'} leads to 
    \begin{align*}
    \sum_{\chi\in\mathfrak{X}}\big|h_{\chi}(m)\big|\leq c\|m\|^N,\ m\in M_R(F)Z_{M_R}(\mathbb{A}_F)\backslash M_R(\mathbb{A}_F),
    \end{align*}
    for some positive constants $c$ and $N,$ since the domain $N_{Q_i}(\mathbb{A}_F)N(F)\backslash N(\mathbb{A}_F)$ is compact. Note that \eqref{65'} amounts to that $\sum_{\chi}h_{\chi}(m)=0$ for any $m$ in the space $Z_Q(\mathbb{A}_F)M_Q(F)\backslash M_Q(\mathbb{A}_F).$ Therefore, $I(h_{\chi'},\phi_{\chi'})=0$ as well. Then for any $\chi, \chi'\in\mathfrak{X},$ $I(h_{\chi},\phi_{\chi'})=0.$ Note that the function $h_{\chi}$ is continuous, then by Lemma 3.7 of \cite{Lan76} one concludes that $h_{\chi}\equiv0,$ $\forall$ $\chi\in\mathfrak{X}.$

    \end{proof}

    \begin{prop}\label{29'}
    Let notation be as before. Let $\chi\in\mathfrak{X}$ be a cuspidal datum. Then when $T$ is sufficiently regular and depends only on the support of $\varphi,$ one has $\mathcal{F}_1\Lambda^T_2\K_{\chi}(x,x)=\mathcal{F}_1k^T_{\chi}(x)$ on $Z_G(\mathbb{A}_F)G(F)\backslash G(\mathbb{A}_F).$
    \end{prop}
    \begin{proof}
    If $\phi\in\mathcal{B}_{loc}\left(Z_G(\mathbb{A}_F)G(F)\backslash G(\mathbb{A}_F)\right),$ we define $\Lambda^{T,P}\phi$ to be the function in $\mathcal{B}_{loc}\left(Z_G(\mathbb{A}_F)M_P(F)N_P(\mathbb{A}_F)\backslash G(\mathbb{A}_F)\right)$ whose value at $x$ is equal to
    \begin{align*}
    \sum_{Q:\ B\subset Q\subset P}(-1)^{\dim(A_Q/A_P)}\sum_{\delta\in Q(F)\backslash P(F)}\int_{N_Q(F)\backslash N_Q(\mathbb{A}_F)}\phi(n\delta x)\widehat{\tau}_Q^P\left(H_Q(\delta x)-T\right)dn.
    \end{align*}
    Then by Lemma 1.5 of \cite{Art80}, for any $\phi\in\mathcal{B}_{loc}\left(Z_G(\mathbb{A}_F)G(F)\backslash G(\mathbb{A}_F)\right),$ one has
    \begin{align*}
    \phi(x)=\sum_{P\supset B}\sum_{\delta\in P(F)\backslash G(F)}\Lambda^{T,P}\phi(\delta x)\tau_P(H_P(\delta x)-T).
    \end{align*}
    More generally, if $\phi\in \mathcal{B}_{loc}\left(Z_G(\mathbb{A}_F)P(F)\backslash G(\mathbb{A}_F)\right)$ for some standard parabolic subgroup $P,$ then we have the following expansion:
    \begin{equation}\label{62}
    \int_{N_P(F)\backslash N_P(\mathbb{A}_F)}\phi(nx)dn=\sum_{B\subset Q\subset P}\sum_{\delta\in Q(F)\backslash P(F)}\Lambda^{T,Q}\phi(\delta x)\tau_Q^P(H_Q(\delta x)-T).
    \end{equation}
    Note that the function $\mathcal{F}_1\K_{P,\chi}(x,y)$ is invariant under the left translation of the second variable by $N_P(\mathbb{A}_F).$ In particular, we can write
    \begin{align*}
    \mathcal{F}_1\K_{P,\chi}(x,y)=\int_{N_P(F)\backslash N_P(\mathbb{A}_F)}\mathcal{F}_1\K_{P,\chi}( x,ny)dn.
    \end{align*} 
    Substituting this expression into the definition implies that $\mathcal{F}_1k_{\chi}^T$ is equal to
    \begin{align*}
    \sum_{P}(-1)^{\dim\left(A^P_G\right)}\sum_{\delta\in P(F)\backslash G(F)}\widehat{\tau}_P\left(H_P(\delta x)-T\right)\int_{N_P(F)\backslash N_P(\mathbb{A}_F)}\mathcal{F}_1\K_P(\delta x,n\delta x)dn,
    \end{align*}
    where $A^P_G:=A_P/A_G.$ The integral over $n$ can then be expanded via \eqref{62}. Combine the resulting sum over $Q(F)\backslash P(F)$ with that over $P(F)\backslash G(F)$ to give an expression
    \begin{align*}
    \sum_{Q\subset P}(-1)^{\dim\left(A^P_G\right)}\sum_{\delta}\widehat{\tau}_P\left(H_P(\delta x)-T\right)\tau_Q^P\left(H_Q(\delta x)-T\right)\mathcal{F}_1\Lambda_2^{T,Q}\K_{P,\chi}(\delta x,\delta x)
    \end{align*}
    of $\mathcal{F}_1k^T_{\chi}(x),$ where $\delta$ runs over $Q(F)\backslash G(F).$ Note that one has 
    $$
    \widehat{\tau}_P\left(H_P(\delta x)-T\right)\tau_Q^P\left(H_Q(\delta x)-T\right)=\widehat{\tau}_P\left(H_Q(\delta x)-T\right)\tau_Q^P\left(H_Q(\delta x)-T\right).
    $$ 
    Then applying the identity, for any $H\in\mathfrak{a}_Q,$
    \begin{align*}
    \tau_Q^P(H)\widehat{\tau}_P(H)=\sum_{P_1:\ P_1\supset P}\sum_{P_2:\ Q\subset P_2\subset P_1}(-1)^{\dim(A_Q/A_{P_1})}\tau_Q^{P_2}(H)\widehat{\tau}_{P_2}(H),
    \end{align*}
    we can rewrite $\widehat{\tau}_P\left(H_P(\delta x)-T\right)\tau_Q^P\left(H_Q(\delta x)-T\right)$ as
    \begin{equation}\label{63'}
    \sum_{P_1:\ P_1\supset P}\sum_{P_2:\ P_2\supset Q}(-1)^{\dim(A_Q/A_{P_1})}\tau_Q^{P_2}(H_Q(\delta x)-T)\widehat{\tau}_{P_2}(H_Q(\delta x)-T).
    \end{equation}
    Denote the inner sum in \eqref{63'} by $\sigma_{Q}^{P_1}.$ It then follows that $\mathcal{F}_1k^T_{\chi}(x)$ has an expansion
    \begin{equation}\label{63}
    \sum_{Q\subset P_1}\sum_{\delta\in Q(F)\backslash G(F)}\sigma_{Q}^{P_1}(H_Q(\delta x)-T)\mathcal{F}_1\Lambda_2^{T,Q}\K_{Q,P_1,\chi}(\delta x,\delta x),
    \end{equation}
    where for any $x,y\in Z_G(\mathbb{A}_F)\backslash G(\mathbb{A}_F),$
    $$
    \K_{Q,P_1,\chi}(x,y):=\sum_{P:\ Q\subset P\subset P_1}(-1)^{\dim(A_P/A_G)}\K_{P,\chi}(x,y).
    $$
    
    By Lemma 6.3 of \cite{Art78}, $\sigma_{Q}^{P_1}=0$ if $Q=P_1\neq G.$ So the corresponding summand in \eqref{63} vanishes. If $Q=P_1=G,$ $\sigma_{Q}^{P_1}=1,$ and the corresponding summand in \eqref{63} is equal to $\mathcal{F}_1\Lambda_2^{T}\K_{\chi}(x,x).$ Therefore, the difference $\mathcal{F}_1k^T_{\chi}(x)-\mathcal{F}_1\Lambda^T_2\K_{\chi}(x,x)$ is equal to the modified expression of \eqref{63}:
    \begin{equation}\label{65}
    \sum_{Q\subsetneq P_1}\sum_{\delta\in Q(F)\backslash G(F)}\sigma_{Q}^{P_1}(H_Q(\delta x)-T)\mathcal{F}_1\Lambda_2^{T,Q}\K_{Q,P_1,\chi}(\delta x,\delta x).
    \end{equation}
    Consider the integral over $Z_G(\mathbb{A}_F)G(F)\backslash G(\mathbb{A}_F)$ of the absolute value of the expression in \eqref{65}. Combining the integral with the sum over $Q(F)\backslash G(F)$ leads to
    \begin{align*}
    &\sum_{\chi\in\mathfrak{X}}\int_{Z_G(\mathbb{A}_F)G(F)\backslash G(\mathbb{A}_F)}\Big|\mathcal{F}_1k^T_{\chi}(x)-\mathcal{F}_1\Lambda^T_2\K_{\chi}(x,x)\Big|dx\\
    \leq&\sum_{\chi\in\mathfrak{X}}\sum_{Q\subsetneq P_1}\int_{Z_G(\mathbb{A}_F)Q(F)\backslash G(\mathbb{A}_F)}\sigma_{Q}^{P_1}(H_Q(x)-T)\Big|\mathcal{F}_1\Lambda^{T,Q}_2\K_{Q,P_1,\chi}(x,x)\Big|dx.
    \end{align*}
    Let $J_{Q,P_1,\chi}^T:=\sigma_{Q}^{P_1}(H_Q(\delta x)-T)\big|\mathcal{F}_1\Lambda^{T,Q}_2\K_{Q,P_1,\chi}(x,x)\big|,$ $Q\subsetneq P_1.$ Then we will show that for $T$ highly regular, $J_{Q,P_1,\chi}^T$ actually vanishes. For any $Q\subsetneq P_1,$ set
    \begin{align*}
    \mathcal{F}_1\K_{Q,P_1}(x,y)=\sum_{P:\ Q\subset P\subset P_1}(-1)^{\dim(A_P/A_G)}\int_{N(F)\backslash N(\mathbb{A}_F)}\K_{P}(n_1x,y)\theta(n_1)dn_1.
    \end{align*}
    Let $J_2(x,y)$ be the constant term of the partial Fourier expansion of $\mathcal{F}_1\K_{Q,P_1}(x,y)$ with respect to $y$-variable along $N_Q(F)\backslash N_Q(\mathbb{A}_F),$ explicitly,
    \begin{align*}
    J_2(x,y)=\int_{N_Q(F)\backslash N_Q(\mathbb{A}_F)}\mathcal{F}_1\K_{Q,P_1}(x,ny)dn.
    \end{align*}
    
    Then use the geometric expansion of $\K_{P}(n_1x,y)$ we see $J_2(x,y)$ is equal to 
    \begin{align*}
    \sum_{P:\ Q\subset P\subset P_1}(-1)^{\dim(A^P_G)}\int_{[N]}\int_{[N_Q]}\int_{N_P(\mathbb{A}_F)}\sum_{\gamma\in M_P(F)}\varphi(x^{-1}n_1^{-1}\gamma nn_2y)dndn_2\theta(n_1)dn_1.
    \end{align*}
    In the last summand corresponding to $P,$ for the inner triple integral we have
    \begin{align*}
    &\int_{N_Q(F)\backslash N_Q(\mathbb{A}_F)}\int_{N_P(\mathbb{A}_F)}\sum_{\gamma\in M_P(F)}\varphi(x^{-1}n_1^{-1}\gamma nn_2y)dndn_2\\
    =&\int_{N_Q(\mathbb{A}_F)}\int_{M_P(F)N_P(\mathbb{A}_F)/N_Q(F)}\varphi(x^{-1}n_1^{-1}pn_2y)dpdn_2\\
    =&\sum_{M_P(F)/M_P(F)\cap N_Q(F)}\int_{N_P(\mathbb{A}_F)/N_P(F)}\int_{N_Q(\mathbb{A}_F)}\varphi(x^{-1}n_1^{-1}\gamma nn_2y)dndn_2.
    \end{align*}
    The integral over $N_P(\mathbb{A}_F)/N_P(F)$ can then be absorbed in the integral over $N_Q(\mathbb{A}_F).$ Note that $M_P(F)/M_P(F)\cap N_Q(F)\simeq P(F)/Q(F)\times M_Q(F),$ then $J_2(x,y)$ equals
    \begin{align*}
    \sum_{P}(-1)^{\dim(A^P_G)}\int_{\mathcal{N}_Q}\sum_{\gamma\in Q(F)\backslash P(F)}\int_{N_Q(\mathbb{A}_F)}\sum_{\eta\in M_Q(F)}\varphi(x^{-1}n_1^{-1}\gamma^{-1}\eta n_2y)dn_2\theta(n_1)dn_1,
    \end{align*}
    where $\mathcal{N}_Q=N_Q(\mathbb{A}_F)N(F)\backslash N(\mathbb{A}_F),$ $A^P_G=A_P/A_G,$ and $P$ runs over parabolic subgroups of $G$ such that $Q\subset P\subset P_1.$ Clearly the above expression is equal to
    \begin{equation}\label{66}
    \sum_{P:\ Q\subset P\subset P_1}(-1)^{\dim(A_P/A_G)}\int_{\mathcal{N}_Q}\sum_{\gamma\in Q(F)\backslash P(F)}\K_Q(\gamma n_1x,y)\theta(n_1)dn_1.
    \end{equation}
    For any $P_0$ such that $Q\subset P_0\subset P_1,$ let $F(Q,P_0)$ be the collection of all elements in $Q(F)\backslash P_0(F)$ which do not lie in $Q(F)\backslash P(F)$ for any $Q\subset P\subsetneq P_0.$ Since the sum over $P$ such that $Q\subset P\subset P_1$ is finite, then \eqref{66} becomes
    \begin{align*}
    \int_{\mathcal{N}_Q}\sum_{P_0:\ Q\subset P_0\subset P_1}\sum_{\gamma\in F(Q,P_0)}\sum_{P:\ P_0\subset P\subset P_1}(-1)^{\dim(A_P/A_G)}\K_Q(\gamma n_1x,y)\theta(n_1)dn_1.
    \end{align*}
    An elementary combinatorial computation shows that the inner triple sum vanishes unless $P_0=P_1,$ in which case the multiplicity of each $\gamma\in F(Q,P_1)$ equals $(-1)^{\dim(A_{P_1}/A_G)}.$ Therefore, we have shown
    \begin{equation}\label{67}
    J_2(x,y)=(-1)^{\dim(A_{P_1}/A_G)}\int_{\mathcal{N}_Q}\sum_{\gamma\in F(Q,P_1)}\K_Q(\gamma n_1x,y)\theta(n_1)dn_1.
    \end{equation}
    For any $m\in Z_Q(\mathbb{A}_F)M_Q(F)\backslash M_Q(\mathbb{A}_F),$ define $h(m)$ to be
    \begin{align*}
    c_PJ_2(x,my)-(-1)^{\dim(A_{P_1}/A_G)}\sum_{\gamma\in F(Q,P_1)}\int_{[N_Q]}\int_{\mathcal{N}_Q}\K_Q(\gamma n_1x,my)\theta(n_1)dn_1dn,
    \end{align*}
    where $c_Q=\vol\left(N_Q(F)\backslash N_Q(\mathbb{A}_F)\right).$ Since $\K_Q(x,y)$ is left $N_Q(\mathbb{A}_F)$-invariant on both variables, then \eqref{67} implies that $h(m)\equiv0$ on $Z_Q(\mathbb{A}_F)M_Q(F)\backslash M_Q(\mathbb{A}_F).$ Set 
    \begin{align*}
    J_2^{\chi}(x,y)=\int_{N_Q(F)\backslash N_Q(\mathbb{A}_F)}\mathcal{F}_1\K_{Q,P_1,\chi}(x,ny)dn,
    \end{align*}
    where $\chi\in\mathfrak{X}$ is any cuspidal datum, and $\mathcal{F}_1\K_{Q,P_1}(x,y)$ is defined to be
    \begin{align*}
    \sum_{P:\ Q\subset P\subset P_1}(-1)^{\dim(A_P/A_G)}\int_{N_P(\mathbb{A}_F)N(F)\backslash N(\mathbb{A}_F)}\K_{P,\chi}(n_1x,y)\theta(n_1)dn_1.
    \end{align*}
    Hence, by Lemma \ref{2.3} and Claim \ref{30} we have, for any $\chi\in\mathfrak{X},$ that $h_{\chi}(m)$ vanishes for any $m\in Z_Q(\mathbb{A}_F)M_Q(F)\backslash M_Q(\mathbb{A}_F),$ where
    \begin{align*}
    c_PJ_2^{\chi}(x,my)-(-1)^{\dim(A_{P_1}/A_G)}\sum_{\gamma\in F(Q,P_1)}\int_{[N_Q]}\int_{\mathcal{N}_Q}\K_{Q,\chi}(\gamma n_1x,my)\theta(n_1)dn_1dn.
    \end{align*}
    Taking $m=I_n,$ the trivial element, leads to $h_{\chi}(I_n)=0,$ which implies that 
    \begin{equation}\label{68}
    J_2^{\chi}(x,y)=(-1)^{\dim(A_{P_1}/A_G)}\sum_{\gamma\in F(Q,P_1)}\int_{\mathcal{N}_Q}\K_{Q,\chi}(\gamma n_1x,y)\theta(n_1)dn_1dn.
    \end{equation}
    By definition, $\mathcal{F}_1\Lambda^{T,Q}_2\K_{Q,P_1,\chi}(x,y)$ is equal to
    \begin{align*}
    \sum_{Q_1}(-1)^{\dim(A_{Q_1}/A_Q)}\sum_{\delta\in Q_1(F)\backslash Q(F)}\widehat{\tau}_{Q_1}^Q\left(H_{Q_1}(\delta x)-T\right)\int_{[N_{Q_1}]}\mathcal{F}_1\K_{Q,P_1,\chi}(x,n\delta y)dn,
    \end{align*}
    where $Q_1$ runs over $B\subset Q_1\subset Q.$ Clearly one sees that  
    $$
    \mathcal{F}_1\Lambda^{T,Q}_2\K_{Q,P_1,\chi}(x,ny)=\mathcal{F}_1\Lambda^{T,Q}_2\K_{Q,P_1,\chi}(x,y),
    $$ 
    for any $n\in N_Q(F)\backslash N_Q(\mathbb{A}_F)$ and any $x, y\in G(\mathbb{A}_F).$ Therefore, by \eqref{68} and Claim \ref{30} one obtains that $\mathcal{F}_1\Lambda^{T,Q}_2\K_{Q,P_1,\chi}(x,y)$ is equal to
   \begin{equation}\label{79'}
   (-1)^{\dim(A^{P_1}_G)}\sum_{\gamma\in F(Q,P_1)}\int_{\mathcal{N}_Q}\Lambda^{T,Q}_2\K_{Q,\chi}(\gamma n_1x,y)\theta(n_1)dn_1.
   \end{equation}
   According to definition, $\Lambda^{T,Q}_2\K_{Q,\chi}(\gamma n_1x,y)$ is equal to
   \begin{equation}\label{80}
   \sum_{P:\ B\subset P\subset Q}(-1)^{\dim(A^P_Q)}\sum_{\delta\in P(F)\backslash Q(F)}\int_{[N_P]}\K_{Q,\chi}(\gamma n_1x,n\delta y)\widehat{\tau}_P^Q\left(H_P(\delta y)-T\right)dn.
   \end{equation}
   Let $\chi$ be represented by a pair $(Q,\sigma).$ Then one has 
   \begin{equation}\label{81}
   \K_{Q,\chi}(\gamma n_1x,n\delta y)=\int_{Z_{M_Q}(\mathbb{A}_F)M_Q(F)\backslash M_Q(\mathbb{A}_F)}\K_{Q}(\gamma n_1x,mn\delta y)\sigma(m)dm.
   \end{equation}
   Substituting \eqref{80} and \eqref{81} into \eqref{79'}, combining with Claim \ref{30}, we can write $\mathcal{F}_1\Lambda^{T,Q}_2\K_{Q,P_1,\chi}(x,y)$ as $(-1)^{\dim(A^{P_1}_G)}$ multiplying 
   \begin{align*}
   \int_{\mathcal{N}_Q}\sum_{P:\ B\subset P\subset Q}(-1)^{\dim(A^P_Q)}\sum_{\delta}\widehat{\tau}_P^Q\left(H_P(\delta y)-T\right)\int_{[N_P]}K(n_1,n;\delta)dn\theta(n_1)dn_1
   \end{align*}
   where $\delta$ runs over $P(F)\backslash Q(F)\simeq P(F)\cap M_Q(F)\backslash M_Q(F),$ and $K(n_1,n;\delta)$ is 
   \begin{align*}
   \sum_{\gamma\in F(Q,P_1)}\int_{X_M}\int_{N_Q(\mathbb{A}_F)}\sum_{\mu\in Z_{M_Q}(F)\backslash M_Q(F)}\varphi(x^{-1}n_1^{-1}\gamma^{-1}n_Q\mu mn\delta y)dn_Q\sigma(m)dm,
   \end{align*}
   where $X_M=Z_{M_Q}(\mathbb{A}_F)M_Q(F)\backslash M_Q(\mathbb{A}_F).$ For any $x\in Z_G(\mathbb{A}_F)Q(F)\backslash G(\mathbb{A}_F),$ write it as $x=n_2m_2ak,$ where $n_2\in N_Q(F)\backslash N_Q(\mathbb{A}_F),$ $m_2\in X_M,$ $a\in Z_G(\mathbb{A}_F)\backslash Z_{M_Q}(\mathbb{A}_F),$ and $k\in K.$ Therefore, $J_{Q,P_1,\chi}^T\neq 0$ unless
   \begin{equation}\label{82}
   \begin{cases}
   \sigma_{Q}^{P_1}(H_Q(a)-T)=1,\\
   a^{-1}m_2^{-1}n_2^{-1}n_1^{-1}\gamma^{-1}n_Q\mu mn\delta m_2a\in K\cdot\supp\varphi\cdot K,
   \end{cases}
   \end{equation}
   for some $a,$ $m_2,$ $n_2,$ $n_1,$ $\gamma,$ $n_Q,$ $\mu,$ $m,$ $n$ and $\delta$ defined as above. By definition of $\gamma\in F(Q, P_1)$ with $Q\subsetneq P_1,$ $\gamma\in Q(F)\backslash (P_1(F)-Q(F)).$ Hence, by Bruhat decomposition, we may write $\gamma=wn_{\gamma},$ where $w\in W_{Q}\backslash W_{P_1}-\{Id\},$ and $n_{\gamma}\in N(F).$ Then $a^{-1}m_2^{-1}n_2^{-1}n_1^{-1}\gamma^{-1}n_Q\mu mn\delta m_2a$ is of the form $q_1a^{-1}waq_2,$ where $q_1, q_2\in Q(\mathbb{A}_F).$ By assumption, $a^{-1}wa\notin Q(\mathbb{A}_F).$ An explicit computation shows that as an $n\times n$ matrix, there is some entry of $a^{-1}wa$ from diagonal elements of $a.$ Therefore, writing $a=\diag(a_1,\cdots,a_n),$ we conclude that at least one of these $a_i$'s is bounded absolutely, that is, depending only on $\supp\varphi.$ This contradicts the first equation in \eqref{82} if one takes $T$ sufficiently regular in terms of $\supp\varphi.$ This completes the proof of Proposition \ref{29'}.
   \end{proof}
    
    \begin{claim}\label{30}
    Given any $Q\subsetneq P_1,$ parabolic subgroups of $G.$ Let $x, y\in G(\mathbb{A}_F)^1$ be fixed. Then for any $n_1\in N_Q(\mathbb{A}_F)N(F)\backslash N(\mathbb{A}_F),$ $\K_{Q,\chi}(\gamma n_1x,y)$ vanishes unless $\gamma$ belongs to a finite subset of $F(Q,P_1);$ moreover, this finite subset depends only on $x ,y$ and $\supp\varphi.$ In particular, it is independent of $\chi.$
    \end{claim}
    \begin{proof}[Proof of Claim \ref{30}]
    Write $y=y_Qk,$ where $y_Q\in Q(\mathbb{A}_F)\cap G(\mathbb{A}_F)^1$ and $k\in K.$ Let $m\in M_Q(F)\backslash M_Q(\mathbb{A}_F)^1$ such that $\K_{Q}(\gamma n_1x,my)\neq0$ for any $\gamma\in F(Q,P_1)$ and any $n_1\in N_Q(\mathbb{A}_F)N(F)\backslash N(\mathbb{A}_F).$ Then there is a compact subset $\mathcal{C}$ of $G(\mathbb{A}_F)^1,$ depending only on $\supp \varphi$ such that $x^{-1}n_1^{-1}\gamma^{-1}n\mu my_1\in\mathcal{C},$ where $n\in N_Q(\mathbb{A}_F)$ and $\mu\in M_Q(F).$ Fix $\varpi\in \widehat{\Delta}_{Q}$ and let $d_0>0$ be such that $d_0\varphi$ is the highest weight of the standard representation of $G.$ Denote by $v_0$ the highest vector $(1,0,\cdots,0).$ Then one can take some constant $c_0$ such that $\|(x^{-1}n_1^{-1}\gamma^{-1}n\mu my_1)v_0\|\leq c_0$ whenever $x^{-1}n_1^{-1}\gamma^{-1}n\mu my_1\in\mathcal{C}.$ Note that $\|(x^{-1}n_1^{-1}\gamma^{-1}n\mu my_1)v_0\|$ is equal to
    $$
    e^{d_0\varpi(H_0(y_1))}\|(x^{-1}n_1^{-1}\gamma^{-1})v_0\|=e^{d_0\varpi(H_0(y))}\|(x^{-1}n_1^{-1}\gamma^{-1})v_0\|.
    $$
    Then there exists some $c_1>0$ such that 
    $$
    \|(x^{-1}n_1^{-1}\gamma^{-1}n\mu my_1)v_0\|\geq c_1 e^{-d_0\varpi(H_0(y))}e^{d\varpi(H_0(\gamma n_1 x))}.
    $$
    Therefore, we can choose some point $T_0\in \mathfrak{a}_0,$ depending only on the support of our test function $\varphi,$ such that $\widehat{\tau}_Q\left(H_0(\gamma n_1x)-H_0(y)-T_0\right)=1$ whenever $\K_{Q}(\gamma n_1x,my)$ does not vanish in $m\in M_Q(F)\backslash M_Q(\mathbb{A}_F)^1.$ On the other hand, there is some $c_2>0$ (see Lemma 5.1 of \cite{Art78}) such that $\varpi(H_0(\delta x))\leq c_2\left(1+\log\|x\|\right)$ for all $\varpi\in \widehat{\Delta}_{0},$ $x\in G(\mathbb{A}_F)^1$ and $\delta\in G(F).$ Hence for any $\gamma\in G(F)$ such that 
    \begin{equation}\label{72'}
    \widehat{\tau}_Q\left(H_0(\gamma n_1x)-H_0(y)-T_0\right)=1,\ \forall\ n_1\in N_Q(\mathbb{A}_F)N(F)\backslash N(\mathbb{A}_F),
    \end{equation}
    we have $\|\gamma n_1x\|$ is bounded by a constant multiple of a power of $\|n_1x\|\cdot e^{\|H_0(y)+T_0\|}.$ Since $N_Q(\mathbb{A}_F)N(F)\backslash N(\mathbb{A}_F)$ is compact, hence for any $n_1,$ $\|\gamma n_1x\|$ is bounded by a constant multiple of a power of $\|x\|\cdot e^{\|H_0(y)+T_0\|}.$ Then  $\|\gamma\|$ is bounded by a constant multiple of a power of $\|x\|\cdot e^{\|H_0(y)+T_0\|},$ since  $\|\gamma\|\leq \|\gamma n_1x\|\cdot \|n_1x\|^{-1}\leq C_0\|\gamma n_1x\|\cdot \|n_1x\|^{N_0}.$ Because $G(F)$ is a discrete subgroup of $G(\mathbb{A}_F)^1,$ then from the fact that the volume in $G(\mathbb{A}_F)^1$ of the set $\{x\in G(\mathbb{A}_F)^1:\ \|x\|\leq t\}$ is bounded by a constant multiple of a power of $t,$ we conclude that for any $n_1\in N_Q(\mathbb{A}_F)N(F)\backslash N(\mathbb{A}_F),$ $\K_{Q}(\gamma n_1x,y)$ vanishes unless $\gamma$ belongs to a finite subset of $F(Q,P_1).$ Then Claim \ref{30} follows from Lemma \ref{2.3}.
    \end{proof}
    
    \begin{lemma}\label{32}
    Let notation be as before. Let $\chi\in\mathfrak{X}$ be a cuspidal datum. Let $R$ be a slowly increasing function on $X_G.$ Then one has that 
    \begin{align*}
    \sum_{P\in \mathcal{P}}\frac{1}{k_P!(2\pi)^{k_P}}\int_{i\mathfrak{a}^*_P/i\mathfrak{a}^*_G}\int_{X_G}\Big|\sum_{\phi\in \mathfrak{B}_{P,\chi}}\mathcal{F}E(x,\mathcal{I}_P(\lambda,\varphi)\phi,\lambda)\overline{E(x,\phi,\lambda)}\cdot R(x)\Big|dxd\lambda
    \end{align*}
    is finite. In particular, one has that 
    \begin{equation}\label{75}
    \int_{Z_G(\mathbb{A}_F)N(F)\backslash G(\mathbb{A}_F)}\Big|\mathcal{F}_1\K_{\chi}(x,x)\cdot R(x)\Big|dx<\infty.
    \end{equation}
    \end{lemma}
    \begin{proof}
    Note that $\mathcal{F}E(x,\mathcal{I}_P(\lambda,\varphi)\phi,\lambda)$ is actually equal to $W(x,\mathcal{I}_P(\lambda,\varphi)\phi,\lambda),$ the Whittaker function corresponding to $\mathcal{I}_P(\lambda,\varphi)\phi$. Consider the spectral expansion of $\mathcal{I}_P(\lambda,\varphi)\phi$ in $\mathfrak{B}_{P,\chi},$ we then have
    \begin{align*}
    W(x,\mathcal{I}_P(\lambda,\varphi)\phi,\lambda)=\sum_{\phi_1\in \mathfrak{B}_{P,\chi}}\langle\mathcal{I}_P(\lambda,\varphi)\phi,\phi_1\rangle W(x,\phi_1,\lambda),
    \end{align*}
    which is a finite sum due to $K$-finiteness of $\varphi.$ Moreover, one has 
    \begin{claim}\label{34}
    Let $\chi\in\mathfrak{X}$ be a cuspidal datum. Then for any $\phi, \phi_1\in\mathfrak{B}_{P,\chi},$ the Selberg transform $\langle\mathcal{I}_P(\lambda,\varphi)\phi,\phi_1\rangle$ is rapidly decaying as a function of $\lambda\in i\mathfrak{a}^*_P/i\mathfrak{a}^*_G.$ 
    \end{claim}
    By Lemma 8.3.3 of \cite{JPSS79} $W_{\infty}(x,\phi_1,\lambda)$ is dominated uniformly (independent of $\lambda$) by a gauge. Hence it is rapidly decreasing. Also, according to the analysis in the proof of Corollary \ref{28'}, each local Whittaker function $W_{v}(x,\phi_{1,v},\lambda)$ is a Schwartz-Bruhat function. Combining these with the fact that $E(x,\phi,\lambda)$ can be bounded by a polynomial function independent of $\lambda$, we then conclude that 
    \begin{align*}
    \int_{i\mathfrak{a}^*_P/i\mathfrak{a}^*_G}\int_{X_G}\Big|\sum_{\phi\in \mathfrak{B}_{P,\chi}}\mathcal{F}E(x,\mathcal{I}_P(\lambda,\varphi)\phi,\lambda)\overline{E(x,\phi,\lambda)}\cdot R(x)\Big|dxd\lambda<\infty.
    \end{align*}
    Then the proof ends by noting that there are only finitely many $P$'s in $\mathcal{P}.$
    \end{proof}
    \begin{proof}[Proof of Claim \ref{34}]
    Write $\chi=\{(M,\sigma)\},$ where $M$ is a Levi subgroup of some parabolic subgroup of $G,$ and $\sigma$ is a cuspidal representation of $M.$ If $P$ is a parabolic subgroup of $G$ such that $M_P\neq M,$ then $\langle\mathcal{I}_P(\lambda,\varphi)\phi,\phi_1\rangle$ vanishes. So we may assume $P$ is such that $M_P=M.$ Recall that the test function $\varphi$ is $K$-finite, then it is $K_M$-finite. Let $\tau$ be the representation of $K_M$ on the $K_M$-span of $\varphi.$ We may assume $\varphi$ is so chosen that $\tau$ is irreducible. Write $\tau=\otimes_v\tau_v.$ Then $\tau_v$ is the trivial representation for almost all $v.$ For each $v,$ let $\tilde{\tau}_v$ be an irreducible representation of $K_v$ such that $\tau_v=\tilde{\tau}_v\mid_{K_{M,v}}.$ Assume $\tilde{\tau}_v=1$ whenever $\tau_v=1.$ Set $\tilde{\tau}=\otimes_v\tilde{\tau}_v.$ Write $x\in G(\mathbb{A}_F)$ as $nmk,$ where $m\in M(\mathbb{A}_F),$ $n\in N_P(\mathbb{A}_F)$ and $k\in K_P\simeq P(\mathbb{A}_F)\backslash G(\mathbb{A}_F),$ define the transform $S_{\lambda}\phi$ of $\phi$ evaluated at $x=nmk$ by
    \begin{equation}\label{74}
    S_{\lambda}\phi(nmk)=e^{\langle\lambda+\rho_P,H_P(m)\rangle}\int_{K_M}\langle\tilde{\tau}(k^{-1})\tilde{\tau}(k_0)\phi_{\tau},{\phi^{\vee}_{\tau}}\rangle \left(\sigma(mk_0)\phi\right)(e)dk_0,
    \end{equation}
    where $\phi_{\tau}$ and $\phi^{\vee}_{\tau}$ are vectors in the spaces of $\tilde{\tau}$ and $\tilde{\tau}^{\vee},$ where $\tilde{\tau}^{\vee}$ denotes the contragradient of $\tilde{\tau}.$ Let $Y=Z_G(\mathbb{A}_F)\backslash G(\mathbb{A}_F).$ Then $\langle\mathcal{I}_P(\lambda,\varphi)\phi,\phi_1\rangle$ is equal to
    \begin{equation}\label{75'}
    \int_{Y}\int_{K}S_{\lambda}\phi(kx)\overline{\phi_1(k)}dk\varphi(x)dx=\int_{K}\int_{Y}S_{\lambda}\phi(x)\overline{\phi_1(k)}dk\varphi(k^{-1}x)dx
    \end{equation}
    for some $\phi_{\tau}$ and $\phi^{\vee}_{\tau}.$ Apply Iwasawa decomposition to $G(\mathbb{A}_F)=N_P(\mathbb{A}_F)M(\mathbb{A}_F)K_P,$ then substitute \eqref{74} into the right hand side of \eqref{75'} to write 
    \begin{equation}\label{92}
    \langle\mathcal{I}_P(\lambda,\varphi)\phi,\phi_1\rangle=\int_{Z_G(\mathbb{A}_F)\backslash Z_M(\mathbb{A}_F)}\mathcal{F}(a)\omega_{\sigma}(a)e^{\langle\lambda,H_P(a)\rangle}d^{\times}a,
    \end{equation}
    where $\omega_{\sigma}$ is the central character of $\sigma,$ and the function $\mathcal{F}(a)$ is defined by
    \begin{align*}
    a\mapsto\int_{Z_M(\mathbb{A}_F)\backslash M(\mathbb{A}_F)}\int_{N_P(\mathbb{A}_F)}\int_{K_P}\int_{K}\int_{K_M}\mathcal{F}_0(m;k_0,k)\varphi(k^{-1}nmak')dk_0dkdk'dndm,
    \end{align*}
    with $\mathcal{F}_0(m;k_0,k)=\langle\tilde{\tau}(k^{-1})\tilde{\tau}(k_0)\phi_{\tau},{\phi^{\vee}_{\tau}}\rangle \left(\sigma(mk_0)\phi\right)(e)\overline{\phi_1(k)}.$ Then clearly $a\mapsto \mathcal{F}(a)\omega_{\sigma}(a)$ is of compact support on $Z_G(\mathbb{A}_F)\backslash Z_M(\mathbb{A}_F),$ which is nontrivial. Therefore $\langle\mathcal{I}_P(\lambda,\varphi)\phi,\phi_1\rangle$ is the Mellin transform of a function of compact support. Hence it is rapidly decreasing, ending the proof.
    \end{proof}
    \begin{prop}\label{30'}
   	Let notation be as before. Let $\chi\in\mathfrak{X}$ be a cuspidal datum. Then there exists some $T_0\in\mathfrak{a}_0$ depending only on the support of $\varphi,$ such that for any $T\in\mathfrak{a}_0$ with $T-T_0\in\mathfrak{a}_0^+,$ one has 
   	\begin{align*}
   	\int_{Z_G(\mathbb{A}_F)N(F)\backslash G(\mathbb{A}_F)}\mathcal{F}_1\Lambda^T_2\K_{\chi}(x,x)\cdot R(x)dx
   	\end{align*}
   	converges absolutely, and it is of the form 
   	\begin{equation}\label{67'}
   	\sum_{w\in W_n}\sum_QC^Q_1(T_0;w,\chi,R)e^{-\lambda_w(T)}+\sum_{w\in W_n}\sum_QC_2^Q(T_0;w,\chi,R)P_{w,Q}(T;T_0),
   	\end{equation}
   	where $C^Q_1(T_0;w,\chi,R)$ and $C^Q_2(T_0;w,\chi,R)$ are constants depending on $w,$ $\chi,$ $R$ and $T_0;$ $\lambda_{w}$ is a point $\left(\mathfrak{a}_0^{*}\right)^+,$ decided by $w\in W_n;$ and $P_{w,Q}(T;T_0)$ is a polynomial depending on $w$ and $Q,$ with $\deg P_{w,Q}(T;T_0)\leq \dim\mathfrak{a}_Q^G.$ 
    \end{prop}
    \begin{proof}
    The absolute convergence is shown by \eqref{56'}. Let $T_0\in\mathfrak{a}_0$ be a fixed point with $\beta(T_0)$ large for any $\beta\in\Delta_0,$ and let $T\in\mathfrak{a}_0$ be a variable point satisfying $\beta(T-T_0)>0$ for every $\beta.$ Then by Proposition \ref{29'} it is enough to show that the function 
    \begin{equation}\label{60}
    T\longmapsto I_R^T:=\int_{Z_G(\mathbb{A}_F)N(F)\backslash G(\mathbb{A}_F)}\left(\mathcal{F}_1k_{\chi}^T(x)-\mathcal{F}_1k_{\chi}^{T_0}(x)\right)\cdot R(x)dx
    \end{equation}
    is a constant function. Then substitute the definition of $\mathcal{F}_1k^T(x)$ and $\mathcal{F}_1k^{T_0}(x)$ to see that the only terms in \eqref{60} that depend on $T$ and $T_0$ are differences of characteristic functions $\widehat{\tau}_P\left(H_P(\delta x)-T\right)-\widehat{\tau}_P\left(H_P(\delta x)-T_0\right).$
    
    Let $P,$ $Q$ be standard parabolic subgroups of $G.$ Write $\tau_P^Q$ for the characteristic function of the subset $\mathfrak{a}_P^+=\{t\in\mathfrak{a}_P:\ \alpha(t)>0,\ \forall\ \alpha\in\Delta_P\}.$ Assume $P\subset Q.$ Then $P\backslash Q=(P\cap M_Q)N_Q\backslash M_QN_Q\cong (P\cap M_Q)\backslash M_Q.$ We write $\tau_P^Q=\tau_{P\cap M_Q}$ and $\widehat{\tau}_P^Q=\widehat{\tau}_{P\cap M_Q}.$ We shall regard these two functions as characteristic functions on $\mathfrak{a}_0$ that depend only on the projection of $\mathfrak{a}_0$ onto $\mathfrak{a}_P^Q,$ relative to the decomposition $\mathfrak{a}_0=\mathfrak{a}_0^P\oplus\mathfrak{a}_P^Q\oplus\mathfrak{a}_Q.$ Define the functions $\Gamma_P'(H,X)$ on $\mathfrak{a}_0^G\times \mathfrak{a}_0^G$ by 
    \begin{align*}
    \Gamma_P'(H,X)=\sum_{Q:\ Q\supset P}(-1)^{\dim(A_Q/A_G)} \tau_P^Q(H)\widehat{\tau}_Q(H-X),
    \end{align*}
    where $Q$ runs over standard parabolic subgroups containing $P.$ One can check that 
    \begin{equation}\label{58}
    \widehat{\tau}_P(H-X)=\sum_{Q:\ Q\supset P}(-1)^{\dim(A_Q/A_G)}\widehat{\tau}_P^Q(H)\Gamma_Q'(H,X),
    \end{equation}
    for any $P.$ In fact, one has (see \cite{Art81}) that for any $X$ and $P,$ the function $H\mapsto \Gamma_P'(H,X),$ $H\in \mathfrak{a}_P^G,$ is compactly supported. Set $H=H_P(\delta X)-T_0$ and $X=T-T_0.$ Then $H-X=H_P(\delta X)-T,$ so \eqref{58} implies that $\widehat{\tau}_P(H_P(\delta X)-T)$ is equal to 
    \begin{equation}\label{59}
    \sum_{Q\supset P}(-1)^{\dim(A^Q_G)}\widehat{\tau}_P^Q(H_P(\delta x)-T_0)\Gamma_Q'(H_P(\delta x)-T_0,T-T_0),
    \end{equation}
    where we write $A^Q_G$ for $A_Q/A_G.$ Define the function $C(x)$ on $Z^G(\mathbb{A}_F)\backslash G(\mathbb{A}_F)$ by 
    \begin{align*}
    C_{\chi}(x)=C_{\chi;P,Q}(x;T)=\widehat{\tau}_P^Q(H_P(x)-T_0)\Gamma_Q'(H_Q(x)-T_0,T-T_0)\mathcal{F}_1\K_{P,\chi}(x,x).
    \end{align*}
    Substitute formula \eqref{59} of $\widehat{\tau}_P(H_P(\delta x)-T)$ into the definition of $\mathcal{F}_1k^T(x)$ to see
    \begin{align*}
    I_R^T&=\int_{X_G}\sum_P(-1)^{\dim(A_P/A_G)}\sum_{\delta\in P(F)\backslash G(F)}\sum_{Q\supset P}(-1)^{\dim{A_Q/A_G}}\cdot C_{\chi}(\delta x)R(x)dx\\
    &=\sum_Q\int_{X_G}\sum_{P\subset Q}(-1)^{\dim(A_P/A_Q)}\sum_{\delta\in Q(F)\backslash G(F)}\sum_{\gamma\in P(F)\backslash Q(F)}C_{\chi}(\gamma\delta x)R(x)dx.
    \end{align*}
    Let $Q$ be a proper parabolic subgroup of $G.$ Let $\delta$ be a class in $Q(F)\backslash G(F)/N(F),$ identify it with an Weyl element $w_{\delta}\in W_Q\backslash W_n.$ 
    Set $N^{\delta}=N\cap w_{\delta}^{-1}Qw_{\delta}.$ Let $N_{\delta}$ be the unipotent subgroup of $G$ such that $N_{\delta}\simeq N^{\delta}\backslash N.$
    Define the partial distribution with respect to $\delta$ to be
    \begin{align*}
    I_{R}^T(\delta;Q)=\int_{X_G}\sum_{P\subset Q}(-1)^{\dim(A_P/A_Q)}\sum_{n_{\delta}\in N_{\delta}(F)}\sum_{\gamma\in P(F)\backslash Q(F)}C_{\chi}(\gamma w_{\delta}n_{\delta}x)R(x)dx.
    \end{align*}
    Since $P(F)\backslash Q(F)\simeq P(F)\cap M_Q(F)\backslash M_Q(F),$ we have 
    \begin{align*}
    I_{R}^T(\delta;Q)=\int_{X_G^{\delta}}\sum_{P\subset Q}(-1)^{\dim(A_P/A_Q)}\sum_{\gamma\in P(F)\cap M_Q(F)\backslash M_Q(F)}C_{\chi}(\gamma w_{\delta}x)R(x)dx,
    \end{align*}
    where $X^{\delta}_G=N^{\delta}(F)\backslash N(\mathbb{A}_F)\times A(\mathbb{A}_F)\times K.$ Let $N^{\delta}_1=w_{\delta}N^{\delta}w_{\delta}^{-1}\cap N_Q$ and $N^{\delta}_2=w_{\delta}N^{\delta}w_{\delta}^{-1}\cap M_Q.$ For any $x\in X_G^{\delta},$ write $x=n_Qm_Qw_{\delta}^{-1}a_Qw_{\delta}k,$ for variables $n_Q\in N^{\delta}\backslash N^{\delta}(\mathbb{A}_F),$ $m\in N_{\delta}(\mathbb{A}_F)A(\mathbb{A}_F)\cap M_Q(\mathbb{A}_F)^1,$ $a_Q\in A_Q(F_{\infty})^0\cap G(\mathbb{A}_F)^1,$ and $k\in K.$ Then one has $dx=e^{\lambda'_{\delta}H_Q(a_Q)}dn_Qdm_Qda_Qdk,$ for some $\lambda'_{\delta}\in \mathfrak{a}_0^*.$
    Note that for any $\phi\in\mathcal{B}_{Q,\chi},$ the Eisenstein series $E_Q^P\left(\gamma w_{\delta}x,\phi,\lambda\right)$ is equal to 
    $$
    E_Q^P\left(\gamma w_{\delta}n_Qm_Qw_{\delta}^{-1}a_Qw_{\delta}k,\phi,\lambda\right)=\omega_{\chi}(a_Q)E_Q^P\left(\gamma n_Q^2w_{\delta}m_Qk,\phi,\lambda\right),
    $$
    where $n_Q^2\in N^{\delta}_2(F)\backslash N^{\delta}_2(\mathbb{A}_F),$ and $\omega_{\chi}$ is the (unitary) central character of the representation in $\chi.$ Likewise, for any $n_1\in N_Q(\mathbb{A}_F)N(F)\backslash N(\mathbb{A}_F),$ we have $\phi\in\mathcal{B}_{Q,\chi},$ $E_Q^P\left(n_1\gamma w_{\delta}x,\phi,\lambda\right)$ equals
    $$
    E_Q^P\left(n_1\gamma w_{\delta}n_Qm_Qw_{\delta}^{-1}a_Qw_{\delta}k,\phi,\lambda\right)=\omega_{\chi}(a_Q)E_Q^P\left(n_1\gamma n_Q^2w_{\delta}m_Qk,\phi,\lambda\right).
    $$
    Recall that the kernel function $K_{P,\chi}(x,y)$ has spectral expansion 
    \begin{align*}
    K_{P,\chi}(x,y)=\sum_{Q\subset P}\frac{1}{n_Q^P}\int_{i\mathfrak{a}^*_{Q}/i\mathfrak{a}^*_{H}}\sum_{\phi\in\mathcal{B}_{Q,\chi}}E_Q^P\left(x,\mathcal{I}_{Q}(\lambda, \varphi)\phi,\lambda\right)\overline{E_Q^P\left(y,\phi,\lambda\right)}d\lambda.
    \end{align*}
    Hence $\mathcal{F}_1K_{P,\chi}(\gamma w_{\delta}x,\gamma w_{\delta}x)=\mathcal{F}_1K_{P,\chi}(\gamma w_{\delta}n_Qm_Qk,\gamma w_{\delta}n_Qnm_Qk).$ Also, note that $\widehat{\tau}_P^Q(H_P(\gamma x)-T_0)=\widehat{\tau}_P^Q(H_P(\gamma m_Q)-T_0).$ Therefore, one has 
    \begin{align*}
     I_{R}^T(\delta;Q)=\sum_{P\subset Q}C_{\chi,\delta}^{P,Q}\int_{A_Q(F_{\infty})^0\cap G(\mathbb{A}_F)^1}\Gamma_Q'(H_P(a_Q)-T_0,T-T_0)e^{-\lambda_{\delta}H_Q(a_Q)}da_Q,
    \end{align*}
    where $C_{\chi,\delta}^{P,Q}$ is some finite constant; $\lambda_{\delta}=-\lambda_{\delta}'+\rho_Q\in\mathfrak{a}_Q^*$ with $\rho_Q$ the half sum of positive root in $\mathfrak{a}_Q^*.$ Hence $\lambda_{\delta}$ is nonnegative. According to Lemma 2.2 in  \cite{Art81}, when $\lambda_{\delta}=0,$ the integral 
    \begin{align*}
    J_Q(\lambda_{\delta})=\int_{A_Q(F_{\infty})^0\cap G(\mathbb{A}_F)^1}\Gamma_Q'(H_P(a_Q)-T_0,T-T_0)e^{-\lambda_{\delta}H_Q(a_Q)}da_Q
    \end{align*}
    is a homogeneous polynomial of degree equal to $\dim\mathfrak{a}_Q^G.$ When $\lambda_{\delta}\neq 0,$ 
    \begin{align*}
    J_Q(\lambda_{\delta})=\sum_{Q'\supset Q}(-1)^{\dim A_{Q'}/A_Q}e^{-\lambda_{\delta,Q'}(T)}\widehat{\theta}_Q^{Q'}(\lambda_{\delta})^{-1}\theta_{Q'}(\lambda_{\delta})^{-1},
    \end{align*}
    where $\theta_{Q'}(\lambda)=\vol\left(\mathfrak{a}_{M_{Q'}}^G/\mathbb{Z}\left(\Delta_{Q'}^{\vee}\right)\right)^{-1}\prod_{\alpha\in\Delta_{Q'}}\lambda(\alpha^{\vee}),$ and
    \begin{align*}
    \widehat{\theta}_Q^{Q'}(\lambda)=\widehat{\theta}_{Q\cap M_{Q'}}(\lambda)=\vol\left(\mathfrak{a}_Q^{Q'}/\mathbb{Z}\left(\left(\widehat{\Delta}_Q^{Q'}\right)^{\vee}\right)\right)^{-1}\prod_{\alpha\in\widehat{\Delta}_Q^{Q'}}\lambda(\alpha^{\vee}).
    \end{align*}
    Then \eqref{67'} follows from $I_R^T=\sum_{Q}\sum_{\delta} I_{R}^T(\delta;Q)$ and this sum is finite.
    \end{proof}
    \begin{remark}
    If we assume that the test function $\varphi$ is supported in the regular elliptic conjugacy classes, therefore, for any $m\in M_P(\mathbb{A}_F),$ $n_0\in N_P(\mathbb{A}_F),$ whenever $P\subsetneq G,$ one has
    \begin{align*}
    \K_P(x,n_0mx)=\int_{N_P(\mathbb{A}_F)}\sum_{\mu\in Z_H(F)\setminus M_{P}(F)}\varphi(x^{-1}\mu nn_0mx)dn=0,
    \end{align*}
    since the inner summand is always vanishing. Noting that $\K_{P,\chi}(x,y)$ is $N_P(\mathbb{A}_F)$-invariant on the second variable, then by Lemma \ref{2.3} we have $\K_{P,\chi}(x,x)=0,$ for any $\chi\in\mathfrak{X}.$ Hence $I_{R}^T(\delta;P)=0$ for any $P\subsetneq G.$ To summarize, we have 
    \begin{align*}
    I_{R}^T=\sum_{P\subsetneq G}\sum_{\delta\in P(F)\backslash G(F)/N(F)}I_{R}^T(\delta;P)+I_{R}^T(G)=I_{R}^T(G),
    \end{align*}
    where $I_{R}^T(G)$ is defined by 
    \begin{align*}
    \int_{X_G}\Gamma_G'(H_G(x)-T_0,T-T_0)\int_{N(F)\backslash N(\mathbb{A}_F)}\K(n_1x,x)\theta(n_1)dn_1\cdot R(x)dx.
    \end{align*}
    which is a homogeneous polynomial of degree $\dim\mathfrak{a}_G^G=0.$ Namely, $I_{R}^T$ is a independent of $T.$ 	
    \end{remark}
    Let $R$ be a slowly increasing function on $X_G.$ Define, at least formally, that 
    \begin{align*}
    J_R=\int_{X_G}\sum_{\chi}\mathcal{F}_1\K_{\chi}(x,x)\cdot R(x)dx.
    \end{align*}
    \begin{cor}\label{36'}
    Let notation be as above. Then for any slowing increasing left $Z_G(\mathbb{A}_F)N(F)$-invariant function $R,$ $J_R$ is well defined. Moreover, we have
    \begin{equation}\label{79}
    J_R=\int_{Y_G}\sum_{\chi\in\mathfrak{X}}\widehat{\K}_{\chi}(x,x)\cdot R(x)dx=\sum_{\chi\in\mathfrak{X}}\int_{Y_G}\widehat{\K}_{\chi}(x,x)\cdot R(x)dx,
    \end{equation}
    where $Y_G:=Z_G(\mathbb{A}_F)N(\mathbb{A}_F)\backslash G(\mathbb{A}_F),$ and for any $\chi\in\mathfrak{X},$
    \begin{align*}
    \widehat{\K}_{\chi}(x,y)=\int_{N(F)\backslash N(\mathbb{A}_F)}\int_{N(F)\backslash N(\mathbb{A}_F)}\K_{\chi}(n_1x,n_2y)\theta(n_1)\bar{\theta}(n_2)dn_1dn_2.
    \end{align*}
    \end{cor}
    \begin{proof}
    Clearly the first equality in \eqref{79} comes from changing of variables if $J_R$ is well defined. To verify that the integral defining $J_R$ converges, we need 
    \begin{claim}\label{36}
    $C_2^Q(T_0;w,\chi,R)$ in \eqref{67'} is vanishing unless $\deg P_{w,Q}(T;T_0)$ is zero. 
    \end{claim}
    Then based on Proposition \ref{30'}, we have shown that 
    \begin{align*}
    \int_{Z_G(\mathbb{A}_F)N(F)\backslash G(\mathbb{A}_F)}\mathcal{F}_1\Lambda^T_2\K_{\chi}(x,x)\cdot R(x)dx=\sum_{w\in W_n}\sum_QC^Q_1(w,\chi,R)e^{-\lambda_w(T)}+C_{\chi},
    \end{align*}
    where $C^Q_1(w,\chi,R)$ is a constant depending on $w,$ $\chi,$ $R;$ $\lambda_{w}$ is a point $\left(\mathfrak{a}_0^{*}\right)^+,$ decided by $w\in W_n;$ and $C_{\chi}$ is a constant. Moreover, one can thus take $T\rightarrow \infty$ to see 
    $$
    C_{\chi}=\int_{Z_G(\mathbb{A}_F)N(F)\backslash G(\mathbb{A}_F)}\mathcal{F}_1\K_{\chi}(x,x)\cdot R(x)dx.
    $$ 
    A similar proof to that of \eqref{67'} leads to 
    \begin{align*}
    \int_{Z_G(\mathbb{A}_F)N(F)\backslash G(\mathbb{A}_F)}\mathcal{F}_1\Lambda^T_2\K(x,x)\cdot R(x)dx=\sum_{w\in W_n}\sum_QC^Q_1(w,R)e^{-\lambda_w(T)}+C,
    \end{align*}
    for some constant $C^Q_1(w,R)$ and $C.$ Since the right hand side is uniformly bounded with respect to $T$ (as long as $T$ is regular enough), one can take limit to get
    \begin{align*}
    C=\int_{Z_G(\mathbb{A}_F)N(F)\backslash G(\mathbb{A}_F)}\mathcal{F}_1\K(x,x)\cdot R(x)dx=J_R.
    \end{align*}
    Hence $J_R$ is well defined and the sum of $C_{\chi}$ over $\chi\in\mathfrak{X}$ is equal to $C$ by Corollary \ref{28'}, and this proves the second identity in \eqref{79}.
    \end{proof}
    \begin{proof}[Proof of Claim \ref{36}]
    Note that one always has $\deg P_{w,Q}(T;T_0)\leq n-1.$ Let $m\leq n-1$ be the integer such that among the nonvanishing terms $C_2^Q(T_0;w,\chi,R)P_{w,Q}(T;T_0),$ $m$ is the maximal degree. Assume that $m\geq 1.$ Then we can consider the difference operator $\Delta_{T_0}$ of step $T_0\in\mathfrak{a}_0^+,$ namely, for any function $\phi(T),$ $T\in\mathfrak{a}_0,$ $\Delta_{T_0}\phi(T)$ is defined to be $\phi(T+T_0)-\phi(T).$ For any $n\geq 1,$ let $\Delta_{T_0}^n\phi:=\Delta_{T_0}\left(\Delta_{T_0}^{n-1}\phi\right).$ Therefore, by our assumption one has that 
    \begin{align*}
    J_m^T=\frac{1}{m!}\int_{Z_G(\mathbb{A}_F)N(F)\backslash G(\mathbb{A}_F)}\Delta_{T_0}^m\mathcal{F}_1\Lambda^T_2\K_{\chi}(x,x)\cdot R(x)dx
    \end{align*}
    is exactly the coefficient of the term with the highest power $m$ plus finitely many functions which decay exponentially, by Proposition \ref{30'}. So when we let $T\rightarrow\infty,$ $J_m^T$ tends to a nonzero constant. On the other hand, $J_m^T$ converges absolutely and has a uniform upper bound. Therefore by dominant control theorem, 
    \begin{align*}
    \lim_{T\rightarrow\infty}J_m^T=\frac{1}{m!}\int_{Z_G(\mathbb{A}_F)N(F)\backslash G(\mathbb{A}_F)}\lim_{T\rightarrow\infty}\Delta_{T_0}^m\mathcal{F}_1\Lambda^T_2\K_{\chi}(x,x)\cdot R(x)dx.
    \end{align*}
    By Proposition 2.1 of \cite{Jac95} for any $x,$ $\mathcal{F}_1\Lambda^T_2\K_{\chi}(x,x)=\mathcal{F}_1\K_{\chi}(x,x)$ as long as $T$ is regular enough. Therefore, $\lim_{T\rightarrow\infty}J_m^T=0$ by \eqref{75}. Hence we have a contradiction by assuming that $m\geq1.$ So $m$ must be zero, then Claim \ref{36} follows.
    \end{proof}

    \begin{prop}\label{38'}
    	Let notation be as above. Let $A_{\varphi,fin}$ be a compact subgroup of $Z_G(\mathbb{A}_{F,fin})\backslash T(\mathbb{A}_{F,fin})$ depending only on $\varphi$ and $F.$ Let $R(x)$ be a slowly increasing function on $\mathcal{S}_0.$ Then we have
    	\begin{equation}\label{57''}
    	\int_{Z_G(\mathbb{A}_F)N(F)\backslash G(\mathbb{A}_F)}\sum_{\pi}\Big|\mathcal{F}_1\K_{\pi}(x,x)\cdot R(x)\Big|dx<\infty,
    	\end{equation}
    	where $\pi$ runs through $\mathcal{A}_0\left(G(F)\setminus G(\mathbb{A}_F),\omega^{-1}\right),$ the cuspidal spectrum.
    \end{prop}
    \begin{proof}
    	Consider the spectral expansion of $\K_0(x,y)$ (see \eqref{ker_0}):
    	\begin{align*}
    	\K_0(x,y)=\sum_{\pi}\K_{\pi}(x,y),\ \text{where}\ \K_{\pi}(x,y)=\sum_{\phi\in\mathcal{B}_{\pi}}\pi(\varphi)\phi(x)\overline{\phi(y)},
    	\end{align*}
    	where $\pi$ runs through $\mathcal{A}_0\left(G(F)\setminus G(\mathbb{A}_F),\omega^{-1}\right),$ the cuspidal spectrum; and $\mathcal{B}_{\pi}$ is a standard orthonormal basis of $\pi.$ All of the functions are rapidly decreasing in $x$ and $y.$ The first converges in $L^2$ and hence also in the space of rapidly decaying functions, by the usual estimates on the growth of cusp forms. Thus the sum over $\pi$ converges absolutely. The second sum is actually finite uniformly in $x$ and $y$ for the given $K$-finite test function $\varphi.$ Then we have
    	\begin{equation}\label{57'}
    	\int_K\int_{N(F)\backslash N(\mathbb{A}_F)}\int_{A_{F,\infty}A_{\varphi,fin}}\sum_{\pi}\Big|\mathcal{F}_1\K_{\pi}(nak,nak)\cdot R(ak)\Big|d^{\times}adndk<\infty.
    	\end{equation}
    	The proof of \eqref{57'} is similar to that of \eqref{56''}. Now \eqref{57''} follows from \eqref{57'}, this is exactly the same as that \eqref{56''} implies \eqref{56'}. 
    \end{proof}
    
	\begin{thmx}\label{39'}
	Let notation be as before. Let $s\in\mathbb{C}$ be such that $\Re(s)>1.$ Let $Y_G=Z_G(\mathbb{A}_F)N(\mathbb{A}_F)\backslash G(\mathbb{A}_F).$ Then the following integral
	\begin{align*}
	\sum_{\chi\in\mathfrak{X}}\sum_{P\in \mathcal{P}}\sum_{\phi_1\in \mathfrak{B}_{P,\chi}}\sum_{\phi_2\in \mathfrak{B}_{P,\chi}}\int_{\Lambda^*}\int_{Y_G} \Big|\langle\mathcal{I}_P(\lambda,\varphi)\phi_2,\phi_1\rangle W_{1}(x;\lambda)\overline{W_{2}(x;\lambda)}f(x,s)\Big|dxd\lambda
	\end{align*}
	is finite, and is uniformly bounded if $s$ lies in some compact subset of the right half plane $\{z:\ \Re(z)>1\}$. In particular, $I_{\infty}^{(1)}(s)$ converges absolutely for $\Re(s)>1.$ Moreover, when $\Re(s)>1,$ $I_{\infty}^{(1)}(s)$ is equal to 
	\begin{align*}
    \sum_{\chi}\sum_{P\in \mathcal{P}}\frac{1}{c_P}\sum_{\phi_1\in \mathfrak{B}_{P,\chi}}\sum_{\phi_2\in \mathfrak{B}_{P,\chi}}\int_{\Lambda^*}\langle\mathcal{I}_P(\lambda,\varphi)\phi_2,\phi_1\rangle\int_{Y_G} W_{1}(x;\lambda)\overline{W_{2}(x;\lambda)}f(x,s)dxd\lambda,
	\end{align*}
	where $\chi$ runs over proper cuspidal data, i.e., $\chi$ is not of the form $\{(G,\pi)\}.$ Particularly, as a function of $s,$ $I_{\infty}^{(1)}(s)$ is analytic in the right half plane $\{z:\ \Re(z)>1\}$. 
	\end{thmx}
	\begin{proof}
	By Lemma \ref{25'} and Corollary \ref{36'}, for any slowly increasing function $R(x)$ on $X_G=Z_G(\mathbb{A}_F)N(\mathbb{A}_F)\backslash G(\mathbb{A}_F),$
	and for any test function $\varphi\in\mathcal{C}_0\left(Z_G(\mathbb{A}_F)\backslash G(\mathbb{A}_F)\right),$ 
	\begin{align*}
	\sum_{\chi\in\mathfrak{X}}\sum_{P\in \mathcal{P}}\sum_{\phi_1\in \mathfrak{B}_{P,\chi}}\sum_{\phi_2\in \mathfrak{B}_{P,\chi}}\int_{\Lambda^*}\int_{Y_G} \langle\mathcal{I}_P(\lambda,\varphi)\phi_2,\phi_1\rangle W_{1}(x;\lambda)\overline{W_{2}(x;\lambda)}R(x)dxd\lambda
	\end{align*}
	converges. Take test functions of the form $\varphi_0*\varphi_0^*,$ where $\varphi_0^*(x)=\overline{\varphi_0(x^{-1})},$ and take $R$ to be nonnegative. Then the above integral becomes
	\begin{equation}\label{85}
	\sum_{\chi\in\mathfrak{X}}\sum_{P\in \mathcal{P}}\sum_{\phi\in \mathfrak{B}_{P,\chi}}\int_{\Lambda^*}\int_{Y_G} W(x;\mathcal{I}_P(\lambda,\varphi_0)\phi,\lambda)\overline{W(x;\mathcal{I}_P(\lambda,\varphi_0)\phi,\lambda)}R(x)dxd\lambda,
	\end{equation}
	where the Whittaker functions are defined by
	\begin{align*}
	W(x;\mathcal{I}_P(\lambda,\varphi_0)\phi,\lambda)=\int_{N(\mathbb{A}_F)}\left(\mathcal{I}_P(\lambda,\varphi_0)\phi\right)(w_0nx)e^{(\lambda+\rho_P)H_P(w_0 nx)}\theta(n)dn,
	\end{align*}
	where $w_0$ is the longest element in the Weyl group $W_n.$ Hence \eqref{85} is convergent and also nonnegative. So it converges absolutely.
	
	For arbitrary test function $\varphi\in\mathcal{C}_0\left(G(\mathbb{A}_F)\right),$ one can write $\varphi$ as a finite linear combination of convolutions $\varphi_{j,1}*\varphi_{j,2}$ with functions $\varphi_{j,i}\in C_c^r\left(G(\mathbb{A}_F)\right),$ whose archimedean components are differentiable of arbitrarily high order $r,$ $1\leq i\leq 2,$ and $j\in J$ is a finite set. Then one applies H\"older inequality to it to see
	\begin{align*}
	&\sum_{\chi\in\mathfrak{X}}\sum_{P\in \mathcal{P}}\sum_{\phi_1\in \mathfrak{B}_{P,\chi}}\sum_{\phi_2\in \mathfrak{B}_{P,\chi}}\int_{\Lambda^*}\int_{Y_G} \Big|\langle\mathcal{I}_P(\lambda,\varphi)\phi_2,\phi_1\rangle W_{1}(x;\lambda)\overline{W_{2}(x;\lambda)}R(x)\Big|dxd\lambda\\
	\leq&\sum_{j\in J}\prod_{i=1}^2\Bigg[\sum_{\chi\in\mathfrak{X}}\sum_{P\in \mathcal{P}}\sum_{\phi\in \mathfrak{B}_{P,\chi}}\int_{\Lambda^*}\int_{Y_G} W_{j,i}(x;\lambda)\overline{W_{i,j}(x;\lambda)}\cdot\big|R(x)\big|dxd\lambda\Bigg]^{1/2}<\infty,
	\end{align*}
	where $W_{j,i}(x;\lambda)=W(x;\mathcal{I}_P(\lambda,\varphi_{j,i})\phi,\lambda),$ for any $1\leq i\leq 2,$ and $j\in J.$ This proves the first part of Theorem \ref{39'}. Then apply $R(x)=f(x,s)$ for $s\in\mathbb{C}$ such that $\Re(s)>1.$ Note that when $\Re(s)>1,$ $f(x,s)$ is slowly increasing, and is uniformly bounded by a polynomial when $s$ lies in some compact subset of the right half plane $\{z:\ \Re(z)>1\},$ then the remaining part of the proof follows from subtracting the cuspidal contribution, which is justified by Proposition \ref{38'}.
	\end{proof}
    \begin{remark}
    If the base field $F$ is a function field, then it has no archimedean places. Thus $\supp W_{i}(x;\lambda)\mid_{A(\mathbb{A}_F)}\subseteq A_{\varphi,fin},$ $\forall$ $\lambda\in i\mathfrak{a}^*_P/i\mathfrak{a}^*_G,$ $1\leq i\leq 2,$ namely, the support of $\widehat{\K}_{\infty}(x,x)$ is compact. Also, in the function field case the cuspidal datums have no infinitesimal characters, so the sum over $\chi$'s is only finite. Therefore, the conclusion comes from Lemma \ref{32}.
    \end{remark}
    Note that for any $\chi$ and $P,$ the space $\mathfrak{B}_{P,\chi}$ depends only on the support and $K$-finite type of the test function $\varphi.$ Hence, given any $\lambda_P^{\circ}=(\lambda_1^{\circ},\lambda_2^{\circ},\cdots, \lambda_r^{\circ})\in \mathfrak{a}_P^*(\mathbb{C})=\mathfrak{a}_P^*\otimes\mathbb{C},$ the function $\varphi(\cdot)\exp\langle\lambda_P^{\circ},H_{M_P}(\cdot)\rangle$ shares the same support and $K$-finite type with the test function $\varphi.$ Hence, one can replace $\varphi$ in Theorem \ref{39'} with $\varphi(\cdot)\exp\langle\lambda^{\circ},H_{M_P}(\cdot)\rangle$ to get that 
    \begin{cor}\label{43cor}
    Let notation be as before. Let $s\in\mathbb{C}$ be such that $\Re(s)>1;$ and for any standard parabolic subgroup $P,$ let $\lambda_P^{\circ}=(\lambda_1^{\circ},\lambda_2^{\circ},\cdots, \lambda_r^{\circ})$ be a fixed point in $ \mathfrak{a}_P^*(\mathbb{C}).$  Let $Y_G=Z_G(\mathbb{A}_F)N(\mathbb{A}_F)\backslash G(\mathbb{A}_F).$ Then the following integral
    \begin{align*}
    \sum_{\chi\in\mathfrak{X}}\sum_{P\in \mathcal{P}}\sum_{\phi_1, \phi_2\in \mathfrak{B}_{P,\chi}}\int_{\Lambda^*}\int_{Y_G} \Big|\langle\mathcal{I}_P(\lambda+\lambda_P^{\circ},\varphi)\phi_2,\phi_1\rangle W_{1}(x;\lambda)\overline{W_{2}(x;\lambda)}f(x,s)\Big|dxd\lambda
    \end{align*}
    is finite, and is uniformly bounded if $s$ lies in some compact subset of the right half plane $\{z:\ \Re(z)>1\}$.
    \end{cor} 
   \begin{remark}
   	Let notation be as before, and let $\varphi\in \mathcal{C}_0(G(\mathbb{A}_F)),$ to apply Theorem \ref{39'}, one still needs to verify that the function $\varphi(\cdot)\exp\langle\lambda^{\circ},H_{M_P}(\cdot)\rangle$ lies in $\mathcal{C}_0(G(\mathbb{A}_F))$ as well. Noting that they have the same support, which excludes the possible singularities of $\exp\langle\lambda^{\circ},H_{M_P}(\cdot)\rangle,$ one concludes that for any $\varphi\in \mathcal{C}_0(G(\mathbb{A}_F)),$ the function  $\varphi(\cdot)\exp\langle\lambda^{\circ},H_{M_P}(\cdot)\rangle\in \mathcal{C}_0(G(\mathbb{A}_F)).$ Then Corollary \ref{43cor} follows from Theorem \ref{39'}.
   	\end{remark}

   \section{Rankin-Selberg Convolutions for Generic Representations}\label{6.2.}
   By Theorem \ref{39'}, we see that when $\Re(s)>1,$ $I_{\infty}^{(1)}(s)$ is equal to 
   \begin{equation}\label{95}
   \sum_{\chi}\sum_{P\in \mathcal{P}}\frac{1}{c_P}\sum_{\phi_1\in \mathfrak{B}_{P,\chi}}\sum_{\phi_2\in \mathfrak{B}_{P,\chi}}\int_{\Lambda^*}\langle\mathcal{I}_P(\lambda,\varphi)\phi_2,\phi_1\rangle\Psi\left(s,W_{1},W_{2};\lambda\right)d\lambda,
   \end{equation} 
   where $\Psi\left(s,W_{1},W_{2};\lambda\right)=\int_{Z_G(\mathbb{A}_F)N(\mathbb{A}_F)\backslash G(\mathbb{A}_F)} W_{1}(x;\lambda)\overline{W_{2}\left(x;\lambda\right)} f(x,s)dx,$ and the Whittaker function $W_{i}\left(x,\lambda\right)=\int_{N(\mathbb{A}_F)}\phi_{i}(w_0nx)e^{(\lambda+\rho_P)H_P(w_0 nx)}\theta(n)dn,$ $1\leq i\leq 2.$
   
   For our purpose, we need to show that $I_{\infty}^{(1)}(s)$ is a holomorphic multiple of $L(s,\tau).$ So we have to compute \eqref{95} explicitly (up to an entire factor), then continue it to a meromorphic function which is a holomorphic multiple of $L(s,\tau)$ as we desired. To achieve that, we start with computing each $\Psi_{P,\chi}\left(s,W_{1},W_{2};\lambda\right)$ associated to a standard parabolic subgroup $P$ and a cuspidal datum $\chi=(M_P,\sigma)\in\mathfrak{X}.$ 
   
   Let $P$ be a standard parabolic subgroup of $G$ of type $(n_1,n_2,\cdots,n_r),$ $1\leq r\leq n,$ with $n_1+n_2+\cdots n_r=n.$ Let $\chi\in \mathfrak{X}$ be represented by $(M_P,\sigma).$ Let $\mathfrak{B}_{P,\chi}$ be a orthonormal basis of the Hilbert space $\mathcal{H}_{P,\chi}.$ For $\phi_i\in \mathfrak{B}_{P,\chi},$ $1\leq i\leq 2,$ define the Whittaker function associated to $\phi_{i}$ parameterized by $\lambda\in i\mathfrak{a}^*_P/i\mathfrak{a}^*_G$ by
   \begin{align*}
   W_{P,\chi,i}\left(x,\lambda\right)=W\left(x,\phi_{i},\lambda\right):=\int_{N(\mathbb{A}_F)}\phi_{i}(w_0nx)e^{(\lambda+\rho_P)H_P(w_0 nx)}\theta(n)dn,
   \end{align*}
   where $w_0$ is the longest element in the Weyl group $W_n.$ Define
   \begin{align*}
   \Psi_{P,\chi}\left(s,W_{1},W_{2};\lambda,\Phi\right)=\int_{Z_G(\mathbb{A}_F)N(\mathbb{A}_F)\backslash G(\mathbb{A}_F)} W_{P,\chi,1}(x;\lambda)\overline{W_{P,\chi,2}\left(x;\lambda\right)} f(x,s)dx.
   \end{align*}
   
   From now on, we fix such a standard parabolic subgroup $P$ of type $(n_1,\cdots,n_r)$ and a cuspidal datum $\chi=(M_P,\sigma)\in\mathfrak{X},$ where $\sigma$ is a unitary representation of $M$ of central character $\omega.$ Then there exist $r$ cuspidal representations $\pi_i$ of $\GL_{n_i}(\mathbb{A}_F),$ $1\leq i\leq r,$ such that $\sigma\simeq \pi_1\oplus\pi_2\oplus\cdots\oplus\pi_r.$ Let $\pi= \Ind_{P(\mathbb{A}_F)}^{G(\mathbb{A}_F)}\left(\pi_{1},\pi_{2}\,\cdots,\pi_{r}\right).$ For any $\lambda=(\lambda_1,\lambda_2,\cdots,\lambda_r)\in i\mathfrak{a}^*_P/i\mathfrak{a}^*_G,$ denote by
   \begin{align*}
   \pi_{\lambda}= \Ind_{P(\mathbb{A}_F)}^{G(\mathbb{A}_F)}\left(\pi_{1}\otimes|\cdot|_{\mathbb{A}_F}^{\lambda_1},\pi_{2}\otimes|\cdot|_{\mathbb{A}_F}^{\lambda_2},\cdots,\pi_{r}\otimes|\cdot|_{\mathbb{A}_F}^{\lambda_r}\right).
   \end{align*}
   Then $\pi_{\lambda}$ is also a unitary automorphic representation of $G(\mathbb{A}_F).$ Fix $\phi_1,$ $\phi_2\in\mathcal{B}_{P,\chi}$ and a point $\lambda=(\lambda_1,\lambda_2,\cdots,\lambda_r)\in i\mathfrak{a}^*_P/i\mathfrak{a}^*_G.$ Write $W_{i}\left(x,\lambda\right)=W_{P,\chi,i}\left(x,\lambda\right),$ and $\Psi\left(s,W_{1},W_{2};\lambda,\Phi\right)=\Psi_{P,\chi}\left(s,W_{1},W_{2};\lambda,\Phi\right).$ Since $\lambda\in i\mathfrak{a}^*_P/i\mathfrak{a}^*_G,$ $\lambda=-\bar{\lambda},$ one has
   \begin{align*}
   \Psi\left(s,W_{1},W_{2};\lambda,\Phi\right)=\int_{Z_G(\mathbb{A}_F)N(\mathbb{A}_F)\backslash G(\mathbb{A}_F)} W_{1}(x;\lambda)\overline{W_{2}\left(x;-\bar{\lambda}\right)} f(x,s)dx.
   \end{align*}
   Since $W_1(x;\lambda)$ and $W_{2}\left(x;-\bar{\lambda}\right)$ are dominant by some gauge, and $f(x,s)$ is slowly increasing when $\Re(s)>1,$ then $\Psi\left(s,W_{1},W_{2};\lambda,\Phi\right)$ converges absolutely and normally when $\Re(s)>1.$
   Note that $\pi_{i}=\otimes_v'\pi_{i,v}$, $1\leq i \leq r,$ where, for each $v\in\Sigma_F,$ $\pi_{i,v}$ is a unitary irreducible representation of $\GL_{n_i}(F_v),$ of Whittaker tape. Then for each $v\in\Sigma_F$ and $\lambda=(\lambda_1,\lambda_2,\cdots,\lambda_r)\in i\mathfrak{a}^*_P/i\mathfrak{a}^*_G,$ denote by $\pi_{v}=\Ind_{M_P(F_v)}^{G(F_v)}\left(\pi_{1,v},\pi_{2,v},\cdots,\pi_{r,v}\right)$ and
   \begin{align*}
   \pi_{\lambda,v}=\Ind_{M_P(F_v)}^{G(F_v)}\left(\pi_{1,v}\otimes|\cdot|_{F_v}^{\lambda_1},\pi_{2,v}\otimes|\cdot|_{F_v}^{\lambda_2},\cdots,\pi_{r,v}\otimes|\cdot|_{F_v}^{\lambda_r}\right).
   \end{align*}
   Then $\pi=\otimes_v'\pi_v$ and $\pi_{\lambda}=\otimes_v'\pi_{\lambda,v}.$ Recall that $f(x,s)=\prod_vf_v(x_v,s),$ where
   $$
   f_v(x_v,s)=\tau_v(\det x_v)|\det x_v|_{F_v}^s\int_{Z_G(F_v)}\Phi_v[(0,\cdots,t_v)x_v]\tau_v^{n}(t_v)|t_v|_{F_v}^{ns}d^{\times}t_v,
   $$
   if $\Phi=\otimes_v'\Phi_{v}.$ Since $\phi_1$ and $\phi_2$ both have central character $\omega_{\lambda}=\omega,$ which is unitary. So is $W_{1}(x;\lambda)$ and $W_{2}(x;\lambda).$ Hence one can rewrite $\Psi\left(s,W_{1},W_{2};\lambda,\Phi\right)$ as 
   \begin{equation}\label{100}
   \int_{N(\mathbb{A}_F)\backslash G(\mathbb{A}_F)} W_{1}(x;\lambda)\overline{W_{2}\left(x;-\bar{\lambda}\right)} \Phi(\eta x)\tau(\det x)|\det x|_{\mathbb{A}_F}^sdx,
   \end{equation}
   where $\eta=\left(0,\cdots,0,1\right)\in F^n.$ According to the definition we can write $\phi_i=\otimes_v'\phi_{i,v},$ $1\leq i\leq 2$. Thus one can factor $W_{i}(x;\lambda)$ as $\prod_{v\in\Sigma_F}W_{i,v}(x_v;\lambda),$ where
   \begin{align*}
   W_{i,v}(x_v;\lambda)=\int_{N(F_v)}\phi_{i,v}(w_0nx_v)e^{(\lambda+\rho_P)H_P(w_0 nx_v)}\theta(n)dn,\ x_v\in F_v,\ 1\leq i\leq 2.
   \end{align*}
   We may assume $\Phi=\otimes_v'\Phi_{v}$ and $\phi_i=\otimes_{v}'\phi_{i,v},$ $i=1,2.$ Then one has
   \begin{align*}
   \Psi\left(s,W_{1,v},W_{2,v};\lambda,\Phi\right)=\prod_{v\in \Sigma_F}\Psi_v\left(s,W_{1,v},W_{2,v};\lambda,\Phi_v\right),
   \end{align*}
   where each local factor $\Psi_v\left(s,W_{1,v},W_{2,v};\lambda,\Phi_v\right)$ is defined to be
   \begin{equation}\label{101}
   \int_{N(F_v)\backslash G(F_v)} W_{1,v}(x_v;\lambda)\overline{W_{2,v}\left(x_v;-\bar{\lambda}\right)} \Phi_v(\eta x_v)\tau(\det x_v)|\det x_v|_{F_v}^sdx_v,
   \end{equation}
   where $W_{i,v}(x_v;\lambda)=\int_{N(F_v)}\phi_{i,v}(w_0nx)e^{(\lambda+\rho_P)H_{P,v}(w_0 nx)}\theta(n)dn,$ $1\leq i\leq 2.$ Since $W_{1,v}(x;\lambda)$ and $W_{2,v}\left(x;-\bar{\lambda}\right)$ are dominant by some local gauge, and $f_v(x_v,s)$ is slowly increasing when $\Re(s)>1,$ then $\Psi\left(s,W_{1,v},W_{2,v};\lambda,\Phi_v\right)$ converges absolutely and normally when $\Re(s)>1,$ for any $v\in\Sigma_F.$
   
   \subsection{Local Theory for $\Psi_v\left(s,W_{1,v},W_{2,v};\lambda,\Phi_v\right)$} In this section, we shall compute each local integral representation $\Psi_v\left(s,W_{1,v},W_{2,v};\lambda,\Phi_v\right)$ defined via \eqref{101}. Let $v\in \Sigma_F$ be a place of $F.$ Note that $v$ may be archimedean or nonarchimedean. Let $\mathfrak{u}=(u_{j,l})_{1\leq j,l\leq n}\in N(F_v),$ the unipotent of $\GL_n,$ we denote by $N_j^0(\mathfrak{u})$ the matrix
   \begin{align*}
   \begin{pmatrix}
   1&u_{12}&\cdots&\cdots&u_{1,n-j+2}&u_{1,n-j+3}&\cdots&\cdots&u_{1,n}\\
   &1&\cdots&\cdots&u_{2,n-j+2}&u_{2,n-j+3}&\cdots&\cdots&u_{2,n}\\
   &&\ddots&\vdots&\vdots&\vdots&\vdots&\vdots&\vdots\\
   &&&1&u_{n-j+1,n-j+2}&u_{n-j+1,n-j+3}&\cdots&\cdots& u_{n-j+1,n}\\
   &&&&1&0&\cdots&\cdots& 0\\
   &&&&&1&\cdots&\cdots& u_{n-j+3,n}\\
   &&&&&&\ddots&\vdots&\vdots\\
   &&&&&&&1&u_{n-1,n}\\
   &&&&&&&&1 
   \end{pmatrix}
   \end{align*}
   associated to $\mathfrak{u}.$ Denote by $N_j^1(\mathfrak{u})$ the matrix
   \begin{align*}
   \begin{pmatrix}
   1&u_{12}&\cdots&\cdots&u_{1,n-j+2}^1&u_{1,n-j+3}&\cdots&\cdots&u_{1,n}\\
   &1&\cdots&\cdots&u_{2,n-j+2}^1&u_{2,n-j+3}&\cdots&\cdots&u_{2,n}\\
   &&\ddots&\vdots&\vdots&\vdots&\vdots&\vdots&\vdots\\
   &&&1&0&u_{n-j+1,n-j+3}&\cdots&\cdots& u_{n-j+1,n}\\
   &&&&1&0&\cdots&\cdots& 0\\
   &&&&&1&\cdots&\cdots& u_{n-j+3,n}\\
   &&&&&&\ddots&\vdots&\vdots\\
   &&&&&&&1&u_{n-1,n}\\
   &&&&&&&&1 
   \end{pmatrix}
   \end{align*}
   associated to $\mathfrak{u},$ where $u^1_{l,n-j+2}=u_{l,n-j+2}-u_{l,n-j+1}u_{n-j-1,n-j+2},$ $1\leq l\leq n-j.$ For any $2\leq j\leq n,$ we denote by $N_{j}^0(\mathfrak{u}^*)$ the matrix
   \begin{align*}
   \begin{pmatrix}
   1&u_{12}&\cdots&u_{1,n-j+1}&u_{1,n}&u_{1,n-j+2}&\cdots&\cdots&u_{1,n-1}\\
   &1&\cdots&u_{2,n-j+1}&u_{2,n}&u_{2,n-j+2}&\cdots&\cdots&u_{2,n-1}\\
   &&\ddots&\vdots&\vdots&\vdots&\vdots&\vdots&\vdots\\
   &&&1&u_{n-j+1,n}&u_{n-j+1,n-j+2}&\cdots&\cdots& u_{n-j+1,n-1}\\
   &&&&1&0&\cdots&\cdots& 0\\
   &&&&&1&\cdots&\cdots& u_{n-j+2,n-1}\\
   &&&&&&\ddots&\vdots&\vdots\\
   &&&&&&&1&u_{n-2,n-1}\\
   &&&&&&&&1 
   \end{pmatrix};
   \end{align*}
   and for $\mathfrak{u}=(u_{j,l})_{1\leq j,l\leq n}\in N(F_v),$ we let $N_{j+2}^0(\mathfrak{u}'')$ represent the matrix 
   \begin{align*}
   \begin{pmatrix}
   1&u_{12}&\cdots&u_{1,n-j-1}&u_{1,n}&u_{1,n-j}'&\cdots&\cdots&u_{1,n-1}\\
   &1&\cdots&u_{2,n-j-1}&u_{2,n}&u_{2,n-j}'&\cdots&\cdots&u_{2,n-1}\\
   &&\ddots&\vdots&\vdots&\vdots&\vdots&\vdots&\vdots\\
   &&&1&u_{n-j-1,n}&u_{n-j-1,n-j}'&\cdots&\cdots& u_{n-j-1,n-1}\\
   &&&&1&0&\cdots&\cdots& 0\\
   &&&&&1&\cdots&\cdots& u_{n-j,n-1}\\
   &&&&&&\ddots&\vdots&\vdots\\
   &&&&&&&1&u_{n-2,n-1}\\
   &&&&&&&&1 
   \end{pmatrix},
   \end{align*}
   with $u_{k,n-j}'=u_{k,n-j}+s_{n-j,n}u_{k,n-j-1},$ $1\leq k\leq n-j-1.$ 
   
   Let $w_{j}$ be the simple root of $\GL_n$ corresponding to the permutation $(j,j+1),$ $1\leq j\leq n-1.$ Let $\tau_n$ be the longest element in the Weyl group $W_n$ of $\GL_n,$ $n\geq 2.$ Then $\tau_n=w_{n-1}w_{n-2}w_{n-1}w\cdots w_{1}w_2\cdots w_{n-1},$ for any $n\geq 2.$
   
   Let $\alpha=(\alpha_1,\alpha_2,\cdots,\alpha_{n-1})\in F_v^{n-1}.$ Denote by $\theta_{\alpha}$ the character on $F_v^{n-1}$ such that $\theta_{\alpha}(x_1,\cdots,x_{n-1})=\langle\alpha_1x_1+\cdots+\alpha_{n-1}x_{n-1}\rangle.$ Extending $\theta_{\alpha}$ to a character on $N(F_v)$ by $\theta_{\alpha}(\mathfrak{u})=\theta(\alpha_{1}u_{12}+\alpha_2u_{23}+\cdots+\alpha_{n-1}u_{n-1,n}),$ where $\mathfrak{u}=(u_{k,l})_{1\leq k,l \leq n}\in N(F_v).$ Let $\phi_v\in \pi_v.$ Define the Whittaker function associated to $\phi$ and $\alpha$ by
   \begin{align*}
   W_v(\alpha,\lambda)=\int_{N(F_v)}\phi_v\left(\tau_n\mathfrak{u}\right)e^{(\lambda+\rho)H_B(\tau_n\mathfrak{u})}\theta_{\alpha}(\mathfrak{u})d\mathfrak{u}.
   \end{align*}
   Let $\pi_{v,\lambda}=\Ind_{B(F_v)}^{GL_{n}(F_v)}\left(\chi_{v,1}|\cdot|^{\lambda_1},\cdots,\chi_{v,n}|\cdot|^{\lambda_n}\right)$ be a principal series. Let $\mathfrak{u}_{n-1}=(u_1,u_2,\cdots,c_{n-1})\in F_v^{n-1}.$  For any $\alpha=(\alpha_1,\alpha_2,\cdots,\alpha_{n-1})\in F_v^{n-1},$ we set
   \begin{align*}
   W_{v,j}(\widetilde{\alpha}_{n-1};\widetilde{\lambda})=\int_{N_{n-1}(F_v)}\phi_v\left(\tau_{n-1}\mathfrak{u}\right)e^{(\lambda+\rho)H_{B_{n-1}}(\tau_{n-1}\mathfrak{u})}\theta_{\widetilde{\alpha}_{n-1}}(\mathfrak{u})d\mathfrak{u},
   \end{align*}
   with $\widetilde{\alpha}_{n-1}=(a(u_1)^{-1}a(u_2)\alpha_1,\cdots,a(u_{n-2})^{-1}a(u_{n-1})\alpha_{n-2})\in F_v^{n-2}.$ Hence the function $W_{v,j}(\widetilde{\alpha}_{n-1};\widetilde{\lambda})$ is a Whittaker function on $\GL_{n-1}$ associated to the principal series representation $\Ind_{B_{n-1}(F_v)}^{GL_{n-1}(F_v)}\left(\chi_{v,2}|\cdot|^{\lambda_2},\cdots,\chi_{v,n}|\cdot|^{\lambda_n}\right)$ and parameter $\widetilde{\alpha}_{n-1}\in F_v^{n-2}.$ Let $\chi_{v,l}^{ur}=|\cdot|_v^{i\nu_l},$ $\nu_l\in\mathbb{R},$ $1\leq l\leq n.$ Denote by $z_{k,l}=\lambda_k-\lambda_l+i\nu_k-i\nu_l\in\mathbb{C},$ $1\leq k<l\leq n.$

   Let $x_v\in F_v.$ If $v$ is an archimedean place, then define the functions $a$ and $s$ on $F_v$ as follows:
   \begin{align*}
   a(x_v)=\begin{cases}
   (1+|x_v|_v^2)^{-1/2},\ s(x_v)=x_va(x_v),\ \text{if $F_v\simeq \mathbb{R};$}\\
   (1+|x_v|_v)^{-1/2},\ s(x_v)=\overline{x}_va(x_v),\ \text{if $F_v\simeq \mathbb{C}.$}\\
   \end{cases}
   \end{align*}	
   If $v$ is a nonarchimedean place, then define the functions $a$ and $s$ on $F_v$ as follows:
   \begin{align*}
   a(x_v)=\begin{cases}
   1,\  \text{if $x_v\in\mathcal{O}_{F_v};$}\\
   x_v^{-1},\ \text{otherwise;}\\
   \end{cases}\ s(x_v)=\begin{cases}
   0,\ \text{if $x_v\in\mathcal{O}_{F_v};$}\\
   1,\ \text{otherwise.}
   \end{cases}
   \end{align*}
   \begin{lemma}\label{43lem}
   	Let notation be as above. Assume that $\pi_v$ is right $K_v$-finite. Let $v\in\Sigma_F$ be an arbitrary place and let  $\pi_v=\Ind_{B(F_v)}^{GL_{n}(F_v)}\left(\chi_{v,1},\cdots,\chi_{v,n}\right)$ be a principal series. Then the Whittaker function $W_{v}(\alpha;\lambda)$ is equal to
   	\begin{equation}\label{43*}
   	\sum_{j\in\boldsymbol{J}}\int_{F_v^{n-1}}W_{v,j}(\widetilde{\alpha}_{n-1};\widetilde{\lambda})\widetilde{\theta}_n(u_1,\cdots,u_{n-1})\prod_{l=2}^{n}|a(u_{n-l+1})|_v^{1+z_{1,l}}\prod_{j=1}^{n-1}du_{j},
   	\end{equation}
   	where $j$ runs over a finite index set depending only on the $K_v$-type of $\phi_v;$ and $\widetilde{\theta}_n(u_1,\cdots,u_{n-1})=\theta(\alpha_{n-1}u_{n-1}-\alpha_{n-2}c(u_{n-1})u_{n-2}-\alpha_{n-3}c(u_{n-2})u_{n-3}-\cdots-\alpha_{n-j}c(u_{n-j+1})u_{n-j}-\cdots-\alpha_{1}c(u_2)u_{1}).$ 
   \end{lemma}
   \begin{proof}
   	Assume that $v$ is an archimedean place. Let $r\in\mathbb{R}$ and $\theta\in[0,2\pi).$ Then by a straightforward computation we have the Iwasawa decomposition 
   	\begin{equation}\label{106e}
   	\begin{pmatrix}
   	1 &re^{i\theta}\\
   	& 1
   	\end{pmatrix}=\begin{pmatrix}
   	1 &\\
   	\frac{re^{-i\theta}}{1+r^2}& 1
   	\end{pmatrix}\begin{pmatrix}
   	e^{-i\theta}\sqrt{1+r^2} &\\
   	& \frac{e^{i\theta}}{\sqrt{1+r^2} }
   	\end{pmatrix}\begin{pmatrix}
   	\frac{e^{i\theta}}{\sqrt{1+r^2} }&\frac{re^{2i\theta}}{\sqrt{1+r^2} }\\
   	-\frac{re^{-2i\theta}}{\sqrt{1+r^2} }& \frac{e^{-i\theta}}{\sqrt{1+r^2} }
   	\end{pmatrix}.
   	\end{equation}
   	
   	If $v$ is a nonarchimedean place of $F.$ We then fix an uniformizer $\varpi_v$ of $F_v^{\times}.$ For any $u\in F_v,$ one can write $u=u^{\circ}\varpi_v^m,$ for some $m\in\mathbb{Z},$ where $u^{\circ}\in \mathcal{O}_{F_v}^{\times}.$ If $m\geq 0,$ then $u\in\mathcal{O}_{F_v},$ implying that $\begin{pmatrix}
   	1 &u\\
   	& 1
   	\end{pmatrix}\in GL(n,\mathcal{O}_{F_v}).$ If $m<0,$ then one has that   
   	\begin{equation}\label{106f}
   	\begin{pmatrix}
   	1 &u\\
   	& 1
   	\end{pmatrix}=\begin{pmatrix}
   	1 &\\
   	u^{-1}& 1
   	\end{pmatrix}\begin{pmatrix}
   	u&\\
   	& u^{-1}
   	\end{pmatrix}\begin{pmatrix}
   	u^{-1}&1\\
   	1&0
   	\end{pmatrix}.
   	\end{equation}
   	Let $v$ be arbitrary and $u\in F_v.$ For any $2\leq j\leq n,$ $1\leq l\leq 4,$ let $M_l(u)=M_{l,j}(u)$ be the matrix defined by
   	\begin{align*}
   	M_1(u)&=\begin{pmatrix}
   	I_{n-j}&\\
   	& 1&u\\
   	& &1\\
   	&&&I_{j-2}
   	\end{pmatrix},\quad M_2(u)=\begin{pmatrix}
   	I_{n-j}&\\
   	& a(u)^{-1}\\
   	& s(u)&a(u)\\
   	&&&I_{j-2}
   	\end{pmatrix};\\
   	M_3(u)&=\begin{pmatrix}
   	I_{n-j}&\\
   	&{a(u)}^{-1} \\
   	& &a(u)\\
   	&&&I_{j-2}
   	\end{pmatrix},\quad M_4(u)=\begin{pmatrix}
   	I_{n-j}&\\
   	& 1\\
   	& c(u)&1\\
   	&&&I_{j-2}
   	\end{pmatrix},
   	\end{align*}
   	where $c(u)=a(u)s(u).$ Let $\tau_{n,j}=w_{n-1}\cdots w_{n-j+1},$ $2\leq j\leq n.$ 
   	
   	Denote by $w_0=Id_n,$ the identity element. Then one has $N_2^0(\mathfrak{u})=\mathfrak{u},$ and for any $j\geq 2,$ $N_j^0(\mathfrak{u})=N_j^1(\mathfrak{u})M_1(u_{n-j+1,n-j+2})$ and $w_{n-j+1}N_j^1(\mathfrak{u})w_{n-j+1}=N_{j+1}^0(\mathfrak{u}'),$ where $\mathfrak{u}'=(u_{k,l}')_{1\leq k,l\leq n}\in N(F_v)$ is defined by $u_{k,l}'=u_{k,l-1}$ if $l=n-j+2;$ $u_{k,l}'=u_{k,l+1}$ if $l=n-j+1;$ $u_{k,l}'=u_{k+1,l}$ if $k=n-j+2;$ and $u_{k,l}'=u_{k,l},$ otherwise. Now applying \eqref{106f} one then has that $M_1(u_{n+1-j,n-j+2})=M_2(u_{n+1-j,n-j+2})k=M_4(u_{n+1-j,n-j+2})M_3(u_{n+1-j,n-j+2})k,$ where $k\in K(F_v),$ the maximal compact subgroup of $\GL(n,F_v).$ Consequently, we have $\tau_nN_j^0(\mathfrak{u})=\tau_nN_j^1(\mathfrak{u})M_1(u_{n-j+1,n-j+2})=\tau_nw_{n-j+1}N_{j+1}^0(\mathfrak{u}')w_{n-j+1}M_1(u_{n-j+1,n-j+2}),$ which is equal to $\tau_nw_{n-j+1}N_{j+1}^0(\mathfrak{u}')w_{n-j+1}M_4(u_{n+1-j,n-j+2})M_3(u_{n+1-j,n-j+2})k.$ 
   	
   	Note that we have $w_{n-j+1}M_4(u_{n+1-j,n-j+2})=M_1(u_{n+1-j,n-j+2})w_{n-j+1}$ and $N_{j+1}^0(\mathfrak{u}')M_1(u_{n+1-j,n-j+2})=M_1(u_{n+1-j,n-j+2})N_{j+1}^0(\tilde{\mathfrak{u}}),$ where $\tilde{\mathfrak{u}}=(\tilde{u}_{k,l})_{1\leq k,l\leq n}\in N(F_v)$ is defined by $\tilde{u}_{k,l}=u_{k,l}'+u_{n+1-j,n-j+2}u_{k+1,l}'$ if $k=n-j+1;$ $\tilde{u}_{k,l}=u_{k,l}'+u_{n+1-j,n-j+2}u_{k,l-1}'$ if $l=n-j+2;$ and $\tilde{u}_{k,l}=u_{k,l}'$ otherwise. Therefore,
   	\begin{align*}
   	\tau_nN_j^0(\mathfrak{u})&=\tau_nw_{n-j+1}M_1(u_{n+1-j,n-j+2})N_{j+1}^0(\tilde{\mathfrak{u}})w_{n-j+1}M_3(u_{n+1-j,n-j+2})k\\
   	&=M'_1(u_{n+1-j,n-j+2})\tau_nw_{n-j+1}N_{j+1}^0(\tilde{\mathfrak{u}})M_3(u_{n+1-j,n-j+2})^{-1}w_{n-j+1}k,
   	\end{align*}
   	where $M'_1(u_{n+1-j,n-j+2})=\tau_nw_{n-j+1}M_1(u_{n+1-j,n-j+2})w_{n-j+1}\tau_n^{-1}.$ Then one has that $M'_1(u_{n+1-j,n-j+2})\in N(F_v).$ Let $2\leq j\leq n$ and $\phi_{v,\lambda}=\phi_ve^{(\lambda+\rho)H_B(\cdot))}.$ Then
   	\begin{align*}
   	W_v(\alpha;\lambda)&=\int_{N(F_v)}\phi_v\left(\tau_nN_2^0(\mathfrak{u})\right)e^{(\lambda+\rho)H_B(\tau_nN_2^0(\mathfrak{u}))}\theta_{\alpha}(\mathfrak{u})d\mathfrak{u}\\
   	&=\int_{N(F_v)}\phi_{v,\lambda}\left(M'_1(u_{n-1,n})\tau_nw_{n-1}N_{3}^0(\tilde{\mathfrak{u}})M_3(u_{n-1,n})^{-1}w_{n-1}k\right)\theta_{\alpha}(\mathfrak{u})d\mathfrak{u}\\
   	&=\int_{N(F_v)}\phi_{v,\lambda}\left(M_3^{\tau_n}(u_{n-1,n})\tau_nw_{n-1}N_{3}^0(\mathfrak{u}^*)w_{n-1}k\right)\theta_{\alpha}^{(2)}(\mathfrak{u})d\mathfrak{u},
   	\end{align*}
   	where $M_3^{\tau_n}(u_{n-1,n})=\diag(a(u_{n-1,n}),a(u_{n-1,n})^{-1},I_{n-2});$ $\theta_{\alpha}^{(2)}(\mathfrak{u})=\theta(\alpha_1u_{12}+\cdots+\alpha_{n-3}u_{n-3,n-2}+\alpha_{n-2}a(u_{n-1,n})u_{n-2,n-1}+\alpha_{n-1}u_{n-1,n})\cdot \theta(-\alpha_{n-2}c(u_{n-1,n})u_{n-2,n}).$
   	
   	Denote by $M_{2}^{\tau_n}(\mathfrak{u})=I_{n}.$ Let $j\geq 3$ and $M_l,$ $3\leq l\leq j,$ be matrices. Denote by $\prod_{l=2}^jM_l$ the matrix $M_2\cdots M_j.$ Define the matrix
   	\begin{align*}
   	M_{j}^{\tau_n}(\mathfrak{u})=\prod_{l=3}^j
   	\begin{pmatrix}
   	a(u_{n-l+2,n})&\\
   	&I_{l-3}\\
   	&&a(u_{n-l+2,n})^{-1} \\
   	&&&I_{n-l+1}
   	\end{pmatrix}.
   	\end{align*}
   	Write $a_{k,l}$ for $a(u_{k,l});$ and $c_{k,l}$ for $c(u_{k,l}).$ Let $\beta_{k}(\mathfrak{u})=a_{k,n}^{-1}a_{k+1,n}.$ Denote by $\theta_{\alpha}^{(j)}(\mathfrak{u})$ the product of $\theta(\alpha_1u_{12}+\cdots+\alpha_{n-j-1}u_{n-j-1,n-j}+\alpha_{n-j}a_{n-j+1,n}u_{n-j,n-j+1}+\alpha_{n-j+1}\beta_{n-j+1}(\mathfrak{u})u_{n-j+1,n-j+2}+\cdots+\alpha_{n-2}\beta_{n-2}(\mathfrak{u})u_{n-2,n-1})$ and $\theta(\alpha_{n-1}u_{n-1,n}-\alpha_{n-2}c_{n-1,n}u_{n-2,n}-\alpha_{n-3}c_{n-2,n}u_{n-3,n}-\cdots-\alpha_{n-j}c_{n-j+1,n}u_{n-j,n}),$ for any $2\leq j\leq n-1;$ and $\theta_{\alpha}^{(n)}(\mathfrak{u})=\theta(\alpha_1\beta_{1}(\mathfrak{u})u_{12}+\cdots+\alpha_{n-j}\beta_{n-j}(\mathfrak{u})u_{n-j,n-j+1}+\cdots+\alpha_{n-2}\beta_{n-2}(\mathfrak{u})u_{n-2,n-1})\cdot \theta(\alpha_{n-1}u_{n-1,n}-\alpha_{n-2}c_{n-1,n}u_{n-2,n}-\alpha_{n-3}c_{n-2,n}u_{n-3,n}-\cdots-\alpha_{n-j}c_{n-j+1,n}u_{n-j,n}-\cdots-\alpha_{1}c_{2,n}u_{1,n}).$ Let
   	\begin{align*}
   	W_v^{(j)}(\alpha;\lambda)=\int_{N(F_v)}\phi_{v,\lambda}\left(M_{j+1}^{\tau_n}(\mathfrak{u})\tau_n\tau_{n,j}N_{j+1}^0(\mathfrak{u}^*)k_{j+1}(\mathfrak{u})\right)\theta_{\alpha}^{(j)}(\mathfrak{u})\prod_{l=2}^{j}a_{n-l+1,n}^{2-l}d\mathfrak{u}.
   	\end{align*}
   	where $k_{j+1}(\mathfrak{u})=\tau_{n,j}^{-1}k_{j}(\mathfrak{u})$ and $k_2(\mathfrak{u})=k.$ Then $W_v(\alpha;\lambda)=W_v^{(2)}(\alpha;\lambda).$ Let $N_{j+1}^*(\mathfrak{u}^*)=(u_{k,l}'')_{1\leq k,l\leq n}$ such that $u_{k,l}''=0$ if $(k,l)=(n-j+1,n);$ and $u_{k,l}''=u_{k,l}^*$ otherwise. Let $M_4=M_4(u_{n-j+1,n}),$ $M_3=M_3(u_{n-j+1,n}).$ Then a changing of variables leads to that $W_v^{(j)}(\alpha;\lambda)$ is equal to 
   	\begin{align*}
   	&\int_{N(F_v)}\phi_{v,\lambda}\left(M_{j+1}^{\tau_n}(\mathfrak{u})\tau_n\tau_{n,j}N_{j+1}^*(\mathfrak{u}^*)M_4M_3k_{j+1}(\mathfrak{u})\right)\theta_{\alpha}^{(j)}(\mathfrak{u})\prod_{l=2}^{j}a_{n-l+1,n}^{2-l}d\mathfrak{u}=\\
   	&\int_{N(F_v)}\phi_{v,\lambda}\left(M_{j+1}^{\tau_n}(\mathfrak{u})M_4'\tau_n\tau_{n,j+1}N_{j+2}^0(\mathfrak{u}'')w_{n-j}M_3k_{j+1}(\mathfrak{u})\right)\theta_{\alpha}^{(j)}(\mathfrak{u})\prod_{l=2}^{j}a_{n-l+1,n}^{2-l}d\mathfrak{u},
   	\end{align*}
   	where $W_4'\in N(F_v).$ Since $\phi_{v,\lambda}$ is left $N(F_v)$-invariant, the right hand side of the above equality is equal to 
   	\begin{align*}
   	\int_{N(F_v)}\phi_{v,\lambda}\left(M_{j+1}^{\tau_n}(\mathfrak{u})\tau_nM_3\tau_{n,j+1}N_{j+2}^0(\mathfrak{u})w_{n-j}k_{j+1}(\mathfrak{u})\right)\theta_{\alpha}^{(j+1)}(\mathfrak{u})\prod_{l=2}^{j}a_{n-l+1,n}^{2-l}d\mathfrak{u},
   	\end{align*}
   	implying that for any $2\leq j\leq n-2,$ one has $W_v^{(j)}(\alpha;\lambda)=W_v^{(j+1)}(\alpha;\lambda).$ By our definition of $\theta_{\alpha}^{(n)},$ a similar computation to the above shows that $W_v^{(n-1)}(\alpha;\lambda)=W_v^{(n)}(\alpha;\lambda),$ namely, one has that $W_v(\alpha;\lambda)$ is equal to
   	\begin{equation}\label{111}
   	\int_{N(F_v)}\phi_{v,\lambda}\left(M_{n+1}^{\tau_n}(\mathfrak{u})\tau_n\tau_{n,n}N_{n+1}^0(\mathfrak{u}^*)k_{n+1}(\mathfrak{u})\right)\theta_{\alpha}^{(n)}(\mathfrak{u})\prod_{l=2}^{n}a_{n-l+1,n}^{2-l}d\mathfrak{u}.
   	\end{equation}
   	By definition, one has, for any $\phi_{v,\lambda}\in \pi_{v,\lambda},$ that 
   	\begin{equation}\label{1100}
   	\phi_v(t_vx_v)=\prod_{j=1}^n\chi_{v,j}(t_{v,j})|t_{v,j}|_v^{\frac{n+1}2-j+\lambda_j}\cdot\phi_v(x_v),\ x_v\in GL(n,F_v).
   	\end{equation}
   	Substituting \eqref{1100} into \eqref{111} one then sees that $W_v(\alpha;\lambda)$ is equal to 
   	\begin{equation}\label{112..}
   	\int_{N(F_v)}\phi_{v,\lambda}\left(\tau_n\tau_{n,n}N_{n+1}^0(\mathfrak{u}^*)k_{n+1}(\mathfrak{u})\right)\theta_{\alpha}^{(n)}(\mathfrak{u})\prod_{l=2}^{n}\chi_{1,l}(a_{n-l+1,n})a_{n-l+1,n}^{1+\lambda_1-\lambda_{l}}d\mathfrak{u},
   	\end{equation}
   	where $\chi_{1,l}(a_{n-l+1,n})=\chi_{v,1}(a_{n-l+1,n})\chi_{v,l}(a_{n-l+1,n})^{-1}.$ Since $\pi_v$ is right $K_v$-finite, one then sees, according to \eqref{112..}, that $W_v(\alpha;\lambda)$ is equal to
   	\begin{equation}\label{113'}
   	\sum_{j\in\boldsymbol{J}}\int_{N(F_v)}\phi_{v,\lambda}^{(j)}\left(\tau_n\tau_{n,n}N_{n+1}^0(\mathfrak{u}^*)\right)\theta_{\alpha}^{(n)}(\mathfrak{u})\prod_{l=2}^{n}\chi_{1,l}(a_{n-l+1,n})a_{n-l+1,n}^{1+\lambda_1-\lambda_{l}}d\mathfrak{u},
   	\end{equation}
   	where $J$ is a finite set of indexes, whose cardinality depends only on the $K_v$-finite type of $\pi_v;$ and each $\phi_{v,\lambda}^{(j)}\in \pi_{v,\lambda}.$ Let $W_{v,j}(\alpha;\lambda)$ be the summand of \eqref{113'} corresponding to the index $j\in \boldsymbol{J}.$ Let $\widetilde{\lambda}_j=\lambda_{j}+\lambda_1/(n-1),$ $2\leq j\leq n.$ Denote by $B_{n-1}$ the standard Borel subgroup of $\GL_{n-1}$ and $N_{n-1}$ the unipotent of $B_{n-1}.$ Then a change of variables implies that $W_{v,j}(\alpha;\lambda)$ is equal to
   	\begin{align*}
   	\int_{F_v^{n-1}}W_{v,j}(\widetilde{\alpha}_{n-1};\widetilde{\lambda})\widetilde{\theta}_n(u_1,\cdots,u_{n-1})\prod_{l=2}^{n}\chi_{1,l}(a_{n-l+1})|a_{n-l+1}|_v^{1+\lambda_1-\lambda_{l}}\prod_{j=1}^{n-1}du_{j}.
   	\end{align*}
   	Then Lemma \ref{43lem} follows.
   \end{proof}

   Let $v\in \Sigma_F$ be a fixed nonarchimedean place, let $\widetilde{\pi}_{\lambda,v}$ be the contragredient of $\pi_{\lambda,v}.$ Let $\varpi_v$ be a uniformizer of $\mathcal{O}_{F,v},$ the ring of integers of $F_v.$ Let $q_v=N_{F_v/\mathbb{Q}_p}(\varpi_v),$ where $p$ is the rational prime such that $v$ is above $p.$ Denote by
   \begin{align*}
   R_v(s,W_{1,v},W_{2,v};\lambda):=\frac{\Psi_v\left(s,W_{1,v},W_{2,v};\lambda,\Phi_v\right)}{L_v(s,\pi_{\lambda,v}\otimes\tau_v\times\widetilde{\pi}_{-\lambda,v})},\quad \Re(s)>1.
   \end{align*}
   \begin{prop}[Nonarchimedean Case]\label{43'}
   	Let notation be as before. Let $s\in \mathbb{C}$ be such that $\Re(s)>1$.Then we have
   	\begin{description}
   		\item[(a)] $R_v(s,W_{1,v},W_{2,v};\lambda)$ is a polynomial in $\{q_v^s, q_v^{-s}, q_v^{\lambda_i}, q_v^{-\lambda_i}:\ 1\leq i\leq r\}.$ 
   		\item[(b)] We have the local functional equation
   		\begin{align*}
   		\frac{\Psi_v\left(s,W_{1,v},W_{2,v};\lambda,\Phi_v\right)}{ L_v(s,\pi_{\lambda,v}\otimes\tau_v\times\widetilde{\pi}_{-\lambda,v})}=\varepsilon(s,\pi_{\lambda,v}\times\widetilde{\pi}_{-\lambda,v},\theta)\cdot\frac{\Psi_v(1-s,\widetilde{W}_{1,v},\widetilde{W}_{2,v};-\bar{\lambda},\widehat{\Phi}_v)}{L_v(1-s,\widetilde{\pi}_{-\bar{\lambda},v}\otimes\bar{\tau}_v\times\pi_{\lambda,v})},
   		\end{align*}
   		where $\varepsilon(s,\pi_{\lambda,v}\times\widetilde{\pi}_{-\lambda,v},\theta)$ is a polynomial in $\{q_v^s, q_v^{-s}, q_v^{\lambda_i}, q_v^{-\lambda_i}:\ 1\leq i\leq r\}.$ 
   	\end{description}
   \end{prop}
   \begin{proof}
   	We shall only prove Part (a), since Part (b) will follow form \cite{JS81}.
   	
   	Let $T(F_v)$ be the maximal torus of $G(F_v),$ and for any $m\in\mathbb{Z},$ let $T^{(m)}(F_v)=\{t\in T(F_v):\ |\det t|_{F_v}=q_v^{-m}\}.$ Using Iwasawa decomposition and the fact that $W_{i,v}$ and $\Phi_v$ are right $G(\mathcal{O}_{F,v})$-finite, we can rewrite $\Psi_v\left(s,W_{1,v},W_{2,v};\lambda,\Phi_v\right)$ as
   	\begin{align*}
   	\sum_{j\in J}\int_{T(F_v)}W^{(j)}_{1,v}(a_v;\lambda)\overline{W^{(j)}_{2,v}\left(a_v;-\bar{\lambda}\right)} \Phi_{j,v}(\eta a_v)\tau_v(\det a_v)\delta_{T}^{-1}(a_v)|\det a_v|_{F_v}^sda_v,
   	\end{align*}
   	where the sum over a finite set $J,$ $W^{(j)}_{i,v}(a_v;\lambda)$ is a Whittaker function associated to some smooth functions in $\mathcal{H}_{P,\chi},$ $1\leq i\leq 2,$ and $\Phi_{j,v}$ is some Schwartz-Bruhat function. Note that for $1\leq i\leq 2$ and $j\in J,$ $W_{i,v}^{(j)}(x_v;\lambda)$ is right $G(\mathcal{O}_{F,v})$-finite. So there exists a compact subgroup $N_{0,v}\subseteq G(\mathcal{O}_{F,v})\cap N(F_v),$ depending only on $\varphi,$ such that 
   	$W^{(j)}_{i,v}(t_vu_v;\lambda)=W^{(j)}_{i,v}(t_v;\lambda),$ for all $t_v\in T(F_v)$ and $u_v\in N_{0,v}.$
   	On the other hand, $W^{(j)}_{i,v}(t_vu_v;\lambda)=\theta_{t_v}(u_v)W^{(j)}_{i,v}(t_v;\lambda),$ where $\theta_{t_v}(n_v)=\theta(t_vn_vt_v^{-1}),$ for any $n_v\in N(F_v).$ But then, there exists a constant $C_v$ depending only on $N_{0,v}$ and $\theta$ (hence not on $\lambda$) such that $\theta_{t_v}(u_v)=1$ if and only if $|\alpha_i(t_v)|\leq C_v,$ where $\alpha_i$'s are the simple roots of $G(F).$ Thus each $W^{(j)}_{i,v}(x_v;\lambda)$ is compactly supported for a fixed $\lambda\in i\mathfrak{a}_P/i\mathfrak{a}_G.$ Therefore, for a fixed $\lambda,$ $\Psi_v\left(s,W_{1,v},W_{2,v};\lambda,\Phi_v\right)$ is a formal Laurent series in $q_v^{-s}.$ Indeed, one can chose some nonnegative integer $M$ independent of $\lambda$ (but depending on $\pi$ and $\varphi$), such that 
   	\begin{align*}
   	\Psi_v\left(s,W_{1,v},W_{2,v};\lambda,\Phi_v\right)=\sum_{m\geq -M}\Psi_v^{(m)}\left(W_{1,v},W_{2,v};\lambda,\Phi_v\right)\cdot q_v^{-ms},
   	\end{align*}
   	where $\Psi_v^{(m)}\left(s,W_{1,v},W_{2,v};\lambda,\Phi_v\right)$ is defined by the integral
   	\begin{align*}
   	\sum_{j\in J}\int_{T^{(m)}(F_v)}W^{(j)}_{1,v}(a_v;\lambda)\overline{W^{(j)}_{2,v}\left(a_v;-\bar{\lambda}\right)} \Phi_{j,v}(\eta a_v)\tau_v(\det a_v)\delta_{T}^{-1}(a_v)da_v.
   	\end{align*}
   	Apply the above analysis on $\supp W_{i,v}(a_v;\lambda),$ we see similarly that 
   	$$
   	\supp W^{(j)}_{i,v}(a_v;\lambda)\subseteq \{t\in T^{(m)}(F_v):\ |\alpha_l(t)|_{F_v}\leq C_v^{(j)},\ 1\leq l\leq n-1\}
   	$$ 
   	for some constants $C_v^{(j)}.$ Hence, for each $j\in J,$ $m\geq -N$ and $a_v\in T^{(m)}(F_v),$ the function $a_v\mapsto W^{(j)}_{1,v}(a_v;\lambda)\overline{W^{(j)}_{2,v}\left(a_v;-\bar{\lambda}\right)}$ is analytic and is a formal Laurent series in $\{q_v^{-\lambda_i}:\ 1\leq i\leq r\}$ by (2.5.2) of \cite{JS81}, and the function
   	$$
   	a_v\mapsto W^{(j)}_{1,v}(a_v;\lambda)\overline{W^{(j)}_{2,v}\left(a_v;-\bar{\lambda}\right)} \Phi_{j,v}(\eta a_v)\tau(\det a_v)\delta_{T}^{-1}(a_v)
   	$$ 
   	is locally constant. Therefore, $\Psi_v^{(m)}\left(W_{1,v},W_{2,v};\lambda,\Phi_v\right)$ is an analytic function of $\lambda$ and is a formal Laurent series in $\{q_v^{-\lambda_i}:\ 1\leq i\leq r\}$. 
   	
   	Since $\pi_{\lambda,v}$ is of Whittaker type, we can use Theorem 2.7 of \cite{JS81} to see that, for fixed $\lambda\in i\mathfrak{a}_P/i\mathfrak{a}_G,$ $\Psi_v\left(s,W_{1,v},W_{2,v};\lambda,\Phi_v\right)\cdot L_v(s,\pi_{\lambda,v}\otimes\tau_v\times\widetilde{\pi}_{-\lambda,v})^{-1}$ is a polynomial in $\{q_v^{s}, q_v^{-s}\}$ with coefficients functions of $\lambda.$ Moreover, $L_v(s,\pi_{\lambda,v}\otimes\tau_v\times\widetilde{\pi}_{-\lambda,v})^{-1}$ is a polynomial in $\{q_v^s, q_v^{-s}, q_v^{\lambda_i}, q_v^{-\lambda_i}:\ 1\leq i\leq r\}.$ So we can write
   	\begin{align*}
   	L_v(s,\pi_{\lambda,v}\otimes\tau_v\times\widetilde{\pi}_{\lambda,v})^{-1}=\sum_{|l|\leq N}Q_l(\lambda)q_v^{-ls},
   	\end{align*}
   	where $N$ is a positive integer and $Q_l(\lambda)$ are polynomials in $\{q_v^{\lambda_i}, q_v^{-\lambda_i}:\ 1\leq i\leq r\}.$ Then for  $\lambda\in i\mathfrak{a}_P/i\mathfrak{a}_G,$ $\Psi_v\left(s,W_{1,v},W_{2,v};\lambda,\Phi_v\right)\cdot L(s,\pi_{\lambda,v}\otimes\tau_v\times\widetilde{\pi}_{-\lambda,v})^{-1}$ is equal to the sum over $m\geq -N-M$ of $R_m\left(s,W_{1,v},W_{2,v};\lambda,\Phi_v\right)q_v^{-ms},$ where
   	\begin{align*}
   	R_m\left(s,W_{1,v},W_{2,v};\lambda,\Phi_v\right)=\sum_{\substack{i+j=m\\ |i|\leq N, j\geq-M}}Q_i(\lambda)\Psi_v^{(m)}\left(W_{1,v},W_{2,v};\lambda,\Phi_v\right).
   	\end{align*}
   	Since the sum on the right hand side is finite, $R_l\left(s,W_{1,v},W_{2,v};\lambda,\Phi_v\right)$ is analytic in $\lambda.$ Moreover, it is a formal Laurent series in $\{q_v^{\lambda_i}, q_v^{-\lambda_i}:\ 1\leq i\leq r\}.$ Therefore, part (a) of Proposition \ref{43'} follows from Claim \ref{44''} below.
   \end{proof}
   \begin{claim}\label{44''}
   	There exists some $M_0\in\mathbb{Z},$ independent of $\lambda\in i\mathfrak{a}_P/i\mathfrak{a}_G,$ such that $R_m\left(s,W_{1,v},W_{2,v};\lambda,\Phi_v\right)=0$ for all $m\geq M_0$ and for all $\lambda\in i\mathfrak{a}_P/i\mathfrak{a}_G.$ Moreover, for each $m\in \mathbb{Z},$ $R_m\left(s,W_{1,v},W_{2,v};\lambda,\Phi_v\right)$ is a polynomial in $\{q_v^{\pm s}, q_v^{\pm\lambda_i}:\ 1\leq i\leq r\}.$
   \end{claim}
   \begin{proof}[Proof of Claim \ref{44''}]
   	Let $l\in\mathbb{Z}.$ One then defines  
   	\begin{align*}
   	\Lambda_{l}=\big\{\lambda\in  i\mathfrak{a}_P/i\mathfrak{a}_G:\ R_m\left(s,W_{1,v},W_{2,v};\lambda,\Phi_v\right)=0\ \text{for all $m\geq l$}\big\}.
   	\end{align*}
   	Then each $\Lambda_{l}$ is closed since $R_m\left(s,W_{1,v},W_{2,v};\lambda,\Phi_v\right)=0$ are analytic (hence continuous) in $\lambda.$ Since $R_v(s,\lambda)=\sum_{m}R_m\left(s,W_{1,v},W_{2,v};\lambda,\Phi_v\right)q_v^{-ms}\in\mathbb{C}[q_v^s, q_v^{-s}],$ for fixed $\lambda\in i\mathfrak{a}_P/i\mathfrak{a}_G,$ there exists some $M(\lambda)$ such that $R_m\left(s,W_{1,v},W_{2,v};\lambda,\Phi_v\right)=0$ as long as $m\geq M(\lambda).$ Therefore, $i\mathfrak{a}_P/i\mathfrak{a}_G$ is covered by the union of all $\Lambda_{l}.$ Noting that $i\mathfrak{a}_P/i\mathfrak{a}_G\simeq R^{r-1}$ is a Banach space, by Baire category theorem there exists some $\Lambda_{l_0}$ having nonempty interior, $\Int(\Lambda_{l_0}),$ say. Then for any $\lambda\in \Int(\Lambda_{l_0}),$ $R_m\left(s,W_{1,v},W_{2,v};\lambda,\Phi_v\right)=0$ for any $m\geq l_0.$ Since $R_m\left(s,W_{1,v},W_{2,v};\lambda,\Phi_v\right)$ is analytic for any $l\in\mathbb{Z},$ $R_m\left(s,W_{1,v},W_{2,v};\lambda,\Phi_v\right)=0$ for all $\lambda\in i\mathfrak{a}_P/i\mathfrak{a}_G,$ proving the first part. For the remaining part, we consider the functional equation (see \cite{JS81}):
   	\begin{align*}
   	\frac{\Psi_v\left(s,W_{1,v},W_{2,v};\lambda,\Phi_v\right)}{ L_v(s,\pi_{\lambda,v}\otimes\tau_v\times\widetilde{\pi}_{-\lambda,v})}=\varepsilon(s,\pi_{\lambda,v}\times\widetilde{\pi}_{-\lambda,v},\theta)\cdot\frac{\Psi_v(1-s,\widetilde{W}_{1,v},\widetilde{W}_{2,v};-\bar{\lambda},\widehat{\Phi}_v)}{L_v(1-s,\widetilde{\pi}_{-\bar{\lambda},v}\otimes\bar{\tau}_v\times\pi_{\lambda,v})},
   	\end{align*}
   	where $\varepsilon(s,\pi_{\lambda,v}\times\widetilde{\pi}_{-\lambda,v},\theta)$ is a polynomial in $\{q_v^s, q_v^{-s}, q_v^{\lambda_i}, q_v^{-\lambda_i}:\ 1\leq i\leq r\}.$ 
   	
   	We can interpret the functional as an identity between formal Laurent series in $\{q_v^{\lambda_i}, q_v^{-\lambda_i}:\ 1\leq i\leq r\}.$ The left hand side are formal Laurent series of the form $\sum_{m_1\geq -M_1}q_v^{m_1\lambda_i},$ while the right hand side are formal Laurent series of the form $\sum_{m_2\geq -M_2}q_v^{-m_2\lambda_i}.$ Since they are equal, they must be both polynomials in $\{q_v^s, q_v^{-s}, q_v^{\lambda_i}, q_v^{-\lambda_i}:\ 1\leq i\leq r\}.$ Then the proof of Claim \ref{44''} follows.
   \end{proof}
   
   One will see that Proposition \ref{43'} is insufficient for our continuation in next few sections. Hence we need to compute $R_v(s,W_{1,v},W_{2,v};\lambda)$ more explicitly. We will do principal series case below since this is the only case we need for the particular purpose of this paper. 
   
   \begin{lemma}\label{46lem}
   	Let $v$ be a nonarchimedean place of $F.$ Let $\pi_v$ be a principal series characters $\chi_{v,1},\chi_{v,2},\cdots,\chi_{v,n}$. Assume that $\pi_v$ is right $K_v$-finite. Let $\alpha\in \mathbb{G}_m(F_v)^{n-1}$ and let $W_v(\alpha,\lambda)$ be a Whittaker function associated to $\pi_{v,\lambda}$ and $\alpha.$ Then $W_v(\alpha,\lambda)$ is of the form $\mathcal{B}_v(\alpha,\lambda)\mathcal{L}_v(\lambda),$ where $\mathcal{B}_v(\alpha,\lambda)$ is a holomorphic function, and
   	\begin{align*}
   	\mathcal{L}_v(\lambda)=\prod_{1\leq i<r}\prod_{i< j\leq r}L_v(1+\lambda_i-\lambda_j,\chi_{v,i}\overline{\chi}_{v,j})^{-1}.
   	\end{align*}
   \end{lemma}
   \begin{proof}
   	It follows from Lemma \ref{43lem} and induction that Lemma \ref{46lem} holds for any $n$ if it holds for $n=2$ case. Now we show that Lemma \ref{46lem} holds for $n=2.$ 
   	
   	We may assume that $\chi_{1,2}=\chi_{v,1}\chi_{v,2}^{-1}$ is unramified. Otherwise, the local $L$-function $L(s,\chi_1\overline{\chi}_2)$ is trivial, and Lemma \ref{46lem} follows from Part (a) of Proposition \ref{43'}.  According to \eqref{106f} and the $K_v$-finiteness condition, one has 
   	\begin{align*}
   	W_v(\alpha,\lambda)&=\sum_{j\in\boldsymbol{J}}\sum_{l=1}^{\infty}c_j\int_{\varpi_v^{-l}\mathcal{O}_{v}^{\times}}\chi_{12}(a(u))|a(u)|_v^{1+\lambda_1-\lambda_2}\theta(\alpha u)du+W_v^{\circ}(\alpha,\lambda),\\
   	W_v^{\circ}(\alpha,\lambda)&=\sum_{j\in\boldsymbol{J}}c_j\int_{\mathcal{O}_{v}}\chi_{12}(a(u))|a(u)|_v^{1+\lambda_1-\lambda_2}\theta(\alpha u)du=\sum_{j\in\boldsymbol{J}}c_j\int_{\mathcal{O}_{v}}\theta(\alpha u)du,
   	\end{align*}
   	where $j$ runs over a finite set $\boldsymbol{J}$ and $c_j$'s are constants; moreover, $\boldsymbol{J}$ and $c_j$'s relay only on the $K_v$-type of $\pi_v.$ For $u\in F_v^{\times},$ write $u=u^{\circ}\varpi_v^l,$ where $u^{\circ}\in\mathcal{O}_v^{\times}=\mathcal{O}_{F_v}^{\times},$ and $l\in\mathbb{Z}.$ Write $\alpha=\alpha^{\circ}\varpi_{v}^k,$ where $\alpha^{\circ}\in\mathcal{O}_v^{\times}.$ Recall that by definition the one sees that the conductor of $\theta$ is precisely the inverse different of $F_v,$ which is $\mathfrak{D}_{F_v}^{-1}=\{x_v\in F_v:\ \tr_{F_v/\mathbb{Q}_p}(x_v)\in \mathbb{Z}_p\},$ where $p$ is the characteristic of residue field of $\mathcal{O}_v.$ Note that $\mathfrak{D}_{F_v}^{-1}$ is a $\mathbb{Z}_p$-module of $F_v$ and thus has the representation $\mathfrak{D}_{F_v}^{-1}=\varpi_v^{-d}\mathcal{O}_v,$ where $d\in\mathbb{N}_{\geq0}.$ Hence one sees that 
   	\begin{align*}
   	I=\int_{\mathcal{O}_{v}}\theta(\alpha u)du=\int_{\mathcal{O}_{v}}\theta(\alpha^{\circ} u\varphi_v^k)du=\int_{\mathcal{O}_{v}}\theta(u\varpi_v^k)du
   	\end{align*}
   	is vanishing if $k\leq -d-1.$ Clearly $I=1$ if $k\geq -d.$ Note that 
   	\begin{align*}
   	\int_{\varpi_v^{-l}\mathcal{O}_{v}^{\times}}\chi_{12}(a(u))|a(u)|_v^{1+\lambda_1-\lambda_2}\theta(\alpha u)du=\chi_{12}(\varpi_v)^l|\varpi_v|_v^{(1+\lambda_1-\lambda_2)l}\int_{\varpi_v^{-l}\mathcal{O}_{v}^{\times}}\theta(\alpha u)du
   	\end{align*}
   	is vanishing if $l\geq k+d+2.$ Let $q_v=|\varpi_v|_v^{-1}.$ Then one sees that 
   	\begin{equation}\label{b}
   	W_v(\alpha,\lambda)=C+C\sum_{l=1}^{k+d}(1-q_v^{-1})\chi_{12}(\varpi_v)^lq^{-(\lambda_1-\lambda_2)l}_v+C\cdot W_{re},
   	\end{equation}
   	where $C$ is a constant depending only on $F$ and $K_v$-type of $\phi_v$ and 
   	\begin{equation}\label{c}
   	W_{re}=\chi_{12}(\varpi_v)^{k+d+1}q^{-(k+d+1)(1+\lambda_1-\lambda_2)}_v\int_{\varpi_v^{-k-d-1}\mathcal{O}_v^{\times}}\theta(u\varpi_v^{k})du.
   	\end{equation}
   	Since $\theta$ is nontrivial on $\varpi_v^{-d-1}\mathcal{O}_v,$ then $\int_{\varpi_v^{-k-d-1}\mathcal{O}_v}\theta(u\varpi_v^{k})du=0.$ Note that  $\varpi_v^{-k-d-1}\mathcal{O}_v^{\times}=\varpi_v^{-k-d-1}\mathcal{O}_v\setminus \varpi_v^{-k-d}\mathcal{O}_v.$ Then one has that 
   	\begin{align*}
   	\int_{\varpi_v^{-k-d-1}\mathcal{O}_v^{\times}}\theta(u\varpi_v^{k})du&=\int_{\varpi_v^{-k-d-1}\mathcal{O}_v}\theta(u\varpi_v^{k})du-\int_{\varpi_v^{-k-d}\mathcal{O}_v}\theta(u\varpi_v^{k})du\\
   	&=-\int_{\varpi_v^{-k-d}\mathcal{O}_v}\theta(u\varpi_v^{k})du=-\vol(\varpi_v^{-k-d}\mathcal{O}_v)=-q_v^{k+d}.
   	\end{align*}
   	Then it follows form \eqref{b} and \eqref{c} that $W_v(\alpha,\lambda)$ is equal to $C$ multiplying 
   	\begin{align*}
   	L=1+\sum_{l=1}^{k+d}(1-q_v^{-1})\chi_{12}(\varpi_v)^lq^{-(\lambda_1-\lambda_2)l}_v-\chi_{12}(\varpi_v)^{k+d+1}q^{-(k+d+1)(\lambda_1-\lambda_2)-1}_v.
   	\end{align*}
   	An elementary computation leads to the identity 
   	\begin{equation}\label{d}
   	L=(1-\chi_{12}(\varpi_v)q_v^{-(1+\lambda_1-\lambda_2)})\cdot P(\chi_{12}(\varpi_v)q_v^{-(\lambda_1-\lambda_2)}),
   	\end{equation}
   	where $P(z)=(1-z^{k+d+1})\cdot(1-z)^{-1}=1+z+\cdots+z^{k+d}\in\mathbb{C}[z].$
   	
   	Therefore, one has that $W_v(\alpha,\lambda)=CQ(\chi_{12}(\varpi_v)q_v^{-(\lambda_1-\lambda_2)})\cdot L_v(1+\lambda_1-\lambda_2,\chi_{12}),$ where $Q(z)=P(z)$ if $k\geq -d;$ $Q(z)\equiv0,$ otherwise. Taking $\mathcal{B}_v(\alpha,\lambda)$ to be the function $CQ(\chi_{12}(\varpi_v)q_v^{-(\lambda_1-\lambda_2)})$ we then obtain Lemma \ref{46lem} in $n=2$ case. The general case follows from this and induction, since integral with respect to $\chi_{l,j}$ is exactly the same as above, $1\leq l<j\leq n.$
   \end{proof}

   \begin{prop}[Principal Series Case: nonarchimedean]\label{prop47}
   	Let $v$ be a nonarchimedean place of $F.$ Let $\pi_v$ be a principal series characters $\chi_{1,v},\chi_{2,v},\cdots,\chi_{n,v}$. Assume that $\pi_v$ is right $K_v$-finite. Then the function $R_v(s,W_{1,v},W_{2,v};\lambda)$ is of the form  $Q_v(s,\lambda)\mathcal{L}_v(\lambda),$ where the function $Q_v(s,\lambda)\in\mathbb{C}[q_v^{\pm s},q_v^{\pm\lambda_i}:\ 1\leq i\leq n];$ and $\mathcal{L}_v(\lambda)$ is defined to be
   	\begin{align*}
   	\prod_{1\leq i<r}\prod_{i< j\leq r}L_v(1+\lambda_i-\lambda_j,\chi_{i,v}\overline{\chi}_{j,v})^{-1}\cdot L_v(1-\lambda_i+\lambda_j,\overline{\chi}_{i,v}\chi_{j,v})^{-1}.
   	\end{align*}
   \end{prop}
   \begin{proof}
   	By Lemma \ref{46lem} the function 
   	\begin{align*}
   	W_v(x_v;\phi_{1,v},\lambda)\prod_{1\leq i<r}\prod_{i< j\leq r}L_v(1+\lambda_i-\lambda_j,\chi_{i,v}\overline{\chi}_{j,v})\in \mathbb{C}[q_v^{\pm\lambda_j}:\ 1\leq j\leq n].
   	\end{align*}
   	Then applying expansions in \cite{JS81} and changing orders of summations we see that 
   	\begin{align*}
   	R_v(s,W_{1,v},W_{2,v};\lambda)	\prod_{1\leq i<r}\prod_{i< j\leq r}L_v(1+\lambda_i-\lambda_j,\chi_{i,v}\overline{\chi}_{j,v})\cdot L_v(1-\lambda_i+\lambda_j,\overline{\chi}_{i,v}\chi_{j,v})
   	\end{align*}
   	lies in $\mathbb{C}[q_v^{\pm s},q_v^{\pm\lambda_i}:\ 1\leq i\leq n].$ Done.
   \end{proof} 
   
   \begin{cor}\label{102}
   	Let $v\in \Sigma_{F,fin}$ be a finite place such that $\pi_v$ is unramified and $\Phi_v=\Phi_v^{\circ}$ is the characteristic function of $G(\mathcal{O}_{F,v}).$ Assume that $\phi_{1,v}=\phi_{2,v}=\phi_v^{\circ}$ be the unique $G(\mathcal{O}_{F,v})$-fixed vector in the space of $\pi_v$ such that $\phi_v^0(e)=1.$ Then $R_v(s,W_{1,v},W_{2,v};\lambda)$ is equal to
   	\begin{align*}
   	\prod_{1\leq i<r}\prod_{i< j\leq r}L_v(1+\lambda_i-\lambda_j,\pi_{i,v}\times\widetilde{\pi}_{j,v})^{-1}\cdot L_v(1-\lambda_i+\lambda_j,\widetilde{\pi}_{i,v}\times\pi_{j,v})^{-1}.
   	\end{align*}
   	In particular, $R_v(s,\lambda)$ is independent of $s.$
   \end{cor}
   \begin{proof}
   	Fix $\lambda\in i\mathfrak{a}_P/i\mathfrak{a}_G.$ Let $W_{i,v}^{\circ}$ be the $G(\mathcal{O}_{F,v})$-invariant vectors such that $W_{i,v}^{\circ}(e)=1,$ $1\leq i\leq 2.$ Then by the computation from \cite{JS81}, we know that $\Psi_v\left(s,W^{\circ}_{1,v},W^{\circ}_{2,v};\lambda,\Phi_v^{\circ}\right)/ L_v(s,\pi_{\lambda,v}\otimes\tau_v\times\widetilde{\pi}_{-\lambda,v})=1,$ where $\Phi_v^{\circ}$ is the characteristic function of $\mathcal{O}_{F,v}^n.$ Then Corollary \eqref{102} follows from induction and unramified computations of nonconstant Fourier coefficients of Eisenstein series (see Chap. 7 of \cite{Sha10}).
   \end{proof}

   Now we move to the archimedean case. In the current state of affairs the local $L$-functions $L_{\infty}(s,\pi_{\lambda}\times\tau\times\widetilde{\pi}_{-\lambda})=\prod_{v\mid\infty}L_{v}(s,\pi_{\lambda,v}\times\tau_v\times\widetilde{\pi}_{-\lambda,v})$ are not defined intrinsically through the integrals as nonarchimedean case, but rather extrinsically through the Langlands correspondence and then related to the integrals. Let $\Gamma_{\mathbb{R}}(s)=\pi^{-s/2}\Gamma(s/2)$ and $\Gamma_{\mathbb{C}}(s)=2^{1-s}\pi^{-s}\Gamma(s).$ Then by Langlands classification (e.g. see \cite{Kna94}), each archimedan $L$-function $L_{v}(s,\pi_{\lambda,v}\times\tau_v\times\pi_{-\lambda,v})$ is of the form 
   \begin{equation}\label{09}
   \prod_{i\in I}\Gamma_{\mathbb{R}}(s+\mu_{i})\prod_{j\in J}\Gamma_{\mathbb{C}}(s+\mu_{j}'),
   \end{equation}
   where $I$ and $J$ are finite set of inters satisfying $\#I+\#J\leq n;$ $\mu_i, \mu_j'\in\mathbb{C}.$
   
   Combining results from  \cite{Jac09} and well known estimates on archimedean Satake parameters one concludes the following result.    
   \begin{prop}[Archimedean Case]\label{45'}
   	Let notation be as before. Let $v\in \Sigma_{F,\infty}$ be an archimedean place. Then we have
   	\begin{description}
   		\item[(a)] $\Psi_v\left(s,W_{1,v},W_{2,v};\lambda,\Phi_v\right)$ converges absolutely and normally in the right half plane $\{s\in\mathbb{C}:\ \Re(s)>1-2/(n^2+1)\},$ uniformly in $\lambda\in  i\mathfrak{a}_P/i\mathfrak{a}_G.$
   		\item[(b)] When $\Re(s)>1-2/(n^2+1),$ for any fixed $\lambda\in i\mathfrak{a}_P^*/i\mathfrak{a}_G^*,$ the function $R_v(s,W_{1,v},W_{2,v};\lambda)$ is equal to a polynomial $Q_v(s,\lambda)\in\mathbb{C}[s],$ with coefficients depending on $\lambda$ continuously. Hence, for any fixed $\lambda\in i\mathfrak{a}_P^*/i\mathfrak{a}_G^*,$ $\Psi_v\left(s,W_{1,v},W_{2,v};\lambda,\Phi_v\right)=R_v(s,W_{1,v},W_{2,v};\lambda)L_{v}(s,\pi_{\lambda,v}\times\tau_v\times\pi_{-\lambda,v})$ admits a meromorphic continuation to the whole complex plane. 
   		\item[(d)] We have the local functional equation
   		\begin{align*}
   		\frac{\Psi_v\left(s,W_{1,v},W_{2,v};\lambda,\Phi_v\right)}{ L_v(s,\pi_{\lambda,v}\otimes\tau_v\times\widetilde{\pi}_{-\lambda,v})}=\varepsilon(s,\pi_{\lambda,v}\times\widetilde{\pi}_{-\lambda,v},\theta)\cdot\frac{\Psi_v(1-s,\widetilde{W}_{1,v},\widetilde{W}_{2,v};-\bar{\lambda},\widehat{\Phi}_v)}{L_v(1-s,\widetilde{\pi}_{-\bar{\lambda},v}\otimes\bar{\tau}_v\times\pi_{\lambda,v})},
   		\end{align*}
   		where $\varepsilon(s,\pi_{\lambda,v}\times\widetilde{\pi}_{-\lambda,v},\theta)$ is a holomoprhic function.
   	\end{description}
   \end{prop}
   \begin{remark}
   	It follows from Lemma 5.4 in \cite{Jac09} that if both $\pi$ is tempered, then the Rankin-Selberg convolution $\Psi_v\left(s,W_{1,v},W_{2,v};\lambda,\Phi_v\right)$ converges absolutely and normally in the right half plane $\{s\in\mathbb{C}:\ \Re(s)>0\},$ uniformly in $\lambda\in  i\mathfrak{a}_P/i\mathfrak{a}_G.$ 
   \end{remark}
   
   We need a more explicit description of the polynomial $Q_v(s,\lambda)$ in Proposition \ref{45'} when $\pi_v$ is a principal series. To start with, we recall the definition of archimedean $L$-function associated to a unitary Hecke character. If $F_v\simeq\mathbb{R},$ then the only possible choices for a unitary Gr\"{o}ssencharacter are $x_v\mapsto \sgn(x_v)^k|x_v|_v^{i\nu}$ for $k\in\{0,1\}$ and $\nu\in\mathbb{R}.$ If $F_v\simeq\mathbb{C},$ then the only possible choices for a unitary Gr\"{o}ssencharacter are $x_v\mapsto (x_v\cdot|x_v|_v^{-1/2})^k|x_v|_v^{i\nu}$ for $k\in\mathbb{Z}$ and $\nu\in\mathbb{R}.$ Furthermore, since the units are killed by such a character, then the sum of those $\nu$'s must be 0. The Gamma factors at the real infinite places are $\Gamma((s+i\nu+k)/2)$ and at the complex places are $\Gamma(s+i\nu+|k|/2).$ To prove Proposition \ref{prop}, we need some preparation. 
   
   \begin{lemma}\label{48lem}
   	Let $v$ be an archimedean place of $F.$ Let $\pi_v$ be a principal series characters $\chi_{v,1},\chi_{v,2},\cdots,\chi_{v,n}$. Assume that $\pi_v$ is right $K_v$-finite. Let $\alpha\in \mathbb{G}_m(F_v)^{n-1}$ and let $W_v(\alpha,\lambda)$ be a Whittaker function associated to $\pi_{v,\lambda}$ and $\alpha.$ Then $W_v(\alpha,\lambda)$ is of the form $\mathcal{B}_v(\alpha,\lambda)\mathcal{L}_v(\lambda),$ where $\mathcal{B}_v(\alpha,\lambda)$ is a holomorphic function, and
   	\begin{align*}
   	\mathcal{L}_v(\lambda)=\prod_{1\leq i<r}\prod_{i< j\leq r}L_v(1+\lambda_i-\lambda_j,\chi_{v,i}^{ur}\overline{\chi}_{v,j}^{ur})^{-1},
   	\end{align*}
   	where for any $\chi_{v,l},$ $1\leq l\leq n,$ $\chi_{v,l}^{ur}=\chi_{v,l}\circ|\cdot|_v^{1/[F_v:\mathbb{R}]}$ is the unramified part of $\chi_{v,l}.$ 
   \end{lemma}
   \begin{proof}
   	It follows from \eqref{43*} and induction that Lemma \ref{48lem} holds for any $n$ if it holds for $n=2$ case. Now we show that Lemma \ref{48lem} holds for $n=2.$
   	
   	By \eqref{106e} it suffices to show that for any $\alpha, z\in\mathbb{C},$ one has
   	\begin{equation}\label{111'}
   	\int_{F_v}|a(u)|_v^z\theta_v(\alpha u)du\sim {\Gamma_{F_v}(z+1)}^{-1}.
   	\end{equation}
   	Since the proof is similar, we only consider real places. Let $F_v\simeq\mathbb{R}.$ Then 
   	\begin{align*}
   	\Gamma_{\mathbb{R}}(z+1)\int_{F_v}|a(u)|_v^z\theta_v(\alpha u)du&=\int_{0}^{\infty}\int_{\mathbb{R}}e^{-t}t^{\frac{z+1}2}\cdot\frac{e^{2\pi i\alpha u}}{(1+u^2)^{\frac{z+1}2}}du\frac{dt}{t}\\
   	&=\int_{0}^{\infty}\int_{\mathbb{R}}e^{-t(1+u^2)+2\pi i\alpha u}t^{\frac{z+1}2}du\frac{dt}{t}\\
   	&=\pi^{\frac{z+1}2}|\alpha|^{z/2}\int_{0}^{\infty}e^{-\pi|\alpha|(t+t^{-1})}t^{z/2}\frac{dt}{t}.
   	\end{align*}
   	Since the function $g(t)=e^{-\pi|\alpha|(t+t^{-1})}$ is Schwartz, then its Mellin transform $\int_{0}^{\infty}g(t)t^{z/2}d^{\times}t$ is entire. Hence one has a continuation of Whittaker functions and proves \eqref{111'}.
   \end{proof}
   \begin{remark}
   	Note that in the proof of Lemma \ref{48lem}, $\int_{0}^{\infty}g(t)t^{z/2}d^{\times}t\neq0$ for any $\alpha, z\in \mathbb{C}.$ Hence $\int_{F_v}|a(u)|_v^z\theta_v(\alpha u)du$ never vanishes. Then by induction one concludes that $W_v(\alpha,\lambda)/\mathcal{L}_v(\lambda)\neq0,$ for any $\alpha$ and $\lambda.$
   \end{remark}
   
   Combining Lemma \ref{48lem} and  \cite{JS81} one then concludes the following result. 
   \begin{prop}[Principal Series Case: archimedean]\label{prop}
   	Let $v$ be an archimedean place of $F.$ Let $\pi_v$ be a principal series characters $\chi_{v,1},\chi_{v,2},\cdots,\chi_{v,n}$. Then the function $R_v(s,W_{1,v},W_{2,v};\lambda)$ is of the form  $Q_v(s,\lambda)\mathcal{B}_v(\lambda)\mathcal{L}_v(\lambda),$ where $Q_v(s,\lambda)$ is a polynomial in $s$ and $\lambda_j,$ $1\leq j\leq n;$ $\mathcal{B}_v(\lambda)$ is a holomorphic function, and
   	\begin{align*}
   	\mathcal{L}_v(\lambda)=\prod_{1\leq i<r}\prod_{i< j\leq r}L_v(1+\lambda_i-\lambda_j,\chi_{v,i}^{ur}\overline{\chi}_{v,j}^{ur})^{-1}\cdot L_v(1-\lambda_i+\lambda_j,\overline{\chi}_{v,i}^{ur}\chi_{v,j}^{ur})^{-1}.
   	\end{align*}
   \end{prop}

   \begin{remark}
   	Let $v\in \Sigma_{F,\infty}$ be an archimedean place such that $\pi_v$ is unramified and $\Phi_v=\Phi_v^{\circ}$ is the characteristic function of $G(\mathcal{O}_{F,v}).$ Assume that $\phi_{1,v}=\phi_{2,v}=\phi_v^{\circ}$ be the unique $G(\mathcal{O}_{F,v})$-fixed vector in the space of $\pi_v$ such that $\phi_v^0(e)=1.$ Applying the result in \cite{Sta02} we then have that  $\Psi_v\left(s,W_{1,v},W_{2,v};\lambda,\Phi_v\right)$ is equal to
   	\begin{align*}
   	\prod_{1\leq i<r}\prod_{i< j\leq r}\frac{L_v(s,\pi_{\lambda,v}\otimes\tau_v\times\widetilde{\pi}_{-\lambda,v})}{L_v(1+\lambda_i-\lambda_j,\pi_{i,v}\times\widetilde{\pi}_{j,v})\cdot L_v(1-\lambda_i+\lambda_j,\widetilde{\pi}_{i,v}\times\pi_{j,v})}.
   	\end{align*}
   	In particular, $R_v(s,\lambda)$ is independent of $s.$
   \end{remark}

   \subsection{Global Theory for $\Psi\left(s,W_{1},W_{2};\lambda\right)$}
   In this section, we shall compute the global integral representation $\Psi\left(s,W_{1},W_{2};\lambda,\Phi\right)$ defined via \eqref{100}. 
   
   Let $\widetilde{\pi}_{\lambda,v}$ be the contragredient of $\pi_{\lambda,v}.$ Let $\varpi_v$ be a uniformizer of $\mathcal{O}_{F,v},$ the ring of integers of $F_v.$ Let $q_v=N_{F_v/\mathbb{Q}_p}(\varpi_v),$ where $p$ is the rational prime such that $v$ is above $p.$ Denote by
   \begin{align*}
   R(s,W_{1},W_{2};\lambda):=\prod_{v\in\Sigma_{F}}\frac{\Psi_v\left(s,W_{1,v},W_{2,v};\lambda,\Phi_v\right)}{L_v(s,\pi_{\lambda,v}\otimes\tau_v\times\widetilde{\pi}_{-\lambda,v})},\quad \Re(s)>1.
   \end{align*}
   Then $R(s,W_{1},W_{2};\lambda)$ is holomorphic for any $\lambda\in i\mathfrak{a}_P^*/i\mathfrak{a}_G^*.$ Putting the local computations together in the last section, we get
   \begin{prop}[Global Case]\label{46'}
   	Let notation be as before. Let $s\in \mathbb{C}$ be such that $\Re(s)>1$.Then 
   	\begin{description}
   		\item[(a)] The integral $\Psi\left(s,W_{1},W_{2};\lambda,\Phi\right)$ converges absolutely in $\Re(s)>1.$
   		\item[(b)] We have the global functional equation for $\Re(s)>1:$
   		\begin{align*}
   		\Psi\left(1-s,\widetilde{W}_{1},\widetilde{W}_{2};\lambda,\tau^{-1},\widehat{\Phi}\right)=\Psi\left(s,{W}_{1},{W}_{2};\lambda,\tau,{\Phi}\right).
   		\end{align*}
   		\item[(c)] For any fixed $\lambda\in i\mathfrak{a}_P^*/i\mathfrak{a}_G^*,$ $R(s,W_{1},W_{2};\lambda)$ can be continued to an entire function. 
   	\end{description}
   \end{prop} 
   \begin{remark}
   	By Proposition \ref{prop47}, we see that if the irreducible representation  $\pi=\Ind_{P(\mathbb{A}_F)}^{G(\mathbb{A}_F)}(\pi_1,\cdots,\pi_{n})$ is a principal series which is $K$-finite, then 
   	\begin{equation}\label{1}
   	\frac{\Psi_{f}\left(s,W_{1},W_{2};\lambda,\Phi\right)}{L(s,\pi_{\lambda}\otimes\tau\times\widetilde{\pi}_{-\lambda})}=\mathcal{H}_{f}(s,\lambda)\underset{\substack{1\leq i,j\leq n\\ i\neq j}}{\prod\prod}\frac{1}{L(1+\lambda_i-\lambda_j,\pi_i\times\widetilde{\pi}_j)},
   	\end{equation}
   	where $\Psi_{f}\left(s,W_{1},W_{2};\lambda,\Phi\right)$ is the finite component of $\Psi\left(s,W_{1},W_{2};\lambda,\Phi\right)$ and $\mathcal{H}_f(s,\lambda)$ is a finite product of polynomials, depending on the $K$-type of $\pi.$  
   	
   	According to Proposition \ref{prop} we have, for each irreducible representation  $\pi=\Ind_{P(\mathbb{A}_F)}^{G(\mathbb{A}_F)}(\pi_1,\cdots,\pi_{n})$ is a principal series which is $K$-finite, that 
   	\begin{equation}\label{1.}
   	\frac{\Psi_{\infty}\left(s,W_{1},W_{2};\lambda,\Phi\right)}{L_{\infty}(s,\pi_{\lambda}\otimes\tau\times\widetilde{\pi}_{-\lambda})}=\mathcal{H}_{\infty}^*(s,\lambda)\underset{\substack{1\leq i,j\leq n\\ i\neq j}}{\prod\prod}\frac{1}{L_{\infty}(1+\lambda_i-\lambda_j,\pi_i^{ur}\times\widetilde{\pi}_j^{ur})},
   	\end{equation}
   	where $\mathcal{H}_{\infty}^*(s,\lambda)$ is a product of polynomials and Mellin transform of Schwartz functions. Moreover, $\mathcal{H}_{\infty}^*(s,\lambda)$ is nonvanishing when $\Re(s)\geq 1-2/(n^2+1).$ Let $v$ be an archimedean place. Let $\Sigma_1$ be the set of archimedean places such that $F_v\simeq \mathbb{R}$ and $\pi_{v,i}\widetilde{\pi}_{v,j}$ is ramified. Then for any $v\in\Sigma_1,$ one has  $L_{\infty}(1+\lambda_i-\lambda_j,\pi_i^{ur}\times\widetilde{\pi}_j^{ur})L_{\infty}(1+\lambda_i-\lambda_j,\pi_i\times\widetilde{\pi}_j)^{-1}=\Gamma_{\mathbb{R}}(1+\lambda_i-\lambda_j)\cdot \Gamma_{\mathbb{R}}(2+\lambda_i-\lambda_j)^{-1}.$ Let $\Sigma_2$ be the set of archimedean places such that  $F_v\simeq \mathbb{C}$ and $\pi_{v,i}\widetilde{\pi_{v,j}}$ is ramified, then for any $v\in\Sigma_2,$ one has  $L_{\infty}(1+\lambda_i-\lambda_j,\pi_i^{ur}\times\widetilde{\pi}_j^{ur})L_{\infty}(1+\lambda_i-\lambda_j,\pi_i\times\widetilde{\pi}_j)^{-1}=\Gamma_{\mathbb{C}}(1+\lambda_i-\lambda_j)\cdot \Gamma_{\mathbb{C}}(k_v+1+\lambda_i-\lambda_j)^{-1}=\prod_{l=0}^{k_v-1}(l+1+\lambda_i-\lambda_j)^{-1},$ where $k_v\in\mathbb{N}_{\geq 1}.$
   	
   	Let $\mathcal{H}_{\infty}(s,\lambda)$ be the product of $\mathcal{H}_{\infty}^*(s,\lambda)$ and the function 
   	\begin{align*}
   	\underset{\substack{1\leq i,j\leq n\\ i\neq j}}{\prod\prod}\prod_{v_2\in\Sigma_2}\prod_{l=0}^{k_{v_2}-1}(l+1+\lambda_i-\lambda_j)\prod_{v_1\in\Sigma_1}{\Gamma\left(\frac{2+\lambda_i-\lambda_j}2\right)}\Gamma\left(\frac{1+\lambda_i-\lambda_j}2\right)^{-1}.
   	\end{align*}
   	Then $\mathcal{H}_{\infty}(s,\lambda)$ is holomorphic with respect to $s\in\mathbb{C}$ and with respect to $\lambda=(\lambda_1,\cdots,\lambda_n)$ in the domain $|\lambda_i-\lambda_j|<2,$ $1\leq i, j\leq n.$
   	
   	Let $\mathcal{H}(s,\lambda)=\mathcal{H}_{\infty}(s,\lambda)\mathcal{H}_{f}(s,\lambda).$ Then by \eqref{1} and \eqref{1.} we have 
   	\begin{equation}\label{ar}
   	\frac{\Psi\left(s,W_{1},W_{2};\lambda,\Phi\right)}{\Lambda(s,\pi_{\lambda}\otimes\tau\times\widetilde{\pi}_{-\lambda})}=\mathcal{H}(s,\lambda)\underset{\substack{1\leq i,j\leq n\\ i\neq j}}{\prod\prod}\frac{1}{\Lambda(1+\lambda_i-\lambda_j,\pi_i\times\widetilde{\pi}_j)}.
   	\end{equation}
   \end{remark}
   
   Let notation be as before, we then define, for $\lambda\in i\mathfrak{a}_P^*/i\mathfrak{a}_G^*$ and $\phi_2\in \mathfrak{B}_{P,\chi},$ that 
   \begin{equation}\label{47''}
   R_{\varphi}(s,\lambda;\phi_2)=\sum_{\phi_1\in \mathfrak{B}_{P,\chi}}\langle\mathcal{I}_P(\lambda,\varphi)\phi_1,\phi_2\rangle\cdot\frac{\Psi(s,W_1,W_2;\lambda)}{\Lambda(s,\pi_{\lambda}\otimes\tau\times\widetilde{\pi}_{-\lambda})},\ \Re(s)>1,
   \end{equation}
   where $\Lambda(s,\pi_{\lambda}\otimes\tau\times\widetilde{\pi}_{-\lambda})$ is the complete $L$-function, defined by $\prod_{v\in\Sigma_F}L_v(s,\pi_{\lambda,v}\otimes\tau_v\times\widetilde{\pi}_{-\lambda,v}).$ Write $\varphi$ as a finite sum of convolutions $\varphi_{\alpha}*\varphi_{\beta}.$ Since $\mathfrak{B}_{P,\chi}$ is finite dimensional, we have, when $\Re(s)>1,$ that 
   \begin{equation}\label{105'''}
   \sum_{\phi_1\in \mathfrak{B}_{P,\chi}}\langle\mathcal{I}_P(\lambda,\varphi)\phi_1,\phi\rangle\cdot\frac{\Psi(s,W_1,W_2;\lambda)}{\Lambda(s,\pi_{\lambda}\otimes\tau\times\widetilde{\pi}_{-\lambda})}=\sum_{\alpha}\sum_{\beta}\frac{\Psi(s,W_{\alpha},W_{\beta};\lambda)}{\Lambda(s,\pi_{\lambda}\otimes\tau\times\widetilde{\pi}_{-\lambda})},
   \end{equation}
   where $W_{\beta}(x;\lambda)=W(x,\mathcal{I}_P(\lambda,\varphi_{\beta})\phi;\lambda)$ and $W_{\alpha}(x;\lambda)$ is the Whittaker function defined by $W_{\alpha}(x;\lambda)=W(x,\mathcal{I}_P(\lambda,\varphi_{\alpha})\phi;\lambda).$ Then we have $\Psi(s,W_{\alpha},W_{\beta};\lambda)$ equal to $\prod\Psi_v(s,W_{\alpha,v},W_{\beta,v};\lambda),$ $\Re(s)>1;$ and each $\Psi_v(s,W_{\alpha,v},W_{\beta,v};\lambda)$ is a finite sum of $\Psi_v(s,W_{1,v},W_{2,v};\lambda).$ Then according to Proposition \ref{43'} and Proposition \ref{45'} we see that, when $\Re(s)>1,$ $\Psi_v(s,W_{\alpha,v},W_{\beta,v};\lambda)\Lambda_v(s,\pi_{\lambda,v}\otimes\tau_v\times\widetilde{\pi}_{-\lambda},v)^{-1}$ are independent of $s$ for all but finitely many places $v,$ and as a function of $s,$ is a finite product of holomorphic function in $\Re(s)>0.$ Hence both sides of \eqref{105'''} are well defined and is meromorphic in $\Re(s).$ Then after continuation we have, for $\Re(s)>0,$ that 
   \begin{equation}\label{106'}
   R_{\varphi}(s,\lambda;\phi)=\sum_{\phi_1\in \mathfrak{B}_{P,\chi}}\langle\mathcal{I}_P(\lambda,\varphi)\phi_1,\phi\rangle R(s,W_1,W_2;\lambda)=\sum_{\alpha,\beta}R(s,W_{\alpha},W_{\beta};\lambda).
   \end{equation}
   Then clearly by Theorem \ref{39'} we have that, for $\Re(s)>1,$ $I_{\infty}^{(1)}(s)$ is equal to 
   \begin{align*}
   \sum_{\chi}\sum_{P\in \mathcal{P}}\frac{1}{c_P}\sum_{\phi\in \mathfrak{B}_{P,\chi}}\int_{\Lambda^*} R_{\varphi}(s,\lambda;\phi)\Lambda(s,\pi_{\lambda}\otimes\tau\times\widetilde{\pi}_{-\lambda})d\lambda.
   \end{align*}
   Note that the integrands make sense in the critical strip  $0<\Re(s)<1.$ To continue $I_{\infty}^{(1)}(s)$ to a meromorphic function in the right half plane $\Re(s)>0,$ we need to show that the summation expressing $I_{\infty}^{(1)}(s)$ in Theorem \ref{39'} in fact converges absolutely in the strip $S_{[0,1]}=\{s\in \mathbb{C}:\ 0<\Re(s)<1\},$ which we call the critical strip.
   \section{Absolute Convergence in the Critical Strip $S_{(0,1)}$}\label{3.}
   Let $1\leq m\leq n$ be an integer and $\pi\in\mathcal{A}(GL_m(F)\backslash GL_{m}(\mathbb{A}_F))$ be a cuspidal representation of $\GL_m$ over $F.$ For $v\in \Sigma_{fin},$ let $f(\pi_v)$ be the conductor of $\pi_v,$ set $C(\pi_v)=q_v^{f(\pi_v)},$ where $q_v$ is the cardinality of the residual field of $F_v,$ then $C(\pi_v)=1$ for all but finitely many finite places $v.$ For $v\in\Sigma_{\infty},$ then $F_v\simeq \mathbb{R}$ or $F_v\simeq \mathbb{C}.$ Let $L_v(s,\pi)=\prod_j\Gamma_{F_v}(s+\mu_{\pi_v,j})$ be the associated $L$-factor of $\pi_v.$ Denote in this case by $C(\pi_v;t)=\prod_{j}(2+|it+\mu_{\pi,j}|_{F_v})^{[F_v: \mathbb{R}]},$ $t\in\mathbb{R},$ and set $C(\pi_v)=C(\pi_v;0).$ 
   \begin{defn}[Analytic Conductor]
   	Let notation be as above, denote by $C(\pi;t)=\prod_{v\in\Sigma_F}C(\pi_v;t),$ $t\in\mathbb{R}.$ We call $C(\pi)=C(\pi;0)$ the analytic conductor of $\pi.$ Note that it is well defined.
   \end{defn}

   To prove our crucial Theorem \ref{47'}, we need an explicit upper bound for Rankin-Selberg $L$-functions in the critical strip in terms of the corresponding analytic conductors. However, the standard convexity bound $L(1/2,\sigma\otimes\tau\times\sigma')\ll_{\epsilon} C(\sigma\otimes \tau\times\sigma')^{1/2+\epsilon}$ is unknown unconditionally for general cuspidal representations $\sigma\in\mathcal{A}_0(GL_m(F)\backslash GL_{m}(\mathbb{A}_F))$ and $\sigma'\in\mathcal{A}_0(GL_{m'}(F)\backslash GL_{m'}(\mathbb{A}_F)).$ To remedy this, we prove a preconvex estimate (but is sufficient for our purpose) for $L(s,\sigma\otimes\tau\times\sigma')$ in the critical strip $0<\Re(s)<1.$ 
   \begin{lemma}[Preconvex bound]\label{49lem}
   	Let $1\leq m,m'\leq n$ be two integers. Let $\sigma\in\mathcal{A}_0(GL_m(F)\backslash GL_{m}(\mathbb{A}_F))$ and $\sigma'\in\mathcal{A}_0(GL_{m'}(F)\backslash GL_{m'}(\mathbb{A}_F)).$ Let $\beta_{m,m'}=1-1/(m^2+1)-1/(m'^2+1).$ Then for $s\in\mathbb{C}$ such that $0<\Re(s)<1,$ we have
   	\begin{equation}\label{105''}
   	L(s,\sigma\otimes\tau\times\sigma')\ll_{F,\epsilon} \left(1+|s(s-1)|^{-1}\right)C(\sigma\otimes \tau\times\sigma';s)^{\frac{1+\beta_{m,m'}-\Re(s)}{2}+\epsilon},
   	\end{equation}
   	where the implies constant is absolute, depending only on $\epsilon$ and the base field $F.$
   \end{lemma}
   \begin{proof}
   	By definition, $\tau$ extends to a character on $G(\mathbb{A}_F)$ via composing with the determinant map, i.e., by setting $\tau(x)=\tau(|\det x|_{\mathbb{A}_F}),$ for any $x\in G(\mathbb{A}_F).$ Thus $\tau$ is automorphic and invariant on $N(\mathbb{A}_F).$ Hence $\sigma\otimes\tau$ is also cuspidal.  We may write the cuspidal representations as $\sigma\otimes\tau=\otimes_v(\sigma_v\otimes\tau_v)$ and $\sigma'=\otimes'_v\sigma'_v.$ For prime ideals $\mathfrak{p}$ at which neither $\sigma_{\mathfrak{p}}$ or $\sigma'_{\mathfrak{p}}$ is ramified, let $\{St_{\sigma\otimes\tau,i}(\mathfrak{p})\}_{i=1}^m$ and $\{St_{\sigma',j}(\mathfrak{p})\}_{j=1}^{m'}$ be the respective Satake parameters of $\sigma\otimes\tau$ and $\sigma'.$ The Rankin-Selberg $L$-function at such a $\mathfrak{p}$ (there are all but finitely many such primes) is defined to be 
   	\begin{align*}
   	L_{\mathfrak{p}}(s,\sigma_{\mathfrak{p}}\otimes\tau_{\mathfrak{p}}\times\sigma'_{\mathfrak{p}})=\prod_{i=1}^m\prod_{j=1}^{m'}\left(1-St_{\sigma\otimes\tau,i}(\mathfrak{p})St_{\sigma',j}(\mathfrak{p})N_{F/\mathbb{Q}}(\mathfrak{p})^{-s}\right)^{-1}.
   	\end{align*}
   	Since $\tau$ is unitary, by \cite{LRS99} we have $\big|\log_{N_{F/\mathbb{Q}}(\mathfrak{p})}|St_{\sigma\otimes\tau,i}(\mathfrak{p})|\leq 1/2-1/(m^2+1),$ and $\big|\log_{N_{F/\mathbb{Q}}(\mathfrak{p})}|St_{\sigma',j}(\mathfrak{p})|\leq 1/2-1/(m'^2+1).$ For the remaining places $\mathfrak{p},$ The Rankin-Selberg $L$-function at such a $\mathfrak{p}$ can be written as 
   	\begin{align*}
   	L_{\mathfrak{p}}(s,\sigma_{\mathfrak{p}}\otimes\tau_{\mathfrak{p}}\times\sigma'_{\mathfrak{p}})=\prod_{i=1}^m\prod_{j=1}^{m'}\left(1-St_{\sigma\otimes\tau\times \sigma',i,j}(\mathfrak{p})N_{F/\mathbb{Q}}(\mathfrak{p})^{-s}\right)^{-1},
   	\end{align*}
   	with $\big|\log_{N_{F/\mathbb{Q}}(\mathfrak{p})}|St_{\sigma\otimes\tau\times \sigma',i,j}(\mathfrak{p})|\big|\leq \big|\log_{N_{F/\mathbb{Q}}(\mathfrak{p})}|St_{\sigma\otimes\tau,i}(\mathfrak{p})|+\big|\log_{N_{F/\mathbb{Q}}(\mathfrak{p})}|St_{\sigma',j}(\mathfrak{p})|,$ which is bounded by $\beta_{m,m'}=1-1/(m^2+1)-1/(m'^2+1).$ Then an easy estimate implies that for any $s$ such that $\beta=\Re(s)>1+\beta_{m,m'},$ we have
   	\begin{align*}
   	\big|L(s,\sigma\otimes\tau\times\sigma')\big|&=\prod_{\mathfrak{p}}\big|L_{\mathfrak{p}}(s,\sigma_{\mathfrak{p}}\otimes\tau_{\mathfrak{p}}\times\sigma'_{\mathfrak{p}})\big|\leq \prod_{\mathfrak{p}}\prod_{i=1}^m\prod_{j=1}^{m'}\big|1-N_{F/\mathbb{Q}}(\mathfrak{p})^{\beta_{m,m'}-\beta}\big|^{-1}\\
   	&=\prod_{\mathfrak{p}}\big|1-N_{F/\mathbb{Q}}(\mathfrak{p})^{\beta_{m,m'}-\beta}\big|^{-mm'}=\zeta_F(\beta-\beta_{m,m'})^{mm'},
   	\end{align*}
   	where $\zeta_F(s)$ is the Dedekind zeta function associated to $F/\mathbb{Q}.$ In particular, 
   	\begin{equation}\label{105'}
   	\big|L(\beta+i\gamma,\sigma\otimes\tau\times\sigma')\big|\leq \zeta_F(\beta-\beta_{m,m'})^{mm'}=O_{\beta_0}(1),\ \beta\geq\beta_0>1+\beta_{m,m'}.
   	\end{equation}
   	Also, at each infinite place $v\mid\infty,$ there exists a set of $mm'$ complex parameters $\{\mu_{\sigma\otimes\tau\times \sigma';v,i,j}:\ 1\leq i\leq m,\ 1\leq j\leq m'\}$ such that each local $L$-factor at $v$ is 
   	\begin{align*}
   	L_v(s,\sigma_v\otimes\tau_v\times\sigma'_v)=Q_v(s)\prod_{i=1}^m\prod_{j=1}^{m'}\Gamma_{F_v}\left(s+\mu_{\sigma\otimes\tau\times \sigma';v,i,j}\right),
   	\end{align*}
   	where $Q_v(s)$ is entire. Likewise, we have $\big|\mu_{\sigma\otimes\tau\times \sigma';v,i,j}\big|\leq \beta_{m,m'},$ according to loc. cit. Moreover, since $\widetilde{\sigma_v\otimes\tau_v}=\widetilde{\sigma}_v\otimes\widetilde{\tau}_v,$ the finite set $\{\overline{\mu_{\sigma\otimes\tau\times \sigma';v,i,j}}:\ 1\leq i\leq m,\ 1\leq j\leq m'\}$ is equal to $\{\mu_{\widetilde{\sigma}\otimes\widetilde{\tau}\times\widetilde{\sigma'};v,i,j}:\ 1\leq i\leq m,\ 1\leq j\leq m'\}$ for any $v\in\Sigma_{F,\infty}.$ Note that by Stirling's formula one has, for $s=\beta+i\gamma,$ where $\beta<1-\beta_{m,m'},$ that 
   	\begin{align*}
   	{\Gamma\left({1-s+\bar{\mu}}/{2}\right)}\cdot{\Gamma\left({s+\mu}/{2}\right)}^{-1}\ll_{\beta}\left(1+|i\gamma+\mu|\right)^{1/2-\beta},
   	\end{align*}
   	for any $\mu\in\mathbb{C}$ such that $\Re(\mu)>-1+\beta_{m,m'}.$ Then combining these with the duplication formula $\Gamma_{\mathbb{C}}(s)=\Gamma_{\mathbb{R}}(s)\Gamma_{\mathbb{R}}(s+1)$ we have, for $s=\beta+i\gamma$ with $\beta<1-\beta_{m,m'},$ that 
   	\begin{align*}
   	\prod_{v\mid\infty}\frac{L_v(1-s,\widetilde{\sigma}_v\otimes\widetilde{\tau}_v\times\widetilde{\sigma'}_v)}{L_v(s,\sigma_v\otimes\tau_v\times\sigma'_v)}\ll_{\beta}\prod_{v\mid\infty}C(\sigma_v\otimes\tau_v\times\sigma'_v;\gamma)^{1/2-\beta}.
   	\end{align*}
   	Hence together with the functional equation we have 
   	\begin{equation}\label{106}
   	L(\beta+i\gamma,\sigma\otimes\tau\times\sigma')=O\left(C(\sigma\otimes\tau\times\sigma';\gamma)^{1/2-\beta}\right),\ \beta\leq\beta_0<-\beta_{m,m'}.
   	\end{equation}
   	
   	If $\sigma\otimes\tau\ncong\sigma',$ then according to \cite{JS81}, $L(s,\sigma\otimes\tau\times\sigma')$ is entire. Then combining the nice analytic properties of $L(s,\sigma\otimes\tau\times\sigma')$ (see loc. cit.) and Phragm\'{e}n-Lindel\"{o}f principle with \eqref{105'} and \eqref{106} we obtain the following preconvex bound in the interval $-\beta_{m,m'}\leq\beta\leq1+\beta_{m,m'}:$
   	\begin{equation}\label{107}
   	L(\beta+i\gamma,\sigma\otimes\tau\times\sigma')\ll_{\epsilon}C(\sigma\otimes\tau\times\sigma';\gamma)^{\frac{1+\beta_{m,m'}-\beta}{2}+\epsilon}.
   	\end{equation}
   	
   	If $\sigma\otimes\tau\simeq\sigma',$ then according to loc. cit., $L(s,\sigma\otimes\tau\times\sigma')$ has simple poles precisely at $s=1$ and possibly at $s=0.$ Consider instead the function $f(s)=s(s-1)(s+2)^{-(5+\beta_{m,m'}-\beta)/2}L(s,\sigma\otimes\tau\times\sigma').$ Then clearly $f(s)$ is holomorphic and of order 1 in the right half plane $\Re(s)>-\beta_{m,m'}.$ Hence by \eqref{105'}, \eqref{106} and Phragm\'{e}n-Lindel\"{o}f principle we have that $f(s)$ is bounded by $O_{\epsilon}\left(C(\sigma\otimes\tau\times\sigma';\gamma)^{(1+\beta_{m,m'}-\beta)/{2}+\epsilon}\right)$ in the strip $-\beta_{m,m'}\leq\Re(s)\leq 1+\beta_{m,m'},$ leading to the estimate
   	\begin{equation}\label{108}
   	L(\beta+i\gamma,\sigma\otimes\tau\times\sigma')\ll_{\epsilon}|(\beta+i\gamma)(\beta+i\gamma-1)|^{-1}C(\sigma\otimes\tau\times\sigma';\gamma)^{\frac{1+\beta_{m,m'}-\beta}{2}+\epsilon},
   	\end{equation}
   	where $0<\beta<1.$ Now \eqref{105''} follows from \eqref{107} and \eqref{108}.
   \end{proof}
   \begin{lemma}\label{48'}
   	Let $s=\beta+i\gamma$ such that $\beta>0$ and $\gamma\in\mathbb{R}.$ Then one has
   	\begin{equation}\label{109'}
   	|s|^{-1-2|\gamma|}\cdot e^{-C(s)}\leq\big|\Gamma(s)\big|\leq \beta\Gamma(\beta)\cdot|s|^{-1},
   	\end{equation}
   	where $C(s)=\min\big\{\frac{\pi^2|s|^2}6+\beta, 2|s|\big\}.$
   	Moreover, if $|\gamma|\geq 1,$ we have uniformly that 
   	\begin{equation}\label{109}
   	\Gamma(s)=\sqrt{2\pi}e^{-\frac{\pi}2|\gamma|}|\gamma|^{\beta-\frac12}e^{i\gamma(\log|\gamma|-1)}e^{\frac{i\pi\delta(s)}2\cdot(\beta-\frac12)}\left(1+\lambda(s)|\gamma|^{-1}\right),
   	\end{equation}
   	where $\delta(s)=1$ if $\gamma\geq1,$ and $\delta(s)=-1$ if $\gamma\leq -1;$ and $|\lambda(s)|\leq e^{1/3+{\beta^2}/{2}+{\beta^3}/{3}}-1.$ 
   \end{lemma}
   \begin{proof}
   	Consider $1/\Gamma(s),$ which is an entire function. Take the logarithm of its Hadamard decomposition $1/\Gamma(s)=se^{\gamma_0 s}\prod (1+s/n)e^{-s/n}$ (here $\gamma_0=0.57721\cdots$ is the Euler-Mascheroni constant) and take real parts on both sides to get 
   	\begin{align*}
   	\log\big|{\Gamma(s)}^{-1}\big|=\log|s|+\gamma_0\beta+\Re\sum_{n< 2|s|}\left(\log\left(1+\frac{s}n\right)-\frac{s}n\right)+S_I,
   	\end{align*} 
   	where $S_I=\Re\sum_{n\geq 2|s|}\left(\log(1+s/n)-s/n\right).$ Expand the logarithm to see 
   	\begin{align*}
   	\big|S_I\big|=\Bigg|\Re\sum_{n\geq 2|s|}\sum_{k\geq 2}\frac{(-1)^{k-1}s^k}{kn^k}\Bigg|\leq \frac12\sum_{n\geq2|s|}\frac{{|s|}^2{n}^{-2}}{1-{|s|}/{n}}\leq \sum_{n\geq2|s|}\left(\frac{|s|}{n}\right)^{2}\leq C_1(s),
   	\end{align*}
   	where $C_1(s)=\min\{\pi^2/6,  {|s|^2}/(2|s|-1)\}.$ Therefore, $\log\big|\frac1{\Gamma(s)}\big|$ is no more than
   	\begin{align*}
   	\log|s|+\gamma_0\beta+\sum_{n\leq 2|s|}\left(\frac{|s|}n-\frac{\beta}n\right)+C_1(s)\leq (1+2|s|-2\beta)\log|s|+C(s),
   	\end{align*}
   	which is further bounded by $(1+2\gamma)\log|s|+C(s).$ This proves the left inequality of \eqref{109'}. For the right hand side, consider the integral representation of $\Gamma(s+1)=s\Gamma(s),$ we have $  |s\Gamma(s)|=\big|\int_{0}^{\infty}t^{s}e^{-t}dt\big|\leq \int_{0}^{\infty}\big|t^{s}\big|e^{-t}dt=\int_{0}^{\infty}t^{\beta}e^{-t}dt=\beta\Gamma(\beta),$ which proves the right inequality of \eqref{109'}. Hence \eqref{109'} holds.
   	
   	To prove \eqref{109}, we may assume that $\gamma\geq1.$ Write $s=\beta+i\gamma=\rho e^{i\theta},$ then 
   	\begin{equation}\label{110}
   	0<\theta=\frac{\pi}2-\arctan\frac{\beta}{\gamma}\leq 
   	\begin{cases}
   	\frac{\pi}2,\ \text{if $\beta\geq 0$;}\\
   	\frac{\pi}2-\arctan \beta_0,\ \text{if $\beta_0\leq \beta<0$,}
   	\end{cases}
   	\end{equation}
   	where $\arctan x$ is taken its principal value, i.e., $-\pi/2<\arctan x<\pi/2,$ $\forall$ $x\in\mathbb{R}.$
   	
   	A standard application of Euler-MacLaurin summation formula leads to that
   	\begin{equation}\label{112}
   	\log\Gamma(s)=(s-1/2)\log s-s+1/2\log2\pi+\int_{0}^{\infty}\frac{b(u)}{(u+s)^2}du,
   	\end{equation}
   	where $b(u)=1/2\{u\}-1/2\{u\}^2,$ here $\{u\}:=u-[u]$ with $[u]$ denoting the Gauss symbol, i.e., $[u]$ is the largest integer no more than $u.$ Then 
   	\begin{equation}\label{113}
   	\Big|\int_{0}^{\infty}\frac{b(u)}{(u+s)^2}du\Big|\leq {1/2}\left(\cos\frac{\theta}2\right)^{-2}\int_{0}^{\infty}\frac{du}{(\rho+u)^2}\leq\frac1{6\rho}\left(\cos\frac{\theta}2\right)^{-2}\leq \frac1{3\gamma},
   	\end{equation}
   	since $0<\theta/2\leq \pi/4$ according to \eqref{110}. Substitute \eqref{113} into \eqref{112} to get
   	\begin{align*}
   	\log\Gamma(s)&=(\beta+i\gamma-1/2)\log \sqrt{\beta^2+\gamma^2}+i\theta-\beta-i\gamma+\log\sqrt{2\pi}+C_1(\gamma)\\
   	&=(\beta-1/2)\log \sqrt{\beta^2+\gamma^2}-\gamma\theta
   	-\beta+\frac12\log2\pi+C_1(\gamma)+iC_2(s),
   	\end{align*}
   	where $C_2(s)=\gamma\log\sqrt{\beta^2+\gamma^2}+(\beta-1/2)\cdot\theta-\gamma.$ Also one has elementary inequalities
   	\begin{equation}\label{114}
   	\begin{cases}
   	\big|\arctan{\beta}{\gamma}^{-1}-{\beta}{\gamma}^{-1}\big|\leq \big|{\beta}^3/(3{\gamma}^3)\big|,\\
   	\big|\log\sqrt{\beta^2+\gamma^2}-\log\gamma\big|\leq {\beta^2}/({2\gamma^2}).
   	\end{cases}
   	\end{equation}
   	Then plugging \eqref{114} into the expansion of $\log\Gamma(s)$ to get that $\log\Gamma(s)$ equal to
   	\begin{equation}\label{115}
   	\frac12\log2\pi-\frac{\pi\gamma}2+\left(\beta-\frac12\right)\log\gamma+i\Big\{\gamma\log\gamma-\gamma+\frac{\pi}2\left(\beta-\frac12\right)\Big\}+C_3(s),
   	\end{equation}
   	where $\big|C_3(s)\big|\leq \left(1/3+{\beta^2}/{2}+{\beta^3}/{3}\right)\cdot|{\gamma}|^{-1}.$ Then the case $\gamma\geq 1$ of \eqref{109} follows from \eqref{115} and the elementary inequality $|e^{cx}-1|\leq (e^c-1)x,$ for any $c>0$ and $0<x\leq 1.$ Taking the complex conjugate of both sides gives the case where $\gamma\leq -1.$ Hence the lemma follows. 
   \end{proof}
   
   \begin{cor}\label{50''}
   	Let $1\leq m,m'\leq n$ be two integers. Let $\sigma\in\mathcal{A}_0(GL_m(F)\backslash GL_{m}(\mathbb{A}_F))$ and $\sigma'\in\mathcal{A}_0(GL_{m'}(F)\backslash GL_{m'}(\mathbb{A}_F)).$ Let $v$ be an archimedean place. Let $\beta\geq5.$ Then for each $s=\beta_0+i\gamma\in\mathbb{C}$ such that $1-1/(n^2+1)<\beta_0<1$ and $\gamma\in\mathbb{R},$ we have
   	\begin{equation}\label{115'}
   	\big|L_v(s,\sigma_v\otimes\tau_v\times\sigma'_v)\big|\leq C_{\beta}\cdot \Big|\frac{L_v(\beta+i\gamma,\sigma_v\otimes\tau_v\times\sigma'_v)}{C(\sigma_v\otimes\tau_v\times\sigma'_v;\gamma)}\Big|,
   	\end{equation}
   	where $C_{\beta}$ is an absolute constant depending only on $\beta,$ $n$ and the base field $F.$
   \end{cor}
   \begin{proof}
   	Let $t_{\beta}=2e^{1/3+{\beta^2}/{2}+{\beta^3}/{3}}.$ Then $t_{\beta}>2.$ Recall that by definition 
   	\begin{equation}\label{116}
   	L_v(s,\sigma_v\otimes\tau_v\times\sigma'_v)=\prod_{j=1}^m\prod_{k=1}^{m'}\Gamma_{F_v}\left(s+\mu_{\sigma\otimes\tau\times \sigma';v,j,k}\right).
   	\end{equation}
   	We can write $\mu_{\sigma\otimes\tau\times \sigma';v,j,k}=\beta_{j,k}+i\gamma_{j,k}.$ Then $|\beta_{j,k}|\leq 1-1/(1+m^2)-1/(1+m'^2)\leq 1-2/(1+n^2).$ Let $t_{j,k}=\gamma+\gamma_{j,k},$ $1\leq j\leq m,$ $1\leq k\leq m'.$ Let $\delta_v=2/[F_v: \mathbb{R}].$
   	\begin{itemize}
   		\item[Case 1:] If $|t_{j,k}|\leq t_{\beta}.$ Then by the estimate \eqref{109'} we see that
   		\begin{align*}
   		&\Bigg|\frac{\Gamma_{F_v}\left(s+\mu_{\sigma\otimes\tau\times \sigma';v,j,k}\right)\cdot\left(2+|i\gamma+\mu_{\sigma\otimes\tau\times \sigma';v,i,j}|_{F_v}\right)^{[F_v: \mathbb{R}]}}{\Gamma_{F_v}\left(\beta+i\gamma+\mu_{\sigma\otimes\tau\times \sigma';v,j,k}\right)}\Bigg|\\
   		\leq&\Bigg|\Gamma\left(\frac{\beta_0+\beta_{j,k}}{\delta_{v}}\right)\cdot\left(2+|\beta_{j,k}|_{F_v}+|t_{\beta}|_{F_v}\right)^{2}e^{\frac{2(\beta+|\beta_{j,k}|+t_{\beta})}{\delta_v}} \cdot \bigg[\frac{\beta+|\beta_{j,k}|+t_{\beta}}{\delta_v}\bigg]^{1+2t_{\beta}}\Bigg|,
   		\end{align*}
   		which can be seen clearly to be bounded by 
   		\begin{align*}
   		C_{1;j,k}(\beta)\triangleq(\beta+|t_{\beta}|^2+1)^{2t_{\beta}+3}e^{2(\beta+|t_{\beta}|+1)}\cdot\max_{1/(n^2+1)\leq \beta'\leq 2}\Gamma({\beta'}/{\delta_{v}}).
   		\end{align*}
   		\item[Case 2:] If $|t_{j,k}|\geq t_{\beta}>2.$ Then by the estimate \eqref{109} we see that
   		\begin{align*}
   		&\Bigg|\frac{\Gamma_{F_v}\left(s+\mu_{\sigma\otimes\tau\times \sigma';v,j,k}\right)\cdot\left(2+|i\gamma+\mu_{\sigma\otimes\tau\times \sigma';v,i,j}|_{F_v}\right)^{[F_v: \mathbb{R}]}}{\Gamma_{F_v}\left(\beta+i\gamma+\mu_{\sigma\otimes\tau\times \sigma';v,j,k}\right)}\Bigg|\\
   		\leq&\Bigg|\frac{\left(2+|\beta_{j,k}|_{F_v}+|t_{\beta}|_{F_v}\right)^{2}\cdot\left(\left(1+|\lambda(s+\mu_{\sigma\otimes\tau\times \sigma';v,j,k})|\right)\cdot|t_{j,k}|^{-1}\right)}{|\delta_v^{-1}t_{j,k}|^{\beta-\beta_0}\cdot\left(\left(1-|\lambda(\beta+i\gamma+\mu_{\sigma\otimes\tau\times \sigma';v,j,k})|\right)\cdot|t_{j,k}|^{-1}\right)}\Bigg|,
   		\end{align*}
   		which is bounded by $3(3+|t_{j,k}|^2)^{2}|\delta_v^{-1}t_{j,k}|^{1-\beta}\triangleq C_{3;j,k}(t_{j,k}).$ Note that $C_{3;j,k}(t_{j,k})$ is bounded in the interval $[t_{\beta},\infty).$ So we can define 
   		$$
   		C_{2;j,k}(\beta)=\sup_{|t_{j,k}|\geq t_{\beta}}C_{3;j,k}(t_{j,k}).
   		$$
   	\end{itemize}
   	Now let $C_{j,k}(\beta)=\max\{C_{1;j,k}(\beta),C_{2;j,k}(\beta)\}.$ Then $C_{j,k}(\beta)$ is well defined, $1\leq j\leq m,$ $1\leq k\leq m'.$ Set $C_{\beta}=\prod_{j}\prod_{k}C_{j,k}(\beta).$ Then \eqref{115'} follows from \eqref{116} and 
   	\begin{align*}
   	C(\sigma_v\otimes\tau_v\times\sigma'_v;\gamma)=\prod_{i=1}^m\prod_{j=1}^{m'}\left(2+|i\gamma+\mu_{\sigma\otimes\tau\times \sigma';v,i,j}|_{F_v}\right)^{[F_v: \mathbb{R}]}.
   	\end{align*}
   \end{proof}
   \begin{remark}
   	Let $1\leq m,m'\leq n$ be two integers. Let $\sigma\in\mathcal{A}_0(GL_m(F)\backslash GL_{m}(\mathbb{A}_F))$ and $\sigma'\in\mathcal{A}_0(GL_{m'}(F)\backslash GL_{m'}(\mathbb{A}_F)).$ Let $v$ be an archimedean place. Let $N\geq 1$ and $\beta\geq 4N+1.$ Then essentially the same proof of Corollary \ref{50''} leads to the result that for each $s=\beta_0+i\gamma\in\mathbb{C}$ such that $0<\beta_0<1$ and $\gamma\in\mathbb{R},$ we have
   	\begin{equation}\label{115''}
   	\big|L_v(s,\sigma_v\otimes\tau_v\times\sigma'_v)\big|\leq C_{N,\beta}\cdot \Big|\frac{L_v(\beta+i\gamma,\sigma_v\otimes\tau_v\times\sigma'_v)}{C(\sigma_v\otimes\tau_v\times\sigma'_v;\gamma)^N}\Big|,
   	\end{equation}
   	where $C_{N,\beta}$ is an absolute constant depending only on $N,$ $\beta,$ $n$ and the base field $F.$ This slightly general bound \eqref{115''} will be used in \cite{Yan19}.
   \end{remark}
   Let $v\in\Sigma_{fin}$ be a nonarchimedean place of $F.$ Let $\Phi_{v,l}$ be a constant multiplying the characteristic function of some connected compact subset of $F_v^n.$ 
   Now we consider integrals $\Psi_{v}^{*}\left(s,W_{1,v},W_{1,v};\lambda,\Phi_{v,l}\right)$ defined by
   \begin{align*}
   \int_{N(F_v)\backslash G(F_v)} W_{1,v}(x_v;\lambda)\overline{W_{1,v}\left(x_v;-\bar{\lambda}\right)}\cdot\Phi_{v,l}(\eta x_v)|\det x_v|_{F_v}^{s}dx_v.
   \end{align*}
   Let $\widetilde{W}_{1,v}$ be the Whittaker function of $\widetilde{\pi}_{v,-\lambda},$ defined via $\widetilde{W}_{1,v}(x)=W_{1,v}(wx^{\iota}),$ where $x\in G(F_v)$ and $w$ is the longest element in $W_P\backslash W/W_P.$ Define the integral $\Psi_{v}^*(s,\widetilde{W}_{1,v},\widetilde{W}_{1,v};\lambda,\widehat{\Phi_{v,l}})$ by
   \begin{align*}
   \int_{N(F_v)\backslash G(F_v)} \widetilde{W}_{1,v}\overline{\widetilde{W}_{1,v}\left(x_v;-\bar{\lambda}\right)}\cdot\widehat{\Phi_{v,l}}(\eta x_v)|\det x_v|_{F_v}^{s}dx_v.
   \end{align*}

   \begin{lemma}\label{50'}
   	Let notation be as before. Let $0<\epsilon<1/2.$ Let $q_v$ be the cardinality of the residue field of $F_v.$ Let $W_1$ be a Whittaker function associated to $\chi\in\mathfrak{X}_P.$ Then there exists a constant $c_v$ depending only on the test function $\varphi$ such that 
   	\begin{equation}\label{117'}
   	\big|\Psi_{v}^*\left(s,W_{1,v},W_{1,v};\lambda,\Phi_{v,l}\right)\big|\leq q_v^{c_v} \big|\Psi_{v}^*\left(1-\epsilon,W_{1,v},W_{1,v};\lambda,\Phi_{v,l}\right)\big|,
   	\end{equation}
   	for any $s\in\mathbb{C}$ such that $\Re(s)=\epsilon.$
   \end{lemma}
   \begin{proof}
   	We can apply the same argument on the support of $\widetilde{W}_{1,v}$ as that of $W_{1,v}$ in the proof of Corollary \ref{28'} shows that there exists a positive integer ${m_v}=m_v(\varphi_v),$ depending only on the place $v\in S(\pi,\Phi)$ and the $K$-finite type of the test function $\varphi,$ such that $\supp \widetilde{W}_{1,v}\mid_{A(F_v)}\subseteq \{x=\diag(x_1,\cdots,x_n)\in A(F_v):\ \max\{|x_i|_v\}\leq q_v^{m_v}\}.$ Then one has, for any $s$ such that $\Re(s)=\epsilon,$ that
   	\begin{align*}
   	&q_v^{-nm_v\epsilon}\Big|\int_{N(F_v)\backslash G(F_v)} \widetilde{W}_{1,v}(x_v;\lambda)\overline{\widetilde{W}_{1,v}\left(x_v;-\bar{\lambda}\right)}\cdot\widehat{\Phi_{v,l}}(\eta x_v)|\det x_v|_{F_v}^{\epsilon}dx_v\Big|\\
   	\geq &q_v^{-nm_v(1-\epsilon)}\Big|\int_{N(F_v)\backslash G(F_v)} \widetilde{W}_{1,v}(x_v;\lambda)\overline{\widetilde{W}_{1,v}\left(x_v;-\bar{\lambda}\right)}\cdot\widehat{\Phi_{v,l}}(\eta x_v)|\det x_v|_{F_v}^{1-\epsilon}dx_v\Big|,
   	\end{align*}
   	from which one easily obtains the inequality that 
   	\begin{equation}\label{118'}
   	\big|\Psi_v^*(\epsilon,\widetilde{W}_{1,v},\widetilde{W}_{1,v};\lambda,\widehat{\Phi_{v,l}})\big|\geq q_v^{-nm_v} \big|\Psi_v^*(1-\epsilon,\widetilde{W}_{1,v},\widetilde{W}_{1,v};\lambda,\widehat{\Phi_{v,l}})\big|.
   	\end{equation}
   	On the other hand, we have the functional equation
   	\begin{equation}\label{125}
   	\Bigg|\frac{\Psi_{v}^*(1-s,\widetilde{W}_{1,v},\widetilde{W}_{1,v};\lambda,\widehat{\Phi_{v,l}})}{L_v(1-s,\pi_{\lambda,v}\times\widetilde{\pi}_{-\lambda,v})}\Bigg|=\Bigg|\frac{\epsilon(s,\pi_{v},\lambda)\Psi_{v}^*\left(s,W_{1,v},W_{1,v};\lambda,\Phi_{v,l}\right)}{L_v(s,\pi_{\lambda,v}\times\widetilde{\pi}_{-\lambda,v})}\Bigg|,
   	\end{equation}
   	where  $\epsilon(s,\pi_{v},\lambda)=\gamma(s,\pi_v,\lambda)L_v(s,\pi_{\lambda,v}\times\widetilde{\pi}_{-\lambda,v})L_v(1-s,\pi_{\lambda,v}\times\widetilde{\pi}_{-\lambda,v})^{-1}$ is the $\epsilon$-factor, here $\gamma(s,\pi_v,\lambda)$ is the $\gamma$-factor. By the stability of $\gamma$-factors and \cite{CP17}, we have the stability of $\epsilon(s,\pi_{v},\lambda).$ Thus $\epsilon(s,\pi_{v},\lambda)=\prod\prod\epsilon(s+\lambda_i-\lambda_j,\sigma_{v,i}\times\widetilde{\sigma}_{v,j},\lambda).$ Let $q_v=N_{F/\mathbb{Q}}(\mathfrak{p}).$ Then one has that (see \cite{JPSS83}) each $\epsilon(s+\lambda_i-\lambda_j,\sigma_{v,i}\times\widetilde{\sigma}_{v,j},\lambda)$ is of the form $cq_v^{-f_vs},$ where $|c|=q_v^{1/2}$ and $f_v$ is the local conductor, which is bounded by an absolute constant depending only on $K_v$-type of the test function $\varphi.$ Hence there exists some absolute constant $e_{v}\in\mathbb{N}_{\geq0}$, relying only on $\varphi,$ such that \begin{equation}\label{120''}
   	\big|\epsilon(s,\pi_{v},\lambda)\epsilon(1-s,\pi_{v},\lambda)^{-1}\big|\geq q_v^{-e_v\epsilon}.
   	\end{equation}
   	Then combine \eqref{118'}, \eqref{125} and \eqref{120''} we have
   	\begin{equation}\label{122'}
   	\Bigg|\frac{\Psi_{v}^*\left(\epsilon,W_{1,v},W_{1,v};\lambda,\Phi_{v,l}\right)}{L_v(\epsilon,\pi_{\lambda,v}\times\widetilde{\pi}_{-\lambda,v})^{2}}\Bigg|\leq q^{nm_v+e_v\epsilon} \Bigg|\frac{\Psi_{v}^*\left(1-\epsilon,W_{1,v},W_{1,v};\lambda,\Phi_{v,l}\right)}{L_v(1-\epsilon,\pi_{\lambda,v}\times\widetilde{\pi}_{-\lambda,v})^{2}}\Bigg|.
   	\end{equation}
   	Since $\pi_{v,\lambda}\in\mathfrak{X}_P$ is generic, then it is irreducible. Hence, according to \cite{CP17}, we have
   	\begin{equation}\label{104}
   	L_v(s,\pi_{\lambda,v}\times\widetilde{\pi}_{-\lambda,v})=\prod_{i=1}^r\prod_{j=1}^rL_v(s+\lambda_i-\lambda_j,\sigma_{v,i}\times\widetilde{\sigma}_{v,j}).
   	\end{equation}
   	Let $\mathfrak{p}$ be the prime ideal representing the place $v\in \Sigma_{fin}.$ Then for any $s,$
   	\begin{align*}
   	L_v(s,\pi_{\lambda,v}\times\widetilde{\pi}_{-\lambda,v})^{-1}=\prod_{i=1}^r\prod_{j=1}^r\prod_{k=1}^{n_i}\prod_{l=1}^{n_j}\left(1-St_{\sigma\times \sigma',k,l}(\mathfrak{p})N_{F/\mathbb{Q}}(\mathfrak{p})^{-s-\lambda_i+\lambda_j}\right)
   	\end{align*}
   	is a finite product, hence it is an entire function. Moreover, since $\big|St_{\sigma\times \sigma',k,l}(\mathfrak{p})\big|\leq N_{F/\mathbb{Q}}(\mathfrak{p})^{\beta_{i,j}},$ where $\beta_{i,j}=1-1/(n_i^2+1)-1/(n_j^2+1),$ we then have
   	\begin{equation}\label{124'}
   	\Big|L_v(s,\pi_{\lambda,v}\times\widetilde{\pi}_{-\lambda,v})^{-1}\Big|\leq\prod_{i=1}^r\prod_{j=1}^r\left(1+N_{F/\mathbb{Q}}(\mathfrak{p})^{-\Re(s)+\beta_{i,j}}\right)^{n_i+n_j},
   	\end{equation}
   	where $n_i$ and $n_j$ are ranks of components of Levi subgroup of $P$ respectively. Also, 
   	\begin{equation}\label{124''}
   	\Big|L_v(s,\pi_{\lambda,v}\times\widetilde{\pi}_{-\lambda,v})^{-1}\Big|\geq\prod_{i=1}^r\prod_{j=1}^r\left(1-N_{F/\mathbb{Q}}(\mathfrak{p})^{-\Re(s)+\beta_{i,j}}\right)^{n_i+n_j}.
   	\end{equation}
   	Then it follows from \eqref{122'}, \eqref{124'} and \eqref{124''} that 
   	\begin{equation}\label{125'}
   	\big|\Psi_{v}^*\left(\epsilon,W_{1,v},W_{1,v};\lambda,\Phi_{v,l}\right)\big|\leq q_v^{c_v} \big|\Psi_{v}^*\left(1-\epsilon,W_{1,v},W_{1,v};\lambda,\Phi_{v,l}\right)\big|,
   	\end{equation}
   	where $c_v$ is a constant depending only on the test function $\varphi.$ Noting that $\Phi_{v,l}$ is a constant multiplying the characteristic function of some connected compact subset of $F_v^n,$ so $\big|\Psi_{v}^{*}\left(s,W_{1,v},W_{1,v};\lambda,\Phi_{v,l}\right)\big|\leq \big|\Psi_{v}^{*}\left(\Re(s),W_{1,v},W_{1,v};\lambda,\Phi_{v,l}\right)\big|.$ Then \eqref{117'} follows from this inequality and \eqref{125'}.
   \end{proof}

   With the preparation above, now we can prove the following result:    
   \begin{thmx}\label{47'}
   	Let $s\in\mathbb{C}$ be such that $0<\Re(s)<1,$ then 
   	\begin{equation}\label{105}
   	\sum_{\chi}\sum_{P\in \mathcal{P}}\frac{1}{c_P}\sum_{\phi\in \mathfrak{B}_{P,\chi}}\int_{\Lambda^*} R_{\varphi}(s,\lambda;\phi)\Lambda(s,\pi_{\lambda}\otimes\tau\times\widetilde{\pi}_{-\lambda})d\lambda,
   	\end{equation} 
   	converges absolutely, normally with respect to $s,$ where $R_{\varphi}(s,\lambda;\phi)$ is defined in \eqref{47''} and $\Lambda(s,\pi_{\lambda}\otimes\tau\times\widetilde{\pi}_{-\lambda})$ is the complete $L$-function.
   \end{thmx}
   \begin{proof}
   	Fix a proper parabolic subgroup $P\in\mathcal{P}$ of type $(n_1,n_2,\cdots,n_r).$ Let $\mathfrak{X}_P$ be the subset of cuspidal data $\chi=\{(M,\sigma)\}$ such that $M=M_P.$ Denote by 
   	\begin{align*}
   	J_P(s)=\sum_{\chi\in \mathfrak{X}_P}\sum_{\phi\in \mathfrak{B}_{P,\chi}}\int_{\Lambda^*} R_{\varphi}(s,\lambda;\phi)\Lambda(s,\pi_{\lambda}\otimes\tau\times\widetilde{\pi}_{-\lambda})d\lambda.
   	\end{align*}
   	
   	Let $M_P=\diag(M_1,M_2,\cdots,M_r),$ where $M_i$ is $n_i$ by $n_i$ matrix, $1\leq i\leq r.$ We may write $\sigma=(\sigma_1,\sigma_2,\cdots,\sigma_r),$ where $\sigma_i\in\mathcal{A}_0(M_i(F)\backslash M_i(\mathbb{A}_F)).$ By the $K$-finiteness of $\varphi,$ each $\sigma_i$ has a fixed finite type, so its conductor is bounded uniformly (depending only on $\varphi$). Let $C_{\infty}(\sigma_i\otimes\tau\times\sigma_j;t)=\prod_{v\in\Sigma_{F,\infty}}C(\sigma_{i,v}\otimes\tau_v\times\sigma_{j,v};t).$ Then one has $C_{\infty}(\sigma_i\otimes\tau\times\sigma_j;t)\asymp_{\varphi}C(\sigma_i\otimes\tau\times\sigma_j;t),$ where the implies constant depends only on $\supp\varphi.$ For any $\Phi=\Phi_{\infty}\cdot\prod_{v<\infty}\Phi_{v}\in\mathcal{S}_0(\mathbb{A}_F^n),$ where $\Phi_{\infty}=\prod_{v\mid\infty}\Phi_v.$ Let $x_v=(x_{v,1},x_{v,2},\cdots,x_{v,n})\in F_v^n,$ then by definition, $\Phi_v$ is of the form
   	\begin{equation}\label{117}
   	\Phi_v(x_v)=e^{-\pi \sum_{j=1}^nx_{v,j}^2}\cdot \sum_{k=1}^mQ_k(x_{v,1},x_{v,2},\cdots,x_{v,n}),
   	\end{equation}
   	where $F_v\simeq \mathbb{R},$ $Q_k(x_{v,1},x_{v,2},\cdots,x_{v,n})\in \mathbb{C}[x_{v,1},x_{v,2},\cdots,x_{v,n}]$ are monomials; and
   	\begin{equation}\label{118}
   	\Phi_v(x_v)=e^{-2\pi \sum_{j=1}^nx_{v,j}\bar{x}_{v,j}}\cdot\sum_{k=1}^m Q_k(x_{v,1},\bar{x}_{v,1},x_{v,2},\bar{x}_{v,2},\cdots,x_{v,n},\bar{x}_{v,n}),
   	\end{equation}
   	where $F_v\simeq \mathbb{C}$ and $Q_k(x_{v,1},\bar{x}_{v,1},x_{v,2},\bar{x}_{v,2},\cdots,x_{v,n},\bar{x}_{v,n})$ are monomials in the ring $\mathbb{C}[x_{v,1},\bar{x}_{v,1},x_{v,2},\bar{x}_{v,2},\cdots,x_{v,n},\bar{x}_{v,n}].$ Thus there exists a finite index set $J$ such that 
   	\begin{align*}
   	\Phi_{\infty}(x_{\infty})=\sum_{\boldsymbol{j}=(j_v)_{v\mid\infty}\in J}\prod_{v\mid\infty}\Phi_{v,j_v}(x_v),\quad x_{\infty}=\prod_{v\mid\infty}x_v\in G(\mathbb{A}_{F,\infty}),
   	\end{align*}
   	where each $\Phi_{v,j_v}$ is of the form in \eqref{117} or \eqref{118} with $m=1.$ Let $\Phi_{\infty,\boldsymbol{j}}=\prod_{v\mid\infty}\Phi_{v,j-v},$ $\boldsymbol{j}=(j_v)_{v\mid\infty}\in J.$ Then $\Phi$ is equal to the sum over $\boldsymbol{j}\in J$ of each $\Phi_{\boldsymbol{j}}=\Phi_{\infty,\boldsymbol{j}}\prod_{v<\infty}\Phi_{v}\in\mathcal{S}_0(\mathbb{A}_F^n).$ 
   	
   	Since for each $v\mid\infty$ and $j\in J,$ $\Psi_v\left(s,W_{\alpha,v},W_{\beta,v};\lambda,\Phi_{v,j}\right)$ converges absolutely in $\Re(s)>0$ (see \cite{Jac09}), one has 
   	\begin{equation}\label{119}
   	\Big|\prod_{v\mid\infty}\Psi_v\left(s,W_{\alpha,v},W_{\beta,v};\lambda,\Phi_{v}\right)\Big|\leq \sum_{\boldsymbol{j}\in J}\Big|\prod_{v\mid\infty}\Psi_v\left(s,W_{\alpha,v},W_{\beta,v};\lambda,\Phi_{v,j_v}\right)\Big|.
   	\end{equation}
   	Since each $\Phi_{v,j_v}$ is a monomial multiplying an exponential function with negative exponent, $\Psi_v\left(s,W_{\alpha,v},W_{\beta,v};\lambda,\Phi_{v,j_v}\right)$ is in fact of the form $c_1\pi^{c_2s}\prod_i\prod_j\Gamma(s+\nu_{i,j}),$ where $c_1,$ $c_2$ and $\nu_i$ are some constants and the product is finite. Although these parameters depend on the representations $\sigma$ and $\tau,$ the local Rankin-Selberg integral $\Psi_v\left(s,W_{\alpha,v},W_{\beta,v};\lambda,\Phi_{v,j_v}\right)$ is either nonvanishing in $\Re(s)>0$ or vanishing identically (i.e. $c_1=0$). Note that for each archimedean place $v,$ there exists a polynomial $Q(s)=Q(s,\lambda)$ (see loc. cit.) depending on $\pi_{\infty}$ and $\lambda,$ such that 
   	\begin{align*}
   	L_v(s,\pi_{\lambda,v}\otimes\tau_v\times\widetilde{\pi}_{-\lambda,v})=Q(s,\lambda)\prod_{i=1}^r\prod_{j=1}^rL_v(s+\lambda_i-\lambda_j,\sigma_{v,i}\otimes\tau_v\times\widetilde{\sigma}_{v,j}),
   	\end{align*}
   	where $\Re(s)>\beta_{n,n}=1-2/(n^2+1).$ Clearly $Q(s,\lambda)$ is nonvanishing in $\Re(s)>\beta_{n,n}.$ Combining this with the preceding discussion we conclude that there exists a polynomial $Q_{v,j}(s;\lambda)$ (depending on $\pi$) for each $\lambda$ such that 
   	\begin{align*}
   	\Psi_v\left(s,W_{\alpha,v},W_{\beta,v};\lambda,\Phi_{v,j_v}\right)=Q_{v,j_v}(s;\lambda)\prod_{i=1}^r\prod_{j=1}^rL_v(s+\lambda_i-\lambda_j,\sigma_{v,i}\otimes\tau_v\times\widetilde{\sigma}_{v,j}).
   	\end{align*}
   	Then the above analysis leads to that each $Q_v(s;\lambda)$ is either nonvanishing in $\Re(s)>\beta_{n,n}$ or vanishing identically (i.e. $c_1=0$). Write $Q_{v,j_v}(s;\lambda)=\prod(s-\varrho_{\lambda}),$ with $\Re(\varrho_{\lambda})\leq 0.$ Let $s_0=\beta_0+i\gamma_0$ such that $\beta_0, \gamma_0\in\mathbb{R}$ and $1-1/(n^2+1)\leq\beta_0<1.$ Let $U_{\epsilon}(s_0)\subsetneq S[0,1]$ (with $0<\epsilon<(1-\Re(s_0))/10$) be a neighborhood of $s_0.$ Then $|s-\varrho_{\lambda}|\leq |s'-\varrho_{\lambda}|,$ for any $s\in U_{\epsilon}(s_0)$ and $s'=\beta+i\Im(s),$ $\beta\geq 5.$ Therefore, 
   	\begin{equation}\label{120'}
   	|Q_{v,j_v}(s;\lambda)|\leq |Q_{v,j_v}(s';\lambda)|,\ \forall\ v\mid\infty,\ \boldsymbol{j}=(j_v)_{v\mid\infty}\in J,\ \lambda\in i\mathfrak{a}_P/i\mathfrak{a}_G^*,
   	\end{equation}
   	where $s=\beta+i\gamma\in U_{\epsilon}(s_0)$ and $s'=5+i\gamma.$ Combining \eqref{120'} with \eqref{115'} leads to 
   	\begin{equation}\label{120}
   	\Big|\prod_{v\mid\infty}\Psi_v\left(s,W_{\alpha,v},W_{\beta,v};\lambda,\Phi_{v,j_v}\right)\Big|\leq C_5^{d_F}\cdot\Bigg|\prod_{v\mid\infty}\frac{\Psi_v\left(s',W_{\alpha,v},W_{\beta,v};\lambda,\Phi_{v,j_v}\right)}{C_{v}(\pi_{\lambda,v}\otimes\tau_v\times\widetilde{\pi}_{-\lambda,v};\gamma)}\Bigg|,
   	\end{equation}
   	where $d_F=[F:\mathbb{Q}].$ Let $S(\pi,\Phi)$ be the finite set of nonarchimedean places such that   $\pi_v$ is unramified and $\Phi_v=\Phi_v^{\circ}$ is the characteristic function of $G(\mathcal{O}_{F,v})$ outside $\Sigma_{F,\infty}\cup S(\pi,\Phi).$ Then by Proposition \ref{43'} we have
   	\begin{align*}
   	R_{S(\pi,\Phi)}(s,\lambda)=\prod_{v\in S(\pi,\Phi)}R_v(s,W_{\alpha,v},W_{\beta,v};\lambda)\in \bigotimes_{v\in S(\pi,\Phi)}\mathbb{C}[q_v^{\pm s},  q_v^{\pm\lambda_i}:\ 1\leq i\leq r].
   	\end{align*}
   	Let $v\in S(\pi,\Phi).$ Write $R_{S(\pi,\Phi)}(s,W_{\alpha};\lambda)=\prod_{v\in S(\pi,\Phi)}R_v(s,W_{\alpha,v},W_{\alpha,v};\lambda);$ write $R_{S(\pi,\Phi)}(s,W_{\beta};\lambda)=\prod_{v\in S(\pi,\Phi)}R_v(s,W_{\beta,v},W_{\beta,v};\lambda).$ Then they both lie in the ring $\bigotimes_{v\in S(\pi,\Phi)}\mathbb{C}[q_v^{\pm s},  q_v^{\pm\lambda_i}:\ 1\leq i\leq r].$ Write $R_v(s;\lambda)=R_v(s,W_{\alpha,v},W_{\beta,v};\lambda)$ in this proof. By definition we have  $R_v(s;\lambda)=\Psi_v\left(s,W_{\alpha,v},W_{\beta,v};\lambda,\Phi_{v,j}\right)\cdot L_v(s,\pi_{\lambda,v}\otimes\tau_v\times\widetilde{\pi}_{-\lambda,v})^{-1},$ when $\Re(s)>1.$  Recall that $\Psi_v\left(s,W_{1,v},W_{2,v};\lambda,\Phi_{v}\right)$ is equal to
   	\begin{align*}
   	\int_{N(F_v)\backslash G(F_v)} W_{\alpha,v}(x_v;\lambda)\overline{W_{\beta,v}\left(x_v;-\bar{\lambda}\right)} \Phi_v(\eta x_v)\tau(\det x_v)|\det x_v|_{F_v}^sdx_v,
   	\end{align*}
   	which converges normally in $\Re(s)>0,$ uniformly in $\lambda\in i\mathfrak{a}_P^*/i\mathfrak{a}_G^*,$ due to the standard estimate on Whittaker functions (they are bounded by compactly supported functions in this case). Thus it defines an holomorphic function with respect to $s$ in the region $\Re(s)>0.$ By definition and gauge argument we see that the integral $\Psi_v\left(s,W_{\alpha,v},W_{\beta,v};\lambda,\Phi_{v}\right)$ converges normally in $\Re(s)>0.$ Therefore,   $\Psi_v\left(s,W_{\alpha,v},W_{\beta,v};\lambda,\Phi_{v}\right)\cdot L_v(s,\pi_{\lambda,v}\otimes\tau_v\times\widetilde{\pi}_{-\lambda,v})^{-1}$ is exactly $R_v(s,\lambda)$ for any  $\Re(s)>0.$ Since $\Phi_v$ is a Schwartz-Bruhat function, we can write $\Phi_v$ as a finite sum of $\Phi_{v,l},$ where each $\Phi_{v,l}$ is a constant multiplying a characteristic function of some connected compact subset of $F_v^n.$ Then the Fourier transform of $\Phi_{v,l}$ is of the same form. Recall that the integral $\Psi_{v}^{*}\left(s,W_{\alpha,v},W_{\alpha,v};\lambda,\Phi_{v,l}\right)$ is defined by
   	\begin{align*}
   	\int_{N(F_v)\backslash G(F_v)} W_{\alpha,v}(x_v;\lambda)\overline{W_{\alpha,v}\left(x_v;-\bar{\lambda}\right)}\cdot\Phi_{v,l}(\eta x_v)|\det x_v|_{F_v}^{s}dx_v.
   	\end{align*}
   	Hence $\big|\Psi_{v}^{*}\left(s,W_{\alpha,v},W_{\alpha,v};\lambda,\Phi_{v,l}\right)\big|\leq \big|\Psi_{v}^{*}\left(\Re(s),W_{\alpha,v},W_{\alpha,v};\lambda,\Phi_{v,l}\right)\big|.$ Likewise, one has $\big|\Psi_{v}^*(s,\widetilde{W}_{\alpha,v},\widetilde{W}_{\alpha,v};\lambda,\widehat{\Phi_{v,l}})\big|\leq \big|\Psi_{v}^{*}(\Re(s),\widetilde{W}_{\alpha,v},\widetilde{W}_{\alpha,v};\lambda,\widehat{\Phi_{v,l}})\big|.$ Define 
   	\begin{align*}
   	H_{v,l}(s,c)=\frac{q_v^{cs}\Psi_{v}^*\left(s,W_{\alpha,v},W_{\alpha,v};\lambda,\Phi_{v,l}\right)}{(6-s)(s+5)L_v(s,\pi_{\lambda,v}\times\widetilde{\pi}_{-\lambda,v})},\ 1/2\leq \Re(s)\leq 5,
   	\end{align*}
   	where $v\in S(\pi,\Phi)$ and $c>0$ is a parameter to be determined, depending only on the test function $\varphi$. Clearly for any $c$, $H_{v,l}(s,c)$ is bounded in the strip $1/2\leq \Re(s)\leq 5,$ tends to zero as $\Im(s)$ tends to infinity. Let $0<\epsilon<2/(n^2+1)$ and $s_0'\in(1-\epsilon,1).$ Then by maximal principle, there exists an $s_1$ such that $\Re(s_1)=5$ or $\Re(s_1)=1/2$ such that $\big|H_{v,l}(s_0',c)\big|\leq \big|H_{v,l}(s_1,c)\big|.$ Now we assume $\Re(s_1)=1/2.$ Consider the functional equation \eqref{125}. Let $q_v=N_{F/\mathbb{Q}}(\mathfrak{p}).$ By the stability of $\epsilon(s,\pi_{v},\lambda),$ one has that (see \cite{JPSS83}) $\epsilon(s,\pi_{v},\lambda)=\prod\prod\epsilon(s+\lambda_i-\lambda_j,\sigma_{v,i}\times\widetilde{\sigma}_{v,j},\lambda)$ and each $\epsilon(s+\lambda_i-\lambda_j,\sigma_{v,i}\times\widetilde{\sigma}_{v,j},\lambda)$  is of the form $cq_v^{-f_vs},$ where $|c|=q_v^{1/2}$ and $f_v$ is the local conductor, which is bounded by an absolute constant depending only on $K_v$-type of the test function $\varphi.$ Hence there exists some absolute constant $e_{v}\in\mathbb{N}_{\geq0}$, relying only on $\varphi,$ such that 
   	\begin{equation}\label{128'}
   	\big|\epsilon(s,\pi_{v},\lambda)\epsilon(1/2,\pi_{v},\lambda)^{-1}\big|\geq q_v^{-e_v\Re(s)}.
   	\end{equation}
   	
   	The same argument on the support of $\widetilde{W}_{1,v}$ as that of $W_{1,v}$ in the proof of Corollary \ref{28'} shows that there exists a positive integer ${m_v}=m_v(\varphi_v),$ depending only on the place $v\in S(\pi,\Phi)$ and the $K$-finite type of the test function $\varphi,$ such that $\supp \widetilde{W}_{1,v}\mid_{A(F_v)}\subseteq \{x=\diag(x_1,\cdots,x_n)\in A(F_v):\ \max\{|x_i|_v\}\leq q_v^{m_v}\}.$ Then one has, for any $s$ in the strip $0<\Re(s)\leq1/2,$ that
   	\begin{align*}
   	&q_v^{-nm_v\Re(s)}\Big|\int_{N(F_v)\backslash G(F_v)} \widetilde{W}_{\alpha,v}(x_v;\lambda)\overline{\widetilde{W}_{\alpha,v}\left(x_v;-\bar{\lambda}\right)}\cdot\widehat{\Phi_{v,l}}(\eta x_v)|\det x_v|_{F_v}^{\Re(s)}dx_v\Big|\\
   	\geq &q_v^{-nm_v/2}\Big|\int_{N(F_v)\backslash G(F_v)} \widetilde{W}_{\alpha,v}(x_v;\lambda)\overline{\widetilde{W}_{\alpha,v}\left(x_v;-\bar{\lambda}\right)}\cdot\widehat{\Phi_{v,l}}(\eta x_v)|\det x_v|_{F_v}^{1/2}dx_v\Big|.
   	\end{align*}
   	Therefore we can substitute $s=1-s_0'$ into the above inequality to get
   	\begin{equation}\label{127'}
   	\Big|\Psi_v^*(1-s_0',\widetilde{W}_{\alpha,v},\widetilde{W}_{\alpha,v};\lambda,\widehat{\Phi_{v,l}})\Big|\geq q_v^{nm_v(\frac12-s_0')} \Big|\Psi_v^*(1/2,\widetilde{W}_{\alpha,v},\widetilde{W}_{\alpha,v};\lambda,\widehat{\Phi_{v,l}})\Big|.
   	\end{equation}
   	Then combining this with \eqref{125}, \eqref{128'} and \eqref{127'} one has
   	\begin{equation}\label{128}
   	\Bigg|\frac{\Psi_{v}^*\left(s_0',W_{\alpha,v},W_{\alpha,v};\lambda,\Phi_{v,l}\right)}{L_v(s_0',\pi_{\lambda,v}\times\widetilde{\pi}_{-\lambda,v})}\Bigg|\geq q_v^{\nu(s_0')}\cdot \Bigg|\frac{\Psi_{v}^*\left(1/2,W_{\alpha,v},W_{\alpha,v};\lambda,\Phi_{v,l}\right)}{L_v(1-s_0',\pi_{\lambda,v}\times\widetilde{\pi}_{-\lambda,v})}\Bigg|,
   	\end{equation}
   	where $\nu(s_0')=nm_v(1/2-s_0')+e_vs_0'$ is a constant depending only on the test function $\varphi.$ Denote by  $R_{v,l}^*(s,\lambda)$ the function $\Psi_{v}^*\left(s,W_{\alpha,v},W_{\alpha,v};\lambda,\Phi_{v,l}\right)/{L_v(s,\pi_{\lambda,v}\times\widetilde{\pi}_{-\lambda,v})},$ $\Re(s)>0.$ Then one combines \eqref{128} with \eqref{124'} and \eqref{124''} to get
   	\begin{equation}\label{129}
   	\big|R_{v,l}^*(s_0,\lambda)\big|\geq q_v^{\nu(s_0)}\cdot\prod_{i=1}^r\prod_{j=1}^r \Bigg|\frac{1-q_v^{s_0-1+\beta_{i,j}}}{1+q_v^{-1/2+\beta_{i,j}}}\Bigg|^{n_i+n_j}\cdot\big|R_{v,l}^*(1/2,\lambda)\big|.
   	\end{equation}
   	Let $c=c_{v,0}$ be a positive constant such that 
   	\begin{align*}
   	q_v^{(s_0'-1/2)c_{v,0}}>q_v^{-\nu(s_0')}\Big|\frac{(6-s_0')(s_0'+5)}{(6-1/2)(1/2+5)}\Big|\cdot\prod_{i=1}^r\prod_{j=1}^r \Bigg|\frac{1-q_v^{-1/2+\beta_{i,j}}}{1-q_v^{s_0'-1+\beta_{i,j}}}\Bigg|^{n_i+n_j}.
   	\end{align*}
   	Note that such a $c_{v,0}$ always exists since $s_0'>1/2.$ Then it follows from \eqref{124'}, \eqref{124''} and \eqref{129} that $\big|H_{v,l}(s_0',c)\big|\leq \big|H_{v,l}(s_1,c_{v,0})\big|.$ Therefore, we have a contradiction by assuming that $\Re(s_1)=1/2.$ Hence, we have $\Re(s_1)=5$ if $c=c_{v,0}.$ Then for any $s$ such that $\Re(s)=s_0',$ $\big|\Psi_v\left(s,W_{\alpha,v},W_{\alpha,v};\lambda,\Phi_{v,l}\right)\big|\leq \big|\Psi_{v}^*\left(s_0',W_{\alpha,v},W_{\alpha,v};\lambda,\Phi_{v,l}\right)\big|,$ which bounded, since $\big|H_{v,l}(s_0',c_{v,0})\big|\leq \big|H_{v,l}(s_1,c_{v,0})\big|,$ by
   	\begin{align*}
   	\Bigg|\frac{q_v^{5c_{v,0}}(6-s_0')(s_0'+5)L_v(s_0',\pi_{\lambda,v}\times\widetilde{\pi}_{-\lambda,v})}{10q_v^{c_{v,0}s_0'}L_v(s_1,\pi_{\lambda,v}\otimes\tau_v\times\widetilde{\pi}_{-\lambda,v})}\Bigg|\cdot  \Big|\Psi_{v}^*\left(5,W_{\alpha,v},W_{\alpha,v};\lambda,\Phi_{v,l}\right)\Big|.
   	\end{align*}
   	Then by \eqref{124''} and trivial estimate on $L_v(s_1,\pi_{\lambda,v}\otimes\tau_v\times\widetilde{\pi}_{-\lambda,v})$ one concludes that there exists some constant $c_v'',$ depending only on $\varphi,$ such that 
   	\begin{equation}\label{131}
   	\big|\Psi_v^*\left(s_0',W_{\alpha,v},W_{\alpha,v};\lambda,\Phi_{v}\right)\big|\leq q_v^{c_v''} \sum_{l=1}^{L_v}\Big|\Psi_{v}^*\left(5,W_{\alpha,v},W_{\alpha,v};\lambda,\Phi_{v,l}\right)\Big|.
   	\end{equation}
   	Then combining \eqref{131} and Lemma \ref{50'} we have, for any $s$ with $\epsilon\leq \Re(s)\leq 1-\epsilon,$ 
   	\begin{equation}\label{132}
   	\big|\Psi_v\left(s,W_{\alpha,v},W_{\alpha,v};\lambda,\Phi_{v}\right)\big|\leq q_v^{c_v'} \sum_{l=1}^{L_v}\Big|\Psi_{v}^*\left(5,W_{\alpha,v},W_{\alpha,v};\lambda,\Phi_{v,l}\right)\Big|,
   	\end{equation}
   	where $c_v'$ is a constant depending only on the test function $\varphi.$ 
   	
   	Note that when $\varphi_v$ is $G(\mathcal{O}_{F,v})$-invariant, then $\pi_{v,\lambda}$ is unramified. So the cardinality of the finite set $S(\pi,\Phi)$ is bounded in terms of $\tau,$ $\Phi$ and the $K$-finite type of the test function $\varphi.$ Namely, there exists a finite set $S_{\varphi,\tau,\Phi}$ of prime ideals such that for any $\pi$ from some cuspidal datum $\chi\in\mathfrak{X}_P,$ one has $S(\pi,\Phi)\subseteq S_{\varphi,\tau,\Phi}.$ Therefore, we conclude that 
   	\begin{equation}\label{124}
   	\big|R_{S(\pi,\Phi)}(s,W_{\alpha};\lambda)\big|\leq\sum_{\boldsymbol{l}=(l_v)_{v\in S(\pi,\Phi)}}\prod_{v\in S(\pi,\Phi)}q_v^{c_v'}\Big|\Psi_{v}^*\left(5,W_{\alpha,v},W_{\alpha,v};\lambda,\Phi_{v,l_v}\right)\Big|,
   	\end{equation}
   	where the sum over multi-index $\boldsymbol{l}$ is finite in terms of $\varphi,$ $\tau$ and $\Phi.$ Similarly,
   	\begin{equation}\label{124'''}
   	\big|R_{S(\pi,\Phi)}(s,W_{\beta};\lambda)\big|\leq\sum_{\boldsymbol{l}=(l_v)_{v\in S(\pi,\Phi)}}\prod_{v\in S(\pi,\Phi)}q_v^{c_v'}\Big|\Psi_{v}^*\left(5,W_{\beta,v},W_{\beta,v};\lambda,\Phi_{v,l_v}\right)\Big|.
   	\end{equation}
   	
   	By Proposition \ref{43'} we have, when $v\in\Sigma_{fin}-S(\pi,\Phi)\triangleq S_{\pi,\Phi}^{u.r.},$ that
   	\begin{align*}
   	R_v(s,\lambda)=\prod_{1\leq i<r}\prod_{i< j\leq r}L_v(1+\lambda_i-\lambda_j,\pi_{i,v}\times\widetilde{\pi}_{j,v})^{-1}\cdot L_v(1-\lambda_i+\lambda_j,\widetilde{\pi}_{i,v}\times\pi_{j,v})^{-1}
   	\end{align*}
   	is independent of $s.$ So we write $R_v(\lambda)$ for $R_v(s,\lambda)$ in this case. 
   	
   	Let $s\in U_{\epsilon}(s_0)$ and $s'=5+i\Im(s).$  Then by \eqref{119} and \eqref{124} we see that when $\phi_1=\phi_2\in\mathfrak{B}_{P,\chi},$ $\big|R(s,\lambda)\Lambda(s,\pi_{\lambda}\otimes\tau\times\widetilde{\pi}_{-\lambda})\big|$ is bounded by $\big|R_{S(\pi,\Phi)}(s,\lambda)\big|$ multiplying 
   	\begin{align*}
   	\big|L(s,\pi_{\lambda}\otimes\tau\times\widetilde{\pi}_{-\lambda})\big|\prod_{v\in S_{\pi,\Phi}^{u.r.}}\big|R_v(\lambda)\big|\cdot\sum_{\boldsymbol{j}\in J}\prod_{v\mid\infty}\Big|\Psi_v\left(s,W_{\alpha,v},W_{\beta,v};\lambda,\Phi_{v,j_v}\right)\Big|.	
   	\end{align*}
   	
   	By \eqref{105''} and \eqref{104} we have the preconvex bound $L(s,\pi_{\lambda}\otimes\tau\times\widetilde{\pi}_{-\lambda})\ll_{F,\epsilon} C_{\infty}(\pi_{\lambda}\otimes\tau\times\widetilde{\pi}_{-\lambda};\Im(s)).$ Then combining this bound with \eqref{120} we have
   	\begin{equation}\label{126}
   	\prod_{v\mid\infty}\Big|\Psi_v\left(s,W_{\alpha,v},W_{\beta,v};\lambda,\Phi_{v,j_v}\right)\Big|	\ll\prod_{v\mid\infty}\Bigg|\frac{\Psi_v\left(s',W_{\alpha,v},W_{\beta,v};\lambda,\Phi_{v,j_v}\right)}{L(s,\pi_{\lambda}\otimes\tau\times\widetilde{\pi}_{-\lambda})}\Bigg|,
   	\end{equation}
   	where the implied constant is absolute. Let 
   	$$
   	\Psi_{\infty}\left(s',W_{\alpha},W_{\beta};\lambda,\Phi_{\boldsymbol{j}}\right)=\prod_{v\mid\infty}\Psi_v\left(s',W_{\alpha,v},W_{\beta,v};\lambda,\Phi_{v,j_v}\right)
   	$$
   	if $\boldsymbol{j}=(j_v)_{v\mid\infty}.$ Similarly, for any $\boldsymbol{l}=(l_v)_{v\in S(\pi,\Phi)},$ we denote by 
   	$$
   	\Psi_{ra}^*\left(s',W_{\alpha};\lambda,\Phi_{\boldsymbol{l}}\right)=\prod_{v\in S(\pi,\Phi)}\Psi_{v}^*\left(s',W_{\alpha,v},W_{\alpha,v};\lambda,\Phi_{v,l_v}\right).
   	$$
   	Let $C=\prod_{v\in S(\pi,\Phi)}q_v^{c_v'}<\infty.$ Then combine \eqref{124}, \eqref{124'''} and \eqref{126} to conclude that 
   	\begin{align*}
   	\big|J_P(s)\big|\leq&\sum_{\alpha}\sum_{\beta}\sum_{\chi\in \mathfrak{X}_P}\sum_{\phi\in \mathfrak{B}_{P,\chi}}\int_{\Lambda^*}\big|R(s,W_{\alpha},W_{\beta};\lambda)\cdot\Lambda(s,\pi_{\lambda}\otimes\tau\times\widetilde{\pi}_{-\lambda})\big||d\lambda|\\
   	\leq&C\sum_{\alpha}\sum_{\beta}\Bigg[\sum_{\chi\in \mathfrak{X}_P}\sum_{\phi\in \mathfrak{B}_{P,\chi}}\int_{\Lambda^*}J_{\alpha}(\lambda)|d\lambda|\Bigg]^{\frac12}\Bigg[\sum_{\chi\in \mathfrak{X}_P}\sum_{\phi\in \mathfrak{B}_{P,\chi}}\int_{\Lambda^*}J_{\beta}(\lambda)|d\lambda|\Bigg]^{\frac12},
   	\end{align*}
   	where $J_{\alpha}(\lambda)=J_{\alpha}(\lambda;\chi,\phi)$ is defined by 
   	\begin{align*}
   	\sum_{\boldsymbol{j}\in J}\sum_{\boldsymbol{l}}\Big|\Psi_{\infty}\left(s',|W_{\alpha}|,|W_{\alpha}|;\lambda,|\Phi_{\boldsymbol{j}}|\right)\Psi_{ra}^*\left(5,W_{\alpha};\lambda,\Phi_{\boldsymbol{l}}\right)\Big|\cdot\prod_{v\in S_{\pi,\Phi}^{u.r.}}\big|R_v(\lambda)\big|.
   	\end{align*}
   	Likewise, we have definition of $J_{\beta}(\lambda)=J_{\beta}(\lambda;\chi,\phi)$ of the same form. Note that 
   	\begin{align*}
   	&\Big|\Psi_{\infty}\left(s',|W_{\alpha}|,|W_{\alpha}|;\lambda,|\Phi_{\boldsymbol{j}}|\right)\Psi_{ra}^*\left(5,W_{\alpha};\lambda,\Phi_{\boldsymbol{l}}\right)\Big|\cdot\prod_{v\in S_{\pi,\Phi}^{u.r.}}\big|R_v(\lambda)\big|\\
   	\leq &\Big|\Psi_{\infty}\left(s',|W_{\alpha}|,|W_{\alpha}|;\lambda,|\Phi_{\boldsymbol{j}}|\right)\Psi_{ra}^*\left(5,W_{\alpha};\lambda,\Phi_{\boldsymbol{l}}\right)\Big|\cdot\prod_{v\in S_{\pi,\Phi}^{u.r.}}\Bigg|\frac{\Psi_v^*\left(s',W_{\alpha,v};\lambda,|\Phi_{v}|\right)}{L(s',\pi_{\lambda}\otimes\tau\times\widetilde{\pi}_{-\lambda})}\Bigg|\\
   	\leq&C_0\prod_{v\in \Sigma_{F}}\int_{N(F_v)\backslash G(F_v)} \Big|W_{\alpha}(x_v;\lambda)\overline{W_{\alpha}\left(x_v;\lambda\right)} \Phi_{\boldsymbol{j},\boldsymbol{l},v}(\eta x_v)\Big|\cdot|\det x_v|_{F_v}^{5}dx_v,
   	\end{align*}
   	where $\Phi_{\boldsymbol{j},\boldsymbol{l},v}=|\Phi_{v,j_v}|$ if $v$ is archimedean; $\Phi_{\boldsymbol{j},\boldsymbol{l},v}=|\Phi_{v,l_v}|,$  if $v\in S_{\varphi,\tau,\Phi};$ and $\Phi_{\boldsymbol{j},\boldsymbol{l},v}=|\Phi_{v}|$ otherwise; and $C_0$ is an absolute constant, independent of $\pi$ and $\lambda.$ Note that $\Phi_{\boldsymbol{j},\boldsymbol{l}}\in \mathcal{S}_0(\mathbb{A}_F^n).$ Denote by $\Psi^*\left(5,W_{\alpha};\lambda,\Phi_{\boldsymbol{j},\boldsymbol{l}}\right)$ the last integral in the above inequalities. Then we have by the first part of Theorem \ref{39'} that
   	\begin{align*}
   	\sum_{\chi\in \mathfrak{X}_P}\sum_{\phi\in \mathfrak{B}_{P,\chi}}\int_{\Lambda^*}J_{\alpha}(\lambda)d\lambda\leq \sum_{\boldsymbol{j}\in J}\sum_{\boldsymbol{l}}\sum_{\chi\in \mathfrak{X}_P}\sum_{\phi\in \mathfrak{B}_{P,\chi}}\int_{\Lambda^*}\Psi^*\left(5,W_{\alpha};\lambda,\Phi_{\boldsymbol{j},\boldsymbol{l}}\right)d\lambda<\infty,
   	\end{align*}
   	since the sums over $\boldsymbol{j}$ and $\boldsymbol{l}$ are finite. Similarly, one has
   	\begin{align*}
   	\sum_{\chi\in \mathfrak{X}_P}\sum_{\phi\in \mathfrak{B}_{P,\chi}}\int_{\Lambda^*}J_{\beta}(\lambda)d\lambda\leq \sum_{\boldsymbol{j}\in J}\sum_{\boldsymbol{l}}\sum_{\chi\in \mathfrak{X}_P}\sum_{\phi\in \mathfrak{B}_{P,\chi}}\int_{\Lambda^*}\Psi^*\left(5,W_{\beta};\lambda,\Phi_{\boldsymbol{j},\boldsymbol{l}}\right)d\lambda<\infty.
   	\end{align*}
   	Since the sums over $\alpha$ and $\beta$ are finite, and since there are only finitely many standard parabolic subgroups $P$ of $G,$ we have shown that 
   	\begin{equation}\label{*}
   	\sum_{\chi}\sum_{P\in \mathcal{P}}\frac{1}{|c_P|}\sum_{\phi\in \mathfrak{B}_{P,\chi}}\int_{\Lambda^*} \max_{\Re(s)=s_0}\Big|R_{\varphi}(s,\lambda;\phi)\Lambda(s,\pi_{\lambda}\otimes\tau\times\widetilde{\pi}_{-\lambda})\Big|d\lambda<\infty,
   	\end{equation}
   	for any $1-1/(n^2+1)\leq s_0<1.$ Now we apply Proposition \ref{46'} to this result to see that \eqref{*} holds for any $0<\Re(s)\leq 1/(n^2+1).$ Note that  $R_{\varphi}(s,\lambda;\phi)\Lambda(s,\pi_{\lambda}\otimes\tau\times\widetilde{\pi}_{-\lambda})$ is analytic inside the strip $1/(n^2+1)\leq\Re(s)\leq 1-1/(n^2+1).$ Then by Phragm\'{e}n-Lindel\"{o}f principle we have that $\big|R_{\varphi}(s,\lambda;\phi)\Lambda(s,\pi_{\lambda}\otimes\tau\times\widetilde{\pi}_{-\lambda})\big|$ is bounded by
   	\begin{align*}
   	\max_{s_0\in \{1/(n^2+1),1-1/(n^2+1)\}}\max_{\Re(s)=s_0}\Big|R_{\varphi}(s,\lambda;\phi)\Lambda(s,\pi_{\lambda}\otimes\tau\times\widetilde{\pi}_{-\lambda})\Big|.
   	\end{align*}
   	Therefore, \eqref{105} holds for all $s\in\mathcal{S}_{(0,1)}.$
   \end{proof}

Gathering Theorem \ref{39'} and Theorem \ref{47'} we then conclude 
\begin{cor}
Let notation be as before. Assume $\tau$ is such that $\tau^k\neq 1$ for all $1\leq k\leq n.$ Then 
\begin{align*}
\sum_{\chi}\sum_{P\in \mathcal{P}}\frac{1}{c_P}\sum_{\phi\in \mathfrak{B}_{P,\chi}}\int_{\Lambda^*} R_{\varphi}(s,\lambda;\phi)\Lambda(s,\pi_{\lambda}\otimes\tau\times\widetilde{\pi}_{-\lambda})d\lambda
\end{align*} 
admits a holomorphic continuation to the whole $s$-plane.
\end{cor}
   
   \section{Holomorphic Continuation via Multidimensional Residues}\label{7.2}
   From preceding estimates, we see that when $\Re(s)>1,$ $I_{\infty}^{(1)}(s)$ is a combination of Rankin-Selberg convolutions for automorphic functions which are not of rapid decay. Zagier \cite{Zag81} computed the Rankin-Selberg transform of some type of automorphic functions and derived the desired holomorphic continuation for $n=2$ and $F=\mathbb{Q}$ case. However, general Eisenstein series for $\GL(n)$ do not have the asymptotic properties as Zagier considered, since there are mixed terms in the Fourier expansion (see Proposition \ref{Fourier}). Thus one needs to develop a different approach to obtain the continuation.
   
   $I_{\infty}^{(1)}(s)$ can be written as a sum of functions $\int_{\Lambda}\mathcal{F}(s,\lambda)d\lambda,$ which is well defined when $\Re(s)>1.$ Moreover, for each $s_0$ with $\Re(s_0)=1,$ there exists some $\lambda_0\in\Lambda$ such that $F(s,\lambda_0)$ is singular at $s=s_0.$ Hence the original integral representations for $I_{\infty}^{(1)}(s)$ have singularities at all points on the line $\Re(s)=1.$ We shall use contour-shifting and Cauchy's theorem to continue $I_{\infty}^{(1)}(s).$ To illustrate the underlying idea, we simply "think" $\mathcal{F}(s,\lambda)=(s-1-\lambda)^{-1}\cdot(s-1+\lambda)^{-1}$ and $\Lambda=i\mathbb{R},$ namely,
   \begin{align*}
   I_{\infty}^{(1)}(s)=\int_{-i\infty}^{i\infty}\frac{1}{(s-1-\lambda)(s-1+\lambda)}d\lambda,\ \ \Re(s)>1.
   \end{align*}
   
   Now we fix $s$ such that $1<\Re(s)<1+\epsilon/2,$ for some small $\epsilon>0.$ Then shift contour to see 
   \begin{equation}\label{C}
   I_{\infty}^{(1)}(s)=\int_{\epsilon-i\infty}^{\epsilon+i\infty}\frac{1}{(s-1-\lambda)(s-1+\lambda)}d\lambda-\frac{1}{2(s-1)}. 
   \end{equation}
   Note that the right hand side of \eqref{C} defines a meromorphic function in the region $1-\epsilon/2<\Re(s)<1+\epsilon/2,$ with a simple at $s=1.$ Hence we obtain a meromorphic continuation of $I_{\infty}^{(1)}(s)$ to the region $\Re(s)>1-\epsilon/2.$ Do this process one more time one then gets a meromorphic continuation to the whole complex plane, with explicit description on poles. 
   
   Just as the above prototype, the genuine situation admits the same idea of continuation, but with more delicate techniques required, since $I_{\infty}^{(1)}(s)/\Lambda(s,\tau)$ is typically infinitely many sums of such integrals. Details will be provide in the following subsections. Moreover, we find all possible explicit poles of the continuation of each such integral as well, and show they cancel with each other except for $s=1/2,$ where $I_{\infty}^{(1)}(s)/\Lambda(s,\tau)$ has at most a simple pole if $\tau^2=1.$
   
   \subsection{Continuation via a Zero-free Region}\label{7.11}
   Recall that we fix the unitary character $\tau.$ Let $\mathcal{D}_{\tau}$ be a standard (open) zero-free region of $L_F(s,\tau)$ (e.g. see \cite{Bru06}). We fix such a $\mathcal{D}_{\tau}$ once for all. We thus can form a domain
   \begin{equation}\label{R}
   \mathcal{R}(1/2;\tau)^-:=\{s\in\mathbb{C}:\ 2s\in \mathcal{D}_{\tau}\}\supsetneq \{s\in\mathbb{C}:\ \Re(s)\geq 1/2\}.
   \end{equation}
   In Section \ref{7.22}, we will continue $I_{\infty}^{(1)}(s)$ to the open set $\mathcal{R}(1/2;\tau)^-.$ Invoking \eqref{R} with functional equation we then obtain a meromorphic continuation of $I_{\infty}^{(1)}(s)$ to the whole complex plane.
   
   Let $P$ be a standard parabolic subgroup of $G$ of type $(n_1,n_2,\cdots,n_r).$ Let $\mathfrak{X}_P$ be the subset of cuspidal data $\chi=\{(M,\sigma)\}$ such that $M=M_P=\diag(M_1,M_2,\cdots,M_r),$ where $M_i$ is $n_i$ by $n_i$ matrix, $1\leq i\leq r.$ We may write $\sigma=(\sigma_1,\sigma_2,\cdots,\sigma_r),$ where $\sigma_i\in\mathcal{A}_0(M_i(F)\backslash M_i(\mathbb{A}_F)).$ Let $\pi$ be a representation induced from $\chi=\{(M,\sigma)\}.$ 
   
   For any $\boldsymbol{\lambda}=(\lambda_1,\lambda_2,\cdots,\lambda_r)\in i\mathfrak{a}_P^*/i\mathfrak{a}_G^*\simeq (i\mathbb{R})^{r-1},$ satisfying that $\lambda_1+\lambda_2+\cdots+\lambda_r=0,$ we let $\boldsymbol{\kappa}=(\kappa_1,\kappa_2,\cdots,\kappa_{r})\in \mathbb{C}^{r-1}$ be such that 
   \begin{equation}\label{145}
   \begin{cases}
   \kappa_j=\lambda_j-\lambda_{j+1},\ 1\leq j\leq r-1,\\
   \kappa_r=\lambda_1-\lambda_r=\kappa_1+\kappa_2+\cdots+\kappa_{r-1}.
   \end{cases}
   \end{equation}
   Then we have a bijection $i\mathfrak{a}_P^*/i\mathfrak{a}_G^*\xleftrightarrow[]{1:1}i\mathfrak{a}_P^*/i\mathfrak{a}_G^*,$ $\boldsymbol{\lambda}\mapsto\boldsymbol{\kappa}$ given by \eqref{145}, which induces a change of coordinates with $d\boldsymbol{\lambda}=m_Pd\boldsymbol{\kappa},$ where $m_P$ is an absolute constant (the determinant of the transform \eqref{145}). So that we can write $\boldsymbol{\lambda}=\boldsymbol{\lambda}(\boldsymbol{\kappa}).$ Let $R_{\varphi}(s,\boldsymbol{\lambda};\phi)$ be defined by \eqref{47''} and $\Lambda(s,\pi_{\lambda}\otimes\tau\times\widetilde{\pi}_{-\lambda})$ be the complete $L$-function. Then we can write $R_{\varphi}(s,\boldsymbol{\lambda};\phi)=R_{\varphi}(s,\boldsymbol{\kappa};\phi)$ and $\Lambda(s,\pi_{\lambda}\otimes\tau\times\widetilde{\pi}_{-\lambda})=\Lambda(s,\pi_{\boldsymbol{\kappa}}\otimes\tau\times\widetilde{\pi}_{-\boldsymbol{\kappa}}).$ Recall that if $v\in \Sigma_{F,fin}$ is a finite place such that $\pi_v$ is unramified and $\Phi_v=\Phi_v^{\circ}$ is the characteristic function of $G(\mathcal{O}_{F,v}).$ Assume further that $\phi_{1,v}=\phi_{2,v}=\phi_v^{\circ}$ be the unique $G(\mathcal{O}_{F,v})$-fixed vector in the space of $\pi_v$ such that $\phi_v^0(e)=1.$ Then $R_v(s,W_{1,v},W_{2,v};\boldsymbol{\lambda})=R_v(s,W_{1,v},W_{2,v};\boldsymbol{\kappa})$ is equal to \eqref{102}, which is, in the $\boldsymbol{\kappa}$-coordinate, that  
   \begin{equation}\label{146}
   \prod_{1\leq i<r}\prod_{i< j\leq r}L_v(1+\boldsymbol{\kappa}_{i,j},\sigma_{i,v}\times\widetilde{\sigma}_{j,v})^{-1}\cdot L_v(1-\boldsymbol{\kappa}_{i,j},\widetilde{\sigma}_{i,v}\times\sigma_{j,v})^{-1},
   \end{equation}
   where $\boldsymbol{\kappa}_{i,j}=\boldsymbol{\kappa}_{i}+\cdots+\boldsymbol{\kappa}_{j-1}.$ By the $K$-finiteness of $\varphi,$ there exists a finite set $S_{\varphi,\tau,\Phi}$ of nonarchimedean places such that for any $\pi$ from some cuspidal datum $\chi\in\mathfrak{X}_P,$ $R_v(s,W_{1,v},W_{2,v};\boldsymbol{\kappa})$ is equal to the formula in \eqref{146}. Then according to Proposition \ref{43'} and Proposition \ref{45'} we see that, when $\Re(s)>0,$ $R_v(s,W_{1,v},W_{2,v};\boldsymbol{\kappa})$ are independent of $s$ for all but finitely many places $v.$ Therefore, as a function of $s,$ $R_{\varphi}(s,\boldsymbol{\kappa};\phi)$ is a finite product of holomorphic function in $\Re(s)>0;$ for any given $s$ such that $\Re(s)>0,$ as a complex function of multiple variables with respect to $\boldsymbol{\kappa},$ $R_{\varphi}(s,\boldsymbol{\kappa};\phi)$ has the property that $R_{\varphi}(s,\boldsymbol{\kappa};\phi)L_{S}(\boldsymbol{\kappa},\pi,\widetilde{\pi})$ is holomorphic, where $L_S(\boldsymbol{\kappa},\pi,\widetilde{\pi})$ is denoted by the meromorphic function 
   \begin{align*}
   \prod_{1\leq i<r}\prod_{i< j\leq r}\prod_{v\in S_{\varphi,\tau,\Phi}}L_v(1+\boldsymbol{\kappa}_{i,j},\sigma_{i,v}\times\widetilde{\sigma}_{j,v})\cdot L_v(1-\boldsymbol{\kappa}_{i,j},\widetilde{\sigma}_{i,v}\times\sigma_{j,v}).
   \end{align*}
   Hence $R_{\varphi}(s,\boldsymbol{\kappa};\phi)$ is holomorphic in some domain $\mathcal{D}$ if $L_{S}(\boldsymbol{\kappa},\pi,\widetilde{\pi})$ is nonvanishing in $\mathcal{D}.$ Now we are picking up such a zero-free region $\mathcal{D}$ explicitly. 
   
   Let $1\leq m,m'\leq n$ be two integers. Let $\sigma\in\mathcal{A}_0(GL_m(F)\backslash GL_{m}(\mathbb{A}_F))$ and $\sigma'\in\mathcal{A}_0(GL_{m'}(F)\backslash GL_{m'}(\mathbb{A}_F)).$ Fix $\epsilon_0>0.$ For any $c'>0,$ let $\mathcal{D}_{c'}(\sigma,\sigma')$ be 
   \begin{equation}\label{A}
   \Bigg\{\kappa=\beta+i\gamma:\ \beta\geq 1-c'\cdot\Big[\frac{(C(\sigma)C(\sigma'))^{-2(m+m')}}{(|\gamma|+3)^{2mm'[F:\mathbb{Q}]}}\Big]^{\frac12+\frac{1}{2(m+m')}-\epsilon_0}\Bigg\},
   \end{equation}
   if $\sigma'\ncong\widetilde{\sigma};$ and let $\mathcal{D}_{c'}(\sigma,\sigma')$ denote by the region 
   \begin{equation}\label{B}
   \Bigg\{\kappa=\beta+i\gamma:\ \beta\geq 1-c'\cdot\Big[\frac{(C(\sigma))^{-8m}}{(|\gamma|+3)^{2mm^2[F:\mathbb{Q}]}}\Big]^{-\frac78+\frac{5}{8m}-\epsilon_0}\Bigg\},
   \end{equation}
   if $\sigma'\simeq \widetilde{\sigma}.$ According to \cite{Bru06} and the Appendix of \cite{Lap13}, there exists a constant $c_{m,m'}>0$ depending only on $m$ and $m',$ such that $L(\boldsymbol{\kappa},\sigma\times\sigma')$ does not vanish in $\boldsymbol{\kappa}=(\kappa_1,\cdots,\kappa_{r})\in \mathcal{D}_{c_{m,m'}}(\sigma,\sigma')\times\cdots\times \mathcal{D}_{c_{m,m'}}(\sigma,\sigma').$ Let $c=\min_{1\leq m,m'\leq n}c_{m,m'}$ and  $\mathcal{C}(\sigma,\sigma')$ be the boundary of $\mathcal{D}_{c}(\sigma,\sigma').$ We may assume that $c$ is small such that the curve $\mathcal{C}(\sigma,\sigma')$ lies in the strip $1-1/(n+4)<\Re(\kappa_j)<1,$ $1\leq j\leq r.$ Fix such a $c$ henceforth. Note that by our choice of $c,$ $L(\boldsymbol{\kappa},\sigma\times\sigma')$ is nonvanishing in $\mathcal{D}_c(\sigma,\sigma')\times \cdots\times \mathcal{D}_c(\sigma,\sigma')$ for any $1\leq m,m'\leq n.$ For $v\in S_{\varphi,\tau,\Phi},$ we have that 
   \begin{align*}
   \big|L_v(\boldsymbol{\kappa},\sigma_{v}\times{\sigma'}_{v})^{-1}\big|\leq\prod_{i=1}^r\prod_{j=1}^r\left(1+q_v^{1-\frac1{m^2+1}-\frac1{m'^2+1}}\right)^{n_i+n_j}<\infty,
   \end{align*}
   for any $\boldsymbol{\kappa}$ such that each $\Re(\kappa_j)\geq 0,$ $1\leq j\leq r.$ Let $L_{S}(\boldsymbol{\kappa},\sigma\times\sigma')=L(\boldsymbol{\kappa},\sigma\times\sigma')\prod_{v\in S_{\varphi,\tau,\Phi}}L_v(\boldsymbol{\kappa},\sigma_{v}\times{\sigma'}_{v})^{-1}.$ Then $L_{S}(\boldsymbol{\kappa},\sigma\times\sigma')$ is nonvanishing in $\mathcal{D}_c(\sigma,\sigma')\times \cdots\times \mathcal{D}_c(\sigma,\sigma')$ for any $1\leq m,m'\leq n.$
   
   Let $\chi\in \mathfrak{X}_P$ and $\pi=\Ind_{P(\mathbb{A}_F)}^{G(\mathbb{A}_F)}(\sigma_1,\sigma_2,\cdots,\sigma_r)\in\chi.$ For any $\epsilon\in (0,1]$ we set
   \begin{align*}
   \mathcal{D}_{\chi}(\epsilon)=\bigcap_{1\leq i\leq r}\bigcap_{i<j\leq r}\Big\{\kappa\in\mathbb{C}:\ \Re(\kappa)\geq0,\ 1-\kappa\in \mathcal{D}_{c\epsilon}(\sigma_i,\sigma_j)\Big\}.
   \end{align*}
   Also, for $\epsilon=0,$ we set $\mathcal{D}_{\chi}(\epsilon)=\big\{\kappa\in\mathbb{C}:\ \Re(\kappa)\geq0\big\}.$ Then by the above discussion, as a function of $\boldsymbol{\kappa},$ $L_S(\boldsymbol{\kappa},\pi,\widetilde{\pi})$ is nonzero in the region $\mathcal{D}_{\chi}(\boldsymbol{\epsilon})=\big\{\boldsymbol{\kappa}=(\kappa_1,\cdots,\kappa_r)\in\mathbb{C}^r:\ \kappa_l\in \mathcal{D}_{\chi}(\epsilon_l)\big\},$ where $\boldsymbol{\epsilon}=(\epsilon_1,\cdots,\epsilon_r)\in [0,1]^r.$ We can write $\mathcal{D}_{\chi}(\boldsymbol{\epsilon})$ as a product space $\mathcal{D}_{\chi}(\boldsymbol{\epsilon})=\prod_{l=1}^r\mathcal{D}_{\chi}(\epsilon_l),$ and let $\partial\mathcal{D}_{\chi}(\epsilon_l)$ be the boundary of $\mathcal{D}_{\chi}(\epsilon_l).$ Then when $\epsilon_l>0,$ $\partial\mathcal{D}_{\chi}(\epsilon_l)$ has two connected components and one of which is exactly the imaginary axis. Let $\mathcal{C}_{\chi}(\epsilon_l)$ be the other component, which is a continuous curve, where $0\leq\epsilon_l\leq 1.$ When $\epsilon_l=0,$ let $\mathcal{C}_{\chi}(\epsilon_l)$ be the maginary axis. Set $\mathcal{C}_{\chi}(\boldsymbol{\epsilon})=\mathcal{C}_{\chi}(\epsilon_1)\times \cdots\times \mathcal{C}_{\chi}(\epsilon_{r-1}),$ $0\leq \epsilon_l\leq 1,$ $1\leq l\leq r-1.$
   
   Let $\boldsymbol{\epsilon}=(\epsilon_1,\cdots,\epsilon_{r-1})\in [0,1]^{r-1}.$ Then by the above construction, $R_{\varphi}(s,\boldsymbol{\kappa};\phi)$ is holomorphic in $\mathcal{D}_{\chi}(\boldsymbol{\epsilon}).$ Hence $R_{\varphi}(s,\boldsymbol{\kappa};\phi)\Lambda(s,\pi_{\boldsymbol{\kappa}}\otimes\tau\times\widetilde{\pi}_{-\boldsymbol{\kappa}})$ is holomorphic in $\mathcal{D}_{\chi}(\boldsymbol{\epsilon}).$ Moreover,  $L_S(\boldsymbol{\kappa},\pi,\widetilde{\pi})\neq0$ on $\mathcal{C}_{\chi}(\boldsymbol{\epsilon}),$ for any $\boldsymbol{\epsilon}=(\epsilon_1,\cdots,\epsilon_{r-1})\in [0,1]^{r-1}$ and any cuspidal datum $\chi\in \mathfrak{X}_P.$ Let $\Re(s)>1.$ For any $\phi\in \mathfrak{B}_{P,\chi}$ and $\boldsymbol{\epsilon}=(\epsilon_1,\cdots,\epsilon_{r-1})\in [0,1]^{r-1},$ let
   \begin{align*}
   J_{P,\chi}(s;\phi,\mathcal{C}_{\chi}(\boldsymbol{\epsilon}))=\int_{\mathcal{C}_{\chi}(\boldsymbol{\epsilon})} R_{\varphi}(s,\boldsymbol{\kappa};\phi)\Lambda(s,\pi_{\boldsymbol{\kappa}}\otimes\tau\times\widetilde{\pi}_{-\boldsymbol{\kappa}})d\boldsymbol{\kappa}.
   \end{align*}
   which is well defined because $J_{P,\chi}(s;\phi,\mathcal{C}_{\chi}(\boldsymbol{\epsilon}))=J_{P,\chi}(s;\phi,\mathcal{C}_{\chi}(\boldsymbol{0}))$ by Cauchy integral formula. Therefore, according to Theorem \ref{39'}, 
   \begin{align*}
   \sum_{P}\frac1{c_P}\sum_{\chi\in\mathfrak{X}_P}\sum_{\phi\in \mathfrak{B}_{P,\chi}}\int_{\mathcal{C}_{\chi}(\boldsymbol{\epsilon})} \big|R_{\varphi}(s,\boldsymbol{\kappa};\phi)\Lambda(s,\pi_{\boldsymbol{\kappa}}\otimes\tau\times\widetilde{\pi}_{-\boldsymbol{\kappa}})\big|d\boldsymbol{\kappa}<\infty
   \end{align*}
   for any $\Re(s)>1,$ $\boldsymbol{\epsilon}=(\epsilon_1,\cdots,\epsilon_{r-1})\in [0,1]^{r-1}.$
   
   Let $\boldsymbol{\epsilon}=(\epsilon_1,\cdots,\epsilon_{r-1})\in [0,1]^{r-1}.$ For any $\beta\geq 1/2,$ we denote by 
   \begin{equation}\label{149}
   \mathcal{R}(\beta;\chi,\boldsymbol{\epsilon})=\Big\{s\in 1+\mathcal{D}_{\chi}(\boldsymbol{\epsilon})\Big\}\bigcup \Big\{s\in 1-\mathcal{D}_{\chi}(\boldsymbol{\epsilon})\Big\}.
   \end{equation}
   
   \begin{lemma}\label{51'}
   	Let notation be as before. Let $P\in\mathcal{P}$ and let $\boldsymbol{\epsilon}=(1/n,1/n,\cdots,1/n)\in\mathbb{R}^{n-1}.$ Then for any $s\in \mathcal{R}(1;\chi,\boldsymbol{\epsilon})\setminus \{1\},$ we have 
   	\begin{equation}\label{147}
   	\sum_{\chi\in\mathfrak{X}_P}\sum_{\phi\in \mathfrak{B}_{P,\chi}}\int_{\mathcal{C}_{\chi}(\boldsymbol{\epsilon})} \big|R_{\varphi}(s,\boldsymbol{\kappa};\phi)\Lambda(s,\pi_{\boldsymbol{\kappa}}\otimes\tau\times\widetilde{\pi}_{-\boldsymbol{\kappa}})\big|d\boldsymbol{\kappa}<\infty.
   	\end{equation}
   \end{lemma}
   \begin{proof}
   	We start with a variant of Lemma \ref{25'}:
   	\begin{claim}\label{25''}
   		Let notation be as before. Let 
   		\begin{equation}\label{51}
   		\widehat{\K}_{\infty}(x,y):=\int_{N(F)\backslash N(\mathbb{A}_F)}\int_{N(F)\backslash N(\mathbb{A}_F)}\K_{\infty}(n_1x,n_2y)\theta(n_1)\bar{\theta}(n_2)dn_1dn_2.
   		\end{equation} 
   		Then $\widehat{\K}_{\infty}(x,y)$ is equal to
   		\begin{equation}\label{53''}
   		\sum_{\chi\in\mathfrak{X}}\sum_{P\in \mathcal{P}}\frac{1}{k_P!(2\pi)^{k_P}}\int_{\mathcal{C}_{\chi}(\boldsymbol{\epsilon})}\sum_{\phi\in \mathfrak{B}_{P,\chi}}W(x,\mathcal{I}_P(\lambda,\varphi)\phi,\lambda)\overline{W(y,\phi,\lambda)}d\lambda.
   		\end{equation}
   	\end{claim}
   	Then the proof is similar as that of Theorem \ref{47'} except that Lemma \ref{25'} should be replaced with Claim \ref{25''} and the constant $e_v$ in \eqref{128'} is replaced with $e_v+1.$
   \end{proof}
   \begin{proof}[Proof of Claim \ref{25''}]
   	The main idea of the proof is similar to Lemma \ref{25'}.
   	For any $P\in \mathcal{P}$, let $c_P=k_P!(2\pi)^{k_P}.$ Applying Cauchy's integral formula we see that $\K_{\infty}(x,y)$ is equal to
   	\begin{equation}\label{52''}
   	\sum_{\chi\in\mathfrak{X}}\sum_{P\in \mathcal{P}}\frac{1}{k_P!(2\pi)^{k_P}}\int_{\mathcal{C}_{\chi}(\boldsymbol{\epsilon})}\sum_{\phi\in \mathfrak{B}_{P,\chi}}E(x,\mathcal{I}_P(\lambda,\varphi)\phi,\lambda)\overline{E(y,\phi,\lambda)}d\lambda,
   	\end{equation}
   	the absolute convergence of \eqref{52''} is justified in \cite{Art79} invoking Langlands' work on Eisenstein theory (see \cite{Lan76}).
   	
   	Substitute \eqref{52''} into \eqref{51} to get an at least formal expansion of $\widehat{\K}_{\infty}(x,y),$ which is clearly dominated by the following formal expression
   	\begin{align*}
   	\int_{[N]}\int_{[N]}\sum_{\chi\in\mathfrak{X}}\sum_{P\in \mathcal{P}}\frac{1}{c_P}\bigg|\int_{\mathcal{C}_{\chi}(\boldsymbol{\epsilon})}\sum_{\phi\in \mathfrak{B}_{P,\chi}}E(x,\mathcal{I}_P(\lambda,\varphi)\phi,\lambda)\overline{E(y,\phi,\lambda)}d\lambda\bigg|dn_1dn_2.
   	\end{align*}
   	Denote by $J_G$ the above integral. We will show $J_G$ is finite, hence 
   	\begin{align*}
   	\widehat{\K}_{\infty}(x,y)=\sum_{\chi\in\mathfrak{X}}\sum_{P\in \mathcal{P}}\frac{1}{k_P!(2\pi)^{k_P}}\int_{i\mathfrak{a}^*_P/i\mathfrak{a}^*_G}\sum_{\phi\in \mathfrak{B}_{P,\chi}}W_{\Eis,1}(x;\lambda)\overline{W_{\Eis,2}(y;\lambda)}d\lambda,
   	\end{align*}
   	is well defined. One can write the test function $\varphi$ as a finite linear combination of convolutions $\varphi_1*\varphi_2$ with functions $\varphi_i\in C_c^r\left(G(\mathbb{A}_F)\right),$ whose archimedean components are differentiable of arbitrarily high order $r.$ Then one applies H\"older inequality to it. Clearly it is enough to deal with the special case that $\varphi=\varphi_j*\varphi_j^*,$ where $\varphi_j^*(x)=\overline{\varphi_j(x^{-1})},$ and $x=y.$ Note that $\mathfrak{B}_{P,\chi}$ is finite due to the $K$-finiteness assumption, and Eisenstein series converge absolutely for our $\lambda,$ hence the integrand 
   	\begin{align*}
   	\sum_{\phi\in \mathfrak{B}_{P,\chi}}E(x,\mathcal{I}_P(\lambda,\varphi)\phi,\lambda)\overline{E(x,\phi,\lambda)}=\sum_{\phi\in \mathfrak{B}_{P,\chi}}E(x,\mathcal{I}_P(\lambda,\varphi_j)\phi,\lambda)\overline{E(x,\mathcal{I}_P(\lambda,\varphi_j)\phi,\lambda)}
   	\end{align*}
   	is well defined and obviously nonnegative. In fact, the double integral over $\lambda$ and $\phi$ can be expressed as an increasing limit of nonnegative functions, each of which is the kernel of the restriction of $R(\varphi_j*\varphi_j^*),$ a positive semidefinite operator, to an invariant subspace. Since this limit is bounded by the nonnegative function 
   	$$
   	\K_j(x,x)=\sum_{\gamma\in Z(F)\backslash G(F)}\varphi_j*\varphi_j^*(x^{-1}\gamma x),
   	$$
   	and the domain $[N]=N(F)\backslash N(\mathbb{A}_F)$ is compact, the integral $J_G$ converges.
   \end{proof}
   
   \subsubsection{Meromorphic continuation of $J_{P,\chi}(s;\phi,\mathcal{C}_{\chi}(\boldsymbol{\epsilon}))$ across the critical line $\Re(s)=1$} \label{7.1.1.}
   Let $\boldsymbol{\epsilon}=(1/n,1/n,\cdots,1/n)\in\mathbb{R}^{n-1}$ and $s\in 1+\mathcal{D}_{\chi}(\boldsymbol{\epsilon})$ and $\Re(s)>1.$ Then by \eqref{1} we see that  $R(s,W_1, W_2;\boldsymbol{\kappa},\phi)\Lambda(s,\pi_{\boldsymbol{\kappa}}\otimes\tau\times\widetilde{\pi}_{-\boldsymbol{\kappa}})$ is equal to a holomorphic function multiplying 
   \begin{align*}
   \prod_{k=1}^r\Lambda(s,\sigma_k\otimes\tau\times\widetilde{\sigma}_k)\prod_{j=1}^{r-1}\prod_{i=1}^j\frac{\Lambda(s+\kappa_{i,j},\sigma_i\otimes\tau\times\widetilde{\sigma}_{j+1})\Lambda(s-\kappa_{i,j},\sigma_{j+1}\otimes\tau\times\widetilde{\sigma}_{i})}{\Lambda(1+\kappa_{i,j},\sigma_{i}\times\widetilde{\sigma}_{j+1})\Lambda(1-\kappa_{i,j},\sigma_{j+1}\times\widetilde{\sigma}_{i})}.
   \end{align*}
   Let $\mathcal{G}(\boldsymbol{\kappa};s)=\mathcal{G}(\boldsymbol{\kappa};s,P,\chi)$ denotes the above product. Also, for simplicity, we denote by $\mathcal{F}(\boldsymbol{\kappa};s)=\mathcal{F}(\boldsymbol{\kappa};s,P,\chi)$ the function $R_{\varphi}(s,\boldsymbol{\kappa};\phi)\Lambda(s,\pi_{\boldsymbol{\kappa}}\otimes\tau\times\widetilde{\pi}_{-\boldsymbol{\kappa}})$ if $\chi$ is fixed in the context. Then the Rankin-Selberg theory implies that  $\mathcal{F}(\boldsymbol{\kappa};s)/\mathcal{G}(\boldsymbol{\kappa};s)$ can be continued to an entire function. We will write $\mathcal{C}$ for the boundary $C_{\chi}(1),$ and $(0)$ for the imaginary axis. Then an analysis on the potential poles of $\mathcal{G}(\boldsymbol{\kappa};s)$ leads to an expression for the integral $J_{P,\chi}(s;\phi,\mathcal{C}_{\chi}(\boldsymbol{0}))=J_{P,\chi}(s;\phi,\mathcal{C})-\mathcal{J}_{\chi}(s),$ where
   \begin{align*}
   \mathcal{J}_{\chi}(s)=\sum_{j=1}^{r-1}\sum_{i=1}^j\int_{(0)}\cdots\int_{(0)}d\kappa_{j-1}\cdots d\kappa_1\int_{\mathcal{C}}\cdots\int_{\mathcal{C}}\underset{\kappa_{i,j}=s-1}{\Res}\mathcal{F}(\boldsymbol{\kappa};s)d\kappa_{r-1}\cdots d\kappa_{j+1},
   \end{align*}
   where $\underset{\kappa_{i,j}=s-1}{\Res}\mathcal{F}(\boldsymbol{\kappa};s)$ is not identically vanishing unless $\sigma_i\otimes\tau\simeq \sigma_{j+1},$ in which case one must have $n_i=n_{j+1}.$ Let $S(r)$ be the symmetric group acting on $\{1,2,\cdots,r\}.$ To obtain meromorphic continuation of $J_{P,\chi}(s;\phi,\mathcal{C}_{\chi}(\boldsymbol{\epsilon}))$ to the critical strip $0<\Re(s)<1,$ we start with the following initial step:
   \begin{prop}\label{57prop}
   	Let notation be as before. Let $\chi\in\mathfrak{X}_P.$ Let $\boldsymbol{\epsilon}=(1/n,1/n,\cdots,1/n)$ and $s\in 1+\mathcal{D}_{\chi}(\boldsymbol{\epsilon})$ and $\Re(s)>1.$ Then 
   	\begin{equation}\label{150}
   	\sum_{\phi\in \mathfrak{B}_{P,\chi}}J_{P,\chi}(s;\phi,\mathcal{C}_{\chi}(\boldsymbol{0}))=\sum_{\phi\in \mathfrak{B}_{P,\chi}}J_{P,\chi}(s;\phi,\mathcal{C})-\sum_{\phi\in \mathfrak{B}_{P,\chi}}\mathcal{J}(s;\phi,\mathcal{C}),
   	\end{equation}
   	where $\mathcal{C}=\mathcal{C}_{\chi},$ and the summand $\mathcal{J}(s;\phi,\mathcal{C})$ is defined to be
   	\begin{align*}
   	\sum_{m=1}^{r-1}\underset{\substack{j_m,j_{m-1},\cdots,j_1\\1\leq j_m<\cdots<j_1\leq r-1}}{\sum\cdots\sum}c_{j_1,\cdots,j_m}\int_{\mathcal{C}}\cdots\cdots\int_{\mathcal{C}}\underset{\kappa_{j_m}=s-1}{\Res}\cdots\underset{\kappa_{j_1}=s-1}{\Res}\mathcal{F}(\boldsymbol{\kappa};s)\frac{d\kappa_{r-1}\cdots d\kappa_{1}}{d\kappa_{j_m}\cdots d\kappa_{j_1}},
   	\end{align*}
   	where $c_{j_1,\cdots,j_m}$'s are some explicit integers, and ${d\kappa_{r-1}\cdots d\kappa_{1}}/(d\kappa_{j_m}\cdots d\kappa_{j_1})$ means $d\kappa_{r-1}\cdots \widehat{d\kappa_{j_m}}\cdots \widehat{d\kappa_{j_1}}\cdots d\kappa_{1}$; namely, omitting $d\kappa_{j_m}, \cdots, d\kappa_{j_1}$. Moreover, the terms in \eqref{150} converges absolutely and normally inside $\mathcal{R}(1;\chi,\boldsymbol{\epsilon})\setminus\{1\},$ where $\mathcal{R}(1;\chi,\boldsymbol{\epsilon})$ is defined in \eqref{149}. Hence \eqref{150} gives a meromorphic continuation of the function $\sum_{\phi\in \mathfrak{B}_{P,\chi}}J_{P,\chi}(s;\phi,\mathcal{C}_{\chi}(\boldsymbol{0}))$ to $\mathcal{R}(1;\chi,\boldsymbol{\epsilon}),$ with a potential pole at $s=1.$ 
   \end{prop}
   \begin{proof}
   	For any $1\leq j\leq r-1$ and $1\leq i\leq j,$ if $n_i=n_{j+1},$ we can take the following change of variables to simplify the integral of $\underset{\kappa_{i,j}=s-1}{\Res}\mathcal{F}(\boldsymbol{\kappa};s):$
   	\begin{align*}
   	\begin{cases}
   	\lambda_l'=\lambda_l,\ l\neq i,j;\\
   	\lambda_i'=\lambda_j,\ \lambda_j'=\lambda_i,
   	\end{cases}\ \text{and}\ \ 
   	\begin{cases}
   	\sigma_l'=\sigma_l,\ l\neq i,j;\\
   	\sigma_i'=\sigma_j,\ \sigma_j'=\sigma_i.
   	\end{cases}
   	\end{align*}
   	Let $\kappa_{l}'=\lambda_{l}'-\lambda_{l+1}',$ $1\leq l\leq r-1;$ and $\kappa_{l,m}'=\kappa_{l}'+\cdots +\kappa_{m}',$ $1\leq l\leq m\leq r-1.$ To describe the relation between $\{\kappa_{l}:\ 1\leq l\leq r-1\}$ and $\{\kappa_{l}':\ 1\leq l\leq r-1\},$ we need to consider separately as follows:
   	\begin{itemize}
   		\item[Case 1] If $i=j-1.$ Then a direct computation shows that 
   		\begin{align*}
   		\begin{cases}
   		\kappa_l=\kappa_l',\ 1\leq l\leq r-1, l\neq i-1, i, i+1;\\
   		\kappa_{i-1}=\kappa_{i-1,i}',\ \kappa_i=-\kappa_i',\ \kappa_{i+1}=\kappa_{i,i+1}'.
   		\end{cases}
   		\end{align*}
   		Hence, the domains $\Re(\kappa_l)=0,$ $1\leq l\leq i=j-1$ are equivalent to $\Re(\kappa_l')=0,$ $1\leq l\leq i=j-1.$ Note that $\det\{\partial\kappa_{l}/\partial\kappa_{m}'\}_{1\leq l,m\leq r-1}=-1,$ and $\kappa_{j}'=\kappa_{i,j}=s-1,$ then one has 
   		\begin{align*}
   		&\int_{(0)}\cdots\int_{(0)}d\kappa_{j-1}\cdots d\kappa_1\int_{\mathcal{C}}\cdots\int_{\mathcal{C}}\underset{\kappa_{i,j}=s-1}{\Res}\mathcal{F}(\boldsymbol{\kappa};s)d\kappa_{r-1}\cdots d\kappa_{j+1}\\
   		=&-\int_{(0)}\cdots\int_{(0)}d\kappa_{j-1}'\cdots d\kappa_1'\int_{\mathcal{C}}\cdots\int_{\mathcal{C}}\underset{\kappa_{j}=s-1}{\Res}\mathcal{F}(\boldsymbol{\kappa};s,P,\chi')d\kappa_{r-1}'\cdots d\kappa_{j+1}',
   		\end{align*}
   		where $\chi'$ is the cuspidal datum attached to representations $(\sigma_1',\cdots,\sigma_{r}').$ Hence $\chi'=\chi$ as an equivalent class.
   		\item[Case 2] If $i\leq j-2.$ Then a direct computation leads to that 
   		\begin{align*}
   		\begin{cases}
   		\kappa_l=\kappa_l',\ 1\leq l\leq r-1, l\neq i-1, i, j-1, j;\\
   		\kappa_{i-1}=\kappa_{i-1,j-1}',\ \kappa_i=-\kappa_{i+1,j-1}',\ \kappa_{j-1}=-\kappa_{i,j-2}',\ \kappa_{j}=\kappa_{i,j}'.
   		\end{cases}
   		\end{align*}
   		One can show inductively that the domains $\Re(\kappa_l)=0,$ $1\leq l\leq i=j-1$ are equivalent to $\Re(\kappa_l')=0,$ $1\leq l\leq i=j-1.$ Note that the determinant of transition matrix $\det\{\partial\kappa_{l}/\partial\kappa_{m}'\}_{1\leq l,m\leq r-1}=-1,$ and $\kappa_{j}'=\kappa_{i,j},$ so again
   		\begin{align*}
   		&\int_{(0)}\cdots\int_{(0)}d\kappa_{j-1}\cdots d\kappa_1\int_{\mathcal{C}}\cdots\int_{\mathcal{C}}\underset{\kappa_{i,j}=s-1}{\Res}\mathcal{F}(\boldsymbol{\kappa};s)d\kappa_{r-1}\cdots d\kappa_{j+1}\\
   		=&-\int_{(0)}\cdots\int_{(0)}d\kappa_{j-1}'\cdots d\kappa_1'\int_{\mathcal{C}}\cdots\int_{\mathcal{C}}\underset{\kappa_{j}=s-1}{\Res}\mathcal{F}(\boldsymbol{\kappa};s,P,\chi')d\kappa_{r-1}'\cdots d\kappa_{j+1}'.
   		\end{align*}
   	\end{itemize}
   	While if  $n_i\neq n_{j+1},$ then $\underset{\kappa_{i,j}=s-1}{\Res}\mathcal{F}(\boldsymbol{\kappa};s)=0.$ In all, we have
   	\begin{align*}
   	&\sum_{\phi\in \mathfrak{B}_{P,\chi}}\int_{(0)}\cdots\int_{(0)}d\kappa_{j-1}\cdots d\kappa_1\int_{\mathcal{C}}\cdots\int_{\mathcal{C}}\underset{\kappa_{i,j}=s-1}{\Res}\mathcal{F}(\boldsymbol{\kappa};s)d\kappa_{r-1}\cdots d\kappa_{j+1}\\
   	=&-\sum_{\phi\in \mathfrak{B}_{P,\chi}}\int_{(0)}\cdots\int_{(0)}d\kappa_{j-1}\cdots d\kappa_1\int_{\mathcal{C}}\cdots\int_{\mathcal{C}}\underset{\kappa_{j}=s-1}{\Res}\mathcal{F}(\boldsymbol{\kappa};s)d\kappa_{r-1}\cdots d\kappa_{j+1}.
   	\end{align*}
   	Therefore, we see that $\sum_{\phi\in \mathfrak{B}_{P,\chi}}J_{P,\chi}(s;\phi,\mathcal{C}_{\chi}(\boldsymbol{0}))-\sum_{\phi\in \mathfrak{B}_{P,\chi}}J_{P,\chi}(s;\phi,\mathcal{C})$ equals 
   	\begin{align*}
   	\sum_{\phi\in \mathfrak{B}_{P,\chi}}\sum_{j=1}^{r-1}c_{j}'\int_{(0)}\cdots\int_{(0)}d\kappa_{j-1}\cdots d\kappa_1\int_{\mathcal{C}}\cdots\int_{\mathcal{C}}\underset{\kappa_{j}=s-1}{\Res}\mathcal{F}(\boldsymbol{\kappa};s)d\kappa_{r-1}\cdots d\kappa_{j+1},
   	\end{align*}
   	where $c_j'$'s are some explicit constants, depending only on the type of $P$. Consider 
   	$$
   	\int_{(0)}\cdots\int_{(0)}d\kappa_{j_1-1}\cdots d\kappa_1\int_{\mathcal{C}}\cdots\int_{\mathcal{C}}\underset{\kappa_{j_1}=s-1}{\Res}\mathcal{F}(\boldsymbol{\kappa};s)d\kappa_{r-1}\cdots d\kappa_{j_1+1},\ 1\leq j_1\leq r-1.
   	$$
   	Then by Cauchy integral formula we can write it as the sum of
   	\begin{align*}
   	\int_{\mathcal{C}}\cdots\int_{\mathcal{C}}d\kappa_{j_1-1}\cdots d\kappa_1\int_{\mathcal{C}}\cdots\int_{\mathcal{C}}\underset{\kappa_{j_1}=s-1}{\Res}\mathcal{F}(\boldsymbol{\kappa};s)d\kappa_{r-1}\cdots d\kappa_{j_1+1}\ \text{and}
   	\end{align*}
   	\begin{align*}
   	\sum_{j_2=1}^{j_1-1}\sum_{i_2=1}^{j_2}c_{i_2,j_2}'\int_{(0)}\cdots\int_{(0)}d\kappa_{j_2-1}\cdots d\kappa_1\int_{\mathcal{C}}\cdots\int_{\mathcal{C}}{\Res}\mathcal{F}(\boldsymbol{\kappa};s)\frac{d\kappa_{r-1}\cdots d\kappa_{j_2+1}}{d\kappa_{j_1}},
   	\end{align*}
   	where $c_{i_2,j_2}'$ are some explicit integers depending only on the type of $P$.  ${\Res}\mathcal{F}(\boldsymbol{\kappa};s)=\underset{\substack{ \kappa_{i_2,j_2}=s-1}}{\Res}\underset{\substack{\kappa_{j_1}=s-1}}{\Res}\mathcal{F}(\boldsymbol{\kappa};s).$ 
   	Then one can do the similar analysis to replace $\kappa_{i_2,j_2}=s-1$ with $\kappa_{j_2}=s-1.$ Then by induction (or simply continue this process until $m=r-1$) we obtain the expression \eqref{150}. Recall that by definition 
   	\begin{align*}
   	\mathcal{F}(\boldsymbol{\kappa};s)=\sum_{\phi_1\in \mathfrak{B}_{P,\chi}}\langle\mathcal{I}_P(\lambda,\varphi)\phi_1,\phi_2\rangle\cdot\Psi(s,W_1,W_2;\lambda).
   	\end{align*}
   	Then $\mathcal{F}(\boldsymbol{\kappa};s)$ is a Schwartz function of $\boldsymbol{\kappa}$ by Claim \ref{34}. Hence all the above integrals converge absolutely. Then the proof is completed.
   \end{proof}
   
   Let notation be as in Proposition \ref{57prop}. Denote by $\mathcal{I}_{0}(s;\chi)$ the summand of the first term of the right hand side of \eqref{150}, i.e., 
   \begin{align*}
   \mathcal{I}_{0,\chi}(s)=\sum_{\phi\in \mathfrak{B}_{P,\chi}}J_{P,\chi}(s;\phi,\mathcal{C}), \ \ s\in 1+\mathcal{D}_{\chi}(\boldsymbol{\epsilon}),\  \Re(s)>1.
   \end{align*}
   
   \begin{prop}\label{58prop}
   	Let notation be as before. Let $s\in 1+\mathcal{D}_{\chi}(\boldsymbol{\epsilon})$ and $\Re(s)>1.$  Then 
   	\begin{equation}\label{153}
   	\mathcal{I}_{0,\chi}(s)=\sum_{\phi\in \mathfrak{B}_{P,\chi}}J_{P,\chi}(s;\phi,\mathcal{C_{\chi}(\boldsymbol{0})})+\sum_{\phi\in \mathfrak{B}_{P,\chi}}\mathcal{J}_{\chi}^0(s),
   	\end{equation}
   	where $\mathcal{C}=\mathcal{C}_{\chi};$ and the summand $\mathcal{J}_{\chi}^0(s)$ is defined to be 
   	\begin{align*}
   	\sum_{m=1}^{r-1}\underset{\substack{j_m,j_{m-1},\cdots,j_1\\1\leq j_m<\cdots<j_1\leq r-1}}{\sum\cdots\sum}\tilde{c}_{j_1,\cdots,j_m}\int_{(0)}\cdots\cdots\int_{(0)}\underset{\kappa_{j_m}=1-s}{\Res}\cdots\underset{\kappa_{j_1}=1-s}{\Res}\mathcal{F}(\boldsymbol{\kappa};s)\frac{d\kappa_{r-1}\cdots d\kappa_{1}}{d\kappa_{j_m}\cdots d\kappa_{j_1}},
   	\end{align*}
   	where $\tilde{c}_{j_1,\cdots,j_m}$'s are some explicit integers, depending only on $P;$ and the measure  ${d\kappa_{r-1}\cdots d\kappa_{1}}/(d\kappa_{j_m}\cdots d\kappa_{j_1})$ means $d\kappa_{r-1}\cdots \widehat{d\kappa_{j_m}}\cdots \widehat{d\kappa_{j_1}}\cdots d\kappa_{1}$. Moreover, the terms in \eqref{153} converges absolutely and normally inside any bounded strip.
   \end{prop}
   \begin{proof}
   	The proof is pretty similar to that of Proposition \ref{57prop}. Hence we will omit it.
   \end{proof}

   \subsection{Meromorphic Continuation Inside the Critical Strip}\label{7.22}
   Let $s\in \mathcal{R}(1;\chi,\boldsymbol{\epsilon})$ and $1\leq m\leq r-1.$ Let $j_m, j_{m-1},\cdots,j_1$ be $m$ integers such that $1\leq j_m<\cdots<j_1\leq r-1.$ Consider the summand in the second term of \eqref{150}:
   \begin{align*}
   \mathcal{I}_{m,\chi}(s):=\sum_{\phi\in \mathfrak{B}_{P,\chi}}\int_{\mathcal{C}}\cdots\cdots\int_{\mathcal{C}}\underset{\kappa_{j_m}=s-1}{\Res}\cdots\underset{\kappa_{j_1}=s-1}{\Res}\mathcal{F}(\boldsymbol{\kappa};s)\frac{d\kappa_{r-1}\cdots d\kappa_{1}}{d\kappa_{j_m}\cdots d\kappa_{j_1}}.
   \end{align*}
   Then each $\mathcal{I}_{m,\chi}(s)$ is naturally meromorphic in $\mathcal{R}(1;\chi,\boldsymbol{\epsilon})$ with a possible at $s=1.$
   \begin{thmx}\label{57}
   	Let notation be as before. Let $n\leq 4.$ Let $\chi\in\mathfrak{X}_P.$ Assume that the adjoint L-function $L(s,\sigma,\Ad\otimes \tau)$ is holomorphic inside the strip $\mathcal{S}_{(0,1)}$ for any cuspidal representation  $\sigma\in\mathcal{A}_0\left(GL(k,\mathbb{A}_F)\right),$ and any $k\leq n-1.$ Then for any $0\leq m\leq r-1,$ the function 
   	\begin{align*}
   	\sum_{\phi\in \mathfrak{B}_{P,\chi}}{\mathcal{I}_{m,\chi}(s)},\quad \ s\in \mathcal{R}(1;\chi,\boldsymbol{\epsilon}),
   	\end{align*} 
   	admits a meromorphic continuation to the area $\mathcal{R}(1/2;\tau)^-$, with possible simple poles at $s\in\{1/2, 2/3,\cdots, (n-1)/n,1\},$ where $\mathcal{R}(1/2;\tau)^-$ is defined in \eqref{R}. Moreover, for any $3\leq k\leq n,$ if $L_F((k-1)/k,\tau)=0,$ then $s=(k-1)/k$ is not a pole.
   \end{thmx}
   \begin{remark}
   	In can be seen from the proof that when $n\leq 3,$ we can continue the functions $\sum_{\phi\in \mathfrak{B}_{P,\chi}}{\mathcal{I}_{m,\chi}(s)}$ to $\Re(s)>1/3.$ When $n=4,$ we can only continue $\sum_{\phi\in \mathfrak{B}_{P,\chi}}{\mathcal{I}_{m,\chi}(s)}$ to $\mathcal{R}(1/2;\tau)^-,$ an open set just containing the right half plane $\Re(s)\geq1/2.$ This is because some of its components involve $\Lambda(2s,\tau^2)^{-1}$ as a factor. The key ingredient is that $\mathcal{R}(1/2;\tau)^-$ is uniform with respect to $\chi\in\mathfrak{X}_P.$ 
   \end{remark}
   \begin{remark}
   	We restrict ourselves to the case $n\leq 4$ for the following two reasons. On the one hand, we actually need to assume Dedekind Conjecture of degree $n$ to handle the contribution from geometric side. This conjecture has been confirmed when $n\leq 4,$ so we will get unconditional results if $n\leq 4.$ On the other hand, when $n\geq 5,$ the procedure of meromorphic continuation is even more complicated, since we are lack of a symmetrical description of this process. Thus, we will focus on $n\leq 4$ case in this paper.
   \end{remark}
   Since the case $n=2$ has been done in \cite{GJ78}, we only need to care about the situation where $3\leq n\leq 4.$ To prove Theorem \ref{57} in these cases, we deal with $n=3$ and $n=4$ separately, since we want to give explicit descriptions. 
   \medskip 
   
   Let notation be as before. To simplify our computations below, we shall write, for any $\beta\in\mathbb{R},$ that $\mathcal{R}(\beta)=\mathcal{R}(\beta;\chi,\boldsymbol{\epsilon}),$ $\mathcal{R}(\beta)^-=\mathcal{R}(\beta;\chi,\boldsymbol{\epsilon})\cap\{s:\ \Re(s)<\beta\},$ and $\mathcal{R}(\beta)^+=\mathcal{R}(\beta;\chi,\boldsymbol{\epsilon})\cap\{s:\ \Re(s)>\beta\}.$ Recall also that we use $\mathcal{S}_{(a,b)}$ to denote the strip $a<\Re(s)<b,$ for any $a<b.$
   \subsubsection{Proof of Theorem \ref{57} when $n=3$}\label{7.1.}
   \begin{proof}
   	Let $n=3.$ Then there are two possibilities for $r:$ $r=2$ or $r=3.$ If $r=2,$ then the parabolic subgroup $P$ is maximal, and any associated cuspidal datum is of the form $\chi\simeq(\sigma_1,\sigma_2),$ where $\sigma_1$ is a cuspidal representation of $\GL(2,\mathbb{A}_F)$ and $\sigma_2$ is a Hecke character on $\mathbb{A}_F^{\times}.$ In this case, $\mathcal{F}(\boldsymbol{\kappa},s)$ is equal to an entire function multiplying 
   	\begin{equation}\label{2'}
   	\frac{\Lambda(s+\kappa_{1},\sigma_1\otimes\tau\times\widetilde{\sigma}_{2})\Lambda(s-\kappa_{1},\sigma_{2}\otimes\tau\times\widetilde{\sigma}_{1})}{\Lambda(1+\kappa_{1},\sigma_{1}\times\widetilde{\sigma}_{2})\Lambda(1-\kappa_{1},\sigma_{2}\times\widetilde{\sigma}_{1})}\cdot\prod_{k=1}^2\Lambda(s,\sigma_k\otimes\tau\times\widetilde{\sigma}_k).
   	\end{equation}
   	Since each completed $L$-functions in \eqref{2'} is entire inside $\mathcal{S}_{(0,1)},$ then $\mathcal{F}(\boldsymbol{\kappa},s)$ is holomorphic (after continuation) when $0<\Re(s)<1.$ On the other hand, $\mathcal{F}(\boldsymbol{\kappa},s)$ vanishes when $\Im(\kappa_1)\rightarrow\infty.$ Let $\Re(s)>1.$ By Cauchy integral formula, 
   	\begin{align*}
   	J_{P,\chi}(s;\phi,\mathcal{C}(\boldsymbol{0}))=\sum_{\phi\in \mathfrak{B}_{P,\chi}}\int_{(0)}\mathcal{F}(\boldsymbol{\kappa},s)d\kappa_1=\sum_{\phi\in \mathfrak{B}_{P,\chi}}\int_{\mathcal{C}}\mathcal{F}(\boldsymbol{\kappa},s)d\kappa_1,
   	\end{align*}
   	which gives holomorphic continuation to an area $\Re(s)>1-\epsilon_1,$ for some $\epsilon_1>0.$ Hence we obtain holomorphic continuation of $J_{P,\chi}(s;\phi,\mathcal{C}(\boldsymbol{0}))$ to $\Re(s)>0.$
   	
   	Now we handle the more complicated case where $r=3.$ In this case, cuspidal data $\chi$ correspond to $(\chi_1,\chi_2,\chi_3),$ where $\chi_i$'s are unitary Hecke characters such that $\chi_1\chi_2\chi_3=\omega,$ the fixed central character. Then $\mathcal{F}(\boldsymbol{\kappa},s)$ is equal to
   	\begin{equation}\label{17}
   	\mathcal{H}(s,\boldsymbol{\kappa})\Lambda_F(s,\tau)^3\prod_{j=1}^{2}\prod_{i=1}^j\frac{\Lambda_F(s+\kappa_{i,j},\tau\chi_i\overline{\chi}_{j+1})\Lambda_F(s-\kappa_{i,j},\tau\chi_{j+1}\overline{\chi}_{i})}{\Lambda_F(1+\kappa_{i,j},\chi_i\overline{\chi}_{j+1})\Lambda_F(1-\kappa_{i,j},\chi_{j+1}\overline{\chi}_{i})},
   	\end{equation}
   	where $\mathcal{H}(s,\boldsymbol{\kappa})$ is an entire function and $\Lambda_F(s,\chi')$ is the completed Hecke $L$-function associated to the unitary Hecke character $\chi'$ over $F.$ Let $\sum_{\phi}$ denote the double summation over ${\phi\in \mathfrak{B}_{P,\chi}}.$ Then by Proposition \ref{57prop}, 
   	\begin{align*}
   	J_{P,\chi}(s;\phi,\mathcal{C}(\boldsymbol{0}))=&\sum_{\phi} \int_{\mathcal{C}}\int_{\mathcal{C}}\mathcal{F}(\boldsymbol{\kappa},s)d\kappa_2d\kappa_1-c_{1,2}\sum_{\phi} \underset{\kappa_{1}=s-1}{\Res}\underset{\kappa_{2}=s-1}{\Res}\mathcal{F}(\boldsymbol{\kappa},s)\\
   	&-c_1\sum_{\phi} \int_{\mathcal{C}}\underset{\kappa_{1}=s-1}{\Res}\mathcal{F}(\boldsymbol{\kappa},s)d\kappa_2-c_2\sum_{\phi} \int_{\mathcal{C}}\underset{\kappa_{2}=s-1}{\Res}\mathcal{F}(\boldsymbol{\kappa},s)d\kappa_1,
   	\end{align*}
   	for some integers $c_1,$ $c_2$ and $c_{1,2};$ and $s\in 1+\mathcal{D}(\boldsymbol{\epsilon}).$ Denote by $J_{P,\chi}^{1}(s;\phi,\mathcal{C}(\boldsymbol{0}))$ the right hand side of the above equality. Then $J_{P,\chi}^{1}(s;\phi,\mathcal{C}(\boldsymbol{0}))$ is meromorphic in the domain $s\in \mathcal{R}(1).$ Then we get a meromorphic continuation inside $\mathcal{R}(1)^-$ with a possible pole at $s=1.$ We will handle these integrals respectively.
   	
   	Denote, for any meromorphic functions $A(s)$ and $B(s),$ by $A(s)\sim B(s)$ if there exists some holomorphic function $C(s)$ such that $A(s)=C(s)B(s).$ Then by \eqref{17}, 
   	\begin{align*}
   	&\underset{\kappa_{1}=s-1}{\Res}\mathcal{F}(\boldsymbol{\kappa},s)\sim \frac{\Lambda(s-\kappa_2,\chi_1\overline{\chi}_2\tau)\Lambda(2s-1+\kappa_2,\chi_2\overline{\chi}_1\tau^2)\Lambda(2s-1,\tau^2)\Lambda(s,\tau)^2}{\Lambda(1+\kappa_2,\chi_2\overline{\chi}_1)\Lambda(2-s-\kappa_2,\chi_1\overline{\chi}_2{\tau}^{-1})\Lambda(2-s,\tau^{-1})};\\
   	&\underset{\kappa_{2}=s-1}{\Res}\mathcal{F}(\boldsymbol{\kappa},s)\sim \frac{\Lambda(s-\kappa_1,\chi_2\overline{\chi}_1\tau)\Lambda(2s-1+\kappa_1,\chi_1\overline{\chi}_2\tau^2)\Lambda(2s-1,\tau^2)\Lambda(s,\tau)^2}{\Lambda(1+\kappa_1,\chi_1\overline{\chi}_2)\Lambda(2-s-\kappa_1,\chi_2\overline{\chi}_1{\tau}^{-1})\Lambda(2-s,\tau^{-1})};\\
   	&\underset{\kappa_{1}=s-1}{\Res}\underset{\kappa_{2}=s-1}{\Res}\mathcal{F}(\boldsymbol{\kappa},s)\sim \frac{\Lambda(3s-2,\tau^3)\Lambda(2s-1,\tau^2)\Lambda(s,\tau)}{\Lambda(3-2s,\tau^{-2})\Lambda(2-s,\tau^{-1})}.
   	\end{align*}
   	Hence by Cauchy integral formula we have, for $s\in \mathcal{R}(1)^-,$ that  
   	\begin{equation}\label{152}
   	\int_{\mathcal{C}}\underset{\kappa_{1}=s-1}{\Res}\mathcal{F}(\boldsymbol{\kappa},s)d\kappa_2=\int_{(0)}\underset{\kappa_{1}=s-1}{\Res}\mathcal{F}(\boldsymbol{\kappa},s)d\kappa_2-\underset{\kappa_{2}=2-2s}{\Res}\underset{\kappa_{1}=s-1}{\Res}\mathcal{F}(\boldsymbol{\kappa},s),
   	\end{equation}
   	where the right hand side is holomorphic inside $1/2<\Re(s)<1,$ since
   	\begin{equation}\label{160'}
   	\underset{\kappa_{2}=2-2s}{\Res}\underset{\kappa_{1}=s-1}{\Res}\mathcal{F}(\boldsymbol{\kappa},s)\sim \frac{\Lambda(3s-2,\tau^3)\Lambda(2s-1,\tau^2)\Lambda(s,\tau)}{\Lambda(3-2s,\tau^{-2})\Lambda(2-s,\tau^{-1})}.
   	\end{equation}
   	From \eqref{160'} we see $\int_{\mathcal{C}}{\Res}_{\kappa_{1}=s-1}\mathcal{F}(\boldsymbol{\kappa},s)d\kappa_2$ has a potential pole at $s=2/3$ when $\tau^3=1.$ Likewise, we have the continuation for $\int_{\mathcal{C}}{\Res}_{\kappa_{2}=s-1}\mathcal{F}(\boldsymbol{\kappa},s)d\kappa_1:$
   	\begin{equation}\label{161.}
   	\int_{\mathcal{C}}\underset{\kappa_{2}=s-1}{\Res}\mathcal{F}(\boldsymbol{\kappa},s)d\kappa_1=\int_{(0)}\underset{\kappa_{2}=s-1}{\Res}\mathcal{F}(\boldsymbol{\kappa},s)d\kappa_1-\underset{\kappa_{1}=2-2s}{\Res}\underset{\kappa_{2}=s-1}{\Res}\mathcal{F}(\boldsymbol{\kappa},s),
   	\end{equation}
   	where the right hand side is holomorphic inside $1/2<\Re(s)<1,$ since
   	\begin{equation}\label{162}
   	\underset{\kappa_{1}=2-2s}{\Res}\underset{\kappa_{2}=s-1}{\Res}\mathcal{F}(\boldsymbol{\kappa},s)\sim \frac{\Lambda(3s-2,\tau^3)\Lambda(2s-1,\tau^2)\Lambda(s,\tau)}{\Lambda(3-2s,\tau^{-2})\Lambda(2-s,\tau^{-1})}.
   	\end{equation}
   	From \eqref{162} we see $\int_{\mathcal{C}}{\Res}_{\kappa_{2}=s-1}\mathcal{F}(\boldsymbol{\kappa},s)d\kappa_1$ has a potential pole at $s=2/3$ when $\tau^3=1.$ Now we deal with the remaining term $\int_{\mathcal{C}}\int_{\mathcal{C}}\mathcal{F}(\boldsymbol{\kappa},s)d\kappa_2d\kappa_1.$ By Proposition \ref{58prop}, for $s\in \mathcal{R}(1)^-,$ there are integers $c_1,$ $c_2$ and $c_{1,2},$ such that  
   	\begin{align*}
   	\int_{\mathcal{C}}\int_{\mathcal{C}}\mathcal{F}(\boldsymbol{\kappa},s)d\kappa_2d\kappa_1=&\sum_{\phi}\int_{(0)}\int_{(0)}\mathcal{F}(\boldsymbol{\kappa},s)d\kappa_2d\kappa_1-c_{1,2}'\sum_{\phi}\underset{\kappa_{1}=1-s}{\Res}\underset{\kappa_{2}=1-s}{\Res}\mathcal{F}(\boldsymbol{\kappa},s)\\
   	-&c_1'\sum_{\phi}\int_{(0)}\underset{\kappa_{1}=1-s}{\Res}\mathcal{F}(\boldsymbol{\kappa},s)d\kappa_2-c_2'\sum_{\phi}\int_{(0)}\underset{\kappa_{2}=1-s}{\Res}\mathcal{F}(\boldsymbol{\kappa},s)d\kappa_1.
   	\end{align*}
   	According to \eqref{17}, one can compute the partial residues of $\mathcal{F}(\boldsymbol{\kappa},s):$
   	\begin{align*}
   	&\underset{\kappa_{2}=1-s}{\Res}\mathcal{F}(\boldsymbol{\kappa},s)\sim \frac{\Lambda(s+\kappa_1,\chi_1\overline{\chi}_2\tau)\Lambda(2s-1-\kappa_1,\chi_2\overline{\chi}_1\tau^2)\Lambda(2s-1,\tau^2)\Lambda(s,\tau)^2}{\Lambda(1-\kappa_1,\chi_2\overline{\chi}_1)\Lambda(2-s+\kappa_1,\chi_1\overline{\chi}_2{\tau}^{-1})\Lambda(2-s,\tau^{-1})};\\
   	&\underset{\kappa_{1}=1-s}{\Res}\mathcal{F}(\boldsymbol{\kappa},s)\sim \frac{\Lambda(s+\kappa_2,\chi_2\overline{\chi}_1\tau)\Lambda(2s-1-\kappa_2,\chi_1\overline{\chi}_2\tau^2)\Lambda(2s-1,\tau^2)\Lambda(s,\tau)^2}{\Lambda(1-\kappa_2,\chi_1\overline{\chi}_2)\Lambda(2-s+\kappa_2,\chi_2\overline{\chi}_1{\tau}^{-1})\Lambda(2-s,\tau^{-1})};\\
   	&\underset{\kappa_{1}=1-s}{\Res}\underset{\kappa_{2}=1-s}{\Res}\mathcal{F}(\boldsymbol{\kappa},s)\sim \frac{\Lambda(3s-2,\tau^3)\Lambda(2s-1,\tau^2)\Lambda(s,\tau)}{\Lambda(3-2s,\tau^{-2})\Lambda(2-s,\tau^{-1})}.
   	\end{align*}
   	From the above formulas and combining with the analytic behavior of the function  ${\Res}_{\kappa_{1}=s-1}{\Res}_{\kappa_{2}=s-1}\mathcal{F}(\boldsymbol{\kappa},s)$ we conclude that $\int_{\mathcal{C}}\int_{\mathcal{C}}\mathcal{F}(\boldsymbol{\kappa},s)d\kappa_2d\kappa_1$ admits a meromorphic continuation to $1/2<\Re(s)<1,$ with a possible pole at $s=2/3$ when $\tau^3=1.$ Denote by $J_{P,\chi}^{(1/2,1)}(s;\phi,\mathcal{C}(\boldsymbol{0}))$ this continuation. Now we continue our meromorphic continuation to some open set containing $\Re(s)\geq 1/2.$ Let $s\in \mathcal{R}(1/2)^{+}.$ Then one can plug \eqref{152} and \eqref{161.} into formulas for $\int_{(0)}\int_{(0)}\mathcal{F}(\boldsymbol{\kappa},s)d\kappa_2d\kappa_1$ and $\int_{\mathcal{C}}\int_{\mathcal{C}}\mathcal{F}(\boldsymbol{\kappa},s)d\kappa_2d\kappa_1$ and shift contours to see that $J_{P,\chi}^{(1/2,1)}(s;\phi,\mathcal{C}(\boldsymbol{0}))$ is equal to the sum over and $\phi\in\mathfrak{B}_{P,\chi}$ of 
   	\begin{align*}
   	&\int_{(0)}\int_{(0)}\mathcal{F}(\boldsymbol{\kappa},s)d\kappa_2d\kappa_1-c_{1,2}'\underset{\kappa_{1}=1-s}{\Res}\underset{\kappa_{2}=1-s}{\Res}\mathcal{F}(\boldsymbol{\kappa},s)-c_1'\int_{(0)}\underset{\kappa_{1}=1-s}{\Res}\mathcal{F}(\boldsymbol{\kappa},s)d\kappa_2\\
   	&-c_2'\int_{(0)}\underset{\kappa_{2}=1-s}{\Res}\mathcal{F}(\boldsymbol{\kappa},s)d\kappa_1-c_1\int_{(0)}\underset{\kappa_{1}=s-1}{\Res}\mathcal{F}(\boldsymbol{\kappa},s)d\kappa_2-c_2\int_{(0)}\underset{\kappa_{2}=s-1}{\Res}\mathcal{F}(\boldsymbol{\kappa},s)d\kappa_1\\
   	&+c_1\underset{\kappa_{2}=2-2s}{\Res}\underset{\kappa_{1}=s-1}{\Res}\mathcal{F}(\boldsymbol{\kappa},s)+c_2\underset{\kappa_{1}=2-2s}{\Res}\underset{\kappa_{2}=s-1}{\Res}\mathcal{F}(\boldsymbol{\kappa},s)-c_{1,2}\underset{\kappa_{1}=s-1}{\Res}\underset{\kappa_{2}=s-1}{\Res}\mathcal{F}(\boldsymbol{\kappa},s)\\
   	=&\int_{(0)}\int_{(0)}\mathcal{F}(\boldsymbol{\kappa},s)d\kappa_2d\kappa_1-c_{1,2}'\underset{\kappa_{1}=1-s}{\Res}\underset{\kappa_{2}=1-s}{\Res}\mathcal{F}(\boldsymbol{\kappa},s)-c_1'\int_{\mathcal{C}}\underset{\kappa_{1}=1-s}{\Res}\mathcal{F}(\boldsymbol{\kappa},s)d\kappa_2\\
   	&-c_2'\int_{\mathcal{C}}\underset{\kappa_{2}=1-s}{\Res}\mathcal{F}(\boldsymbol{\kappa},s)d\kappa_1-c_1\int_{\mathcal{C}}\underset{\kappa_{1}=s-1}{\Res}\mathcal{F}(\boldsymbol{\kappa},s)d\kappa_2-c_2\int_{\mathcal{C}}\underset{\kappa_{2}=s-1}{\Res}\mathcal{F}(\boldsymbol{\kappa},s)d\kappa_1\\
   	&+c_1\underset{\kappa_{2}=2-2s}{\Res}\underset{\kappa_{1}=s-1}{\Res}\mathcal{F}(\boldsymbol{\kappa},s)+c_2\underset{\kappa_{1}=2-2s}{\Res}\underset{\kappa_{2}=s-1}{\Res}\mathcal{F}(\boldsymbol{\kappa},s)-c_{1,2}\underset{\kappa_{1}=s-1}{\Res}\underset{\kappa_{2}=s-1}{\Res}\mathcal{F}(\boldsymbol{\kappa},s)\\
   	&+c_1'\underset{\kappa_{2}=2s-1}{\Res}\underset{\kappa_{1}=1-s}{\Res}\mathcal{F}(\boldsymbol{\kappa},s)+c_2'\underset{\kappa_{1}=2s-1}{\Res}\underset{\kappa_{2}=1-s}{\Res}\mathcal{F}(\boldsymbol{\kappa},s),
   	\end{align*}
   	where the right hand side of the above equality has a natural meromorphic continuation to the domain $\mathcal{R}(1/2).$ Denote by $J_{P,\chi}^{1/2}(s;\phi,\mathcal{C}(\boldsymbol{0}))$ the last expression. Note that a direct computation leads to that 
   	\begin{align*}
   	\underset{\kappa_{2}=2s-1}{\Res}\underset{\kappa_{1}=1-s}{\Res}\mathcal{F}(\boldsymbol{\kappa},s)&\sim \frac{\Lambda(3s-1,\tau^3)\Lambda(2s-1,\tau^2)\Lambda(s,\tau)^2}{\Lambda(2-2s,\tau^{-2})\Lambda(2-s,\tau^{-1})\Lambda(1+s,\tau)};\\
   	\underset{\kappa_{1}=2s-1}{\Res}\underset{\kappa_{2}=1-s}{\Res}\mathcal{F}(\boldsymbol{\kappa},s)&\sim \frac{\Lambda(3s-1,\tau^3)\Lambda(2s-1,\tau^2)\Lambda(s,\tau)^2}{\Lambda(2-2s,\tau^{-2})\Lambda(2-s,\tau^{-1})\Lambda(1+s,\tau)}.
   	\end{align*}
   	Also, when $s\in \mathcal{R}(1/2),$ $2-2s$ lies in the zero-free region of $L_F(s,\tau^{-2})$ and $L_{\infty}(2-2s,\tau^{-2})$ is holomorphic (hence nonvanishing), then $\Lambda(2-2s,\tau^{-2})\neq0.$ So the last two terms of $J_{P,\chi}^{1/2}(s;\phi,\mathcal{C}(\boldsymbol{0}))$ is meromorphic in $\mathcal{R}(1/2)$ with a possible simple pole at $s=1/2$ when $\tau^2=1.$ Hence, we have a meromorphic continuation of $J_{P,\chi}(s;\phi,\mathcal{C}(\boldsymbol{0}))=J_{P,\chi}^{1/2}(s;\phi,\mathcal{C}(\boldsymbol{0}))$ to the region $\mathcal{R}(1/2)$ with a possible simple pole at $s=1/2$ when $\tau^2=1.$ 
   	
   	Now consider $J_{P,\chi}^{1/2}(s;\phi,\mathcal{C}(\boldsymbol{0})),$ where $s\in\mathcal{R}(1/2)^-.$ Invoking the analytic behaviors of $\underset{\kappa_{1}=s-1}{\Res}\mathcal{F}(\boldsymbol{\kappa},s),$ $\underset{\kappa_{2}=s-1}{\Res}\mathcal{F}(\boldsymbol{\kappa},s),$ $\underset{\kappa_{1}=1-s}{\Res}\mathcal{F}(\boldsymbol{\kappa},s)$ and $\underset{\kappa_{2}=1-s}{\Res}\mathcal{F}(\boldsymbol{\kappa},s)$ with Cauchy's formula we obtain that 
   	$J_{P,\chi}^{(1/3,1/2)}(s;\phi,\mathcal{C}(\boldsymbol{0}))$ is equal to the sum over and $\phi\in\mathfrak{B}_{P,\chi}$ of 
   	\begin{align*}
   	&\int_{(0)}\int_{(0)}\mathcal{F}(\boldsymbol{\kappa},s)d\kappa_2d\kappa_1-c_{1,2}'\underset{\kappa_{1}=1-s}{\Res}\underset{\kappa_{2}=1-s}{\Res}\mathcal{F}(\boldsymbol{\kappa},s)-c_1'\int_{\mathcal{C}}\underset{\kappa_{1}=1-s}{\Res}\mathcal{F}(\boldsymbol{\kappa},s)d\kappa_2\\
   	&-c_2'\int_{\mathcal{C}}\underset{\kappa_{2}=1-s}{\Res}\mathcal{F}(\boldsymbol{\kappa},s)d\kappa_1-c_1\int_{(0)}\underset{\kappa_{1}=s-1}{\Res}\mathcal{F}(\boldsymbol{\kappa},s)d\kappa_2-c_2\int_{(0)}\underset{\kappa_{2}=s-1}{\Res}\mathcal{F}(\boldsymbol{\kappa},s)d\kappa_1\\
   	&+c_1\underset{\kappa_{2}=2-2s}{\Res}\underset{\kappa_{1}=s-1}{\Res}\mathcal{F}(\boldsymbol{\kappa},s)+c_2\underset{\kappa_{1}=2-2s}{\Res}\underset{\kappa_{2}=s-1}{\Res}\mathcal{F}(\boldsymbol{\kappa},s)-c_{1,2}\underset{\kappa_{1}=s-1}{\Res}\underset{\kappa_{2}=s-1}{\Res}\mathcal{F}(\boldsymbol{\kappa},s)\\
   	&+c_1'\underset{\kappa_{2}=2s-1}{\Res}\underset{\kappa_{1}=1-s}{\Res}\mathcal{F}(\boldsymbol{\kappa},s)+c_2'\underset{\kappa_{1}=2s-1}{\Res}\underset{\kappa_{2}=1-s}{\Res}\mathcal{F}(\boldsymbol{\kappa},s)+c_1\underset{\kappa_{2}=1-2s}{\Res}\underset{\kappa_{1}=s-1}{\Res}\mathcal{F}(\boldsymbol{\kappa},s)\\
   	&+c_2\underset{\kappa_{1}=1-2s}{\Res}\underset{\kappa_{2}=s-1}{\Res}\mathcal{F}(\boldsymbol{\kappa},s).
   	\end{align*}
   	Denote by $J_{P,\chi}^{1/2}(s;\phi,\mathcal{C}(\boldsymbol{0}))$ the last expression. Note that we have
   	\begin{align*}
   	\underset{\kappa_{2}=1-2s}{\Res}\underset{\kappa_{1}=s-1}{\Res}\mathcal{F}(\boldsymbol{\kappa},s)&\sim \frac{\Lambda(3s-1,\tau^3)\Lambda(2s-1,\tau^2)\Lambda(s,\tau)^2}{\Lambda(2-2s,\tau^{-2})\Lambda(2-s,\tau^{-1})\Lambda(1+s,\tau)};\\
   	\underset{\kappa_{1}=1-2s}{\Res}\underset{\kappa_{2}=s-1}{\Res}\mathcal{F}(\boldsymbol{\kappa},s)&\sim \frac{\Lambda(3s-1,\tau^3)\Lambda(2s-1,\tau^2)\Lambda(s,\tau)^2}{\Lambda(2-2s,\tau^{-2})\Lambda(2-s,\tau^{-1})\Lambda(1+s,\tau)}.
   	\end{align*}
   	Note also that the integrals in $J_{P,\chi}^{1/2}(s;\phi,\mathcal{C}(\boldsymbol{0}))$ converges locally normally when $1/3<\Re(s)<1/2.$ Hence $J_{P,\chi}^{1/2}(s;\phi,\mathcal{C}(\boldsymbol{0}))$ has a natural continuation to $1/3<\Re(s)<1/2,$ where it is holomorphic. In all, we obtain the meromorphic continuation of $J_{P,\chi}(s;\phi,\mathcal{C}(\boldsymbol{0}))$ to $\mathcal{S}_{(1/3,1)}\cup\mathcal{R}(1)$ as follows:
   	\begin{align*}
   	J_{P,\chi}(s;\phi,\mathcal{C}(\boldsymbol{0}))=\begin{cases}
   	J_{P,\chi}^{1}(s;\phi,\mathcal{C}(\boldsymbol{0})),\ s\in\mathcal{R}(1);\\
   	J_{P,\chi}^{(1/2,1)}(s;\phi,\mathcal{C}(\boldsymbol{0})),\ s\in  \mathcal{S}_{(1/2,1)};\\
   	J_{P,\chi}^{1/2}(s;\phi,\mathcal{C}(\boldsymbol{0})),\ s\in\mathcal{R}(1/2);\\
   	J_{P,\chi}^{(1/3,1/2)}(s;\phi,\mathcal{C}(\boldsymbol{0})),\ s\in\mathcal{S}(1/3,1/2).
   	\end{cases}
   	\end{align*}
   	Moreover, the continued function $J_{P,\chi}(s;\phi,\mathcal{C}(\boldsymbol{0}))$ is meromorphic inside $\mathcal{R}(1/2)\cup \mathcal{S}_{(1/2,1)},$ with possible simple poles at $s=1/2$ and $s=2/3$ when $\tau^2=1$ and $\tau^3=1,$ respectively.
   \end{proof}
   
   \subsubsection{Proof of Theorem \ref{57} when $n=4$}\label{7.2.}
   The case $n=4$ seems to be much more complicated than $n=3,$ but they share the same underlying idea. The proof is similar, but does not quite follow from $\GL(3)$ case. In fact, the essential difficulty as $n$ increases is the determination of partial residues of each continuation: there are roughly $O(n^2)$ such multiple residues, and there is not likely a simple systematical description of them, so we give a proof by explicitly dealing with all possible cases. Some careful computation and continuation are carried out in the appendix (see Section \ref{app} for details).
   \begin{proof}
   	Let $n=4.$ Then there are three possibilities for $r:$ $r=2$, $r=3$ or $r=4.$ We will deal with these cases separately.
   	\begin{itemize}
   		\item[$r=2:$] In this case, the parabolic subgroup $P$ is of type $(2,2),$ and any associated cuspidal datum is of the form $\chi\simeq(\sigma_1,\sigma_2),$ where $\sigma_1$ and $\sigma_2$ are cuspidal representations of $\GL(2,\mathbb{A}_F).$ In this case, $\mathcal{F}(\boldsymbol{\kappa},s)$ is equal to an entire function multiplying 
   		\begin{equation}\label{163'}
   		\frac{\Lambda(s+\kappa_{1},\sigma_1\otimes\tau\times\widetilde{\sigma}_{2})\Lambda(s-\kappa_{1},\sigma_{2}\otimes\tau\times\widetilde{\sigma}_{1})}{\Lambda(1+\kappa_{1},\sigma_{1}\times\widetilde{\sigma}_{2})\Lambda(1-\kappa_{1},\sigma_{2}\times\widetilde{\sigma}_{1})}\cdot\prod_{k=1}^2\Lambda(s,\sigma_k\otimes\tau\times\widetilde{\sigma}_k).
   		\end{equation}
   		Let $s\in\mathcal{R}(1)^+.$ Since $\mathcal{F}(\boldsymbol{\kappa},s)$ vanishes when $\Im(\kappa_1)\rightarrow\infty,$ then by Cauchy integral formula, we have that 
   		\begin{equation}\label{163.}
   		J_{P,\chi}(s;\phi,\mathcal{C}(\boldsymbol{0}))=\sum_{\phi\in \mathfrak{B}_{P,\chi}}\int_{\mathcal{C}}\mathcal{F}(\boldsymbol{\kappa},s)d\kappa_1-\sum_{\phi\in \mathfrak{B}_{P,\chi}}\underset{\kappa_{1}=s-1}{\Res}\mathcal{F}(\boldsymbol{\kappa},s).
   		\end{equation}
   		The term $\Res_{\kappa_1=s-1}\mathcal{F}(\boldsymbol{\kappa},s)$ is nonvanishing unless $\sigma_1\simeq \sigma_2\otimes \tau.$ Hence 
   		\begin{align*}
   		\underset{\kappa_{1}=s-1}{\Res}\mathcal{F}(\boldsymbol{\kappa},s)\sim \frac{\Lambda(2s-1,\sigma_1\otimes\tau^2\times\widetilde{\sigma}_{1})\Lambda(s,\sigma_1\otimes\tau\times\widetilde{\sigma}_1)}{\Lambda(2-s,\sigma_{1}\otimes \tau^{-1}\times\widetilde{\sigma}_{1})}.
   		\end{align*}
   		So $\Res_{\kappa_1=s-1}\mathcal{F}(\boldsymbol{\kappa},s)$ admits a meromorphic continuation inside the domain $\mathcal{R}(1/2)\cup \mathcal{S}_{(1/2,1)},$ with possible simple poles at $s=1/2.$ Now the right hand side of \eqref{163.} is meromorphic inside $\mathcal{R}(1),$ with a possible pole at $s=1.$ Denote by $J^{1}_{P,\chi}(s;\phi,\mathcal{C}(\boldsymbol{0}))$ the continuation of $J_{P,\chi}(s;\phi,\mathcal{C}(\boldsymbol{0}))$ in $\mathcal{R}(1).$ Apply Cauchy formula again to get
   		\begin{equation}\label{165}
   		\int_{\mathcal{C}}\mathcal{F}(\boldsymbol{\kappa},s)d\kappa_1=\int_{(0)}\mathcal{F}(\boldsymbol{\kappa},s)d\kappa_1+\underset{\kappa_{1}=1-s}{\Res}\mathcal{F}(\boldsymbol{\kappa},s),
   		\end{equation}
   		where $s\in \mathcal{R}(1)^-.$ By \eqref{163'}, $\int_{(0)}\mathcal{F}(\boldsymbol{\kappa},s)d\kappa_1$ is holomorphic inside $\mathcal{S}_{(1/2,1)};$ also, $\Res_{\kappa_1=1-s}\mathcal{F}(\boldsymbol{\kappa},s)$ is nonvanishing unless $\sigma_2\simeq \sigma_1\otimes \tau,$ in which case one has 
   		\begin{align*}
   		\underset{\kappa_{1}=1-s}{\Res}\mathcal{F}(\boldsymbol{\kappa},s)\sim \frac{\Lambda(2s-1,\sigma_2\otimes\tau^2\times\widetilde{\sigma}_{2})\Lambda(s,\sigma_2\otimes\tau\times\widetilde{\sigma}_2)}{\Lambda(2-s,\sigma_{2}\otimes \tau^{-1}\times\widetilde{\sigma}_{2})}.
   		\end{align*}
   		So $\Res_{\kappa_1=1-s}\mathcal{F}(\boldsymbol{\kappa},s)$ admits a meromorphic continuation to $\mathcal{R}(1/2)\cup \mathcal{S}_{(1/2,1)},$ with possible simple poles at $s=1/2.$ Substituting this with \eqref{165} into \eqref{163.} we conclude that $J^{1}_{P,\chi}(s;\phi,\mathcal{C}(\boldsymbol{0}))$ admits a meromorphic continuation to the domain $\mathcal{R}(1/2)\cup \mathcal{S}_{(1/2,1)}.$ Denote by $J^{1/2}_{P,\chi}(s;\phi,\mathcal{C}(\boldsymbol{0}))$ this continuation. Hence we have
   		\begin{align*}
   		J_{P,\chi}(s;\phi,\mathcal{C}(\boldsymbol{0}))=\begin{cases}
   		J_{P,\chi}^{1}(s;\phi,\mathcal{C}(\boldsymbol{0})),\ s\in\mathcal{R}(1);\\
   		J_{P,\chi}^{1/2}(s;\phi,\mathcal{C}(\boldsymbol{0})),\ s\in \mathcal{S}_{(0,1)}.
   		\end{cases}
   		\end{align*}
   		Moreover, by assumption $\Lambda(s,\sigma_2\otimes\tau\times\widetilde{\sigma}_2)L_F(s,\tau)^{-1}$ is holomorphic in $\mathcal{S}_{(0,1)},$ then from the expressions above we see that $J_{P,\chi}(s;\phi,\mathcal{C}(\boldsymbol{0}))L_F(s,\tau)^{-1}$ admits a meromorphic continuation in $s\in \mathcal{S}_{(1/3,1)}$ with a possible simple pole at $s=1/2$ when $\tau^2=1.$
   		\item[$r=3:$] In this case, the parabolic subgroup $P$ is of type $(2,1,1),$ and any associated cuspidal datum is of the form $\chi\simeq(\sigma_1,\chi_2,\chi_3),$ where $\sigma_1$ is a cuspidal representations of $\GL(2,\mathbb{A}_F);$ and $\chi_2,$ $\chi_3$ are unitary Hecke characters on $\mathbb{A}_F^{\times}.$ Since $\Lambda(s,\sigma_1\otimes \tau\times\chi_i)$ is entire, $2\leq i\leq 3,$ then $\mathcal{F}(\boldsymbol{\kappa},s)$ is equal to an entire function $\mathcal{H}(\boldsymbol{\kappa},s)$ multiplying 
   		\begin{equation}\label{166'}
   		\frac{\Lambda(s+\kappa_{2},\chi_2\overline{\chi}_3\tau)\Lambda(s-\kappa_{2},\chi_3\overline{\chi}_2\tau)\Lambda(s,\sigma_1\otimes\tau\times\widetilde{\sigma}_1)\Lambda(s,\tau)^2}{\Lambda(1+\kappa_{2},\chi_2\overline{\chi}_3\tau)\Lambda(1-\kappa_{2},\chi_3\overline{\chi}_2\tau)}.
   		\end{equation}
   		Let $s\in\mathcal{R}(1)^+.$ Since $\mathcal{F}(\boldsymbol{\kappa},s)$ vanishes when $\Im(\kappa_1)\rightarrow\infty,$ then by Cauchy integral formula, we have that $J_{P,\chi}(s;\phi,\mathcal{C}(\boldsymbol{0}))$ is equal to 
   		\begin{equation}\label{167.}
   		\sum_{\phi\in \mathfrak{B}_{P,\chi}}\int_{\mathcal{C}}\int_{\mathcal{C}}\mathcal{F}(\boldsymbol{\kappa},s)d\kappa_1d\kappa_2-\sum_{\phi\in \mathfrak{B}_{P,\chi}}\int_{\mathcal{C}}\underset{\kappa_{2}=s-1}{\Res}\mathcal{F}(\boldsymbol{\kappa},s)d\kappa_1.
   		\end{equation}
   		The term $\Res_{\kappa_2=s-1}\mathcal{F}(\boldsymbol{\kappa},s)$ is nonvanishing unless $\chi_1\simeq \chi_2\otimes \tau.$ Hence 
   		\begin{align*}
   		\underset{\kappa_{2}=s-1}{\Res}\mathcal{F}(\boldsymbol{\kappa},s)=\mathcal{H}(s,\kappa_1) \frac{\Lambda(2s-1,\tau^2)\Lambda(s,\tau)\Lambda(s,\sigma_1\otimes\tau\times\widetilde{\sigma}_1)}{\Lambda(2-s, \tau^{-1})},
   		\end{align*}
   		where $\mathcal{H}(s,\kappa_1)$ is an holomorphic function. So $\Res_{\kappa_2=s-1}\mathcal{F}(\boldsymbol{\kappa},s)$ admits a meromorphic continuation inside the domain $ \mathcal{S}_{(0,1)},$ with possible simple poles at $s=1/2.$ Now the right hand side of \eqref{167.} is meromorphic inside $\mathcal{R}(1),$ with a possible pole at $s=1.$ Denote by $J^{1}_{P,\chi}(s;\phi,\mathcal{C}(\boldsymbol{0}))$ the continuation of $J_{P,\chi}(s;\phi,\mathcal{C}(\boldsymbol{0}))$ in $\mathcal{R}(1).$ Apply Cauchy formula again to get
   		\begin{equation}\label{165'}
   		\int_{\mathcal{C}}\int_{\mathcal{C}}\mathcal{F}(\boldsymbol{\kappa},s)d\kappa_1d\kappa_2=\int_{\mathcal{C}}\int_{(0)}\mathcal{F}(\boldsymbol{\kappa},s)d\kappa_2d\kappa_1+\int_{\mathcal{C}}\underset{\kappa_{2}=1-s}{\Res}\mathcal{F}(\boldsymbol{\kappa},s)d\kappa_1,
   		\end{equation}
   		where $s\in \mathcal{R}(1)^-.$ By \eqref{166'}, $\int_{\mathcal{C}}\int_{(0)}\mathcal{F}(\boldsymbol{\kappa},s)d\kappa_2d\kappa_1$ is holomorphic inside $\mathcal{S}_{(1/3,1)};$ also, $\Res_{\kappa_1=1-s}\mathcal{F}(\boldsymbol{\kappa},s)$ is nonvanishing unless $\sigma_2\simeq \sigma_1\otimes \tau,$ in which case one has 
   		\begin{align*}
   		\underset{\kappa_{2}=1-s}{\Res}\mathcal{F}(\boldsymbol{\kappa},s)\sim \frac{\Lambda(2s-1,\tau^2)\Lambda(s,\tau)\Lambda(s,\sigma_1\otimes\tau\times\widetilde{\sigma}_1)}{\Lambda(2-s, \tau^{-1})}.
   		\end{align*}
   		So $\int_{\mathcal{C}}\Res_{\kappa_2=1-s}\mathcal{F}(\boldsymbol{\kappa},s)d\kappa_1$ admits a meromorphic continuation to $\mathcal{S}_{(1/3,1)},$ with possible simple poles at $s=1/2.$ Substituting this and \eqref{165'} into \eqref{163.} we conclude that $J^{1}_{P,\chi}(s;\phi,\mathcal{C}(\boldsymbol{0}))$ admits a meromorphic continuation to the domain $\mathcal{S}_{(1/3,1)}.$ Denote by $J^{(1/3,1)}_{P,\chi}(s;\phi,\mathcal{C}(\boldsymbol{0}))$ this continuation. Hence invoking the above discussion we have 
   		\begin{align*}
   		J_{P,\chi}(s;\phi,\mathcal{C}(\boldsymbol{0}))=\begin{cases}
   		J_{P,\chi}^{1}(s;\phi,\mathcal{C}(\boldsymbol{0})),\ s\in\mathcal{R}(1);\\
   		J_{P,\chi}^{(1/3,1)}(s;\phi,\mathcal{C}(\boldsymbol{0})),\ s\in\mathcal{S}_{(1/3,1)}.
   		\end{cases}
   		\end{align*}
   		Moreover, by assumption $\Lambda(s,\sigma_2\otimes\tau\times\widetilde{\sigma}_2)L_F(s,\tau)^{-1}$ is holomorphic in $\mathcal{S}_{(0,1)},$ then from the expressions above we see that $J_{P,\chi}(s;\phi,\mathcal{C}(\boldsymbol{0}))L_F(s,\tau)^{-1}$ admits a meromorphic continuation in $s\in\mathcal{S}_{(1/3,1)}$ with a possible simple pole at $s=1/2$ when $\tau^2=1.$
   		
   		\item[$r=4:$] In this case, the parabolic subgroup $P$ is of type $(1,1,1,1),$ and any associated cuspidal datum is of the form $\chi\simeq(\chi_1,\chi_2,\chi_3,\chi_4),$ where $\chi_i$'s are unitary Hecke characters on $\mathbb{A}_F^{\times}$ such that $\chi_1\chi_2\chi_3\chi_4=\omega.$ Then there exists an entire function $\mathcal{H}(s,\boldsymbol{\kappa})$ such that $\mathcal{F}(\boldsymbol{\kappa},s)$ is equal to 
   		\begin{equation}\label{169.}
   		\mathcal{H}(s,\boldsymbol{\kappa})\Lambda_F(s,\tau)^4\prod_{j=1}^{3}\prod_{i=1}^j\frac{\Lambda_F(s+\kappa_{i,j},\tau\chi_i\overline{\chi}_{j+1})\Lambda_F(s-\kappa_{i,j},\tau\chi_{j+1}\overline{\chi}_{i})}{\Lambda_F(1+\kappa_{i,j},\chi_i\overline{\chi}_{j+1})\Lambda_F(1-\kappa_{i,j},\chi_{j+1}\overline{\chi}_{i})},
   		\end{equation}
   		where $\Lambda_F(s,\chi')$ is the completed Hecke $L$-function associated to the unitary Hecke character $\chi'$ over $F.$ Then by Proposition \ref{57prop}, when $s\in\mathcal{R}(1)^+,$ $J_{P,\chi}(s;\phi,\mathcal{C}(\boldsymbol{0}))$ is equal to 
   		\begin{align*}
   		&\sum_{\phi} \int_{\mathcal{C}}\int_{\mathcal{C}}\int_{\mathcal{C}}\mathcal{F}(\boldsymbol{\kappa},s)d\kappa_3d\kappa_2d\kappa_1-c_1\sum_{\phi} \int_{\mathcal{C}}\int_{\mathcal{C}}\underset{\kappa_{1}=s-1}{\Res}\mathcal{F}(\boldsymbol{\kappa},s)d\kappa_3d\kappa_2-\\
   		&c_2\sum_{\phi} \int_{\mathcal{C}}\int_{\mathcal{C}}\underset{\kappa_{2}=s-1}{\Res}\mathcal{F}(\boldsymbol{\kappa},s)d\kappa_3d\kappa_1-c_3\sum_{\phi} \int_{\mathcal{C}}\int_{\mathcal{C}}\underset{\kappa_{3}=s-1}{\Res}\mathcal{F}(\boldsymbol{\kappa},s)d\kappa_2d\kappa_1-\\
   		&c_{1,2}\sum_{\phi}\int_{\mathcal{C}} \underset{\kappa_{1}=s-1}{\Res}\underset{\kappa_{2}=s-1}{\Res}\mathcal{F}(\boldsymbol{\kappa},s)d\kappa_3-c_{1,3}\sum_{\phi}\int_{\mathcal{C}} \underset{\kappa_{1}=s-1}{\Res}\underset{\kappa_{3}=s-1}{\Res}\mathcal{F}(\boldsymbol{\kappa},s)d\kappa_2-\\
   		&c_{2,3}\sum_{\phi} \int_{\mathcal{C}}\underset{\kappa_{2}=s-1}{\Res}\underset{\kappa_{3}=s-1}{\Res}\mathcal{F}(\boldsymbol{\kappa},s)d\kappa_1-c_{1,2,3}\sum_{\phi}\underset{\kappa_{1}=s-1}{\Res}\underset{\kappa_{2}=s-1}{\Res}\underset{\kappa_{3}=s-1}{\Res}\mathcal{F}(\boldsymbol{\kappa},s),
   		\end{align*}
   		where the coefficients $c_1,$ $c_2,$ $c_3,$ $c_{1,2},$ $c_{1,3},$ $c_{2,3}$ and $c_{1,2,3}$ are some absolute integers; and the sum with respect to $\phi$ in taken over $\phi\in \mathfrak{B}_{P,\chi}.$ 
   		
   		Due to the finiteness of $\mathfrak{B}_{P,\chi}$ and rapidly decay of $\mathcal{F}(\boldsymbol{\kappa},s)$ as a function of $\boldsymbol{\kappa}$ (see Claim \ref{34}), each term in the above expression converges absolutely and locally normally. Hence we only need to consider each summand in the above expression. Denote by  $\chi_{ij}=\chi_i\overline{\chi}_j,$ $1\leq i,j\leq 4.$ By \eqref{169.} we have
   		\begin{align*}
   		\underset{\kappa_{3}=s-1}{\Res}\mathcal{F}(\boldsymbol{\kappa},s)\sim& 
   		\frac{\Lambda(s+\kappa_1,\chi_{12}\tau)\Lambda(s-\kappa_1,\chi_{21}\tau)\Lambda(s-\kappa_2,\chi_{32}\tau)\Lambda(s-\kappa_{12},\chi_{31}\tau)}{\Lambda(1+\kappa_1,\chi_{12})\Lambda(1-\kappa_1,\chi_{21})\Lambda(1+\kappa_2,\chi_{23})\Lambda(2-s-\kappa_2,\chi_{32}\tau^{-1})}\times\\
   		&\frac{\Lambda(2s-1+\kappa_2,\chi_{23}\tau^2)\Lambda(2s-1+\kappa_{12},\chi_{13}\tau^2)\Lambda(2s-1,\tau^2)\Lambda(s,\tau)^3}{\Lambda(1+\kappa_{12},\chi_{13})\Lambda(2-s-\kappa_{12},\chi_{31}\tau^{-1})\Lambda(2-s,\tau^{-1})};\\
   		\underset{\kappa_{2}=s-1}{\Res}\mathcal{F}(\boldsymbol{\kappa},s)\sim& 
   		\frac{\Lambda(s-\kappa_1,\chi_{21}\tau)\Lambda(s-\kappa_3,\chi_{43}\tau)\Lambda(s+\kappa_{13},\chi_{14}\tau)\Lambda(s-\kappa_{13},\chi_{41}\tau)}{\Lambda(1+\kappa_1,\chi_{12})\Lambda(1+\kappa_3,\chi_{34})\Lambda(1+\kappa_{13},\chi_{14})\Lambda(2-s-\kappa_1,\chi_{21}\tau^{-1})}\\
   		&\cdot\frac{\Lambda(2s-1+\kappa_3,\chi_{34}\tau^2)\Lambda(2s-1+\kappa_1,\chi_{12}\tau^2)\Lambda(2s-1,\tau^2)\Lambda(s,\tau)^3}{\Lambda(1-\kappa_{13},\chi_{41})\Lambda(2-s-\kappa_{3},\chi_{43}\tau^{-1})\Lambda(2-s,\tau^{-1})};\\
   		\underset{\kappa_{1}=s-1}{\Res}\mathcal{F}(\boldsymbol{\kappa},s)\sim& 
   		\frac{\Lambda(s+\kappa_3,\chi_{34}\tau)\Lambda(s-\kappa_3,\chi_{43}\tau)\Lambda(s-\kappa_2,\chi_{32}\tau)\Lambda(s-\kappa_{23},\chi_{42}\tau)}{\Lambda(1+\kappa_2,\chi_{23})\Lambda(1-\kappa_3,\chi_{43})\Lambda(1+\kappa_3,\chi_{34})\Lambda(2-s-\kappa_2,\chi_{32}\tau^{-1})}\times\\
   		&\frac{\Lambda(2s-1+\kappa_2,\chi_{23}\tau^2)\Lambda(2s-1+\kappa_{23},\chi_{24}\tau^2)\Lambda(2s-1,\tau^2)\Lambda(s,\tau)^3}{\Lambda(1+\kappa_{23},\chi_{24})\Lambda(2-s-\kappa_{23},\chi_{42}\tau^{-1})\Lambda(2-s,\tau^{-1})}.
   		\end{align*}
   		Hence from the above expressions we see that $\underset{\kappa_{2}=s-1}{\Res}\underset{\kappa_{3}=s-1}{\Res}\mathcal{F}(\boldsymbol{\kappa},s)$ is equal to some holomorphic function multiplying
   		\begin{equation}\label{170.}
   		\frac{\Lambda(s-\kappa_1,\chi_{21}\tau)\Lambda(3s-2+\kappa_1,\chi_{12}\tau^3)\Lambda(3s-2,\tau^{3})\Lambda(2s-1,\tau^2)\Lambda(s,\tau)^2}{\Lambda(1+\kappa_1,\chi_{12})\Lambda(3-2s-\kappa_1,\chi_{21}\tau^{-2})\Lambda(3-2s,\tau^{-2})\Lambda(2-s,\tau^{-1})}.
   		\end{equation}
   		Likewise, $\underset{\kappa_{1}=s-1}{\Res}\underset{\kappa_{3}=s-1}{\Res}\mathcal{F}(\boldsymbol{\kappa},s)$ equals some holomorphic function multiplying the product of $\Lambda(2s-1,\tau^2)^2\Lambda(s,\tau)^2\Lambda(2-s,\tau^{-1})^{-2}$ and 
   		\begin{equation}\label{171.}
   		\frac{\Lambda(1-\kappa_2,\chi_{31})\Lambda(s-\kappa_2,\chi_{32}\tau)\Lambda(2s-1+\kappa_2,\chi_{23}\tau^2)\Lambda(3s-2+\kappa_2,\chi_{23}\tau^{3})}{\Lambda(1+\kappa_2,\chi_{23})\Lambda(s+\kappa_2,\chi_{23}\tau)\Lambda(2-s-\kappa_2,\chi_{32}\tau^{-1})\Lambda(3-2s-\kappa_2,\chi_{32}\tau^{-2})}.
   		\end{equation}
   		Also, the function $\underset{\kappa_{1}=s-1}{\Res}\underset{\kappa_{2}=s-1}{\Res}\mathcal{F}(\boldsymbol{\kappa},s)$ is equal to some holomorphic function multiplying the following function
   		\begin{equation}\label{172.}
   		\frac{\Lambda(s-\kappa_3,\chi_{43}\tau)\Lambda(3s-2+\kappa_3,\chi_{34}\tau^3)\Lambda(3s-2,\tau^{3})\Lambda(2s-1,\tau^2)\Lambda(s,\tau)^2}{\Lambda(1+\kappa_3,\chi_{34})\Lambda(3-2s-\kappa_3,\chi_{43}\tau^{-2})\Lambda(3-2s,\tau^{-2})\Lambda(2-s,\tau^{-1})}.
   		\end{equation}
   		Moreover, one can continue the computation to see that 
   		\begin{align*}
   		\underset{\kappa_{1}=s-1}{\Res}\underset{\kappa_{2}=s-1}{\Res}\underset{\kappa_{3}=s-1}{\Res}\mathcal{F}(\boldsymbol{\kappa},s)\sim& 
   		\frac{\Lambda(4s-3,\tau^4)\Lambda(3s-2,\tau^{3})\Lambda(2s-1,\tau^2)\Lambda(s,\tau)}{\Lambda(4-3s,\tau^{-3})\Lambda(3-2s,\tau^{-2})\Lambda(2-s,\tau^{-1})}.
   		\end{align*}
   		Therefore,  we have, from the above expressions, that $J_{P,\chi}(s;\phi,\mathcal{C}(\boldsymbol{0}))$ admits a meromorphic continuation to $s\in \mathcal{S}(1).$ Denote by $J^1_{P,\chi}(s;\phi,\mathcal{C}(\boldsymbol{0}))$ the continuation. Then clearly $J^1_{P,\chi}(s;\phi,\mathcal{C}(\boldsymbol{0}))$ is holomorphic when $s\in\mathcal{R}(1)^-.$
   		
   		Let $s\in\mathcal{R}(1)^-.$ Let $L_F(s,\tau)$ be the finite part of Hecke $L$-function with respect to $\tau.$ Then by Cauchy integral formula we have that 
   		\begin{claim}\label{60claim}
   			$\int_{\mathcal{C}}\int_{\mathcal{C}}\underset{\kappa_{1}=s-1}{\Res}\mathcal{F}(\boldsymbol{\kappa},s)d\kappa_3d\kappa_2$ admits a meromorphic continuation to the domain $\mathcal{S}_{(1/3,\infty)}.$ When restricted to $\mathcal{R}(1/2;\tau)^-\cup\mathcal{S}_{(1/2,1)},$ it only has  possible simple poles at $s=3/4,$ $s=2/3$ and $s=1/2.$ Moreover, if $L_F(3/4,\tau)=0,$ then $s=3/4$ is not a pole; if $L_F(2/3,\tau)=0,$ then $s=2/3$ is not a pole.
   		\end{claim}  
   		\begin{claim}\label{61claim}
   			$\int_{\mathcal{C}}\int_{\mathcal{C}}\underset{\kappa_{2}=s-1}{\Res}\mathcal{F}(\boldsymbol{\kappa},s)d\kappa_3d\kappa_1$ admits a meromorphic continuation to the domain $\mathcal{S}_{(1/3,\infty)}.$ When restricted to $\mathcal{R}(1/2;\tau)^-\cup\mathcal{S}_{(1/2,1)},$ it only has possible simple poles at $s=3/4,$ $s=2/3$ and $s=1/2.$ Moreover, if $L_F(3/4,\tau)=0,$ then $s=3/4$ is not a pole; if $L_F(2/3,\tau)=0,$ then $s=2/3$ is not a pole.
   		\end{claim}  
   		\begin{claim}\label{62claim}
   			$\int_{\mathcal{C}}\int_{\mathcal{C}}\underset{\kappa_{3}=s-1}{\Res}\mathcal{F}(\boldsymbol{\kappa},s)d\kappa_2d\kappa_1$ admits a meromorphic continuation to the domain $\mathcal{S}_{(1/3,\infty)}.$ When restricted to $\mathcal{R}(1/2;\tau)^-\cup\mathcal{S}_{(1/2,1)},$ it only has possible simple poles at $s=3/4,$ $s=2/3$ and $s=1/2.$ Moreover, if $L_F(3/4,\tau)=0,$ then $s=3/4$ is not a pole; if $L_F(2/3,\tau)=0,$ then $s=2/3$ is not a pole.
   		\end{claim}  
   		\begin{claim}\label{63claim}
   			$\int_{\mathcal{C}} \underset{\kappa_{1}=s-1}{\Res}\underset{\kappa_{2}=s-1}{\Res}\mathcal{F}(\boldsymbol{\kappa},s)d\kappa_3$ admits a meromorphic continuation to the domain $\mathcal{S}_{(1/3,\infty)}.$ When restricted to $\mathcal{R}(1/2;\tau)^-\cup\mathcal{S}_{(1/2,1)},$ it only has  possible simple poles at $s=3/4,$ $s=2/3$ and $s=1/2.$ Moreover, if $L_F(3/4,\tau)=0,$ then $s=3/4$ is not a pole; if $L_F(2/3,\tau)=0,$ then $s=2/3$ is not a pole.
   		\end{claim} 
   		\begin{claim}\label{64claim}
   			$\int_{\mathcal{C}} \underset{\kappa_{1}=s-1}{\Res}\underset{\kappa_{3}=s-1}{\Res}\mathcal{F}(\boldsymbol{\kappa},s)d\kappa_2$ admits a meromorphic continuation to the domain $\mathcal{S}_{(1/3,\infty)}.$ When restricted to $\mathcal{R}(1/2;\tau)^-\cup\mathcal{S}_{(1/2,1)},$ it only has  possible simple poles at $s=3/4,$ $s=2/3$ and $s=1/2.$ Moreover, if $L_F(3/4,\tau)=0,$ then $s=3/4$ is not a pole; if $L_F(2/3,\tau)=0,$ then $s=2/3$ is not a pole.
   		\end{claim}   
   		\begin{claim}\label{65claim}
   			$\int_{\mathcal{C}}\underset{\kappa_{2}=s-1}{\Res}\underset{\kappa_{3}=s-1}{\Res}\mathcal{F}(\boldsymbol{\kappa},s)d\kappa_1$ admits a meromorphic continuation to the domain $\mathcal{S}_{(1/3,\infty)}.$ When restricted to $\mathcal{R}(1/2;\tau)^-\cup\mathcal{S}_{(1/2,1)},$ it only has possible simple poles at $s=3/4,$ $s=2/3$ and $s=1/2.$ Moreover, if $L_F(3/4,\tau)=0,$ then $s=3/4$ is not a pole; if $L_F(2/3,\tau)=0,$ then $s=2/3$ is not a pole.
   		\end{claim}  
   		
   		By Proposition \ref{58prop}, for $s\in \mathcal{R}(1)^-,$ there are integers $c'_1,$ $c'_2,$ $c'_3,$ $c'_{1,2},$ $c'_{1,3},$ $c'_{2,3}$ and $c'_{1,2,3},$ such that $\sum_{\phi} \int_{\mathcal{C}}\int_{\mathcal{C}}\int_{\mathcal{C}}\mathcal{F}(\boldsymbol{\kappa},s)d\kappa_3d\kappa_2d\kappa_1$ is equal to 
   		\begin{align*}
   		&\sum_{\phi} \int_{(0)}\int_{(0)}\int_{(0)}\mathcal{F}(\boldsymbol{\kappa},s)d\kappa_3d\kappa_2d\kappa_1-c'_1\sum_{\phi} \int_{(0)}\int_{(0)}\underset{\kappa_{1}=1-s}{\Res}\mathcal{F}(\boldsymbol{\kappa},s)d\kappa_3d\kappa_2-\\
   		&c'_2\sum_{\phi} \int_{(0)}\int_{(0)}\underset{\kappa_{2}=1-s}{\Res}\mathcal{F}(\boldsymbol{\kappa},s)d\kappa_3d\kappa_1-c'_3\sum_{\phi} \int_{(0)}\int_{(0)}\underset{\kappa_{3}=1-s}{\Res}\mathcal{F}(\boldsymbol{\kappa},s)d\kappa_2d\kappa_1-\\
   		&c'_{1,2}\sum_{\phi}\int_{(0)} \underset{\kappa_{1}=1-s}{\Res}\underset{\kappa_{2}=1-s}{\Res}\mathcal{F}(\boldsymbol{\kappa},s)d\kappa_3-c'_{1,3}\sum_{\phi}\int_{(0)} \underset{\kappa_{1}=1-s}{\Res}\underset{\kappa_{3}=1-s}{\Res}\mathcal{F}(\boldsymbol{\kappa},s)d\kappa_2-\\
   		&c'_{2,3}\sum_{\phi} \int_{(0)}\underset{\kappa_{2}=1-s}{\Res}\underset{\kappa_{3}=1-s}{\Res}\mathcal{F}(\boldsymbol{\kappa},s)d\kappa_1-c'_{1,2,3}\sum_{\phi}\underset{\kappa_{1}=1-s}{\Res}\underset{\kappa_{2}=1-s}{\Res}\underset{\kappa_{3}=1-s}{\Res}\mathcal{F}(\boldsymbol{\kappa},s),
   		\end{align*}
   		where the coefficients $c_1',$ $c_2',$ $c_3',$ $c'_{1,2},$ $c'_{1,3},$ $c'_{2,3}$ and $c'_{1,2,3}$ are some absolute integers; and the sum with respect to $\phi$ in taken over $\phi\in \mathfrak{B}_{P,\chi}.$ 
   		
   		Due to the finiteness of $\mathfrak{B}_{P,\chi}$ and rapidly decay of $\mathcal{F}(\boldsymbol{\kappa},s)$ as a function of $\boldsymbol{\kappa}$ (see Claim \ref{34}), each term in the above expression converges absolutely and locally normally. Hence we only need to consider each summand in the above expression. According to \eqref{169.}, we have that 
   		\begin{align*}
   		\underset{\kappa_{3}=1-s}{\Res}\mathcal{F}(\boldsymbol{\kappa},s)\sim& 
   		\frac{\Lambda(s+\kappa_1,\chi_{12}\tau)\Lambda(s-\kappa_1,\chi_{21}\tau)\Lambda(s+\kappa_2,\chi_{23}\tau)\Lambda(s+\kappa_{12},\chi_{13}\tau)}{\Lambda(1+\kappa_1,\chi_{12})\Lambda(1-\kappa_1,\chi_{21})\Lambda(1-\kappa_2,\chi_{32})\Lambda(2-s+\kappa_2,\chi_{23}\tau^{-1})}\times\\
   		&\frac{\Lambda(2s-1-\kappa_2,\chi_{32}\tau^2)\Lambda(2s-1-\kappa_{12},\chi_{31}\tau^2)\Lambda(2s-1,\tau^2)\Lambda(s,\tau)^3}{\Lambda(1-\kappa_{12},\chi_{31})\Lambda(2-s+\kappa_{12},\chi_{13}\tau^{-1})\Lambda(2-s,\tau^{-1})};\\
   		\underset{\kappa_{2}=1-s}{\Res}\mathcal{F}(\boldsymbol{\kappa},s)\sim& 
   		\frac{\Lambda(s+\kappa_1,\chi_{12}\tau)\Lambda(s+\kappa_3,\chi_{34}\tau)\Lambda(s+\kappa_{13},\chi_{14}\tau)\Lambda(s-\kappa_{13},\chi_{41}\tau)}{\Lambda(1-\kappa_1,\chi_{21})\Lambda(1-\kappa_3,\chi_{43})\Lambda(1+\kappa_{13},\chi_{14})\Lambda(2-s+\kappa_1,\chi_{12}\tau^{-1})}\\
   		&\cdot\frac{\Lambda(2s-1-\kappa_3,\chi_{43}\tau^2)\Lambda(2s-1-\kappa_1,\chi_{21}\tau^2)\Lambda(2s-1,\tau^2)\Lambda(s,\tau)^3}{\Lambda(1-\kappa_{13},\chi_{41})\Lambda(2-s+\kappa_{3},\chi_{34}\tau^{-1})\Lambda(2-s,\tau^{-1})};\\
   		\underset{\kappa_{1}=1-s}{\Res}\mathcal{F}(\boldsymbol{\kappa},s)\sim& 
   		\frac{\Lambda(s+\kappa_3,\chi_{34}\tau)\Lambda(s-\kappa_3,\chi_{43}\tau)\Lambda(s+\kappa_2,\chi_{23}\tau)\Lambda(s+\kappa_{23},\chi_{24}\tau)}{\Lambda(1-\kappa_2,\chi_{32})\Lambda(1-\kappa_3,\chi_{43})\Lambda(1+\kappa_3,\chi_{34})\Lambda(2-s+\kappa_2,\chi_{23}\tau^{-1})}\times\\
   		&\frac{\Lambda(2s-1-\kappa_2,\chi_{32}\tau^2)\Lambda(2s-1-\kappa_{23},\chi_{42}\tau^2)\Lambda(2s-1,\tau^2)\Lambda(s,\tau)^3}{\Lambda(1-\kappa_{23},\chi_{42})\Lambda(2-s+\kappa_{23},\chi_{24}\tau^{-1})\Lambda(2-s,\tau^{-1})}.
   		\end{align*}
   		Hence from the above expressions we see that $\underset{\kappa_{2}=1-s}{\Res}\underset{\kappa_{3}=1-s}{\Res}\mathcal{F}(\boldsymbol{\kappa},s)$ is equal to some holomorphic function multiplying
   		\begin{equation}\label{173.}
   		\frac{\Lambda(s+\kappa_1,\chi_{12}\tau)\Lambda(3s-2-\kappa_1,\chi_{21}\tau^3)\Lambda(3s-2,\tau^{3})\Lambda(2s-1,\tau^2)\Lambda(s,\tau)^2}{\Lambda(1-\kappa_1,\chi_{21})\Lambda(3-2s+\kappa_1,\chi_{12}\tau^{-2})\Lambda(3-2s,\tau^{-2})\Lambda(2-s,\tau^{-1})}.
   		\end{equation}
   		Likewise, $\underset{\kappa_{1}=1-s}{\Res}\underset{\kappa_{3}=1-s}{\Res}\mathcal{F}(\boldsymbol{\kappa},s)$ equals some holomorphic function multiplying the product of $\Lambda(2s-1,\tau^2)^2\Lambda(s,\tau)^2\Lambda(2-s,\tau^{-1})^{-2}$ and 
   		\begin{equation}\label{174.}
   		\frac{\Lambda(1+\kappa_2,\chi_{13})\Lambda(s+\kappa_2,\chi_{23}\tau)\Lambda(2s-1-\kappa_2,\chi_{32}\tau^2)\Lambda(3s-2-\kappa_2,\chi_{32}\tau^{3})}{\Lambda(1-\kappa_2,\chi_{32})\Lambda(s-\kappa_2,\chi_{32}\tau)\Lambda(2-s+\kappa_2,\chi_{23}\tau^{-1})\Lambda(3-2s+\kappa_2,\chi_{23}\tau^{-2})}.
   		\end{equation}
   		Also, the function $\underset{\kappa_{1}=1-s}{\Res}\underset{\kappa_{2}=1-s}{\Res}\mathcal{F}(\boldsymbol{\kappa},s)$ is equal to some holomorphic function multiplying the following function
   		\begin{equation}\label{175.}
   		\frac{\Lambda(s+\kappa_3,\chi_{34}\tau)\Lambda(3s-2-\kappa_3,\chi_{43}\tau^3)\Lambda(3s-2,\tau^{3})\Lambda(2s-1,\tau^2)\Lambda(s,\tau)^2}{\Lambda(1-\kappa_3,\chi_{43})\Lambda(3-2s+\kappa_3,\chi_{34}\tau^{-2})\Lambda(3-2s,\tau^{-2})\Lambda(2-s,\tau^{-1})}.
   		\end{equation}
   		Moreover, one can continue the computation to see that 
   		\begin{align*}
   		\underset{\kappa_{1}=1-s}{\Res}\underset{\kappa_{2}=1-s}{\Res}\underset{\kappa_{3}=1-s}{\Res}\mathcal{F}(\boldsymbol{\kappa},s)\sim& 
   		\frac{\Lambda(4s-3,\tau^4)\Lambda(3s-2,\tau^{3})\Lambda(2s-1,\tau^2)\Lambda(s,\tau)}{\Lambda(4-3s,\tau^{-3})\Lambda(3-2s,\tau^{-2})\Lambda(2-s,\tau^{-1})}.
   		\end{align*}
   		Let $s\in\mathcal{R}(1)^-.$ Then by Cauchy integral formula we have that 
   		\begin{claim}\label{66claim}
   			$\int_{(0)}\int_{(0)}\underset{\kappa_{1}=1-s}{\Res}\mathcal{F}(\boldsymbol{\kappa},s)d\kappa_3d\kappa_2$ admits a meromorphic continuation to the domain $\mathcal{S}_{(1/3,\infty)}.$ When restricted to $\mathcal{R}(1/2;\tau)^-\cup\mathcal{S}_{(1/2,1)},$ it only has  possible simple poles at $s=3/4,$ $s=2/3$ and $s=1/2.$ Moreover, if $L_F(3/4,\tau)=0,$ then $s=3/4$ is not a pole; if $L_F(2/3,\tau)=0,$ then $s=2/3$ is not a pole.
   		\end{claim}  
   		\begin{claim}\label{67claim}
   			$\int_{(0)}\int_{(0)}\underset{\kappa_{2}=1-s}{\Res}\mathcal{F}(\boldsymbol{\kappa},s)d\kappa_3d\kappa_1$ admits a meromorphic continuation to the domain $\mathcal{S}_{(1/3,\infty)}.$ When restricted to $\mathcal{R}(1/2;\tau)^-\cup\mathcal{S}_{(1/2,1)},$ it only has  possible simple poles at $s=3/4,$ $s=2/3$ and $s=1/2.$ Moreover, if $L_F(3/4,\tau)=0,$ then $s=3/4$ is not a pole; if $L_F(2/3,\tau)=0,$ then $s=2/3$ is not a pole.
   		\end{claim}  
   		\begin{claim}\label{68claim}
   			$\int_{(0)}\int_{(0)}\underset{\kappa_{3}=1-s}{\Res}\mathcal{F}(\boldsymbol{\kappa},s)d\kappa_2d\kappa_1$ admits a meromorphic continuation to the domain $\mathcal{S}_{(1/3,\infty)}.$ When restricted to $\mathcal{R}(1/2;\tau)^-\cup\mathcal{S}_{(1/2,1)},$ it only has  possible simple poles at $s=3/4,$ $s=2/3$ and $s=1/2.$ Moreover, if $L_F(3/4,\tau)=0,$ then $s=3/4$ is not a pole; if $L_F(2/3,\tau)=0,$ then $s=2/3$ is not a pole.
   		\end{claim}  
   		\begin{claim}\label{69claim}
   			$\int_{(0)} \underset{\kappa_{1}=1-s}{\Res}\underset{\kappa_{2}=1-s}{\Res}\mathcal{F}(\boldsymbol{\kappa},s)d\kappa_3$ admits a meromorphic continuation to the domain $\mathcal{S}_{(1/3,\infty)}.$ When restricted to $\mathcal{R}(1/2;\tau)^-\cup\mathcal{S}_{(1/2,1)},$ it only has  possible simple poles at $s=3/4,$ $s=2/3$ and $s=1/2.$ Moreover, if $L_F(3/4,\tau)=0,$ then $s=3/4$ is not a pole; if $L_F(2/3,\tau)=0,$ then $s=2/3$ is not a pole.
   		\end{claim} 
   		\begin{claim}\label{70claim}
   			$\int_{(0)} \underset{\kappa_{1}=1-s}{\Res}\underset{\kappa_{3}=1-s}{\Res}\mathcal{F}(\boldsymbol{\kappa},s)d\kappa_2$ admits a meromorphic continuation to the domain $\mathcal{S}_{(1/3,\infty)}.$ When restricted to $\mathcal{R}(1/2;\tau)^-\cup\mathcal{S}_{(1/2,1)},$ it only has  possible simple poles at $s=3/4,$ $s=2/3$ and $s=1/2.$ Moreover, if $L_F(3/4,\tau)=0,$ then $s=3/4$ is not a pole; if $L_F(2/3,\tau)=0,$ then $s=2/3$ is not a pole.
   		\end{claim}   
   		\begin{claim}\label{71claim}
   			$\int_{(0)}\underset{\kappa_{2}=1-s}{\Res}\underset{\kappa_{3}=1-s}{\Res}\mathcal{F}(\boldsymbol{\kappa},s)d\kappa_1$ admits a meromorphic continuation to the domain $\mathcal{S}_{(1/3,\infty)}.$ When restricted to $\mathcal{R}(1/2;\tau)^-\cup\mathcal{S}_{(1/2,1)},$ it only has  possible simple poles at $s=3/4,$ $s=2/3$ and $s=1/2.$ Moreover, if $L_F(3/4,\tau)=0,$ then $s=3/4$ is not a pole; if $L_F(2/3,\tau)=0,$ then $s=2/3$ is not a pole.
   		\end{claim}  
   	\end{itemize}
   	The proof of these claims are given in the Appendix \ref{app}. Then Theorem \ref{57} follows.
   \end{proof}

\section{Proof of Theorems in Applications}\label{sec9}
From Theorem \ref{reg ell}, Theorem \ref{aa} and Theorem \ref{39'}  we conclude the first part of Theorem \ref{A.}, obtaining \eqref{m}, namely,
\begin{equation}\label{44}
I_0^{\varphi}(s,\tau)=I_{\Reg}^{\varphi}(s,\tau)+I_{\Con}^{\varphi}(s,\tau)+I_{\Sin}^{\varphi}(s,\tau)+ \sum_{\chi}\int_{(i\mathbb{R})^{n-1}}I_{\chi}^{\varphi}(s,\tau,\lambda)d\lambda.
\end{equation}
Moreover, $I_{\Reg}^{\varphi}(s,\tau)$ and $I_{\Con}^{\varphi}(s,\tau)$ admits a meromorphic to the whole $s$-plane.

Assume $\tau$ is such that $\tau^k\neq 1,$ $1\leq k\leq n.$ Then gathering Theorem \ref{39'} with Theorem \ref{47'} we conclude that $\sum_{\chi}\int_{(i\mathbb{R})^{n-1}}I_{\chi}^{\varphi}(s,\tau,\lambda)d\lambda$ has a meromorphic continuation to $\Re(s)>0.$ Then by functional equation of Eisenstein series, we conclude that $\sum_{\chi}\int_{(i\mathbb{R})^{n-1}}I_{\chi}^{\varphi}(s,\tau,\lambda)d\lambda$ has a meromorphic continuation to the whole $s$-plane.

Also, note that $I_0^{\varphi}(s,\tau)$ admit a meromorphic continuation to the whole $s$-plane. Hence, by \eqref{44}, $I_{\Sin}^{\varphi}(s,\tau)$ can be continued to a meromorphic on $\mathbb{C}.$ This proves Theorem \ref{A.}.
\medskip 

To prove Theorem \ref{D} and Theorem \ref{N}, we need to choose some special test functions, which are introduced in the following Sec. \ref{9.} and Sec. \ref{9..}.

\subsection{Simple Test Functions}\label{9.}
Let $\Sigma=\Sigma_{\infty}\coprod\Sigma_f$ be the set of places of $F,$ where $\Sigma_{\infty}$ denotes the subset of archimedean places of $F$, and $\Sigma_f$ denotes the subset of nonarchimedean places of $F.$

For a place $v\in \Sigma_f,$ we say that a test function $\varphi=\otimes_v\varphi_v\in \mathcal{H}\left(G(\mathbb{A}_F)\right)$ is $discrete$ $at$ $v$ if $\varphi_v$ is supported on the intersection of $G(\mathcal{O}_{F_v})$ and the regular elliptic subset of $G(F_v).$
\begin{defn}[Space of Test Functions]
	Let $\omega$ be a character of $\mathbb{A}_F^{\times}/F^{\times}.$ Let $\mathcal{F}^*(\omega)$ be the set of smooth functions $\varphi=\otimes_v'\varphi_v:$ $G(\mathbb{A}_F) \rightarrow \mathbb{C}$ which is left and right $K$-finite, is discrete at some $v\in\Sigma_{f},$ transforms by the character $\omega$ of $Z_G\left(\mathbb{A}_F\right),$ and has compact support modulo $Z_G\left(\mathbb{A}_F\right).$ Let $\mathcal{F}(\omega)$ be the space spanned linearly by functions in $\mathcal{F}^*(\omega).$ 
\end{defn}
We then take $\varphi\in\mathcal{F}(\omega).$ Then \eqref{44} would be simplified as 
\begin{equation}\label{45}
I_0^{\varphi}(s,\tau)=\frac1n\sum_{[E:F]=n}\sum_{\gamma_E}\int_{G(\mathbb{A}_F)}\varphi(x^{-1}\gamma_E x)\Phi(\eta x)\tau(\det x)|\det x|^sdx,
\end{equation}
where $\eta=(0,\cdots,0,1)\in \mathbb{A}_F^n,$ and $\gamma_E$ runs over $E^{\times}/F^{\times}-\{1\}.$
\medskip

\subsection{Tate's Thesis Revisited: An Extension}\label{9..}
Let $x\in G(\mathbb{A}_F).$ Denote by 
\begin{align*}
T(s,x)=\int_{\mathbb{A}_E^{\times}}\Phi[(0,\cdots, 0,1) tx]\tau(\det t)|\det t|^sd^{\times}t.
\end{align*}

Then by Tate's thesis, $T(s,x)$ is an integral representation for $\Lambda(s,\tau\circ N_{E/F}).$ So there exists a holomorphic function $Q(s)$ such that $T(s,x)=Q(s)\Lambda(s,\tau\circ N_{E/F}).$ Note that $Q(s)$ depends also on $x,$ so we shall write $Q(s,x)$ instead of $Q(s).$ To prove Theorem \ref{D}, we need more delicate description on the analytic behavior of $Q(s,x)$ as a function on $x.$

Let $l\geq 0$ be an integer. Let $b_l(x)$ be the sum of $l$-th power of entries in the $n$-th row of $x.$ Denote by $\eta=(0,\cdots,0,1)\in F^n.$ Let $x\in G(\mathbb{A}_F)$ be elliptic regular. Then according the ramification of $\tau,$ we will choose the test function  $\Phi=\otimes_v'\Phi_v\in \mathcal{S}_0(\mathbb{A}_F^n)$ for each type of local fields.

\begin{enumerate}
	\item[1.] Suppose $E_v\simeq \mathbb{R}.$ If $\tau_v$ is unramified, i.e., $\tau_v\circ N_{E_v/F_v}=|\cdot |_v^{i\mu}$ for some $\mu\in \mathbb{R}.$ We then take $\Phi_v(y_v)=e^{-\pi \sum_{j=1}^ny_{v,j}^2}.$ If If $\tau_v\circ N_{E_v/F_v}$ is ramified, i.e., $\tau_v=\sgn|\cdot |_v^{i\mu}$ for some $\mu\in \mathbb{R}.$ We then take $\Phi_v(y_v)=e^{-\pi \sum_{j=1}^ny_{v,j}^2}\cdot \sum_{j=1}^ny_{v,j}.$ Then local computation shows that 
	\begin{align*}
	\int_{E_v^{\times}}\Phi_v[(0,\cdots, 0,1) t_vx_v]\tau(\det t_v)|\det t_v|^sd^{\times}t_v=\frac{L_v(s,\tau\circ N_{E/F})}{\|\eta x\|_v^s},
	\end{align*}
	if $\tau_v\circ N_{E/F}$ is unramified, where $\|\eta x\|_v=\big[\sum_{j=1}^n|x_{n,j}|_v^2\big]^{1/2}$; and 
	\begin{align*}
	\int_{E_v^{\times}}\Phi_v[(0,\cdots, 0,1) t_vx_v]\tau(\det t_v)|\det t_v|^sd^{\times}t_v=\frac{b_1(x_v)L_v(s,\tau\circ N_{E/F})}{\|\eta x\|_v^{s+1}},
	\end{align*}
	if $\tau_v\circ N_{E_v/F_v}$ is ramified.
	\medskip 
	
	\item[2.] Suppose $E_v\simeq \mathbb{C}.$ Then $\tau_v\circ N_{E_v/F_v}$ is of the form $\chi_{m}:$ $re^{i\theta}\mapsto r^{i\mu}e^{im\theta}$ for some $\mu\in \mathbb{R}$ and some uniquely defined $m\in\mathbb{Z}.$ Let 
	\begin{align*}
	\Phi_v(y_v)=\begin{cases}
	\frac{1}{2\pi}e^{2\pi \sum_{j=1}^ny_{v,j}\overline{y}_{v,j}}\sum_{j=1}^n\overline{y}_{v,j}^m, & \text{if}\ m\geq 0 \\
	\frac{1}{2\pi}e^{2\pi \sum_{j=1}^ny_{v,j}\overline{y}_{v,j}}\sum_{j=1}^ny_{v,j}^{-m}, & \text{otherwise}.
	\end{cases}
	\end{align*}
	
	Then local computation shows that 
	\begin{align*}
	\int_{E_v^{\times}}\Phi_v[(0,\cdots, 0,1) t_vx_v]\tau(\det t_v)|\det t_v|^sd^{\times}t_v=\frac{\overline{b_m(x_v)}L_v(s,\tau\circ N_{E/F})}{\|\eta x\|_v^{s+|m|/2}},
	\end{align*}
	if $m\geq 0;$ and when $m<0,$ we have 
	\begin{align*}
	\int_{E_v^{\times}}\Phi_v[(0,\cdots, 0,1) t_vx_v]\tau(\det t_v)|\det t_v|^sd^{\times}t_v=\frac{b_{-m}(x_v)L_v(s,\tau\circ N_{E/F})}{\|\eta x\|_v^{s+|m|/2}}.
	\end{align*}
	\medskip

	\item[3.] Suppose $E_v$ is nonarchimedean and $\tau_v\circ N_{E_v/F_v}$ is unramified. Let 
	\begin{align*}
	\Phi_v(y_v)=\prod_{j=1}^n\psi_v(y_{v,j})\textbf{1}_{\mathcal{O}_{E,v}}(y_{v,j}),
	\end{align*}
	where $\psi$ is a nontrivial unramified additive character on $\mathcal{O}_{E,v}.$ Then local computation shows that 
	\begin{align*}
	\int_{E_v^{\times}}\Phi_v[(0,\cdots, 0,1) t_vx_v]\tau(\det t_v)|\det t_v|^sd^{\times}t_v=q_v^{c_v(x_v)s}\vol(\mathcal{O}_{E,v}^{\times})L_v(s,\tau\circ N_{E/F}),
	\end{align*}
	where $q_v$ is the cardinality of the residual field of $E$ at $v,$ and 
	\begin{align*}
	c_v(x_v)=\min_{1\leq j\leq n}\nu(x_{v,j}).
	\end{align*}
	Here $\nu$ is denoted by the normalized valuation on the local field $E_v.$
	\medskip

	\item[4.] Suppose $E_v$ is nonarchimedean and $\tau_v\circ N_{E_v/F_v}$ is ramified. Let $\mathfrak{p}^m$ be the conductor of $\tau_v\circ N_{E_v/F_v},$ where $\mathfrak{p}$ is the maximal ideal in $\mathcal{O}_{E,v}.$ Let 
	\begin{align*}
	\Phi_v(y_v)=\prod_{j=1}^n\psi_v(y_{v,j})\textbf{1}_{\mathfrak{p}^{-m}}(y_{v,j}),
	\end{align*}
	where $\psi$ is a nontrivial unramified additive character on $\mathcal{O}_{E,v}.$ Let $\varpi_v$ be the uniformizer. Then 
	\begin{align*}
	\omega_v^{-m}\mathcal{O}_{E,v}-\{0\}=\bigsqcup_{l=-m}^{\infty}\omega_v^l\mathcal{O}_{E,v}^{\times}.
	\end{align*} 
	Let $T_v(s)=\int_{E_v^{\times}}\Phi_v[(0,\cdots, 0,1) t_vx_v]\tau(\det t_v)|\det t_v|^sd^{\times}t_v.$ Then substituting the choice of $\Phi_v(y_v)$ and doing local computation we obtain 
	\begin{align*}
	T_v(s)=&\sum_{l=-c_v(x_v)-m}^{\infty}\int_{\mathcal{O}_{E,v}^{\times}}\psi_{\varpi_v^{c_v(x_v)}}(\varpi_v^lt_v)\tau_v\circ N_{E_v/F_v}(\varpi_v^lt_v)|\varpi_v^lt_v|_v^{ns}d^{\times}t_v\\
	=&\sum_{l=-c_v(x_v)-m}q_v^{-ls}G(\tau_v\circ N_{E_v/F_v},\psi_{\varpi_{v}^{l+c_v(x_v)}}),
	\end{align*}
	where $\psi_{\varpi_{v}^k}(t_v)=\psi(\varpi_{v}^kt_v)$ for any integer $k\in\mathbb{Z};$ and $G(\tau_v\circ N_{E_v/F_v},\psi_{\varpi_{v}^k})$ is the Gauss sum. Hence, by properties of Gauss sums (e.g., see \cite{RV13}), we conclude that 
	\begin{align*}
	T_v(s)=q_v^{(c_v(x_v)+m)s}G(\tau_v\circ N_{E_v/F_v},\psi_{\varpi_{v}^{m}}).
	\end{align*}
\end{enumerate}
\medskip

Let $\Sigma_{F,\tau}$ be the set of complex places $v\in \Sigma_F$ such that $\tau_v\circ N_{E_v/F_v}$ is of the form $\chi_{m_v}:$ $re^{i\theta}\mapsto r^{i\mu}e^{im_v\theta}$ for some $\mu\in \mathbb{R}$ and some uniquely defined $m_v\in\mathbb{Z}.$ 

Then from the above discussion, we conclude that 

\begin{lemma}\label{Tate}
	Let notation be as before. Assume $b_{m_v}(x)>0$ for all $v\in \Sigma_{F,\tau}.$ Then there exists a test function $\Phi\in\mathcal{S}_0(\mathbb{A}_F)$ such that $T(s,x)/\Lambda(s,\tau\circ N_{E/F})$ is positive when $s\in\mathbb{R}.$ 
\end{lemma}

Also, we need to choose some test function $\varphi\in\mathcal{F}(\omega)$ to fit the above Tate integral $T(s,x),$ from which we will finish the proof of Theorem \ref{D}.
\begin{lemma}\label{l}
	Let $C$ be a compact subset of $G(\mathbb{A}_F).$ Let $\gamma\in G(F)$ be elliptic regular. Let $\iota:$ $G_{\gamma}(\mathbb{A}_F)\backslash G(\mathbb{A}_F)\rightarrow G(\mathbb{A}_F)$ given by $x\mapsto x^{-1}\gamma x.$ Then there exists some $\varphi=\otimes'_v\varphi_v\in \mathcal{F}(\omega)$ such that $\varphi_v\geq 0$ at each $v$ and $\supp\varphi\circ \iota\subseteq [C],$ where $[C]$ is the image of $C$ into $G_{\gamma}(\mathbb{A}_F)\backslash G(\mathbb{A}_F).$ 
\end{lemma}
\begin{proof}
	Let $\gamma^C=\{c^{-1}\gamma c:\ c\in C\}.$ Since $C$ is compact, $\gamma^C$ is also compact. Take $\varphi\in \mathcal{F}(\omega)$ such that $\varphi\geq 0$ and $\supp\varphi\subseteq \gamma^C.$ Then for any $x\in Z_G(\mathbb{A}_F)\backslash G(\mathbb{A}_F)$ such that $\varphi\circ\iota(x)\neq0,$ one has $x^{-1}\gamma x\in \gamma^C.$ Hence there exists some $c\in C$ such that $cx^{-1}\in G_{\gamma}(\mathbb{A}_F),$ implying that $x\in G_{\gamma}(\mathbb{A}_F)C.$ Hence $\supp\varphi\circ \iota\subseteq [C].$ 
\end{proof}

Now we pick a suitable compact subset $C=\otimes_vC_v$ in Lemma \ref{l}. Consider the map
\begin{align*}
\sigma:\ G(F)\longrightarrow F^n,\quad \gamma\mapsto (a_{n-1}(\gamma),\cdots, a_1(\gamma), a_0(\gamma)),
\end{align*}
where $a_i(\gamma)$'s are the coefficients of characteristic polynomial $f_{\gamma}$ of $\gamma,$ namely, $f_{\gamma}(t)=t^n+a_{n-1}(\gamma)t^{n-1}+\cdots+a_1(\gamma)t+a_0(\gamma).$
\medskip

Fix a field extension $E/F$ of degree $n.$ Let $\gamma_0\in G(\mathcal{O}_F)$ be such that $F[\gamma_0]^{\times}=E.$ Although such $\gamma_0$'s are not unique, we fix one $\gamma_0.$ Note that when $\gamma$ runs through the elliptic regular set of $G(F),$ the image $\xi(\gamma)$ is discrete in $\mathbb{A}_F^n.$  Take a compact neighborhood $U_{\gamma_0}$ of $\xi(\gamma_0)$ in $\mathbb{A}_{F}^n,$ such that $U_{\gamma_0}$ does not intersect with other $\sigma(\gamma)$ when $\xi(\gamma)\neq \sigma(\gamma_0).$ Then the pullback $\sigma^{-1}(U_{\gamma_0})$ is a compact subset of $G(\mathbb{A}_F).$ 

Let $C_1$ be a compact subset of $G(\mathbb{A}_F)$ such that $\tau\circ\det \mid_{C_1}$ is trivial. Let $\Sigma_{F,\tau}$ be as before, i.e., the set of complex places $v\in \Sigma_F$ such that $\tau_v\circ N_{E_v/F_v}$ is of the form $\chi_{m_v}:$ $re^{i\theta}\mapsto r^{i\mu}e^{im_v\theta}$ for some $\mu\in \mathbb{R}$ and some uniquely defined $m_v\in\mathbb{Z}.$ Let $C_2$ be a compact subset of $G(\mathbb{A}_F)$ such that any $x=\otimes_v'x_v\in C_2$ has the property that $100^{-1}\leq b_{m_v}(\eta x_v)\leq 100,$ $\forall$ $v\in \Sigma_{F,\tau}.$ 
\medskip

\begin{proof}[Proof of Theorem \ref{D}]
	Let $C=\sigma^{-1}(U_{\gamma_0})\cap C_1\cap C_2.$ Then $C$ is compact. By Lemma \ref{l}, there exists a nonnegative $0\not\equiv \widetilde{\varphi}\in\mathcal{F}(\omega)$ such that $\supp\widetilde{\varphi}\circ \iota\subseteq [C].$

	\medskip

	On the other hand, we can choose a pseudo-coefficient at a place $u$ such that it is compatible with $\gamma_0.$ Let $u$ be a place of $F$ such that $u$ splits in $E,$ $\tau_{u}$ is unramified, and $\widetilde{\varphi}_u$ is the characteristic function of $G(\mathcal{O}_{F_{u}})$. Then $\gamma_{0,u}\in G(\mathcal{O}_{F_{u}}).$ Let $\rho$ be a finite dimensional admissible representation of $G(\mathcal{O}_{F_{u}})\leq G(F_u).$ Let $\rho^{\vee}$ be the contragredient of $\rho.$ Denote by $\Theta_{\rho^{\vee}}$ the character of $\rho^{\vee}.$ Since $\gamma_0$ is elliptic regular, we can take such a $\rho$ such that $\Theta_{\rho^{\vee}}(\gamma_{0,u})\neq0$ and the compact induction $\pi_u=c-\Ind_{G(\mathcal{O}_{F_{u}})}^{G(F_u)}\rho$ is irreducible. Hence $\pi_u$ is supercuspidal. Let 
	\begin{align*}
	m_{\pi_u}(x)=\begin{cases}
	\Theta_{\rho^{\vee}}(x),\ &\text{if $x\in G(\mathcal{O}_{F_{u}})$};\\
	0,\ &\text{otherwise}.
	\end{cases}
	\end{align*}

	Now take $\varphi(x)=\otimes_{v\neq u}\widetilde{\varphi}_v\otimes m_{\pi_u}(x).$ Then $\supp\varphi\circ \iota\subseteq [C].$ Thus, if $x\in G_{\gamma_0}(\mathbb{A}_F)\backslash G(\mathbb{A}_F)$ be such that $\varphi(x^{-1}\gamma_0 x )\neq0,$ then $x\in [C].$ Therefore, substituting this choice of $\varphi$ into Proposition \ref{reg} we obtain 
	\begin{equation}\label{37}
	I_{0}^{\varphi}(s,\tau)=\frac{1}{n}\int_{[C]}\varphi(x^{-1}\gamma_0 x)|\det x|^sT(s,x)dx.
	\end{equation}
	
	Applying Lemma \ref{Tate} into \eqref{37} to deduce by Paley-Wiener Theorem, that 
	\begin{equation}\label{39}
	\frac{I_{0}^{\varphi}(s,\tau)}{\Lambda(s,\tau\circ N_{E/F})}=\frac{1}{n}\int_{[C]}\varphi(x^{-1}\gamma_0 x)|\det x|^s\cdot \frac{T(s,x)}{\Lambda(s,\tau\circ N_{E/F})}dx
	\end{equation}
	the meromorphic function ${I_{0}^{\varphi}(s,\tau)}/{\Lambda(s,\tau\circ N_{E/F})}$ is actually holomorphic for all $s\in\mathbb{C}.$ It is moreover nonvanishing everywhere. In fact, assume there exists some $s_0=\sigma_0+it_0\in\mathbb{C},$ $\sigma_0, t_0\in\mathbb{R},$ such that ${I_{0}^{\varphi}(s,\tau)}/{\Lambda(s,\tau\circ N_{E/F})}$ vanishes at $s=s_0.$ Then we have 
	\begin{align*}
	\frac{1}{n}\prod_{v}\int_{[C_v]}\varphi_v(x_v^{-1}\gamma_{0,_v} x_v)|\det x_v|_v^{\sigma_0}\cdot \frac{T_v(\sigma_0,x_v)}{\Lambda_v(\sigma_0,\tau_v\circ N_{E/F})}dx_v=0.
	\end{align*}
	
	However, by Lemma \ref{Tate}, ${T(\sigma_0,x_v)}/{\Lambda(\sigma_0,\tau_v\circ N_{E/F})}\geq 0$ for all $x_v\in [C_v].$ Thus, the function $\varphi_v(x_v^{-1}\gamma x)|\det x_v|_v^{\sigma_0}\cdot {T_v(\sigma_0,x_v)}/{\Lambda_v(\sigma_0,\tau_v\circ N_{E/F})}$ is nonnegative when $x_v\in[C_v]$ and is not identically zero, if $v\neq u$. Moreover, when $v=u,$ 
	\begin{align*}
	&\int_{[C_u]}\varphi_v(x_u^{-1}\gamma_{0,_u} x_u)|\det x_u|_u^{\sigma_0}\cdot \frac{T_u(\sigma_0,x_u)}{\Lambda_u(\sigma_0,\tau_u\circ N_{E/F})}dx_u\\
	=&\Theta_{\rho^{\vee}}(\gamma_{0,u})\int_{[C_u]}|\det x_u|_u^{\sigma_0}\cdot \frac{T_u(\sigma_0,x_u)}{\Lambda_u(\sigma_0,\tau_u\circ N_{E/F})}dx_u\neq 0.
	\end{align*}
	
	Therefore, the integration on the right hand side of \eqref{39} is nonvanishing at $s=\sigma_0$, a contradiction! Thus the holomorphic function ${I_{0}^{\varphi}(s,\tau)}/{\Lambda(s,\tau\circ N_{E/F})}$ is nonvanishing everywhere. 
	\medskip 
	
	Assume that the twisted adjoint $L$-function $L(s,\pi,\Ad\otimes\tau)$ is holomorphic outside $s= 1$ for all $\pi\in\mathcal{A}_{0}^{\dis}(G(F)\backslash G(\mathbb{A}_F),\omega^{-1}).$ Then by spectral expansion, ${I_{0}^{\varphi}(s,\tau)}/{\Lambda(s,\tau)}$ is holomorphic at $s\neq1.$ Thus 
	\begin{align*}
	\frac{\Lambda(s,\tau\circ N_{E/F})}{\Lambda(s,\tau)}\cdot \frac{I_{0}^{\varphi}(s,\tau)}{\Lambda(s,\tau\circ N_{E/F})}
	\end{align*}
	is holomorphic at $s\neq 1.$ However, ${I_{0}^{\varphi}(s,\tau)}/{\Lambda(s,\tau\circ N_{E/F})}$ is nonvanishing everywhere. Therefore, the meromorphic function ${\Lambda(s,\tau\circ N_{E/F})}/{\Lambda(s,\tau)}$ is holomorphic outside $s=1,$ namely, the $\tau$-twisted Dedekind conjecture holds.
	\medskip
	
	For the other direction. Assume the $\tau$-twisted Dedekind conjecture for all field extensions $E/F$ such that $[E:F]=n.$ Then by Theorem \ref{A.}, Theorem \ref{reg ell}, Theorem \ref{aa}, Theorem \ref{39'} and Theorem \ref{47'}, ${I_{0}^{\varphi}(s,\tau)}/\Lambda(s,\tau)$ is holomorphic at $s\neq 1.$ Let $s_0$ be a zero of $\Lambda(s,\tau)$ of order $m.$ Then 
	\begin{align*}
	\int_{G(F)Z(\mathbb{A}_F)\backslash  G(\mathbb{A}_F)}\K_0^{\varphi}(x,x)\cdot \frac{\partial^{m-1} E(s,x)}{\partial s^{m-1}}\mid_{s=s_0}dx=0,
	\end{align*}
	for all $\varphi\in\mathcal{F}(\omega).$  By Rankin-Selberg theory, we obtain 
	\begin{equation}\label{38}
\int_{G(F)Z(\mathbb{A}_F)\backslash  G(\mathbb{A}_F)}\phi_1(x)\overline{\phi_2(x)}\cdot \frac{\partial^{m-1} E(s,x)}{\partial s^{m-1}}\mid_{s=s_0}dx=0, 
	\end{equation}
	for all cuspidal representations $\pi\in \mathcal{A}_0^{\dis}\left(G(F)\setminus G(\mathbb{A}_F),\omega^{-1}\right),$ and all $K$-finite functions $\phi_1, \phi_2\in\V_{\pi}.$
	\medskip
	
	Then Theorem \ref{D} follows from Rankin-Selberg theory by switching the order of taking partial derivative and integration.
\end{proof}

\medskip 
\begin{proof}[Proof of Theorem \ref{N}]
Let $E$ be a field extension of $F$ of degree $n,$ such that $\zeta_E(1/2)\neq0.$ By the proof of Theorem \ref{D}, one can choose some test function $\varphi\in\mathcal{F}(\omega),$ such that the right hand side of \eqref{39} does not vanish. Under the choice of this test function, we obtain $I_{0}^{\varphi}(s,\tau)=\Lambda(s,\tau\circ N_{E/F})\cdot Q(s,\varphi),$ where
\begin{equation}\label{40}
Q(s,\varphi)=\frac{1}{n}\int_{[C]}\varphi(x^{-1}\gamma x)|\det x|^s\cdot \frac{T(s,x)}{\Lambda(s,\tau\circ N_{E/F})}dx
\end{equation}
is nonvanishing everywhere. Hence we can evaluate at $s=1/2$ to obtain  
\begin{equation}\label{41}
I_{0}^{\varphi}(1/2,\tau)=\Lambda(1/2,\tau\circ N_{E/F})\cdot Q(1/2,\varphi)\neq 0.
\end{equation}

Theorem \ref{N} then follows from \eqref{41} and the spectral expansion \eqref{ker_0} of the cuspidal kernel function $\K_0(x,x).$
\end{proof}

   \section{Appendix: Continuation Across the Critical Line for $\GL(4)$}\label{app}
   In this appendix, we shall prove the claims in our proceeding proof of Theorem \ref{57} when $n=4$ in subsection \ref{7.2.}. The processes here are in the same flavor of those in the $n=3$ case in subsection \ref{7.1.}, but they are typically much more complicated. 
   \begin{proof}[Proof of Claim \ref{60claim}]
   	Let $s\in\mathcal{R}(1)^+.$ Let $J_1(s):=\int_{(0)}\int_{(0)}\underset{\kappa_{1}=s-1}{\Res}\mathcal{F}(\boldsymbol{\kappa},s)d\kappa_3d\kappa_2,$ and  $J^1_1(s):=\int_{\mathcal{C}}\int_{\mathcal{C}}\underset{\kappa_{1}=s-1}{\Res}\mathcal{F}(\boldsymbol{\kappa},s)d\kappa_3d\kappa_2.$ By the analytic property of $\underset{\kappa_{1}=s-1}{\Res}\mathcal{F}(\boldsymbol{\kappa},s)$ we see that $J^1_1(s)$ is meromorphic in the domain $\mathcal{R}(1),$ with a possible pole at $s=1.$ Let $s\in\mathcal{R}(1)^-.$ Applying Cauchy integral formula we then see that  
   	\begin{equation}\label{175"}
   	J_1^1(s)=\int_{\mathcal{C}}\int_{(0)}\underset{\kappa_{1}=s-1}{\Res}\mathcal{F}(\boldsymbol{\kappa},s)d\kappa_2d\kappa_3+\int_{\mathcal{C}}\underset{\kappa_{2}=2-2s}{\Res}\underset{\kappa_{1}=s-1}{\Res}\mathcal{F}(\boldsymbol{\kappa},s)d\kappa_3,
   	\end{equation}
   	where $\underset{\kappa_{2}=2-2s}{\Res}\underset{\kappa_{1}=s-1}{\Res}\mathcal{F}(\boldsymbol{\kappa},s)$ is equal to some holomorphic function multiplying
   	\begin{equation}\label{175_}
   	\frac{\Lambda(s+\kappa_3,\chi_{34}\tau)\Lambda(3s-2-\kappa_3,\chi_{43}\tau^3)\Lambda(3s-2,\tau^{3})\Lambda(2s-1,\tau^2)\Lambda(s,\tau)^2}{\Lambda(1-\kappa_3,\chi_{43})\Lambda(3-2s+\kappa_3,\chi_{34}\tau^{-2})\Lambda(3-2s,\tau^{-2})\Lambda(2-s,\tau^{-1})}.
   	\end{equation}
   	Then $J_1^1(s)$ is equal to, after applications of Cauchy integral formula to \eqref{175"},
   	\begin{align*}
   	&\int_{(0)}\int_{(0)}\underset{\kappa_{1}=s-1}{\Res}\mathcal{F}(\boldsymbol{\kappa},s)d\kappa_2d\kappa_3+\int_{(0)}\underset{\kappa_{3}=1-s}{\Res}\underset{\kappa_{1}=s-1}{\Res}\mathcal{F}(\boldsymbol{\kappa},s)d\kappa_2+\int_{(0)}\underset{\kappa_{3}=2-2s-\kappa_2}{\Res}\\
   	&\underset{\kappa_{1}=s-1}{\Res}\mathcal{F}(\boldsymbol{\kappa},s)d\kappa_2+\int_{(0)}\underset{\kappa_{2}=2-2s}{\Res}\underset{\kappa_{1}=s-1}{\Res}\mathcal{F}(\boldsymbol{\kappa},s)d\kappa_3+\underset{\kappa_{3}=1-s}{\Res}\underset{\kappa_{2}=2-2s}{\Res}\underset{\kappa_{1}=s-1}{\Res}\mathcal{F}(\boldsymbol{\kappa},s),
   	\end{align*}
   	where $\underset{\kappa_{3}=1-s}{\Res}\underset{\kappa_{1}=s-1}{\Res}\mathcal{F}(\boldsymbol{\kappa},s)$ is equal to some holomorphic function multiplying the product of $\Lambda(2s-1,\tau^2)^2\Lambda(s,\tau)^2\cdot\Lambda(2-s,\tau^{-1})^{-2}$ and 
   	\begin{equation}\label{173_}
   	\frac{\Lambda(s-\kappa_2,\chi_{32}\tau)\Lambda(2s-1-\kappa_2,\chi_{32}\tau^2)\Lambda(2s-1+\kappa_2,\chi_{23}\tau^{2})\Lambda(s+\kappa_2,\chi_{23}\tau)}{\Lambda(1+\kappa_2,\chi_{23})\Lambda(1-\kappa_2,\chi_{32})\Lambda(2-s-\kappa_2,\chi_{32}\tau^{-1})\Lambda(2-s+\kappa_2,\chi_{23}\tau^{-1})};
   	\end{equation}
   	and $\underset{\kappa_{3}=2-2s-\kappa_2}{\Res}\underset{\kappa_{1}=s-1}{\Res}\mathcal{F}(\boldsymbol{\kappa},s)$ is equal to some holomorphic function multiplying
   	\begin{equation}\label{174_}
   	\frac{\Lambda(s-\kappa_2,\chi_{32}\tau)\Lambda(3s-2+\kappa_2,\chi_{23}\tau^3)\Lambda(3s-2,\tau^{3})\Lambda(2s-1,\tau^2)\Lambda(s,\tau)^2}{\Lambda(1+\kappa_2,\chi_{23})\Lambda(3-2s-\kappa_2,\chi_{32}\tau^{-2})\Lambda(3-2s,\tau^{-2})\Lambda(2-s,\tau^{-1})}.
   	\end{equation}
   	
   	From the formula \eqref{175_}, we see that $\underset{\kappa_{3}=1-s}{\Res}\underset{\kappa_{2}=2-2s}{\Res}\underset{\kappa_{1}=s-1}{\Res}\mathcal{F}(\boldsymbol{\kappa},s)$ is equal to some holomorphic function multiplying
   	\begin{equation}\label{176.}
   	\frac{\Lambda(4s-3,\tau^4)\Lambda(3s-2,\tau^{3})\Lambda(2s-1,\tau^2)\Lambda(s,\tau)}{\Lambda(4-3s,\tau^{-3})\Lambda(3-2s,\tau^{-2})\Lambda(2-s,\tau^{-1})}.
   	\end{equation}
   	
   	We thus see from the proceeding computations of analytic behaviors of the functions  $\underset{\kappa_{1}=s-1}{\Res}\mathcal{F}(\boldsymbol{\kappa},s),$ $\underset{\kappa_{3}=1-s}{\Res}\underset{\kappa_{1}=s-1}{\Res}\mathcal{F}(\boldsymbol{\kappa},s)$ and $\underset{\kappa_{3}=1-s}{\Res}\underset{\kappa_{2}=2-2s}{\Res}\underset{\kappa_{1}=s-1}{\Res}\mathcal{F}(\boldsymbol{\kappa},s),$ that  $\int_{(0)}\int_{(0)}\underset{\kappa_{1}=s-1}{\Res}\mathcal{F}(\boldsymbol{\kappa},s)d\kappa_2d\kappa_3$ and $\int_{(0)}\underset{\kappa_{3}=1-s}{\Res}\underset{\kappa_{1}=s-1}{\Res}\mathcal{F}(\boldsymbol{\kappa},s)d\kappa_2$ admit meromorphic continuation to the domain $1/2<\Re(s)<1,$ with a possible pole at $s=2/3$ if $\tau^3=1;$ and $\underset{\kappa_{3}=1-s}{\Res}\underset{\kappa_{2}=2-2s}{\Res}\underset{\kappa_{1}=s-1}{\Res}\mathcal{F}(\boldsymbol{\kappa},s)$ admits a meromorphic continuation to the domain $\mathcal{R}(1/2)^-\cup\mathcal{S}_{[1/2,1)},$ with possible simple poles at $s=3/4,$ $s=2/3$ and $s=1/2,$ when $\tau^4=1,$ $\tau^3=1$ and $\tau^2=1,$ respectively, according to \eqref{176.}.
   	
   	From \eqref{174_} we see that the function $\int_{(0)}\underset{\kappa_{3}=2-2s-\kappa_2}{\Res}\underset{\kappa_{1}=s-1}{\Res}\mathcal{F}(\boldsymbol{\kappa},s)d\kappa_2$ admits holomorphic continuation to the domain $2/3<\Re(s)<1.$ From \eqref{175_} we see that the function $\int_{(0)}\underset{\kappa_{2}=2-2s}{\Res}\underset{\kappa_{1}=s-1}{\Res}\mathcal{F}(\boldsymbol{\kappa},s)d\kappa_3$ admits holomorphic continuation to the domain $2/3<\Re(s)<1.$ Then combining these with  \eqref{173_} and \eqref{176.} one sees that $J_1^1(s)$ admits a holomorphic continuation to the domain $2/3<\Re(s)<1.$ Denote by $J_1^{(2/3,1)}(s)$ this continuation, where $2/3<\Re(s)<1.$
   	
   	Let $s\in \mathcal{R}(2/3)^+,$ then by Cauchy integral formula we have 
   	\begin{equation}\label{179}
   	\int_{(0)}\underset{\kappa_{3}=2-2s-\kappa_2}{\Res}\underset{\kappa_{1}=s-1}{\Res}\mathcal{F}(\boldsymbol{\kappa},s)d\kappa_2=\int_{\mathcal{C}}\underset{\kappa_{3}=2-2s-\kappa_2}{\Res}\underset{\kappa_{1}=s-1}{\Res}\mathcal{F}(\boldsymbol{\kappa},s)d\kappa_2.
   	\end{equation}
   	Likewise, for $s\in \mathcal{R}(2/3)^+,$ $\int_{(0)}\underset{\kappa_{2}=2-2s}{\Res}\underset{\kappa_{1}=s-1}{\Res}\mathcal{F}(\boldsymbol{\kappa},s)d\kappa_3$ is equal to 
   	\begin{equation}\label{180}
   	\int_{\mathcal{C}}\underset{\kappa_{2}=2-2s}{\Res}\underset{\kappa_{1}=s-1}{\Res}\mathcal{F}(\boldsymbol{\kappa},s)d\kappa_3-\underset{\kappa_3=3s-2}{\Res}\underset{\kappa_{2}=2-2s}{\Res}\underset{\kappa_{1}=s-1}{\Res}\mathcal{F}(\boldsymbol{\kappa},s).
   	\end{equation}
   	
   	Then according to \eqref{175_}, \eqref{173_}, \eqref{174_}, \eqref{179}, \eqref{180}, and the computation that the function $\underset{\kappa_3=3s-2}{\Res}\underset{\kappa_{2}=2-2s}{\Res}\underset{\kappa_{1}=s-1}{\Res}\mathcal{F}(\boldsymbol{\kappa},s)$ is equal to some holomorphic function multiplying 
   	\begin{equation}\label{183}
   	\frac{\Lambda(4s-2,\tau^4)\Lambda(3s-2,\tau^{3})\Lambda(2s-1,\tau^2)\Lambda(s,\tau)^2}{\Lambda(3-3s,\tau^{-3})\Lambda(3-2s,\tau^{-2})\Lambda(2-s,\tau^{-1})\Lambda(1+s,\tau)},
   	\end{equation}
   	we see that $J_1^{(2/3,1)}(s)$ admits a meromorphic continuation to the domain $\mathcal{R}(2/3),$ with a possible pole at $s=2/3$ when $\tau^3=1.$ Denote by  $J_1^{2/3}(s)$ this continuation, $s\in \mathcal{R}(2/3).$ Now let $s\in \mathcal{R}(2/3)^-.$ Then we have
   	\begin{align*}
   	J_1^{2/3}(s)=&\int_{(0)}\int_{(0)}\underset{\kappa_{1}=s-1}{\Res}\mathcal{F}(\boldsymbol{\kappa},s)d\kappa_2d\kappa_3+\int_{(0)}\underset{\kappa_{3}=1-s}{\Res}\underset{\kappa_{1}=s-1}{\Res}\mathcal{F}(\boldsymbol{\kappa},s)d\kappa_2\\
   	&+\int_{\mathcal{C}}\underset{\kappa_{3}=2-2s-\kappa_2}{\Res}\underset{\kappa_{1}=s-1}{\Res}\mathcal{F}(\boldsymbol{\kappa},s)d\kappa_2+\int_{\mathcal{C}}\underset{\kappa_{2}=2-2s}{\Res}\underset{\kappa_{1}=s-1}{\Res}\mathcal{F}(\boldsymbol{\kappa},s)d\kappa_3\\
   	&+\underset{\kappa_{3}=1-s}{\Res}\underset{\kappa_{2}=2-2s}{\Res}\underset{\kappa_{1}=s-1}{\Res}\mathcal{F}(\boldsymbol{\kappa},s)-\underset{\kappa_3=3s-2}{\Res}\underset{\kappa_{2}=2-2s}{\Res}\underset{\kappa_{1}=s-1}{\Res}\mathcal{F}(\boldsymbol{\kappa},s).
   	\end{align*}
   	According to \eqref{175_}, \eqref{174_}, \eqref{179} and \eqref{180}, the terms in the right hand side of the above formula are holomorphic in $1/2<\Re(s)<2/3$ except the term $\int_{\mathcal{C}}\underset{\kappa_{3}=2-2s-\kappa_2}{\Res}\underset{\kappa_{1}=s-1}{\Res}\mathcal{F}(\boldsymbol{\kappa},s)d\kappa_2,$ which is equal to, by Cauchy integral formula, that 
   	\begin{equation}\label{184}
   	\int_{(0)}\underset{\kappa_{3}=2-2s-\kappa_2}{\Res}\underset{\kappa_{1}=s-1}{\Res}\mathcal{F}(\boldsymbol{\kappa},s)d\kappa_2+\underset{\kappa_{2}=2-3s}{\Res}\underset{\kappa_{3}=2-2s-\kappa_2}{\Res}\underset{\kappa_{1}=s-1}{\Res}\mathcal{F}(\boldsymbol{\kappa},s).
   	\end{equation}
   	By \eqref{174_}, one sees that $\underset{\kappa_{2}=2-3s}{\Res}\underset{\kappa_{3}=2-2s-\kappa_2}{\Res}\underset{\kappa_{1}=s-1}{\Res}\mathcal{F}(\boldsymbol{\kappa},s)$ is equal to some holomorphic function multiplying 
   	\begin{equation}\label{185}
   	\frac{\Lambda(4s-2,\tau^4)\Lambda(3s-2,\tau^{3})\Lambda(2s-1,\tau^2)\Lambda(s,\tau)^2}{\Lambda(3-3s,\tau^{-3})\Lambda(3-2s,\tau^{-2})\Lambda(2-s,\tau^{-1})\Lambda(1+s,\tau)}.
   	\end{equation}
   	By \eqref{184} and \eqref{185} one sees that $\int_{\mathcal{C}}\underset{\kappa_{3}=2-2s-\kappa_2}{\Res}\underset{\kappa_{1}=s-1}{\Res}\mathcal{F}(\boldsymbol{\kappa},s)d\kappa_2$ admits a meromorphic continuation to $\mathcal{S}_{(1/3,2/3)}$ with a at most double pole at $s=1/2$ when $\tau^2=1.$ Hence we obtain a meromorphic continuation of $J_1^{2/3}(s)$ to the strip $1/2<\Re(s)<2/3.$ Denote by $J_1^{(1/2,2/3)}$ this continuation. Then 
   	\begin{align*}
   	J_1^{(1/2,2/3)}(s)=&\int_{(0)}\int_{(0)}\underset{\kappa_{1}=s-1}{\Res}\mathcal{F}(\boldsymbol{\kappa},s)d\kappa_2d\kappa_3+\int_{(0)}\underset{\kappa_{3}=1-s}{\Res}\underset{\kappa_{1}=s-1}{\Res}\mathcal{F}(\boldsymbol{\kappa},s)d\kappa_2\\
   	&+\int_{(0)}\underset{\kappa_{3}=2-2s-\kappa_2}{\Res}\underset{\kappa_{1}=s-1}{\Res}\mathcal{F}(\boldsymbol{\kappa},s)d\kappa_2+\int_{(0)}\underset{\kappa_{2}=2-2s}{\Res}\underset{\kappa_{1}=s-1}{\Res}\mathcal{F}(\boldsymbol{\kappa},s)d\kappa_3\\
   	&+\underset{\kappa_{3}=1-s}{\Res}\underset{\kappa_{2}=2-2s}{\Res}\underset{\kappa_{1}=s-1}{\Res}\mathcal{F}(\boldsymbol{\kappa},s)-\underset{\kappa_3=3s-2}{\Res}\underset{\kappa_{2}=2-2s}{\Res}\underset{\kappa_{1}=s-1}{\Res}\mathcal{F}(\boldsymbol{\kappa},s)\\
   	&+\underset{\kappa_{2}=2-3s}{\Res}\underset{\kappa_{3}=2-2s-\kappa_2}{\Res}\underset{\kappa_{1}=s-1}{\Res}\mathcal{F}(\boldsymbol{\kappa},s).
   	\end{align*}
   	One sees clearly that the terms in the right hand side of the above expression are meromorphic in $\mathcal{R}(1/2),$ except the terms $\int_{(0)}\int_{(0)}\underset{\kappa_{1}=s-1}{\Res}\mathcal{F}(\boldsymbol{\kappa},s)d\kappa_2d\kappa_3$ and  $\int_{(0)}\underset{\kappa_{3}=1-s}{\Res}\underset{\kappa_{1}=s-1}{\Res}\mathcal{F}(\boldsymbol{\kappa},s)d\kappa_2,$ which by Cauchy integral formula and \eqref{173_}, is equal to
   	\begin{equation}\label{186} 
   	\int_{\mathcal{C}}\underset{\kappa_{3}=1-s}{\Res}\underset{\kappa_{1}=s-1}{\Res}\mathcal{F}(\boldsymbol{\kappa},s)d\kappa_2-\underset{\kappa_{2}=2s-1}{\Res}\underset{\kappa_{3}=1-s}{\Res}\underset{\kappa_{1}=s-1}{\Res}\mathcal{F}(\boldsymbol{\kappa},s),
   	\end{equation}
   	where $s\in\mathcal{R}(1/2)^+.$ From the formula \eqref{173_}, we see that $\underset{\kappa_{2}=2s-1}{\Res}\underset{\kappa_{3}=1-s}{\Res}\underset{\kappa_{1}=s-1}{\Res}\mathcal{F}(\boldsymbol{\kappa},s)$ is equal to some holomorphic function multiplying
   	\begin{equation}\label{187}
   	\frac{\Lambda(4s-2,\tau^4)\Lambda(3s-1,\tau^{3})\Lambda(2s-1,\tau^2)^2\Lambda(s,\tau)^2\Lambda(1-s,\tau^{-1})}{\Lambda(3-3s,\tau^{-3})\Lambda(2s,\tau^{2})\Lambda(1+s,\tau)\Lambda(2-2s,\tau^{-2})\Lambda(2-s,\tau^{-1})^2}.
   	\end{equation}
   	We then apply the functional equation $\Lambda(2-2s,\tau^{-2})\sim \Lambda(2s-1,\tau^{2})$ to \eqref{187} to see that $\underset{\kappa_{3}=1-s}{\Res}\underset{\kappa_{2}=2-2s}{\Res}\underset{\kappa_{1}=s-1}{\Res}\mathcal{F}(\boldsymbol{\kappa},s)$ equals some holomorphic function multiplying
   	\begin{equation}\label{188}
   	\frac{\Lambda(4s-2,\tau^4)\Lambda(3s-1,\tau^{3})\Lambda(2s-1,\tau^2)\Lambda(s,\tau)}{\Lambda(3-3s,\tau^{-3})\Lambda(2s,\tau^{2})\Lambda(1+s,\tau)\Lambda(2-s,\tau^{-1})^2}.
   	\end{equation}
   	Note that when $s\in\mathcal{R}(1/2)^-,$ $2s$ lies in a zero-free region of $\Lambda(s,\tau^2).$ Also, Note that $\int_{(0)}\int_{(0)}\underset{\kappa_{1}=s-1}{\Res}\mathcal{F}(\boldsymbol{\kappa},s)d\kappa_2d\kappa_3=\int_{\mathcal{C}}\int_{\mathcal{C}}\underset{\kappa_{1}=s-1}{\Res}\mathcal{F}(\boldsymbol{\kappa},s)d\kappa_2d\kappa_3$ when $s\in \mathcal{R}(1/2)^+.$ Then by \eqref{186} and \eqref{187} we conclude that $J_1^{(1/2,2/3)}(s)$ admits a meromorphic continuation to the area $\mathcal{R}(1/2).$ Denote by $J_1^{1/2}(s)$ this continuation, then 
   	\begin{align*}
   	J_1^{1/2}(s)=&\int_{\mathcal{C}}\int_{\mathcal{C}}\underset{\kappa_{1}=s-1}{\Res}\mathcal{F}(\boldsymbol{\kappa},s)d\kappa_2d\kappa_3+\int_{\mathcal{C}}\underset{\kappa_{3}=1-s}{\Res}\underset{\kappa_{1}=s-1}{\Res}\mathcal{F}(\boldsymbol{\kappa},s)d\kappa_2\\
   	&+\int_{(0)}\underset{\kappa_{3}=2-2s-\kappa_2}{\Res}\underset{\kappa_{1}=s-1}{\Res}\mathcal{F}(\boldsymbol{\kappa},s)d\kappa_2+\int_{(0)}\underset{\kappa_{2}=2-2s}{\Res}\underset{\kappa_{1}=s-1}{\Res}\mathcal{F}(\boldsymbol{\kappa},s)d\kappa_3\\
   	&+\underset{\kappa_{3}=1-s}{\Res}\underset{\kappa_{2}=2-2s}{\Res}\underset{\kappa_{1}=s-1}{\Res}\mathcal{F}(\boldsymbol{\kappa},s)-\underset{\kappa_3=3s-2}{\Res}\underset{\kappa_{2}=2-2s}{\Res}\underset{\kappa_{1}=s-1}{\Res}\mathcal{F}(\boldsymbol{\kappa},s)\\
   	&+\underset{\kappa_{2}=2-3s}{\Res}\underset{\kappa_{3}=2-2s-\kappa_2}{\Res}\underset{\kappa_{1}=s-1}{\Res}\mathcal{F}(\boldsymbol{\kappa},s)-\underset{\kappa_{2}=2s-1}{\Res}\underset{\kappa_{3}=1-s}{\Res}\underset{\kappa_{1}=s-1}{\Res}\mathcal{F}(\boldsymbol{\kappa},s).
   	\end{align*}
   	Let $s\in\mathcal{R}(1/2)^-.$ By Cauchy's formula we have
   	\begin{align*}
   	\int_{\mathcal{C}}\int_{\mathcal{C}}\underset{\kappa_{1}=s-1}{\Res}&\mathcal{F}(\boldsymbol{\kappa},s)d\kappa_2d\kappa_3=\int_{(0)}\int_{(0)}\underset{\kappa_{1}=s-1}{\Res}\mathcal{F}(\boldsymbol{\kappa},s)d\kappa_2d\kappa_3\\&+\int_{\mathcal{C}}\underset{\kappa_{2}=1-2s}{\Res}\underset{\kappa_{1}=s-1}{\Res}\mathcal{F}(\boldsymbol{\kappa},s)d\kappa_3+\int_{(0)}\underset{\kappa_{3}=1-2s-\kappa_2}{\Res}\underset{\kappa_{1}=s-1}{\Res}\mathcal{F}(\boldsymbol{\kappa},s)d\kappa_2;
   	\end{align*}
   	and the function  $\int_{\mathcal{C}}\underset{\kappa_{3}=1-s}{\Res}\underset{\kappa_{1}=s-1}{\Res}\mathcal{F}(\boldsymbol{\kappa},s)d\kappa_2$ is equal to 
   	\begin{align*}
   	\int_{(0)}\underset{\kappa_{3}=1-s}{\Res}\underset{\kappa_{1}=s-1}{\Res}\mathcal{F}(\boldsymbol{\kappa},s)d\kappa_2+\underset{\kappa_{2}=1-2s}{\Res}\underset{\kappa_{3}=1-s}{\Res}\underset{\kappa_{1}=s-1}{\Res}\mathcal{F}(\boldsymbol{\kappa},s),
   	\end{align*}
   	where we have $\underset{\kappa_{2}=1-2s}{\Res}\underset{\kappa_{3}=1-s}{\Res}\underset{\kappa_{1}=s-1}{\Res}\mathcal{F}(\boldsymbol{\kappa},s)\sim \underset{\kappa_{2}=2s-1}{\Res}\underset{\kappa_{3}=1-s}{\Res}\underset{\kappa_{1}=s-1}{\Res}\mathcal{F}(\boldsymbol{\kappa},s).$
   	Now we have a continuation of $J_1^{(1/2)}(s)$ to the region $1/3<\Re(s)<1/2.$ Denote by $J_1^{(1/3,1/2)}(s)$ this continuation, then 
   	\begin{align*}
   	J_1^{(1/3,1/2)}(s)=&\int_{(0)}\int_{(0)}\underset{\kappa_{1}=s-1}{\Res}\mathcal{F}(\boldsymbol{\kappa},s)d\kappa_2d\kappa_3+\int_{(0)}\underset{\kappa_{3}=1-s}{\Res}\underset{\kappa_{1}=s-1}{\Res}\mathcal{F}(\boldsymbol{\kappa},s)d\kappa_2\\
   	&+\int_{\mathcal{C}}\underset{\kappa_{2}=1-2s}{\Res}\underset{\kappa_{1}=s-1}{\Res}\mathcal{F}(\boldsymbol{\kappa},s)d\kappa_3+\int_{(0)}\underset{\kappa_{3}=1-2s-\kappa_2}{\Res}\underset{\kappa_{1}=s-1}{\Res}\mathcal{F}(\boldsymbol{\kappa},s)d\kappa_2\\
   	&+\int_{(0)}\underset{\kappa_{3}=2-2s-\kappa_2}{\Res}\underset{\kappa_{1}=s-1}{\Res}\mathcal{F}(\boldsymbol{\kappa},s)d\kappa_2+\int_{(0)}\underset{\kappa_{2}=2-2s}{\Res}\underset{\kappa_{1}=s-1}{\Res}\mathcal{F}(\boldsymbol{\kappa},s)d\kappa_3\\
   	&+\underset{\kappa_{3}=1-s}{\Res}\underset{\kappa_{2}=2-2s}{\Res}\underset{\kappa_{1}=s-1}{\Res}\mathcal{F}(\boldsymbol{\kappa},s)-\underset{\kappa_3=3s-2}{\Res}\underset{\kappa_{2}=2-2s}{\Res}\underset{\kappa_{1}=s-1}{\Res}\mathcal{F}(\boldsymbol{\kappa},s)\\
   	&+\underset{\kappa_{2}=2-3s}{\Res}\underset{\kappa_{3}=2-2s-\kappa_2}{\Res}\underset{\kappa_{1}=s-1}{\Res}\mathcal{F}(\boldsymbol{\kappa},s)-\underset{\kappa_{2}=2s-1}{\Res}\underset{\kappa_{3}=1-s}{\Res}\underset{\kappa_{1}=s-1}{\Res}\mathcal{F}(\boldsymbol{\kappa},s)\\
   	&+\underset{\kappa_{2}=1-2s}{\Res}\underset{\kappa_{3}=1-s}{\Res}\underset{\kappa_{1}=s-1}{\Res}\mathcal{F}(\boldsymbol{\kappa},s).
   	\end{align*}
   	Thus we obtain a meromorphic continuation of $J_1(s)$ to the area $\mathcal{S}_{(1/3,\infty)}:$ 
   	\begin{equation}\label{claim1}
   	\widetilde{J}_1(s)=\begin{cases}
   	J_1(s),\ s\in \mathcal{S}_{(1,+\infty)};\\
   	J_1^1(s),\ s\in\mathcal{R}(1);\\
   	J_1^{(2/3,1)}(s),\ s\in \mathcal{S}_{(2/3,1)};\\
   	J_1^{2/3}(s),\ s\in \mathcal{R}(2/3);\\
   	J_1^{(1/2,2/3)}(s),\ s\in \mathcal{S}_{(1/2,2/3)};\\
   	J_1^{1/2}(s),\ s\in \mathcal{R}(1/2);\\
   	J_1^{(1/3,1/2)}(s),\ s\in \mathcal{S}_{(1/3,1/2)};\\
   	\end{cases}
   	\end{equation}
   	
   	From the above formulas one sees that $\widetilde{J}_1(s)$ has possible poles at $s=3/4,$ $s=2/3$ and $s=1/2;$ and the potential poles at $s=3/4,$ $s=2/3$ are at most simple, the possible pole at $s=1/2$ has order at most 2. Moreover, from the above explicit expressions of $\widetilde{J}_1(s),$ we see that $\widetilde{J}_1(s)\cdot\Lambda(s,\tau)^{-1}$ has at most a simple pole at $s=1/2$ if $L_F(1/2,\tau)=0.$
   	\begin{itemize}
   		\item[Case 1:]
   		If $L_F(3/4,\tau)=0,$ then by functional equation we have that $\Lambda(1/4,\tau^{-1})=0.$ Suppose that $\widetilde{J}_1(s)$ has a pole at $s=3/4,$ then from the proceeding explicit expressions, we must have that $\tau^4=1,$ and the singular part of $\widetilde{J}_1(s)$ around $s=3/4$ is a holomorphic function multiplying $\Lambda(4s-3,\tau^4)\Lambda(3s-2,\tau^3).$ Note that $\Lambda(3s-2,\tau^3)\mid_{s=3/4}=\Lambda(1/4,\tau^3)=\Lambda(1/4,\tau^{-1})=0.$ Hence, when $L_F(3/4,\tau)=0,$ $\widetilde{J}_1(s)$ is holomorphic at $s=3/4.$
   		\item[Case 2:]
   		If $L_F(2/3,\tau)=0,$ then by functional equation we have that $\Lambda(1/3,\tau^{-1})=0.$ Suppose that $\widetilde{J}_1(s)$ has a pole at $s=2/3,$ then from the proceeding explicit expressions, we must have that $\tau^3=1,$ and the singular part of $\widetilde{J}_1(s)$ around $s=2/3$ is a holomorphic function multiplying $\Lambda(3s-2,\tau^3)\Lambda(2s-1,\tau^2).$ Note that $\Lambda(2s-1,\tau^2)\mid_{s=2/3}=\Lambda(1/3,\tau^2)=\Lambda(1/3,\tau^{-1})=0.$ Hence, when $L_F(2/3,\tau)=0,$ $\widetilde{J}_1(s)$ is holomorphic at $s=2/3.$
   	\end{itemize}
   	Now the proof of Claim \ref{60claim} is complete.
   \end{proof}

   \begin{proof}[Proof of Claim \ref{61claim}]
   	Let $s\in\mathcal{R}(1)^+.$ Let $J_2(s):=\int_{(0)}\int_{(0)}\underset{\kappa_{2}=s-1}{\Res}\mathcal{F}(\boldsymbol{\kappa},s)d\kappa_1d\kappa_3,$ and  $J^1_2(s):=\int_{\mathcal{C}}\int_{\mathcal{C}}\underset{\kappa_{2}=s-1}{\Res}\mathcal{F}(\boldsymbol{\kappa},s)d\kappa_1d\kappa_3.$ By the analytic property of $\underset{\kappa_{2}=s-1}{\Res}\mathcal{F}(\boldsymbol{\kappa},s)$ we see that $J^1_2(s)$ is meromorphic in the domain $\mathcal{R}(1),$ with a possible pole at $s=1.$ Let $s\in\mathcal{R}(1)^-.$ Applying Cauchy integral formula we then see that  
   	\begin{equation}\label{190}
   	J_2^1(s)=\int_{\mathcal{C}}\int_{(0)}\underset{\kappa_{2}=s-1}{\Res}\mathcal{F}(\boldsymbol{\kappa},s)d\kappa_1d\kappa_3+\int_{\mathcal{C}}\underset{\kappa_{1}=2-2s}{\Res}\underset{\kappa_{2}=s-1}{\Res}\mathcal{F}(\boldsymbol{\kappa},s)d\kappa_3,
   	\end{equation}
   	where $\underset{\kappa_{1}=2-2s}{\Res}\underset{\kappa_{2}=s-1}{\Res}\mathcal{F}(\boldsymbol{\kappa},s)$ is equal to some holomorphic function multiplying the product of $\Lambda(s,\tau)^2$ and the meromorphic function 
   	\begin{equation}\label{193}
   	\frac{\Lambda(2s-1-\kappa_3,\chi_{43}\tau^2)\Lambda(2s-1+\kappa_3,\chi_{34}\tau^{2})\Lambda(3s-2,\tau^3)\Lambda(2s-1,\tau^2)}{\Lambda(2-s+\kappa_3,\chi_{34})\Lambda(2-s-\kappa_3,\chi_{43})\Lambda(3-2s,\tau^{-2})\Lambda(2-s,\tau^{-1})}.
   	\end{equation}
   	Then $J_2^1(s)$ is equal to, after applications of Cauchy integral formula to \eqref{190},
   	\begin{align*}
   	&\int_{(0)}\int_{(0)}\underset{\kappa_{2}=s-1}{\Res}\mathcal{F}(\boldsymbol{\kappa},s)d\kappa_1d\kappa_3+\int_{(0)}\underset{\kappa_{3}=2-2s}{\Res}\underset{\kappa_{2}=s-1}{\Res}\mathcal{F}(\boldsymbol{\kappa},s)d\kappa_1+\int_{(0)}\underset{\kappa_{3}=2-2s-\kappa_1}{\Res}\\
   	&\underset{\kappa_{2}=s-1}{\Res}\mathcal{F}(\boldsymbol{\kappa},s)d\kappa_1+\int_{(0)}\underset{\kappa_{1}=2-2s}{\Res}\underset{\kappa_{2}=s-1}{\Res}\mathcal{F}(\boldsymbol{\kappa},s)d\kappa_3+\underset{\kappa_{3}=2-2s}{\Res}\underset{\kappa_{1}=2-2s}{\Res}\underset{\kappa_{2}=s-1}{\Res}\mathcal{F}(\boldsymbol{\kappa},s),
   	\end{align*}
   	where $\underset{\kappa_{3}=2-2s}{\Res}\underset{\kappa_{2}=s-1}{\Res}\mathcal{F}(\boldsymbol{\kappa},s)$ is equal to some holomorphic function multiplying the product of the meromorphic function $\Lambda(s,\tau)^2$ and
   	\begin{equation}\label{191}
   	\frac{\Lambda(2s-1+\kappa_1,\chi_{12}\tau^2)\Lambda(2s-1-\kappa_1,\chi_{21}\tau^2)\Lambda(3s-2,\tau^{3})\Lambda(2s-1,\tau^2)}{\Lambda(2-s+\kappa_1,\chi_{12})\Lambda(2-s-\kappa_1,\chi_{21})\Lambda(3-2s,\tau^{-2})\Lambda(2-s,\tau^{-1})};
   	\end{equation} 
   	also, $\underset{\kappa_{3}=2-2s-\kappa_1}{\Res}\underset{\kappa_{2}=s-1}{\Res}\mathcal{F}(\boldsymbol{\kappa},s)$ is equal to some holomorphic function multiplying the product of $\Lambda(2s-1,\tau^2)^2\Lambda(s,\tau)^2\cdot\Lambda(2-s,\tau^{-1})^{-2}$ and 
   	\begin{equation}\label{192}
   	\frac{\Lambda(1-\kappa_1,\chi_{21})\Lambda(s-\kappa_1,\chi_{21}\tau)\Lambda(3s-2+\kappa_1,\chi_{12}\tau^3)\Lambda(2s-1+\kappa_1,\chi_{12}\tau^{2})}{\Lambda(1+\kappa_1,\chi_{12})\Lambda(s+\kappa_1,\chi_{12}\tau)\Lambda(3-2s-\kappa_1,\chi_{21}\tau^{-2})\Lambda(2-s-\kappa_1,\tau^{-1})}.
   	\end{equation}
   	
   	From the formula \eqref{193}, we see that $\underset{\kappa_{3}=2-2s}{\Res}\underset{\kappa_{1}=2-2s}{\Res}\underset{\kappa_{2}=s-1}{\Res}\mathcal{F}(\boldsymbol{\kappa},s)$ is equal to some holomorphic function multiplying
   	\begin{equation}\label{194}
   	\frac{\Lambda(4s-3,\tau^4)\Lambda(3s-2,\tau^{3})\Lambda(2s-1,\tau^2)\Lambda(s,\tau)}{\Lambda(4-3s,\tau^{-3})\Lambda(3-2s,\tau^{-2})\Lambda(2-s,\tau^{-1})}.
   	\end{equation}
   	
   	We thus see from the proceeding computations of analytic behaviors of the functions  $\underset{\kappa_{2}=s-1}{\Res}\mathcal{F}(\boldsymbol{\kappa},s),$ $\underset{\kappa_{3}=2-2s}{\Res}\underset{\kappa_{2}=s-1}{\Res}\mathcal{F}(\boldsymbol{\kappa},s),$ $\underset{\kappa_{1}=2-2s}{\Res}\underset{\kappa_{2}=s-1}{\Res}\mathcal{F}(\boldsymbol{\kappa},s)$ and the function  $\underset{\kappa_{3}=2-2s}{\Res}\underset{\kappa_{1}=2-2s}{\Res}\underset{\kappa_{2}=s-1}{\Res}\mathcal{F}(\boldsymbol{\kappa},s),$ that the functions  $\int_{(0)}\int_{(0)}\underset{\kappa_{2}=s-1}{\Res}\mathcal{F}(\boldsymbol{\kappa},s)d\kappa_1d\kappa_3,$ $\int_{(0)}\underset{\kappa_{3}=2-2s}{\Res}\underset{\kappa_{2}=s-1}{\Res}\mathcal{F}(\boldsymbol{\kappa},s)d\kappa_1$ and $\int_{(0)}\underset{\kappa_{1}=2-2s}{\Res}\underset{\kappa_{2}=s-1}{\Res}\mathcal{F}(\boldsymbol{\kappa},s)f\kappa_3$ admit meromorphic continuation to the domain $1/2<\Re(s)<1,$ with a possible pole at $s=2/3$ if $\tau^3=1;$ and $\underset{\kappa_{3}=2-2s}{\Res}\underset{\kappa_{1}=2-2s}{\Res}\underset{\kappa_{2}=s-1}{\Res}\mathcal{F}(\boldsymbol{\kappa},s)$ admits a meromorphic continuation to the domain $\mathcal{R}(1/2)^-\cup\mathcal{S}_{[1/2,1)},$ with possible simple poles at $s=3/4,$ $s=2/3$ and $s=1/2,$ when $\tau^4=1,$ $\tau^3=1$ and $\tau^2=1,$ respectively. 
   	
   	From \eqref{192} we see that the function $\underset{\kappa_{3}=2-2s-\kappa_1}{\Res}\underset{\kappa_{2}=s-1}{\Res}\mathcal{F}(\boldsymbol{\kappa},s)$ might have infinitely many poles in the strip $1/2<\Re(s)<1.$ These poles come from nontrivial zeros of $L_F(s,\chi_{12}\tau)$ is this strip. Hence we may have a problem shifting contours if we try to continue   $\int_{(0)}\underset{\kappa_{3}=2-2s-\kappa_1}{\Res}\underset{\kappa_{2}=s-1}{\Res}\mathcal{F}(\boldsymbol{\kappa},s)d\kappa_1$ directly. To remedy this, we need to first deal with the factor $\Lambda(s+\kappa_1,\chi_{12}\tau).$ Thanks to the uniform zero-free region of Rankin-Selberg L-functions defined in Section \ref{7.11}, the function  $\underset{\kappa_{3}=2-2s-\kappa_1}{\Res}\underset{\kappa_{2}=s-1}{\Res}\mathcal{F}(\boldsymbol{\kappa},s)$ is holomorphic in the domain $\mathcal{R}(1-s).$ Then we can apply Cauchy integral formula to obtain that 
   	\begin{equation}\label{195.}
   	\int_{(0)}\underset{\kappa_{3}=2-2s-\kappa_1}{\Res}\underset{\kappa_{2}=s-1}{\Res}\mathcal{F}(\boldsymbol{\kappa},s)d\kappa_1=\int_{(1-s)}\underset{\kappa_{3}=2-2s-\kappa_1}{\Res}\underset{\kappa_{2}=s-1}{\Res}\mathcal{F}(\boldsymbol{\kappa},s)d\kappa_1,
   	\end{equation}
   	where the integral on the right hand side is taken over $(1-s):=\{z\in\mathbb{C}:\ \Re(z)=1-\Re(s)\}.$ Let $\kappa_1'=\kappa_1+s-1,$ $\kappa_2'=\kappa_2$ and $\kappa_3'=\kappa_3.$ Denote by $\boldsymbol{\kappa}'=(\kappa_1',\kappa_2',\kappa_3').$ Then $d\kappa_j'=d\kappa_j,$ $1\leq j\leq 3.$ Hence we have 
   	\begin{equation}\label{196.}
   	\int_{(0)}\underset{\kappa_{3}=2-2s-\kappa_1}{\Res}\underset{\kappa_{2}=s-1}{\Res}\mathcal{F}(\boldsymbol{\kappa},s)d\kappa_1=\int_{(0)}\underset{\kappa_{3}'=1-s-\kappa_1'}{\Res}\underset{\kappa_{2}'=s-1}{\Res}\mathcal{F}(\boldsymbol{\kappa}',s)d\kappa_1',
   	\end{equation}
   	where by \eqref{192}, $\underset{\kappa_{3}'=1-s-\kappa_1'}{\Res}\underset{\kappa_{2}'=s-1}{\Res}\mathcal{F}(\boldsymbol{\kappa}',s)$ is equal to some holomorphic function multiplying the product of $\Lambda(2s-1,\tau^2)^2\Lambda(s,\tau)^2\cdot\Lambda(2-s,\tau^{-1})^{-2}$ and 
   	\begin{equation}\label{197.}
   	\frac{\Lambda(s+\kappa_1',\chi_{12}\tau)\Lambda(s-\kappa_1',\chi_{21}\tau)\Lambda(2s-1+\kappa_1',\chi_{12}\tau^{2})\Lambda(2s-1-\kappa_1',\chi_{21}\tau^2)}{\Lambda(1+\kappa_1',\chi_{12})\Lambda(1-\kappa_1',\chi_{21})\Lambda(2-s+\kappa_1',\chi_{12}\tau^{-1})\Lambda(2-s-\kappa_1',\chi_{21}\tau^{-1})}.
   	\end{equation}
   	Then from \eqref{195.}, \eqref{196.} and \eqref{197.} we conclude that $\int_{(0)}\underset{\kappa_{3}=2-2s-\kappa_1}{\Res}\underset{\kappa_{2}=s-1}{\Res}\mathcal{F}(\boldsymbol{\kappa},s)d\kappa_1$ admits a meromorphic continuation to the strip $1/2<\Re(s)<1.$ We then have a meromorphic continuation of $J_2^1(s)$ to the area $\mathcal{S}_{(1/2,1)}.$ Denote by $J_2^{(1/2,1)}$ this continuation. Then $J_2^{(1/2,1)}(s)$ is equal to 
   	\begin{align*}
   	&\int_{(0)}\int_{(0)}\underset{\kappa_{2}=s-1}{\Res}\mathcal{F}(\boldsymbol{\kappa},s)d\kappa_1d\kappa_3+\int_{(0)}\underset{\kappa_{3}=2-2s}{\Res}\underset{\kappa_{2}=s-1}{\Res}\mathcal{F}(\boldsymbol{\kappa},s)d\kappa_1+\int_{(0)}\underset{\kappa_{3}'=1-s-\kappa_1'}{\Res}\\
   	&\underset{\kappa_{2}'=s-1}{\Res}\mathcal{F}(\boldsymbol{\kappa}',s)d\kappa_1'+\int_{(0)}\underset{\kappa_{1}=2-2s}{\Res}\underset{\kappa_{2}=s-1}{\Res}\mathcal{F}(\boldsymbol{\kappa},s)d\kappa_3+\underset{\kappa_{3}=2-2s}{\Res}\underset{\kappa_{1}=2-2s}{\Res}\underset{\kappa_{2}=s-1}{\Res}\mathcal{F}(\boldsymbol{\kappa},s).
   	\end{align*}
   	Let $s\in\mathcal{R}(1/2)^+.$ Then by \eqref{191} and Cauchy integral formula we see that the function  $\int_{(0)}\underset{\kappa_{3}=2-2s}{\Res}\underset{\kappa_{2}=s-1}{\Res}\mathcal{F}(\boldsymbol{\kappa},s)d\kappa_1$ is equal to
   	\begin{equation}\label{198}
   	\int_{\mathcal{C}}\underset{\kappa_{3}=2-2s}{\Res}\underset{\kappa_{2}=s-1}{\Res}\mathcal{F}(\boldsymbol{\kappa},s)d\kappa_1-\underset{\kappa_{1}=2s-1}{\Res}\underset{\kappa_{3}=2-2s}{\Res}\underset{\kappa_{2}=s-1}{\Res}\mathcal{F}(\boldsymbol{\kappa},s),
   	\end{equation} 
   	and $\underset{\kappa_{1}=2s-1}{\Res}\underset{\kappa_{3}=2-2s}{\Res}\underset{\kappa_{2}=s-1}{\Res}\mathcal{F}(\boldsymbol{\kappa},s)$ is equal to some holomorphic function multiplying 
   	\begin{equation}\label{198'}
   	\frac{\Lambda(4s-2,\tau^4)\Lambda(3s-2,\tau^{3})\Lambda(2s-1,\tau^2)\Lambda(s,\tau)^2}{\Lambda(3-3s,\tau^{-3})\Lambda(3-2s,\tau^{-2})\Lambda(2-s,\tau^{-1})\Lambda(1+s,\tau)}.
   	\end{equation}
   	
   	By \eqref{197.} and Cauchy formula we see that  $\int_{(0)}\underset{\kappa_{3}'=1-s-\kappa_1'}{\Res}\underset{\kappa_{2}'=s-1}{\Res}\mathcal{F}(\boldsymbol{\kappa}',s)d\kappa_1'$ equals
   	\begin{equation}\label{199}
   	\int_{\mathcal{C}}\underset{\kappa_{3}'=1-s-\kappa_1'}{\Res}\underset{\kappa_{2}'=s-1}{\Res}\mathcal{F}(\boldsymbol{\kappa}',s)d\kappa_1'-\underset{\kappa_{1}'=2s-1}{\Res}\underset{\kappa_{3}'=2-2s}{\Res}\underset{\kappa_{2}'=s-1}{\Res}\mathcal{F}(\boldsymbol{\kappa},s),
   	\end{equation} 
   	and $\underset{\kappa_{1}'=2s-1}{\Res}\underset{\kappa_{3}'=2-2s}{\Res}\underset{\kappa_{2}'=s-1}{\Res}\mathcal{F}(\boldsymbol{\kappa},s)$ is equal to some holomorphic function multiplying 
   	\begin{equation}\label{199'}
   	\frac{\Lambda(4s-2,\tau^4)\Lambda(3s-1,\tau^{3})\Lambda(2s-1,\tau^2)^2\Lambda(1-s,\tau{-1})\Lambda(s,\tau)^2}{\Lambda(3-3s,\tau^{-3})\Lambda(2-2s,\tau^{-2})\Lambda(2-s,\tau^{-1})^2\Lambda(2s,\tau^2)\Lambda(1+s,\tau)}.
   	\end{equation}
   	
   	By \eqref{193} and Cauchy formula we see that  $\int_{(0)}\underset{\kappa_{1}=2-2s}{\Res}\underset{\kappa_{2}=s-1}{\Res}\mathcal{F}(\boldsymbol{\kappa},s)d\kappa_3$ equals
   	\begin{equation}\label{200}
   	\int_{\mathcal{C}}\underset{\kappa_{1}=2-2s}{\Res}\underset{\kappa_{2}=s-1}{\Res}\mathcal{F}(\boldsymbol{\kappa},s)d\kappa_3-\underset{\kappa_{3}=2s-1}{\Res}\underset{\kappa_{1}=2-2s}{\Res}\underset{\kappa_{2}=s-1}{\Res}\mathcal{F}(\boldsymbol{\kappa},s),
   	\end{equation}
   	and $\underset{\kappa_{3}=2s-1}{\Res}\underset{\kappa_{1}=2-2s}{\Res}\underset{\kappa_{2}=s-1}{\Res}\mathcal{F}(\boldsymbol{\kappa},s)$ is equal to some holomorphic function multiplying 
   	\begin{equation}\label{200'}
   	\frac{\Lambda(4s-2,\tau^4)\Lambda(3s-2,\tau^{3})\Lambda(2s-1,\tau^2)\Lambda(s,\tau)^2}{\Lambda(3-3s,\tau^{-3})\Lambda(3-2s,\tau^{-2})\Lambda(2-s,\tau^{-1})\Lambda(1+s,\tau)}.
   	\end{equation}
   	
   	Note that $\int_{\mathcal{C}}\underset{\kappa_{3}=2-2s}{\Res}\underset{\kappa_{2}=s-1}{\Res}\mathcal{F}(\boldsymbol{\kappa},s)d\kappa_1,$ $\int_{\mathcal{C}}\underset{\kappa_{3}'=1-s-\kappa_1'}{\Res}\underset{\kappa_{2}'=s-1}{\Res}\mathcal{F}(\boldsymbol{\kappa}',s)d\kappa_1'$ and the function  $\int_{\mathcal{C}}\underset{\kappa_{1}=2-2s}{\Res}\underset{\kappa_{2}=s-1}{\Res}\mathcal{F}(\boldsymbol{\kappa},s)d\kappa_3$ are meromorphic inside $\mathcal{R}(1/2),$ with a potential pole of order less or equal to 2 at $s=1/2.$ Moreover, it follows from \eqref{195.}, \eqref{196.} and \eqref{197.} that if $L_F(1/2,\tau)=0,$ then these three integrals are holomorphic at $s=1/2;$ and the ratio of these integrals and $\Lambda(s,\tau)$ have at most a simple pole at $s=1/2.$ In particular, combining equations \eqref{198}, \eqref{198'}, \eqref{199}, \eqref{199'}, \eqref{200} and \eqref{200'}, one thus has a meromorphic continuation of $J_2^{(1/2,1)}(s)$ to the domain $\mathcal{R}(1/2),$ with a potential pole of order less or equal to 2 at $s=1/2.$ Denote by $J_2^{1/2}(s)$ this continuation. Then $J_2^{1/2}(s)\cdot\Lambda(s,\tau)^{-1}$ has at most a simple pole at $s=1/2$ if $L_F(1/2,\tau)=0.$ Explicitly, by Cauchy's formula we have
   	\begin{align*}
   	J_2^{1/2}(s)=&\int_{\mathcal{C}}\int_{\mathcal{C}}\underset{\kappa_{2}=s-1}{\Res}\mathcal{F}(\boldsymbol{\kappa},s)d\kappa_1d\kappa_3+\int_{\mathcal{C}}\underset{\kappa_{3}=2-2s}{\Res}\underset{\kappa_{2}=s-1}{\Res}\mathcal{F}(\boldsymbol{\kappa},s)d\kappa_1\\
   	&+\int_{\mathcal{C}}\underset{\kappa_{3}'=1-s-\kappa_1'}{\Res}\underset{\kappa_{2}'=s-1}{\Res}\mathcal{F}(\boldsymbol{\kappa}',s)d\kappa_1'+\int_{\mathcal{C}}\underset{\kappa_{1}=2-2s}{\Res}\underset{\kappa_{2}=s-1}{\Res}\mathcal{F}(\boldsymbol{\kappa},s)d\kappa_3\\
   	&+\underset{\kappa_{3}=2-2s}{\Res}\underset{\kappa_{1}=2-2s}{\Res}\underset{\kappa_{2}=s-1}{\Res}\mathcal{F}(\boldsymbol{\kappa},s)-\underset{\kappa_{1}=2s-1}{\Res}\underset{\kappa_{3}=2-2s}{\Res}\underset{\kappa_{2}=s-1}{\Res}\mathcal{F}(\boldsymbol{\kappa},s)\\
   	&-\underset{\kappa_{1}'=2s-1}{\Res}\underset{\kappa_{3}'=2-2s}{\Res}\underset{\kappa_{2}'=s-1}{\Res}\mathcal{F}(\boldsymbol{\kappa},s)-\underset{\kappa_{3}=2s-1}{\Res}\underset{\kappa_{1}=2-2s}{\Res}\underset{\kappa_{2}=s-1}{\Res}\mathcal{F}(\boldsymbol{\kappa},s).
   	\end{align*}
   	Let $s\in\mathcal{R}(1/2)^-.$ Then $\int_{\mathcal{C}}\int_{\mathcal{C}}\underset{\kappa_{2}=s-1}{\Res}\mathcal{F}(\boldsymbol{\kappa},s)d\kappa_1d\kappa_3$ is equal to
   	\begin{align*}
   	\int_{(0)}\int_{(0)}&\underset{\kappa_{2}=s-1}{\Res}\mathcal{F}(\boldsymbol{\kappa},s)d\kappa_1d\kappa_3+\int_{(0)}\underset{\kappa_{1}=1-2s}{\Res}\underset{\kappa_{2}=s-1}{\Res}\mathcal{F}(\boldsymbol{\kappa},s)d\kappa_3\\
   	&+\int_{(0)}\underset{\kappa_{3}=1-2s}{\Res}\underset{\kappa_{2}=s-1}{\Res}\mathcal{F}(\boldsymbol{\kappa},s)d\kappa_1+\underset{\kappa_{3}=1-2s}{\Res}\underset{\kappa_{1}=1-2s}{\Res}\underset{\kappa_{2}=s-1}{\Res}\mathcal{F}(\boldsymbol{\kappa},s).
   	\end{align*}
   	Likewise, the function  $\int_{\mathcal{C}}\underset{\kappa_{3}=2-2s}{\Res}\underset{\kappa_{2}=s-1}{\Res}\mathcal{F}(\boldsymbol{\kappa},s)d\kappa_1$ is equal to 
   	\begin{align*}
   	\int_{(0)}\underset{\kappa_{3}=2-2s}{\Res}\underset{\kappa_{2}=s-1}{\Res}\mathcal{F}(\boldsymbol{\kappa},s)d\kappa_1+\underset{\kappa_{1}=1-2s}{\Res}\underset{\kappa_{3}=2-2s}{\Res}\underset{\kappa_{2}=s-1}{\Res}\mathcal{F}(\boldsymbol{\kappa},s);
   	\end{align*}
   	the function  $\int_{\mathcal{C}}\underset{\kappa_{3}'=1-s-\kappa_1'}{\Res}\underset{\kappa_{2}'=s-1}{\Res}\mathcal{F}(\boldsymbol{\kappa}',s)d\kappa_1'$ is equal to 
   	\begin{align*}
   	\int_{(0)}\underset{\kappa_{3}'=1-s-\kappa_1'}{\Res}\underset{\kappa_{2}'=s-1}{\Res}\mathcal{F}(\boldsymbol{\kappa}',s)d\kappa_1'+\underset{\kappa_{1}'=1-2s}{\Res}\underset{\kappa_{3}'=1-s-\kappa_1'}{\Res}\underset{\kappa_{2}'=s-1}{\Res}\mathcal{F}(\boldsymbol{\kappa}',s);
   	\end{align*}
   	and the function $\int_{\mathcal{C}}\underset{\kappa_{1}=2-2s}{\Res}\underset{\kappa_{2}=s-1}{\Res}\mathcal{F}(\boldsymbol{\kappa},s)d\kappa_3$ is equal to 
   	\begin{align*}
   	\int_{(0)}\underset{\kappa_{1}=2-2s}{\Res}\underset{\kappa_{2}=s-1}{\Res}\mathcal{F}(\boldsymbol{\kappa},s)d\kappa_3+\underset{\kappa_{3}=1-2s}{\Res}\underset{\kappa_{1}=2-2s}{\Res}\underset{\kappa_{2}=s-1}{\Res}\mathcal{F}(\boldsymbol{\kappa},s).
   	\end{align*}
   	As before, a computation of the above integrals leads to a meromorphic continuation of $J_2^{1/2}(s)$ to the region $1/3<\Re(s)<1/2.$ Denote by this continuation $J_2^{(1/3,1/2)}(s),$ then 
   	\begin{align*}
   	J_2^{(1/3,1/2)}(s)=&\int_{(0)}\int_{(0)}\underset{\kappa_{2}=s-1}{\Res}\mathcal{F}(\boldsymbol{\kappa},s)d\kappa_1d\kappa_3+\int_{(0)}\underset{\kappa_{3}=2-2s}{\Res}\underset{\kappa_{2}=s-1}{\Res}\mathcal{F}(\boldsymbol{\kappa},s)d\kappa_1\\
   	&+\int_{(0)}\underset{\kappa_{3}'=1-s-\kappa_1'}{\Res}\underset{\kappa_{2}'=s-1}{\Res}\mathcal{F}(\boldsymbol{\kappa}',s)d\kappa_1'+\int_{(0)}\underset{\kappa_{1}=2-2s}{\Res}\underset{\kappa_{2}=s-1}{\Res}\mathcal{F}(\boldsymbol{\kappa},s)d\kappa_3\\
   	&+\underset{\kappa_{3}=2-2s}{\Res}\underset{\kappa_{1}=2-2s}{\Res}\underset{\kappa_{2}=s-1}{\Res}\mathcal{F}(\boldsymbol{\kappa},s)-\underset{\kappa_{1}=2s-1}{\Res}\underset{\kappa_{3}=2-2s}{\Res}\underset{\kappa_{2}=s-1}{\Res}\mathcal{F}(\boldsymbol{\kappa},s)\\
   	&-\underset{\kappa_{1}'=2s-1}{\Res}\underset{\kappa_{3}'=2-2s}{\Res}\underset{\kappa_{2}'=s-1}{\Res}\mathcal{F}(\boldsymbol{\kappa},s)-\underset{\kappa_{3}=2s-1}{\Res}\underset{\kappa_{1}=2-2s}{\Res}\underset{\kappa_{2}=s-1}{\Res}\mathcal{F}(\boldsymbol{\kappa},s)\\
   	&+\int_{(0)}\underset{\kappa_{3}=1-2s}{\Res}\underset{\kappa_{2}=s-1}{\Res}\mathcal{F}(\boldsymbol{\kappa},s)d\kappa_1+\underset{\kappa_{3}=1-2s}{\Res}\underset{\kappa_{1}=1-2s}{\Res}\underset{\kappa_{2}=s-1}{\Res}\mathcal{F}(\boldsymbol{\kappa},s)\\
   	&+\int_{(0)}\underset{\kappa_{1}=1-2s}{\Res}\underset{\kappa_{2}=s-1}{\Res}\mathcal{F}(\boldsymbol{\kappa},s)d\kappa_3+\underset{\kappa_{1}=1-2s}{\Res}\underset{\kappa_{3}=2-2s}{\Res}\underset{\kappa_{2}=s-1}{\Res}\mathcal{F}(\boldsymbol{\kappa},s)\\
   	&+\underset{\kappa_{1}'=1-2s}{\Res}\underset{\kappa_{3}'=1-s-\kappa_1'}{\Res}\underset{\kappa_{2}'=s-1}{\Res}\mathcal{F}(\boldsymbol{\kappa}',s)+\underset{\kappa_{3}=1-2s}{\Res}\underset{\kappa_{1}=2-2s}{\Res}\underset{\kappa_{2}=s-1}{\Res}\mathcal{F}(\boldsymbol{\kappa},s).
   	\end{align*}
   	
   	Thus we obtain a meromorphic continuation of $J_2(s)$ to $\mathcal{S}_{(1/3,\infty)}:$ 
   	\begin{equation}\label{claim2}
   	\widetilde{J}_2(s)=\begin{cases}
   	J_2(s),\ s\in\mathcal{S}_{(1,+\infty)};\\
   	J_2^1(s),\ s\in\mathcal{R}(1);\\
   	J_2^{(1/2,1)}(s),\ s\in \mathcal{S}_{(1/2,1)};\\
   	J_2^{1/2}(s),\ s\in \mathcal{R}(1/2);\\
   	J_2^{(1/3,1/2)}(s),\ s\in \mathcal{S}_{(1/3,1/2)}.
   	\end{cases}
   	\end{equation}
   	
   	From the above formulas one sees that $\widetilde{J}_2(s)$ has possible poles at $s=3/4,$ $s=2/3$ and $s=1/2;$ and the potential poles at $s=3/4,$ $s=2/3$ are at most simple, the possible pole at $s=1/2$ has order at most 2. Moreover, from the above explicit expressions of $\widetilde{J}_2(s),$ we see that $\widetilde{J}_2(s)\cdot\Lambda(s,\tau)^{-1}$ has at most a simple pole at $s=1/2$ if $L_F(1/2,\tau)=0.$
   	\begin{itemize}
   		\item[Case 1:]
   		If $L_F(3/4,\tau)=0,$ then by functional equation we have that $\Lambda(1/4,\tau^{-1})=0.$ Suppose that $\widetilde{J}_2(s)$ has a pole at $s=3/4,$ then from the proceeding explicit expressions, we must have that $\tau^4=1,$ and the singular part of $\widetilde{J}_2(s)$ around $s=3/4$ is a holomorphic function multiplying $\Lambda(4s-3,\tau^4)\Lambda(3s-2,\tau^3).$ Note that $\Lambda(3s-2,\tau^3)\mid_{s=3/4}=\Lambda(1/4,\tau^3)=\Lambda(1/4,\tau^{-1})=0.$ Hence, when $L_F(3/4,\tau)=0,$ $\widetilde{J}_2(s)$ is holomorphic at $s=3/4.$
   		\item[Case 2:]
   		If $L_F(2/3,\tau)=0,$ then by functional equation we have that $\Lambda(1/3,\tau^{-1})=0.$ Suppose that $\widetilde{J}_2(s)$ has a pole at $s=2/3,$ then from the proceeding explicit expressions, we must have that $\tau^3=1,$ and the singular part of $\widetilde{J}_2(s)$ around $s=2/3$ is a holomorphic function multiplying $\Lambda(3s-2,\tau^3)\Lambda(2s-1,\tau^2).$ Note that $\Lambda(2s-1,\tau^2)\mid_{s=2/3}=\Lambda(1/3,\tau^2)=\Lambda(1/3,\tau^{-1})=0.$ Hence, when $L_F(2/3,\tau)=0,$ $\widetilde{J}_2(s)$ is holomorphic at $s=2/3.$
   	\end{itemize}
   	Now the proof of Claim \ref{61claim} is complete.
   \end{proof}
   
   \begin{proof}[Proof of Claim \ref{62claim}]
   	Let $s\in\mathcal{R}(1)^+.$ Let $J_3(s):=\int_{(0)}\int_{(0)}\underset{\kappa_{3}=s-1}{\Res}\mathcal{F}(\boldsymbol{\kappa},s)d\kappa_1d\kappa_2,$ and  $J^1_3(s):=\int_{\mathcal{C}}\int_{\mathcal{C}}\underset{\kappa_{3}=s-1}{\Res}\mathcal{F}(\boldsymbol{\kappa},s)d\kappa_1d\kappa_2.$ By the analytic property of $\underset{\kappa_{3}=s-1}{\Res}\mathcal{F}(\boldsymbol{\kappa},s)$ we see that $J^1_3(s)$ is meromorphic in the domain $\mathcal{R}(1),$ with a possible pole at $s=1.$ Let $s\in\mathcal{R}(1)^-.$ Applying Cauchy integral formula we then see that  
   	\begin{equation}\label{205}
   	J_3^1(s)=\int_{\mathcal{C}}\int_{(0)}\underset{\kappa_{3}=s-1}{\Res}\mathcal{F}(\boldsymbol{\kappa},s)d\kappa_1d\kappa_2+\int_{\mathcal{C}}\underset{\kappa_{1}=1-s}{\Res}\underset{\kappa_{3}=s-1}{\Res}\mathcal{F}(\boldsymbol{\kappa},s)d\kappa_2,
   	\end{equation}
   	where $\underset{\kappa_{1}=1-s}{\Res}\underset{\kappa_{3}=s-1}{\Res}\mathcal{F}(\boldsymbol{\kappa},s)$ is equal to some holomorphic function multiplying the product of $\Lambda(2s-1,\tau^2)^2\Lambda(s,\tau)^2\cdot\Lambda(2-s,\tau^{-1})^{-2}$ and 
   	\begin{equation}\label{206}
   	\frac{\Lambda(s+\kappa_2,\chi_{23}\tau)\Lambda(s-\kappa_2,\chi_{32}\tau)\Lambda(2s-1+\kappa_2,\chi_{23}\tau^2)\Lambda(2s-1-\kappa_2,\chi_{32}\tau^{2})}{\Lambda(1+\kappa_2,\chi_{23})\Lambda(1-\kappa_2,\chi_{32})\Lambda(2-s+\kappa_2,\chi_{23}\tau^{-1})\Lambda(2-s-\kappa_2,\tau^{-1})}.
   	\end{equation}
   	Then after applications of Cauchy integral formula to \eqref{205}, we obtain that 
   	\begin{align*}
   	J_1^1(s)=&\int_{(0)}\int_{(0)}\underset{\kappa_{3}=s-1}{\Res}\mathcal{F}(\boldsymbol{\kappa},s)d\kappa_1d\kappa_2+\int_{(0)}\underset{\kappa_{2}=2-2s}{\Res}\underset{\kappa_{3}=s-1}{\Res}\mathcal{F}(\boldsymbol{\kappa},s)d\kappa_1\\
   	&+\int_{(0)}\underset{\kappa_{2}=2-2s-\kappa_1}{\Res}\underset{\kappa_{3}=s-1}{\Res}\mathcal{F}(\boldsymbol{\kappa},s)d\kappa_1+\int_{(0)}\underset{\kappa_{1}=1-s}{\Res}\underset{\kappa_{3}=s-1}{\Res}\mathcal{F}(\boldsymbol{\kappa},s)d\kappa_2\\
   	&+\underset{\kappa_{2}=1-s}{\Res}\underset{\kappa_{1}=1-s}{\Res}\underset{\kappa_{3}=s-1}{\Res}\mathcal{F}(\boldsymbol{\kappa},s)+\underset{\kappa_{2}=2-2s}{\Res}\underset{\kappa_{1}=1-s}{\Res}\underset{\kappa_{3}=s-1}{\Res}\mathcal{F}(\boldsymbol{\kappa},s),
   	\end{align*}
   	where $s\in\mathcal{R}(1)^-$ and $\underset{\kappa_{2}=2-2s}{\Res}\underset{\kappa_{3}=s-1}{\Res}\mathcal{F}(\boldsymbol{\kappa},s)$ is equal to some holomorphic function multiplying the product of the meromorphic function  $\Lambda(s,\tau)^2$ and 
   	\begin{equation}\label{207}
   	\frac{\Lambda(s+\kappa_1,\chi_{12}\tau)\Lambda(3s-2-\kappa_1,\chi_{21}\tau^3)\Lambda(3s-2,\tau^3)\Lambda(2s-1,\tau^2)}{\Lambda(1-\kappa_1,\chi_{21})\Lambda(3-2s+\kappa_1,\chi_{12})\Lambda(3-2s,\tau^{-2})\Lambda(2-s,\tau^{-1})};
   	\end{equation}
   	and $\underset{\kappa_{2}=2-2s-\kappa_1}{\Res}\underset{\kappa_{3}=s-1}{\Res}\mathcal{F}(\boldsymbol{\kappa},s)$ is equal to some holomorphic function multiplying
   	\begin{equation}\label{208}
   	\frac{\Lambda(s-\kappa_1,\chi_{21}\tau)\Lambda(3s-2+\kappa_1,\chi_{12}\tau^3)\Lambda(3s-2,\tau^{3})\Lambda(2s-1,\tau^2)\Lambda(s,\tau)^2}{\Lambda(1+\kappa_1,\chi_{12})\Lambda(3-2s-\kappa_1,\chi_{21}\tau^{-2})\Lambda(3-2s,\tau^{-2})\Lambda(2-s,\tau^{-1})}.
   	\end{equation}
   	
   	From the formula \eqref{206}, we see that $\underset{\kappa_{2}=2-2s}{\Res}\underset{\kappa_{1}=1-s}{\Res}\underset{\kappa_{3}=s-1}{\Res}\mathcal{F}(\boldsymbol{\kappa},s)$ is equal to some holomorphic function multiplying
   	\begin{equation}\label{209}
   	\frac{\Lambda(4s-3,\tau^4)\Lambda(3s-2,\tau^{3})\Lambda(2s-1,\tau^2)\Lambda(s,\tau)}{\Lambda(4-3s,\tau^{-3})\Lambda(3-2s,\tau^{-2})\Lambda(2-s,\tau^{-1})}.
   	\end{equation}
   	Moreover, by \eqref{206} and the fact that $\Lambda(s+\kappa_2,\chi_{23}\tau)\cdot \Lambda(2-s-\kappa_2,\chi_{32}\tau^{-1})$ is holomorphic at $\kappa_2=1-s$ when $\chi_{23}=\tau^{-1},$ we deduce that  $\underset{\kappa_{2}=1-s}{\Res}\underset{\kappa_{1}=1-s}{\Res}\underset{\kappa_{3}=s-1}{\Res}\mathcal{F}(\boldsymbol{\kappa},s)\equiv0.$
   	
   	We thus see from the proceeding computations of analytic behaviors of the functions  $\underset{\kappa_{3}=s-1}{\Res}\mathcal{F}(\boldsymbol{\kappa},s),$ $\underset{\kappa_{1}=1-s}{\Res}\underset{\kappa_{3}=s-1}{\Res}\mathcal{F}(\boldsymbol{\kappa},s)$ and $\underset{\kappa_{2}=2-2s}{\Res}\underset{\kappa_{1}=1-s}{\Res}\underset{\kappa_{3}=s-1}{\Res}\mathcal{F}(\boldsymbol{\kappa},s),$ that  $\int_{(0)}\int_{(0)}\underset{\kappa_{3}=s-1}{\Res}\mathcal{F}(\boldsymbol{\kappa},s)d\kappa_1d\kappa_2$ and $\int_{(0)}\underset{\kappa_{1}=1-s}{\Res}\underset{\kappa_{3}=s-1}{\Res}\mathcal{F}(\boldsymbol{\kappa},s)d\kappa_2$ admit meromorphic continuation to the domain $1/2<\Re(s)<1;$ and $\underset{\kappa_{2}=2-2s}{\Res}\underset{\kappa_{1}=1-s}{\Res}\underset{\kappa_{3}=s-1}{\Res}\mathcal{F}(\boldsymbol{\kappa},s)$ admits a meromorphic continuation to the domain $\mathcal{R}(1/2)\cup\mathcal{S}_{(1/2,1)},$ with possible simple poles at $s=3/4,$ $s=2/3$ and $s=1/2,$ when $\tau^4=1,$ $\tau^3=1$ and $\tau^2=1,$ respectively, according to \eqref{209}.
   	
   	From \eqref{207} we see that the function $\int_{(0)}\underset{\kappa_{2}=2-2s}{\Res}\underset{\kappa_{3}=s-1}{\Res}\mathcal{F}(\boldsymbol{\kappa},s)d\kappa_1$ admits holomorphic continuation to the domain $2/3<\Re(s)<1.$ From \eqref{208} we see that the function $\int_{(0)}\underset{\kappa_{2}=2-2s-\kappa_1}{\Res}\underset{\kappa_{3}=s-1}{\Res}\mathcal{F}(\boldsymbol{\kappa},s)d\kappa_1$ admits holomorphic continuation to the domain $2/3<\Re(s)<1.$ Then combining these with  \eqref{207} and \eqref{209} one sees that $J_3^1(s)$ admits a holomorphic continuation to the domain $2/3<\Re(s)<1.$ Denote by $J_3^{(2/3,1)}(s)$ this continuation, where $2/3<\Re(s)<1.$
   	
   	Let $s\in \mathcal{R}(2/3)^+,$ then by Cauchy integral formula we have 
   	\begin{equation}\label{210}
   	\int_{(0)}\underset{\kappa_{2}=2-2s-\kappa_1}{\Res}\underset{\kappa_{3}=s-1}{\Res}\mathcal{F}(\boldsymbol{\kappa},s)d\kappa_1=\int_{\mathcal{C}}\underset{\kappa_{2}=2-2s-\kappa_1}{\Res}\underset{\kappa_{3}=s-1}{\Res}\mathcal{F}(\boldsymbol{\kappa},s)d\kappa_1.
   	\end{equation}
   	Likewise, for $s\in \mathcal{R}(2/3)^+,$ by \eqref{207}, $\int_{(0)}\underset{\kappa_{2}=2-2s}{\Res}\underset{\kappa_{3}=s-1}{\Res}\mathcal{F}(\boldsymbol{\kappa},s)d\kappa_1$ is equal to 
   	\begin{equation}\label{211}
   	\int_{\mathcal{C}}\underset{\kappa_{2}=2-2s}{\Res}\underset{\kappa_{3}=s-1}{\Res}\mathcal{F}(\boldsymbol{\kappa},s)d\kappa_1-\underset{\kappa_{1}=3s-2}{\Res}\underset{\kappa_{2}=2-2s}{\Res}\underset{\kappa_{3}=s-1}{\Res}\mathcal{F}(\boldsymbol{\kappa},s),
   	\end{equation}
   	where $\underset{\kappa_3=3s-2}{\Res}\underset{\kappa_{2}=2-2s}{\Res}\underset{\kappa_{1}=s-1}{\Res}\mathcal{F}(\boldsymbol{\kappa},s)$ equals some holomorphic function multiplying 
   	\begin{equation}\label{212}
   	\frac{\Lambda(4s-2,\tau^4)\Lambda(3s-2,\tau^{3})\Lambda(2s-1,\tau^2)\Lambda(s,\tau)^2}{\Lambda(3-3s,\tau^{-3})\Lambda(3-2s,\tau^{-2})\Lambda(2-s,\tau^{-1})\Lambda(1+s,\tau)}.
   	\end{equation}
   	
   	Then according to \eqref{207}, \eqref{208}, \eqref{210}, \eqref{211} and \eqref{212}, we see that $J_1^{(2/3,1)}(s)$ admits a meromorphic continuation to the domain $\mathcal{R}(2/3),$ with a possible pole at $s=2/3$ when $\tau^3=1.$ Denote by  $J_3^{2/3}(s)$ this continuation, $s\in \mathcal{R}(2/3).$ Now let $s\in \mathcal{R}(2/3)^-.$ Then we have that
   	\begin{align*}
   	J_3^{2/3}(s)=&\int_{(0)}\int_{(0)}\underset{\kappa_{3}=s-1}{\Res}\mathcal{F}(\boldsymbol{\kappa},s)d\kappa_1d\kappa_2+\int_{\mathcal{C}}\underset{\kappa_{2}=2-2s}{\Res}\underset{\kappa_{3}=s-1}{\Res}\mathcal{F}(\boldsymbol{\kappa},s)d\kappa_1\\
   	&+\int_{\mathcal{C}}\underset{\kappa_{2}=2-2s-\kappa_1}{\Res}\underset{\kappa_{3}=s-1}{\Res}\mathcal{F}(\boldsymbol{\kappa},s)d\kappa_1+\int_{(0)}\underset{\kappa_{1}=1-s}{\Res}\underset{\kappa_{3}=s-1}{\Res}\mathcal{F}(\boldsymbol{\kappa},s)d\kappa_2\\
   	&+\underset{\kappa_{2}=2-2s}{\Res}\underset{\kappa_{1}=1-s}{\Res}\underset{\kappa_{3}=s-1}{\Res}\mathcal{F}(\boldsymbol{\kappa},s)-\underset{\kappa_{1}=3s-2}{\Res}\underset{\kappa_{2}=2-2s}{\Res}\underset{\kappa_{3}=s-1}{\Res}\mathcal{F}(\boldsymbol{\kappa},s).
   	\end{align*}
   	According to \eqref{208}, \eqref{209}, \eqref{210}, \eqref{211} and \eqref{212}, the terms in the right hand side of the above formula are holomorphic in $1/2<\Re(s)<2/3$ except the term $\int_{\mathcal{C}}\underset{\kappa_{2}=2-2s-\kappa_1}{\Res}\underset{\kappa_{3}=s-1}{\Res}\mathcal{F}(\boldsymbol{\kappa},s)d\kappa_1,$ which is equal to, by Cauchy integral formula, that 
   	\begin{equation}\label{213}
   	\int_{(0)}\underset{\kappa_{2}=2-2s-\kappa_1}{\Res}\underset{\kappa_{3}=s-1}{\Res}\mathcal{F}(\boldsymbol{\kappa},s)d\kappa_1+\underset{\kappa_{1}=2-3s}{\Res}\underset{\kappa_{2}=2-2s-\kappa_1}{\Res}\underset{\kappa_{3}=s-1}{\Res}\mathcal{F}(\boldsymbol{\kappa},s),
   	\end{equation}
   	where $s\in\mathcal{R}(2/3)^-.$ By \eqref{208}, one sees that $\underset{\kappa_{2}=2-3s}{\Res}\underset{\kappa_{3}=2-2s-\kappa_2}{\Res}\underset{\kappa_{1}=s-1}{\Res}\mathcal{F}(\boldsymbol{\kappa},s)$ is equal to some holomorphic function multiplying 
   	\begin{equation}\label{214}
   	\frac{\Lambda(4s-2,\tau^4)\Lambda(3s-2,\tau^{3})\Lambda(2s-1,\tau^2)\Lambda(s,\tau)^2}{\Lambda(3-3s,\tau^{-3})\Lambda(3-2s,\tau^{-2})\Lambda(2-s,\tau^{-1})\Lambda(1+s,\tau)}.
   	\end{equation}
   	By \eqref{208}, \eqref{213} and \eqref{214} one sees that $\int_{\mathcal{C}}\underset{\kappa_{2}=2-2s-\kappa_1}{\Res}\underset{\kappa_{3}=s-1}{\Res}\mathcal{F}(\boldsymbol{\kappa},s)d\kappa_1$ admits a meromorphic continuation to $\mathcal{S}_{(1/3,2/3)}$ with a at most double pole at $s=1/2$ when $\tau^2=1.$ Hence we obtain a meromorphic continuation of $J_1^{2/3}(s)$ to the strip $1/2<\Re(s)<2/3.$ Denote by $J_3^{(1/2,2/3)}$ this continuation, namely,
   	\begin{align*}
   	J_3^{(1/2,2/3)}(s)=&\int_{(0)}\int_{(0)}\underset{\kappa_{3}=s-1}{\Res}\mathcal{F}(\boldsymbol{\kappa},s)d\kappa_1d\kappa_2+\int_{(0)}\underset{\kappa_{2}=2-2s}{\Res}\underset{\kappa_{3}=s-1}{\Res}\mathcal{F}(\boldsymbol{\kappa},s)d\kappa_1\\
   	&+\int_{(0)}\underset{\kappa_{2}=2-2s-\kappa_1}{\Res}\underset{\kappa_{3}=s-1}{\Res}\mathcal{F}(\boldsymbol{\kappa},s)d\kappa_1+\int_{(0)}\underset{\kappa_{1}=1-s}{\Res}\underset{\kappa_{3}=s-1}{\Res}\mathcal{F}(\boldsymbol{\kappa},s)d\kappa_2\\
   	&+\underset{\kappa_{2}=2-2s}{\Res}\underset{\kappa_{1}=1-s}{\Res}\underset{\kappa_{3}=s-1}{\Res}\mathcal{F}(\boldsymbol{\kappa},s)-\underset{\kappa_{1}=3s-2}{\Res}\underset{\kappa_{2}=2-2s}{\Res}\underset{\kappa_{3}=s-1}{\Res}\mathcal{F}(\boldsymbol{\kappa},s)\\
   	&+\underset{\kappa_{1}=2-3s}{\Res}\underset{\kappa_{2}=2-2s-\kappa_1}{\Res}\underset{\kappa_{3}=s-1}{\Res}\mathcal{F}(\boldsymbol{\kappa},s).
   	\end{align*}
   	One sees clearly that the terms in the right hand side of the above expression are meromorphic in $\mathcal{R}(1/2),$ except the term $\int_{(0)}\underset{\kappa_{1}=1-s}{\Res}\underset{\kappa_{3}=s-1}{\Res}\mathcal{F}(\boldsymbol{\kappa},s)d\kappa_2,$ which by Cauchy integral formula and \eqref{206}, is equal to
   	\begin{equation}\label{215} 
   	\int_{\mathcal{C}}\underset{\kappa_{1}=1-s}{\Res}\underset{\kappa_{3}=s-1}{\Res}\mathcal{F}(\boldsymbol{\kappa},s)d\kappa_2-\underset{\kappa_{2}=2s-1}{\Res}\underset{\kappa_{1}=1-s}{\Res}\underset{\kappa_{3}=s-1}{\Res}\mathcal{F}(\boldsymbol{\kappa},s),
   	\end{equation}
   	where $s\in\mathcal{R}(1/2)^+.$ By formula \eqref{206}, we see that $\underset{\kappa_{2}=2s-1}{\Res}\underset{\kappa_{1}=1-s}{\Res}\underset{\kappa_{3}=s-1}{\Res}\mathcal{F}(\boldsymbol{\kappa},s)$ is equal to some holomorphic function multiplying
   	\begin{equation}\label{216}
   	\frac{\Lambda(4s-2,\tau^4)\Lambda(3s-1,\tau^{3})\Lambda(2s-1,\tau^2)^2\Lambda(s,\tau)^2\Lambda(1-s,\tau^{-1})}{\Lambda(3-3s,\tau^{-3})\Lambda(2s,\tau^{2})\Lambda(1+s,\tau)\Lambda(2-2s,\tau^{-2})\Lambda(2-s,\tau^{-1})^2}.
   	\end{equation}
   	We then apply the functional equation $\Lambda(2-2s,\tau^{-2})\sim \Lambda(2s-1,\tau^{2})$ to \eqref{216} to see that $\underset{\kappa_{2}=2s-1}{\Res}\underset{\kappa_{1}=1-s}{\Res}\underset{\kappa_{3}=s-1}{\Res}\mathcal{F}(\boldsymbol{\kappa},s)$ equals some holomorphic function multiplying
   	\begin{equation}\label{217}
   	\frac{\Lambda(4s-2,\tau^4)\Lambda(3s-1,\tau^{3})\Lambda(2s-1,\tau^2)\Lambda(s,\tau)^2\Lambda(1-s,\tau^{-1})}{\Lambda(3-3s,\tau^{-3})\Lambda(2s,\tau^{2})\Lambda(1+s,\tau)\Lambda(2-s,\tau^{-1})^2}.
   	\end{equation}
   	Note that when $s\in\mathcal{R}(1/2)^-,$ $2s$ lies in a zero-free region of $\Lambda(s,\tau^2).$ Then by \eqref{215} and \eqref{217} we conclude that $J_1^{(1/2,2/3)}(s)$ admits a meromorphic continuation to the region $\mathcal{R}(1/2).$ Denote by $J_3^{1/2}(s)$ this continuation, then 
   	\begin{align*}
   	J_1^{1/2}(s)=&\int_{\mathcal{C}}\int_{\mathcal{C}}\underset{\kappa_{3}=s-1}{\Res}\mathcal{F}(\boldsymbol{\kappa},s)d\kappa_1d\kappa_2+\int_{(0)}\underset{\kappa_{2}=2-2s}{\Res}\underset{\kappa_{3}=s-1}{\Res}\mathcal{F}(\boldsymbol{\kappa},s)d\kappa_1\\
   	&+\int_{(0)}\underset{\kappa_{2}=2-2s-\kappa_1}{\Res}\underset{\kappa_{3}=s-1}{\Res}\mathcal{F}(\boldsymbol{\kappa},s)d\kappa_1+\int_{\mathcal{C}}\underset{\kappa_{1}=1-s}{\Res}\underset{\kappa_{3}=s-1}{\Res}\mathcal{F}(\boldsymbol{\kappa},s)d\kappa_2\\
   	&+\underset{\kappa_{2}=2-2s}{\Res}\underset{\kappa_{1}=1-s}{\Res}\underset{\kappa_{3}=s-1}{\Res}\mathcal{F}(\boldsymbol{\kappa},s)-\underset{\kappa_{1}=3s-2}{\Res}\underset{\kappa_{2}=2-2s}{\Res}\underset{\kappa_{3}=s-1}{\Res}\mathcal{F}(\boldsymbol{\kappa},s)\\
   	&+\underset{\kappa_{1}=2-3s}{\Res}\underset{\kappa_{2}=2-2s-\kappa_1}{\Res}\underset{\kappa_{3}=s-1}{\Res}\mathcal{F}(\boldsymbol{\kappa},s)-\underset{\kappa_{2}=2s-1}{\Res}\underset{\kappa_{1}=1-s}{\Res}\underset{\kappa_{3}=s-1}{\Res}\mathcal{F}(\boldsymbol{\kappa},s).
   	\end{align*}
   	Let $s\in\mathcal{R}(1/2)^-.$ Then by Cauchy's integral formula we have 
   	\begin{align*}
   	\int_{\mathcal{C}}\int_{\mathcal{C}}\underset{\kappa_{3}=s-1}{\Res}&\mathcal{F}(\boldsymbol{\kappa},s)d\kappa_1d\kappa_2=\int_{(0)}\int_{(0)}\underset{\kappa_{3}=s-1}{\Res}\mathcal{F}(\boldsymbol{\kappa},s)d\kappa_2d\kappa_1\\
   	&+\int_{\mathcal{C}}\underset{\kappa_{2}=1-2s}{\Res}\underset{\kappa_{3}=s-1}{\Res}\mathcal{F}(\boldsymbol{\kappa},s)d\kappa_1+\int_{(0)}\underset{\kappa_{1}=1-2s}{\Res}\underset{\kappa_{3}=s-1}{\Res}\mathcal{F}(\boldsymbol{\kappa},s)d\kappa_2.
   	\end{align*}
   	Also, by \eqref{171.} we have $\int_{\mathcal{C}}\underset{\kappa_{1}=1-s}{\Res}\underset{\kappa_{3}=s-1}{\Res}\mathcal{F}(\boldsymbol{\kappa},s)d\kappa_2$ equal to
   	\begin{align*}
   	\int_{(0)}\underset{\kappa_{1}=1-s}{\Res}\underset{\kappa_{3}=s-1}{\Res}\mathcal{F}(\boldsymbol{\kappa},s)d\kappa_2+\underset{\kappa_{2}=1-2s}{\Res}\underset{\kappa_{1}=1-s}{\Res}\underset{\kappa_{3}=s-1}{\Res}\mathcal{F}(\boldsymbol{\kappa},s).
   	\end{align*}
   	Substituting the above equalities into the expression of $J_3^{1/2}(s)$ we then obtain a continuation of $J_3^{1/2}(s)$ into $1/3<\Re(s)<1/2.$ Denote this continuation by $J_3^{(1/2,2/3)}(s),$ then 
   	\begin{align*}
   	J_3^{(1/2,2/3)}(s)=&\int_{(0)}\int_{(0)}\underset{\kappa_{3}=s-1}{\Res}\mathcal{F}(\boldsymbol{\kappa},s)d\kappa_1d\kappa_2+\int_{(0)}\underset{\kappa_{2}=2-2s}{\Res}\underset{\kappa_{3}=s-1}{\Res}\mathcal{F}(\boldsymbol{\kappa},s)d\kappa_1\\
   	&+\int_{(0)}\underset{\kappa_{2}=2-2s-\kappa_1}{\Res}\underset{\kappa_{3}=s-1}{\Res}\mathcal{F}(\boldsymbol{\kappa},s)d\kappa_1+\int_{(0)}\underset{\kappa_{1}=1-s}{\Res}\underset{\kappa_{3}=s-1}{\Res}\mathcal{F}(\boldsymbol{\kappa},s)d\kappa_2\\
   	&+\underset{\kappa_{2}=2-2s}{\Res}\underset{\kappa_{1}=1-s}{\Res}\underset{\kappa_{3}=s-1}{\Res}\mathcal{F}(\boldsymbol{\kappa},s)-\underset{\kappa_{1}=3s-2}{\Res}\underset{\kappa_{2}=2-2s}{\Res}\underset{\kappa_{3}=s-1}{\Res}\mathcal{F}(\boldsymbol{\kappa},s)\\
   	&+\underset{\kappa_{1}=2-3s}{\Res}\underset{\kappa_{2}=2-2s-\kappa_1}{\Res}\underset{\kappa_{3}=s-1}{\Res}\mathcal{F}(\boldsymbol{\kappa},s)-\underset{\kappa_{2}=2s-1}{\Res}\underset{\kappa_{1}=1-s}{\Res}\underset{\kappa_{3}=s-1}{\Res}\mathcal{F}(\boldsymbol{\kappa},s)\\
   	&+\int_{\mathcal{C}}\underset{\kappa_{2}=1-2s}{\Res}\underset{\kappa_{3}=s-1}{\Res}\mathcal{F}(\boldsymbol{\kappa},s)d\kappa_1+\int_{(0)}\underset{\kappa_{1}=1-2s}{\Res}\underset{\kappa_{3}=s-1}{\Res}\mathcal{F}(\boldsymbol{\kappa},s)d\kappa_2\\
   	&+\underset{\kappa_{2}=1-2s}{\Res}\underset{\kappa_{1}=1-s}{\Res}\underset{\kappa_{3}=s-1}{\Res}\mathcal{F}(\boldsymbol{\kappa},s).
   	\end{align*}
   	Thus we obtain a meromorphic continuation of $J_3(s)$ to the area $\mathcal{S}_{(1/3,\infty)}:$ 
   	\begin{equation}\label{claim3}
   	\widetilde{J}_3(s)=\begin{cases}
   	J_3(s),\ s\in \mathcal{S}_{(1,+\infty)};\\
   	J_3^1(s),\ s\in\mathcal{R}(1);\\
   	J_3^{(2/3,1)}(s),\ s\in \mathcal{S}_{(2/3,1)};\\
   	J_3^{2/3}(s),\ s\in \mathcal{R}(2/3);\\
   	J_3^{(1/2,2/3)}(s),\ s\in \mathcal{S}_{(1/2,2/3)};\\
   	J_3^{1/2}(s),\ s\in \mathcal{R}(1/2);\\
   	J_3^{(1/3,1/2)}(s),\ s\in \mathcal{S}(1/3,1/2).
   	\end{cases}
   	\end{equation}
   	
   	From the above formulas one sees that $\widetilde{J}_3(s)$ has possible poles at $s=3/4,$ $s=2/3$ and $s=1/2;$ and the potential poles at $s=3/4,$ $s=2/3$ are at most simple, the possible pole at $s=1/2$ has order at most 2. Moreover, from the above explicit expressions of $\widetilde{J}_3(s),$ we see that $\widetilde{J}_3(s)\cdot\Lambda(s,\tau)^{-1}$ has at most a simple pole at $s=1/2$ if $L_F(1/2,\tau)=0.$
   	\begin{itemize}
   		\item[Case 1:]
   		If $L_F(3/4,\tau)=0,$ then by functional equation we have that $\Lambda(1/4,\tau^{-1})=0.$ Suppose that $\widetilde{J}_3(s)$ has a pole at $s=3/4,$ then from the proceeding explicit expressions, we must have that $\tau^4=1,$ and the singular part of $\widetilde{J}_3(s)$ around $s=3/4$ is a holomorphic function multiplying $\Lambda(4s-3,\tau^4)\Lambda(3s-2,\tau^3).$ Note that $\Lambda(3s-2,\tau^3)\mid_{s=3/4}=\Lambda(1/4,\tau^3)=\Lambda(1/4,\tau^{-1})=0.$ Hence, when $L_F(3/4,\tau)=0,$ $\widetilde{J}_3(s)$ is holomorphic at $s=3/4.$
   		\item[Case 2:]
   		If $L_F(2/3,\tau)=0,$ then by functional equation we have that $\Lambda(1/3,\tau^{-1})=0.$ Suppose that $\widetilde{J}_3(s)$ has a pole at $s=2/3,$ then from the proceeding explicit expressions, we must have that $\tau^3=1,$ and the singular part of $\widetilde{J}_3(s)$ around $s=2/3$ is a holomorphic function multiplying $\Lambda(3s-2,\tau^3)\Lambda(2s-1,\tau^2).$ Note that $\Lambda(2s-1,\tau^2)\mid_{s=2/3}=\Lambda(1/3,\tau^2)=\Lambda(1/3,\tau^{-1})=0.$ Hence, when $L_F(2/3,\tau)=0,$ $\widetilde{J}_3(s)$ is holomorphic at $s=2/3.$
   	\end{itemize}
   	Now the proof of Claim \ref{62claim} is complete.
   \end{proof}
   
   \begin{proof}[Proof of Claim \ref{63claim}]
   	Let $s\in\mathcal{R}(1)^+.$ Let $J_{12}(s)=\int_{(0)}\underset{\kappa_{1}=s-1}{\Res}\underset{\kappa_{2}=s-1}{\Res}\mathcal{F}(\boldsymbol{\kappa},s)d\kappa_3,$ and  $J_{12}^1(s)=\int_{\mathcal{C}}\underset{\kappa_{1}=s-1}{\Res}\underset{\kappa_{2}=s-1}{\Res}\mathcal{F}(\boldsymbol{\kappa},s)d\kappa_3.$ Then by \eqref{172.} one sees that $J_{12}^1(s)$ is meromorphic in the region $\mathcal{R}(1),$ with a possible pole at $s=1.$
   	
   	Let $s\in \mathcal{R}(1)^-.$ Applying Cauchy integral formula we then have that
   	\begin{equation}\label{219}
   	J_{12}^1(s)=\int_{(0)}\underset{\kappa_{1}=s-1}{\Res}\underset{\kappa_{2}=s-1}{\Res}\mathcal{F}(\boldsymbol{\kappa},s)d\kappa_3+\underset{\kappa_{3}=3-3s}{\Res}\underset{\kappa_{1}=s-1}{\Res}\underset{\kappa_{2}=s-1}{\Res}\mathcal{F}(\boldsymbol{\kappa},s),
   	\end{equation}
   	where $\underset{\kappa_{3}=3-3s}{\Res}\underset{\kappa_{1}=s-1}{\Res}\underset{\kappa_{2}=s-1}{\Res}\mathcal{F}(\boldsymbol{\kappa},s)$ equals some holomorphic function multiplying
   	\begin{equation}\label{220}
   	\frac{\Lambda(4s-3,\tau^4)\Lambda(3s-2,\tau^{3})\Lambda(2s-1,\tau^2)\Lambda(s,\tau)}{\Lambda(4-3s,\tau^{-3})\Lambda(3-2s,\tau^{-2})\Lambda(2-s,\tau^{-1})}.
   	\end{equation}
   	Then one sees,  by \eqref{172.} and \eqref{220}, that  $\int_{(0)}\underset{\kappa_{1}=s-1}{\Res}\underset{\kappa_{2}=s-1}{\Res}\mathcal{F}(\boldsymbol{\kappa},s)d\kappa_3$ and the function  $\underset{\kappa_{3}=3-3s}{\Res}\underset{\kappa_{1}=s-1}{\Res}\underset{\kappa_{2}=s-1}{\Res}\mathcal{F}(\boldsymbol{\kappa},s)$ are meromorphic in the strip $2/3<\Re(s)<1,$ with possible simple poles at $s=3/4$ if $\tau^4=1.$ Hence, by \eqref{219}, we obtain a meromorphic continuation of $J_{12}^1(s)$ to the strip $2/3<\Re(s)<1,$ with possible simple poles at $s=3/4$ if $\tau^4=1.$ Denote by $J_{12}^{(2/3,1)}(s)$ this continuation. 
   	
   	Let $s\in\mathcal{R}(2/3)^+.$ Applying Cauchy integral formula to \eqref{172.} to see that the function  $\int_{(0)}\underset{\kappa_{1}=s-1}{\Res}\underset{\kappa_{2}=s-1}{\Res}\mathcal{F}(\boldsymbol{\kappa},s)d\kappa_3$ is equal to 
   	\begin{equation}\label{221}
   	\int_{\mathcal{C}}\underset{\kappa_{1}=s-1}{\Res}\underset{\kappa_{2}=s-1}{\Res}\mathcal{F}(\boldsymbol{\kappa},s)d\kappa_3-\underset{\kappa_{3}=2-3s}{\Res}\underset{\kappa_{1}=s-1}{\Res}\underset{\kappa_{2}=s-1}{\Res}\mathcal{F}(\boldsymbol{\kappa},s),
   	\end{equation}
   	and $\underset{\kappa_{3}=2-3s}{\Res}\underset{\kappa_{1}=s-1}{\Res}\underset{\kappa_{2}=s-1}{\Res}\mathcal{F}(\boldsymbol{\kappa},s)$ is equal to some holomorphic function multiplying 
   	\begin{equation}\label{222}
   	\frac{\Lambda(4s-3,\tau^4)\Lambda(3s-2,\tau^{3})\Lambda(2s-1,\tau^2)\Lambda(s,\tau)^2}{\Lambda(3-3s,\tau^{-3})\Lambda(1+s,\tau)\Lambda(3-2s,\tau^{-2})\Lambda(2-s,\tau^{-1})}.
   	\end{equation}
   	Then by \eqref{221}, \eqref{222} and the fact that $\int_{\mathcal{C}}\underset{\kappa_{1}=s-1}{\Res}\underset{\kappa_{2}=s-1}{\Res}\mathcal{F}(\boldsymbol{\kappa},s)d\kappa_3$ is holomorphic in $\mathcal{R}(2/3),$ we obtain a meromorphic continuation of $J_{12}^{(2/3,1)}(s)$ to the region $\mathcal{R}(2/3),$ with a possible simple pole at $s=2/3,$ if $\tau^3=1.$ Denote by $J_{12}^{2/3}(s)$ the continuation.
   	
   	Let $s\in\mathcal{R}(2/3)^-.$ Then by \eqref{219} and \eqref{221} one has 
   	\begin{align*}
   	J_{12}^{2/3}(s)=&\int_{\mathcal{C}}\underset{\kappa_{1}=s-1}{\Res}\underset{\kappa_{2}=s-1}{\Res}\mathcal{F}(\boldsymbol{\kappa},s)d\kappa_3-\underset{\kappa_{3}=2-3s}{\Res}\underset{\kappa_{1}=s-1}{\Res}\underset{\kappa_{2}=s-1}{\Res}\mathcal{F}(\boldsymbol{\kappa},s)\\
   	&+\underset{\kappa_{3}=3-3s}{\Res}\underset{\kappa_{1}=s-1}{\Res}\underset{\kappa_{2}=s-1}{\Res}\mathcal{F}(\boldsymbol{\kappa},s).
   	\end{align*}
   	Since the right hand side is meromorphic in the strip $\mathcal{S}_{(0,2/3)},$ with a possible simple pole at $s=1/2$ if $\tau^2=1.$ We thus obtain a meromorphic continuation of $J_{12}^{2/3}(s)$ to the region $0<\Re(s)<2/3,$ with a possible simple pole at $s=1/2$ if $\tau^2=1.$ Denote by $J_{12}^{(1/3,2/3)}(s)$ this continuation. Thus we obtain a meromorphic continuation of $J_{12}(s)$ to the area $\mathcal{S}_{(1/3,\infty)}:$ 
   	\begin{equation}\label{claim4}
   	\widetilde{J}_{12}(s)=\begin{cases}
   	J_{12}(s),\ s\in \mathcal{S}_{(1,+\infty)};\\
   	J_{12}^1(s),\ s\in\mathcal{R}(1);\\
   	J_{12}^{(2/3,1)}(s),\ s\in \mathcal{S}_{(2/3,1)};\\
   	J_{12}^{2/3}(s),\ s\in \mathcal{R}(2/3);\\
   	J_{12}^{(1/3,2/3)}(s),\ s\in \mathcal{R}(1/2)\cup\mathcal{S}_{(1/3,2/3)}.
   	\end{cases}
   	\end{equation}
   	
   	From the above formulas one sees that $\widetilde{J}_{12}(s)$ has possible poles at $s=3/4,$ $s=2/3$ and $s=1/2;$ and these potential poles are all at most simple. Moreover, from the above explicit expressions of $\widetilde{J}_{12}(s),$ we see that $\widetilde{J}_{12}(s)\cdot\Lambda(s,\tau)^{-1}$ has at most a simple pole at $s=1/2.$ We discuss the other two possible poles separately.
   	\begin{itemize}
   		\item[Case 1:]
   		If $L_F(3/4,\tau)=0,$ then by functional equation we have that $\Lambda(1/4,\tau^{-1})=0.$ Suppose that $\widetilde{J}_{12}(s)$ has a pole at $s=3/4,$ then from the proceeding explicit expressions, we must have that $\tau^4=1,$ and the singular part of $\widetilde{J}_{12}(s)$ around $s=3/4$ is a holomorphic function multiplying $\Lambda(4s-3,\tau^4)\Lambda(3s-2,\tau^3).$ Note that $\Lambda(3s-2,\tau^3)\mid_{s=3/4}=\Lambda(1/4,\tau^3)=\Lambda(1/4,\tau^{-1})=0.$ Hence, when $L_F(3/4,\tau)=0,$ $\widetilde{J}_{12}(s)$ is holomorphic at $s=3/4.$
   		\item[Case 2:]
   		If $L_F(2/3,\tau)=0,$ then by functional equation we have that $\Lambda(1/3,\tau^{-1})=0.$ Suppose that $\widetilde{J}_{12}(s)$ has a pole at $s=2/3,$ then from the proceeding explicit expressions, we must have that $\tau^3=1,$ and the singular part of $\widetilde{J}_{12}(s)$ around $s=2/3$ is a holomorphic function multiplying $\Lambda(3s-2,\tau^3)\Lambda(2s-1,\tau^2).$ Note that $\Lambda(2s-1,\tau^2)\mid_{s=2/3}=\Lambda(1/3,\tau^2)=\Lambda(1/3,\tau^{-1})=0.$ Hence, when $L_F(2/3,\tau)=0,$ $\widetilde{J}_{12}(s)$ is holomorphic at $s=2/3.$
   	\end{itemize}
   	Now the proof of Claim \ref{63claim} is complete.
   \end{proof}
   
   \begin{proof}[Proof of Claim \ref{64claim}]
   	Let $s\in\mathcal{R}(1)^+.$ Let $J_{13}(s)=\int_{(0)}\underset{\kappa_{1}=s-1}{\Res}\underset{\kappa_{3}=s-1}{\Res}\mathcal{F}(\boldsymbol{\kappa},s)d\kappa_2,$ and  $J_{13}^1(s)=\int_{\mathcal{C}}\underset{\kappa_{1}=s-1}{\Res}\underset{\kappa_{3}=s-1}{\Res}\mathcal{F}(\boldsymbol{\kappa},s)d\kappa_2.$ Then by \eqref{171.} one sees that $J_{13}^1(s)$ is meromorphic in the region $\mathcal{R}(1),$ with a possible pole at $s=1.$
   	
   	Let $s\in \mathcal{R}(1)^-.$ Applying Cauchy integral formula we then have that
   	\begin{equation}\label{224}
   	J_{13}^1(s)=\int_{(0)}\underset{\kappa_{1}=s-1}{\Res}\underset{\kappa_{3}=s-1}{\Res}\mathcal{F}(\boldsymbol{\kappa},s)d\kappa_2+R_1(s)+R_2(s),
   	\end{equation}
   	where $R_1(s):=\underset{\kappa_{2}=2-2s}{\Res}\underset{\kappa_{1}=s-1}{\Res}\underset{\kappa_{3}=s-1}{\Res}\mathcal{F}(\boldsymbol{\kappa},s),$ and  $R_2(s)$ denotes the meromorphic function $\underset{\kappa_{2}=3-3s}{\Res}\underset{\kappa_{1}=s-1}{\Res}\underset{\kappa_{3}=s-1}{\Res}\mathcal{F}(\boldsymbol{\kappa},s).$ 
   	Then by \eqref{171.} $R_1(s)\equiv0$ if $\chi_{23}\tau^2\neq 1.$ Let $\chi_{23}\tau^2=1.$ Then the function $G(\kappa_2)=\Lambda_F(2s-1+\kappa_2)\cdot\Lambda_F(3-2s-\kappa_2)^{-1}$ is holomorphic at $\kappa_2=2-2s.$ Hence $R_1(s)\equiv 0.$ Also, according to \eqref{171.}, 
   	\begin{equation}\label{225}
   	R_2(s)\sim\frac{\Lambda(4s-3,\tau^4)\Lambda(3s-2,\tau^{3})\Lambda(2s-1,\tau^2)\Lambda(s,\tau)}{\Lambda(4-3s,\tau^{-3})\Lambda(3-2s,\tau^{-2})\Lambda(2-s,\tau^{-1})}.
   	\end{equation} 
   	Thanks to the uniform zero-free region of Rankin-Selberg L-functions defined in Section \ref{7.11}, the function  $\underset{\kappa_{1}=s-1}{\Res}\underset{\kappa_{3}=s-1}{\Res}\mathcal{F}(\boldsymbol{\kappa},s)$ is holomorphic in the domain $\mathcal{R}(1-s).$ Then we can apply Cauchy integral formula to obtain that 
   	\begin{equation}\label{226}
   	\int_{(0)}\underset{\kappa_{1}=s-1}{\Res}\underset{\kappa_{3}=s-1}{\Res}\mathcal{F}(\boldsymbol{\kappa},s)d\kappa_2=\int_{(1-s)}\underset{\kappa_{1}=s-1}{\Res}\underset{\kappa_{3}=s-1}{\Res}\mathcal{F}(\boldsymbol{\kappa},s)d\kappa_2,
   	\end{equation}
   	where the integral on the right hand side is taken over $(1-s):=\{z\in\mathbb{C}:\ \Re(z)=1-\Re(s)\}.$ Let $\kappa_2'=\kappa_2+s-1,$ $\kappa_1'=\kappa_1$ and $\kappa_3'=\kappa_3.$ Denote by $\boldsymbol{\kappa}'=(\kappa_1',\kappa_2',\kappa_3').$ Then $d\kappa_j'=d\kappa_j,$ $1\leq j\leq 3.$ Hence we have 
   	\begin{equation}\label{227}
   	\int_{(0)}\underset{\kappa_{1}=s-1}{\Res}\underset{\kappa_{3}=s-1}{\Res}\mathcal{F}(\boldsymbol{\kappa},s)d\kappa_2=\int_{(0)}\underset{\kappa_{1}'=s-1}{\Res}\underset{\kappa_{3}'=s-1}{\Res}\mathcal{F}(\boldsymbol{\kappa}',s)d\kappa_2',
   	\end{equation}
   	where by \eqref{171.}, $\underset{\kappa_{1}'=s-1}{\Res}\underset{\kappa_{3}'=s-1}{\Res}\mathcal{F}(\boldsymbol{\kappa}',s)$ is equal to some holomorphic function multiplying the product of $\Lambda(2s-1,\tau^2)^2\Lambda(s,\tau)^2\cdot\Lambda(2-s,\tau^{-1})^{-2}$ and 
   	\begin{equation}\label{228}
   	\frac{\Lambda(s+\kappa_2',\chi_{23}\tau)\Lambda(s-\kappa_2',\chi_{32}\tau)\Lambda(2s-1+\kappa_2',\chi_{23}\tau^{2})\Lambda(2s-1-\kappa_2',\chi_{32}\tau^2)}{\Lambda(1+\kappa_2',\chi_{23})\Lambda(1-\kappa_2',\chi_{32})\Lambda(2-s+\kappa_2',\chi_{23}\tau^{-1})\Lambda(2-s-\kappa_2',\chi_{32}\tau^{-1})}.
   	\end{equation}
   	Then from \eqref{227} and \eqref{228}, we conclude that $\int_{(0)}\underset{\kappa_{1}=s-1}{\Res}\underset{\kappa_{3}=s-1}{\Res}\mathcal{F}(\boldsymbol{\kappa},s)d\kappa_2$ admits a meromorphic continuation to the strip $1/2<\Re(s)<1.$ Combining this with equations  \eqref{224} and \eqref{225}, we then obtain a meromorphic continuation of $J_{13}^1(s)$ to the area $\mathcal{S}_{(1/2,1)}.$ Denote by $J_{13}^{(1/2,1)}$ this continuation. Then 
   	\begin{equation}\label{229}
   	J_{13}^{(1/2,1)}(s)=\int_{(0)}\underset{\kappa_{1}'=s-1}{\Res}\underset{\kappa_{3}'=s-1}{\Res}\mathcal{F}(\boldsymbol{\kappa}',s)d\kappa_2'+\underset{\kappa_{2}=3-3s}{\Res}\underset{\kappa_{1}=s-1}{\Res}\underset{\kappa_{3}=s-1}{\Res}\mathcal{F}(\boldsymbol{\kappa},s).
   	\end{equation}
   	Let $s\in\mathcal{R}(1/2)^+.$ Applying Cauchy integral formula to \eqref{229} to obtain that 
   	\begin{equation}\label{230}
   	J_{13}^{(1/2,1)}(s)=\int_{\mathcal{C}}\underset{\kappa_{1}'=s-1}{\Res}\underset{\kappa_{3}'=s-1}{\Res}\mathcal{F}(\boldsymbol{\kappa}',s)d\kappa_2'+R_2(s)-R_3(s),
   	\end{equation}
   	where $R_3(s):=\underset{\kappa_{2}'=2s-1}{\Res}\underset{\kappa_{1}'=s-1}{\Res}\underset{\kappa_{3}'=s-1}{\Res}\mathcal{F}(\boldsymbol{\kappa}',s).$ By \eqref{171.}, we have that 
   	\begin{equation}\label{231}
   	R_3(s)\sim \frac{\Lambda(4s-2,\tau^4)\Lambda(3s-1,\tau^{3})\Lambda(2s-1,\tau^2)^2\Lambda(1-s,\tau{-1})\Lambda(s,\tau)^2}{\Lambda(3-3s,\tau^{-3})\Lambda(2-2s,\tau^{-2})\Lambda(2-s,\tau^{-1})^2\Lambda(2s,\tau^2)\Lambda(1+s,\tau)}.
   	\end{equation}
   	By \eqref{231}, the right hand side is meromorphic in $\mathcal{R}(1/2),$ with a possible pole at $s=1/2$ of order at most 1 according to the functional equation $\Lambda(2s-1,\tau^2)\sim \Lambda(2-2s,\tau^{-2}).$ Hence we obtain a meromorphic continuation of $J_{13}^{(1/2,1)}(s)$ to the domain $\mathcal{R}(1/2).$ Denote by $J_{13}^{1/2}(s)$ this continuation. 
   	
   	Let $s\in\mathcal{R}(1/2)^-.$ Then by \eqref{228}, $\int_{\mathcal{C}}\underset{\kappa_{1}'=s-1}{\Res}\underset{\kappa_{3}'=s-1}{\Res}\mathcal{F}(\boldsymbol{\kappa}',s)d\kappa_2'$ is equal to 
   	\begin{align*}
   	\int_{(0)}\underset{\kappa_{1}'=s-1}{\Res}\underset{\kappa_{3}'=s-1}{\Res}\mathcal{F}(\boldsymbol{\kappa}',s)d\kappa_2'+\underset{\kappa_{2}'=1-2s}{\Res}\underset{\kappa_{1}'=s-1}{\Res}\underset{\kappa_{3}'=s-1}{\Res}\mathcal{F}(\boldsymbol{\kappa}',s),
   	\end{align*}
   	where $\underset{\kappa_{2}'=1-2s}{\Res}\underset{\kappa_{1}'=s-1}{\Res}\underset{\kappa_{3}'=s-1}{\Res}\mathcal{F}(\boldsymbol{\kappa}',s)$ is equal to a holomorphic function multiplying 
   	\begin{align*}
   	\frac{\Lambda(2s-1,\tau^2)^2\Lambda(s,\tau)^2\Lambda(3s-1,\tau^3)\Lambda(1-s,\tau^{-1})\Lambda(4s-2,\tau^4)}{\Lambda(2-s,\tau^{-1})^{2}\Lambda(2-2s,\tau^{-2})\Lambda(2s,\tau^2)\Lambda(3-3s,\tau^{-3})\Lambda(s+1,\tau)}.
   	\end{align*}
   	
   	Thus we obtain a meromorphic continuation of $J_{12}(s)$ to the area $\mathcal{S}_{(1/3,\infty)}:$ 
   	\begin{equation}\label{claim}
   	\widetilde{J}_{13}(s)=\begin{cases}
   	J_{13}(s),\ s\in \mathcal{S}_{(1,+\infty)};\\
   	J_{13}^1(s),\ s\in\mathcal{R}(1);\\
   	J_{13}^{(1/2,1)}(s),\ s\in \mathcal{S}_{(1/2,1)};\\
   	J_{13}^{1/2}(s),\ s\in \mathcal{R}(1/2);\\
   	J_{13}^{(1/3,1/2)}(s),\ s\in \mathcal{S}_{(1/3,1/2)}.
   	\end{cases}
   	\end{equation}
   	
   	From the above formulas one sees that $\widetilde{J}_{13}(s)$ has possible poles at $s=3/4,$ $s=2/3$ and $s=1/2;$ and these potential poles are all at most simple. Moreover, from the above explicit expressions of $\widetilde{J}_{13}(s),$ we see that $\widetilde{J}_{13}(s)\cdot\Lambda(s,\tau)^{-1}$ has at most a simple pole at $s=1/2.$ We discuss the other two possible poles separately.
   	\begin{itemize}
   		\item[Case 1:]
   		If $L_F(3/4,\tau)=0,$ then by functional equation we have that $\Lambda(1/4,\tau^{-1})=0.$ Suppose that $\widetilde{J}_{13}(s)$ has a pole at $s=3/4,$ then from the proceeding explicit expressions, we must have that $\tau^4=1,$ and the singular part of $\widetilde{J}_{13}(s)$ around $s=3/4$ is a holomorphic function multiplying $\Lambda(4s-3,\tau^4)\Lambda(3s-2,\tau^3).$ Note that $\Lambda(3s-2,\tau^3)\mid_{s=3/4}=\Lambda(1/4,\tau^3)=\Lambda(1/4,\tau^{-1})=0.$ Hence, when $L_F(3/4,\tau)=0,$ $\widetilde{J}_{13}(s)$ is holomorphic at $s=3/4.$
   		\item[Case 2:]
   		If $L_F(2/3,\tau)=0,$ then by functional equation we have that $\Lambda(1/3,\tau^{-1})=0.$ Suppose that $\widetilde{J}_{13}(s)$ has a pole at $s=2/3,$ then from the proceeding explicit expressions, we must have that $\tau^3=1,$ and the singular part of $\widetilde{J}_{13}(s)$ around $s=2/3$ is a holomorphic function multiplying $\Lambda(3s-2,\tau^3)\Lambda(2s-1,\tau^2).$ Note that $\Lambda(2s-1,\tau^2)\mid_{s=2/3}=\Lambda(1/3,\tau^2)=\Lambda(1/3,\tau^{-1})=0.$ Hence, when $L_F(2/3,\tau)=0,$ $\widetilde{J}_{13}(s)$ is holomorphic at $s=2/3.$
   	\end{itemize}
   	Now the proof of Claim \ref{64claim} is complete.
   \end{proof}
   
   \begin{proof}[Proof of Claim \ref{65claim}]
   	Let $s\in\mathcal{R}(1)^+.$ Let $J_{23}(s)=\int_{(0)}\underset{\kappa_{2}=s-1}{\Res}\underset{\kappa_{3}=s-1}{\Res}\mathcal{F}(\boldsymbol{\kappa},s)d\kappa_1,$ and  $J_{23}^1(s)=\int_{\mathcal{C}}\underset{\kappa_{2}=s-1}{\Res}\underset{\kappa_{3}=s-1}{\Res}\mathcal{F}(\boldsymbol{\kappa},s)d\kappa_1.$ Then by \eqref{170.} one sees that $J_{23}^1(s)$ is meromorphic in the region $\mathcal{R}(1),$ with a possible pole at $s=1.$
   	
   	Let $s\in \mathcal{R}(1)^-.$ Applying Cauchy integral formula we then have that
   	\begin{equation}\label{233.}
   	J_{23}^1(s)=\int_{(0)}\underset{\kappa_{2}=s-1}{\Res}\underset{\kappa_{3}=s-1}{\Res}\mathcal{F}(\boldsymbol{\kappa},s)d\kappa_1+\underset{\kappa_{1}=3-3s}{\Res}\underset{\kappa_{2}=s-1}{\Res}\underset{\kappa_{3}=s-1}{\Res}\mathcal{F}(\boldsymbol{\kappa},s),
   	\end{equation}
   	where $\underset{\kappa_{1}=3-3s}{\Res}\underset{\kappa_{2}=s-1}{\Res}\underset{\kappa_{3}=s-1}{\Res}\mathcal{F}(\boldsymbol{\kappa},s)$ equals some holomorphic function multiplying
   	\begin{equation}\label{234}
   	\frac{\Lambda(4s-3,\tau^4)\Lambda(3s-2,\tau^{3})\Lambda(2s-1,\tau^2)\Lambda(s,\tau)}{\Lambda(4-3s,\tau^{-3})\Lambda(3-2s,\tau^{-2})\Lambda(2-s,\tau^{-1})}.
   	\end{equation}
   	Then one sees,  by \eqref{170.} and \eqref{234}, that  $\int_{(0)}\underset{\kappa_{2}=s-1}{\Res}\underset{\kappa_{3}=s-1}{\Res}\mathcal{F}(\boldsymbol{\kappa},s)d\kappa_1$ and the function  $\underset{\kappa_{1}=3-3s}{\Res}\underset{\kappa_{2}=s-1}{\Res}\underset{\kappa_{3}=s-1}{\Res}\mathcal{F}(\boldsymbol{\kappa},s)$ are meromorphic in the strip $2/3<\Re(s)<1,$ with possible simple poles at $s=3/4$ if $\tau^4=1.$ Hence, by \eqref{233.}, we obtain a meromorphic continuation of $J_{23}^1(s)$ to the strip $2/3<\Re(s)<1,$ with possible simple poles at $s=3/4$ if $\tau^4=1.$ Denote by $J_{23}^{(2/3,1)}(s)$ this continuation. 
   	
   	Let $s\in\mathcal{R}(2/3)^+.$ Applying Cauchy integral formula to \eqref{172.} to see that the function  $\int_{(0)}\underset{\kappa_{2}=s-1}{\Res}\underset{\kappa_{3}=s-1}{\Res}\mathcal{F}(\boldsymbol{\kappa},s)d\kappa_1$ is equal to 
   	\begin{equation}\label{235}
   	\int_{\mathcal{C}}\underset{\kappa_{2}=s-1}{\Res}\underset{\kappa_{3}=s-1}{\Res}\mathcal{F}(\boldsymbol{\kappa},s)d\kappa_1-\underset{\kappa_{1}=2-3s}{\Res}\underset{\kappa_{2}=s-1}{\Res}\underset{\kappa_{3}=s-1}{\Res}\mathcal{F}(\boldsymbol{\kappa},s),
   	\end{equation}
   	and $\underset{\kappa_{1}=2-3s}{\Res}\underset{\kappa_{2}=s-1}{\Res}\underset{\kappa_{3}=s-1}{\Res}\mathcal{F}(\boldsymbol{\kappa},s)$ is equal to some holomorphic function multiplying 
   	\begin{equation}\label{236}
   	\frac{\Lambda(4s-3,\tau^4)\Lambda(3s-2,\tau^{3})\Lambda(2s-1,\tau^2)\Lambda(s,\tau)^2}{\Lambda(3-3s,\tau^{-3})\Lambda(1+s,\tau)\Lambda(3-2s,\tau^{-2})\Lambda(2-s,\tau^{-1})}.
   	\end{equation}
   	Then by \eqref{235}, \eqref{236} and the fact that $\int_{\mathcal{C}}\underset{\kappa_{2}=s-1}{\Res}\underset{\kappa_{3}=s-1}{\Res}\mathcal{F}(\boldsymbol{\kappa},s)d\kappa_1$ is holomorphic in $\mathcal{R}(2/3),$ we obtain a meromorphic continuation of $J_{23}^{(2/3,1)}(s)$ to the region $\mathcal{R}(2/3),$ with a possible simple pole at $s=2/3,$ if $\tau^3=1.$ Denote by $J_{23}^{2/3}(s)$ the continuation.
   	
   	Let $s\in\mathcal{R}(2/3)^-.$ Then by \eqref{233.} and \eqref{235} one has 
   	\begin{align*}
   	J_{23}^{2/3}(s)=&\int_{\mathcal{C}}\underset{\kappa_{2}=s-1}{\Res}\underset{\kappa_{3}=s-1}{\Res}\mathcal{F}(\boldsymbol{\kappa},s)d\kappa_1-R_1(s)+R_2(s),
   	\end{align*}
   	where $R_1(s)=\underset{\kappa_{1}=2-3s}{\Res}\underset{\kappa_{2}=s-1}{\Res}\underset{\kappa_{3}=s-1}{\Res}\mathcal{F}(\boldsymbol{\kappa},s);$ $R_2(s)=\underset{\kappa_{1}=3-3s}{\Res}\underset{\kappa_{2}=s-1}{\Res}\underset{\kappa_{3}=s-1}{\Res}\mathcal{F}(\boldsymbol{\kappa},s).$
   	
   	Since the right hand side is meromorphic in the strip $\mathcal{S}_{(0,2/3)},$ with a possible simple pole at $s=1/2$ if $\tau^2=1.$ We thus obtain a meromorphic continuation of $J_{23}^{2/3}(s)$ to the region $0<\Re(s)<2/3,$ with a possible simple pole at $s=1/2$ if $\tau^2=1.$ Denote by $J_{23}^{(1/3,2/3)}(s)$ this continuation. Thus we obtain a meromorphic continuation of $J_{23}(s)$ to the area $\mathcal{S}_{(1/3,\infty)}:$ 
   	\begin{equation}\label{claim5}
   	\widetilde{J}_{23}(s)=\begin{cases}
   	J_{23}(s),\ s\in \mathcal{S}_{(1,+\infty)};\\
   	J_{23}^1(s),\ s\in\mathcal{R}(1);\\
   	J_{23}^{(2/3,1)}(s),\ s\in \mathcal{S}_{(2/3,1)};\\
   	J_{23}^{2/3}(s),\ s\in \mathcal{R}(2/3);\\
   	J_{23}^{(1/3,2/3)}(s),\ s\in \mathcal{S}_{(1/3,2/3)}.
   	\end{cases}
   	\end{equation}
   	
   	From the above formulas one sees that $\widetilde{J}_{23}(s)$ has possible poles at $s=3/4,$ $s=2/3$ and $s=1/2;$ and these potential poles are all at most simple. Moreover, from the above explicit expressions of $\widetilde{J}_{23}(s),$ we see that $\widetilde{J}_{23}(s)\cdot\Lambda(s,\tau)^{-1}$ has at most a simple pole at $s=1/2.$ We discuss the other two possible poles separately.
   	\begin{itemize}
   		\item[Case 1:]
   		If $L_F(3/4,\tau)=0,$ then by functional equation we have that $\Lambda(1/4,\tau^{-1})=0.$ Suppose that $\widetilde{J}_{23}(s)$ has a pole at $s=3/4,$ then from the proceeding explicit expressions, we must have that $\tau^4=1,$ and the singular part of $\widetilde{J}_{23}(s)$ around $s=3/4$ is a holomorphic function multiplying $\Lambda(4s-3,\tau^4)\Lambda(3s-2,\tau^3).$ Note that $\Lambda(3s-2,\tau^3)\mid_{s=3/4}=\Lambda(1/4,\tau^3)=\Lambda(1/4,\tau^{-1})=0.$ Hence, when $L_F(3/4,\tau)=0,$ $\widetilde{J}_{23}(s)$ is holomorphic at $s=3/4.$
   		\item[Case 2:]
   		If $L_F(2/3,\tau)=0,$ then by functional equation we have that $\Lambda(1/3,\tau^{-1})=0.$ Suppose that $\widetilde{J}_{23}(s)$ has a pole at $s=2/3,$ then from the proceeding explicit expressions, we must have that $\tau^3=1,$ and the singular part of $\widetilde{J}_{23}(s)$ around $s=2/3$ is a holomorphic function multiplying $\Lambda(3s-2,\tau^3)\Lambda(2s-1,\tau^2).$ Note that $\Lambda(2s-1,\tau^2)\mid_{s=2/3}=\Lambda(1/3,\tau^2)=\Lambda(1/3,\tau^{-1})=0.$ Hence, when $L_F(2/3,\tau)=0,$ $\widetilde{J}_{23}(s)$ is holomorphic at $s=2/3.$
   	\end{itemize}
   	Now the proof of Claim \ref{65claim} is complete.
   \end{proof}
   
   \begin{proof}[Proof of Claim \ref{66claim}]
   	Let $s\in\mathcal{R}(1)^-.$ Let $H_1^{(1/2,1)}(s):=\int_{(0)}\int_{(0)}\underset{\kappa_{1}=1-s}{\Res}\mathcal{F}(\boldsymbol{\kappa},s)d\kappa_3d\kappa_2.$ Recall that we have computed the analytic property of $\underset{\kappa_{1}=1-s}{\Res}\mathcal{F}(\boldsymbol{\kappa},s):$
   	\begin{align*}
   	\underset{\kappa_{1}=1-s}{\Res}\mathcal{F}(\boldsymbol{\kappa},s)\sim& 
   	\frac{\Lambda(s+\kappa_3,\chi_{34}\tau)\Lambda(s-\kappa_3,\chi_{43}\tau)\Lambda(s+\kappa_2,\chi_{23}\tau)\Lambda(s+\kappa_{23},\chi_{24}\tau)}{\Lambda(1-\kappa_2,\chi_{32})\Lambda(1-\kappa_3,\chi_{43})\Lambda(1+\kappa_3,\chi_{34})\Lambda(2-s+\kappa_2,\chi_{23}\tau^{-1})}\times\\
   	&\frac{\Lambda(2s-1-\kappa_2,\chi_{32}\tau^2)\Lambda(2s-1-\kappa_{23},\chi_{42}\tau^2)\Lambda(2s-1,\tau^2)\Lambda(s,\tau)^3}{\Lambda(1-\kappa_{23},\chi_{42})\Lambda(2-s+\kappa_{23},\chi_{24}\tau^{-1})\Lambda(2-s,\tau^{-1})}.
   	\end{align*}  
   	
   	Therefore, we see that $H_1^{(1/2,1)}(s)$ is holomorphic in the strip $1/2<\Re(s)<1.$ Let $s\in\mathcal{R}(1/2)^+.$ By Cauchy integral formula we have 
   	\begin{equation}\label{238}
   	H_1^{(1/2,1)}(s)=\int_{(0)}\int_{\mathcal{C}}\underset{\kappa_{1}=1-s}{\Res}\mathcal{F}(\boldsymbol{\kappa},s)d\kappa_2d\kappa_3-\int_{(0)}\big[R_1(\kappa_3)+R_2(\kappa_3)\big]d\kappa_3,
   	\end{equation}
   	where $R_1(\kappa_3)=R_1(\kappa_3;s)=\underset{\kappa_{2}=2s-1}{\Res}\underset{\kappa_{1}=1-s}{\Res}\mathcal{F}(\boldsymbol{\kappa},s),$ and $R_2(\kappa_3)=R_2(\kappa_3;s)=\underset{\kappa_{2}=2s-1-\kappa_3}{\Res}\underset{\kappa_{1}=1-s}{\Res}\mathcal{F}(\boldsymbol{\kappa},s).$ By functional equation of Hecke L-functions over $F$ we see that $R_1(\kappa_3)$ is equal to some holomorphic function multiplying the product of $\Lambda(3s-1,\tau^3)\Lambda(s,\tau)^3\cdot\Lambda(2-s,\tau^{-1})^{-1}$ and 
   	\begin{equation}\label{239}
   	\frac{\Lambda(s+\kappa_3,\chi_{34}\tau)\Lambda(s-\kappa_3,\chi_{43}\tau)\Lambda(3s-1+\kappa_3,\chi_{34}\tau^3)}{\Lambda(1-\kappa_3,\chi_{43})\Lambda(2-2s-\kappa_3,\chi_{43}\tau^{-2})\Lambda(1+s+\kappa_3,\chi_{34}\tau)\Lambda(1+s,\tau)}.
   	\end{equation}
   	Also, applying functional equation of Hecke L-functions to $\underset{\kappa_{2}=2s-1-\kappa_3}{\Res}\underset{\kappa_{1}=1-s}{\Res}\mathcal{F}(\boldsymbol{\kappa},s)$ leads to that $R_2(\kappa_3)$ is equal to some holomorphic function multiplying the product of $\Lambda(3s-1,\tau^3)\Lambda(s,\tau)^3\cdot\Lambda(2-s,\tau^{-1})^{-1}$ and 
   	\begin{equation}\label{240}
   	\frac{\Lambda(s+\kappa_3,\chi_{34}\tau)\Lambda(s-\kappa_3,\chi_{43}\tau)\Lambda(3s-1-\kappa_3,\chi_{43}\tau^3)}{\Lambda(1+\kappa_3,\chi_{34})\Lambda(2-2s+\kappa_3,\chi_{34}\tau^{-2})\Lambda(1+s-\kappa_3,\chi_{43}\tau)\Lambda(1+s,\tau)}.
   	\end{equation}
   	Due to the uniform zero-free region discussed in Section \ref{7.11}, one sees that both $\int_{(0)}R_1(\kappa_3)d\kappa_3$ and $\int_{(0)}R_2(\kappa_3)d\kappa_3$ converges normally in the region $\mathcal{R}(1/2).$ Hence they are holomorphic in this are. Also, note that $\int_{(0)}\int_{\mathcal{C}}\underset{\kappa_{1}=1-s}{\Res}\mathcal{F}(\boldsymbol{\kappa},s)d\kappa_2d\kappa_3$ is meromorphic in the region $\mathcal{R}(1/2),$ with a possible simple pole at $s=1/2$ if $\tau^2=1.$ Denote by $H_1^{1/2}(s)$ this continuation. It's clear that $H_1^{1/2}(s)$ admits a natural meromorphic continuation to the region $1/3<\Re(s)<1/2.$ Denote by $H_1^{(1/3,1/2)}(s)$ this continuation. Then we obtain $\widetilde{H}_1(s),$ a meromorphic continuation of $H_1^{(1/2,1)}(s)$ to the domain $\mathcal{S}_{(1/3,1)},$ by \eqref{238}, \eqref{239} and \eqref{240}. Explicitly, we have that 
   	\begin{equation}\label{claim7}
   	\widetilde{H}_1(s)=\begin{cases}
   	H_1^{(1/2,1)}(s),\ s\in \mathcal{S}_{(1/2,1)};\\
   	H_1^{1/2}(s),\ s\in \mathcal{R}(1/2);\\
   	H_1^{(1/3,1/2)}(s).
   	\end{cases}
   	\end{equation}
   	Moreover, $\widetilde{H}_1(s)$ has a possible simple pole at $s=1/2$ if $\tau^2=1.$ Now the proof of Claim \ref{66claim} is complete.
   \end{proof}
   
   \begin{proof}[Proof of Claim \ref{67claim}]
   	Let $s\in\mathcal{R}(1)^-.$ Let $H_2^{(1/2,1)}(s):=\int_{(0)}\int_{(0)}\underset{\kappa_{2}=1-s}{\Res}\mathcal{F}(\boldsymbol{\kappa},s)d\kappa_3d\kappa_1.$ Recall that we have computed the analytic property of $\underset{\kappa_{2}=1-s}{\Res}\mathcal{F}(\boldsymbol{\kappa},s):$
   	\begin{align*}
   	\underset{\kappa_{2}=1-s}{\Res}\mathcal{F}(\boldsymbol{\kappa},s)\sim& 
   	\frac{\Lambda(s+\kappa_1,\chi_{12}\tau)\Lambda(s+\kappa_3,\chi_{34}\tau)\Lambda(s+\kappa_{13},\chi_{14}\tau)\Lambda(s-\kappa_{13},\chi_{41}\tau)}{\Lambda(1-\kappa_1,\chi_{21})\Lambda(1-\kappa_3,\chi_{43})\Lambda(1+\kappa_{13},\chi_{14})\Lambda(2-s+\kappa_1,\chi_{12}\tau^{-1})}\\
   	&\cdot\frac{\Lambda(2s-1-\kappa_3,\chi_{43}\tau^2)\Lambda(2s-1-\kappa_1,\chi_{21}\tau^2)\Lambda(2s-1,\tau^2)\Lambda(s,\tau)^3}{\Lambda(1-\kappa_{13},\chi_{41})\Lambda(2-s+\kappa_{3},\chi_{34}\tau^{-1})\Lambda(2-s,\tau^{-1})}.
   	\end{align*} 
   	
   	Therefore, we see that $H_2^{(1/2,1)}(s)$ is holomorphic in the strip $1/2<\Re(s)<1.$ Let $s\in\mathcal{R}(1/2)^+.$ By Cauchy integral formula we have 
   	\begin{align*}
   	H_2^{(1/2,1)}(s)=&\int_{(0)}\int_{\mathcal{C}}\underset{\kappa_{2}=1-s}{\Res}\mathcal{F}(\boldsymbol{\kappa},s)d\kappa_1d\kappa_3-\int_{(0)}\big[R_1(\kappa_3)+R_2(\kappa_3)\big]d\kappa_3\\
   	=&\int_{\mathcal{C}}\int_{\mathcal{C}}\Res_{2}(\boldsymbol{\kappa},s)d\kappa_1d\kappa_3-\int_{(0)}\big[R_1(\kappa_3)+R_2(\kappa_3)\big]d\kappa_3-\int_{C}R(\kappa_1)d\kappa_1,
   	\end{align*}
   	where $\Res_{2}(\boldsymbol{\kappa},s)=\underset{\kappa_{2}=1-s}{\Res}\mathcal{F}(\boldsymbol{\kappa},s);$ $R_1(\kappa_3)=R_1(\kappa_3;s)=\underset{\kappa_{1}=2s-1}{\Res}\underset{\kappa_{2}=1-s}{\Res}\mathcal{F}(\boldsymbol{\kappa},s),$ $R_2(\kappa_3)=R_2(\kappa_3;s)=\underset{\kappa_{1}=2s-1-\kappa_3}{\Res}\underset{\kappa_{2}=1-s}{\Res}\mathcal{F}(\boldsymbol{\kappa},s),$ and the meromorphic function  $R(\kappa_1)=R(\kappa_1;s)=\underset{\kappa_{3}=2s-1}{\Res}\underset{\kappa_{2}=1-s}{\Res}\mathcal{F}(\boldsymbol{\kappa},s).$ By analytic properties of $\underset{\kappa_{2}=1-s}{\Res}\mathcal{F}(\boldsymbol{\kappa},s)$ and functional equation of Hecke L-functions over $F$ we see that $R_1(\kappa_3)$ is equal to some holomorphic function multiplying the product of $\Lambda(2s-1,\tau^2)\Lambda(s,\tau)^3\cdot\Lambda(2-s,\tau^{-1})^{-1}\cdot\Lambda(2-2s,\tau^{-2})^{-1}$ and 
   	\begin{equation}\label{242}
   	\frac{\Lambda(2s+\kappa_3,\chi_{34}\tau^2)\Lambda(2s-1-\kappa_3,\chi_{43}\tau^2)\Lambda(1+\kappa_3,\chi_{34})\Lambda(3s-1,\tau^3)}{\Lambda(1-\kappa_3,\chi_{43})\Lambda(1+s+\kappa_3,\chi_{34}\tau)\Lambda(2-s+\kappa_3,\chi_{34}\tau^{-1})\Lambda(1+s,\tau)}.
   	\end{equation}
   	Also, applying functional equation of Hecke L-functions to $\underset{\kappa_{1}=2s-1-\kappa_3}{\Res}\underset{\kappa_{2}=1-s}{\Res}\mathcal{F}(\boldsymbol{\kappa},s)$ leads to that $R_2(\kappa_3)$ is equal to some holomorphic function multiplying 
   	\begin{equation}\label{243}
   	\frac{\Lambda(s+\kappa_3,\chi_{34}\tau)\Lambda(3s-1-\kappa_3,\chi_{43}\tau^3)\Lambda(2s,\tau^2)\Lambda(2s-1,\tau^2)\Lambda(s,\tau)^2}{\Lambda(2-s+\kappa_3,\chi_{34}\tau^{-1})\Lambda(1+s-\kappa_3,\chi_{43}\tau)\Lambda(2-s,\tau^{-1})\Lambda(1+s,\tau)}.
   	\end{equation}
   	Again, by functional equation of Hecke L-functions we see that $R(\kappa_1)$ is equal to some holomorphic function multiplying the product of $\Lambda(2s-1,\tau^2)\Lambda(s,\tau)^2\cdot\Lambda(2-s,\tau^{-1})^{-1}\cdot\Lambda(1+s,\tau)^{-1}$ and the meromorphic function
   	\begin{equation}\label{244}
   	\frac{\Lambda(1+\kappa_1,\chi_{12})\Lambda(2s-1-\kappa_1,\chi_{21}\tau^2)\Lambda(2s+\kappa_1,\chi_{12}\tau^2)\Lambda(3s-1,\tau^3)}{\Lambda(1-\kappa_1,\chi_{21})\Lambda(2-s+\kappa_1,\chi_{12}\tau^{-1})\Lambda(1+s+\kappa_1,\chi_{12}\tau)\Lambda(2-2s,\tau^{-2})}.
   	\end{equation}
   	Due to the uniform zero-free region discussed in Section \ref{7.11}, one sees from \eqref{243} and \eqref{244} that both $\Lambda(2s,\tau^2)^{-1}\cdot\Lambda(2s-1,\tau^2)^{-1}\cdot\int_{(0)}R(\kappa_1;s)d\kappa_3$ and $\Lambda(2s-1,\tau^2)^{-1}\cdot\int_{\mathcal{C}}R(\kappa_1;s)d\kappa_1$ converge normally for any $s\in \mathcal{R}(1/2).$ Hence they are holomorphic in this area. Then we obtain a meromorphic continuation of $\int_{(0)}R(\kappa_1;s)d\kappa_3$ to $\mathcal{R}(1/2),$ with a possible pole of order at most 2 at $s=1/2$ if $\tau^2=1;$ and a meromorphic continuation of $\int_{\mathcal{C}}R(\kappa_1;s)d\kappa_1$ to $\mathcal{R}(1/2),$ with a possible simple pole at $s=1/2$ if $\tau^2=1.$ Moreover, if $L_F(1/2,\tau)=0,$ then both $\int_{(0)}R(\kappa_1;s)d\kappa_3$ and $\int_{\mathcal{C}}R(\kappa_1;s)d\kappa_1$ are holomorphic at $s=1/2.$
   	
   	By \eqref{242}, one can apply Cauchy integral formula to deduce that
   	\begin{align*}
   	H_2^{(1/2,1)}(s)=&\int_{\mathcal{C}}\int_{\mathcal{C}}\Res_{2}(\boldsymbol{\kappa},s)d\kappa_1d\kappa_3-\int_{\mathcal{C}}\big[R_1(\kappa_3)+R_2(\kappa_3)\big]d\kappa_3-\int_{C}R(\kappa_1)d\kappa_1\\
   	&+\underset{\kappa_{3}=2s-1}{\Res}\underset{\kappa_{1}=2s-1}{\Res}\underset{\kappa_{2}=1-s}{\Res}\mathcal{F}(\boldsymbol{\kappa},s),
   	\end{align*}
   	where $\underset{\kappa_{3}=2s-1}{\Res}\underset{\kappa_{1}=2s-1}{\Res}\underset{\kappa_{2}=1-s}{\Res}\mathcal{F}(\boldsymbol{\kappa},s)$ is equal to, according to \eqref{242}, some holomoprhic function multiplying the meromorphic function
   	\begin{equation}\label{245}
   	\frac{\Lambda(4s-1,\tau^4)\Lambda(3s-1,\tau^3)\Lambda(2s-1,\tau^2)\Lambda(2s,\tau^2)\Lambda(s,\tau)^3}{\Lambda(2-2s,\tau^{-2})^2\Lambda(2-s,\tau^{-1})\Lambda(1+s,\tau)^2\Lambda(3s,\tau^{3})}.
   	\end{equation}
   	Hence $\underset{\kappa_{3}=2s-1}{\Res}\underset{\kappa_{1}=2s-1}{\Res}\underset{\kappa_{2}=1-s}{\Res}\mathcal{F}(\boldsymbol{\kappa},s)$ admits a meromorphic continuation to $\mathcal{R}(1/2),$ with a possible pole of order at most 2 at $s=1/2$ if $\tau^2=1.$ Moreover, if $L_F(1/2,\tau)=0,$ then $\underset{\kappa_{3}=2s-1}{\Res}\underset{\kappa_{1}=2s-1}{\Res}\underset{\kappa_{2}=1-s}{\Res}\mathcal{F}(\boldsymbol{\kappa},s)$ is holomorphic at $s=1/2.$ 
   	
   	Also, note that $\int_{\mathcal{C}}\int_{\mathcal{C}}\Res_{2}(\boldsymbol{\kappa},s)d\kappa_1d\kappa_3,$ $\int_{\mathcal{C}}R_1(\kappa_3)d\kappa_3$ and $\int_{\mathcal{C}}R_2(\kappa_3)d\kappa_3$ are meromorphic in $\mathcal{S}_{(1/3,1/2)}\cup\mathcal{R}(1/2),$ with a possible pole of order at most 2 at $s=1/2$ if $\tau^2=1.$ Moreover, $L_F(1/2,\tau)^{-1}\cdot\int_{\mathcal{C}}\int_{\mathcal{C}}\Res_{2}(\boldsymbol{\kappa},s)d\kappa_1d\kappa_3,$ $L_F(1/2,\tau)^{-1}\cdot\int_{\mathcal{C}}R_1(\kappa_3)d\kappa_3$ and $L_F(1/2,\tau)^{-1}\cdot\int_{\mathcal{C}}R_2(\kappa_3)d\kappa_3$ all have at most a simple pole at $s=1/2.$ Denote by $H_2^{(1/3,1/2]}(s)$ this continuation of $H_2^{(1/2,1)}(s)$ to $\mathcal{R}(1/2).$ 
   	
   	Thus, we obtain $\widetilde{H}_2(s),$ a meromorphic continuation of $H_2^{(1/2,1)}(s)$ to the domain $\mathcal{S}_{(1/3,\infty)},$ by \eqref{242}, \eqref{243}, \eqref{244} and \eqref{245}. Explicitly, we have that 
   	\begin{equation}\label{claim8}
   	\widetilde{H}_2(s)=\begin{cases}
   	H_2^{(1/2,1)}(s),\ s\in \mathcal{S}_{(1/2,1)};\\
   	H_2^{(1/3,1/2]}(s),\ s\in \mathcal{S}_{(1/3,1/2)}\cup\mathcal{R}(1/2).
   	\end{cases}
   	\end{equation}
   	Moreover, $\widetilde{H}_2(s)\cdot \Lambda_F(1/2,\tau)^{-1}$ has a possible pole of order at most 1 at $s=1/2$ if $\tau^2=1.$ Moreover, if $L_F(1/2,\tau)=0,$ then $\widetilde{H}_2(s)$ is holomorphic at $s=1/2.$  Now the proof of Claim \ref{67claim} is complete.
   \end{proof}

   \begin{proof}[Proof of Claim \ref{68claim}]
   	Let $s\in\mathcal{R}(1)^-.$ Let $H_3^{(1/2,1)}(s):=\int_{(0)}\int_{(0)}\underset{\kappa_{3}=1-s}{\Res}\mathcal{F}(\boldsymbol{\kappa},s)d\kappa_2d\kappa_1.$ Recall that we have computed the analytic property of $\underset{\kappa_{3}=1-s}{\Res}\mathcal{F}(\boldsymbol{\kappa},s):$
   	\begin{align*}
   	\underset{\kappa_{3}=1-s}{\Res}\mathcal{F}(\boldsymbol{\kappa},s)\sim& 
   	\frac{\Lambda(s+\kappa_1,\chi_{12}\tau)\Lambda(s-\kappa_1,\chi_{21}\tau)\Lambda(s+\kappa_2,\chi_{23}\tau)\Lambda(s+\kappa_{12},\chi_{13}\tau)}{\Lambda(1+\kappa_1,\chi_{12})\Lambda(1-\kappa_1,\chi_{21})\Lambda(1-\kappa_2,\chi_{32})\Lambda(2-s+\kappa_2,\chi_{23}\tau^{-1})}\times\\
   	&\frac{\Lambda(2s-1-\kappa_2,\chi_{32}\tau^2)\Lambda(2s-1-\kappa_{12},\chi_{31}\tau^2)\Lambda(2s-1,\tau^2)\Lambda(s,\tau)^3}{\Lambda(1-\kappa_{12},\chi_{31})\Lambda(2-s+\kappa_{12},\chi_{13}\tau^{-1})\Lambda(2-s,\tau^{-1})}.
   	\end{align*}  
   	
   	Therefore, we see that $H_3^{(1/2,1)}(s)$ is holomorphic in the strip $1/2<\Re(s)<1.$ Let $s\in\mathcal{R}(1/2)^+.$ By Cauchy integral formula we have 
   	\begin{equation}\label{247}
   	H_3^{(1/2,1)}(s)=\int_{(0)}\int_{\mathcal{C}}\underset{\kappa_{3}=1-s}{\Res}\mathcal{F}(\boldsymbol{\kappa},s)d\kappa_2d\kappa_1-\int_{(0)}\big[R_1(\kappa_1)+R_2(\kappa_1)\big]d\kappa_1,
   	\end{equation}
   	where $R_1(\kappa_1)=R_1(\kappa_1;s)=\underset{\kappa_{2}=2s-1}{\Res}\underset{\kappa_{3}=1-s}{\Res}\mathcal{F}(\boldsymbol{\kappa},s),$ and $R_2(\kappa_1)=R_2(\kappa_1;s)=\underset{\kappa_{2}=2s-1-\kappa_1}{\Res}\underset{\kappa_{3}=1-s}{\Res}\mathcal{F}(\boldsymbol{\kappa},s).$ By functional equation of Hecke L-functions over $F$ we see that $R_1(\kappa_1)$ is equal to some holomorphic function multiplying the product of $\Lambda(3s-1,\tau^3)\Lambda(s,\tau)^3\cdot\Lambda(2-s,\tau^{-1})^{-1}$ and 
   	\begin{equation}\label{248}
   	\frac{\Lambda(s+\kappa_1,\chi_{12}\tau)\Lambda(s-\kappa_1,\chi_{21}\tau)\Lambda(3s-1+\kappa_1,\chi_{12}\tau^3)}{\Lambda(1-\kappa_1,\chi_{21})\Lambda(2-2s-\kappa_1,\chi_{21}\tau^{-2})\Lambda(1+s+\kappa_1,\chi_{12}\tau)\Lambda(1+s,\tau)}.
   	\end{equation}
   	Also, applying functional equation of Hecke L-functions to $\underset{\kappa_{2}=2s-1-\kappa_1}{\Res}\underset{\kappa_{3}=1-s}{\Res}\mathcal{F}(\boldsymbol{\kappa},s)$ leads to that $R_2(\kappa_1)$ is equal to some holomorphic function multiplying the product of $\Lambda(3s-1,\tau^3)\Lambda(s,\tau)^3\cdot\Lambda(2-s,\tau^{-1})^{-1}$ and 
   	\begin{equation}\label{249}
   	\frac{\Lambda(s+\kappa_1,\chi_{12}\tau)\Lambda(s-\kappa_1,\chi_{21}\tau)\Lambda(3s-1-\kappa_1,\chi_{21}\tau^3)}{\Lambda(1+\kappa_1,\chi_{12})\Lambda(2-2s+\kappa_1,\chi_{12}\tau^{-2})\Lambda(1+s-\kappa_1,\chi_{21}\tau)\Lambda(1+s,\tau)}.
   	\end{equation}
   	Due to the uniform zero-free region discussed in Section \ref{7.11}, one sees that both $\int_{(0)}R_1(\kappa_1)d\kappa_1$ and $\int_{(0)}R_2(\kappa_1)d\kappa_1$ converges normally in the region $\mathcal{S}_{(1/3,1/2)}\cup\mathcal{R}(1/2).$ Hence they are holomorphic in this are. Also, note that $\int_{(0)}\int_{\mathcal{C}}\underset{\kappa_{3}=1-s}{\Res}\mathcal{F}(\boldsymbol{\kappa},s)d\kappa_2d\kappa_1$ is meromorphic in the region $\mathcal{S}_{(1/3,1/2)}\cup\mathcal{R}(1/2),$ with a possible simple pole at $s=1/2$ if $\tau^2=1.$ Denote by $H_3^{(1/3,1/2]}(s)$ this continuation. Then we obtain $\widetilde{H}_3(s),$ a meromorphic continuation of $H_3^{(1/2,1)}(s)$ to the domain $\mathcal{S}_{(1/3,\infty)},$ by \eqref{247}, \eqref{248} and \eqref{249}. Explicitly, we have that 
   	\begin{equation}\label{claim9}
   	\widetilde{H}_3(s)=\begin{cases}
   	H_3^{(1/2,1)}(s),\ s\in \mathcal{S}_{(1/2,1)};\\
   	H_3^{(1/3,1/2]}(s),\ s\in \mathcal{S}_{(1/3,1/2)}\cup\mathcal{R}(1/2).
   	\end{cases}
   	\end{equation}
   	Moreover, $\widetilde{H}_3(s)$ has a possible simple pole at $s=1/2$ if $\tau^2=1.$ Now the proof of Claim \ref{68claim} is complete.
   \end{proof}
   
   \begin{proof}[Proof of Claim \ref{69claim}]
   	Let $s\in\mathcal{R}(1)^-.$ Let $H_{12}^{(2/3,1)}(s):=\int_{(0)}\underset{\kappa_{1}=1-s}{\Res}\underset{\kappa_{2}=1-s}{\Res}\mathcal{F}(\boldsymbol{\kappa},s)d\kappa_3.$ Recall that by \eqref{175.} one sees that $H_{12}^{(2/3,1)}(s)$ admits a natural holomorphic continuation to the strip $2/3<\Re(s)<1.$ Now let $s\in\mathcal{R}(2/3)^+.$ Then we have 
   	\begin{equation}\label{251}
   	H_{12}^{(2/3,1)}(s)=\int_{\mathcal{C}}\underset{\kappa_{1}=1-s}{\Res}\underset{\kappa_{2}=1-s}{\Res}\mathcal{F}(\boldsymbol{\kappa},s)d\kappa_3-\underset{\kappa_{3}=3s-2}{\Res}\underset{\kappa_{1}=1-s}{\Res}\underset{\kappa_{2}=1-s}{\Res}\mathcal{F}(\boldsymbol{\kappa},s),
   	\end{equation}
   	where $\underset{\kappa_{3}=3s-2}{\Res}\underset{\kappa_{1}=1-s}{\Res}\underset{\kappa_{2}=1-s}{\Res}\mathcal{F}(\boldsymbol{\kappa},s)$ equals some holomorphic function multiplying
   	\begin{align*}
   	\frac{\Lambda(4s-2,\tau^4)\Lambda(3s-2,\tau^3)\Lambda(2s-1,\tau^2)\Lambda(s,\tau)^2}{\Lambda(3-3s,\tau^{-3})\Lambda(3-2s,\tau^{-2})\Lambda(2-s,\tau^{-1})\Lambda(1+s,\tau)}.
   	\end{align*} 
   	Then by functional equation $\Lambda(3s-2,\tau^3)\sim \Lambda(3-3s,\tau^{-3}),$ we have that 
   	\begin{equation}\label{252}
   	\underset{\kappa_{3}=3s-2}{\Res}\underset{\kappa_{1}=1-s}{\Res}\underset{\kappa_{2}=1-s}{\Res}\mathcal{F}(\boldsymbol{\kappa},s)\sim \frac{\Lambda(4s-2,\tau^4)\Lambda(2s-1,\tau^2)\Lambda(s,\tau)^2}{\Lambda(3-2s,\tau^{-2})\Lambda(2-s,\tau^{-1})\Lambda(1+s,\tau)}.
   	\end{equation}
   	
   	Hence $\underset{\kappa_{3}=3s-2}{\Res}\underset{\kappa_{1}=1-s}{\Res}\underset{\kappa_{2}=1-s}{\Res}\mathcal{F}(\boldsymbol{\kappa},s)$ admits a meromorphic continuation to the region $\mathcal{S}_{(1/3,1)},$ with a possible pole of order at most 2 at $1/2$ if $\tau^2=1.$ 
   	
   	Moreover, due to \eqref{175.}, the function  $\int_{\mathcal{C}}\underset{\kappa_{1}=1-s}{\Res}\underset{\kappa_{2}=1-s}{\Res}\mathcal{F}(\boldsymbol{\kappa},s)d\kappa_3$ is meromorphic in the region $\mathcal{S}_{(1/3,2/3)}\cup\mathcal{R}(2/3),$ with a possible simple pole at $s=2/3$ if $\tau^3=1;$ and a possible simple pole at $1/2$ if $\tau^2=1.$ Thus we get a meromorphic continuation of $H_{12}^{(2/3,1)}(s)$ to the region $\mathcal{S}_{(1/3,2/3)}\cup\mathcal{R}(2/3).$ Denote by $H_{12}^{(1/3,2/3]}(s)$ this continuation. Now we obtain from \eqref{251} and \eqref{252} a meromorphic continuation of $H_{12}^{(2/3,1)}(s)$ to the region $\mathcal{S}_{(1/3,1)},$ namely,
   	\begin{equation}\label{claim10}
   	\widetilde{H}_{12}(s)=\begin{cases}
   	H_{12}^{(2/3,1)}(s),\ s\in \mathcal{S}_{(2/3,1)};\\
   	H_{12}^{(1/3,2/3]}(s),\ s\in \mathcal{S}_{(1/3,2/3)}\cup\mathcal{R}(2/3).
   	\end{cases}
   	\end{equation}
   	
   	From \eqref{175.} and the above formulas one sees that $\widetilde{H}_{12}(s)$ has possible poles at $s=2/3$ and $s=1/2;$ and these potential pole at $s=2/3$ is at most simple, the possible pole at $s=1/2$ has order at most 2. Moreover, from the above explicit expressions of $\widetilde{H}_{12}(s),$ we see that $\widetilde{H}_{12}(s)\cdot\Lambda(s,\tau)^{-1}$ has at most a simple pole at $s=1/2$ if $L_F(1/2,\tau)=0.$ In additional, if $L_F(2/3,\tau)=0,$ then by functional equation we have that $\Lambda(1/3,\tau^{-1})=0.$ Suppose that $\widetilde{H}_{12}(s)$ has a pole at $s=2/3.$ Then from the proceeding explicit expressions, we must have that $\tau^3=1,$ and the singular part of $\widetilde{H}_{12}(s)$ around $s=2/3$ is a holomorphic function multiplying $\Lambda(3s-2,\tau^3)\Lambda(2s-1,\tau^2).$ Note that $\Lambda(2s-1,\tau^2)\mid_{s=2/3}=\Lambda(1/3,\tau^2)=\Lambda(1/3,\tau^{-1})=0.$ Hence, when $L_F(2/3,\tau)=0,$ $\widetilde{H}_{12}(s)$ is holomorphic at $s=2/3.$ Now the proof of Claim \ref{69claim} is complete.
   \end{proof}
   
   \begin{proof}[Proof of Claim \ref{70claim}]
   	Let $s\in\mathcal{R}(1)^-.$ Let $H_{12}^{(1/2,1)}(s):=\int_{(0)}\underset{\kappa_{1}=1-s}{\Res}\underset{\kappa_{3}=1-s}{\Res}\mathcal{F}(\boldsymbol{\kappa},s)d\kappa_2.$ Let $\kappa_{2}'=1-s+\kappa_2,$ $\kappa_1'=\kappa_1$ and $\kappa_3'=\kappa_3.$ Denote by $\boldsymbol{\kappa}'=(\kappa_1',\kappa_2',\kappa_3').$ Recall that $\underset{\kappa_{1}=1-s}{\Res}\underset{\kappa_{3}=1-s}{\Res}\mathcal{F}(\boldsymbol{\kappa},s)$ equals some holomorphic function multiplying the product of $\Lambda(2s-1,\tau^2)^2\Lambda(s,\tau)^2\Lambda(2-s,\tau^{-1})^{-2}$ and the function 
   	\begin{align*}
   	\frac{\Lambda(1+\kappa_2,\chi_{13})\Lambda(s+\kappa_2,\chi_{23}\tau)\Lambda(2s-1-\kappa_2,\chi_{32}\tau^2)\Lambda(3s-2-\kappa_2,\chi_{32}\tau^{3})}{\Lambda(1-\kappa_2,\chi_{32})\Lambda(s-\kappa_2,\chi_{32}\tau)\Lambda(2-s+\kappa_2,\chi_{23}\tau^{-1})\Lambda(3-2s+\kappa_2,\chi_{23}\tau^{-2})}.
   	\end{align*}
   	Then after the above changing of variables, we have that  $\underset{\kappa_{1}=1-s}{\Res}\underset{\kappa_{3}=1-s}{\Res}\mathcal{F}(\boldsymbol{\kappa},s)=\underset{\kappa_{1}'=1-s}{\Res}\underset{\kappa_{3}'=1-s}{\Res}\mathcal{F}(\boldsymbol{\kappa}',s)$ is equal to some holomorphic function multiplying the product of $\Lambda(2s-1,\tau^2)^2\Lambda(s,\tau)^2\Lambda(2-s,\tau^{-1})^{-2}$ and the function
   	\begin{equation}\label{254}
   	\frac{\Lambda(s+\kappa_2',\chi_{23}\tau)\Lambda(s-\kappa_2',\chi_{32}\tau)\Lambda(2s-1+\kappa_2',\chi_{23}\tau^2)\Lambda(2s-1-\kappa_2',\chi_{32}\tau^{2})}{\Lambda(1+\kappa_2',\chi_{23})\Lambda(1-\kappa_2',\chi_{32})\Lambda(2-s+\kappa_2',\chi_{23}\tau^{-1})\Lambda(2-s-\kappa_2',\chi_{32}\tau^{-2})}.
   	\end{equation}
   	One then sees that $H_{13}^{(2/3,1)}(s)$ admits a natural holomorphic continuation to the strip $1/2<\Re(s)<1.$ Now let $s\in\mathcal{R}(1/2)^+.$ Then we have 
   	\begin{equation}\label{255}
   	H_{13}^{(1/2,1)}(s)=\int_{\mathcal{C}}\underset{\kappa_{1}'=1-s}{\Res}\underset{\kappa_{3}'=1-s}{\Res}\mathcal{F}(\boldsymbol{\kappa}',s)d\kappa_2'-\underset{\kappa_{2}'=2s-1}{\Res}\underset{\kappa_{1}'=1-s}{\Res}\underset{\kappa_{3}'=1-s}{\Res}\mathcal{F}(\boldsymbol{\kappa}',s),
   	\end{equation}
   	where $\underset{\kappa_{2}'=2s-1}{\Res}\underset{\kappa_{1}'=1-s}{\Res}\underset{\kappa_{3}'=1-s}{\Res}\mathcal{F}(\boldsymbol{\kappa}',s)$ equals some holomorphic function multiplying
   	\begin{equation}\label{2s}
   	\frac{\Lambda(4s-2,\tau^4)\Lambda(3s-1,\tau^3)\Lambda(2s-1,\tau^2)^2\Lambda(s,\tau)^2\Lambda(1-s,\tau^{-1})}{\Lambda(3-3s,\tau^{-3})\Lambda(2-2s,\tau^{-2})\Lambda(2-s,\tau^{-1})^2\Lambda(2s,\tau^2)\Lambda(1+s,\tau)}.
   	\end{equation} 
   	Denote by $R_{213}\mathcal{F}(\boldsymbol{\kappa}',s)=\underset{\kappa_{2}'=2s-1}{\Res}\underset{\kappa_{1}'=1-s}{\Res}\underset{\kappa_{3}'=1-s}{\Res}\mathcal{F}(\boldsymbol{\kappa}',s).$ Applying the functional equation $\Lambda(2s-1,\tau^2)\sim \Lambda(2-2s,\tau^{-2})$ and $\Lambda(1-s,\tau^{-1})\sim \Lambda(s,\tau),$ one then has
   	\begin{equation}\label{256}
   	R_{213}\mathcal{F}(\boldsymbol{\kappa}',s)\sim \frac{\Lambda(4s-2,\tau^4)\Lambda(3s-1,\tau^3)\Lambda(2s-1,\tau^2)\Lambda(s,\tau)^3}{\Lambda(3-3s,\tau^{-3})\Lambda(2-s,\tau^{-1})^2\Lambda(2s,\tau^2)\Lambda(1+s,\tau)}.
   	\end{equation} 
   	
   	Note that for $s\in\mathcal{R}(1/2),$ $\Lambda(3-3s,\tau^{-3})^{-1}\cdot \Lambda(2s,\tau^{2})^{-1}$ is holomophic, since $3-3s$ and $2s$ lie in a zero-free region  (see Section \ref{7.11}). Hence $\underset{\kappa_{2}'=2s-1}{\Res}\underset{\kappa_{1}'=1-s}{\Res}\underset{\kappa_{3}'=1-s}{\Res}\mathcal{F}(\boldsymbol{\kappa}',s)$ admits a meromorphic continuation to the region $\mathcal{S}_{(1/3,1/2)}\cup\mathcal{R}(1/2),$ with a possible pole of order at most 2 at $s=1/2$ if $\tau^2=1.$ Moreover, if $L_F(1/2,\tau)=0,$ then $\Lambda(s,\tau)^{-1}\cdot\underset{\kappa_{2}'=2s-1}{\Res}\underset{\kappa_{1}'=1-s}{\Res}\underset{\kappa_{3}'=1-s}{\Res}\mathcal{F}(\boldsymbol{\kappa}',s)$ is holomorphic at $s=1/2.$
   	
   	On the other hand, the function $\int_{\mathcal{C}}\underset{\kappa_{1}'=1-s}{\Res}\underset{\kappa_{3}'=1-s}{\Res}\mathcal{F}(\boldsymbol{\kappa}',s)d\kappa_2'$ is clearly meromorphic in $\mathcal{R}(1/2),$ with a possible pole of order at most 2 at $s=1/2$ if $\tau^2=1.$ Moreover, if $L_F(1/2,\tau)=0,$ then $\Lambda(s,\tau)^{-1}\cdot\int_{\mathcal{C}}\underset{\kappa_{1}'=1-s}{\Res}\underset{\kappa_{3}'=1-s}{\Res}\mathcal{F}(\boldsymbol{\kappa}',s)d\kappa_2'$ is holomorphic at $s=1/2.$ Then we obtain a meromorphic continuation of $H_{13}^{(1/2,1)}(s)$ to the region $\mathcal{R}(1/2).$ Denote by $H_{13}^{1/2}(s)$ this continuation. 
   	
   	Let $s\in\mathcal{R}(1/2)^{-}.$ Then $\int_{\mathcal{C}}\underset{\kappa_{1}'=1-s}{\Res}\underset{\kappa_{3}'=1-s}{\Res}\mathcal{F}(\boldsymbol{\kappa}',s)d\kappa_2'$ is equal to
   	\begin{equation}\label{319}
   	\int_{(0)}\underset{\kappa_{1}'=1-s}{\Res}\underset{\kappa_{3}'=1-s}{\Res}\mathcal{F}(\boldsymbol{\kappa}',s)d\kappa_2'+\underset{\kappa_{2}'=1-2s}{\Res}\underset{\kappa_{1}'=1-s}{\Res}\underset{\kappa_{3}'=1-s}{\Res}\mathcal{F}(\boldsymbol{\kappa}',s),
   	\end{equation}
   	where $\underset{\kappa_{2}'=1-2s}{\Res}\underset{\kappa_{1}'=1-s}{\Res}\underset{\kappa_{3}'=1-s}{\Res}\mathcal{F}(\boldsymbol{\kappa}',s)$ is equal to a holomorphic function multiplying 
   	\begin{align*}
   	\frac{\Lambda(2s-1,\tau^2)^2\Lambda(s,\tau)^2\Lambda(1-s,\tau^{-1})\Lambda(3s-1,\tau^3)\Lambda(4s-2,\tau^4)}{\Lambda(2-s,\tau^{-1})^{2}\Lambda(2-2s,\tau^{-2})\Lambda(2s,\tau^2)\Lambda(3-3s,\tau^{-3})\Lambda(s+1,\tau)}.
   	\end{align*}
   	
   	Now we obtain from \eqref{254},  \eqref{255}, \eqref{256} and \eqref{319} a meromorphic continuation of $H_{13}^{1/2}(s)$ to the region $\mathcal{S}_{(1/3,1/2)}.$  Denote by $H_{13}^{(1/3,1/2)}(s)$ this continuation, then $H_{13}^{(1/3,1/2)}(s)$ can be expressed as 
   	\begin{align*}
   	\int_{(0)}\underset{\kappa_{1}'=1-s}{\Res}\underset{\kappa_{3}'=1-s}{\Res}\mathcal{F}(\boldsymbol{\kappa}',s)d\kappa_2'+\underset{\kappa_{2}'=1-2s}{\Res}\underset{\kappa_{1}'=1-s}{\Res}\underset{\kappa_{3}'=1-s}{\Res}\mathcal{F}(\boldsymbol{\kappa}',s)-R_{213}\mathcal{F}(\boldsymbol{\kappa}',s).
   	\end{align*}
   	In all, we obtain a meromorphic continuation of $H_{13}^{(1/2,1)}(s)$ to the region $\mathcal{S}_{(1/3,1)}:$
   	\begin{equation}\label{claim11}
   	\widetilde{H}_{13}(s)=\begin{cases}
   	H_{13}^{(1/2,1)}(s),\ s\in \mathcal{S}_{(1/2,1)};\\
   	H_{13}^{1/2}(s),\ s\in\mathcal{R}(1/2);\\
   	H_{13}^{(1/3,1/2)}(s),\ s\in \mathcal{S}_{(1/3,1/2)}.
   	\end{cases}
   	\end{equation}
   	
   	From the above discussions one sees that $\widetilde{H}_{13}(s)$ has a possible pole of order at most 2 at $s=1/2$ if $\tau^2=1.$ Moreover, if $L_F(1/2,\tau)=0,$ then $\Lambda(s,\tau)^{-1}\cdot\widetilde{H}_{13}(s)$ is holomorphic at $s=1/2.$ Now the proof of Claim \ref{70claim} is complete.
   \end{proof}

   \begin{proof}[Proof of Claim \ref{71claim}]
   	Let $s\in\mathcal{R}(1)^-.$ Let $H_{23}^{(2/3,1)}(s):=\int_{(0)}\underset{\kappa_{2}=1-s}{\Res}\underset{\kappa_{3}=1-s}{\Res}\mathcal{F}(\boldsymbol{\kappa},s)d\kappa_1.$ Recall that by \eqref{173.} one sees that $H_{23}^{(2/3,1)}(s)$ admits a natural holomorphic continuation to the strip $2/3<\Re(s)<1.$ Now let $s\in\mathcal{R}(2/3)^+.$ Then we have 
   	\begin{equation}\label{258}
   	H_{23}^{(2/3,1)}(s)=\int_{\mathcal{C}}\underset{\kappa_{2}=1-s}{\Res}\underset{\kappa_{3}=1-s}{\Res}\mathcal{F}(\boldsymbol{\kappa},s)d\kappa_1-\underset{\kappa_{1}=3s-2}{\Res}\underset{\kappa_{2}=1-s}{\Res}\underset{\kappa_{3}=1-s}{\Res}\mathcal{F}(\boldsymbol{\kappa},s),
   	\end{equation}
   	where $\underset{\kappa_{1}=3s-2}{\Res}\underset{\kappa_{2}=1-s}{\Res}\underset{\kappa_{3}=1-s}{\Res}\mathcal{F}(\boldsymbol{\kappa},s)$ equals some holomorphic function multiplying
   	\begin{align*}
   	\frac{\Lambda(4s-2,\tau^4)\Lambda(3s-2,\tau^3)\Lambda(2s-1,\tau^2)\Lambda(s,\tau)^2}{\Lambda(3-3s,\tau^{-3})\Lambda(3-2s,\tau^{-2})\Lambda(2-s,\tau^{-1})\Lambda(1+s,\tau)}.
   	\end{align*} 
   	Then by functional equation $\Lambda(3s-2,\tau^3)\sim \Lambda(3-3s,\tau^{-3}),$ we have that 
   	\begin{equation}\label{259}
   	\underset{\kappa_{1}=3s-2}{\Res}\underset{\kappa_{2}=1-s}{\Res}\underset{\kappa_{3}=1-s}{\Res}\mathcal{F}(\boldsymbol{\kappa},s)\sim \frac{\Lambda(4s-2,\tau^4)\Lambda(2s-1,\tau^2)\Lambda(s,\tau)^2}{\Lambda(3-2s,\tau^{-2})\Lambda(2-s,\tau^{-1})\Lambda(1+s,\tau)}.
   	\end{equation}
   	
   	Hence $\underset{\kappa_{1}=3s-2}{\Res}\underset{\kappa_{2}=1-s}{\Res}\underset{\kappa_{3}=1-s}{\Res}\mathcal{F}(\boldsymbol{\kappa},s)$ admits a meromorphic continuation to the region $\mathcal{R}(1/2)\cup\mathcal{S}_{(1/2,1)},$ with a possible pole of order at most 2 at $1/2$ if $\tau^2=1.$ 
   	
   	Moreover, $\int_{\mathcal{C}}\underset{\kappa_{2}=1-s}{\Res}\underset{\kappa_{3}=1-s}{\Res}\mathcal{F}(\boldsymbol{\kappa},s)d\kappa_1$ is meromorphic in the region $\mathcal{S}_{(1/3,2/3)}\cup\mathcal{R}(2/3),$ with a possible simple pole at $s=2/3$ if $\tau^3=1;$ and a possible simple pole at $1/2$ if $\tau^2=1.$ Thus we get a meromorphic continuation of $H_{23}^{(2/3,1)}(s)$ to the region $\mathcal{S}_{[1/3,2/3)}\cup\mathcal{R}(2/3).$ Denote by $H_{23}^{(1/3,2/3]}(s)$ this continuation. Now we obtain from \eqref{258} and \eqref{259} a meromorphic continuation of $H_{23}^{(2/3,1)}(s)$ to the region $\mathcal{S}_{(1/3,1)},$ namely,
   	\begin{equation}\label{claim12}
   	\widetilde{H}_{23}(s)=\begin{cases}
   	H_{23}^{(2/3,1)}(s),\ s\in \mathcal{S}_{(2/3,1)};\\
   	H_{23}^{(1/3,2/3]}(s),\ s\in \mathcal{S}_{(1/3,2/3)}\cup\mathcal{R}(2/3).
   	\end{cases}
   	\end{equation}
   	
   	From \eqref{173.} and the above formulas one sees that $\widetilde{H}_{23}(s)$ has possible poles at $s=2/3$ and $s=1/2;$ and these potential pole at $s=2/3$ is at most simple, the possible pole at $s=1/2$ has order at most 2. Moreover, from the above explicit expressions of $\widetilde{H}_{23}(s),$ we see that $\widetilde{H}_{23}(s)\cdot\Lambda(s,\tau)^{-1}$ has at most a simple pole at $s=1/2$ if $L_F(1/2,\tau)=0.$ In additional, if $L_F(2/3,\tau)=0,$ then by functional equation we have that $\Lambda(1/3,\tau^{-1})=0.$ Suppose that $\widetilde{H}_{23}(s)$ has a pole at $s=2/3.$ Then from the proceeding explicit expressions, we must have that $\tau^3=1,$ and the singular part of $\widetilde{H}_{23}(s)$ around $s=2/3$ is a holomorphic function multiplying $\Lambda(3s-2,\tau^3)\Lambda(2s-1,\tau^2).$ Note that $\Lambda(2s-1,\tau^2)\mid_{s=2/3}=\Lambda(1/3,\tau^2)=\Lambda(1/3,\tau^{-1})=0.$ Hence, when $L_F(2/3,\tau)=0,$ $\widetilde{H}_{23}(s)$ is holomorphic at $s=2/3.$ Now the proof of Claim \ref{71claim} is complete.
   \end{proof}
   
   \begin{remark}
   	One can of course deal with each individual $\iint\mathcal{F}(\boldsymbol{\kappa},s)$ instead of the infinite sum $\sum_{\chi}\sum_{\phi}\iint\mathcal{F}(\boldsymbol{\kappa},s).$ However, without Proposition \ref{57prop} or Proposition \ref{58prop}, the expression of each single $\iint\mathcal{F}(\boldsymbol{\kappa},s)$ would be super complicated. For example, one needs to consider residues with respect to $\kappa_{12}.$ We give meromorphic continuation of $\iint\mathcal{F}(\boldsymbol{\kappa},s)$ as follows, which involves 56 terms in total for $\GL(3)$ case (also some of them are same but locate in different regions). Let $J(s)=\int_{(0)}\int_{(0)}\mathcal{F}(\boldsymbol{\kappa},s)d\kappa_1d\kappa_2.$ When $s\in\mathcal{R}(1)^+,$ we have, by Cauchy integral formula, that $J(s)$ is equal to 
   	\begin{align*}
   	&\int_{\mathcal{C}}\int_{\mathcal{C}}\mathcal{F}(\boldsymbol{\kappa},s)d\kappa_1d\kappa_2-\int_{\mathcal{C}}\underset{\kappa_1=s-1}{\Res}\mathcal{F}(\boldsymbol{\kappa},s){d\kappa_2}-\int_{\mathcal{C}}\underset{\kappa_2=s-1}{\Res}\mathcal{F}(\boldsymbol{\kappa},s){d\kappa_1}\\
   	&-\int_{\mathcal{C}}\underset{\kappa_2=s-1-\kappa_1}{\Res}\mathcal{F}(\boldsymbol{\kappa},s){d\kappa_1}+\underset{\kappa_1=s-1}{\Res}\underset{\kappa_2=s-1}{\Res}\mathcal{F}(\boldsymbol{\kappa},s)+\underset{\kappa_1=2s-2}{\Res}\underset{\kappa_2=s-1-\kappa_1}{\Res}\mathcal{F}(\boldsymbol{\kappa},s).
   	\end{align*}
   	Since the right hand side is meromorphic in $\mathcal{R}(1),$ we get meromorphic continuation of $J(s)$ in $\mathcal{R}(1)^-.$ Denote by $J_1(s)$ this continuation. Let $s\in \mathcal{R}(1)^-.$ Then 
   	\begin{align*}
   	J_1(s)=&\int_{(0)}\int_{(0)}\mathcal{F}(\boldsymbol{\kappa},s)d\kappa_1d\kappa_2+\int_{(0)}\underset{\kappa_1=1-s}{\Res}\mathcal{F}(\boldsymbol{\kappa},s)d\kappa_2+\int_{(0)}\underset{\kappa_2=1-s-\kappa_1}{\Res}\mathcal{F}(\boldsymbol{\kappa},s)d\kappa_1\\
   	&+\int_{(0)}\underset{\kappa_2=1-s}{\Res}\mathcal{F}(\boldsymbol{\kappa},s)d\kappa_1-\int_{(0)}\underset{\kappa_2=s-1}{\Res}\mathcal{F}(\boldsymbol{\kappa},s)d\kappa_1-\int_{(0)}\underset{\kappa_1=s-1}{\Res}\mathcal{F}(\boldsymbol{\kappa},s)d\kappa_2-\\
   	&\int_{(0)}\underset{\kappa_2=s-1-\kappa_1}{\Res}\mathcal{F}(\boldsymbol{\kappa},s)d\kappa_1+\underset{\kappa_1=1-s}{\Res}\underset{\kappa_2=1-s}{\Res}\mathcal{F}(\boldsymbol{\kappa},s)+\underset{\kappa_1=s-1}{\Res}\underset{\kappa_2=s-1}{\Res}\mathcal{F}(\boldsymbol{\kappa},s)-\\
   	&\underset{\kappa_1=2-2s}{\Res}\underset{\kappa_2=s-1}{\Res}\mathcal{F}(\boldsymbol{\kappa},s)-\underset{\kappa_1=1-s}{\Res}\underset{\kappa_2=s-1-\kappa_1}{\Res}\mathcal{F}(\boldsymbol{\kappa},s)-\underset{\kappa_2=2-2s}{\Res}\underset{\kappa_1=s-1}{\Res}\mathcal{F}(\boldsymbol{\kappa},s)\\
   	&+\underset{\kappa_1=2-2s}{\Res}\underset{\kappa_2=s-1-\kappa_1}{\Res}\mathcal{F}(\boldsymbol{\kappa},s),
   	\end{align*}
   	where the right hand side is meromorphic in $1/2<\Re(s)<1.$ Hence we obtain a meromorphic of $J_1(s)$ to the domain $\mathcal{S}_{(1/2,1)}.$ Denote by $J_2(s)$ this continuation. Let $s\in \mathcal{R}(1/2)^+.$ Then we have, again, by Cauchy integral formula, that 
   	\begin{align*}
   	J_2(s)=&\int_{(0)}\int_{(0)}\mathcal{F}(\boldsymbol{\kappa},s)d\kappa_1d\kappa_2+\int_{\mathcal{C}}\underset{\kappa_1=1-s}{\Res}\mathcal{F}(\boldsymbol{\kappa},s)d\kappa_2+\int_{\mathcal{C}}\underset{\kappa_2=1-s-\kappa_1}{\Res}\mathcal{F}(\boldsymbol{\kappa},s)d\kappa_1\\
   	&+\int_{\mathcal{C}}\underset{\kappa_2=1-s}{\Res}\mathcal{F}(\boldsymbol{\kappa},s)d\kappa_1-\int_{\mathcal{C}}\underset{\kappa_2=s-1}{\Res}\mathcal{F}(\boldsymbol{\kappa},s)d\kappa_1-\int_{\mathcal{C}}\underset{\kappa_1=s-1}{\Res}\mathcal{F}(\boldsymbol{\kappa},s)d\kappa_2-\\
   	&\int_{\mathcal{C}}\underset{\kappa_2=s-1-\kappa_1}{\Res}\mathcal{F}(\boldsymbol{\kappa},s)d\kappa_1+\underset{\kappa_1=1-s}{\Res}\underset{\kappa_2=1-s}{\Res}\mathcal{F}(\boldsymbol{\kappa},s)+\underset{\kappa_1=s-1}{\Res}\underset{\kappa_2=s-1}{\Res}\mathcal{F}(\boldsymbol{\kappa},s)-\\
   	&\underset{\kappa_1=2-2s}{\Res}\underset{\kappa_2=s-1}{\Res}\mathcal{F}(\boldsymbol{\kappa},s)-\underset{\kappa_1=1-s}{\Res}\underset{\kappa_2=s-1-\kappa_1}{\Res}\mathcal{F}(\boldsymbol{\kappa},s)-\underset{\kappa_2=2s-1}{\Res}\underset{\kappa_1=1-s}{\Res}\mathcal{F}(\boldsymbol{\kappa},s)+\\
   	&\underset{\kappa_1=2-2s}{\Res}\underset{\kappa_2=s-1-\kappa_1}{\Res}\mathcal{F}(\boldsymbol{\kappa},s)-\underset{\kappa_2=2-2s}{\Res}\underset{\kappa_1=s-1}{\Res}\mathcal{F}(\boldsymbol{\kappa},s)-\underset{\kappa_1=2s-1}{\Res}\underset{\kappa_2=1-s}{\Res}\mathcal{F}(\boldsymbol{\kappa},s)\\
   	&+\underset{\kappa_1=2s-1}{\Res}\underset{\kappa_2=s-1-\kappa-1}{\Res}\mathcal{F}(\boldsymbol{\kappa},s)-\underset{\kappa_1=2s-1}{\Res}\underset{\kappa_2=1-s-\kappa_1}{\Res}\mathcal{F}(\boldsymbol{\kappa},s),
   	\end{align*}
   	where the right hand side is meromorphic in $\mathcal{R}(1/2).$ Hence we obtain a meromorphic continuation of $J_2(s)$ in $s\in \mathcal{R}(1/2).$ Let $s\in \mathcal{R}(1/2)^-.$ Then we have, again, by Cauchy integral formula, that 
   	\begin{align*}
   	J_2(s)=&\int_{(0)}\int_{(0)}\mathcal{F}(\boldsymbol{\kappa},s)d\kappa_1d\kappa_2+\int_{\mathcal{C}}\underset{\kappa_1=1-s}{\Res}\mathcal{F}(\boldsymbol{\kappa},s)d\kappa_2+\int_{\mathcal{C}}\underset{\kappa_2=1-s-\kappa_1}{\Res}\mathcal{F}(\boldsymbol{\kappa},s)d\kappa_1\\
   	&+\int_{\mathcal{C}}\underset{\kappa_2=1-s}{\Res}\mathcal{F}(\boldsymbol{\kappa},s)d\kappa_1-\int_{(0)}\underset{\kappa_2=s-1}{\Res}\mathcal{F}(\boldsymbol{\kappa},s)d\kappa_1-\int_{(0)}\underset{\kappa_1=s-1}{\Res}\mathcal{F}(\boldsymbol{\kappa},s)d\kappa_2-\\
   	&\int_{(0)}\underset{\kappa_2=s-1-\kappa_1}{\Res}\mathcal{F}(\boldsymbol{\kappa},s)d\kappa_1+\underset{\kappa_1=1-s}{\Res}\underset{\kappa_2=1-s}{\Res}\mathcal{F}(\boldsymbol{\kappa},s)+\underset{\kappa_1=s-1}{\Res}\underset{\kappa_2=s-1}{\Res}\mathcal{F}(\boldsymbol{\kappa},s)-\\
   	&\underset{\kappa_1=2-2s}{\Res}\underset{\kappa_2=s-1}{\Res}\mathcal{F}(\boldsymbol{\kappa},s)-\underset{\kappa_1=1-s}{\Res}\underset{\kappa_2=s-1-\kappa_1}{\Res}\mathcal{F}(\boldsymbol{\kappa},s)-\underset{\kappa_2=2s-1}{\Res}\underset{\kappa_1=1-s}{\Res}\mathcal{F}(\boldsymbol{\kappa},s)+\\
   	&\underset{\kappa_1=2-2s}{\Res}\underset{\kappa_2=s-1-\kappa_1}{\Res}\mathcal{F}(\boldsymbol{\kappa},s)-\underset{\kappa_2=2-2s}{\Res}\underset{\kappa_1=s-1}{\Res}\mathcal{F}(\boldsymbol{\kappa},s)-\underset{\kappa_1=2s-1}{\Res}\underset{\kappa_2=1-s}{\Res}\mathcal{F}(\boldsymbol{\kappa},s)\\
   	&+\underset{\kappa_1=2s-1}{\Res}\underset{\kappa_2=s-1-\kappa-1}{\Res}\mathcal{F}(\boldsymbol{\kappa},s)-\underset{\kappa_1=2s-1}{\Res}\underset{\kappa_2=1-s-\kappa_1}{\Res}\mathcal{F}(\boldsymbol{\kappa},s)+\\
   	&\underset{\kappa_1=1-2s}{\Res}\underset{\kappa_2=s-1}{\Res}\mathcal{F}(\boldsymbol{\kappa},s)+\underset{\kappa_2=1-2s}{\Res}\underset{\kappa_1=s-1}{\Res}\mathcal{F}(\boldsymbol{\kappa},s)+\underset{\kappa_1=1-2s}{\Res}\underset{\kappa_2=s-1-\kappa_1}{\Res}\mathcal{F}(\boldsymbol{\kappa},s),
   	\end{align*}
   	where the right hand side is meromorphic in $1/3<\Re(s)<1/2.$ Hence we obtain a meromorphic continuation of $J_2(s)$ in $s\in \mathcal{S}(1/3,1/2).$ Therefore, putting the above computation together, we get a meromorphic continuation of $J(s)$ to the domain $s\in \mathcal{S}_{(1/3,1)}.$
   	
   	Then one needs to investigate these terms individually. What is worse, the situation would be much more complicated in $\GL_4$ case. 
   \end{remark}
   
\bibliographystyle{alpha}

\bibliography{TF}

\end{document}